\documentclass[francais,12pt]{smfart}

\usepackage{amsmath,amssymb,amsfonts,amscd,footnote}
\usepackage{mathtools}
\usepackage{dsfont}
\usepackage[all]{xy}
\usepackage{pictexwd,dcpic}

\newtheorem{prop0}{Proposition}[section]

\newtheorem{conj0}[prop0]{Conjecture}

\newtheorem{thm0}[prop0]{Th\'eor\`eme}

\newtheorem{prop}{Proposition}[subsection]

\newtheorem{conj}[prop]{Conjecture}

\newtheorem{lem}[prop]{Lemme}
\newtheorem{thm}[prop]{Th\'eor\`eme}
\newtheorem{ex}[prop]{Exemple}
\newtheorem{rem}[prop]{Remarque}
\newtheorem{cor}[prop]{Corollaire}

\newcommand{\norm}{\vert\cdot\vert}

\def\={\buildrel {\rm d\acute ef}\over =}
\newcommand{\smat}[1]{\left( \begin{smallmatrix} #1 \end{smallmatrix} \right)}
\def\ilim#1{\displaystyle \lim_{\stackrel{\longrightarrow}{#1}}}
\def\plim#1{\displaystyle \lim_{\stackrel{\longleftarrow}{#1}}}
\def\varddots{\mathinner{\raise7pt\vbox{\kern3pt\hbox{.}}\mkern1mu\smash{\raise4pt\hbox{.}}\mkern1mu\smash{\raise1pt\hbox{.}}}}

\DeclareMathOperator{\Id}{Id}
\DeclareMathOperator{\im}{im}
\DeclareMathOperator{\HT}{HT}
\DeclareMathOperator{\dR}{dR}
\DeclareMathOperator{\an}{an}
\DeclareMathOperator{\Res}{Res}

\DeclareMathOperator{\tor}{tor}

\DeclareMathOperator{\GL}{GL}
\DeclareMathOperator{\soc}{soc}
\DeclareMathOperator{\sep}{s\acute ep}
\DeclareMathOperator{\alg}{alg}
\DeclareMathOperator{\cInd}{c-Ind}
\DeclareMathOperator{\Frob}{Frob}
\DeclareMathOperator{\End}{End}
\DeclareMathOperator{\Ind}{Ind}

\DeclareMathOperator{\Fil}{Fil}
\DeclareMathOperator{\diag}{diag}
\DeclareMathOperator{\Ext}{Ext}

\DeclareMathOperator{\Hom}{Hom}
\DeclareMathOperator{\Gal}{Gal}

\DeclareMathOperator{\val}{val}

\DeclareMathOperator{\nr}{nr}

\newcommand{\C}{{\mathbb C}}
\newcommand{\Z}{{\mathbb Z}}

\newcommand{\Q}{{\mathbb Q}}
\newcommand{\Qp}{\Q_{p}}
\newcommand{\Zp}{\Z_{p}}

\newcommand{\Rep}{{\rm Rep}^{\rm an}_E}
\newcommand{\Repm}{{\rm Rep}^{\rm an}_{E_m}}

\newcommand{\Qpbar}{\overline\Qp}

\newcommand{\oE}{{\mathcal O}_E}

\newcommand{\G}{{\mathrm{GL}}_n}

\newcommand{\R}{{\mathcal R}_E}
\newcommand{\Rr}{{\mathcal R}^r_{E}}
\newcommand{\Rpr}{{\mathcal R}^{pr}_{E}}
\newcommand{\Rm}{{\mathcal R}_{E_m}}
\newcommand{\Rrm}{{\mathcal R}^r_{E_m}}
\newcommand{\Rprm}{{\mathcal R}^{pr}_{E_m}}

\newcommand{\cF}{\mathcal F}
\newcommand{\cV}{\mathcal V}
\newcommand{\cW}{\mathcal W}

\newcommand{\St}{{\mathrm{St}}^{\infty}}

\newcommand{\mrho}{{\mathfrak m}_\rho}
\newcommand{\gp}{{\Gal}(\Qpbar/\Qp)}

\newcommand{\glp}{{\Gal}(L/\Qp)}
\newcommand{\mg}{\mathfrak{g}}

\newcommand{\mnn}{{\mathfrak{n}}}

\newcommand{\mb}{\mathfrak{b}}

\newcommand{\mpp}{\mathfrak{p}}

\newcommand{\mq}{\mathfrak{q}}
\newcommand{\mt}{\mathfrak{t}}

\newcommand{\fn}{\mathfrak n}
\newcommand{\fh}{\mathfrak h}

\newcommand{\fl}{\mathfrak l}
\newcommand{\fp}{\mathfrak p}
\newcommand{\ug}{\mathfrak g}
\newcommand{\fx}{\mathfrak x}

\newcommand{\fz}{\mathfrak z}

\newcommand{\fc}{\mathfrak c}

\newcommand{\fy}{\mathfrak y}

\newcommand{\ra}{\rightarrow}
\newcommand{\lra}{\longrightarrow}

\author[C. Breuil]{Christophe Breuil}
\address{Laboratoire de Math\'ematique d'Orsay\\
C.N.R.S., Universit\'e Paris-Sud\\
Universit\'e Paris-Saclay\\
91405 Orsay\\
France}
\email{christophe.breuil@math.u-psud.fr}

\author[Y. Ding]{Yiwen Ding}
\address{Beijing Internat. Center for Math. Research\\
Peking University\\
No5 Yiheyuan Road Haidian District\\
Beijing\\
P.R. China 100871}
\email{yiwen.ding@bicmr.pku.edu.cn}

\thanks{Le premier auteur remercie pour leur soutien le CNRS, l'Universit\'e Paris-Sud et l'ANR CLap-CLap. Le second auteur a b\'en\'efici\'e de la Grant n\textsuperscript{o} \!7101500268 de l'Universit\'e de P\'ekin. Les auteurs remercient A. Abbes, L. Berger, G. Dospinescu, O. Fouquet, D. Harari, J. Hauseux, G. Henniart, R. Liu, Z. Qian et B. Schraen pour des discussions ou pour leurs r\'eponses \`a leurs questions.}

\title[Un probl\`eme de compatibilit\'e local-global localement analytique]{Sur un probl\`eme de compatibilit\'e local-global localement analytique}

\begin{document} 

\begin{abstract}
 On r\'einterpr\'ete et on pr\'ecise la conjecture du $\Ext^1$ localement analytique de \cite{Br1} de mani\`ere fonctorielle en utilisant les $(\varphi,\Gamma)$-modules sur l'anneau de Robba (avec \'eventuellement de la $t$-torsion). Puis on d\'emontre plusieurs cas particuliers ou partiels de cette conjecture ``am\'elior\'ee'', notamment pour ${\rm GL}_3(\Qp)$.
\end{abstract}

\begin{altabstract}
We reinterpret the main conjecture of \cite{Br1} on the locally analytic $\Ext^1$ in a functorial way using $(\varphi,\Gamma)$-modules (possibly with $t$-torsion) over the Robba ring, making it more accurate. Then we prove several special or partial cases of this ``improved'' conjecture, notably for ${\rm GL}_3(\Qp)$.
\end{altabstract}

\maketitle

\begin{flushright}
\it \`A Jean-Marc Fontaine et Jean-Pierre Wintenberger
\vspace*{0.5cm}
\end{flushright}

\setcounter{tocdepth}{2}

\tableofcontents

\section{Introduction}\label{intro}

L'id\'ee d'utiliser les $(\varphi,\Gamma)$-modules de Fontaine dans le programme de Langlands $p$-adique est due \`a Colmez. Elle lui a permis de construire un foncteur exact (qui porte son nom) associant un $(\varphi,\Gamma)$-module \'etale de $p^m$-torsion \`a une repr\'esentation de ${\rm GL}_2(\Qp)$ de longueur finie annul\'ee par $p^m$, puis par la suite de d\'emontrer la correspondance de Langlands locale $p$-adique pour ${\rm GL}_2(\Qp)$ (\cite{Co2}, \cite{CDP}).

Lorsque l'on s'attelle \`a d'autres groupes que ${\rm GL}_2(\Qp)$, par exemple ${\rm GL}_n(\Qp)$, les repr\'esentations localement analytiques (ou $p$-adiques) qui apparaissent dans les composantes Hecke-isotypiques des espaces de formes automorphes $p$-adiques sont {\it beaucoup plus} compliqu\'ees. Dans \cite{Br1} est formul\'ee une conjecture qui relie ce qui se passe ``juste apr\`es'' les vecteurs localement alg\'ebriques (dans ces repr\'esentations) au dernier cran de la filtration de Hodge sur les puissances altern\'ees du module filtr\'e de Fontaine sous-jacent. Le but principal de cet article est de r\'einterpr\'eter et de pr\'eciser cette conjecture de mani\`ere fonctorielle en utilisant les $(\varphi,\Gamma)$-modules sur l'anneau de Robba (avec \'eventuellement de la $t$-torsion), puis de d\'emontrer plusieurs cas particuliers ou partiels de cette conjecture ``am\'elior\'ee'', notamment pour ${\rm GL}_3(\Qp)$.

Avant de rentrer dans les d\'etails, rappelons de mani\`ere un peu moins vague la conjecture de \cite{Br1}. Fixons un corps des coefficients $E$ (une extension finie de $\Qp$) et une repr\'esentation de de Rham $\rho_p$ de $\gp$ de dimension $n\geq 2$ sur $E$ et de poids de Hodge-Tate $h_1>h_2>\cdots >h_n$ distincts. Notons $D_{\dR}(\rho_p)$ le module filtr\'e de $\rho_p$ et ${\rm alg}\otimes_E \pi^\infty$ la repr\'esentation localement alg\'ebrique de ${\rm GL}_n(\Qp)$ sur $E$ usuelle associ\'ee \`a $\rho_p$, qui ne d\'epend que des poids de Hodge-Tate et de la repr\'esentation de Weil-Deligne $W$ de $\rho_p$. On devrait pouvoir associer \`a $\rho_p$ une (au moins) repr\'esentation localement analytique admissible (au sens de \cite{ST2}) de ${\rm GL}_n(\Qp)$ sur $E$ contenant ${\rm alg}\otimes_E \pi^\infty$, que l'on note juste $\pi^{\rm an}$ dans cette introduction, soit en prenant une sp\'ecialisation (suppos\'ee non nulle) $V(\rho_p)$ comme dans \cite[\S~2.12]{CEGGPS} (voir aussi \cite{Py}), soit en supposant que $\rho_p$ se globalise en une repr\'esentation galoisienne automorphe $\rho$ (pour un groupe unitaire d\'eploy\'e en $p$ compact \`a l'infini) et en prenant les vecteurs localement analytiques de la composante $\rho$-isotypique dans l'espace des formes automorphes $p$-adiques. La conjecture de \cite{Br1} postule alors l'existence, pour toute racine simple $\alpha=e_j-e_{j+1}$ de ${\rm GL}_n$ ($j\in \{1,\dots,n-1\}$), d'une repr\'esentation localement analytique $\pi^\alpha$ admissible de longueur finie ne d\'ependant que des $h_i$, de $W$ et de $\alpha$, et d'un isomorphisme~:
\begin{equation}\label{isointro}
\Ext^1_{{\rm GL}_n(\Qp)}(\pi^\alpha,{\rm alg}\otimes_E \pi^\infty)\simeq \wedge_E^{n-j}D_{\dR}(\rho_p)
\end{equation}
(ne d\'ependant aussi que des $h_i$ et de $W$) tel que tout plongement ${\rm alg}\otimes_E \pi^\infty\hookrightarrow \pi^{\an}$ s'\'etende en un plongement $\!\begin{xy}(-60,0)*+{({\rm alg}\otimes_E \pi^\infty)}="a";(-40,0)*+{\pi^\alpha}="b";{\ar@{-}"a";"b"}\end{xy}\!\hookrightarrow \pi^{\an}$ o\`u $\!\begin{xy}(-60,0)*+{({\rm alg}\otimes_E \pi^\infty)}="a";(-40,0)*+{\pi^\alpha}="b";{\ar@{-}"a";"b"}\end{xy}\!$ est l'extension (non scind\'ee) donn\'ee par l'image inverse par l'isomorphisme (\ref{isointro}) de la droite $\Fil^{\rm max}_\alpha(\rho_p)$ de $\wedge_E^iD_{\dR}(\rho_p)$ produit altern\'e des $n-j$ derniers crans de la filtration de Hodge sur $D_{\dR}(\rho_p)$, i.e. (cf. (\ref{filmax}))~:
$${\rm Fil}^{\rm max}_\alpha(\rho_p)\={\rm Fil}^{-h_{j+1}}(D_{\dR}(\rho_p))\wedge {\rm Fil}^{-h_{j+2}}(D_{\dR}(\rho_p))\wedge \cdots \wedge {\rm Fil}^{-h_n}(D_{\dR}(\rho_p)).$$

Une des id\'ees \`a la base de cet article a \'et\'e de remarquer qu'il existe un groupe $\Ext^1$ c\^ot\'e $\gp$, ou plut\^ot c\^ot\'e $(\varphi,\Gamma)$-modules, naturellement isomorphe \`a $\wedge_E^{n-j}D_{\dR}(\rho_p)$. Rappelons que Berger dans \cite[Th.~A]{Be2} associe \`a {\it toute} filtration (d\'ecroissante exhaustive) sur $\wedge_E^{n-j}D_{\dR}(\rho_p)$ un certain $(\varphi,\Gamma)$-module (libre de rang fini) sur l'anneau de Robba $\R$ \`a coefficients dans $E$. Par exemple si l'on prend la filtration sur $\wedge_E^{n-j}D_{\dR}(\rho_p)$ induite par la filtration de Hodge sur $D_{\dR}(\rho_p)$, on retrouve le $(\varphi,\Gamma)$-module \'etale sur $\R$ associ\'e \`a $\wedge_E^{n-j}\rho_p$. Si l'on prend la filtration triviale ${\rm Fil}^0=$ tout, ${\rm Fil}^1=0$ on obtient l'\'equation diff\'erentielle $p$-adique $D(\wedge_E^{n-j}W)$ associ\'ee \`a $\wedge_E^{n-j}\rho_p$ (\cite{Be1}), qui ne d\'epend que de $\wedge_E^{n-j}W$.

\begin{prop0}[Proposition \ref{lienfilmax}]\label{filmaxintro}
Pour tout $j\in \{1,\dots,n-1\}$ on a un isomorphisme naturel (o\`u $t=\log(1+X)\in \R$)~:
\begin{equation}\label{filintro}
\Ext^1_{(\varphi,\Gamma)}\big(\R/(t^{h_j-h_{j+1}}),D(\wedge_E^{n-j}W)\otimes_{\R}\R(x^{h_j-h_{j+1}})\big)\buildrel\sim\over\longrightarrow \wedge_E^{n-j}D_{\dR}(\rho_p)
\end{equation}
tel que l'image inverse d'une droite $F\subseteq \wedge_E^{n-j}D_{\dR}(\rho_p)$ (vue \`a isomorphisme pr\`es comme $(\varphi,\Gamma)$-module) est le $(\varphi,\Gamma)$-module associ\'e \`a la filtration ${\rm Fil}^{-(h_j-h_{j+1})}=\wedge_E^{n-j}D_{\dR}(\rho_p)$, ${\rm Fil}^{-(h_j-h_{j+1})+1}=\cdots={\rm Fil}^0=F$, ${\rm Fil}^1=0$.
\end{prop0}

Il est alors naturel de penser que l'isomorphisme conjectural (\ref{isointro}) devrait se d\'ecomposer en deux isomorphismes, l'un donn\'e par (\ref{filintro}) et l'autre donn\'e comme suit (en notant $D(\wedge_E^{n-j}W)(x^{h_j-h_{j+1}})\=D(\wedge_E^{n-j}W)\otimes_{\R}\R(x^{h_j-h_{j+1}})$)~:
\begin{equation}\label{iso1intro}
\Ext^1_{{\rm GL}_n(\Qp)}(\pi^\alpha,{\rm alg}\otimes_E \pi^\infty)\buildrel\sim\over\longrightarrow \Ext^1_{(\varphi,\Gamma)}\big(\R/(t^{h_j-h_{j+1}}),D(\wedge_E^{n-j}W)(x^{h_j-h_{j+1}})\big).
\end{equation}
L'isomorphisme (\ref{iso1intro}) sugg\`ere alors l'existence d'un foncteur contravariant $D_\alpha$ tel que, au moins \`a torsion pr\`es, on ait $D_\alpha({\rm alg}\otimes_E \pi^\infty)\simeq \R/(t^{h_j-h_{j+1}})$ et $D_\alpha(\pi^\alpha)\simeq D(\wedge_E^{n-j}W)(x^{h_j-h_{j+1}})$, tel que $D_\alpha$ induise (\ref{iso1intro}), et tel qu'apparaisse dans $\pi^{\an}$ l'unique extension de $\pi^\alpha$ par ${\rm alg}\otimes_E \pi^\infty$ envoy\'ee par $D_\alpha$ vers le $(\varphi,\Gamma)$-module associ\'e par \cite[Th.~A]{Be2} \`a la filtration ${\rm Fil}^{-(h_j-h_{j+1})}=\wedge_E^{n-j}D_{\dR}(\rho_p)$, ${\rm Fil}^{-(h_j-h_{j+1})+1}=\cdots={\rm Fil}^0={\rm Fil}^{\rm max}_\alpha(\rho_p)$, ${\rm Fil}^1=0$.

Notre id\'ee de d\'epart pour essayer de construire $D_\alpha$ est d'adapter au cadre localement analytique le foncteur d\'efini dans \cite{Br2} \'etendant celui de Colmez \`a des repr\'esentations lisses de ${\rm GL}_n(\Qp)$ (ou de groupes plus g\'en\'eraux) en caract\'eristique $p$ (ou de $p^m$-torsion), plus pr\'ecis\'ement la variante de \cite{Br2} o\`u l'on ne consid\`ere qu'une seule racine simple $\alpha$ (\cite{EZ}, \cite{Za}). R\'esumons la construction de \cite{Br2}, \cite{EZ}. Notons $N^\alpha_0$ un sous-groupe ouvert compact du sous-groupe $N^\alpha(\Qp)$ des unipotents sup\'erieurs $N(\Qp)$ de ${\rm GL}_n(\Qp)$ o\`u l'on a ``enlev\'e'' la racine $\alpha$ (i.e. l'entr\'ee correspondante est nulle), si $\pi$ est une repr\'esentation lisse de ${\rm GL}_n(\Qp)$ en caract\'eristique $p$, on peut munir le dual $(\pi^{N^\alpha_0})^\vee$ d'une structure naturelle de $(\psi,\Gamma)$-module compact o\`u rappelons que l'op\'erateur $\psi$ (dans un $(\varphi,\Gamma)$-module \'etale en caract\'eristique $p$) est l'inverse \`a gauche de $\varphi$. On peut alors associer \`a $\pi$ le foncteur covariant $F_\alpha(\pi)$ sur la cat\'egorie ab\'elienne des $(\varphi,\Gamma)$-modules \'etales en caract\'eristique $p$ en envoyant un tel $(\varphi,\Gamma)$-module $D$ vers $F_\alpha(\pi)(D)\=\Hom_{\psi,\Gamma}((\pi^{N^\alpha_0})^\vee,D)$ (morphismes continus de $(\psi,\Gamma)$-modules). Il r\'esulte alors de \cite[Rem.~5.6(iii)]{Br2} avec \cite[Prop.~3.2(ii)]{Br2} que le foncteur $\pi\mapsto F_\alpha(\pi)$ est exact \`a gauche et que $F_\alpha(\pi)$ est pro-repr\'esentable par un pro-$(\varphi,\Gamma)$-module \'etale $D_\alpha(\pi)$. De plus $D_\alpha(\pi)$ est un vrai $(\varphi,\Gamma)$-module (i.e. de rang fini) au moins lorsque les constituants irr\'eductibles de $\pi$ sont en nombre fini et sous-quotients de s\'eries principales (\cite[Cor.~9.3]{Br2}).

La situation se complique lorsque l'on consid\`ere des repr\'esentations localement a\-nalytiques de ${\rm GL}_n(\Qp)$ (sur des $E$-espaces vectoriels de type compact) au lieu de repr\'esentations en caract\'eristique $p$. On peut penser associer \`a une telle repr\'esentation $\pi$ le foncteur $D\mapsto \Hom_{\psi,\Gamma}((\pi^{N^\alpha_0})^\vee,D)$ pour $D$ dans la cat\'egorie ab\'elienne des $(\varphi,\Gamma)$-modules g\'en\'eralis\'es sur $\R$ (\cite{Li}). Mais on tombe sur plusieurs probl\`emes techniques. Par exemple, \'ecrivant $D={\displaystyle \lim_{r\rightarrow +\infty}}D_r$ o\`u $D_r$ est un $(\psi,\Gamma)$-module sur $\Rr=$ le $E$-espace de Fr\'echet des fonctions rigides analytiques sur la couronne $p^{-1/r}\leq \norm <1$, il n'est d'abord pas clair que l'image d'un morphisme dans $\Hom_{\psi,\Gamma}((\pi^{N^\alpha_0})^\vee,D)$ tombe dans un $D_r$ pour $r\gg 0$, de sorte qu'il vaut mieux consid\'erer $D\mapsto {\displaystyle \lim_{r\rightarrow +\infty}}\Hom_{\psi,\Gamma}((\pi^{N^\alpha_0})^\vee,D_r)$. Par ailleurs, si $\pi$ est une repr\'esentation de Steinberg g\'en\'eralis\'ee localement analytique, alors $D\mapsto {\displaystyle \lim_{r\rightarrow +\infty}}\Hom_{\psi,\Gamma}((\pi^{N^\alpha_0})^\vee,D_r)$ n'est pas le foncteur nul, alors que l'on aimerait que la contribution de ces Steinberg g\'en\'eralis\'ees soit nulle. Dans la th\'eorie classique des repr\'esentations lisse (sur $E$), une mani\`ere d'annuler les Steinberg g\'en\'eralis\'ees est de consid\'erer, plut\^ot que le foncteur de Jacquet usuel $\pi_{N(\Qp)}$, la d\'eriv\'ee sup\'erieure de Bernstein $\pi(\eta^{-1})_{N(\Qp)}$ pour un caract\`ere (g\'en\'erique) non trivial $\eta:N(\Qp)\rightarrow E_\infty^\times$ o\`u $E_\infty=\cup_{m\in \Z_{\geq 0}}E_m$ avec $E_m=E(\sqrt[p^m]{1})$ (i.e. on tord l'action de $N(\Qp)$ par $\eta^{-1}$). Revenant \`a $\pi^{N^\alpha_0}$, on tombe alors sur le fait que $\eta$ peut \^etre tri\-vial sur $N^\alpha_0$, il faut donc remplacer $\pi^{N^\alpha_0}$ par autre chose. Un examen du cas o\`u $\pi$ est une repr\'esentation lisse de ${\rm GL}_n(\Qp)$ sur $E$ montre qu'en fixant une famille croissante $(N_m)_{m\in \Z_{\geq 0}}$ de sous-groupes ouverts compacts de $N(\Qp)$ telle que $\cup_{m}N_m=N(\Qp)$ et en consid\'erant $F_\alpha(\pi):D\mapsto {\displaystyle \lim_{r,m\rightarrow +\infty}}\Hom_{\psi,\Gamma}((\pi\otimes_EE_m)(\eta)_{N^\alpha_m})^\vee,D_r\otimes_EE_m)$ o\`u $N_m^\alpha\=N_m\cap N^\alpha(\Qp)$, on obtient un foncteur $\pi\mapsto F_\alpha(\pi)$ exact \`a gauche qui annule les Steinberg g\'en\'eralis\'ees (lisses). Pour $\pi$ localement analytique, on d\'efinit alors $F_\alpha(\pi)$ comme suit (cf. (\ref{falpha}))~:
\begin{equation}\label{falphaintro}
D\longmapsto {\displaystyle \lim_{r,m\rightarrow +\infty}}\Hom_{\psi,\Gamma}\big((\pi[\mnn^\alpha]\otimes_EE_m)(\eta^{-1})_{N^\alpha_m})^\vee,D_r\otimes_EE_m\big)
\end{equation}
o\`u $D$ est un $(\varphi,\Gamma)$-module g\'en\'eralis\'e sur $\R$ et $\mnn^\alpha$ est l'alg\`ebre de Lie de $N^\alpha(\Qp)$. Comme de plus $\pi$ est d\'efinie sur $E$, les $E_\infty$-espaces vectoriels (\ref{falphaintro}) sont naturellement munis d'une action $E_\infty$-semi-lin\'eaire de $\Gal(E_\infty/E)$. Noter que (\ref{falphaintro}) n'utilise que l'action du Borel $B(\Qp)$ des matrices triangulaires sup\'erieures dans ${\rm GL}_n(\Qp)$ et ne d\'epend pas, \`a isomorphisme pr\`es, du choix des $N_m$ (Proposition \ref{choix}). Noter aussi que $(\pi[\mnn^\alpha]\otimes_EE_m)(\eta^{-1})_{N^\alpha_m}\simeq (\pi\otimes_EE_m)(\eta^{-1})^{N^\alpha_m}$ (cf. Remarque \ref{fixe}), mais cette deuxi\`eme d\'efinition donne des fl\`eches de fonctorialit\'e dans le mauvais sens quand $m$ grandit. Noter enfin que si l'on consid\`ere, par analogie, le foncteur $\displaystyle \sigma \mapsto \lim_{m\rightarrow +\infty}\Hom_{D(T^+(\Qp),E)}((\pi[\mnn]_{N_m})^\vee,\sigma^\vee)$
o\`u $\sigma$ est une repr\'esentation localement analytique du tore $T(\Qp)$ dans la cat\'egorie ${\rm Rep}^z_{\rm la.c}(T(\Qp))$ de \cite{Em2}, $\mnn$ l'alg\`ebre de Lie de $N(\Qp)$ et $D(T(\Qp)^+,E)$ les distributions localement analytiques sur le tore ``positif'' $T(\Qp)^+$ (\cite[Def.~2.2.1]{Em2}), on peut montrer que ce foncteur est repr\'esentable par le dual $J_B(\pi)^\vee$ o\`u $J_B(-)$ est le foncteur de Jacquet-Emerton relativement \`a $B(\Qp)$ (\cite{Em1}, \cite{Em2}), de sorte que (\ref{falphaintro}) semble raisonnable.

\begin{thm0}[Proposition \ref{drex}~\&~Th\'eor\`eme \ref{encoreplusgeneral}]\label{resintro}
(i) Pour toute suite exacte $0\rightarrow \pi''\rightarrow \pi\rightarrow \pi'$ de repr\'esentations localement analytiques de $B(\Qp)$ sur des $E$-espaces vectoriels de type compact, on a une suite exacte~:
$$0\longrightarrow F_{\alpha}(\pi'')\longrightarrow F_{\alpha}(\pi)\longrightarrow F_{\alpha}(\pi')$$
(i.e. $0\longrightarrow F_{\alpha}(\pi'')(D)\longrightarrow F_{\alpha}(\pi)(D)\longrightarrow F_{\alpha}(\pi')(D)$ est exact pour tout $(\varphi,\Gamma)$-module g\'en\'eralis\'e $D$).\\
(ii) Soit $P(\Qp)\subseteq {\rm GL}_n(\Qp)$ un sous-groupe parabolique contenant $B(\Qp)$ tel que $\alpha$ est une racine simple du facteur de Levi $L_P(\Qp)$, $P^-(\Qp)$ le parabolique oppos\'e et $\pi_P$ une repr\'esentation localement analytique de $L_P(\Qp)$ sur un $E$-espace vectoriel de type compact. Alors on a un isomorphisme $F_\alpha\big(\big(\Ind_{P^-(\Qp)}^{{\rm GL}_n(\Qp)} \pi_P\big)^{\an}\big)\cong F_\alpha(\pi_P)$.
\end{thm0}

Le probl\`eme de la repr\'esentabilit\'e du foncteur $F_\alpha(\pi)$ pour, disons, $\pi$ une repr\'esentation de longueur finie de ${\rm GL}_n(\Qp)$ est beaucoup plus d\'elicat. Lorsque l'on se limite, dans un premier temps, aux repr\'esentations construites par Orlik et Strauch dans \cite{OS}, on peut montrer que $F_\alpha(\pi)$ est repr\'esentable dans beaucoup de cas. Rappelons que, si $P(\Qp)\subseteq {\rm GL}_n(\Qp)$ est un sous-groupe parabolique contenant $B(\Qp)$, $P^-(\Qp)$ le parabolique oppos\'e d'alg\`ebre de Lie $\mpp^-$ et ${\mathcal O}^{\mpp^-}_{\rm alg}$ la sous-cat\'egorie de la cat\'egorie ${\mathcal O}^{\mpp^-}$ (\cite{Hu}) des objets avec des poids entiers, alors pour $\pi_P^\infty$ une repr\'esentation lisse de longueur finie du Levi $L_P(\Qp)$ sur $E$, Orlik et Strauch cons\-truisent dans \cite{OS} un foncteur contravariant exact $M\longmapsto {\mathcal F}_{P^-}^G(M,\pi_P^\infty)$ de ${\mathcal O}^{\mpp^-}_{\rm alg}$ dans la cat\'egorie des repr\'esentations localement analytiques (admissibles) de longueur finie de ${\rm GL}_n(\Qp)$ sur $E$. Soit $i\in \{1,\dots,n-1\}$ tel que $\alpha=e_i-e_{i+1}$, on note dans la suite $\lambda_{\alpha^\vee}(x)\=\diag(\underbrace{x,\dots,x}_{i},\underbrace{1,\dots,1}_{n-i})$ (un cocaract\`ere alg\'ebrique de $T(\Qp)$). Si $\lambda$ est un caract\`ere alg\'ebrique de $T(\Qp)$, on peut consid\'erer le $(\varphi,\Gamma)$-module $\R(\lambda\circ \lambda_{\alpha^\vee})$.

\begin{thm0}[Corollaire \ref{generaljoli}~\&~Corollaire \ref{casparticuliers}]\label{resbisintro}
Soit $P(\Qp)$, $M\in {\mathcal O}^{\mpp^-}_{\rm alg}$, $\pi_P^\infty$ comme ci-dessus et $\pi\={\mathcal F}_{P^-}^G(M,\pi_P^\infty)$. On suppose que $\pi_P^\infty$ admet un caract\`ere central $\chi_{\pi_P^\infty}$ et on note $d_{\pi_P^\infty}=\dim_E(\pi_P^\infty\otimes_EE_\infty)(\eta^{-1})_{N_{L_{P}}\!(\Qp)}$ o\`u $N_{L_{P}}\!(\Qp)=N(\Qp)\cap L_P(\Qp)$. Lorsque $\alpha$ n'est pas une racine de $L_P(\Qp)$, on note $Q(\Qp)$ le plus petit sous-groupe parabolique de ${\rm GL}_n(\Qp)$ contenant $P(\Qp)$ tel que $\alpha$ est une racine de $L_Q(\Qp)$ et $\mq^-$ l'alg\`ebre de Lie de $Q^-(\Qp)$.\\
(i) Il existe des caract\`eres alg\'ebriques distincts $\chi_{\lambda_1},\dots,\chi_{\lambda_r}$ de $\Gal(E_\infty/E)$ (cf. (\ref{chilambda}) pour $\chi_{\lambda_i}$) et des $(\varphi,\Gamma)$-modules g\'en\'eralis\'es $D_{\alpha,1}(\pi),\dots ,D_{\alpha,r}(\pi)$ tels que~:
$$F_{\alpha}(\pi)\simeq \bigoplus_{i=1}^rE_\infty(\chi_{\lambda_i})\otimes_E\Hom_{(\varphi,\Gamma)}(D_{\alpha,i}(\pi),-).$$
(ii) Supposons de plus que $M$ est un quotient non nul de $U(\mg)\otimes_{U(\mpp^-)}L^-(\lambda)_{P}$ o\`u $L^-(\lambda)_{P}$ est une repr\'esentation alg\'ebrique irr\'eductible de $L_P(\Qp)$ (pour $\lambda$ un caract\`ere alg\'ebrique de $T(\Qp)$ dominant pour $B^-(\Qp)\cap L_P(\Qp)$). Alors si $\alpha$ n'est pas une racine de $L_P(\Qp)$ et si $M\notin {\mathcal O}^{\mq^-}_{\rm alg}$, on a~:
\begin{equation*}
F_\alpha(\pi)\simeq
E_\infty(\chi_{-\lambda})\otimes_{E}\Hom_{(\varphi,\Gamma)}\Big(\R\big((\lambda\circ \lambda_{\alpha^\vee})(\chi_{\pi_P^\infty}^{-1}\circ \lambda_{\alpha^{\!\vee}})\big)^{\oplus d_{\pi_P^\infty}},-\Big).
\end{equation*}
(iii) Avec les notations de (ii), si $\alpha$ n'est pas une racine de $L_P(\Qp)$ et si $M\in {\mathcal O}^{\mq^-}_{\rm alg}$, ou bien si $\alpha$ est une racine de $L_P(\Qp)$, on a~:
$$F_\alpha(\pi)\simeq E_\infty(\chi_{-\lambda})\otimes_E\Hom_{(\varphi,\Gamma)}\Big(\big(\R(\lambda\circ \lambda_{\alpha^{\!\vee}})/(t^{1-\langle \lambda,\alpha^\vee\rangle})\big)^{\oplus d_{\pi_P^\infty}},-\Big).$$
\end{thm0}

Le Th\'eor\`eme \ref{resintro} et le Th\'eor\`eme \ref{resbisintro} ne sont pas vrais seulement pour ${\rm GL}_n(\Qp)$, on les d\'emontre dans le texte pour $G(\Qp)$ o\`u $G$ est un groupe alg\'ebrique r\'eductif connexe d\'eploy\'e sur $\Qp$ de centre connexe. La preuve utilise comme ingr\'edients le Th\'eor\`eme \ref{resintro}, le cas $\pi$ localement alg\'ebrique (Th\'eor\`eme \ref{caslisse}, dont la preuve utilise de mani\`ere essentielle des r\'esultats de Bernstein et Zelevinsky \cite[\S~3.5]{BZ}), le Th\'eor\`eme \ref{liesocle} qui permet de se ramener \`a une induite parabolique localement analytique, et enfin des d\'evissages parfois techniques (cf. Proposition \ref{descgrdense}) pour montrer que seule la grosse cellule d'une telle induite parabolique a une contribution non nulle \`a $F_\alpha$ (cf. Proposition \ref{celluleouverte} et Proposition \ref{support}). Dans ces d\'evissages, on utilise en particulier \`a maintes reprises le fait technique suivant~: pour certains $\psi$-modules de Fr\'echet dont le $E$-espace vectoriel sous-jacent est de la forme $\prod_{i\geq 1}M_i$ pour des espaces de Fr\'echet $M_i$ (par exemples ceux tels que $\psi$ pr\'eserve chaque facteur $M_i$), tout morphisme $f: \prod_{i\geq 1}M_i\rightarrow D_r$ d'espaces de Fr\'echet commutant \`a $\psi$ est nul sur tous les $M_i$ sauf un nombre fini (voir par exemple la preuve du Lemme \ref{begin}). Noter que le foncteur $\pi\mapsto F_\alpha(\pi)$ n'est pas exact en g\'en\'eral, cf. par exemple le Th\'eor\`eme \ref{liesocle}.

Revenons maintenant au contexte du d\'ebut de l'introduction avec $\rho_p$, $(h_i)_{i\in \{1,\dots,n\}}$, $W$, ${\rm alg}\otimes_E \pi^\infty$, $\pi^{\an}$. Comme cas particulier du (iii) du Th\'eor\`eme \ref{resbisintro} (le cas $P(\Qp)={\rm GL}_n(\Qp)$), on a bien $F_\alpha({\rm alg}\otimes_E \pi^\infty)\simeq E_\infty(\chi_{-\lambda})\otimes_E\Hom_{(\varphi,\Gamma)}(\R(\chi)/(t^{h_j-h_{j+1}}),-)$ o\`u $\lambda\=(n-i-h_i)_{i\in \{1,\dots,n\}}$ et $\chi$ est un certain caract\`ere localement alg\'ebrique (plus pr\'ecis\'ement $x^{h_j-h_{j+1}}\chi$ est le caract\`ere apparaissant en (\ref{dalphaw})). Le foncteur $F_\alpha$ permet maintenant d'\'enoncer la conjecture principale de cet article (on renvoie au \S~\ref{locglob} pour plus de d\'etails).

\begin{conj0}[Conjecture \ref{extglobmieux}]\label{conjintro}
Pour toute racine simple $\alpha=e_j-e_{j+1}$ il existe une repr\'esentation localement analytique admissible de longueur finie $\pi^\alpha$ de ${\rm GL}_n(\Qp)$ sur $E$ ne d\'ependant que des poids de Hodge-Tate $(h_i)_{i\in \{1,\dots,n\}}$, de la repr\'esentation de Weil-Deligne $W$ et de $\alpha$ et v\'erifiant les propri\'et\'es suivantes~:

\noindent
(i) $F_\alpha(\pi^\alpha)\simeq E_\infty(\chi_{-\lambda})\!\otimes_{E}\!\Hom_{(\varphi,\Gamma)}\!\big(D(\wedge_E^{n-j}W)(x^{h_j-h_{j+1}})\otimes_{\R}\R(\chi),-\big)$;

\noindent
(ii) pour toute extension $\pi$ de $\pi^\alpha$ par ${\rm alg}\otimes_E \pi^\infty$ on a~:
$$F_\alpha(\pi)\simeq E_\infty(\chi_{-\lambda})\!\otimes_{E}\!\Hom_{(\varphi,\Gamma)}\!\big(D_\alpha(\pi),-\big)$$
pour un (unique) $(\varphi,\Gamma)$-module g\'en\'eralis\'e $D_\alpha(\pi)$ et le foncteur $\pi\mapsto D_\alpha(\pi)$ induit un isomorphisme~:
\begin{multline}\label{iso1bisintro}
{\rm Ext}^1_{{\rm GL}_n(\Qp)}\big(\pi^\alpha,{\rm alg}\otimes_E \pi^\infty\big)\\\buildrel\sim\over\longrightarrow {\rm Ext}^1_{(\varphi,\Gamma)}\big(\R(\chi)/(t^{h_j-h_{j+1}}),D(\wedge_E^{n-j}W)(x^{h_j-h_{j+1}})\otimes_{\R}\R(\chi)\big);
\end{multline}
(iii) tout plongement ${\rm alg}\otimes_E \pi^\infty\hookrightarrow \pi^{\an}$ s'\'etend en un plongement~:
$$\!\begin{xy}(-60,0)*+{({\rm alg}\otimes_E \pi^\infty)}="a";(-40,0)*+{\pi^\alpha}="b";{\ar@{-}"a";"b"}\end{xy}\!\hookrightarrow \pi^{\an}$$
o\`u $\!\begin{xy}(-60,0)*+{({\rm alg}\otimes_E \pi^\infty)}="a";(-40,0)*+{\pi^\alpha}="b";{\ar@{-}"a";"b"}\end{xy}\!$ est l'unique extension non scind\'ee de $\pi^\alpha$ par ${\rm alg}\otimes_E \pi^\infty$ dont la droite engendr\'ee dans ${\rm Ext}^1_{{\rm GL}_n(\Qp)}(\pi^\alpha,{\rm alg}\otimes_E \pi^\infty)$ s'envoie vers la droite ${\rm Fil}^{\rm max}_\alpha(\rho_p)$ via les isomorphismes (\ref{iso1bisintro}) puis (\ref{filintro}) (tordu par $\chi$).
\end{conj0}

On \'enonce maintenant nos r\'esultats partiels sur la Conjecture \ref{conjintro}. Le premier th\'eor\`eme est valable pour tout $n\geq 2$.

\begin{thm0}[Th\'eor\`eme \ref{localgenplus}, Corollaire \ref{cor05}~\&~Th\'eor\`eme \ref{ncris}]\label{resterintro}
\noindent
(i) Si $\pi^\alpha$ existe et v\'erifie le (i) de la Conjecture \ref{conjintro}, alors pour toute extension $\pi$ de $\pi^\alpha$ par ${\rm alg}\otimes_E \pi^\infty$ on a~:
$$F_\alpha(\pi)\simeq E_\infty(\chi_{-\lambda})\!\otimes_{E}\!\Hom_{(\varphi,\Gamma)}\!\big(D_\alpha(\pi),-\big)$$
et le foncteur $\pi\mapsto D_\alpha(\pi)$ induit un morphisme~:
\begin{multline}\label{iso1bisintrores}
{\rm Ext}^1_{{\rm GL}_n(\Qp)}\big(\pi^\alpha,{\rm alg}\otimes_E \pi^\infty\big)\\
\longrightarrow {\rm Ext}^1_{(\varphi,\Gamma)}\big(\R(\chi)/(t^{h_j-h_{j+1}}),D(\wedge_E^{n-j}W)(x^{h_j-h_{j+1}})\otimes_{\R}\R(\chi)\big).
\end{multline}
(ii) Supposons $\rho_p$ cristalline avec les ratios des valeurs propres du Frobenius sur $D_{\rm cris}(\rho_p)$ distincts de $1$, $p$. Alors il existe $\pi^\alpha$ (admissible de longueur finie) v\'erifiant les (i) et (ii) de la Conjecture \ref{conjintro}. De plus, sous les hypoth\`ese standard de Taylor-Wiles, il existe un plongement~:
$$\!\begin{xy}(-60,0)*+{({\rm alg}\otimes_E \pi^\infty)}="a";(-40,0)*+{\pi^\alpha}="b";{\ar@{-}"a";"b"}\end{xy}\!\hookrightarrow \pi^{\an}$$
avec $\!\begin{xy}(-60,0)*+{({\rm alg}\otimes_E \pi^\infty)}="a";(-40,0)*+{\pi^\alpha}="b";{\ar@{-}"a";"b"}\end{xy}\!$ comme dans le (iii) de la Conjecture \ref{conjintro}.
\end{thm0}

La premi\`ere assertion du (ii) du Th\'eor\`eme \ref{resterintro} requiert en particulier la preuve de \cite[Conj.~3.3.1]{Br1} pour $\GL_n(\Qp)$, ce qui est l'objet du \S~\ref{conjcris}. L'un des ingr\'edients cru\-ciaux de la deuxi\`eme assertion du (ii) du Th\'eor\`eme \ref{resterintro} est l'existence des constituants compagnons dans le cas cristallin (\cite[Th.~1.3]{BHS3} qui requiert les ``hypoth\`eses standard de Taylor-Wiles''). 

Le deuxi\`eme th\'eor\`eme consid\`ere le cas $n=2$. Pour $n=2$, la repr\'esentation $\pi^{\an}$ est bien d\'efinie par la correspondance de Langlands locale pour ${\rm GL}_2(\Qp)$ (et compatible avec la th\'eorie globale, au moins dans beaucoup de cas, cf. \cite{Em3}, \cite{CS1}, \cite{CS2}) et de plus $\pi^{\an}/({\rm alg}\otimes_E \pi^\infty)$ ne d\'epend que des $h_i$ et de $W$ par \cite[Th.~VI.6.43]{Co2}.

\begin{thm0}[Th\'eor\`eme \ref{n=2}]\label{n=2intro}
Supposons $n=2$.\\
(i) La Conjecture \ref{conjintro} est vraie lorsque $W$ est r\'eductible.\\
(ii) Si $W$ est irr\'eductible et si le (i) de la Conjecture \ref{conjintro} est vrai avec $\pi^\alpha\=\pi^{\an}/({\rm alg}\otimes_E \pi^\infty)$, alors on a $F_\alpha(\pi^{\an})\simeq E_\infty\otimes_E\Hom_{(\varphi,\Gamma)}(D_{\rm rig}(\check \rho_p),-)$ o\`u $D_{\rm rig}(\check \rho_p)$ est le $(\varphi,\Gamma)$-module \'etale sur $\R$ associ\'e au dual de Cartier $\check \rho_p$ de $\rho_p$.
\end{thm0}

La d\'emonstration du Th\'eor\`eme \ref{n=2intro} utilise comme ingr\'edients principaux plusieurs r\'esultats de Colmez et Dospinescu (\cite{Co}, \cite{CD}) combin\'es avec le Th\'eor\`eme \ref{resbisintro}, le (i) du Th\'eor\`eme \ref{resterintro} et un r\'esultat de repr\'esentabilit\'e et d'exactitude pour le foncteur $F_\alpha$ appliqu\'e \`a certaines extensions de repr\'esentations localement analytiques de ${\rm GL}_2(\Qp)$ (cf. Th\'eor\`eme \ref{gl3enplus}). Pour d\'emontrer compl\`etement la Conjecture \ref{conjintro} lorsque $n=2$ et $W$ est irr\'eductible, il semble qu'il faille aller plus loin que les r\'esultats \'enonc\'es dans \cite{Co}, \cite{Co2}, \cite{CD}.

Le troisi\`eme th\'eor\`eme, le plus d\'elicat, concerne le cas o\`u $n=3$ et $\pi^\infty$ est la repr\'esentation de Steinberg de ${\rm GL}_3(\Qp)$ \`a torsion pr\`es (en particulier $\rho_p$ est semi-stable non cristalline \`a torsion pr\`es). Ce cas est particuli\`erement int\'eressant car, d'une part la pr\'esence de l'op\'erateur de monodromie ``rigidifie'' la situation et implique que les deux droites ${\rm Fil}^{\rm max}_{e_1-e_2}(\rho_p)$ et ${\rm Fil}^{\rm max}_{e_2-e_3}(\rho_p)$ d\'eterminent la filtration de Hodge sur $D_{\dR}(\rho_p)$ (ce qui n'est pas vrai dans le cas cristallin par exemple), d'autre part on dispose de candidats explicites pour les repr\'esentations $\pi^\alpha$ de ${\rm GL}_3(\Qp)$, cf. \cite[\S~4.5]{Br1}.

\begin{thm0}[Th\'eor\`eme \ref{gl3loc}]\label{gl3locintro}
Supposons $n=3$, $\pi^\infty=$ la repr\'esentation de Steinberg (\`a torsion pr\`es) et soit $\pi^\alpha$ comme dans \cite[\S~4.5]{Br1} (aux notations pr\`es).\\
(i) On a $F_\alpha(\pi^\alpha)\simeq E_\infty(\chi_{-\lambda})\!\otimes_{E}\!\Hom_{(\varphi,\Gamma)}\!\big(D(\wedge_E^{3-j}W)(x^{h_j-h_{j+1}})\otimes_{\R}\R(\chi),-\big)$.\\
(ii) Le morphisme (\ref{iso1bisintrores}) (qui existe par le (i) ci-dessus et le (i) du Th\'eor\`eme \ref{resterintro}) est un isomorphisme~:
\begin{multline*}
{\rm Ext}^1_{{\rm GL}_n(\Qp)}\big(\pi^\alpha,{\rm alg}\otimes_E \pi^\infty\big)\\
\buildrel\sim\over\longrightarrow {\rm Ext}^1_{(\varphi,\Gamma)}\big(\R(\chi)/(t^{h_j-h_{j+1}}),D(\wedge_E^{3-j}W)(x^{h_j-h_{j+1}})\otimes_{\R}\R(\chi)\big).
\end{multline*}
\end{thm0}

La d\'emonstration du Th\'eor\`eme \ref{gl3locintro} utilise essentiellement tous les th\'eor\`emes pr\'ec\'edents, les r\'esultats de \cite{Br1} et le Th\'eor\`eme \ref{gl3enplus} dans le texte qui donne un \'enonc\'e crucial (mais d\'elicat) de repr\'esentabilit\'e et d'exactitude pour le foncteur $F_\alpha$ appliqu\'e \`a certaines extensions de repr\'esentations localement analytiques de ${\rm GL}_3(\Qp)$. La preuve que l'on donne de ce Th\'eor\`eme \ref{gl3enplus} est longue et tr\`es technique, elle utilise (cf. \S~\ref{devis1}) des r\'esultats dus \`a Schraen (\cite{Sc1}) et pas mal d'analyse fonctionnelle $p$-adique assez fastidieuse sur des $(\psi,\Gamma)$-modules de Fr\'echet o\`u l'op\'erateur $\psi$ et le ``fait technique'' de la preuve du Lemme \ref{begin} mentionn\'e pr\'ec\'edemment jouent un r\^ole important (cf. \S\S~\ref{technique1}~\&~\ref{premierscindage}). Noter que, pour d\'emontrer ce r\'esultat d'exactitude, on utilise aussi l'action de $\Gal(E_\infty/E)$ (cf. preuve de la Proposition \ref{pushtechnique}).

Ind\'ependamment du Th\'eor\`eme \ref{gl3locintro}, dans \cite{BD} sous des hypoth\`eses de g\'en\'ericit\'e faibles est associ\'ee \`a $\rho_p$ pour chaque $\alpha$ une extension non scind\'ee $\pi^\alpha(\rho_p)$ de $\pi^\alpha$ par ${\rm alg}\otimes_E \pi^\infty$ (qui d\'epend {\it a priori} de $\rho_p$ ``tout entier'') telle que, au moins lorsque les $h_i$ sont des entiers cons\'ecutifs, tout plongement ${\rm alg}\otimes_E \pi^\infty\hookrightarrow \pi^{\an}$ s'\'etend en un plongement $\pi^\alpha(\rho_p)\hookrightarrow \pi^{\an}$. Le (ii) du Th\'eor\`eme \ref{gl3locintro} permet par ailleurs d'associer \`a $(h_i)_{i\in \{1,2,3\}}$, $W$ et ${\rm Fil}^{\rm max}_{\alpha}(\rho_p)$ (via la Proposition \ref{filmaxintro}) une autre extension non scind\'ee de $\pi^\alpha$ par ${\rm alg}\otimes_E \pi^\infty$ que l'on note $\pi^\alpha({\rm Fil}^{\rm max}_{\alpha}(\rho_p))$. Le (iii) de la Conjecture \ref{conjintro} et \cite[Th.~1.1]{BD} impliquent que, au moins lorsque les $h_i$ sont cons\'ecutifs, ces deux extensions non scind\'ees devraient \^etre les m\^emes. La conjecture suivante (pour des $h_i$ quelconques) est donc naturelle.

\begin{conj0}[Conjecture \ref{deuxpialpha}]\label{frustrantintro}
Avec les notations pr\'ec\'edentes, on a $\pi^\alpha(\rho_p)\simeq \pi^\alpha({\rm Fil}^{\rm max}_{\alpha}(\rho_p))$.
\end{conj0}

Le dernier th\'eor\`eme de cette introduction montre que la Conjecture \ref{frustrantintro} est tr\`es vraisemblable.

\begin{thm0}[Th\'eor\`eme \ref{gl3glob}]\label{dependintro}
La repr\'esentation $\pi^\alpha(\rho_p)$ ne d\'epend que de $(h_i)_{i\in \{1,2,3\}}$, $W$ et ${\rm Fil}^{\rm max}_{\alpha}(\rho_p)$.
\end{thm0}

La preuve du Th\'eor\`eme \ref{dependintro} repose sur les r\'esultats de \cite{BD} et est ind\'ependante des th\'eor\`emes pr\'ec\'edents, mais a \'et\'e fortement inspir\'ee par la Conjecture \ref{frustrantintro}. Il s'agit d'un raffinement de la construction de $\pi^{\alpha}(\rho_p)$ dans \cite{BD} qui consiste \`a utiliser la correspondance localement analytique pour ${\rm GL}_2(\Qp)$ pour construire, par induction parabolique localement analytique et des calculs de cohomologie galoisienne, un accouplement parfait entre deux $\Ext^1$ de dimension $2$, l'un c\^ot\'e $(\varphi,\Gamma)$-modules (sans torsion), l'autre c\^ot\'e repr\'esentations de ${\rm GL}_3(\Qp)$, puis \`a d\'efinir $\pi^\alpha(\rho_p)$ (ou plut\^ot une sous-repr\'esentation $\pi^\alpha(\rho_p)^-$ qui d\'etermine $\pi^\alpha(\rho_p)$) comme l'orthogonal de la droite engendr\'ee par le $(\varphi,\Gamma)$-module associ\'e \`a ${\rm Fil}^{\rm max}_\alpha(\rho_p)$ dans la Proposition \ref{filmaxintro} (\`a torsion et dualit\'e pr\`es) vu comme \'el\'ement du $\Ext^1$ c\^ot\'e $(\varphi,\Gamma)$-modules. Ce raffinement repose sur la Proposition \ref{pLLht} (dont la preuve utilise des r\'esultats de Dospinescu \cite{Do}) et la Proposition \ref{split} (dont la preuve a \'et\'e rel\'egu\'ee en appendice car elle n'utilise que des techniques de \cite{BD} et des calculs d'alg\`ebre de Lie). L'id\'ee nouvelle par rapport \`a \cite{BD} dans ces deux propositions est de consid\'erer des repr\'esentations localement analytiques de ${\rm GL}_2(\Qp)$ admettant un caract\`ere infinit\'esimal.

Nous aurions aim\'e pouvoir montrer compl\`etement la Conjecture \ref{frustrantintro}, mais cela semble difficile sans avoir \`a d\'emontrer encore d'autres propri\'et\'es (redoutables) de repr\'esentablilit\'e et d'exactitude de $F_\alpha$ ni sans avoir \`a subir de nouveaux calculs dans de grosses induites paraboliques localement analytiques de ${\rm GL}_3(\Qp)$. On atteint pro\-bablement l\`a les limites des m\'ethodes relativement ``explicites'' de cet article (et de \cite{Br1}, \cite{BD}). Pour aller plus loin, par exemple pour traiter ${\rm GL}_n(\Qp)$ et $\rho_p$ g\'en\'erique semi-stable non cristalline de dimension $n\geq 4$, il faut probablement trouver de nouvelles m\'ethodes.

Nous nous sommes limit\'es au corps de base $\Qp$ dans cet article pour le confort de n'avoir \`a consid\'erer que des $(\varphi,\Gamma)$-modules sur $\R$. Plusieurs r\'esultats, par exemple la d\'efinition de $F_\alpha$, le Th\'eor\`eme \ref{resintro} et le Th\'eor\`eme \ref{resbisintro}, devraient s'\'etendre au cas de repr\'esentations localement $\sigma$-analytiques de $G(L)$ pour $L$ extension finie de $\Qp$, $\sigma:L\hookrightarrow E$ et $G$ d\'eploy\'e sur $L$ de centre connexe, en rempla\c cant les $(\varphi,\Gamma)$-modules sur $\R$ par la th\'eorie de \cite{BSX}.

Terminons cette introduction avec quelques notations et conventions g\'en\'erales.

Dans tout l'article $E$ d\'esigne un corps de coefficients pour les repr\'esentations, une extension finie de $\Qp$, et $(E_m)_{m\in \Z_{\geq 0}}$ une tour d'extensions galoisiennes de $E$ telle que $E_m=E[\sqrt[p^m]{1}]$ (on peut avoir $E=E_m$ pour $m$ petit). On d\'esigne par $\val$ la valuation $p$-adique normalis\'ee par $\val(p)=1$ et $\norm=1/p^{\val(\cdot)}$ la norme $p$-adique associ\'ee. On note $\Gamma=\Gal(\Qp(\!\sqrt[p^{\infty}]{1})/\Qp)$ et $\varepsilon:\gp\twoheadrightarrow \Gamma \buildrel\sim\over\rightarrow \Zp^\times$ le caract\`ere cyclotomique $p$-adique, qui permet d'identifier $\Gamma$ \`a $\Zp^\times$. On normalise la th\'eorie du corps de classe local de telle sorte que les Frobenius g\'eom\'etriques correspondent aux uniformisantes. Ainsi $\varepsilon(x)=x\vert x\vert$ pour tout $x\in \Qp^\times$. On convient que le poids de Hodge-Tate de $\varepsilon$ est $1$.

Si $G$ est un groupe alg\'ebrique on note $Z_G$ son centre. Si $\mg$ est une alg\`ebre de Lie (sur un corps), on note $U(\mg)$ son alg\`ebre enveloppante. Si $V$ est un $U(\mg)$-module et $X\subseteq \mg$ un sous-ensemble quelconque, on note $V[X]=\{v\in V,\ xv=0\ \forall\ x\in X\}$. Si $M$ est un module sur un anneau commutatif $A$ et $S\subseteq A$ un sous-ensemble quelconque, on note $M_{\tor}$ le sous-$A$-module de torsion de $M$ et $M[S]$ le sous-$A$-module $\{m\in M,\ sm=0\ \forall \ s\in S\}$.

Si $V$ est une repr\'esentation d'un groupe $G$ sur un $E$-espace vectoriel et $\eta:G\rightarrow E$ un ca\-ract\`ere, on note $V(\eta)$ la repr\'esentation de $V$ tordue par $\eta$. Si $V$ est un $\Qp$-espace vectoriel muni d'une topologie localement convexe, on note $V^{\sep}$ son quotient s\'epar\'e, c'est-\`a-dire $V/\overline{\{0\}}$ (o\`u $\overline{\{0\}}$ est l'adh\'erence de $\{0\}$ dans $V$) muni de la topologie quotient, et $V^\vee$ son dual continu muni de la topologie forte (\cite[p.~30]{Sch}). Si $V$ et $W$ sont deux $E$-espaces vectoriels localement convexes, on note $V\otimes_{E,\pi}W$, resp. $V\otimes_{E,\iota}W$, le produit tensoriel muni de la topologie projective, resp. injective, et simplement $V\otimes_{E}W$ lorsque ces deux topologies co\"\i ncident. On note $\widehat\otimes_{E,\pi}$, $\widehat\otimes_{E,\iota}$, $\widehat\otimes_{E}$ au lieu de $\otimes_{E,\pi}$, $\otimes_{E,\iota}$, $\otimes_{E}$ les compl\'et\'es respectifs (\cite[Prop.~7.5]{Sch}).

Si $M$ est une vari\'et\'e localement $\Qp$-analytique paracompacte (ou de mani\`ere \'equivalente strictement paracompacte) de dimension finie (cf. \cite[\S~II]{Sch2}), par exemple un groupe analytique $p$-adique, et $V$ un $E$-espace vectoriel muni d'une topologie localement convexe s\'epar\'ee, on note $C^{\an}(M,V)$ le $E$-espace vectoriel des fonctions localement analytiques de $M$ dans $V$ (\cite[\S~2]{ST1}), $D(M,E)$ l'alg\`ebre des distributions localement analytiques sur $M$ \`a valeurs dans $E$, c'est-\`a-dire le dual continu de $C^{\an}(M,E)$ avec sa topologie forte (cf. {\it loc.cit.}). Si $M$ est compacte et $V$ est un espace de type compact, alors $C^{\an}(M,V)$ est aussi un espace de type compact (\cite[Prop.~2.1.28]{Em}). On note $C^\infty(M,V)$ le sous-$E$-espace vectoriel des fonctions localement constantes.

Si $G$ est un groupe analytique $p$-adique, on note $\Rep(G)$ la cat\'egorie des repr\'esentations localement analytiques de $G$ sur des $E$-espaces vectoriels localement convexes de type compact (\cite[\S~3]{ST1}). La cat\'egorie $\Rep(G)$ n'est pas ab\'elienne, mais elle est munie d'une notion de suite exacte (rappelons qu'une suite exacte courte d'espaces de type compact est n\'ecessairement stricte, de m\^eme qu'une suite exacte courte d'espaces de Fr\'echet). Pour les objets de $\Rep(G)$, on abr\`ege ``absolument topologiquement irr\'eductible'' (i.e. qui reste topologiquement irr\'eductible apr\`es extension des scalaires \`a une extension finie arbitraire de $E$) en ``irr\'eductible''. On rappelle que l'alg\`ebre $D(G,E)$ contient l'alg\`ebre enveloppante $U(\mg)$ de la $\Qp$-alg\`ebre de Lie $\mg$ de $G$. Si $\pi$, $\pi'$ sont deux repr\'esentations localement analytiques admissibles de $G$ (\cite[\S~6]{ST2}), on note ${\rm Ext}^1_G(\pi,\pi')$ le $E$-espace vectoriel des extensions dans la cat\'egorie des repr\'esentations localement analytiques admissibles de $G$.

\section{Foncteurs $F_\alpha$ et $(\varphi,\Gamma)$-modules sur l'anneau de Robba}

On d\'efinit les foncteurs $F_\alpha$ et on montre leurs propri\'et\'es g\'en\'erales.

\subsection{Quelques pr\'eliminaires}\label{prel}

On introduit plusieurs notations et on d\'emontre quelques r\'esultats pr\'eliminaires, dont la Proposition \ref{exactdense}, sur des repr\'esentations de $B(\Qp)$.

On fixe un triplet $(G,B,T)$ o\`u $G$ est un groupe alg\'ebrique r\'eductif connexe d\'eploy\'e sur $\Qp$, $B\subset G$ un sous-groupe de Borel d\'efini sur $\Qp$ et $T\subset B$ un tore ma\-ximal d\'eploy\'e sur $\Qp$. On suppose $G\ne T$ (donc aussi $B\ne T$) et on note $W=N_G(T)/T\ne \{1\}$ le groupe de Weyl de $G$ et $N$ le radical unipotent de $B$. On note $X(T)\=\Hom_{\rm gr}(T,{\mathbb G}_{\rm m})$ le groupe des caract\`eres alg\'ebriques de $T$, $X^\vee(T)\=\Hom_{\rm gr}({\mathbb G}_{\rm m},T)\cong \Hom_{\Z}(X(T),\Z)$ le groupe des coca\-ract\`eres, $\big(X(T),R,X^{\vee}(T),R^{\vee}\big)$ la donn\'ee radicielle de $(G,T)$, $R^+\subset X(T)$ les racines positives relativement \`a $B$ et $S\subset R^+$ les racines simples. On note $s_\alpha\in W$ la r\'eflexion associ\'ee \`a la racine $\alpha\in R$ et $\rho\=\frac{1}{2}\sum_{\alpha\in R^+}\alpha\in \frac{1}{2}X(T)$. Pour $\alpha\in R^+$, on note $N_\alpha\subseteq N$ le sous-groupe radiciel (commutatif) associ\'e et on fixe un isomorphisme $\iota_\alpha:N_\alpha \buildrel\sim\over \rightarrow {\mathbb G}_{\rm a}$ de groupes alg\'ebriques sur $\Qp$ tel que (\cite[\S~II.1.2]{Ja}) :
\begin{equation}\label{alpha}
\iota_\alpha(tx_\alpha t^{-1})=\alpha(t)\iota_\alpha(x_\alpha)\ \ \forall\ t\in T,\ \ \forall\ x_\alpha\in N_\alpha.
\end{equation}
L'application produit donne un isomorphisme de vari\'et\'es alg\'ebriques sur $\Qp$ (pour un ordre quelconque sur les $\alpha\in R^+$) $\prod_{\alpha\in R^+}N_\alpha \buildrel\sim\over\rightarrow N$. On note (cf. \cite[\S~5]{SV}) $\ell:~N\twoheadrightarrow \prod_{\alpha\in S}N_\alpha\buildrel\sum_{\alpha\in S}\iota_\alpha\over\longrightarrow {\mathbb G}_{\rm a}$ ainsi que le morphisme de groupes induit $\ell:N(\Qp)\rightarrow \Qp$.

Pour $\alpha\in S$, on note $N^\alpha\simeq \prod_{\beta\in R^+\backslash\{\alpha\}}N_\beta$ le radical unipotent du sous-groupe parabolique $P_\alpha$ de $G$ contenant $B$ dont le groupe de Levi a pour racines simples $\{\alpha\}$. C'est un sous-groupe alg\'ebrique invariant de $N$ tel que $N_\alpha \buildrel\sim\over\rightarrow N/N^\alpha$. On note aussi $Q_\alpha$ le sous-groupe parabolique maximal de $G$ contenant $B$ dont le Levi a pour racines simples $S\backslash \{\alpha\}$. On d\'esigne avec un $-$ en exposant les sous-groupes paraboliques oppos\'es ainsi que leur radical unipotent : $B^-$, $N^-$, $P_\alpha^-$, $Q_\alpha^-$, etc.

Soit $E$ une extension finie de $\Qp$, $E_\infty=\cup_{m\geq 0}E_m$ et $\eta:\Qp\rightarrow E_\infty$ un caract\`ere additif localement constant non trivial. On rappelle que si $\eta'$ est un autre tel caract\`ere, par exemple $\eta'=g\circ\eta$ pour $g\in {\Gal}(E_\infty/E)$, alors on a $\eta'(-)=\eta(a(-))$ pour un $a\in \Qp^\times$, cf. par exemple \cite[\S~1.7]{BHe} (pour $\eta'=g\circ\eta$ on a $a\in \Zp^\times$). On note encore $\eta:N(\Qp)\rightarrow E_\infty$ le caract\`ere induit $N(\Qp)\buildrel\ell\over\rightarrow \Qp\buildrel\eta\over\rightarrow E_\infty$ ainsi que sa restriction \`a tous les sous-groupes de $N(\Qp)$.

Comme dans \cite[\S~0]{SV} on dit qu'un sous-groupe ouvert compact $N^0\subseteq N$ est totalement d\'ecompos\'e si $\prod_{\alpha\in R^+}N_\alpha \buildrel\sim\over\rightarrow N$ induit~:
\begin{equation}\label{decomp}
\prod_{\alpha\in R^+} (N_\alpha(\Qp) \cap N^0)\buildrel\sim\over\rightarrow N^0
\end{equation}
pour tout ordre sur les racines $\alpha$. On fixe une suite croissante de sous-groupes ouverts compacts de $N(\Qp)$ :
$$N_0\subseteq N_1\subseteq \cdots \subseteq N_m\subseteq N_{m+1}\subseteq \cdots$$
v\'erifiant les conditions suivantes $N_m$ est totalement d\'ecompos\'e pour tout $m\in \Z_{\geq 0}$, $\cup_{m\geq 0}N_m=N(\Qp)$ et $\eta(N_m)\subseteq E_m$ pour tout $m\in \Z_{\geq 0}$. Pour $\alpha\in S$ et $m\in \Z_{\geq 0}$ on note $N_m^\alpha\=N^\alpha(\Qp)\cap N_m\simeq \prod_{\beta\in R^+\backslash\{\alpha\}}N_{\beta}(\Qp)\cap N_m$ et on suppose de plus que $N(\Qp)/N^\alpha(\Qp)\simeq N_\alpha(\Qp)\buildrel\iota_\alpha\over\simeq \Qp$ induit par restriction pour tout $m\in \Z_{\geq 0}$ des isomorphismes~:
\begin{equation}\label{zp}
N_\alpha(\Qp)\cap N_m=N_m/N_m^\alpha\buildrel\substack{\iota_\alpha \\ \sim}\over\longrightarrow \frac{1}{p^{m}}\Zp\subset \Qp.
\end{equation}

On suppose enfin que le centre $Z_G$ est connexe. Pour tout $\alpha\in S$ il existe alors des cocaract\`eres fondamentaux $\lambda_{\alpha^{\!\vee}}\in X^\vee(T)$ v\'erifiant pour $\beta\in S$ (voir par exemple \cite[Prop.2.1.1(iii)]{BH}) :
\begin{equation}\label{cocar}
\langle\beta,\lambda_{\alpha^{\!\vee}}\rangle=\left\{\begin{array}{ccc}1&\ {\rm si}\ &\alpha=\beta\\ 0&\ {\rm si}\ &\alpha \ne \beta. \end{array}\right.
\end{equation}
Rappelons que les $\lambda_{\alpha^{\!\vee}}$ sont d\'efinis \`a addition pr\`es d'un \'el\'ement de $X^0(T)=\{\lambda\in X^\vee(T),\ \langle \beta,\lambda\rangle=0\ \forall\ \beta\in R\}=X^\vee(Z_G)$ (la derni\`ere \'egalit\'e d\'ecoule de $Z_G=\cap_{\beta\in R}\ker(\beta)$, cf. \cite[\S~II.1.6]{Ja}) et que $\lambda_{\alpha^{\!\vee}}:{\mathbb G}_{\rm m}\rightarrow T$ se factorise par le centre $Z_{L_{Q_\alpha}}$ de $L_{Q_\alpha}$, i.e. induit $\lambda_{\alpha^{\!\vee}}:{\mathbb G}_{\rm m}\rightarrow Z_{L_{Q_\alpha}}$. Pour $N^0\subseteq N$ sous-groupe ouvert compact totalement d\'ecompos\'e, comme $N_\alpha(\Qp) \cap N^0$ pour $\alpha\in R^+$ est isomorphe \`a un sous-$\Zp$-module libre de rang $1$ dans $N_\alpha(\Qp)\buildrel\iota_\alpha\over\simeq \Qp$, on d\'eduit de (\ref{alpha}), (\ref{cocar}) et (\ref{decomp}) que l'on a pour $\alpha\in S$ et $z\in \Zp\backslash\{0\}\subseteq \Qp^\times$~:
\begin{equation}\label{incl}
\lambda_{\alpha^{\!\vee}}(z)N^0\lambda_{\alpha^{\!\vee}}(z)^{-1}\subseteq N^0
\end{equation}
avec \'egalit\'e lorsque $z\in \Zp^\times$.

\begin{ex}\label{gln}
{\rm Le cas essentiel pour cet article et celui o\`u $G=\G$ ($n\geq 2$), $B=$ Borel des matrices triangulaires sup\'erieures, $T=$ tore diagonal, $N=$ matrices unipotentes sup\'erieures et $\ell:\smat{1&a_{1,2}&\cdots&\cdots\\0&1&a_{2,3}&\cdots\\\vdots& \cdots &\ddots&\ddots}\longmapsto \sum_{i=1}^{n-1}a_{i,i+1}$ (ce qui d\'etermine les $\iota_\alpha$). Pour $i\in \{1,\dots,n\}$ soit $e_i\in X(T)$ tel que $e_i(\diag(t_1,\dots,t_n))=t_i$, on a $S=\{e_{i}-e_{i+1},\ i\in \{1,\dots,n-1\}\}$ et $R\simeq R^\vee$. On prend $\eta$ tel que $\eta\vert_{\Zp}=1$, $\eta(1/p)\ne 1$ et $N_m\subseteq N(\Qp)$ pour $m\in \Z_{\geq 0}$ le sous-groupe ouvert compact des matrices telles que $a_{i,j}\in \frac{1}{p^{m(j-i)}}\Zp$ pour $1\leq i<j\leq n$. Enfin on prend $\lambda_{e_{i}-e_{i+1}}(x)\=\diag(\underbrace{x,\dots,x}_{i},\underbrace{1,\dots,1}_{n-i})$. Noter que $Z_G={\mathbb G}_{\rm m}$ et $X^\vee(Z_G)=\{x\mapsto \diag(x^i,\dots,x^i),\ i\in \Z\}$.}
\end{ex}

\begin{lem}\label{easy}
(i) Pour tout $\alpha\in S$ et tout $m\in \Z_{\geq 0}$, le sous-groupe $N_m^\alpha$ est normal dans $N_m$.\\
(ii) Pour tout $\alpha\in S$, tout $z\in \Qp^\times$, tout $x\in N(\Qp)$ et tout $x_\alpha\in N^\alpha(\Qp)$ on a $\eta(\lambda_{\alpha^{\!\vee}}(z)x_\alpha\lambda_{\alpha^{\!\vee}}(z)^{-1})=\eta(x_\alpha)=\eta(xx_\alpha x^{-1})$.
\end{lem}
\begin{proof}
(i) Le sous-groupe $N^\alpha(\Qp)$ est normal dans $N(\Qp)$. (ii) Par (\ref{alpha}) et (\ref{cocar}) on a $\ell(\lambda_{\alpha^{\!\vee}}(z)x_\alpha\lambda_{\alpha^{\!\vee}}(z)^{-1})=\ell(x_\alpha)$, et il est clair que $\ell(xx_\alpha x^{-1})=\ell(x_\alpha)$.
\end{proof}

On fixe $\alpha\in S$ jusqu'\`a la fin de ce paragraphe. On note $\mnn^\alpha$ la $\Qp$-alg\`ebre de Lie de $N^\alpha(\Qp)$, ou de mani\`ere \'equivalente de $N_m^\alpha$ pour $m\in \Z_{\geq 0}$. Si $V$ est un $U(\mnn^\alpha)$-module, on note $V_{\mnn^\alpha}\simeq V\otimes_{U(\mnn^\alpha)}\Qp$ le quotient de $V$ par le sous-$\Qp$-espace vectoriel de $V$ engendr\'e par les vecteurs $xv$ pour $(x,v)\in \mnn^\alpha\times V$.

On fixe de plus $m\in \Z_{\geq 0}$ jusqu'\`a la fin de ce paragraphe. Soit $\pi$ une repr\'esentation dans $\Repm(B(\Qp))$, on rappelle que $\pi^\vee$ est un espace de Fr\'echet (nucl\'eaire) r\'eflexif muni d'une structure de $D(B(\Qp),E_m)$-module (\`a gauche) s\'epar\'ement continu (cf. \cite[\S~3]{ST1}). En particulier $\pi$ et $\pi^\vee$ sont munis d'actions continues de $N^\alpha(\Qp)$ et $\mnn^\alpha$ pour tout $\alpha\in S$ qui commutent avec la structure de $E_m$-espace vectoriel. On note $\overline{\langle xv,x\in \mnn^\alpha, v\in \pi^\vee \rangle}$ l'adh\'erence dans $\pi^\vee$ du sous-$E_m$-espace vectoriel ${\langle xv,x\in \mnn^\alpha, v\in \pi^\vee \rangle}$ engendr\'e par $xv$ pour $(x,v)\in \mnn^\alpha\times \pi^\vee$. On a $(\pi^\vee)_{\mnn^\alpha}^{\sep}\=((\pi^\vee)_{\mnn^\alpha})^{\sep}\simeq \pi^\vee/\overline{\langle xv,x\in \mnn^\alpha, v\in \pi^\vee \rangle}$ et le $E_m$-espace vectoriel $(\pi^\vee)_{\mnn^\alpha}^{\sep}$ muni de la topologie quotient est encore un espace de Fr\'echet r\'eflexif (cf. \cite[\S~1]{ST1}). On d\'efinit de m\^eme l'espace de Fr\'echet r\'eflexif $(\pi^\vee)_{N_m^\alpha}^{\sep}\=((\pi^\vee)_{N_m^\alpha})^{\sep}$ en rempla\c cant ${\langle xv,x\in \mnn^\alpha, v\in \pi^\vee \rangle}$ par ${\langle gv-v,g\in N_m^\alpha, v\in \pi^\vee \rangle}$.

Les sous-$E_m$-espaces vectoriels $\pi[{\mnn^\alpha}]$ et $\pi^{N_m^\alpha}$ de $\pi$ sont ferm\'es, donc sont des espaces de type compact (pour la topologie induite) par \cite[Prop.~1.2]{ST1}. De plus $\pi[{\mnn^\alpha}]$ est stable sous l'action de $N_m^\alpha$ puisque $\mnn^\alpha$ est la $\Qp$-alg\`ebre de Lie de $N_m^\alpha$, et cette action de $N_m^\alpha$ sur $\pi[{\mnn^\alpha}]$ est lisse par l'argument de \cite[p.~114]{ST4}. On note $(-)(\eta)$, resp. $(-)(\eta^{-1})$, le tordu de $(-)$ par le caract\`ere $\eta$, resp. $\eta^{-1}$, pour l'action de $N^\alpha_m$.

\begin{lem}\label{gnouf}
(i) On a des isomorphismes topologiques $(\pi^\vee)_{\mnn^\alpha}^{\sep}\simeq \pi[{\mnn^\alpha}]^\vee$ et $(\pi^\vee)(\eta)_{N_m^\alpha}^{\sep}\simeq (\pi(\eta^{-1})^{N_m^\alpha})^\vee$.\\
(ii) Le sous-$E_m$-espace vectoriel de $\pi[{\mnn^\alpha}](\eta^{-1})$ engendr\'e par $xv-v$ pour $(x,v)\in N^\alpha_m\times \pi[\mnn^\alpha](\eta^{-1})$ est ferm\'e.
\end{lem}
\begin{proof}
(i) En proc\'edant par r\'ecurrence comme dans la preuve de \cite[Prop.~8.4]{Br2}, on filtre $N^\alpha$ par une suite croissante~:
$$0=N^{(0)}_{P_\alpha}\subsetneq N^{(1)}_{P_\alpha} \subsetneq \cdots \subsetneq N^{(d_\alpha-1)}_{P_\alpha}\subsetneq N^{(d_\alpha)}_{P_\alpha}=N^\alpha$$
de sous-groupes alg\'ebriques (sur $\Qp$) ferm\'es normaux tels que $\dim N^{(i)}_{P_\alpha}=i$. En notant $\mnn^{(i)}_{P_\alpha}$ l'alg\`ebre de Lie de $N^{(i)}_{P_\alpha}(\Qp)$ et $N_m^{\alpha,(i)}\=N^{(i)}_{P_\alpha}(\Qp)\cap N_m$, on a $\pi[\mnn^{(i)}_{P_\alpha}]\simeq (\pi[\mnn^{(i-1)}_{P_\alpha}])[\mnn^{(i)}_{P_\alpha}/\mnn^{(i-1)}_{P_\alpha}]$ et $\pi(\eta^{-1})^{N_m^{\alpha,(i)}}=(\pi(\eta^{-1})^{N_m^{\alpha,(i-1)}})^{N_m^{\alpha,(i)}/N_m^{\alpha,(i-1)}}$ pour $i\in \{1,\dots,d_\alpha\}$. Le premier isomorphisme du (i) d\'ecoule alors facilement par r\'ecurrence de l'isomorphisme de Fr\'echet r\'eflexifs $\pi[x]^\vee \simeq \pi^\vee/\overline{\langle xv,v\in \pi^\vee\rangle}$ (pour un \'el\'ement quelconque $x\in \mnn^\alpha$ et avec des notations \'evidentes) qui lui m\^eme d\'ecoule par r\'eflexivit\'e de $(\pi^\vee/\overline{\langle xv,v\in \pi^\vee\rangle})^\vee\simeq (\pi^\vee/{\langle xv,v\in \pi^\vee\rangle})^\vee\simeq \pi[x]$. Le deuxi\`eme se d\'eduit de mani\`ere analogue en rempla\c cant $x$ par $g-\Id$.\\
(ii) Notons $\langle xv-v,x\in N^\alpha_m, v\in \pi[\mnn^\alpha](\eta^{-1})\rangle$ le sous-$E_m$-espace vectoriel de $\pi[{\mnn^\alpha}](\eta^{-1})$ engendr\'e par $xv-v$ pour $(x,v)\in N^\alpha_m\times \pi[\mnn^\alpha](\eta^{-1})$ et $\pi[{\mnn^\alpha}](\eta^{-1})_{N_m^\alpha}$ le quotient $\pi[{\mnn^\alpha}](\eta^{-1})/\langle xv-v,x\in N^\alpha_m, v\in \pi[\mnn^\alpha](\eta^{-1})\rangle$. Comme l'action du groupe compact $N_m^\alpha$ sur $\pi[{\mnn^\alpha}]$, et donc sur la repr\'esentation tordue $\pi[{\mnn^\alpha}](\eta^{-1})$, est lisse, on dispose de l'application de projection usuelle~:
\begin{equation}\label{proj}
\pi[{\mnn^\alpha}](\eta^{-1})\twoheadrightarrow \pi[{\mnn^\alpha}](\eta^{-1})^{N_m^\alpha}
\end{equation}
dont le noyau est exactement le sous-espace $\langle xv-v,x\in N^\alpha_m, v\in \pi[\mnn^\alpha](\eta^{-1})\rangle$ (utiliser que le noyau contient ce sous-espace et que la surjection $\pi[{\mnn^\alpha}](\eta^{-1})\twoheadrightarrow \pi[{\mnn^\alpha}](\eta^{-1})_{N_m^\alpha}$ induit une surjection $\pi[{\mnn^\alpha}](\eta^{-1})^{N_m^\alpha}\twoheadrightarrow \pi[{\mnn^\alpha}](\eta^{-1})_{N_m^\alpha}$ par exactitude du foncteur $(-)^{N^\alpha_m}$ sur les repr\'esentations lisses de $N^\alpha_m$). Il suffit donc de montrer que l'application (\ref{proj}) est continue o\`u $\pi[{\mnn^\alpha}](\eta^{-1})^{N_m^\alpha}$ est muni de la topologie induite par $\pi[{\mnn^\alpha}](\eta^{-1})$. Mais cela d\'ecoule de \cite[Prop.~7.1.6]{Em} avec \cite[Cor.~7.1.4]{Em}.
\end{proof}

Par \ le \ (ii) \ du \ Lemme \ \ref{gnouf}, \ le \ $E_m$-espace \ vectoriel \ $\pi[{\mnn^\alpha}](\eta^{-1})_{N_m^\alpha}=\pi[{\mnn^\alpha}](\eta^{-1})/\langle xv-v,x\in N^\alpha_m, v\in \pi[\mnn^\alpha](\eta^{-1})\rangle$ avec la topologie quotient de $\pi[{\mnn^\alpha}](\eta^{-1})$ est (s\'epar\'e) de type compact. De plus l'action de $N_m$ sur $\pi$ pr\'eserve le sous-espace $\pi[\mnn^\alpha]$ (car, par le (i) du Lemme \ref{easy}, l'action adjointe de $N_m$ sur la $\Qp$-alg\`ebre de Lie de $N(\Qp)$ pr\'eserve la sous-alg\`ebre $\mnn^\alpha$) et, toujours par le Lemme \ref{easy}, fait du quotient $\pi[{\mnn^\alpha}](\eta^{-1})_{N_m^\alpha}$ une repr\'esentation localement analytique de $N_m/N_m^\alpha\simeq \frac{1}{p^m}\Zp$ (cf. (\ref{zp})). Pour $z\in \Qp^\times$, l'action de $\lambda_{\alpha^{\!\vee}}(z)$ sur $\pi$ pr\'eserve de m\^eme $\pi[\mnn^\alpha]$, et si de plus $\val(z)\geq 0$ elle passe au quotient $\pi[{\mnn^\alpha}](\eta^{-1})_{N_m^\alpha}$ par (\ref{incl}) (appliqu\'e \`a $N^0=N_m^\alpha$) et la premi\`ere \'egalit\'e du (ii) du Lemme \ref{easy}, et d\'efinit une action continue du mono\"\i de multiplicatif $\Zp\backslash\{0\}$ (vu dans $\Qp^\times$) sur $\pi[{\mnn^\alpha}](\eta^{-1})_{N_m^\alpha}$. De plus, si $x\in \frac{1}{p^m}\Zp$ et $z\in \Zp\backslash\{0\}\subset \Qp^\times$, on a par (\ref{alpha}) et (\ref{cocar})~:
\begin{equation}\label{comut}
z\circ x=zx\circ z\ {\rm sur}\ \pi[{\mnn^\alpha}](\eta^{-1})_{N_m^\alpha}
\end{equation}
o\`u $zx$ est vu dans $ \frac{1}{p^m}\Zp\simeq N_m/N_m^\alpha$.

Par le (i) du Lemme \ref{gnouf}, on a un isomorphisme topologique d'espaces de Fr\'echet r\'eflexifs~:
\begin{equation}\label{dualm}
\big(\pi[{\mnn^\alpha}](\eta^{-1})_{N_m^\alpha}\big)^{\!\vee}=(\pi[{\mnn^\alpha}]^\vee)(\eta)^{N_m^\alpha}\simeq (\pi^\vee)_{\mnn^\alpha}^{\sep}(\eta)^{N_m^\alpha}
\end{equation}
o\`u $(\pi^\vee)_{\mnn^\alpha}^{\sep}(\eta)^{N_m^\alpha}$ est muni de la topologie induite par $(\pi^\vee)_{\mnn^\alpha}^{\sep}(\eta)$ (noter que la premi\`ere \'egalit\'e est imm\'ediate). En particulier $(\pi^\vee)_{\mnn^\alpha}^{\sep}(\eta)^{N_m^\alpha}$ est un $D(\Zp,E_m)$-module continu via $N_0/N^\alpha_0\simeq \Zp$. Si $z\in \Qp^\times$, $\val(z)\leq 0$, on a aussi une action continue de $z$ sur $(\pi^\vee)_{\mnn^\alpha}^{\sep}(\eta)^{N_m^\alpha}$ avec la convention que~:
\begin{equation}\label{actionsurdual}
(zf)(-)\=f(z^{-1}(-))
\end{equation}
si $f\in (\pi[{\mnn^\alpha}](\eta^{-1})_{N_m^\alpha})^\vee$ et $z\in \Qp^\times$, $\val(z)\leq 0$. Si $x\in N_0/N_0^\alpha\simeq \Zp$ et $z\in \Qp^\times$, $\val(z)\leq 0$, on d\'eduit de (\ref{comut}) que l'on a sur $(\pi^\vee)_{\mnn^\alpha}^{\sep}(\eta)^{N_m^\alpha}$~:
\begin{equation}\label{comut2}
z\circ z^{-1}x=x\circ z
\end{equation}
o\`u $z^{-1}x$ est vu dans $\Zp\simeq N_0/N_0^\alpha$.

\begin{rem}\label{fixe}
{\rm La fin de la preuve du (ii) du Lemme \ref{gnouf} montre que pour $\pi$ dans $\Repm(B(\Qp))$ on a une bijection continue, donc un isomorphisme, entre espaces de type compact~:
$$\pi[{\mnn^\alpha}](\eta^{-1})_{N_m^\alpha}\buildrel\sim\over\longrightarrow \pi[{\mnn^\alpha}](\eta^{-1})^{N_m^\alpha}=\pi(\eta^{-1})^{N_m^\alpha}$$
dont on v\'erifie facilement qu'il est $N_m/N_m^\alpha$-\'equivariant et commute aux actions de $\lambda_{\alpha^{\!\vee}}(\Zp\backslash\{0\})$ en d\'efinissant l'action de $\lambda_{\alpha^{\!\vee}}(\Zp\backslash\{0\})$ sur le terme de droite via l'action de $G(\Qp)$ sur $\pi$ suivie de la projection sur $\pi(\eta^{-1})^{N_m^\alpha}$. On a donc aussi un isomorphisme compatible aux actions duales entre espaces de Fr\'echet r\'eflexifs par le (i) du Lemme \ref{gnouf} et (\ref{dualm})~:
$$(\pi^\vee)_{\mnn^\alpha}^{\sep}(\eta)^{N_m^\alpha}\simeq (\pi^\vee)(\eta)_{N_m^\alpha}^{\sep}.$$
On n'utilisera ces isomorphismes qu'\`a $m$ fix\'e (noter que les fl\`eches naturelles quand $m$ grandit ne vont pas dans le m\^eme sens des deux c\^ot\'es). Quand on fera varier $m$ (cf. par exemple le (iii) du Lemme \ref{m+1}) c'est $\pi[{\mnn^\alpha}](\eta^{-1})_{N_m^\alpha}$ et $(\pi^\vee)_{\mnn^\alpha}^{\sep}(\eta)^{N_m^\alpha}$ que l'on utilisera.}
\end{rem}

\begin{lem}\label{dense}
Soit $V'\hookrightarrow V$ un morphisme injectif continu de $E_m$-espaces vectoriels de type compact. Alors le morphisme (continu) $V^\vee\rightarrow {V'}^\vee$ est d'image dense.
\end{lem}
\begin{proof}
Notons $\overline{V'}$ l'adh\'erence de $V'$ dans $V$ (un espace de type compact par \cite[\S~1]{ST1}), on a $V^\vee\twoheadrightarrow \overline{V'}^\vee\hookrightarrow {V'}^\vee$ et il suffit de montrer que $\overline{V'}^\vee\hookrightarrow {V'}^\vee$ est d'image dense. Soit $W$ l'adh\'erence de $\overline{V'}^\vee$ dans ${V'}^\vee$ (un espace de Fr\'echet r\'eflexif par {\it loc.cit.}), par r\'eflexivit\'e des espaces de type compact on d\'eduit $V'\twoheadrightarrow W^\vee\hookrightarrow \overline{V'}$ d'o\`u $V'\buildrel\sim\over\rightarrow W^\vee$ puisque $V'\hookrightarrow \overline{V'}$. Par r\'eflexivit\'e de $W$ on obtient $W\buildrel\sim\over\rightarrow {V'}^\vee$ ce qui ach\`eve la preuve.
\end{proof}

\begin{prop}\label{exactdense}
Soit $0\rightarrow \pi''\rightarrow \pi\rightarrow \pi'$ une suite exacte de repr\'esentations dans $\Repm(B(\Qp))$. Alors on a un complexe d'espaces de Fr\'echet r\'eflexifs~:
$$((\pi')^\vee)_{\mnn^\alpha}^{\sep}(\eta)^{N_m^\alpha}\buildrel f\over \longrightarrow (\pi^\vee)_{\mnn^\alpha}^{\sep}(\eta)^{N_m^\alpha}\buildrel g \over\longrightarrow ((\pi'')^\vee)_{\mnn^\alpha}^{\sep}(\eta)^{N_m^\alpha}\longrightarrow 0$$
o\`u $g$ est une surjection topologique et o\`u l'image de $f$ est dense dans $\ker(g)$.
\end{prop}
\begin{proof}
On a une suite exacte de repr\'esentations lisses de $N_m^\alpha$~:
$$0\longrightarrow \pi''[{\mnn^\alpha}](\eta^{-1})\longrightarrow \pi[{\mnn^\alpha}](\eta^{-1})\longrightarrow \pi'[{\mnn^\alpha}](\eta^{-1})$$
d'o\`u, par exactitude du foncteur des coinvariants $(-)_{N_m^\alpha}$ sur les repr\'esentations lisses de $N_m^\alpha$, une suite exacte de $E_m$-espaces vectoriels de type compact~:
\begin{equation}\label{exactgauche}
0\longrightarrow \pi''[{\mnn^\alpha}](\eta^{-1})_{N_m^\alpha}\longrightarrow \pi[{\mnn^\alpha}](\eta^{-1})_{N_m^\alpha}\longrightarrow \pi'[{\mnn^\alpha}](\eta^{-1})_{N_m^\alpha}
\end{equation}
L'injection de gauche est automatiquement une immersion ferm\'ee (rappelons l'argument : elle induit un isomorphisme continu entre $\pi''[{\mnn^\alpha}](\eta^{-1})_{N_m^\alpha}$ et le noyau de l'application de droite, donc un isomorphisme topologique car ce sont deux espaces de type compact). En dua\-lisant (\ref{exactgauche}), le r\'esultat d\'ecoule du Lemme \ref{dense} appliqu\'e \`a l'injection continue $\pi[{\mnn^\alpha}](\eta^{-1})_{N_m^\alpha}/\pi''[{\mnn^\alpha}](\eta^{-1})_{N_m^\alpha}\hookrightarrow \pi'[{\mnn^\alpha}](\eta^{-1})_{N_m^\alpha}$.
\end{proof}

\subsection{$(\varphi,\Gamma)$-modules g\'en\'eralis\'es}\label{gen}

On donne quelques rappels et r\'esultats simples sur les $(\varphi,\Gamma)$-modules g\'en\'eralis\'es (\cite[\S~4]{Li}).

On fixe une extension finie $E$ de $\Qp$. Pour $r$ dans $\Q_{>0}$, soit $\Rr$ le $E$-espace de Fr\'echet des fonctions rigides analytiques \`a coefficients dans $E$ sur la couronne $p^{-1/r}\leq \norm <1$. On a $\Rr\subseteq {\mathcal R}^{r'}_{E}$ si $r\leq r'$ et on note $\R\={\displaystyle \lim_{r\rightarrow +\infty}}\Rr$ l'anneau de Robba \`a coefficients dans $E$. On note \'egalement $\R^+$ le $E$-espace de Fr\'echet des fonctions rigides analytiques \`a coefficients dans $E$ sur le disque unit\'e ouvert $0\leq \norm <1$ (la ``partie positive'' de $\R$). Un \'el\'ement de $\Rr$ (resp. $\R^+$) est une s\'erie convergente $\sum_{i=-\infty}^{+\infty}a_iX^i$ avec $a_i\in E$ (resp. et $a_i=0$ pour $i<0$). On rappelle que tout sous-$\Rr$-module de type fini de $(\Rr)^d$ (pour un entier $d\geq 0$) est libre (de rang fini) et ferm\'e dans $(\Rr)^d$ pour la topologie de Fr\'echet (\cite[Prop.~1.1.1]{Be3}). Si $D_r$ est un $\Rr$-module de {\it pr\'esentation finie}, on le munit de sa topologie naturelle d'espace de Fr\'echet d\'efinie par la topologie quotient d'une pr\'esentation finie (on v\'erifie facilement que cela ne d\'epend pas du choix de la pr\'esentation). Tout sous-module ferm\'e d'un $\Rr$-module de pr\'esentation finie est de pr\'esentation finie (prendre l'image inverse du sous-module ferm\'e dans une pr\'esentation et utiliser ce qui pr\'ec\`ede).

\begin{lem}\label{plat}
Soit $r,r'\in \Q_{>0}$, $r\leq r'$, alors les injections d'anneaux $\R^+\hookrightarrow \Rr$ et $\Rr\hookrightarrow {\mathcal R}^{r'}_{E}$ sont plates.
\end{lem}
\begin{proof}
On d\'emontre le premier cas seulement, la preuve du deuxi\`eme \'etant strictement analogue. Comme $\R^+$ et $\Rr$ sont des anneaux de B\'ezout int\`egres (\cite[Prop.~4.12(2)]{Be1}), ce sont des anneaux de Pr\"ufer et par \cite[Cor.~5.3.6]{Gl} ils sont coh\'erents. Mais avec des anneaux coh\'erents, il suffit de tester la platitude avec des modules de pr\'esentation finie (merci \`a A. Abbes pour ce point), i.e. il suffit de voir que si $M''\hookrightarrow M$ est une injection de $\R^+$-modules de pr\'esentation finie, alors $\Rr\!\otimes_{\R^+}\!M''\rightarrow \Rr\!\otimes_{\R^+}\!M$ est encore injectif (utiliser la Remarque $1$ de \cite[\S~I.2.3]{Bo1} et le fait que tout id\'eal de type fini d'un anneau coh\'erent est de pr\'esentation finie). Les anneaux $\R^+$ et $\Rr$ sont des alg\`ebres de Fr\'echet-Stein au sens de \cite[\S~3]{ST2}, plus pr\'ecis\'ement on a des isomorphismes topologiques (pour la topologie limite projective)~:
$$\R^+=\lim_{\substack{\longleftarrow \\ s\rightarrow +\infty}}{\mathcal R}_{E}^{[0,s]}\ \ {\rm et}\ \ \Rr=\lim_{\substack{\longleftarrow \\ s\rightarrow +\infty}}{\mathcal R}_{E}^{[r,s]}$$
o\`u ${\mathcal R}_{E}^{[0,s]}$ (resp. ${\mathcal R}_{E}^{[r,s]}$) est l'espace de Banach des fonctions rigides analytiques \`a coefficients dans $E$ sur la couronne $0\leq \norm \leq p^{-1/s}$ (resp. $p^{-1/r}\leq \norm \leq p^{-1/s}$). Pour tout $\R^+$-module, resp. $\Rr$-module, de pr\'esentation finie $M$, on a~:
\begin{equation}\label{limit}
M\buildrel\sim\over\longrightarrow \lim_{\substack{\longleftarrow \\ s\rightarrow +\infty}}\big({\mathcal R}_{E}^{[0,s]}\!\otimes_{\R^+}\!M\big),\ \ {\rm resp.}\ \ M\buildrel\sim\over\longrightarrow \lim_{\substack{\longleftarrow \\ s\rightarrow +\infty}}({\mathcal R}_{E}^{[r,s]}\!\otimes_{\Rr}\!M),
\end{equation}
par \cite[Cor.~3.4(v)]{ST2}. Soit $M''\hookrightarrow M$ une injection de $\R^+$-modules de pr\'esentation finie. Comme $\R^+\hookrightarrow {\mathcal R}_{E}^{[0,s]}$ est plat par \cite[Rem.~3.2]{ST2}, on a encore une injection ${\mathcal R}_{E}^{[0,s]}\!\otimes_{\R^+}\!M''\hookrightarrow {\mathcal R}_{E}^{[0,s]}\!\otimes_{\R^+}\!M$. Comme ${\mathcal R}_{E}^{[0,s]}\hookrightarrow {\mathcal R}_{E}^{[r,s]}$ est aussi plat (car cela correspond \`a une immersion ouverte entre affino\"\i des), on a aussi une injection ${\mathcal R}_{E}^{[r,s]}\!\otimes_{\R^+}\!M''\hookrightarrow {\mathcal R}_{E}^{[r,s]}\!\otimes_{\R^+}\!M$. Comme le foncteur $\displaystyle \lim_{\longleftarrow}$ est exact \`a gauche, on en d\'eduit une injection~:
$$\lim_{\substack{\longleftarrow \\ s\rightarrow +\infty}}\big({\mathcal R}_{E}^{[r,s]}\!\otimes_{\Rr}\!(\Rr\!\otimes_{\R^+}\!M'')\big)\hookrightarrow \lim_{\substack{\longleftarrow \\ s\rightarrow +\infty}}\big({\mathcal R}_{E}^{[r,s]}\!\otimes_{\Rr}\!(\Rr\!\otimes_{\R^+}\!M)\big)$$
d'o\`u une injection $\Rr\!\otimes_{\R^+}\!M''\hookrightarrow \Rr\!\otimes_{\R^+}\!M$ par (\ref{limit}) puisque $\Rr\!\otimes_{\R^+}\!M''$ et $\Rrm\!\otimes_{\R^+}\!M$ sont des $\Rr$-modules de pr\'esentation finie. Cela termine la preuve.
\end{proof}

On munit tous les anneaux pr\'ec\'edents de l'action continue $E$-lin\'eaire de $\Gamma$ usuelle $\gamma(X^i)=\gamma(X)^i=((1+X)^{\varepsilon(\gamma)}-1)^i$ pour $\gamma\in \Gamma$ et $i\in \Z$. On d\'efinit aussi l'op\'erateur de Frobenius $E$-lin\'eaire continu $\varphi:\Rr\rightarrow {\mathcal R}^{pr}_{E}$ pour $r>p-1$ (resp $\varphi:\R\rightarrow \R$, resp. $\varphi:\R^+\rightarrow \R^+$) par $\varphi(X^i)=\varphi(X)^i=((1+X)^{p}-1)^i$ pour $i\in \Z$ (en remarquant que $\vert (1+z)^p-1\vert = \vert z^p\vert = \vert z\vert^p$ si $p^{-1/r}\leq \vert z\vert <1$ et $r>p-1$). On note $t\=\log(1+X)\in \R^+$, qui v\'erifie $\gamma(t)=\varepsilon(\gamma)t$ et $\varphi(t)=pt$.

Pour $r\in \Q_{>p-1}$, on appelle $(\varphi,\Gamma)$-module g\'en\'eralis\'e (resp. $(\varphi,\Gamma)$-module) sur $\Rr$ un $\Rr$-module $D_r$ de pr\'esentation finie (resp. libre de rang fini) muni d'une action semi-lin\'eaire continue de $\Gamma$ et d'un morphisme $\varphi:D_r\rightarrow {\mathcal R}^{pr}_{E}\otimes_{\Rr}D_r$ qui commute \`a $\Gamma$ et induit un isomorphisme ${\mathcal R}^{pr}_{E}$-lin\'eaire~:
\begin{equation}\label{frob}
\Id\otimes \varphi : \Rpr\!\otimes_{\varphi,\Rr}\!D_r \buildrel\sim\over\longrightarrow {\mathcal R}^{pr}_{E}\!\otimes_{\Rr}\!D_r.
\end{equation}
Si $D_r$ est un $(\varphi,\Gamma)$-module g\'en\'eralis\'e sur $\Rr$, l'argument de la preuve de \cite[Prop.~4.1]{Li} montre que pour $r'\gg r$ on a un isomorphisme de ${\mathcal R}^{r'}_{E}$-modules~:
\begin{equation}\label{forme}
{\mathcal R}^{r'}_{E}\!\otimes_{\Rr}\!D_r\simeq ({\mathcal R}^{r'}_{E})^d\oplus (\oplus_{i=1}^{d'} {\mathcal R}^{r'}_{E}/(t^{k_i}))
\end{equation}
pour des entiers $d, d', k_i\geq 0$ (noter au passage que l'injection $\Rr\subseteq {\mathcal R}^{r'}_{E}$ induit une surjection $\Rr/(t^k)\twoheadrightarrow {\mathcal R}^{r'}_{E}/(t^k)$, cf. \cite[Lem.~4.2]{Li}). Si $r'\geq r$, alors $D_{r'}\={\mathcal R}^{r'}_{E}\!\otimes_{\Rr}\!D_r$ avec action de $\Gamma$ par extension des scalaires et~:
\begin{eqnarray*}
\varphi : D_{r'}&\longrightarrow &{\mathcal R}^{pr'}_{E}\!\otimes_{{\mathcal R}^{r'}_{E}}\!D_{r'}\simeq {\mathcal R}^{pr'}_{E}\!\otimes_{{\mathcal R}^{pr}_{E}}\!({\mathcal R}^{pr}_{E}\!\otimes_{\Rr}\!D_r)\\
\lambda \otimes d&\mapsto &\varphi(\lambda) \otimes \varphi(d),\ \lambda\in {\mathcal R}^{r'}_{E}, d\in D_r
\end{eqnarray*}
d\'efinit un foncteur des $(\varphi,\Gamma)$-modules (g\'en\'eralis\'es) sur $\Rr$ vers les $(\varphi,\Gamma)$-modules (g\'en\'eralis\'es) sur ${\mathcal R}^{r'}_{E}$. On appelle $(\varphi,\Gamma)$-module g\'en\'eralis\'e (resp. $(\varphi,\Gamma)$-module) sur $\R$ un $\R$-module $D$ muni d'actions semi-lin\'eaires de $\varphi$ et $\Gamma$ tel que pour $r\gg 0$ il existe un $(\varphi,\Gamma)$-module g\'en\'eralis\'e (resp. un $(\varphi,\Gamma)$-module) $D_r$ sur $\Rr$ et un morphisme $\Rr$-lin\'eaire $D_r\rightarrow D$ qui induit un isomorphisme $\R\!\otimes_{\Rr}\!D_r\buildrel\sim\over\rightarrow D$ compatible \`a $\varphi$ et $\Gamma$. Si $D_r, D'_r$ (resp. $D,D'$) sont deux $(\varphi,\Gamma)$-modules g\'en\'eralis\'es sur $\Rr$ (resp. $\R$) on note $\Hom_{(\varphi,\Gamma)}(D_r,D'_r)$ (resp. $\Hom_{(\varphi,\Gamma)}(D,D')$) les morphismes de $\Rr$-modules (resp. $\R$-modules) qui commutent \`a $\varphi$ et $\Gamma$. Ce sont des $E$-espaces vectoriels de dimension finie.

\begin{rem}
{\rm Les $(\varphi,\Gamma)$-modules g\'en\'eralis\'es sur $\R$ sont d\'efinis de mani\`ere un peu diff\'erente dans \cite[\S~4.1]{Li} mais on montre facilement que la d\'efinition ci-dessus est \'equivalente en utilisant l'argument pr\'ec\'edant \cite[Lem.~4.2]{Li}.}
\end{rem}

\begin{lem}\label{abel}
(i) Les cat\'egories de $(\varphi,\Gamma)$-modules g\'en\'eralis\'es sur $\Rr$ ou sur $\R$ sont ab\'eliennes. La sous-cat\'egorie des $(\varphi,\Gamma)$-modules g\'en\'eralis\'es sur $\R$ de torsion est artinienne.\\
(ii) Soit $D$, $D_r$, $D_{r'}$ des $(\varphi,\Gamma)$-modules g\'en\'eralis\'es sur $\R$, $\Rr$ et ${\mathcal R}^{r'}_{E}$ respectivement tels que l'on a $f_r:\R\otimes_{\Rr}D_r\buildrel\sim\over\rightarrow D$ et $f_{r'}:\R\otimes_{{\mathcal R}^{r'}_{E}}D_{r'}\buildrel\sim\over\rightarrow D$, alors il existe $r''\geq r,r'$ et un isomorphisme ${\mathcal R}^{r''}_{E}\otimes_{\Rr}D_r\buildrel\sim\over \rightarrow {\mathcal R}^{r''}_{E}\otimes_{{\mathcal R}^{r'}_{E}}D_{r'}$ de $(\varphi,\Gamma)$-modules g\'en\'eralis\'es sur ${\mathcal R}^{r''}_{E}$ tels que le diagramme suivant commute~:
$$\xymatrix{{\mathcal R}_{E}\otimes_{\Rr}D_r \ar^{f_{r'}^{-1}\circ f_r}[r] & \R\otimes_{{\mathcal R}^{r'}_{E}}D_{r'} \\
{\mathcal R}^{r''}_{E}\otimes_{\Rr}D_r\ar^{\sim}[r] \ar[u]& {\mathcal R}^{r''}_{E}\otimes_{{\mathcal R}^{r'}_{E}}D_{r'}.\ar[u]}$$
(iii) Tout morphisme $D\rightarrow D'$ de $(\varphi,\Gamma)$-modules g\'en\'eralis\'es sur $\R$ se factorise sous la forme~:
$$\xymatrix{D \ar[r] & D' \\ D_r\ar[r] \ar[u]& D'_r.\ar[u]}$$
o\`u $D_r,D'_r$ sont des $(\varphi,\Gamma)$-modules g\'en\'eralis\'es sur $\Rr$ pour $r\gg 0$ tels que $\R\!\otimes_{\Rr}\!D_r\buildrel\sim\over\rightarrow D$, $\R\!\!\otimes_{\Rr}\!\!D'_r\buildrel\sim\over\rightarrow D'$ et $D_r\rightarrow D'_r$ est un morphisme de $(\varphi,\Gamma)$-modules g\'en\'eralis\'es sur $\Rr$.
\end{lem}
\begin{proof}
(i) Pour $\R$ le premier \'enonc\'e est d\'emontr\'e dans \cite[\S~4.1]{Li} et le deuxi\`eme d\'ecoule de (\ref{forme}) par extension des scalaires \`a $\R$. Comme $\Rr$ est un anneau coh\'erent (cf. la preuve du Lemme \ref{plat}), un morphisme de $(\varphi,\Gamma)$-modules g\'en\'eralis\'es $f:D_r\rightarrow D'_r$ sur $\Rr$ a son noyau et son image de pr\'esentation finie comme $\Rr$-module. Il suffit donc de montrer que $\ker(f)$ et ${\im}(f)$ v\'erifient (\ref{frob}). Comme ${\mathcal R}^{pr}_{E}$ est libre de rang fini ($=p$) sur $\varphi(\Rr)$, le foncteur ${\mathcal R}^{pr}_{E}\otimes_{\varphi,\Rr}\!\!(-)$ est exact sur les $\Rr$-modules. Le foncteur ${\mathcal R}^{pr}_{E}\otimes_{\Rr}\!\!(-)$ est aussi exact par le Lemme \ref{plat}. Le r\'esultat d\'ecoule alors du diagramme commutatif~:
$$\xymatrix{{\mathcal R}^{pr}_{E}\otimes_{\varphi,\Rr}D_r \ar^{\sim}[r] \ar^{\Id\otimes f}[d] & {\mathcal R}^{pr}_{E}\otimes_{\Rr}D_r \ar^{\Id\otimes f}[d] \\
{\mathcal R}^{pr}_{E}\otimes_{\varphi,\Rr}D'_r\ar^{\sim}[r] & {\mathcal R}^{pr}_{E}\otimes_{\Rr}D'_r.}$$
(ii) Quitte \`a augmenter $r$ et $r'$ on peut supposer $D_r$ et $D_{r'}$ de la forme (\ref{forme}). L'assertion d\'ecoule alors facilement du r\'esultat suivant (on laisse les d\'etails au lecteur) : si $x_r\in \Rr/(t^k)$ (resp. $\Rr$) et $x_{r'}\in {\mathcal R}^{r'}_{E}/(t^k)$ (resp. ${\mathcal R}^{r'}_{E}$) ont m\^eme image dans $\R/(t^k)$ (resp. $\R$), alors il existe $r''\geq r,r'$ tel que $x_r$ et $x_{r'}$ ont m\^eme image dans ${\mathcal R}^{r''}_{E}/(t^k)$ (resp. ${\mathcal R}^{r''}_{E}$), cf. \cite[Lem.~4.2]{Li}. L'argument pour (iii) est analogue.
\end{proof}

Soit $D$ un $(\varphi,\Gamma)$-module g\'en\'eralis\'e sur $\R$. On note $I(D)$ la cat\'egorie des triplets $(r,f_r,D_r)$ o\`u $r\in \Q_{>p-1}$, $D_r$ est un $(\varphi,\Gamma)$-module g\'en\'eralis\'e sur $\Rr$ et $f_r:\R\otimes_{\Rr}D_r\buildrel\sim\over\rightarrow D$ un isomorphisme de $(\varphi,\Gamma)$-modules g\'en\'eralis\'es sur $\R$, les morphismes $(r,f_r,D_r)\rightarrow (r',f_{r'},D_{r'})$ \'etant les isomorphismes $f_{r,r'}:{\mathcal R}^{r'}_{E}\!\otimes_{\Rr}\!D_r\buildrel\sim\over\longrightarrow D_{r'}$ pour $r'\geq r$ de $(\varphi,\Gamma)$-modules g\'en\'eralis\'es sur ${\mathcal R}^{r'}_{E}$ tels que le diagramme commute~:
$$\xymatrix{{\mathcal R}_{E}\!\otimes_{{\mathcal R}^{r'}_{E}}\!({\mathcal R}^{r'}_{E}\!\otimes_{\Rr}\!D_r) \ar^{\ \ \ \ \ \ \ \ \Id\otimes f_{r,r'}}[r]\ar[rd]^{f_r} & \R\otimes_{{\mathcal R}^{r'}_{E}}D_{r'}\ar[d]^{f_{r'}} \\
& D.}$$
On d\'eduit du (ii) du Lemme \ref{abel} (et de sa preuve) que la cat\'egorie $I(D)$ est une cat\'egorie filtrante au sens de \cite[\S~IX.1]{ML} et que l'on un isomorphisme de $\R$-modules qui commute \`a $\varphi$ et $\Gamma$~:
\begin{equation}\label{limitind}
\lim_{\substack{\longrightarrow \\ (r,f_r,D_r)\in I(D)}}D_r\buildrel\sim\over\longrightarrow D
\end{equation}
(en particulier cette limite inductive est filtrante). 

\begin{rem}\label{patho}
{\rm Si $D$ est un $(\varphi,\Gamma)$-module g\'en\'eralis\'e sur $\R$, on pourrait penser munir $D$ de la topologie limite inductive (cf. \cite[\S~5]{Sch}) pour tous les morphismes $D_r\rightarrow D$ en (\ref{limitind}). Mais cette topologie n'est pas tr\`es utile car elle n'est en g\'en\'eral pas s\'epar\'ee. Par exemple, le noyau de la surjection continue $\Rr/(t)\twoheadrightarrow \R/(t)$ est dense dans $\Rr/(t)$ (son adh\'erence est un id\'eal principal de $\Rr/(t)$ dont on v\'erifie facilement que le seul g\'en\'erateur possible est $1$) ! S'il \'etait ferm\'e, il serait tout $\Rr/(t)$, ce qui est impossible.}
\end{rem}

On rappelle (ou on v\'erifie) que $\Rpr\simeq \oplus_{i=0}^{p-1}(1+X)^i\varphi(\Rr)$, de sorte que, si $D_r$ est un $(\varphi,\Gamma)$-module (g\'en\'eralis\'e) sur $\Rr$, on a~:
\begin{equation}\label{psi}
\Rpr\!\otimes_{\varphi,\Rr}\!D_r\simeq \oplus_{i=0}^{p-1}(1+X)^i\otimes D_r.
\end{equation}
On note ${\rm pr}_0:\Rpr\!\otimes_{\varphi,\Rr}\!D_r\twoheadrightarrow D_r$ la projection (continue) sur le facteur correspondant \`a $i=0$ et on d\'efinit l'op\'erateur continu $\psi:D_r\longrightarrow D_r$ comme la compos\'ee~:
\begin{equation}\label{verspsi}
D_r\buildrel 1\otimes \Id\over \longrightarrow \Rpr\!\otimes_{\Rr}\!D_r\buildrel {\buildrel (\Id\otimes \varphi)^{-1}\over\sim}\over\longrightarrow \Rpr\!\otimes_{\varphi,\Rr}\!D_r \buildrel {\rm pr}_0 \over\longrightarrow D_r.
\end{equation}
On voit facilement que $\psi$ commute \`a l'action de $\Gamma$, v\'erifie $\psi(\varphi(\lambda)d)=\lambda\psi(d)$ pour $\lambda\in {\mathcal R}^{r/p}_{E}$ et $d\in D_r$, et que l'on a dans ${\mathcal R}^{pr}_{E}\otimes_{\varphi,\Rr}D_r$ pour $d\in D_r$ via (\ref{psi})~:
\begin{equation}\label{psi2}
(\Id \otimes \varphi)^{-1}(1\otimes d)=\sum_{i=0}^{p-1}(1+X)^i\otimes \frac{\psi((1+X)^{p-i}d)}{(1+X)}.
\end{equation}
En rempla\c cant $\Rr$ et $\Rpr$ par $\R$, on d\'efinit de m\^eme $\psi:D\longrightarrow D$ si $D$ est un $(\varphi,\Gamma)$-module (g\'en\'eralis\'e) sur $\R$, et on a encore (\ref{psi2}). Si $D\simeq \R\!\otimes_{\Rr}\!D_r$ pour $D_r$ un $(\varphi,\Gamma)$-module (g\'en\'eralis\'e) sur $\Rr$, le morphisme $D_r\rightarrow D$ commute \`a $\psi$ (utiliser (\ref{verspsi}) pour $D_r$ et $D$).

\subsection{Foncteurs $F_{\alpha}$}\label{def}

Pour tout $\alpha\in S$ on associe \`a toute repr\'esentation dans $\Rep(B(\Qp))$ un foncteur exact \`a gauche $F_\alpha$ sur les $(\varphi,\Gamma)$-modules g\'en\'eralis\'es sur $\R$.

On conserve les notations des \S\S~\ref{prel} et \ref{gen} et on fixe une racine $\alpha\in S$. Pour $m$ un entier $\geq 0$, on rappelle que l'on a un isomorphisme d'alg\`ebres de Fr\'echet $D(\Zp,E_m)\simeq \Rm^+$ qui envoie $\mu\in D(\Zp,E_m)$ vers $\mu((1+X)^z)=\sum_{i=0}^{+\infty}\mu(\binom{z}{i})X^i\in \Rm^+$ (\cite{Am}) et o\`u la multiplication par $z\in \Zp^\times$ (resp. $p$) sur $\Zp$ induit par fonctorialit\'e de $D(\Zp,E_m)$ l'op\'erateur $\gamma_z=\varepsilon^{-1}(z)$ (resp. $\varphi$) sur $\Rm^+$. Soit $\pi$ dans $\Repm(B(\Qp))$, on d\'eduit de la fin du \S~\ref{prel} que $(\pi^\vee)_{\mnn^\alpha}^{\sep}(\eta)^{N_m^\alpha}$ est un $\Rm^+$-module continu sur un espace de Fr\'echet muni d'une action semi-lin\'eaire continue de $\Gamma\buildrel\sim\over\rightarrow \Zp^\times$ ($\subset \Qp^\times$) et d'un endomorphisme continu $\psi$ commutant \`a $\Gamma$ donn\'e par l'action de $1/p\in \Qp^\times$ (cf. (\ref{actionsurdual})) et v\'erifiant $\psi(\varphi(\lambda)v)=\lambda\psi(v)$ pour $\lambda\in \Rm^+$ et $v\in (\pi^{\vee})_{\mnn^\alpha}^{\sep}(\eta)^{N_m^\alpha}$ (cf. (\ref{comut2})). Comme en (\ref{psi2}) on en d\'eduit une application $E_m$-lin\'eaire continue~:
\begin{equation}\label{psi3}
(\pi^\vee)_{\mnn^\alpha}^{\sep}(\eta)^{N_m^\alpha}\rightarrow \Rm^+\!\otimes_{\varphi,\Rm^+}\!(\pi^\vee)_{\mnn^\alpha}^{\sep}(\eta)^{N_m^\alpha},\ v\mapsto \sum_{i=0}^{p-1}(1+X)^i\otimes \frac{\psi((1+X)^{p-i}v)}{(1+X)}
\end{equation}
et on v\'erifie facilement qu'elle est $\Rm^+$-lin\'eaire et commute \`a $\Gamma$. On note pour tout $\pi$ dans $\Repm(B(\Qp))$~:
\begin{equation}\label{malpha}
M_{\alpha}(\pi)\=(\pi^\vee)_{\mnn^\alpha}^{\sep}(\eta)^{N_m^\alpha}\buildrel (\ref{dualm})\over \simeq \big(\pi[{\mnn^\alpha}](\eta^{-1})_{N_m^\alpha}\big)^{\!\vee}.
\end{equation}
Noter que l'on a \'egalement (\`a $m$ fix\'e) par la Remarque \ref{fixe}~:
\begin{equation}\label{malphaf}
M_{\alpha}(\pi)\simeq (\pi^\vee)(\eta)_{N_m^\alpha}^{\sep}.
\end{equation}
Tout morphisme $\pi'\rightarrow \pi$ dans $\Repm(B(\Qp))$ induit un morphisme continu de $\Rm^+$-modules $M_{\alpha}(\pi)\rightarrow M_{\alpha}(\pi')$ qui commute \`a $\Gamma$ et $\psi$.

Soit $\pi$ dans $\Rep(B(\Qp))$, pour $m$ entier positif ou nul on consid\`ere le foncteur de la cat\'egorie ab\'elienne des $(\varphi,\Gamma)$-modules g\'en\'eralis\'es sur $\R$ dans la cat\'egorie des $E_m$-espaces vectoriels~:
\begin{equation}\label{foncteurm}
F_{\alpha,m}(\pi)\ :\ T \longmapsto \!\!\! \lim_{\substack{\longrightarrow \\ (r,f_r,T_r)\in I(T)}}\!\!\Hom_{\psi,\Gamma}\big(M_\alpha(\pi\otimes_EE_m),T_r\otimes_EE_m\big)
\end{equation}
o\`u $\Hom_{\psi,\Gamma}(-,-)$ d\'esigne les morphismes continus de $\Rm^+$-modules qui commutent \`a $\psi$ et $\Gamma$. La limite inductive en (\ref{foncteurm}) est filtrante, cf. (\ref{limitind}), et la fonctorialit\'e en $T$ d\'ecoule des points (iii) et (ii) du Lemme \ref{abel}.

\begin{rem}\label{vrac}
{\rm (i) La commutation d'un morphisme de $\Rm^+$-modules $f:M_\alpha(\pi\otimes_EE_m)\rightarrow T_r\otimes_EE_m$ avec $\psi$ est \'equivalente \`a celle du diagramme~:
\begin{equation}\label{psi5}
\begin{gathered}
\xymatrix{M_{\alpha}(\pi\otimes_EE_m) \ar^{\!\!\!\!\!\!\!\!\!\!\!\!\!\!\!\!\!\!\!\!\!\!\!\!(\ref{psi3})}[r] \ar[d]^{f} & \Rm^+\!\otimes_{\varphi,\Rm^+}\!M_{\alpha}(\pi\otimes_EE_m)\ar^{\Id\otimes f}[d]\\
T_r\otimes_EE_m\ar^{\!\!\!\!\!\!\!\!\!\!\!\!\!\!\!\!\!\!\!\!\!\!\!\!\!\!\!(\ref{psi2})}[r] & \Rprm\!\otimes_{\varphi,\Rrm}\!(T_r\otimes_EE_m).}
\end{gathered}
\end{equation}
(ii) L'application $T_r\rightarrow T$ induit un morphisme canonique~:
$$F_{\alpha,m}(\pi)(T)\longrightarrow \Hom_{\psi,\Gamma}\big(M_\alpha(\pi\otimes_EE_m),T\otimes_EE_m\big),$$
mais sans conditions suppl\'ementaires sur $\pi$ ce morphisme n'a pas de raison d'\^etre injectif ni surjectif.\\
(iii) Soit $D_r,D'_r$ deux $(\varphi,\Gamma)$-modules g\'en\'eralis\'es sur $\Rrm$, alors tout morphisme continu de $\Rm^+$-modules $D_r\rightarrow D'_r$ est un morphisme (continu) de $\Rrm$-modules (il est facile de v\'erifier qu'il commute \`a $\Rm^+[1/X]$ puis on utilise la continuit\'e pour \'etendre \`a $\Rrm$). On a le m\^eme r\'esultat avec des $(\varphi,\Gamma)$-modules g\'en\'eralis\'es sur $\Rm$.\\
(iv) Soit $D_r,D'_r$ deux $(\varphi,\Gamma)$-modules g\'en\'eralis\'es sur $\Rrm$, alors par (\ref{psi2}) tout morphisme dans $\Hom_{\psi,\Gamma}(D_r,D'_r)$ (rappelons qu'il s'agit des morphismes continus de $\Rm^+$-modules - ou de mani\`ere \'equivalente par (iii) de $\Rrm$-modules - qui commutent \`a $\psi$ et $\Gamma$) commute aussi \`a $(\Id\otimes \varphi)^{-1}$, d'o\`u on d\'eduit~:
$$\Hom_{(\varphi,\Gamma)}(D_r,D'_r)\buildrel\sim\over\longrightarrow \Hom_{\psi,\Gamma}(D_r,D'_r).$$
On a le m\^eme r\'esultat avec des $(\varphi,\Gamma)$-modules g\'en\'eralis\'es sur $\Rm$.\\
(v) Soit $D_r,D'_r$ des $(\varphi,\Gamma)$-modules g\'en\'eralis\'es sur $\R^r$, alors on a~: $$\Hom_{(\varphi,\Gamma)}(D_r,D'_r)\otimes_EE_m\buildrel\sim\over\longrightarrow \Hom_{(\varphi,\Gamma)}(D_r\otimes_EE_m,D'_r\otimes_EE_m).$$
Idem avec des $(\varphi,\Gamma)$-modules g\'en\'eralis\'es sur $\R$.}
\end{rem}

\begin{lem}\label{m+1}
(i) Pour tout $\pi$ dans $\Rep(B(\Qp))$ et $m\in \Z_{\geq 0}$ il existe une action naturelle du groupe $\Gal(E_\infty/E)$ sur le foncteur $F_{\alpha,m}(\pi)$ qui est $E_m$-semi-lin\'eaire (via l'action de $\Gal(E_\infty/E)$ sur $E_m$).\\
(ii) Pour tout morphisme $\pi\longrightarrow \pi'$ dans $\Rep(B(\Qp))$ on a un morphisme canonique de foncteurs $F_{\alpha,m}(\pi)\longrightarrow F_{\alpha,m}(\pi')$ qui commute \`a l'action de $\Gal(E_\infty/E)$ en (i).\\
(iii) Pour tout $\pi$ dans $\Rep(B(\Qp))$ et $m\in \Z_{\geq 0}$, on a un morphisme canonique de foncteurs $F_{\alpha,m}(\pi)\longrightarrow F_{\alpha,m+1}(\pi)$ qui commute \`a l'action de $\Gal(E_\infty/E)$ en (i).
\end{lem}
\begin{proof}
(i) Soit $g\in \Gal(E_\infty/E)$ et $\eta'\=g\circ \eta$, on a vu au \S~\ref{prel} que pour tout $x\in N(\Qp)$ on a $\eta'(x)=\eta(ax)$ pour un (unique) $a\in \Zp^\times$. Posons $t_g\=\prod_{\beta\in S\backslash\{\alpha\}}\lambda_{\beta^\vee}(a)\in T(\Qp)$, en utilisant (\ref{alpha}), (\ref{decomp}) et (\ref{cocar}) on v\'erifie que la conjugaison par $t_g$ sur $G(\Qp)$ pr\'eserve $N_m$ et $N_m^\alpha$, est l'identit\'e sur $N_m/N_m^\alpha$, commute \`a $T(\Qp)$ et v\'erifie $\eta(t_gxt_g^{-1})=\eta'(x)$ pour tout $x\in N^\alpha(\Qp)$. L'action de $g$ sur $(\pi\otimes_EE_m)[{\mnn^\alpha}]=\pi[{\mnn^\alpha}]\otimes_EE_m$ (via son action sur $E_m$) induit un isomorphisme $E_m$-semi-lin\'eaire~:
\begin{equation}\label{actiong}
g:(\pi\otimes_EE_m)[{\mnn^\alpha}](\eta^{-1})_{N_m^\alpha}\buildrel\sim\over\longrightarrow (\pi\otimes_EE_m)[{\mnn^\alpha}]({\eta'}^{-1})_{N_m^\alpha}
\end{equation}
et l'action de $t_g\in T(\Qp)$ sur $\pi$ induit un isomorphisme $E_m$-lin\'eaire~:
\begin{equation}\label{actiontg}
t_g:(\pi\otimes_EE_m)[{\mnn^\alpha}]({\eta'}^{-1})_{N_m^\alpha}\buildrel\sim\over\longrightarrow (\pi\otimes_EE_m)[{\mnn^\alpha}](\eta^{-1})_{N_m^\alpha},
\end{equation}
ces deux isomorphismes commutant aux actions de $N_m/N_m^\alpha$ et $\lambda_{\alpha^\vee}(\Zp\backslash\{0\})$. On en d\'eduit une action semi-lin\'eaire de $\Gal(E_\infty/E)$ sur $(\pi\otimes_EE_m)[{\mnn^\alpha}](\eta^{-1})_{N_m^\alpha}$ en faisant agir $g$ par $t_g\circ g$, puis une action semi-lin\'eaire sur le dual $M_\alpha(\pi\otimes_EE_m)$ (cf. (\ref{malpha})) donn\'ee par $(gf)(v)\=g(f((t_g^{-1}\circ g^{-1})v))$ pour $v\in (\pi\otimes_EE_m)[{\mnn^\alpha}]({\eta}^{-1})_{N_m^\alpha}$, $f\in M_\alpha(\pi\otimes_EE_m)$ o\`u $g$ agit sur $f((t_g^{-1}\circ g^{-1})v)\in E_m$ par l'action de $\Gal(E_\infty/E)$ sur $E_m$. Soit $T$ un $(\varphi,\Gamma)$-module g\'en\'eralis\'e sur $\R$, on d\'efinit enfin une action de $\Gal(E_\infty/E)$ sur $F_{\alpha,m}(\pi)(T)$ en envoyant $F\in \Hom_{\psi,\Gamma}(M_\alpha(\pi\otimes_EE_m),T_r\otimes_EE_m)$ (cf. (\ref{foncteurm})) sur $gF$ d\'efini encore par $(gF)(f)\= g(F(g^{-1}f))$ o\`u $g^{-1}f$ est l'action ci-dessus de $g^{-1}$ sur $f\in M_\alpha(\pi\otimes_EE_m)$ et $g$ agit sur $F(g^{-1}f)\in T_r\otimes_EE_m$ par l'action de $\Gal(E_\infty/E)$ sur $E_m$ et l'action triviale sur $T_r$.\\
(ii) La preuve est imm\'ediate et laiss\'ee au lecteur.\\
(iii) Par les r\'esultats du \S~\ref{prel}, l'inclusion $N_{m}^\alpha\subset N_{m+1}^\alpha$ induit un morphisme de $E_{m+1}$-espaces vectoriels de type compact~:
$$(\pi[{\mnn^\alpha}]\otimes_EE_m)(\eta^{-1})_{N_m^\alpha}\otimes_{E_m}E_{m+1}\twoheadrightarrow (\pi[{\mnn^\alpha}]\otimes_EE_{m+1})(\eta^{-1})_{N_{m+1}^\alpha}$$
qui induit par (\ref{dualm}) un morphisme continu de ${\mathcal R}_{E_{m+1}}^+$-modules~:
\begin{equation}\label{augmentem}
M_{\alpha}(\pi\otimes_EE_{m+1})\longrightarrow M_{\alpha}(\pi\otimes_EE_m)\otimes_{E_m}E_{m+1}
\end{equation}
qui commute \`a $\Gamma$, $\psi$ ainsi qu'\`a l'action de $\Gal(E_\infty/E)$ d\'efinie au (ii). Ce morphisme induit pour $(r,f_r,T_r)\in I(T)$~:
\begin{multline*}
\Hom_{\psi,\Gamma}\big(M_\alpha(\pi\otimes_EE_m),T_r\otimes_EE_m\big)\hookrightarrow \Hom_{\psi,\Gamma}\big(M_\alpha(\pi\otimes_EE_m)\otimes_{E_m}E_{m+1},T_r\otimes_EE_{m+1}\big)\\
\longrightarrow \Hom_{\psi,\Gamma}\big(M_\alpha(\pi\otimes_EE_{m+1}),T_r\otimes_EE_{m+1}\big)
\end{multline*}
d'o\`u le r\'esultat en passant \`a la limite inductive sur $(r,f_r,T_r)\in I(T)$.
\end{proof}

Pour $m\in \Z_{\geq 0}$ on note $F(\varphi,\Gamma)_m$ la cat\'egorie des foncteurs covariants additifs exact \`a gauche de la cat\'egorie ab\'elienne des $(\varphi,\Gamma)$-modules g\'en\'eralis\'es sur $\R$ dans la cat\'egorie ab\'elienne des $E_m$-espaces vectoriels avec action $E_m$-semi-lin\'eaire de $\Gal(E_\infty/E)$ (cf. \cite[\S~II.4]{ML}). La cat\'egorie $F(\varphi,\Gamma)_m$ est clairement ab\'elienne. On d\'efinit de m\^eme $F(\varphi,\Gamma)_\infty$ en rempla\c cant $E_m$ par $E_\infty$. Le (i) du Lemme \ref{m+1} montre que $F_{\alpha,m}(\pi)$ est dans $F(\varphi,\Gamma)_m$.

On v\'erifie facilement que pour tout $m\in \Z_{\geq 0}$ et tout $\pi\longrightarrow \pi'$ dans $\Rep(B(\Qp))$ le diagramme de foncteurs suivant donn\'e par les (ii) et (iii) du Lemme \ref{m+1} commute~:
\begin{equation}\label{mpi}
\begin{gathered}
\xymatrix{F_{\alpha,m}(\pi)\ar[r] \ar[d] &F_{\alpha,m}(\pi')\ar[d]\\ F_{\alpha,m+1}(\pi)\ar[r]&F_{\alpha,m+1}(\pi').}
\end{gathered}
\end{equation}
On d\'efinit alors le foncteur additif covariant $F_\alpha$ de la cat\'egorie $\Rep(B(\Qp))$ dans la cat\'egorie ab\'elienne $F(\varphi,\Gamma)_\infty$ par~:
\begin{equation}\label{falpha}
\pi\longmapsto F_{\alpha}(\pi)\=\lim_{m\rightarrow +\infty}F_{\alpha,m}(\pi)
\end{equation}
(i.e. $F_{\alpha}(\pi)(T)=\displaystyle \lim_{m\rightarrow +\infty}F_{\alpha,m}(\pi)(T)$ pour tout $(\varphi,\Gamma)$-module g\'en\'eralis\'e $T$ sur $\R$).

Si $\chi:\Gal(E_\infty/E)\rightarrow E^\times$ est un caract\`ere localement analytique, on note $E_\infty(\chi)$ le $E_\infty$-espace vectoriel de dimension $1$ avec action semi-lin\'eaire de $\Gal(E_\infty/E)$ donn\'ee par $g(x)=g(x)\chi(g)$ pour $x\in E_\infty$. Noter que par Hilbert 90 on a $E_\infty(\chi)\cong E_\infty$ si et seulement $\chi$ est lisse (i.e. $\chi$ se factorise par $\Gal(E_m/E)$ pour $m\gg 0$). Pour $\pi$ dans $\Rep(B(\Qp))$, si l'on a $F_{\alpha}(\pi)(T)\cong E_\infty(\chi)\otimes_E\Hom_{(\varphi,\Gamma)}(D_\alpha(\pi),T)$ (fonctoriellement en $T$) pour tout $(\varphi,\Gamma)$-module g\'en\'eralis\'e $T$ sur $\R$ o\`u $D_\alpha(\pi)$ est un $(\varphi,\Gamma)$-module g\'en\'eralis\'e sur $\R$, alors le $(\varphi,\Gamma)$-module $D_\alpha(\pi)$ est uniquement d\'etermin\'e.

\begin{prop}\label{drex}
(i) Le foncteur $\pi\longmapsto F_\alpha(\pi)$ envoie une suite exacte $0\rightarrow \pi''\rightarrow \pi\rightarrow \pi'$ dans $\Rep(B(\Qp))$ vers une suite exacte dans $F(\varphi,\Gamma)_\infty$~:
$$0\longrightarrow F_{\alpha}(\pi'')\longrightarrow F_{\alpha}(\pi)\longrightarrow F_{\alpha}(\pi').$$
(ii) Si l'on a de plus $F_{\alpha}(-)(T)\cong E_\infty(\chi)\otimes_E\Hom_{(\varphi,\Gamma)}(D_\alpha(-),T)$ (fonctoriellement en $T$) pour tout $(\varphi,\Gamma)$-module g\'en\'eralis\'e $T$ sur $\R$ o\`u $-\in \{\pi,\pi'\}$ et $D_\alpha(-)$ est un $(\varphi,\Gamma)$-module g\'en\'eralis\'e sur $\R$, alors on a $F_{\alpha}(\pi'')(T)=E_\infty(\chi)\otimes_E\Hom_{(\varphi,\Gamma)}(D_\alpha(\pi')/D_\alpha(\pi),T)$.
\end{prop}
\begin{proof}
(i) Par exactitude de $\displaystyle \lim_{m\rightarrow +\infty}$, il suffit de montrer le r\'esultat avec $\pi\mapsto F_{\alpha,m}(\pi)$ pour $\pi$ dans $\Repm(B(\Qp))$. Soit $T$ un $(\varphi,\Gamma)$-module g\'en\'eralis\'e sur $\R$ et $(r,f_r,T_r)$ dans $I(T)$, par la Proposition \ref{exactdense} on a un complexe de $E_m$-espaces vectoriels~:
\begin{multline}\label{T}
\Hom_{\psi,\Gamma}\big(M_\alpha(\pi''),T_r\otimes_EE_m\big)\hookrightarrow \Hom_{\psi,\Gamma}\big(M_\alpha(\pi),T_r\otimes_EE_m\big)\\
\longrightarrow \Hom_{\psi,\Gamma}\big(M_\alpha(\pi'),T_r\otimes_EE_m\big)
\end{multline}
o\`u la premi\`ere application est injective. Si $f: M_\alpha(\pi)\rightarrow T_r\otimes_EE_m$ est nul sur ${\im}(M_\alpha(\pi')\rightarrow M_\alpha(\pi))$, alors $f$ est nul sur $\ker(M_\alpha(\pi)\twoheadrightarrow M_\alpha(\pi''))$ par continuit\'e et la derni\`ere assertion de la Proposition \ref{exactdense}. Donc $f$ se factorise en un morphisme continu $M_\alpha(\pi'')\rightarrow T_r\otimes_EE_m$ par la surjection topologique dans la Proposition \ref{exactdense}. On en d\'eduit que (\ref{T}) est en fait une suite exacte pour tout $T$. Par (\ref{foncteurm}) le r\'esultat d\'ecoule de l'exactitude des limites inductives filtrantes.\\
(ii) Cela d\'ecoule formellement du (i).
\end{proof}

\begin{rem}\label{plusieurschi}
{\rm Le (ii) de la Proposition \ref{drex} peut se g\'en\'eraliser comme suit~: si l'on a des caract\`eres localement analytiques $\chi_1,\dots,\chi_r$ de $\Gal(E_\infty/E)$ distincts modulo les caract\`eres lisses, et des $(\varphi,\Gamma)$-modules g\'en\'eralis\'es $D_{\alpha,1}(-),\dots ,D_{\alpha,r}(-)$ sur $\R$ pour $-\in \{\pi,\pi'\}$ tels que $F_{\alpha}(-)(T)\cong \bigoplus_{i=1}^r(E_\infty(\chi_i)\otimes_E\Hom_{(\varphi,\Gamma)}(D_{\alpha,i}(-),T))$, alors on a $F_{\alpha}(\pi'')(T)\cong \bigoplus_{i=1}^r(E_\infty(\chi_i)\otimes_E\Hom_{(\varphi,\Gamma)}(D_{\alpha,i}(\pi')/D_{\alpha,i}(\pi),T))$.}
\end{rem}

Rappelons que la d\'efinition du foncteur $F_\alpha$ d\'epend d'un certain nombre de choix~: choix des $(\iota_\alpha)_{\alpha\in S}$ v\'erifiant (\ref{alpha}), du caract\`ere additif non trivial $\eta$, de la suite de sous-groupe ouverts compacts $(N_m)_{m\geq 0}$ v\'erifiant les hypoth\`eses du \S~\ref{prel}, et du cocaract\`ere $\lambda_{\alpha^\vee}$.

\begin{prop}\label{choix}
Le foncteur $F_\alpha$ ne d\'epend pas des choix de $(\iota_\alpha)_{\alpha\in S}$, $\eta$ et $(N_m)_{\geq 0}$, i.e. pour un autre choix donnant un foncteur $F'_\alpha$, il existe une \'equivalence naturelle $F'_\alpha \simeq F_\alpha$.
\end{prop}
\begin{proof}
Soit $(N'_m)_{\geq 0}$ une autre suite croissante de sous-groupes ouverts compacts totalement d\'ecompos\'es de $N(\Qp)$ telle que $\cup_{m\geq 0}N_m=N(\Qp)$, $\eta(N_m)\subseteq E_m$ pour $m\in \Z_{\geq 0}$ et telle que $N(\Qp)/N^\alpha(\Qp)\simeq N_\alpha(\Qp)\buildrel\iota_\alpha\over\simeq \Qp$ induise par restriction (pour la racine simple $\alpha$ fix\'ee) $N_\alpha(\Qp)\cap N_0=N_0/N_0^\alpha\buildrel\substack{\iota_\alpha \\ \sim}\over\longrightarrow \Zp\subset \Qp$. Alors on peut d\'efinir des foncteurs $F'_{\alpha,m}$ en rempla\c cant $(N_m)_{\geq 0}$ par $(N'_m)_{\geq 0}$ dans la d\'efinition de $F_{\alpha,m}$ (notons que l'on n'a pas besoin de toute la condition (\ref{zp})). De plus, comme les deux suites $(N'_m)_{\geq 0}$ et $(N_m)_{\geq 0}$ sont croissantes et cofinales dans $N(\Qp)$, en proc\'edant comme dans la preuve du (iii) du Lemme \ref{m+1}, pour tout $m\in \Z_{\geq 0}$ il existe $m_1,m_2\gg 0$ tels que pour tout $\pi$ dans $\Rep(B(\Qp))$ on a des morphismes naturels de foncteurs $F_{\alpha,m}(\pi)\rightarrow F'_{\alpha,m_1}(\pi)$ et $F'_{\alpha,m}(\pi)\rightarrow F_{\alpha,m_2}(\pi)$ qui commutent \`a l'action de $\Gal(E_\infty/E)$. On en d\'eduit des morphismes fonctoriels en $\pi$ dans $F(\varphi,\Gamma)_\infty$~:
$$\lim_{m\rightarrow +\infty}F_{\alpha,m}(\pi)\longrightarrow \lim_{m\rightarrow +\infty}F'_{\alpha,m}(\pi)\ \ {\rm et}\ \lim_{m\rightarrow +\infty}F'_{\alpha,m}(\pi)\longrightarrow \lim_{m\rightarrow +\infty}F_{\alpha,m}(\pi)$$
dont la compos\'ee dans les deux sens est clairement l'identit\'e.

Soit $\eta':\Qp\rightarrow E_\infty$ un autre caract\`ere additif tel que $\eta'(N_m)\subseteq E_m$, alors on a $a\in \Qp^\times$ tel que $\eta'(-)=\eta(a(-))$, et quitte \`a \'echanger $\eta$ et $\eta'$ on peut supposer $\val(a)\geq 0$. Soit $t_a\=\prod_{\beta\in S\backslash\{\alpha\}}\lambda_{\beta^\vee}(a)\in T(\Qp)$ et $N'_m\=t_a^{-1}N_{m}t_a$ pour $m\geq 0$. Les sous-groupes $(N'_m)_{\geq 0}$ sont totalement d\'ecompos\'es et tels que $\cup_{m\geq 0}N'_m=N(\Qp)$, $\eta'(N'_m)\subseteq E_m$ pour $m\in \Z_{\geq 0}$ et $\iota_\alpha$ induit $N_\alpha(\Qp)\cap N'_0=N'_0/{N'_0}^\alpha\buildrel\substack{\iota_\alpha \\ \sim}\over\longrightarrow \Zp$. Par le paragraphe pr\'ec\'edent, le foncteur $F'_\alpha$ d\'efini avec $(\iota_\beta)_{\beta\in R^+}$, $\eta'$ et $(N_m)_{\geq 0}$ est le m\^eme que celui d\'efini avec $(\iota_\beta)_{\beta\in R^+}$, $\eta'$ et $(N'_m)_{\geq 0}$. Mais, comme dans la preuve du (i) du Lemme \ref{m+1}, la multiplication par $t_a^{-1}$ induit un isomorphisme~:
\begin{equation}\label{changebis}
t_a^{-1}:(\pi\otimes_EE_m)[{\mnn^\alpha}](\eta^{-1})_{N_m^\alpha}\buildrel\sim\over\longrightarrow (\pi\otimes_EE_m)[{\mnn^\alpha}]({\eta'}^{-1})_{{N'_m}^{\!\!\alpha}}
\end{equation}
dont on v\'erifie facilement qu'il commute aux actions de $\Zp\buildrel \iota_\alpha^{-1}\over\simeq N'_0/{N'_0}^\alpha=N_0/N_0^\alpha$, $\lambda_{\alpha^\vee}(\Zp\backslash\{0\})$ et $\Gal(E_\infty/E)$. On en d\'eduit $F'_\alpha \simeq F_\alpha$.

Soit $(\iota_\beta)_{\beta\in R^+}$, $\eta$, $(N_m)_{\geq 0}$ avec $\eta(N_m)\subseteq E_m$ et $N_\alpha(\Qp)\cap N_0=N_0/N_0^\alpha\buildrel\substack{\iota_\alpha \\ \sim}\over\longrightarrow \Zp$. La construction du foncteur $F_\alpha$ associ\'ee \`a ces choix ne d\'epend en fait pas des $\iota_\beta$ pour $\beta\in R^+\backslash S$. Fixons $\beta\in S\backslash\{\alpha\}$ et rempla\c cons $\iota_\beta$ par $\iota'_\beta\= a \iota_\beta$ pour un $a\in \Qp^\times$. On a donc une nouvelle application $\ell':N(\Qp)\rightarrow \Qp$ qui induit un nouveau caract\`ere $\eta'\=\eta\circ \ell':N(\Qp)\rightarrow E_\infty$. Alors par (\ref{cocar}) on voit que $\lambda_{\beta^\vee}(a)\in T(\Qp)$ commute avec $N_\alpha(\Qp)$ et v\'erifie $(\eta \circ \ell')(-)=(\eta\circ\ell)(\lambda_{\beta^\vee}(a) (-) \lambda_{\beta^\vee}(a)^{-1})$ (avec (\ref{alpha})). La multiplication par $\lambda_{\beta^\vee}(a)^{-1}$ a alors les m\^emes propri\'et\'es que celle par $t_a^{-1}$ en (\ref{changebis}). Donc changer les $\iota_\beta$ pour $\beta\ne \alpha$ n'affecte pas $F_\alpha$.

Soit enfin $(\iota'_\beta)_{\beta\in R^+}$, $\eta$, $(N'_m)_{\geq 0}$ avec $N_\alpha(\Qp)\cap N'_0=N'_0/{N'_0}^\alpha\buildrel\substack{\iota'_\alpha \\ \sim}\over\longrightarrow \Zp$ et $\iota'_\beta=\iota_\beta$ si $\beta\ne \alpha$. On a $\iota'_\alpha\= a \iota_\alpha$ pour un $a\in \Qp^\times$ ce qui implique $\iota_\alpha^{-1}(z)=\lambda_{\alpha^\vee}(a){\iota_\alpha'}^{-1}(z)\lambda_{\alpha^\vee}(a)^{-1}$ pour $z\in \Zp$ par (\ref{cocar}) et (\ref{alpha}). La multiplication par $\lambda_{\alpha^\vee}(a)^{-1}$ induit un isomorphisme~:
\begin{equation*}
\lambda_{\alpha^\vee}(a)^{-1}:(\pi\otimes_EE_m)[{\mnn^\alpha}](\eta^{-1})_{N_m^\alpha}\buildrel\sim\over\longrightarrow (\pi\otimes_EE_m)[{\mnn^\alpha}]({\eta}^{-1})_{{\lambda_{\alpha^\vee}(a)^{-1}N_m}^{\!\!\alpha}\lambda_{\alpha^\vee}(a)}
\end{equation*}
qui commute aux actions de $\Zp$ (agissant via $\iota_\alpha^{-1}$ \`a gauche et via ${\iota'_\alpha}^{-1}$ \`a droite), $\lambda_{\alpha^\vee}(\Zp\backslash\{0\})$ et $\Gal(E_\infty/E)$. On en d\'eduit que $F_\alpha$ est isomorphe au foncteur obtenu avec $(\iota'_\beta)_{\beta\in R^+}$, $\eta$, $(\lambda_{\alpha^\vee}(a)^{-1}N_m\lambda_{\alpha^\vee}(a))_{\geq 0}$ qui lui m\^eme, par le premier paragraphe, est isomorphe au foncteur obtenu avec $(\iota'_\beta)_{\beta\in R^+}$, $\eta$, $(N'_m)_{\geq 0}$. Cela ach\`eve la preuve.
\end{proof}

\begin{rem}\label{change}
{\rm Le foncteur $F_\alpha$ d\'epend du choix de $\lambda_{\alpha^\vee}$, mais si l'on remplace $\lambda_{\alpha^\vee}$ par $\lambda_{\alpha^\vee}+\mu$ pour $\mu\in X^\vee(Z_G)$, alors lorsque $Z_G(\Qp)$ agit sur $\pi\in \Rep(B(\Qp))$ par un caract\`ere $\chi_\pi:Z_G(\Qp)\rightarrow E^\times$, on change $M_\alpha(\pi)$ en $M_\alpha(\pi)\otimes_{\R^+}\R^+((\chi_\pi\circ \mu)^{-1})$, c'est-\`a-dire que l'on multiplie l'action de $\gamma=\gamma_z\in \Gamma$ par $\chi_\pi(\mu(z))^{-1}$ et celle de $\psi$ par $\chi_\pi(\mu(p))$. Noter au passage que si l'on tord $\pi$ par un caract\`ere localement analytique $\chi:B(\Qp)\rightarrow E^\times$ qui est trivial sur $N(\Qp)\subseteq B(\Qp)$, alors on change $M_\alpha(\pi)$ en $M_\alpha(\pi)\otimes_{\R^+}\R^+(\chi\circ \lambda_{\alpha^\vee}^{-1})$ (cf. le tout d\'ebut du \S~\ref{localg} pour la notation).}
\end{rem}

\section{Foncteurs $F_\alpha$ et th\'eorie d'Orlik-Strauch}\label{alphaqques}

On \'etudie les foncteurs $F_\alpha$ appliqu\'es aux repr\'esentations localement analytiques ${\mathcal F}_{P^-}^G(M,\pi_P^\infty)$ d\'efinies par Orlik et Strauch (\cite{OS}).

\subsection{Quelques notations}\label{notabene}

On commence par quelques rappels et notations.

On conserve les notations des \S\S~\ref{prel}, \ref{gen} et \ref{def}. On note $\mg$ (resp. $\mb$, $\mnn$, $\mt$, $\mb^-$, $\mnn^-$) la $\Qp$-alg\`ebre de Lie de $G(\Qp)$ (resp. $B(\Qp)$, $N(\Qp)$, $T(\Qp)$, $B^-(\Qp)$, $N^-(\Qp)$) et $U(-)$ les alg\`ebres enveloppantes associ\'ees. Afin de pouvoir utiliser les r\'esultats de \cite{OS}, on suppose \cite[As.~5.1]{OS} (comme dans {\it loc.cit.}), i.e. $p>2$ si le syst\`eme de racine associ\'e \`a $(\mg,\mt)$ a des composantes irr\'eductibles de type $B$, $C$, $F_4$ et $p>3$ s'il en a de type $G_2$ (en particulier il n'y a pas d'hypoth\`ese sur $p$ si $G=\G$). Suivant \cite[\S~2]{OS}, on note ${\mathcal O}^{\mb^-}_{\rm alg}$ la sous-cat\'egorie pleine de la cat\'egorie $\mathcal O={\mathcal O}^{\mb^-}$ (cf. \cite[\S~1.1]{Hu}) des $U(\mg)$-modules dont les poids (au sens caract\`eres de $\mt$) sont alg\'ebriques. Si $P\subseteq G$ est un sous-groupe parabolique contenant $B$, on dispose \'egalement de la sous-cat\'egorie pleine ${\mathcal O}^{\mpp^-}_{\rm alg}\subseteq {\mathcal O}^{\mb^-}_{\rm alg}$ o\`u $\mpp^-$ est la $\Qp$-alg\`ebre de Lie de $P^-(\Qp)$ (cf. \cite[\S~2]{OS}). On note $L_P$ le sous-groupe de Levi de $P$, $N_P$ son radical unipotent (donc $P=N_P\rtimes L_P$) et $N_{L_P}\=L_P\cap N$. Pour $w\in W$ et $\lambda\in X(T)$, on note $w\cdot \lambda$ la ``dot-action'' {\it par rapport \`a $B^-$}, i.e. $w\cdot \lambda = w(\lambda-\rho)+\rho\in X(T)$.

Pour $\lambda\in X(T)$, on note $L^-(\lambda)\in {\mathcal O}^{\mb^-}_{\rm alg}$ le $U(\mg)$-module simple sur $E$ de plus haut poids $\lambda$ par rapport \`a $B^-$, c'est-\`a-dire l'unique quotient simple du module de Verma $U(\mg)\otimes_{U(\mb^-)}\lambda$ o\`u $\lambda:U(\mt)\rightarrow E$ est vu comme caract\`ere de $U(\mb^-)$ via $U(\mb^-)\twoheadrightarrow U(\mt)$. Il est dans ${\mathcal O}^{\mpp^-}_{\rm alg}$ si et seulement si $\langle \lambda, \beta^\vee\rangle \leq 0$ pour toute racine simple $\beta\in S$ de $L_P$ et c'est alors aussi l'unique quotient simple du module de Verma g\'en\'eralis\'e $U(\mg)\otimes_{U(\mpp^-)}L^-(\lambda)_P$ o\`u $L^-(\lambda)_P$ est la repr\'esentation alg\'ebrique irr\'eductible de dimension finie de $L_P(\Qp)$ sur $E$ de plus haut poids $\lambda$ (par rapport \`a $B^- \cap L_P$). Rappelons que toute repr\'esentation alg\'ebrique irr\'eductible de $L_P(\Qp)$ sur $E$ est de la forme $L^-(\lambda)_P$ pour $\lambda\in X(T)$ tel que $\langle \lambda, \beta^\vee\rangle \leq 0$ pour toute racine simple $\beta\in S$ de $L_P$, et que $L^-(\lambda)$ (pour $\lambda\in X(T)$) est de dimension finie sur $E$ si et seulement si $\langle \lambda, \beta^\vee\rangle \leq 0$ pour tout $\beta\in S$ i.e. $\lambda$ est dominant par rapport \`a $B^-$. 

Soit $\pi_P^\infty$ une repr\'esentation lisse de longueur finie de $L_P(\Qp)$ sur $E$, Orlik et Strauch construisent dans \cite{OS} un foncteur contravariant exact $M\longmapsto {\mathcal F}_{P^-}^G(M,\pi_P^\infty)$ de la cat\'egorie ${\mathcal O}^{\mpp^-}_{\rm alg}$ dans la cat\'egorie $\Rep(G(\Qp))$ (en fait dans la cat\'egorie des repr\'esentations localement analytiques admissibles de longueur finie de $G(\Qp)$ sur $E$). Pour $L^-(\lambda)_P$ repr\'esentation alg\'ebrique irr\'eductible de $L_P(\Qp)$ sur $E$, ce foncteur envoie le module de Verma g\'en\'eralis\'e $U(\mg)\otimes_{U(\mpp^-)}L^-(\lambda)_P$ vers l'induite parabolique localement analytique~:
\begin{multline*}
\big(\Ind_{P^-(\Qp)}^{G(\Qp)} L^-(\lambda)_P^\vee\otimes_E\pi_P^\infty\big)^{\an}\=\{f:G(\Qp)\rightarrow E\ {\rm loc.\ an.\ tel\ que}\ f(p^-g)=p^-(f(g))\\
{\rm pour\ tous}\ p^-\in P^-(\Qp),g\in G(\Qp)\}
\end{multline*}
o\`u $L^-(\lambda)_P^\vee$ est le dual de la repr\'esentation de dimension finie $L^-(\lambda)_P$ (l'action de $G(\Qp)$ \'etant donn\'ee comme d'habitude par $(g'f)(g)=f(gg')$). Noter que $L^-(\lambda)_P^\vee\simeq L(-\lambda)_P$ o\`u $L(-\lambda)_P$ est la repr\'esentation alg\'ebrique irr\'eductible de $L_P(\Qp)$ de plus haut poids $-\lambda$ par rapport \`a $B\cap L_P$. Si $\lambda$ est dominant par rapport \`a $B^-$, alors on a par \cite[Prop.~4.9(b)]{OS}~:
\begin{equation}\label{induction}
{\mathcal F}_{P^-}^G(L^-(\lambda),\pi_P^\infty)=L(-\lambda)\otimes_E \big(\Ind_{P^-(\Qp)}^{G(\Qp)}\pi_P^\infty\big)^\infty
\end{equation}
o\`u $L(-\lambda)=L(-\lambda)_G=L^-(\lambda)^\vee$ et $(\Ind_{P^-(\Qp)}^{G(\Qp)}\pi_P^\infty\big)^\infty$ est l'induite parabolique lisse. De plus pour un tel $\lambda$ on d\'eduit de \cite[Prop.~II.2.11]{Ja} un isomorphisme $L(-\lambda)[\mnn_{P}]\simeq L(-\lambda)_{P}$.
 
On fixe maintenant jusqu'\`a la fin du \S~\ref{alphaqques} une racine $\alpha\in S$. L'objectif du \S~\ref{alphaqques} est d'\'etudier le foncteur $F_\alpha(\pi)$ pour $\pi={\mathcal F}_{P^-}^G(M,\pi_P^\infty)$ avec $P$, $M$, $\pi_P^\infty$ comme ci-dessus. En particulier nous d\'eterminons explicitement $F_\alpha(\pi)$ lorsque $M$ est un quotient de $U(\mg)\otimes_{U(\mpp^-)}L^-(\lambda)_P$ (pour $L^-(\lambda)_P$ repr\'esentation alg\'ebrique irr\'eductible de $L_P(\Qp)$ sur $E$) et $\pi_P^\infty$ a un caract\`ere central, cf. Corollaire \ref{casparticuliers} ci-dessous. La preuve va utiliser de mani\`ere essentielle deux cas particuliers : le cas localement alg\'ebrique $\pi=L^-(\lambda)\otimes_E\pi^\infty$ avec $\lambda$ dominant par rapport \`a $B^-$ et $\pi^\infty$ repr\'esentation lisse de longueur finie de $G(\Qp)$, qui sera d\'emontr\'e au \S~\ref{localg}, et le cas $\pi=(\Ind_{P^-(\Qp)}^{G(\Qp)} L^-(\lambda)_P^\vee\otimes_E\pi_P^\infty)^{\an}$ qui sera d\'emontr\'e au \S~\ref{endofproof} en utilisant les r\'esultats des \S\S~\ref{topo}, \ref{cellule}, \ref{approx}~\&~\ref{celluledevisse}.

\subsection{Lemmes topologiques}\label{topo}

On montre plusieurs lemmes topologiques (plus ou moins standard) qui seront utilis\'es dans les paragraphes suivantes.

On conserve toutes les notations pr\'ec\'edentes. Si $A$ est un $E$-espace vectoriel localement convexe muni d'une structure d'alg\`ebre s\'epar\'ement continue, et $V$, $W$ deux $E$-espaces vectoriels localement convexes tels que $V$ (resp. $W$) est muni d'une structure de $A$-module \`a droite (resp. \`a gauche) s\'epar\'ement continu, on d\'efinit $V\widehat \otimes_{A,\iota}W$ comme le quotient de $V\widehat \otimes_{E,\iota}W$ par l'adh\'erence du sous-$E$-espace vectoriel engendr\'e par les \'el\'ements de la forme $va\otimes w - v\otimes aw$, $(a,v,w)\in A\times V\times W$, muni de la topologie quotient (cf. \cite[Rem.~1.2.11]{Ko1}). Lorsque $A$, $V$ et $W$ sont de plus des espaces de Fr\'echet, rappelons que les topologies projective et injective co\"\i ncident sur $V\otimes_{E}W$ (\cite[Prop.~17.6]{Sch}) et que $V$ (resp. $W$) est un $A$-module \`a droite (resp. \`a gauche) continu (\cite[Cor.~III.5.2.1]{Bo2}). On note dans ce cas $V\widehat \otimes_{E}W$ et $V\widehat \otimes_{A}W$.

\begin{lem}\label{librehat}
Soit $A$ un $E$-espace vectoriel muni d'une structure d'alg\`ebre s\'epar\'ement continue, et $V$, $W$ deux $E$-espaces vectoriels localement convexes avec $V$ (resp. $W$) muni d'une structure de $A$-module \`a droite (resp. \`a gauche) s\'epar\'ement continu. Alors on a un isomorphisme canonique $V\widehat \otimes_{E,\iota}W\buildrel\sim\over\longrightarrow V\widehat \otimes_{A,\iota} (A\widehat \otimes_{E,\iota} W)$ (resp. $V\widehat \otimes_{E,\iota}W\buildrel\sim\over\longrightarrow (V\widehat \otimes_{E,\iota} A)\widehat \otimes_{A,\iota} W)$).
\end{lem}
\begin{proof}
On montre le premier cas seulement. Par les propri\'et\'es de la topologie injective sur le produit tensoriel (\cite[\S~17.A]{Sch}), on a deux morphismes continus $V\otimes_{E,\iota}W\rightarrow V\otimes_{E,\iota}A\otimes_{E,\iota}W$, $v\otimes w\mapsto v\otimes 1 \otimes w$ et $V\otimes_{E,\iota}A\otimes_{E,\iota}W \rightarrow V\otimes_{E,\iota}W$, $v\otimes a\otimes w\mapsto va\otimes w$. Par universalit\'e du s\'epar\'e-compl\'et\'e (\cite[Prop.~7.5]{Sch}) ils induisent deux morphismes continus $V\widehat\otimes_{E,\iota}W\rightarrow V\widehat\otimes_{E,\iota}A\widehat\otimes_{E,\iota}W$ et $V\widehat\otimes_{E,\iota}A\widehat\otimes_{E,\iota}W \rightarrow V\widehat\otimes_{E,\iota}W$ dont le deuxi\`eme se factorise clairement par $V\widehat \otimes_{A,\iota} (A\widehat \otimes_{E,\iota} W)$ et dont la compos\'ee donne l'identit\'e de $V\widehat\otimes_{E,\iota}W$. Il reste \`a montrer que la compos\'ee~:
$$V\widehat \otimes_{E,\iota} A\widehat \otimes_{E,\iota} W\longrightarrow V\widehat\otimes_{E,\iota}W\longrightarrow V\widehat\otimes_{E,\iota}A\widehat\otimes_{E,\iota}W\twoheadrightarrow V\widehat \otimes_{A,\iota} (A\widehat \otimes_{E,\iota} W)$$
est la surjection canonique $V\widehat\otimes_{E,\iota}A\widehat\otimes_{E,\iota}W\twoheadrightarrow V\widehat \otimes_{A,\iota} (A\widehat \otimes_{E,\iota} W)$, ce qui est clair.
\end{proof}

\begin{lem}\label{strictI}
Soit $0\rightarrow V'\rightarrow V\rightarrow V''\rightarrow 0$ une suite exacte stricte de $E$-espaces vectoriels localement convexes.\\
(i) Pour tout morphisme continu $W\rightarrow V''$ de $E$-espaces vectoriels localement convexes, la suite exacte $0\rightarrow V'\rightarrow V\times_{V''}W\rightarrow W\rightarrow 0$ est stricte o\`u le produit fibr\'e $V\times_{V''}W$ est muni de la topologie induite par la topologie produit sur $V\times W$.\\
(ii) Pour tout morphisme continu $V'\rightarrow W$ de $E$-espaces vectoriels localement convexes, la suite exacte $0\rightarrow W\rightarrow W\oplus_{V'}V\rightarrow V''\rightarrow 0$ est stricte o\`u la somme amalgam\'ee $W\oplus_{V'}V$ est munie de la topologie quotient de la topologie produit sur $W\times V$.
\end{lem}
\begin{proof}
On montre le (i), laissant la preuve du (ii) (qui proc\`ede d'arguments analogues) au lecteur. L'injection $V'\hookrightarrow V\times_{V''}W$ \'etant stricte (car l'injection $V'\simeq V'\times \{0\}\hookrightarrow V\times W$ l'est par \cite[Lem.~5.3(i)]{Sch}), il suffit de montrer que la surjection $V\times_{V''}W \buildrel s \over \twoheadrightarrow W$ est stricte, i.e. que la topologie sur $W$ est la topologie quotient de $V\times_{V''}W$. Soit $X\subseteq W$ un sous-ensemble tel que $s^{-1}(X)$ est ouvert dans $V\times_{V''}W$, nous allons montrer que $X$ est ouvert dans $W$. Il existe donc un ouvert $U$ de $V\times W$ tel que $s^{-1}(X)=U\cap (V\times_{V''}W)$. On peut supposer $U=U+V'$ $(=\{u+v',(u,v')\in U\times V'\})$. En effet, on a clairement une inclusion $s^{-1}(X)\subseteq (U+V')\cap (V\times_{V''}W)$, et on a aussi $(U+V')\cap (V\times_{V''}W)=(U\cap (V\times_{V''}W))+V'$ (car $V'\simeq V'\times \{0\}\subseteq V\times_{V''}W$) et $(U\cap (V\times_{V''}W))+V'=U\cap (V\times_{V''}W)$ car $U\cap (V\times_{V''}W)=s^{-1}(X)=s^{-1}(X)+V'$. Notons $\overline U$ l'image de $U$ dans $V''\times W\simeq V/V' \times W$, comme $U$ est un ouvert de $V\times W$ tel que $U=U+V'$, on voit que $\overline U$ est ouvert dans $V''\times W$ (dont la topologie est la topologie quotient de $V\times W$ par \cite[Lem.~5.3(ii)]{Sch}). Il suffit de montrer que $X=W\cap \overline U$ dans $V''\times W$. On a clairement $X\subseteq W\cap \overline U$. Un \'el\'ement $w$ de $W\cap \overline U$ se rel\`eve dans $V\times_{V''}W$ et dans $U$, la diff\'erence des relev\'es \'etant dans $V'$. Comme $V'\subseteq V\times_{V''}W$, quitte \`a modifier le relev\'e dans $V\times_{V''}W$, on voit que l'on peut relever $w$ dans $U\cap (V\times_{V''}W)=s^{-1}(X)$, donc $w\in X$ et on a bien $X=W\cap \overline U$.
\end{proof}

\begin{rem}\label{strictIf}
{\rm Lorsque les $3$ espaces $V$, $V'$, $V''$ du Lemme \ref{strictI} sont des espaces de Fr\'echet, alors $V\times_{V''}W$ (resp. $W\oplus_{V'}V$) est encore un espace de Fr\'echet~: le premier car c'est le noyau d'un morphisme continu entre espaces de Fr\'echet, le deuxi\`eme car c'est le conoyau d'une immersion ferm\'ee $V'\hookrightarrow W \oplus V$ (compos\'ee des injections $V'\hookrightarrow W \oplus V' \hookrightarrow W\oplus V$ qui sont des immersions ferm\'ees par \cite[Rem.~8.4]{Sch}, \cite[Lem.~5.3(iii)]{Sch} et \cite[Cor.~8.7]{Sch}).}
\end{rem}

\begin{lem}\label{tenseur}
Soit $f:V'\rightarrow V$ un morphisme continu de $E$-espaces vectoriels localement convexes et soit $W$ un $E$-espace vectoriel localement convexe. Si $f$ est d'image dense, alors les morphismes canoniques $f\otimes \Id : V'\otimes_{E,\pi}W\rightarrow V\otimes_{E,\pi}W$ et $f\otimes \Id : V'\otimes_{E,\iota}W\rightarrow V\otimes_{E,\iota}W$ sont aussi d'image dense.
\end{lem}
\begin{proof}
On donne la preuve pour $\otimes_{E,\pi}$, celle avec $\otimes_{E,\iota}$ \'etant analogue. Soit $\overline{{\im}(f\otimes \Id)}\subseteq V\otimes_{E,\pi}W$ l'adh\'erence de l'image de $f\otimes \Id$, il suffit de montrer que tout vecteur de $V\otimes_{E,\pi}W$ est dans $\overline{{\im}(f\otimes \Id)}$. Comme tout vecteur de $V\otimes_{E,\pi}W$ est dans $V\otimes_{E,\pi}W_f=V\otimes_{E}W_f$ pour $W_f\subseteq W$ un sous-$E$-espace vectoriel de dimension finie, il suffit de montrer $V\otimes _EW_f\subseteq \overline{{\im}(f\otimes \Id)}$. Comme l'injection $V\otimes_{E}W_f\hookrightarrow V\otimes_{E,\pi}W$ est continue, $\overline{{\im}(f\otimes \Id)}\cap (V\otimes_{E}W_f)$ est un ferm\'e de $V\otimes_{E}W_f$ qui contient l'image de $(f\otimes \Id)\vert_{V'\otimes_EW_f}:V'\otimes_EW_f\rightarrow V\otimes_EW_f$, donc qui contient son adh\'erence $\overline{{\im}((f\otimes \Id)\vert_{V'\otimes_EW_f})}$. Or cette derni\`ere image est dense par l'hypoth\`ese et parce que $W_f$ est de dimension finie. On en d\'eduit $V\otimes_EW_f\subseteq \overline{{\im}((f\otimes \Id)\vert_{V'\otimes_EW_f})}\subseteq \overline{{\im}(f\otimes \Id)}$.
\end{proof}

\begin{lem}\label{separetypecompact}
Consid\'erons un diagramme commutatif de suites exactes de $E$-espaces vectoriels localement convexes de type compact~:
\begin{equation*}
\xymatrix{0\ar[r] & V' \ar[r]\ar@{^{(}->}[d]^{f'} & V \ar[r]^g\ar@{^{(}->}[d]^{f} & V''\ar@{^{(}->}[d]^{f''} &\\
0\ar[r] & W' \ar[r] &W \ar[r] &W'' }
\end{equation*}
o\`u les injections $f'$, $f''$ et l'application $g$ sont strictes. Alors l'injection $f$ est aussi stricte.
\end{lem}
\begin{proof}
Rappelons d'abord que les injections $V'\hookrightarrow V$ et $W'\hookrightarrow W$ sont automatiquement strictes, qu'une injection d'espaces de type compact est stricte si et seulement si c'est une immersion ferm\'ee, et qu'une compos\'ee d'injections strictes est une injection stricte. On en d\'eduit en particulier un diagramme commutatif d'espaces de type compact~:
\begin{equation*}
\xymatrix{& & V/V' \ar@{^{(}->}[r]^{\overline g}\ar@{^{(}->}[d]^{\overline f} & V''\ar@{^{(}->}[d]^{f''} &\\
0\ar[r] & W'/V' \ar[r]^i &W/V' \ar[r] &W'' }
\end{equation*}
o\`u les injections $\overline g$ et $f''$ sont strictes, puis une suite exacte d'espaces de type compact~:
\begin{equation*}
\xymatrix{0\ar[r] & W'/V'\oplus V/V' \ar[r]^{\ \ \ \ \ \ i\oplus \overline f}& W/V' \ar[r]& W''/(V/V'). &}
\end{equation*}
Cette derni\`ere suite exacte implique que l'injection $i\oplus \overline f$ est une immersion ferm\'ee, d'o\`u \'egalement l'injection $\overline f:V/V'\hookrightarrow W/V'$. L'image inverse du ferm\'e $V/V'$ dans $W$ via la surjection continue $W\twoheadrightarrow W/V'$, c'est-\`a-dire $V$, est donc un ferm\'e de $W$, ce qui termine la preuve.
\end{proof}

On appelle module de Fr\'echet sur $\R^+$ (resp. $\Rr$) un $E$-espace vectoriel de Fr\'echet $M$ muni d'une structure de $\R^+$-module continu (resp. de $\Rr$-module continu). On appelle $(\psi,\Gamma)$-module de Fr\'echet sur $\R^+$ (resp. $\Rr$) un module de Fr\'echet $M$ sur $\R^+$ muni d'une action semi-lin\'eaire continue de $\Gamma$ et d'un morphisme continu $\psi:M\rightarrow M$ qui commute \`a $\Gamma$ et v\'erifie $\psi(\varphi(\lambda)v)=\lambda\psi(v)$ pour $\lambda\in \R^+$ (resp. avec $\Rr$). Un morphisme $M\rightarrow M'$ de $(\psi,\Gamma)$-modules de Fr\'echet sur $\R^+$ est un morphisme continu de $\R^+$-modules qui commute \`a $\psi$ et $\Gamma$, i.e. un \'el\'ement de $\Hom_{\psi,\Gamma}(M,M')$ (cf. apr\`es (\ref{foncteurm})).

\begin{lem}\label{begin}
Soit $M$ un $(\psi,\Gamma)$-module de Fr\'echet sur $\R^+$ et $V$ un $E$-espace vectoriel de dimension d\'enombrable muni de la topologie localement convexe la plus fine (\cite[\S~5]{Sch}). Pour tout $r\in \Q_{>p-1}$ et tout $(\varphi,\Gamma)$-module g\'en\'eralis\'e $T_r$ sur $\R^r$ on a un isomorphisme canonique et fonctoriel (en $T_r$) de $E$-espaces vectoriels~:
$$\Hom_{\psi,\Gamma}\big(M\widehat \otimes_E V^\vee,T_r)\simeq V\otimes_{E}\Hom_{\psi,\Gamma}\big(M,T_r\big).$$
\end{lem}
\begin{proof}
Les hypoth\`eses sur $V$ font qu'on peut l'\'ecrire comme une somme directe topologique $V=\oplus_{i\geq 1}Ee_i$ et on a alors $V^\vee=\displaystyle \prod_{i\geq 1}Ee_i^*$ et $M\widehat \otimes_E V^\vee=\prod_{i\geq 1}(M\otimes_E Ee_i^*)$ o\`u $(e_i^*)_{i\geq 1}$ est la base dual de $(e_i)_{i\geq 1}$.

Montrons d'abord que le morphisme canonique~:
\begin{equation}\label{topologie}
\bigoplus_{i\geq 1}\Hom_{\psi}\!\big(M\otimes_E Ee_i^*,T_r\big)\longrightarrow \Hom_{\psi}\!\Big(\prod_{i\geq 1}(M\otimes_E Ee_i^*),T_r\Big)
\end{equation}
est un isomorphisme o\`u $\Hom_{\psi}=$ homomorphismes continus de $\R^+$-modules commutant \`a $\psi$. Il est injectif, il suffit donc de montrer la surjectivit\'e, c'est-\`a-dire que tout morphisme $\R^+$-lin\'eaire continu $f:\prod_{i\geq 1}(M\otimes_E Ee_i^*)\rightarrow T_r$ commutant \`a $\psi$ et $\Gamma$ est nul sur tous les facteurs $M\otimes_E Ee_i^*$ sauf un nombre fini. Soit $T_{r,\tor}\subseteq T_r$ le sous-$\R^r$-module de torsion (stable par $\varphi$ et $\Gamma$), alors $T_r/T_{r,\tor}$ est un $(\varphi,\Gamma)$-module libre de rang fini sur $\R^r$. Montrons que la compos\'ee :
\begin{equation}\label{topologielibre}
\prod_{i\geq 1}(M\otimes_E Ee_i^*)\buildrel f \over \longrightarrow T_r \twoheadrightarrow (T_r/T_{r,\tor})\simeq (\R^r)^d
\end{equation}
est nulle en restriction \`a $\prod_{i\geq N}(M\otimes_E Ee_i^*)$ pour un entier $N\gg 0$. Il suffit de le montrer apr\`es composition \`a droite par l'injection continue $(\R^r)^d\hookrightarrow (\R^{[r,s]})^d$ pour un $s>r$ quelconque. Soit $W_0$ une boule unit\'e du $E$-espace de Banach $W=(\R^{[r,s]})^d$, l'image inverse de $W_0$ dans $\prod_{i\geq 1}(M\otimes_E Ee_i^*)$ est un sous-$\oE$-module ouvert, donc qui contient $\prod_{i\geq N}(M\otimes_E Ee_i^*)$ pour $N\gg 0$ (par les propri\'et\'es de la topologie produit). Mais comme $M\otimes_E Ee_i^*$ est un $E$-espace vectoriel, son image dans $W_0$ l'est aussi, ce qui implique qu'elle est nulle. Quitte \`a remplacer $f$ par $f\vert_{\prod_{i\geq N}(M\otimes_E Ee_i^*)}$, il suffit donc de montrer l'\'enonc\'e avec $T_{r,\tor}$ au lieu de $T_r$, i.e. on peut supposer $T_r\simeq \oplus_{i=1}^{d} {\mathcal R}^{r}_{E}/(t^{k_i})$. On a alors un isomorphisme topologique d'espaces de Fr\'echet $T_r\simeq \prod_{n\geq n(r)}T_{r,n}$ o\`u $n(r)$ est le plus petit entier $n$ tel que $p^{n(r)-1}(p-1)\geq r$ et $T_{r,n}\simeq \oplus_{i=1}^{d} {\mathcal R}^{r}_{E}/(\varphi^n(q(X)^{k_i}))$ (un $E$-espace vectoriel de dimension finie) avec $q(X)\=\varphi(X)/X\in \R^+$ (utiliser \cite[Lem.~4.2(1)]{Li} et le fait que tout isomorphisme continu entre espaces de Fr\'echet est un isomorphisme topologique). Sur $T_r\simeq \prod_{n\geq n(r)}T_{r,n}$ on v\'erifie que $\psi\vert_{T_{r,n}}:T_{r,n}\rightarrow T_{r,n-1}$ pour $n\geq n(r)+1$ et que $\ker(\psi^j\vert_{\prod_{n\geq n(r)+j}T_{r,n}})$ pour tout entier $j\geq 1$ ne peut contenir de sous-$\R^+$-module non nul (si $v\in \prod_{n\geq n(r)+j}T_{r,n}\subset T_r$ est non nul, il existe toujours un $i\in \{0,\dots,p^j-1\}$ tel que $\psi^j((1+X)^iv)\ne 0$). Maintenant supposons $f\vert_{M\otimes_E Ee_i^*}\ne 0$ pour des $i$ arbitrairement grands, alors il existe $n_i\geq n(r)$ tel que la compos\'ee $M\otimes_E Ee_i^*\buildrel f \over \rightarrow T_r\twoheadrightarrow T_{r,n_i}$ est non nulle, donc aussi par ce qui pr\'ec\`ede la compos\'ee~:
$$M\otimes_E Ee_i^*\buildrel f \over \longrightarrow T_r\twoheadrightarrow T_{r,n_i}\buildrel \psi^{n_i-n(r)}\over\longrightarrow T_{r,n(r)}$$
(puisque l'image de $f$ dans $T_{r,n_i}$ est un sous-$\R^+$-module non nul). Or, comme $\psi^{n_i-n(r)}\circ f=f\circ \psi^{n_i-n(r)}$, on en d\'eduit qu'{\it a fortiori} la compos\'ee $M\otimes_E Ee_i^*\buildrel f \over \rightarrow T_r\twoheadrightarrow T_{r,n(r)}$ est aussi non nulle pour des $i$ arbitrairement grands. Mais l'application compos\'ee~:
$$\prod_{i\geq 1}(M\otimes_E Ee_i^*)\buildrel f \over \longrightarrow T_r\twoheadrightarrow T_{r,n(r)}$$
\'etant continue, elle doit \^etre nulle en restriction \`a $\prod_{i\geq N}(M\otimes_E Ee_i^*)$ pour un $N\gg 0$ car $T_{r,n(r)}$ est de dimension finie, ce qui contredit la phrase d'avant. Donc n\'ecessairement $f\vert_{\prod_{i\geq N}(M\otimes_E Ee_i^*)}=0$ pour un $N\gg 0$. 

On a donc par (\ref{topologie}) et en tenant compte de l'action de $\Gamma$~:
\begin{eqnarray*}
\Hom_{\psi,\Gamma}\big(M\widehat \otimes_E V^\vee,T_r\big)&\simeq &\bigoplus_{i\geq 1}\Hom_{\psi,\Gamma}\big(M\otimes_E Ee_i^*,T_r\big)\\
&\simeq &\bigoplus_{i\geq 1}\big(\Hom_{E}(E e_i^*,E)\otimes_{E}\Hom_{\psi,\Gamma}(M,T_r)\big)\\
&\simeq &\Big(\bigoplus_{i\geq 1}\Hom_{E}(E e_i^*,E)\Big)\otimes_{E}\Hom_{\psi,\Gamma}(M,T_r)\\
&\simeq &V\otimes_{E}\Hom_{\psi,\Gamma}(M,T_r)
\end{eqnarray*}
ce qui termine la preuve du lemme.
\end{proof}

Soit $M$ une vari\'et\'e localement $\Qp$-analytique (strictement) paracompacte de dimension finie et $V$ un $E$-espace vectoriel localement convexe s\'epar\'e. Rappelons que le support d'une fonction dans $C^{\an}(M,V)$ ou $C^\infty(M,V)$ est l'adh\'erence dans $M$ de son lieu de non nullit\'e. On note $C^{\an}_c(M,V)$ (resp. $C^{\infty}_c(M,V)$) le sous-espace de $C^{\an}(M,V)$ (resp. $C^{\infty}(M,V)$) des fonctions \`a support compact.

\begin{lem}\label{supportcompact}
Supposons que $M$ admette un recouvrement disjoint par un nombre d\'enombrable d'ouverts compacts et que $V$ soit un espace de type compact, alors $C^{\an}_c(M,V)$ est naturellement muni d'une topologie localement convexe qui en fait un espace de type compact. De plus l'injection naturelle $C^{\an}_c(M,V)\hookrightarrow C^{\an}(M,V)$ est continue.
\end{lem}
\begin{proof}
On a par d\'efinition $C^{\an}_c(M,V)=\ilim{U}C^{\an}(U,V)$ o\`u la limite inductive est prise sur les ouverts compacts $U$ de $M$ ordonn\'es par l'inclusion, et on munit $C^{\an}_c(M,V)$ de la topologie limite inductive. Il s'agit de montrer que cela en fait un espace de type compact. Soit $(U_i)_{i\in \Z_{\geq 0}}$ un recouvrement d\'enombrable disjoint de $M$ par des ouverts compacts, alors $W_n\=U_1\amalg U_2\amalg \cdots \amalg U_n$ est une suite croissante d'ouverts compacts de $M$ dont on v\'erifie facilement qu'elle est cofinale pour l'inclusion parmi les ouverts compacts de $M$, de sorte que l'on a un isomorphisme topologique $\ilim{n}C^{\an}(W_n,V)\buildrel\sim\over\rightarrow \ilim{U}C^{\an}(U,V)$. Mais $\ilim{n}C^{\an}(W_n,V)\cong \oplus_{i\in \geq 0}C^{\an}(U_i,V)$ avec chaque $C^{\an}(U_i,V)$ de type compact, et son dual topologique est $\prod_{i\in \geq 0}C^{\an}(U_i,V)^\vee$ (\cite[Prop.~9.10]{Sch}), qui est un espace de Fr\'echet (car produit d\'enombrable d'espaces de Fr\'echet \cite[\S~II.4.3]{Bo2}) nucl\'eaire (\cite[Prop.~19.7(i)]{Sch}), donc le dual fort d'un espace de type compact (\cite[Th.~1.3]{ST1}) qui est n\'ecessairement $\oplus_{i\in \geq 0}C^{\an}(U_i,V)$ (\cite[Prop.~9.11]{Sch} et \cite[Th.~1.1]{ST1}). La derni\`ere assertion r\'esulte de la continuit\'e de l'injection $\oplus_{i\in \geq 0}C^{\an}(U_i,V)\hookrightarrow \prod_{i\in \geq 0}C^{\an}(U_i,V)$ (\cite[Lem.~5.2(i)]{Sch}).
\end{proof}

Soit $M$ une vari\'et\'e localement $\Qp$-analytique paracompacte de dimension finie et $C\subseteq M$ un sous-ensemble arbitraire. Rappelons que l'on dispose du sous-$E$-espace vectoriel $D(M,E)_C$ de $D(M,E)$ form\'e des distributions \`a support dans $C$ (\cite[Def.~2.1]{Ko1}). En fait, dans notre cas $C$ sera toujours un sous-ensemble ferm\'e de $M$. Lorsque $C$ est ferm\'e, la preuve de \cite[Lem.~1.2.5]{Ko1} s'\'etend {\it verbatim} (il n'y est en fait pas n\'ecessaire de travailler avec un {\it groupe} analytique $p$-adique) et montre que $D(M,E)_C$ est un sous-espace ferm\'e de $D(M,E)$ et que, muni de la topologie induite, il s'identifie au dual fort de $C^{\an}(M,E)/C^{\an}(M,E)_{M\backslash C}$ o\`u $C^{\an}(M,E)_{M\backslash C}$ est le sous-$E$-espace vectoriel ferm\'e de $C^{\an}(M,E)$ (avec topologie induite) des fonctions dont le support est contenu dans l'ouvert $M\backslash C$. Si $C$ est de plus compact, alors la preuve de \cite[Lem.~1.2.5]{Ko1} (\'etendue) montre aussi que $D(M,E)_C$ s'identifie au dual fort de~:
\begin{equation}\label{germe}
C^\omega_C(M,E)\=\ilim{U}C^{\an}(U,E)
\end{equation}
o\`u, dans la limite inductive topologique de droite (qui est alors un espace de type compact) $U$ parcourt les ouverts (compacts) de $M$ contenant $C$ et les applications de transition sont les restrictions naturelles.

Si $M_1$ et $M_2$ sont deux vari\'et\'es localement $\Qp$-analytiques (paracompactes de dimension finie), rappelons que l'application naturelle $C^{\an}(M_1,E)\otimes_E C^{\an}(M_2,E)\rightarrow C^{\an}(M_1\times M_2,E)$, $f_1\otimes f_2\longmapsto ((m_1,m_2)\mapsto f_1(m_1)f_2(m_2))$ induit un isomorphisme topologique (\cite[Prop.~A.3]{ST3})~:
\begin{equation}\label{produittop}
D(M_1\times M_2,E)\buildrel \sim\over\longrightarrow D(M_1,E)\widehat\otimes_{E,\iota}D(M_2,E).
\end{equation}

\begin{lem}\label{supproduit}
Soit $M_1$ et $M_2$ deux vari\'et\'es localement $\Qp$-analytiques (paracompactes de dimension finie) et $C_i\subseteq M_i$ un sous-ensemble ferm\'e de $M_i$ pour $i\in \{1,2\}$. Alors l'isomorphisme (\ref{produittop}) induit un isomorphisme topologique~:
$$D(M_1\times M_2,E)_{C_1\times C_2}\buildrel \sim \over\longrightarrow D(M_1,E)_{C_1}\widehat\otimes_{E,\iota}D(M_2,E)_{C_2}.$$
\end{lem}
\begin{proof}
On commence par se ramener au cas compact. Si $(M_{1,i})_{i\in I}$ est un recouvrement disjoint de $M_1$ par des ouverts compacts, on a un isomorphisme topologique $C^{\an}(M_1,E)\simeq \prod_{i\in I}C^{\an}(M_{1,i},E)$ o\`u chaque $C^{\an}(M_{1,i},E)$ est un espace de type compact. Si $f=(f_i)_{i\in I}\in C^{\an}(M_1,E)$ avec $f_i\in C^{\an}(M_{1,i},E)$, le support de $f$ est l'union (disjointe) des supports des $f_i$. En effet, le compl\'ementaire dans $M_1$ de l'union disjointe de ces supports est ouvert (car union des compl\'ementaires dans chaque $M_{1,i}$ qui est une union d'ouverts de $M_1$), donc cette union disjointe est ferm\'ee. Notant $C_{1,i}\=C_1\cap M_{1,i}$ on en d\'eduit un isomorphisme topologique $C^{\an}(M_1,E)_{M_1\backslash C_1}\simeq \prod_{i\in I}C^{\an}(M_{1,i},E)_{M_{1,i}\backslash C_{1,i}}$ puis un deuxi\`eme isomorphisme~:
$$C^{\an}(M_1,E)/C^{\an}(M_1,E)_{M_1\backslash C_1}\simeq \prod_{i\in I}\big(C^{\an}(M_{1,i},E)/C^{\an}(M_{1,i},E)_{M_{1,i}\backslash C_{1,i}}\big).$$
Avec \cite[Prop.~9.11]{Sch} et la discussion pr\'ec\'edant ce lemme, on en d\'eduit~:
\begin{equation}\label{iso1}
D(M_1,E)_{C_1}\simeq \bigoplus_{i\in I}D(M_{1,i},E)_{C_{1,i}}.
\end{equation}
Soit $(M_{2,j})_{j\in J}$ un recouvrement disjoint de $M_2$ par des ouverts compacts, on a un isomorphisme analogue \`a (\ref{iso1}) avec $M_2$, $C_2$, $(M_{2,j})_{j\in J}$ et $(C_{2,j}\=C_2\cap M_{2,j})_{j\in J}$. On en d\'eduit~:
\begin{eqnarray}\label{iso2}
D(M_1,E)_{C_1}\widehat\otimes_{E,\iota}D(M_2,E)_{C_2}\!&\!\!\simeq \!\!&\!\big(\!\bigoplus_{i\in I}D(M_{1,i},E)_{C_{1,i}}\big)\widehat\otimes_{E,\iota} \big(\!\bigoplus_{j\in J}D(M_{2,j},E)_{C_{2,j}}\big)\\
\nonumber \!&\!\!\simeq \!\!&\! \bigoplus_{(i,j)\in I\times J}D(M_{1,i},E)_{C_{1,i}}\widehat\otimes_{E,\iota} D(M_{2,j},E)_{C_{2,j}}
\end{eqnarray}
o\`u le deuxi\`eme isomorphisme r\'esulte de l'argument \`a la fin de la preuve de \cite[Prop.~A.3]{ST3}. Par ailleurs, le m\^eme argument que pour d\'emontrer (\ref{iso1}) mais avec $M_1\times M_2$ au lieu de $M_1$ et $C_1\times C_2$ au lieu de $C_1$ donne un isomorphisme topologique~:
\begin{eqnarray}\label{iso3}
D(M_1\times M_2,E)_{C_1\times C_2}&\simeq& \bigoplus_{(i,j)\in I\times J}D(M_{1,i}\times M_{2,j},E)_{C_{1,i}\times C_{2,j}}.
\end{eqnarray}
Comparant (\ref{iso2}) et (\ref{iso3}), on voit que l'on est ramen\'e au cas o\`u $M_1$, $M_2$, $C_1$ et $C_2$ sont compacts.

En utilisant la compacit\'e de $C_1\times C_2$ dans la vari\'et\'e paracompacte $M_1\times M_2$, on v\'erifie que les ouverts de la forme $U_1\times U_2$ pour $U_1$ ouvert compact de $M_1$ contenant $C_1$ et $U_2$ ouvert compact de $M_2$ contenant $C_2$ sont cofinaux dans les ouverts compacts de $M_1\times M_2$ contenant $C_1\times C_2$. Par la discussion pr\'ec\'edant ce lemme et par la version duale de \cite[Prop.~20.13]{Sch}, il suffit alors de montrer que l'on a un isomorphisme d'espaces de type compact~:
$$\ilim{U_1\times U_2}C^{\an}(U_1\times U_2,E)\simeq \big(\ilim{U_1}C^{\an}(U_1,E)\big)\widehat\otimes_E \big(\ilim{U_2}C^{\an}(U_2,E)\big)$$
o\`u $U_1$, resp. $U_2$ parcourt les ouverts compacts de $M_1$, resp. $M_2$ contenant $C_1$, resp. $C_2$. On a un isomorphisme topologique par \cite[Lem.~A.1~\&~Prop.~A.2]{ST3} ~:
\begin{equation}\label{iso4}
\ilim{U_1\times U_2}C^{\an}(U_1\times U_2,E)\simeq \ilim{U_1\times U_2}\big(C^{\an}(U_1,E)\widehat\otimes_E C^{\an}(U_2,E)\big),
\end{equation}
il suffit donc de montrer que l'on a un isomorphisme topologique~:
$$\ilim{U_1\times U_2}\big(C^{\an}(U_1,E)\widehat\otimes_E C^{\an}(U_2,E)\big)\simeq \big(\ilim{U_1}C^{\an}(U_1,E)\big)\widehat\otimes_E \big(\ilim{U_2}C^{\an}(U_2,E)\big).$$
Mais cela d\'ecoule de \cite[Prop.~1.2(2)]{Ko2} et de la preuve de \cite[Prop.~1.2(3)]{Ko2}. Plus pr\'ecis\'ement, le point clef dans la preuve de \cite[Prop.~1.2(3)]{Ko2} est que $\ilim{} (V_i\widehat \otimes_{K,\iota} W_i)$ (avec les notations de {\it loc.cit.}) est complet, ce qui est automatique ici puisqu'on l'applique \`a $\ilim{}(C^{\an}(U_1,E)\widehat\otimes_E C^{\an}(U_2,E))$ qui est complet car de type compact par (\ref{iso4}) (car $\ilim{}C^{\an}(U_1\times U_2,E)$ l'est). Cela ach\`eve la preuve du lemme.
\end{proof}

\subsection{Le cas localement alg\'ebrique}\label{localg}

En utilisant des r\'esultats de Bernstein-Zelevinsky (\cite[\S~3.5]{BZ}), on explicite le foncteur $F_\alpha$ appliqu\'e \`a une repr\'esentation localement alg\'ebrique de $G(\Qp)$.

Si $\delta:\Qp^\times\rightarrow E^\times$ est un caract\`ere continu et $r\in \Q_{>p-1}$, on note $\R^+(\delta)=\R^+ e$ (resp. $\Rr(\delta)=\Rr e$) avec $\varphi(e)= \delta(p)e$ et $\gamma(e)=\delta(\varepsilon(\gamma))e$ (cf. \cite[Not.~6.2.2]{KPX}), et $\R(\delta)$ l'extension des scalaires \`a $\R$. Si $\lambda\in X(T)$, on note~:
\begin{equation}\label{chilambda}
\chi_\lambda:\Gal(E_\infty/E)\longrightarrow E^\times, g\longmapsto \lambda(t_g)
\end{equation}
o\`u $t_g\in T(\Qp)$ est comme dans la preuve du (i) du Lemme \ref{m+1} (en fait $\chi_\lambda$ est une puissance enti\`ere du caract\`ere cyclotomique restreint \`a $\Gal(E_\infty/E)$). L'objectif de ce paragraphe est de montrer le th\'eor\`eme suivant.

\begin{thm}\label{caslisse}
Soit $\pi=L(-\lambda)\otimes_E\pi^\infty$ o\`u $\lambda\in X(T)$ est dominant par rapport \`a $B^-$ et $\pi^\infty$ est une repr\'esentation lisse de longueur finie de $G(\Qp)$ sur $E$, et soit $d_{\pi^\infty}\=\dim_{E_\infty}(\pi^\infty\otimes_EE_\infty)(\eta^{-1})_{N(\Qp)}$. Alors on a un isomorphisme dans $F(\varphi,\Gamma)_\infty$~:
\begin{equation*}
F_\alpha(\pi)\cong E_\infty(\chi_{-\lambda})\otimes_{E}\Hom_{(\varphi,\Gamma)}\Big(\big(\R(\lambda\circ \lambda_{\alpha^{\!\vee}})/(t^{1-\langle \lambda,\alpha^\vee\rangle})\big)^{\oplus d_{\pi^\infty}},-\Big).
\end{equation*}
\end{thm}

Avant de d\'emontrer le Th\'eor\`eme \ref{caslisse}, on a besoin de plusieurs pr\'eliminaires.

On note $P_2$ le sous-groupe de $G(\Qp)$ engendr\'e par $\lambda_{\alpha^\vee}(\Qp^\times)$ et $N_\alpha(\Qp)$. Il est isomorphe au sous-groupe $\smat{* & *\\ 0 & 1}$ de ${\rm GL}_2(\Qp)$ (la notation $P_2$ est celle de \cite{BZ}). Pour $m\geq 0$ on note $P_{2,m}$ le sous-groupe ouvert compact de $P_2$ engendr\'e par $\lambda_{\alpha^\vee}(\Zp^\times)$ et $N_{\alpha,m}\=N_\alpha(\Qp)\cap N_m$, que l'on peut identifier \`a $\smat{\Zp^\times & \frac{1}{p^m}\Zp\\ 0 & 1}$ dans $P_2$. On note $\widehat {\overline E}$ le compl\'et\'e $p$-adique d'une cl\^oture alg\'ebrique $\overline E$ de $E$ contenant $E_\infty$ (le corps $\widehat {\overline E}$ n'est pas sph\'eriquement complet, mais cela ne sera pas un probl\`eme dans ce paragraphe). On pose $W^\infty\=\{f:\Qp^\times\rightarrow \widehat {\overline E},\ f\ {\rm \grave a\ support\ compact}\}$ muni de l'action lisse de $P_2$ d\'efinie par $(\smat{x&0\\0&1}f)(z)=f(zx)$ et $\big(\smat{1 & y\\ 0 & 1}f\big)(z)=\eta(yz)f(z)$ pour $(x,y)\in \Qp^\times\times \Qp$ et $z\in \Qp^\times$. Appelons $E$-structure d'une repr\'esentation lisse $\tau$ d'un groupe topologique sur $\widehat {\overline E}$ tout sous-$E$-espace vectoriel $\tau_E$ de $\tau$ stable par l'action du groupe tel que $\tau_E\otimes_E\widehat {\overline E}\buildrel\sim\over\rightarrow \tau$. Si $\pi$ est une repr\'esentation lisse d'un groupe topologique $H$ sur un corps $K$ et $H'\subseteq H$ est un sous-ensemble, on note $\pi(H')\subseteq \pi$ le sous-$K$-espace vectoriel engendr\'e par $h'v-v$ pour $h'\in H'$ et $v\in \pi$.

\begin{lem}\label{structure}
La repr\'esentation $W^\infty$ est irr\'eductible et admet une unique $E$-structure $W_E^\infty$ \`a multiplication pr\`es par un scalaire dans $\widehat {\overline E}^\times$.
\end{lem}
\begin{proof}
Toute repr\'esentation lisse irr\'eductible g\'en\'erique $\tau_E$ de ${\rm GL}_2(\Qp)$ sur $E$ est telle que~:
$$\ker\big(\tau_E\twoheadrightarrow (\tau_E)_{\smat{1&\Qp\\0&1}}\big)\vert_{P_2}\otimes_E\widehat {\overline E}\simeq W^\infty$$
par \cite[\S~3.5]{BZ}, de sorte que $\ker(\tau_E\twoheadrightarrow (\tau_E)_{\smat{1&\Qp\\0&1}})$ est une $E$-structure de $W^\infty$. On v\'erifie facilement par ailleurs que $W^\infty(N_{\alpha,1})^{P_{2,0}}=\widehat {\overline E}{\mathds 1}_{\Zp^\times}$. Si $\tau_E$ est une $E$-structure sur $W^\infty$, par irr\'eductibilit\'e $\tau_E$ est engendr\'ee dans $W^\infty$ (sur $E$ et sous l'action de $P_2$) par l'un quelconque de ses vecteurs non nuls, par exemple un vecteur de base du $E$-espace de dimension un $\tau_E(N_{\alpha,1})^{P_{2,0}}$. Comme $\tau_E(N_{\alpha,1})^{P_{2,0}}\otimes_E\widehat {\overline E}\buildrel\sim\over\rightarrow W^\infty(N_{\alpha,1})^{P_{2,0}}$, tous ces vecteurs de base pour toutes les $E$-structures $\tau_E$ sont des multiples les uns des autres par des scalaires dans $\widehat {\overline E}^\times$, d'o\`u le r\'esultat. 
\end{proof}

Notons $P_2^+$ le sous-mono\"\i de de $P_2$ engendr\'e par $\lambda_{\alpha^\vee}(p)$ et $P_{2,0}$, ou de mani\`ere \'equivalente par $\lambda_{\alpha^\vee}(\Zp\backslash \{0\})$ et $N_{\alpha,0}$, que l'on peut identifier \`a $\smat{\Zp\backslash\{0\} & \Zp\\ 0 & 1}$ dans $P_2$. Soit $V_E^\infty\=W_E^\infty(N_{\alpha,1})\subseteq W_E^\infty$ et $V^\infty\=V_E^\infty\otimes_E\widehat {\overline E}$, en \'ecrivant $W^\infty=\oplus_{m\in \Z}C^\infty(p^m\Zp^\times,\widehat {\overline E})$, on v\'erifie que $V^\infty=\oplus_{m\in \Z_{\geq 0}}C^\infty(\frac{1}{p^m}\Zp^\times,\widehat {\overline E})$. De plus $V_E^\infty$ est une sous-$P_2^+$-repr\'esentation de $W_E^\infty$ telle que $\dim_E{V_E^\infty}^{P_{2,0}}=1$ avec ${V_E^\infty}^{P_{2,0}}\otimes_E\widehat {\overline E}={V^\infty}^{P_{2,0}}=\widehat {\overline E}{\mathds 1}_{\Zp^\times}\subseteq W^\infty$ (cf. la preuve du Lemme \ref{structure}). Enfin $V_E^\infty$ est engendr\'e sous l'action (it\'er\'ee) de $\lambda_{\alpha^\vee}(p)=\smat{p&0\\0&1}$ par ${V_E^\infty}^{N_{\alpha,0}}$ car c'est le cas de $V^\infty$ (puisque ${V^\infty}^{N_{\alpha,0}}=C^\infty(\Zp^\times,\widehat {\overline E})$) et car ${V_E^\infty}^{N_{\alpha,0}}\otimes_E\widehat {\overline E}\simeq {V^\infty}^{N_{\alpha,0}}$ (en effet, si $\widetilde V_E^\infty$ d\'esigne le sous-espace de $V_E^\infty$ engendr\'ee par ${V_E^\infty}^{N_{\alpha,0}}$ sous $\lambda_{\alpha^\vee}(p)$, l'inclusion $\widetilde V_E^\infty\subseteq V_E^\infty$ devient un isomorphisme apr\`es extension des scalaires \`a $\widehat {\overline E}$, donc est d\'ej\`a un isomorphisme (sinon, consid\'erer son conoyau)). On note $U_E^\infty$ la sous-$P_2^+$-repr\'esentation de $V_E^\infty$ engendr\'ee par ${V_E^\infty}^{P_{2,0}}$ et $U^\infty\=U_E^\infty\otimes_E\widehat {\overline E}$.

\begin{lem}\label{lambdanilp}
Soit $N_E^\infty$ une repr\'esentation de $P_2^+$ sur un $E$-espace vectoriel de dimension d\'enombrable telle que l'action de $P_{2,0}$ est lisse et celle de $\lambda_{\alpha^\vee}(p)$ nilpotente. Alors la repr\'esentation $N_E^\infty$ est localement finie, i.e. on a $N_E^\infty=\ilim{n\in \Z_{\geq 0}}N_n^\infty$ pour des sous-$P_2^+$-repr\'esentations $N_n^\infty$ croissantes de dimension finie sur $E$.
\end{lem}
\begin{proof}
Il suffit de montrer que pour tout $v \in N_E^\infty$ (non nul) le $E$-espace vectoriel $\langle P_2^+ v\rangle$ engendr\'e par $v$ sous $P_2^+$ est de dimension finie. Quitte \`a agrandir $E$, on se ram\`ene facilement au cas o\`u $\smat{a & 0 \\ 0 & 1} v=\chi(a) v$ pour tout $a\in \Zp^\times$ et un caract\`ere lisse $\chi:\Zp^\times\rightarrow E$. Soit $M \in \Z_{>0}$ tel que $\smat{p^M & 0 \\ 0 & 1}v=0$ et $\smat{1 & p^M \Zp \\ 0 & 1}v=v$, on a alors~:
\begin{equation}\label{calculmatrice}
\langle P_2^+ v\rangle = \sum_{m=0}^{M-1} \sum_{i=0}^{p^{m+M}-1} E\smat{p^m & i \\ 0 & 1} v
\end{equation}
qui est en particulier de dimension finie. En effet, on a $\smat{p^N & b\\ 0 & 1}=\smat{1 & b \\ 0 & 1}\smat{p^N & 0 \\ 0 & 1}$ ($b\in \Zp$), d'o\`u $\smat{p^N & b \\ 0 & 1}v=0$ si $N \geq M$. Si $0\leq N < M$ et $b=i+p^{M+N} x$ avec $0\leq i\leq p^{M+N}-1$ et $x\in \Zp$, on a $\smat{p^N & b \\ 0 & 1}=\smat{p^N & i \\ 0 & 1}\smat{1 & p^{M} x \\ 0 & 1}$, d'o\`u $\smat{p^N & b \\ 0 & 1}v=\smat{p^N & i \\ 0 & 1}v$. On en d\'eduit (\ref{calculmatrice}).
\end{proof}

En utilisant l'action explicite de $P_2^+$ sur $V^\infty$, il est facile de v\'erifier que $\lambda_{\alpha^\vee}(p)$ est nilpotent sur la $P_2^+$-repr\'esentation $V^\infty/U^\infty$, donc la $P_2^+$-repr\'esentation $V_E^\infty/U_E^\infty$ v\'erifie les hypoth\`eses du Lemme \ref{lambdanilp}. Par ailleurs, en utilisant le fait que $\lambda_{\alpha^\vee}(p)$ est injectif sur $V^\infty$ et nilpotent sur $V^\infty/U^\infty$, on en d\'eduit que $V^\infty$ est ind\'ecomposable (r\'eductible) comme $P_2^+$-repr\'esentation (consid\'erer une d\'ecomposition $P_2^+$-\'equivariante $V^\infty=S_1\oplus S_2$ et remarquer que le vecteur ${\mathds 1}_{\Zp^\times}\in {V^\infty}^{P_{2,0}}$ est forc\'ement soit dans $S_1$ soit dans $S_2$ puisque $\dim_{\widehat {\overline E}}{V^\infty}^{P_{2,0}}=1$).

Soit $Y_E^\infty$ une repr\'esentation lisse de $P_2^+$ sur un $E$-espace vectoriel de dimension d\'enombrable, l'action diagonale de $P_2^+$ sur $Y_E\=L(-\lambda)_{P_\alpha}\otimes_EY_E^\infty$ (pour $\lambda$ comme dans le Th\'eor\`eme \ref{caslisse}, $P_\alpha$ comme au \S~\ref{prel} et en remarquant que $P_2\subseteq L_{P_\alpha}(\Qp)$) avec la topologie localement convexe la plus fine permet de munir le dual alg\'ebrique $Y_E^\vee\simeq L(-\lambda)_{P_\alpha}^\vee\otimes_E{Y_E^\infty}^\vee$ d'une structure de $(\psi,\Gamma)$-module sur $\R^+$ exactement comme au d\'ebut du \S~\ref{def}. On d\'efinit alors les foncteurs $F_{\alpha,m}(Y_E)$ pour $m\in \Z_{\geq 0}$ de la cat\'egorie ab\'elienne des $(\varphi,\Gamma)$-modules g\'en\'eralis\'es $T$ sur $\R$ \`a valeurs dans les $E_m$-espaces vectoriels comme en (\ref{foncteurm}), c'est-\`a-dire avec les notations de {\it loc.cit.}~:
\begin{eqnarray*}
F_{\alpha,m}(Y_E)(T) &\=&\lim_{\substack{\longrightarrow \\ (r,f_r,T_r)\in I(T)}}\!\!\Hom_{\psi,\Gamma}\big((Y_E)^\vee\otimes_EE_m,T_r\otimes_EE_m\big).
\end{eqnarray*}

\begin{lem}\label{Falphanul}
Soit $N_E^\infty$ comme dans le Lemme \ref{lambdanilp} et $N_E\=L(-\lambda)_{P_\alpha}\otimes_EN_E^\infty$. Alors on a $F_{\alpha,m}(N_E)=0$ pour $m\in \Z_{\geq 0}$. 
\end{lem}
\begin{proof}
Soit $N_n^\infty$ comme dans le Lemme \ref{lambdanilp} et $N_n\=L(-\lambda)_{P_\alpha}\otimes_EN_n^\infty$, on a $N_E^\vee\simeq \plim{n\in \Z_{\geq 0}}N_n^\vee$. De plus $(M_n\=\ker(N_E^\vee\twoheadrightarrow N_{n}^\vee))_{n\in \Z_{\geq 0}}$ est un syst\`eme de voisinages ouverts de $0$ dans $N_E^\vee$ stables par $(\psi,\Gamma)$. Pour montrer que $F_{\alpha,m}(N_E)=0$, il faut montrer que tout morphisme continu $E$-lin\'eaire $f:N_E^\vee \rightarrow T_r\otimes_EE_m$ commutant \`a $\psi$ et $\Gamma$ devient nul dans $T_{r'}\otimes_EE_m$ pour $r'\gg r$. 

Le m\^eme argument que dans la preuve du Lemme \ref{begin} (cf. l'argument pour la surjectivit\'e de (\ref{topologie})) donne que $f\vert_{M_n}=0$ pour $n\gg 0$, i.e. $f$ se factorise par le quotient $N_n^\vee$ de $N_E^\vee$. Comme $N_n^\vee$ est le dual d'une repr\'esentation localement alg\'ebrique de dimension finie de $\smat{1&\Zp\\0&1}$, on en d\'eduit facilement qu'il existe $M\gg 0$ tel que $((1+X)^{p^M}-1)^Mf=f\circ ((1+X)^{p^M}-1)^M=0$. En prenant $r'\gg r$ tel que $(1+X)^{p^M}-1$ est inversible dans ${\mathcal R}^{r'}_{E}$ (i.e. tel que les z\'eros du polyn\^ome $(1+X)^{p^M}-1$ ne sont plus dans la couronne $p^{-1/r'}\leq \norm <1$), on voit que le morphisme $f$ devient nul dans $T_{r'}\otimes_EE_m$.
\end{proof}

On note $U_E\=L(-\lambda)_{P_\alpha}\otimes_EU_E^\infty\subseteq V_E\=L(-\lambda)_{P_\alpha}\otimes_EV_E^\infty\subseteq W_E\=L(-\lambda)_{P_\alpha}\otimes_EW_E^\infty$ que l'on munit de la topologie localement convexe la plus fine.

\begin{lem}\label{backtore}
(i) On a $F_{\alpha,m}(U_E)\buildrel\sim\over\longrightarrow F_{\alpha,m}(V_E)\buildrel\sim\over\longrightarrow F_{\alpha,m}(W_E)$ pour $m\in \Z_{\geq 0}$.\\
(ii) Pour tout $(\varphi,\Gamma)$-module g\'en\'eralis\'e $T$ sur $\R$ on a pour $m\in \Z_{\geq 0}$~:
$$F_{\alpha,m}(W_E)(T)=E_m \otimes_E\Hom_{(\varphi,\Gamma)}\!\big(\R(\lambda\circ \lambda_{\alpha^{\!\vee}})/(t^{1-\langle \lambda,\alpha^\vee\rangle}),T\big).$$
\end{lem}
\begin{proof}
(i) On a des suites exactes de foncteurs $0\rightarrow F_{\alpha,m}(V_E)\rightarrow F_{\alpha,m}(W_E)\rightarrow F_{\alpha,m}(W_E/V_E)$ et $0\rightarrow F_{\alpha,m}(U_E)\rightarrow F_{\alpha,m}(V_E)\rightarrow F_{\alpha,m}(V_E/U_E)$. Il suffit donc de montrer $F_{\alpha,m}(W_E/V_E)=0$ et $F_{\alpha,m}(V_E/U_E)=0$. On a que $\smat{1&1\\0&1}-\Id$ agit par $0$ sur $W^\infty_E/V^\infty_E=(W^\infty_E)_{N_{\alpha,1}}$ et on v\'erifie facilement que $(\smat{1&1\\0&1}-\Id)^{1-\langle \lambda,\alpha^\vee\rangle}$ agit par $0$ sur la repr\'esentation alg\'ebrique $L(-\lambda)_{P_\alpha}$. Cela implique que $X^{1-\langle \lambda,\alpha^\vee\rangle}$ agit par $0$ sur $(W_E/V_E)^\vee$, et comme $X$ est inversible sur $T_r$ on en d\'eduit $F_{\alpha,m}(W_E/V_E)=0$. Enfin on a $F_{\alpha,m}(V_E/U_E)=0$ par le Lemme \ref{Falphanul}.

\noindent
(ii) Soit $\St_2=(\Ind_{B^-(\Qp)}^{{\rm GL}_2(\Qp)} 1)^{\infty}/1$ la repr\'esentation de Steinberg lisse de ${\rm GL}_2(\Qp)$ sur $E$ o\`u $B^-$ est le sous-groupe des matrices triangulaires inf\'erieures, on a un isomorphisme de $P_2$-repr\'esentation $\St_2\vert_{P_2}\simeq C^\infty_c(\Qp,E)$ o\`u l'espace de droite est muni de l'action de $P_2$ donn\'ee par ($f\in C^\infty_c(\Qp,E)$, $z,y\in \Qp$, $x\in \Qp^\times$)~:
\begin{equation}\label{actionP2+}
\big(\smat{1&y\\0&1}f\big)(z)=f(z+y) \ \ \ {\rm et}\ \ \ \big(\smat{x&0\\0&1}f\big)(z)=f\big(\frac{z}{x}\big).
\end{equation}
Par le d\'ebut de la preuve du Lemme \ref{structure} (utilisant \cite[\S~3.5]{BZ}), on a une suite exacte de repr\'esentations lisses de $P_2$ sur $E$~:
\begin{equation}\label{steinberg}
0\longrightarrow W_E^\infty \longrightarrow C^\infty_c(\Qp,E) \longrightarrow (\St_2)_{\smat{1&\Qp\\0&1}}\longrightarrow 0
\end{equation}
et par l'argument au d\'ebut de la preuve du (i) (avec $L(-\lambda)_{P_\alpha}\otimes_E(\St_2)_{\smat{1&\Qp\\0&1}}$ au lieu de $W_E/V_E$), on en d\'eduit $F_{\alpha,m}(W_E)\buildrel\sim\over\rightarrow F_{\alpha,m}(L(-\lambda)_{P_\alpha}\otimes_EC^\infty_c(\Qp,E))$.

Le sous-espace $C^\infty(\Zp,E)$ de $C^\infty_c(\Qp,E)$ est stable par $P_2^+$, on peut donc d\'efinir $F_{\alpha,m}(L(-\lambda)_{P_\alpha}\otimes_EC^\infty(\Zp,E))$ et on a une suite exacte courte~:
\begin{multline*}
0\longrightarrow F_{\alpha,m}(L(-\lambda)_{P_\alpha}\otimes_EC^\infty(\Zp,E))\longrightarrow F_{\alpha,m}(L(-\lambda)_{P_\alpha}\otimes_EC^\infty_c(\Qp,E))\\
\longrightarrow F_{\alpha,m}(L(-\lambda)_{P_\alpha}\otimes_E \pi_0^\infty)
\end{multline*}
o\`u $\pi_0^\infty\=C^\infty_c(\Qp,E)/C^\infty(\Zp,E)\simeq \oplus_{m'\geq 1}C^\infty(\frac{1}{p^{m'}}\Zp^\times,E)$. Montrons que tout morphisme $E$-lin\'eaire continu $f:L(-\lambda)_{P_\alpha}^\vee\otimes_E(\pi_0^\infty)^\vee\rightarrow T_r\otimes_EE_m$ commutant \`a $\psi$ et $\Gamma$ devient nul dans $T_{r'}\otimes_EE_m$ pour $r'\gg r$. On a $L(-\lambda)_{P_\alpha}^\vee\otimes_E(\pi_0^\infty)^\vee\simeq \prod_{n\geq 1}M_{n}$ o\`u $M_{n}\=L(-\lambda)_{P_\alpha}^\vee\otimes_EC^\infty(\frac{1}{p^{n}}\Zp^\times,E)^\vee$ et le m\^eme argument que dans la preuve du Lemme \ref{begin} (cf. la preuve du Lemme \ref{Falphanul}) donne que $f$ est nul sur tous les facteurs $M_{n}$ sauf un nombre fini, i.e. se factorise par un quotient $\prod_{n'\geq n\geq 1}M_{n}$ pour $n'\gg0$. Comme l'op\'erateur $\psi$ sur $L(-\lambda)_{P_\alpha}^\vee\otimes_E(\pi_0^\infty)^\vee$ envoie $M_{n}$ dans $M_{n+1}$, il existe $N\gg 0$ tel que $\psi^N=0$ sur ce quotient. Par l'analogue du diagramme commutatif (\ref{psi5}) o\`u $\varphi$, $\psi$ sont remplac\'es par respectivement $\varphi^{N}$, $\psi^{N}$ et $pr$ par $p^Nr$ (cf. aussi (\ref{psi4}) et (\ref{psi6}) ci-dessous), on obtient que la compos\'ee~:
$$L(-\lambda)_{P_\alpha}^\vee\otimes_E(\pi_0^\infty)^\vee\buildrel f\over \longrightarrow T_r\otimes_EE_m\longrightarrow T_{p^{N}r}\otimes_EE_m$$
est nulle, d'o\`u on d\'eduit $F_{\alpha,m}(L(-\lambda)_{P_\alpha}\otimes_E \pi_0^\infty)=0$ et donc $F_{\alpha,m}(L(-\lambda)_{P_\alpha}\otimes_EC^\infty(\Zp,E))\buildrel \sim\over\rightarrow F_{\alpha,m}(L(-\lambda)_{P_\alpha}\otimes_EC^\infty_c(\Qp,E))$.

Maintenant, $L(-\lambda)_{P_\alpha}\otimes_EC^{\infty}(\Zp,E)$ s'identifie aux fonctions localement alg\'ebriques sur $\Zp$ de degr\'e (local) au plus $-\langle \lambda,\alpha^\vee\rangle\in \Z_{\geq 0}$ avec action de $P_2^+$ induite par (\ref{actionP2+}) et tordue par le caract\`ere $\smat{x&0\\0&1}\mapsto (-\lambda)(\lambda_{\alpha^\vee}(x))$ de $\smat{\Zp\backslash\{0\}&0\\0&1}$. Par ailleurs, dans l'isomorphisme $D(\Zp,E)=C^{\an}(\Zp,E)^\vee\simeq \R^+$ du d\'ebut du \S~\ref{def}, la d\'erivation dans $C^{\an}(\Zp,E)$ induit (en dualisant) la multiplication par $t$ sur $\R^+$. On en d\'eduit que le sous-espace ferm\'e de $C^{\an}(\Zp,E)$ des fonctions localement alg\'ebriques sur $\Zp$ de degr\'e au plus $-\langle \lambda,\alpha^\vee\rangle$, i.e. le noyau de la d\'erivation $(-\langle \lambda,\alpha^\vee\rangle+1)$-i\`eme, a pour dual $\R^+/(t^{-\langle \lambda,\alpha^\vee\rangle+1})$, d'o\`u on d\'eduit un isomorphisme $L(-\lambda)_{P_\alpha}^\vee\otimes_EC^{\infty}(\Zp,E)^\vee\simeq \R^+((-\lambda)\circ \lambda_{\alpha^\vee}^{-1})/(t^{1-\langle \lambda,\alpha^\vee\rangle})\simeq \R^+(\lambda\circ\lambda_{\alpha^\vee})/(t^{1-\langle \lambda,\alpha^\vee\rangle})$ qui commute \`a $\psi$ et $\Gamma$. Avec les (iv) et (v) de la Remarque \ref{vrac} on obtient $F_{\alpha,m}(L(-\lambda)_{P_\alpha}\otimes_EC^\infty(\Zp,E))=E_m \otimes_E\Hom_{(\varphi,\Gamma)}(\R(\lambda\circ \lambda_{\alpha^\vee})/(t^{1-\langle \lambda,\alpha^\vee\rangle}),-)$, d'o\`u le r\'esultat par les deux paragraphes d'avant.
\end{proof}

On d\'emontre maintenant le Th\'eor\`eme \ref{caslisse}.

\noindent
{\bf \'Etape 1}\\
On fixe $m\in \Z_{\geq 0}$. On a des isomorphismes $L(-\lambda)[\mnn^\alpha]\simeq L(-\lambda)_{P_\alpha}$ (cf. \S~\ref{notabene}) et $\pi[\mnn^\alpha](\eta^{-1})_{N_m^\alpha}\simeq L(-\lambda)_{P_\alpha}\otimes_E \pi^\infty(\eta^{-1})_{N_m^\alpha}$. On note $\pi^\infty_\alpha$ la sous-$P_2$-repr\'esentation de $\pi^\infty\vert_{P_2}$ noyau de la surjection naturelle $\pi^\infty\twoheadrightarrow (\pi^\infty)_{N_\alpha(\Qp)}$ et on rappelle que $(\pi^\infty\otimes_EE_m)(\eta^{-1})(N_m^\alpha)$ s'identifie \`a la sous-$P_2^+$-repr\'esentation de $\pi^\infty\otimes_EE_m$ engendr\'ee par les $yv-\eta(y)v$ pour $y\in N_m^\alpha$, $v\in \pi^\infty\otimes_EE_m$. On pose~:
$$\pi^\infty_{\alpha,m}\=(\pi^\infty_\alpha\otimes_EE_m)/\big((\pi^\infty_\alpha\otimes_EE_m)\cap (\pi^\infty\otimes_EE_m)(\eta^{-1})(N_m^\alpha)\big).$$
On a donc une suite exacte de $P_2^+$-repr\'esentations~:
$$0\longrightarrow \pi^\infty_{\alpha,m}\longrightarrow (\pi^\infty\otimes_EE_m)(\eta^{-1})_{N_m^\alpha}
\longrightarrow Q_m\longrightarrow 0$$
o\`u $Q_m$ est un quotient stable par $P_2^+$ de $(\pi^\infty)_{N_\alpha(\Qp)}\otimes_EE_m$. La multiplication par $X$ est nulle sur $Q_m^\vee$ (car elle est nulle sur $((\pi^\infty)_{N_\alpha(\Qp)})^\vee\otimes_EE_m$), donc il existe $M\gg 0$ tel que la multiplication par $X^M$ sur $L(-\lambda)_{P_\alpha}^\vee\otimes_EQ_m^\vee$ est nulle. Cela implique que tout morphisme $\Rm^+$-lin\'eaire continu $L(-\lambda)_{P_\alpha}^\vee\otimes_EQ_m^\vee\rightarrow T_r\otimes_EE_m$ commutant \`a $\psi$ et $\Gamma$ (pour $T_r$ un $(\varphi,\Gamma)$-module g\'en\'eralis\'e sur $\Rr$) est nul. On en d\'eduit~:
\begin{equation}\label{falphalisse}
F_\alpha(\pi)(T)=\lim_{m\rightarrow +\infty}\lim_{\substack{\longrightarrow \\ (r,f_r,T_r)\in I(T)}}\!\!\Hom_{\psi,\Gamma}\big(L(-\lambda)_{P_\alpha}^\vee\otimes_E(\pi^\infty_{\alpha,m})^\vee,T_r\otimes_EE_{m}\big).
\end{equation}

\noindent
{\bf \'Etape 2}\\
On montre que l'on a un isomorphisme de $P_2$-repr\'esentations $\pi^\infty_\alpha\simeq \oplus_{\iota\in I}W_E^\infty$ o\`u $I$ est un ensemble d\'enombrable d'indices. La repr\'esentation $\tau\=\pi^\infty\vert_{P_2}\otimes_E\widehat {\overline E}$ de $P_2$ \'etant lisse, on peut lui appliquer les r\'esultats de \cite[\S~3.5]{BZ} pour $n=2$ (les r\'esultats de {\it loc.cit.} sont sur $\C$, mais la topologie de $\C$ n'est pas utilis\'ee et ils restent donc valables sur le corps alg\'ebriquement clos $\widehat {\overline E}$). En particulier on en d\'eduit une suite exacte $P_2$-\'equivariante~:
$$0\longrightarrow \tau(\eta^{-1})_{N_\alpha(\Qp)}\otimes_{\widehat {\overline E}}W^\infty\longrightarrow \tau\longrightarrow \tau_{N_\alpha(\Qp)}\longrightarrow 0$$
o\`u l'action de $P_2$ est triviale sur $\tau(\eta^{-1})_{N_\alpha(\Qp)}$. Comme $\pi^\infty$ est de dimension d\'enombrable sur $E$ (car $\pi^\infty$ est admissible), c'est {\it a fortiori} le cas de $\tau(\eta^{-1})_{N_\alpha(\Qp)}$ sur $\widehat {\overline E}$, et on a donc un isomorphisme $P_2$-\'equivariant $\pi^\infty_\alpha\otimes_E\widehat {\overline E}\simeq \oplus_{\iota\in I}W^\infty$ avec $I$ d\'enombrable. Pour $\iota\in I$ choisissons une base $v_\iota$ du $\widehat {\overline E}$-espace vectoriel de dimension un $W^\infty(N_{\alpha,1})^{P_{2,0}}$ (cf. la preuve du Lemme \ref{structure}, on prend ici la ``copie $\iota$'' de $W^\infty$), on a~:
\begin{equation}\label{changebase}
(\pi^\infty_\alpha\otimes_E\widehat {\overline E})(N_{\alpha,1})^{P_{2,0}}=\pi^\infty_\alpha(N_{\alpha,1})^{P_{2,0}}\otimes_E\widehat {\overline E}\simeq \oplus_{\iota\in I}\widehat {\overline E}v_\iota.
\end{equation}
Quitte \`a changer la base $v_\iota$ et la d\'ecomposition $\pi^\infty_\alpha\otimes_E\widehat {\overline E}\simeq \oplus_{\iota\in I}W^\infty$ en cons\'equence (rappelons que la $P_2$-repr\'esentation irr\'eductible $W^\infty$ est engendr\'ee par $W^\infty(N_{\alpha,1})^{P_{2,0}}$), (\ref{changebase}) montre que l'on peut supposer $v_\iota\in \pi^\infty_\alpha(N_{\alpha,1})^{P_{2,0}}$, i.e. $\pi^\infty_\alpha(N_{\alpha,1})^{P_{2,0}}=\oplus_{\iota\in I}Ev_\iota$. Mais la $P_2$-repr\'esentation $\pi^\infty_\alpha$ est engendr\'ee par $\pi^\infty_\alpha(N_{\alpha,1})^{P_{2,0}}$~: en effet, si $\widetilde\pi^\infty_\alpha$ d\'esigne la sous-$P_2$-repr\'esentation de $\pi^\infty_\alpha$ engendr\'ee par $\pi^\infty_\alpha(N_{\alpha,1})^{P_{2,0}}$ sur $E$, l'inclusion $\widetilde\pi^\infty_\alpha\subseteq \pi^\infty_\alpha$ devient un isomorphisme apr\`es extension des scalaires \`a $\widehat {\overline E}$ (utiliser (\ref{changebase})), donc est d\'ej\`a un isomorphisme. Donc $\pi^\infty_\alpha$ est engendr\'e sur $E$ sous l'action de $P_2$ par le sous-$E$-espace vectoriel $\oplus_{\iota\in I}Ev_\iota$ de $\oplus_{\iota\in I}W^\infty$, ce qui montre par le Lemme \ref{structure} (et sa preuve) que l'on a un isomorphisme de $P_2$-repr\'esentations $\pi^\infty_\alpha\simeq \oplus_{\iota\in I}W_E^\infty$.

\noindent
{\bf \'Etape 3}\\
On montre que l'on a un isomorphisme de $P_2$-repr\'esentations $\displaystyle \lim_{m\rightarrow +\infty}\pi^\infty_{\alpha,m}\simeq (W_E^\infty\otimes_EE_\infty)^{\oplus d_{\pi^\infty}}$ o\`u l'action de $P_2$ sur $\displaystyle \lim_{m\rightarrow +\infty}\pi^\infty_{\alpha,m}$ est donn\'ee par l'action sur le terme de droite dans l'isomorphisme (que l'on v\'erifie facilement)~:
\begin{equation}\label{limdpi}
\lim_{m\rightarrow +\infty}\pi^\infty_{\alpha,m}\simeq (\pi^\infty_\alpha\otimes_EE_\infty)/\big((\pi^\infty_\alpha\otimes_EE_\infty)\cap (\pi^\infty\otimes_EE_\infty)(\eta^{-1})(N^\alpha(\Qp))\big).
\end{equation}
Appliquant les r\'esultats de \cite[\S~3.5]{BZ} \`a la repr\'esentation lisse $(\pi^\infty\otimes_E\widehat {\overline E})(\eta^{-1})_{N^\alpha(\Qp)}$ de $P_2$ sur $\widehat {\overline E}$ et notant pour abr\'eger $Q\=\big((\pi^\infty\otimes_EE_\infty)(\eta^{-1})_{N^\alpha(\Qp)}\big)_{N_\alpha(\Qp)} $, on d\'eduit \`a partir des d\'efinitions un diagramme commutatif de repr\'esentations lisses de $P_2$ sur $E_\infty$ o\`u toutes les fl\`eches sont surjectives~:
\begin{equation}\small\label{troplong}
\begin{gathered}
\xymatrix{ 0\ar[r] &\pi_\alpha^\infty\otimes_EE_\infty \ar[r] \ar@{^{}->>}[d]&\pi^\infty\otimes_EE_\infty \ar[r] \ar@{^{}->>}[d]&(\pi^\infty\otimes_EE_\infty)_{N_\alpha(\Qp)} \ar@{^{}->>}[d] \ar[r]&0\\
0\ar[r]&\sigma \ar[r]&(\pi^\infty\otimes_EE_\infty)(\eta^{-1})_{N^\alpha(\Qp)} \ar[r] &Q\ar[r]&0}
\end{gathered}
\end{equation}
et o\`u $\sigma\otimes_{E_\infty}\widehat {\overline E}\simeq (W^\infty)^{d_{\pi^\infty}}$. Par des consid\'erations analogues \`a celles de l'\'Etape $2$ en rempla\c cant $E$ par $E_\infty$ et $\pi^\infty_\alpha$ par $\sigma$ on obtient un isomorphisme $\sigma\simeq (W_E^\infty\otimes_EE_\infty)^{d_{\pi^\infty}}$. On voit donc que l'image de $\pi_\alpha^\infty\otimes_EE_\infty$ dans $(\pi^\infty\otimes_EE_\infty)(\eta^{-1})_{N^\alpha(\Qp)}$, i.e. le terme de droite en (\ref{limdpi}), s'identifie \`a $(W^\infty)^{d_{\pi^\infty}}$. Avec (\ref{limdpi}), cela donne le r\'esultat voulu.

\noindent
{\bf \'Etape 4}\\
Pour $m\geq 1$ le groupe $N_{\alpha,1}$ agit sur $\pi^\infty_{\alpha,m}$, de sorte que l'on peut consid\'erer la sous-$P_2^+$-repr\'esentation $\pi^\infty_{\alpha,m}(N_{\alpha,1})\subseteq \pi^\infty_{\alpha,m}$. C'est aussi l'image de $\pi^\infty_{\alpha}(N_{\alpha,1})\otimes_EE_m$ dans le quotient $\pi^\infty_{\alpha,m}$ de $\pi^\infty_{\alpha}\otimes_EE_m$. On d\'eduit des r\'esultats de l'\'Etape $3$ que l'on a un isomorphisme $P_2^+$-\'equivariant $\displaystyle \lim_{m\rightarrow +\infty}\pi^\infty_{\alpha,m}(N_{\alpha,1})\simeq (V_E^\infty\otimes_EE_\infty)^{\oplus d_{\pi^\infty}}$ o\`u les fl\`eches de transition \`a gauche induisent des surjections $\pi^\infty_{\alpha,m}(N_{\alpha,1})\otimes_{E_m}E_{m+1}\twoheadrightarrow \pi^\infty_{\alpha,m+1}(N_{\alpha,1})$. En utilisant le r\'esultat de l'\'Etape $3$ (avec l'irr\'eductibilit\'e de la $P_2$-repr\'esentation $W_E^\infty$), on voit qu'il existe donc $m_0\gg 1$ et un plongement $P_2^+$-\'equivariant $j_{m_0}:(V_E^\infty\otimes_EE_{m_0})^{d_{\pi^\infty}}\hookrightarrow \pi^\infty_{\alpha,m_0}(N_{\alpha,1})$ tels que la compos\'ee~:
$$(V_E^\infty\otimes_EE_{m_0})^{d_{\pi^\infty}}\otimes_{E_{m_0}}E_\infty\buildrel{j_{m_0}}\over\hookrightarrow \pi^\infty_{\alpha,m_0}(N_{\alpha,1})\otimes_{E_{m_0}}E_\infty\twoheadrightarrow (V_E^\infty\otimes_EE_\infty)^{\oplus d_{\pi^\infty}}$$
est un isomorphisme.

Pour $m\geq m_0$ soit $\pi^\infty_{\alpha,m}(N_{\alpha,1})_{\rm nilp}\subseteq \pi^\infty_{\alpha,m}(N_{\alpha,1})$ le sous-$E_m$-espace vectoriel o\`u $\lambda_{\alpha^\vee}(p)$ est nilpotent. Il est stable par $P_2^+$, et donne donc une repr\'esentation de $P_2^+$ comme dans le Lemme \ref{lambdanilp} (avec $E_m$ au lieu de $E$). Par l'\'Etape $2$ on a un isomorphisme $P_2^+$-\'equivariant $\pi^\infty_\alpha(N_{\alpha,1})\simeq \oplus_{\iota\in I}V_E^\infty$. Comme $\pi^\infty_{\alpha,m}(N_{\alpha,1})$ est engendr\'e sous l'action de $\lambda_{\alpha^\vee}(p)$ par les invariants $\pi^\infty_{\alpha,m}(N_{\alpha,1})^{N_{\alpha,0}}$ (car c'est le cas de $V_E^\infty\otimes_EE_m$), le sous-$E_m$-espace vectoriel de $\pi^\infty_{\alpha,m}(N_{\alpha,1})$ engendr\'e sous l'action de $\lambda_{\alpha^\vee}(p)$ par un suppl\'ementaire stable par $\lambda_{\alpha^\vee}(\Zp^\times)$ de $\pi^\infty_{\alpha,m}(N_{\alpha,1})_{\rm nilp}^{N_{\alpha,0}}$ dans $\pi^\infty_{\alpha,m}(N_{\alpha,1})^{N_{\alpha,0}}$ fournit un suppl\'ementaire stable par $P_2^+$ de $\pi^\infty_{\alpha,m}(N_{\alpha,1})_{\rm nilp}$ dans $\pi^\infty_{\alpha,m}(N_{\alpha,1})$ dont on v\'erifie facilement qu'il est une somme directe (d\'enombrable) de repr\'esentations $V_E^\infty\otimes_EE_m$ (utiliser que l'image d'un facteur direct $V_E^\infty\otimes_EE_m$ de $\pi^\infty_\alpha(N_{\alpha,1})\otimes_EE_m$ dans $\pi^\infty_{\alpha,m}(N_{\alpha,1})/\pi^\infty_{\alpha,m}(N_{\alpha,1})_{\rm nilp}$ est soit nulle soit $V_E^\infty\otimes_EE_m$). Autrement dit on a $\pi^\infty_{\alpha,m}(N_{\alpha,1})\simeq (\pi^\infty_{\alpha,m})_{\rm nilp}\oplus (\oplus_J (V_E^\infty\otimes_EE_m))$ avec $J$ d\'enombrable. Par ailleurs, la compos\'ee~:
\begin{multline*}
(V^\infty_E\otimes_EE_{m_0})^{d_{\pi^\infty}}\otimes_{E_{m_0}}E_m\buildrel{j_{m_0}\otimes\Id}\over\hookrightarrow \pi^\infty_{\alpha,m_0}(N_{\alpha,1})\otimes_{E_{m_0}}E_m\twoheadrightarrow \pi^\infty_{\alpha,m}(N_{\alpha,1})\\
\twoheadrightarrow \pi^\infty_{\alpha,m}(N_{\alpha,1})/\pi^\infty_{\alpha,m}(N_{\alpha,1})_{\rm nilp}\simeq \oplus_J (V_E^\infty\otimes_EE_m)
\end{multline*}
est injective (car $\pi^\infty_{\alpha,m}(N_{\alpha,1})_{\rm nilp}$ s'envoie sur $0$ dans $(V_E^\infty\otimes_EE_\infty)^{\oplus d_{\pi^\infty}}$) et on v\'erifie facilement en utilisant le fait que $V_E^\infty$ est ind\'ecomposable sous l'action de $P_2^+$ que son image est n\'ecessairement un facteur direct $(V^\infty_E\otimes_EE_{m})^{d_{\pi^\infty}}$ de $\oplus_J (V_E^\infty\otimes_EE_m)$ (quitte \`a modifier cette d\'ecomposition). On peut donc finalement \'ecrire~:
\begin{equation}\label{decompopenible}
\pi^\infty_{\alpha,m}(N_{\alpha,1})\simeq \pi^\infty_{\alpha,m}(N_{\alpha,1})_{\rm nilp}\oplus (V^\infty_E\otimes_EE_{m})^{d_{\pi^\infty}}\oplus (\oplus_{J'} (V_E^\infty\otimes_EE_m))
\end{equation}
avec $J'$ d\'enombrable et la compos\'ee $(V_E^\infty\otimes_EE_{m})^{d_{\pi^\infty}}\otimes_{E_{m}}E_\infty \hookrightarrow \pi^\infty_{\alpha,m}(N_{\alpha,1})\otimes_{E_{m}}E_\infty\twoheadrightarrow (V_E^\infty\otimes_EE_\infty)^{\oplus d_{\pi^\infty}}$ bijective.

\noindent
{\bf \'Etape 5}\\
Puisque $\smat{1&1\\0&1}-\Id$ annule $\pi^\infty_{\alpha,m}/\pi^\infty_{\alpha,m}(N_{\alpha,1})$, on en d\'eduit pour tout $m\in \Z_{\geq 1}$ un isomorphisme comme dans la preuve du (i) du Lemme \ref{backtore}~:
\begin{multline}\label{bernstein3}
 \lim_{\substack{\longrightarrow \\ (r,f_r,T_r)\in I(T)}}\!\!\Hom_{\psi,\Gamma}\big(L(-\lambda)_{P_\alpha}^\vee\otimes_E\pi^\infty_{\alpha,m}(N_{\alpha,1})^\vee,T_r\otimes_EE_{m}\big)\\
\buildrel\sim\over\longrightarrow \lim_{\substack{\longrightarrow \\ (r,f_r,T_r)\in I(T)}}\!\!\Hom_{\psi,\Gamma}\big(L(-\lambda)_{P_\alpha}^\vee\otimes_E(\pi^\infty_{\alpha,m})^\vee,T_r\otimes_EE_{m}\big).
\end{multline}
Consid\'erons pour $m\geq m_0$ un morphisme $f:L(-\lambda)_{P_\alpha}^\vee\otimes_E\pi^\infty_{\alpha,m}(N_{\alpha,1})^\vee\longrightarrow T_r\otimes_EE_{m}$ continu, $\Rm^+$-lin\'eaire et commutant \`a $\psi$ et $\Gamma$. Une fois fix\'ee une d\'ecomposition comme en (\ref{decompopenible}), le morphisme $f$ induit $f_{\rm nilp}:L(-\lambda)_{P_\alpha}^\vee\otimes_E\pi^\infty_{\alpha,m}(N_{\alpha,1})_{\rm nilp}^\vee\rightarrow T_r\otimes_EE_{m}$ et $f_{\rm libre}:L(-\lambda)_{P_\alpha}^\vee\otimes_E \big((V^\infty_E\otimes_EE_{m})^{d_{\pi^\infty}}\oplus (\oplus_J' (V_E^\infty\otimes_EE_m))\big)^\vee\rightarrow T_r\otimes_EE_{m}$. Par le Lemme \ref{Falphanul} et l'argument pour la surjectivit\'e de (\ref{topologie}) dans la preuve du Lemme \ref{begin}, il existe $r'\gg r$ tel que $f_{\rm nilp}$ devient nul dans $T_{r'}\otimes_EE_{m}$ et que la compos\'ee~:
\begin{equation}\small\label{encorecompo}
L(-\lambda)_{P_\alpha}^\vee\otimes_E \big((V^\infty_E\otimes_EE_{m})^{d_{\pi^\infty}}\oplus (\oplus_{J'} (V_E^\infty\otimes_EE_m))\big)^\vee\buildrel{f_{\rm libre}}\over\longrightarrow T_r\otimes_EE_{m}\longrightarrow T_{r'}\otimes_EE_{m}
\end{equation}
se factorise par un quotient $L(-\lambda)_{P_\alpha}^\vee\otimes_E\big((U^\infty_E\otimes_EE_{m})^{d_{\pi^\infty}}\oplus (\oplus_{J''} (U_E^\infty\otimes_EE_m))\big)^\vee$ pour $J''$ sous-ensemble {\it fini} de $J'$. Mais en utilisant $\displaystyle \lim_{m\rightarrow +\infty}\pi^\infty_{\alpha,m}(N_{\alpha,1})\simeq (V_E^\infty\otimes_EE_\infty)^{\oplus d_{\pi^\infty}}$, la d\'efinition de $U^\infty_E$ et les propri\'et\'es de la d\'ecomposition (\ref{decompopenible}), on voit qu'il existe $m'\gg 0$ tel que dans la compos\'ee~:
$$\big((U^\infty_E\otimes_EE_{m})^{d_{\pi^\infty}}\oplus (\oplus_{J''} (U_E^\infty\otimes_EE_m))\big)\otimes_{E_m}E_{m'}\hookrightarrow \pi^\infty_{\alpha,m}(N_{\alpha,1})\otimes_{E_m}E_{m'}\twoheadrightarrow \pi^\infty_{\alpha,m'}(N_{\alpha,1})$$
l'image de $(\oplus_{J''} (U_E^\infty\otimes_EE_m))\otimes_{E_m}E_{m'}$ est en fait {\it contenue} dans celle de $((U^\infty_E\otimes_EE_{m})^{d_{\pi^\infty}})\otimes_{E_m}E_{m'}$. On d\'eduit de tout cela que la compos\'ee~:
\begin{multline*}
L(-\lambda)_{P_\alpha}^\vee\otimes_E\pi^\infty_{\alpha,m'}(N_{\alpha,1})^\vee\hookrightarrow \big(L(-\lambda)_{P_\alpha}^\vee\otimes_E\pi^\infty_{\alpha,m}(N_{\alpha,1})^\vee\big)\otimes_{E_m}E_{m'}\\
\buildrel f\otimes\Id\over\longrightarrow T_r\otimes_EE_{m'}\longrightarrow T_{r'}\otimes_EE_{m'}
\end{multline*}
se factorise par le quotient $L(-\lambda)_{P_\alpha}^\vee\otimes_E\big((U^\infty_E\otimes_EE_{m'})^{d_{\pi^\infty}}\big)^\vee$ de $L(-\lambda)_{P_\alpha}^\vee\otimes_E\pi^\infty_{\alpha,m'}(N_{\alpha,1})^\vee$ dual de la compos\'ee (injective)~:
\begin{multline*}
(U^\infty_E\otimes_EE_{m_0})^{d_{\pi^\infty}}\otimes_{E_{m_0}}E_{m'}\subseteq (V^\infty_E\otimes_EE_{m_0})^{d_{\pi^\infty}}\otimes_{E_{m_0}}E_{m'}\\
\buildrel{j_{m_0}\otimes\Id}\over\hookrightarrow \pi^\infty_{\alpha,m_0}(N_{\alpha,1})\otimes_{E_{m_0}}E_{m'}\twoheadrightarrow \pi^\infty_{\alpha,m'}(N_{\alpha,1}),
\end{multline*}
et donc devient juste un morphisme continu ${\mathcal R}_{E_{m'}}^+$-lin\'eaire $L(-\lambda)_{P_\alpha}^\vee\otimes_E\big((U^\infty_E\otimes_EE_{m'})^{d_{\pi^\infty}}\big)^\vee\longrightarrow T_{r'}\otimes_EE_{m'}$ commutant \`a $\psi$ et $\Gamma$. Par le Lemme \ref{backtore} avec (\ref{bernstein3}) et (\ref{falphalisse}), on en d\'eduit que l'on a (sans se pr\'eoccuper de l'action de $\Gal(E_\infty/E)$ pour le moment)~:
$$F_\alpha(\pi)(T)\simeq E_\infty\otimes_{E}\Hom_{(\varphi,\Gamma)}\!\Big(\big(\R(\lambda\circ \lambda_{\alpha^\vee})/(t^{1-\langle \lambda,\alpha^\vee\rangle})\big)^{\oplus d_{\pi^\infty}},T\Big).$$

\noindent
{\bf \'Etape 6}\\
Il reste \`a examiner l'action de $\Gal(E_\infty/E)$. Pour $g\in \Gal(E_\infty/E)$ et $t_g$ comme dans la preuve du (i) du Lemme \ref{m+1}, on a des isomorphismes $t_g\circ g:\pi^\infty_{\alpha}\otimes_EE_m\buildrel\sim\over\rightarrow \pi^\infty_{\alpha}\otimes_EE_m$ et $t_g\circ g:\pi^\infty_{\alpha,m}\buildrel\sim\over\rightarrow \pi^\infty_{\alpha,m}$ qui sont $E_m$-semi-lin\'eaires, localement constants et commutent \`a l'action de $P_2$. En particulier, dans l'isomorphisme $\pi^\infty_\alpha\otimes_EE_m\simeq \oplus_{\iota\in I}(W_E^\infty\otimes_EE_m)$ induit par l'\'Etape $2$, on voit que pour tout sous-ensemble {\it fini} $J$ de $I$ la restriction $(t_g\circ g)\vert_{\oplus_J(W_E^\infty\otimes_EE_m)}$ est l'identit\'e pour tout $g$ dans un sous-groupe ouvert suffisamment petit de $\Gal(E_\infty/E)$ (d\'ependant de $J$). Comme l'action de $t_g$ sur $L(-\lambda)_{P_\alpha}$ est la multiplication par $(-\lambda)(t_g)$, il suit alors des consid\'erations dans l'\'Etape $5$ (en particulier le fait que $f$ est non nulle seulement sur un nombre {\it fini} de facteur) que, pour tout $(\varphi,\Gamma$)-module g\'en\'eralis\'e $T$ sur $\R$, l'action de $\Gal(E_\infty/E)$ tordue par $\chi_{-\lambda}^{-1}$ sur le $E_\infty$-espace vectoriel de dimension finie $F_\alpha(\pi)(T)$ en (\ref{falphalisse}) est localement constante (semi-lin\'eaire). Le r\'esultat final d\'ecoule alors facilement du th\'eor\`eme de Hilbert $90$.

\subsection{Le cas de la cellule ouverte d'une induite parabolique}\label{cellule}

On exa\-mine ce que donne le foncteur $F_\alpha$ appliqu\'e \`a la cellule ouverte de certaines induites paraboliques localement analytiques.

On conserve les notations pr\'ec\'edentes. On fixe $P\subseteq G$ un sous-groupe parabolique contenant $B$ et $\pi_P$ une repr\'esentation localement analytique de ${P}^-(\Qp)$ sur un $E$-espace vectoriel de type compact. On a un isomorphisme de vari\'et\'es localement $\Qp$-analytiques~:
\begin{equation}\label{varflag}
G(\Qp)\simeq P^-(\Qp)\times (P^-(\Qp)\backslash G(\Qp))
\end{equation}
compatible \`a la multiplication \`a gauche par $P^-(\Qp)$ (la multiplication sur le terme de droite \'etant triviale sur $P^-(\Qp)\backslash G(\Qp)$) et tel que la projection induite $G(\Qp)\twoheadrightarrow P^-(\Qp)\backslash G(\Qp)$ soit la projection canonique (voir par exemple la preuve de \cite[Lem.~8.1]{Ko2}). L'isomorphisme (\ref{varflag}) induit des isomorphismes topologiques d'espaces de type compact (cf. \cite[(56)~\&~Rem.~5.4]{Ko2})~:
\begin{multline}\label{isoflag}
\big(\Ind_{P^-(\Qp)}^{G(\Qp)} \pi_P\big)^{\an}\simeq C^{\an}(P^-(\Qp)\backslash G(\Qp),\pi_P)\simeq C^{\an}(P^-(\Qp)\backslash G(\Qp),E)\widehat\otimes_{E,\pi}\pi_P\\
\simeq C^{\an}(P^-(\Qp)\backslash G(\Qp),E)\widehat\otimes_{E,\iota}\pi_P
\end{multline}
o\`u l'on a utilis\'e que $C^{\an}(P^-(\Qp)\backslash G(\Qp),E)$ est un espace de type compact (car $P^-(\Qp)\backslash G(\Qp)$ est compact puisque $P^-\backslash G$ est une vari\'et\'e projective).

Soit $U\subseteq G(\Qp)$ un ouvert stable par multiplication \`a gauche par $P^-(\Qp)$, que l'on peut donc \'ecrire par (\ref{varflag}) $U\simeq P^-(\Qp)\times \overline U$ o\`u $\overline U$ est un ouvert de $P^-(\Qp)\backslash G(\Qp)$ (en fait $\overline U\cong P^-(\Qp)\backslash U$). On note $(\cInd_{P^-(\Qp)}^{U}\pi_P)^{\an}$ le sous-espace vectoriel de $(\Ind_{P^-(\Qp)}^{G(\Qp)}\!\pi_P)^{\an}$ des fonctions dont le support est contenu dans $U$. Par (\ref{isoflag}), il s'identifie au sous-espace $C^{\an}(P^-(\Qp)\backslash G(\Qp),\pi_P)_{\overline U}$ de $C^{\an}(P^-(\Qp)\backslash G(\Qp),\pi_P)$ des fonctions \`a support contenu dans l'ouvert $\overline U$, qui est un sous-espace ferm\'e de $C^{\an}(P^-(\Qp)\backslash G(\Qp),\pi_P)$ par \cite[\S~2.3.1]{FdL}, donc encore un espace de type compact. Comme $P^-(\Qp)\backslash G(\Qp)$ est une vari\'et\'e compacte (et qu'un ferm\'e dans un compact est compact), toute fonction dans $C^{\an}(P^-(\Qp)\backslash G(\Qp),\pi_P)_{\overline U}$ est \`a support compact contenu dans $\overline U$, i.e. on a $C^{\an}(P^-(\Qp)\backslash G(\Qp),\pi_P)_{\overline U}=C^{\an}_c(\overline U,\pi_P)$ avec les notations du \S~\ref{topo}. Lorsque $U$ est de plus stable par multiplication \`a droite par $P(\Qp)$, notons que $(\cInd_{P^-(\Qp)}^{U}\pi_P)^{\an}$ est un sous-espace (ferm\'e) de $(\Ind_{P^-(\Qp)}^{G(\Qp)}\pi_P)^{\an}$ stable sous l'action de $P(\Qp)$. 

\begin{rem}\label{rematopo}
{\rm On peut montrer que tout ouvert d'une vari\'et\'e localement $\Qp$-analytique compacte (par exemple $\overline U$ dans $P^-(\Qp)\backslash G(\Qp)$) admet un recouvrement disjoint par un nombre d\'enombrable d'ouverts compacts et que la topologie de type compact sur $C^{\an}_c(\overline U,\pi_P)$ provenant de $C^{\an}(P^-(\Qp)\backslash G(\Qp),\pi_P)_{\overline U}$ s'identifie \`a celle du Lemme \ref{supportcompact} (utiliser qu'une bijection continue entre espaces de type compact est un isomorphisme topologique).}
\end{rem}

Le lemme suivant sera utile.

\begin{lem}\label{attention!}
On a un isomorphisme d'espaces de type compact~:
\begin{equation*}
\big(\cInd_{P^-(\Qp)}^{U}\pi_P\big)^{\an}\cong C^{\an}(P^-(\Qp)\backslash G(\Qp),E)_{\overline U}\widehat\otimes_E\pi_P.
\end{equation*}
\end{lem}
\begin{proof}
Par la discussion juste avant \cite[Def.~2.3.3]{FdL} on a une suite exacte (stricte) d'espaces de type compact (en remarquant que le compl\'ementaire de $\overline U$ dans $P^-(\Qp)\backslash G(\Qp)$ est compact)~:
\begin{equation*}
0\rightarrow C^{\an}(P^-(\Qp)\backslash G(\Qp),\pi_P)_{\overline U}\rightarrow C^{\an}(P^-(\Qp)\backslash G(\Qp),\pi_P)\rightarrow \ilim{V}C^{\an}(V,\pi_P)\rightarrow 0
\end{equation*}
o\`u, dans la limite inductive topologique de droite, $V$ parcourt les ouverts (compacts) de $P^-(\Qp)\backslash G(\Qp)$ contenant le compl\'ementaire de $\overline U$ et les applications de transition sont les restrictions naturelles. On a de m\^eme une suite exacte d'espaces de type compact~:
\begin{equation}\small\label{se1}
0\rightarrow C^{\an}(P^-(\Qp)\backslash G(\Qp),E)_{\overline U}\rightarrow C^{\an}(P^-(\Qp)\backslash G(\Qp),E)\rightarrow \ilim{V}C^{\an}(V,E)\rightarrow 0.
\end{equation}
Notons $A\=C^{\an}(P^-(\Qp)\backslash G(\Qp),E)_{\overline U}\otimes_E\pi_P$, $B\=C^{\an}(P^-(\Qp)\backslash G(\Qp),E)\otimes_E\pi_P$ et $C\= (\ilim{V}C^{\an}(V,E))\otimes_E\pi_P$, on en d\'eduit un diagramme commutatif d'espaces localement convexes o\`u les deux suites sont exactes (comme $E$-espaces vectoriels)~:
\begin{equation}\tiny\label{se1bis}
\begin{gathered}
\xymatrix{0\ar[r] &A \ar[r]\ar[d] &B\ar[r]\ar[d]&C\ar[r]\ar[d]&0\\
0\ar[r] &C^{\an}(P^-(\Qp)\backslash G(\Qp),\pi_P)_{\overline U} \ar[r] & C^{\an}(P^-(\Qp)\backslash G(\Qp),\pi_P) \ar[r] &\ilim{V}C^{\an}(V,\pi_P) \ar[r]&0.}
\end{gathered}
\end{equation}
Puisque les espaces du bas sont tous complets, les fl\`eches verticales se factorisent par les compl\'et\'es respectifs $\widehat A$, $\widehat B$, $\widehat C$ (\cite[Prop.~7.5]{Sch}). Mais par (\ref{se1}), \cite[Lem.~4.13]{Sc1} et \cite[Prop.~20.13]{Sch} on a encore une suite exacte (stricte) d'espaces de type compacts $0\rightarrow \widehat A\rightarrow \widehat B\rightarrow \widehat C\rightarrow 0$, d'o\`u on d\'eduit un diagramme commutatif d'espaces de type compact o\`u les deux suites sont exactes (et n\'ecessairement strictes)~:
\begin{equation}\tiny\label{se2}
\begin{gathered}
\xymatrix{0\ar[r] &\widehat A \ar[r]\ar[d] &\widehat B\ar[r]\ar[d]&\widehat C\ar[r]\ar[d]&0\\
0\ar[r] &C^{\an}(P^-(\Qp)\backslash G(\Qp),\pi_P)_{\overline U} \ar[r] & C^{\an}(P^-(\Qp)\backslash G(\Qp),\pi_P) \ar[r] &\ilim{V}C^{\an}(V,\pi_P) \ar[r]&0.}
\end{gathered}
\end{equation}
Par l'argument \`a la fin de la preuve du Lemme \ref{supproduit} avec \cite[Prop.A.2]{ST3}, on a un isomorphisme d'espaces de type compact~:
$$\widehat C=\big(\ilim{V}C^{\an}(V,E)\big)\widehat\otimes_E\pi_P\cong \ilim{V}\big(C^{\an}(V,E)\widehat\otimes_E\pi_P\big)\cong \ilim{V}C^{\an}(V,\pi_P)$$
de sorte qu'avec (\ref{isoflag}) les deux fl\`eches verticales de droite dans (\ref{se2}) sont des isomorphismes. Celle de gauche l'est donc \'egalement, d'o\`u le r\'esultat avec la discussion pr\'ec\'edant ce lemme.
\end{proof}

On suppose maintenant $U=P^-(\Qp)P(\Qp)\simeq P^-(\Qp)\times N_{P}(\Qp)$ o\`u l'on voit $\overline U\cong N_{P}(\Qp)$ comme ouvert de $P^-(\Qp)\backslash G(\Qp)$. Par la discussion pr\'ec\'edant le Lemme \ref{attention!}, on d\'eduit que la restriction \`a l'ouvert $N_{P}(\Qp)$ induit un isomorphisme d'espaces de type compact~:
\begin{equation}\label{deIaC}
\big(\cInd_{P^-(\Qp)}^{P^-(\Qp)P(\Qp)}\pi_P\big)^{\an}\buildrel\sim\over\longrightarrow C^{\an}_c\big(N_{P}(\Qp),\pi_P\big).
\end{equation}
Cet isomorphisme est de plus $P(\Qp)$-\'equivariant en d\'efinissant l'action de $N_{P}(\Qp)$ sur $C^{\an}_c(N_{P}(\Qp),\pi_P)$ comme la translation \`a droite et celle de $L_{P}(\Qp)$ comme suit pour $f\in C^{\an}_c(N_{P}(\Qp),\pi_P)$~:
\begin{equation}\label{actionP}
(gf)(n)=g\big(f(g^{-1}ng)\big),\ \ g\in L_{P}(\Qp),\ \ n\in N_{P}(\Qp).
\end{equation}
Pour tout $m\in \Z_{\geq 0}$ le sous-espace $C^{\an}(N_{P}(\Qp)\cap N_m,\pi_P)$ de $C^{\an}_c(N_{P}(\Qp),\pi_P)$ des fonctions \`a support dans $N_{P}(\Qp)\cap N_m$ est de type compact et stable par l'action (\ref{actionP}) de $N_m$ et de $\lambda_{\beta^{\!\vee}}(\Zp\backslash \{0\})$ pour $\beta\in S$ (utiliser $N_m=(N_{P}(\Qp)\cap N_m)\rtimes (N_{L_P}(\Qp)\cap N_m)$ et (\ref{incl})). On peut donc d\'efinir le foncteur $F_{\alpha,m}(C^{\an}(N_{P}(\Qp)\cap N_m,\pi_P))$ dans $F(\varphi,\Gamma)_m$ par les formules (\ref{malpha}), (\ref{foncteurm}) et le (i) du Lemme \ref{m+1}. Comme pour (\ref{mpi}), les injections~:
$$C^{\an}(N_{P}(\Qp)\cap N_m,\pi_P)\hookrightarrow C^{\an}(N_{P}(\Qp)\cap N_{m+1},\pi_P)\hookrightarrow C^{\an}_c(N_{P}(\Qp),\pi_P)$$
induisent un diagramme commutatif de foncteurs pour tout $m\in \Z_{\geq 0}$ commutant \`a l'action de $\Gal(E_\infty/E)$~:
\begin{equation}\label{NmNm+1}
\begin{gathered}
\xymatrix{ F_{\alpha,m}\big(C^{\an}(N_{P}(\Qp)\cap N_m,\pi_P)\big) \ar[d]\ar[r]&F_{\alpha,m+1}\big(C^{\an}(N_{P}(\Qp)\cap N_{m+1},\pi_P)\big) \ar[d]\\
F_{\alpha,m}\big(C^{\an}_c(N_{P}(\Qp),\pi_P)\big)\ar[r] & F_{\alpha,m+1}\big(C^{\an}_c(N_{P}(\Qp),\pi_P)\big).}
\end{gathered}
\end{equation}

\begin{lem}\label{niveaum}
Les diagrammes (\ref{NmNm+1}) induisent un isomorphisme dans $F(\varphi,\Gamma)_\infty$~:
\begin{multline*}
\lim_{m\rightarrow +\infty}F_{\alpha,m}\big(C^{\an}(N_{P}(\Qp)\cap N_m,\pi_P)\big)\buildrel\sim\over\longrightarrow \lim_{m\rightarrow +\infty}F_{\alpha,m}\big(C^{\an}_c(N_{P}(\Qp),\pi_P)\big)\\
=F_{\alpha}\big(C^{\an}_c(N_{P}(\Qp),\pi_P)\big).
\end{multline*}
\end{lem}
\begin{proof}
Comme $(N_{P}(\Qp)\cap N_m)\cup (N_{P}(\Qp)\cap (N_{m'}\backslash N_{m'-1}))_{m'\geq m+1}$ est un recouvrement ouvert disjoint de $N_{P}(\Qp)$, on a un isomorphisme topologique pour tout $m\in \Z_{\geq 0}$~:
\begin{multline}\label{decompo0}
C^{\an}_c(N_{P}(\Qp),\pi_P)\simeq C^{\an}(N_{P}(\Qp)\cap N_m,\pi_P)\\
\bigoplus \Big(\!\bigoplus_{m'\geq m+1} C^{\an}(N_{P}(\Qp)\cap (N_{m'}\backslash N_{m'-1}),\pi_P)\Big).
\end{multline}
On a vu que les actions de $N_{m}$ et $\lambda_{\alpha^{\!\vee}}(\Zp\backslash\{0\})$ pr\'eservent $C^{\an}(N_{P}(\Qp)\cap N_{m},\pi_P)$. Pour $m'\geq m+1$ les actions de $N_m$ et $\lambda_{\beta^{\!\vee}}(\Zp^\times)$ pour $\beta\in S$ pr\'eservent chaque $C^{\an}(N_{P}(\Qp)\cap (N_{m'}\backslash N_{m'-1}),\pi_P)$ et l'action de $\lambda_{\alpha^{\!\vee}}(p)$ envoie $C^{\an}(N_{P}(\Qp)\cap (N_{m'}\backslash N_{m'-1}),\pi_P)$ dans $C^{\an}(N_{P}(\Qp)\cap N_{m'},\pi_P)$ (et pr\'eserve $C^{\an}(N_{P}(\Qp)\cap N_{m'},\pi_P)$). Posons pour $m\in \Z_{\geq 0}$~:
\begin{eqnarray}\label{sommem}
\pi_{m}&\=& C^{\an}_c(N_{P}(\Qp),\pi_P)/C^{\an}(N_{P}(\Qp)\cap N_m,\pi_P)\\
\nonumber&\simeq &\oplus_{m'\geq m+1} C^{\an}(N_{P}(\Qp)\cap (N_{m'}\backslash N_{m'-1}),\pi_P)
\end{eqnarray}
muni de l'action quotient de $N_m$ et $\lambda_{\alpha^{\!\vee}}(\Zp\backslash\{0\})$. Par la preuve du (i) de la Proposition \ref{drex}, on a une suite exacte dans $F(\varphi,\Gamma)_m$~:
$$0\longrightarrow F_{\alpha,m}(C^{\an}_c(N_{P}(\Qp)\cap N_{m},\pi_P))\longrightarrow F_{\alpha,m}(C^{\an}_c(N_{P}(\Qp),\pi_P))\longrightarrow F_{\alpha,m}(\pi_m)$$
et par celle du (iii) du Lemme \ref{m+1} on a un morphisme $F_{\alpha,m}(\pi_m)\rightarrow F_{\alpha,m+1}(\pi_{m+1})$ de foncteurs qui commute (en un sens \'evident) avec (\ref{NmNm+1}). Il suffit donc de montrer $\displaystyle \lim_{m\rightarrow +\infty}F_{\alpha,m}(\pi_m)=0$. 

D\'efinissons les $\Rm^+$-modules $M_\alpha(\pi_m\otimes_EE_m)$ et~:
$$M_{\alpha,m'}\=M_{\alpha}\big(C^{\an}(N_{P}(\Qp)\cap (N_{m'}\backslash N_{m'-1}),\pi_P\otimes_EE_m)\big)$$
pour $m'\geq m+1$ avec $M_\alpha(-)$ comme en (\ref{malpha}). On d\'eduit de (\ref{sommem}) avec (\ref{malpha}) et \cite[Prop.~9.10]{Sch} un isomorphisme de $\Rm^+$-modules~:
\begin{equation}\label{pourciter}
M_\alpha(\pi_m\otimes_EE_m)=\prod_{m'\geq m+1}M_{\alpha,m'}
\end{equation}
et par ce qui pr\'ec\`ede l'op\'erateur $\psi$ sur $M_\alpha(\pi_m\otimes_EE_m)$ (provenant de l'action de $\lambda_{\alpha^{\!\vee}}(p)$ sur $\pi_m$, cf. le d\'ebut du \S~\ref{def}) envoie $M_{\alpha,m'}$ dans $\oplus_{m''\geq m'}M_{\alpha,m''}$ et donc pr\'eserve les sous-espaces $\prod_{m''\geq m'}M_{\alpha,m''}$ pour tout $m'\geq m+1$. De plus le morphisme compos\'e~:
\begin{equation}\label{supportm+1}
M_\alpha(\pi_{m+1}\otimes_EE_{m+1})\longrightarrow M_\alpha(\pi_m\otimes_EE_m)\otimes_{E_m}E_{m+1}\twoheadrightarrow M_{\alpha,m+1}\otimes_{E_m}E_{m+1}
\end{equation}
est nul.

Soit maintenant $r\in \Q_{>p-1}$, $T_r$ un $(\varphi,\Gamma)$-module g\'en\'eralis\'e sur $\R^r$ et~:
$$f:M_\alpha(\pi_m\otimes_EE_m)\longrightarrow T_r\otimes_EE_m$$
un morphisme $\Rm^+$-lin\'eaire continu commutant \`a $\psi$. Montrons que $f$ est nul sur tous les facteurs $M_{\alpha,m'}$ sauf un nombre fini. En utilisant que $\psi$ envoie $M_{\alpha,m'}$ dans $\oplus_{m''\geq m'}M_{\alpha,m''}$ pour $m'\geq m+1$, la preuve est alors la m\^eme (aux notations pr\`es) que la preuve de la surjectivit\'e du morphisme (\ref{topologie}) (cet argument est d\'ej\`a utilis\'e dans la preuve du Lemme \ref{Falphanul} et du (ii) du Lemme \ref{backtore}). Donc le morphisme $f$ se factorise par $\prod_{m''\geq m'\geq m+1}M_{\alpha,m'}$ pour un $m''\gg 0$. Par (\ref{supportm+1}) (it\'er\'e) cela implique que la compos\'ee~:
$$M_\alpha(\pi_{m''}\otimes_EE_{m''})\longrightarrow M_\alpha(\pi_m\otimes_EE_m)\otimes_{E_m}E_{m''}\buildrel f\otimes 1\over\longrightarrow T_r\otimes_EE_{m''}$$
est nulle, ce qui ach\`eve la preuve.
\end{proof}

Jusqu'\`a la fin du \S~\ref{cellule} on suppose $\pi_P\cong L(-\lambda)_P\otimes_E \pi_P^\infty$ o\`u $L(-\lambda)_P$ est une repr\'esentation alg\'ebrique irr\'eductible de $L_P(\Qp)$ sur $E$, $\pi_P^\infty$ est une repr\'esentation lisse de dimension d\'enombrable de $L_P(\Qp)$ sur $E$ et o\`u $\pi_P$ est muni de la topologie localement convexe la plus fine (un espace de type compact par la discussion juste avant \cite[Lem.~2.4]{OS}). 

On note $S_P\subseteq S$ le sous-ensemble des racines simples de $L_P$. Rappelons qu'une repr\'esentation lisse irr\'eductible $\pi_P^\infty$ de $L_P(\Qp)$ sur $E$ est dite g\'en\'erique si $(\pi_P^\infty\otimes_EE_\infty)(\eta^{-1})_{N_{L_{P}}}\ne 0$ (en particulier $\pi_P^\infty$ est toujours g\'en\'erique si $L_P=T$).

\begin{prop}\label{celluleouverte}
Supposons que $\alpha\notin S_P$ et que $\lambda_{\alpha^\vee}(\Qp^\times)$ agisse sur $\pi_P^\infty$ par un caract\`ere $\chi_{\pi_P^\infty,\alpha}$ (c'est par exemple le cas si $\pi_P^\infty$ a un caract\`ere central). Alors le foncteur $F_\alpha(C^{\an}_c(N_{P}(\Qp),\pi_P))$ est isomorphe \`a~:
$$E_\infty(\chi_{-\lambda})\otimes_{E_\infty}(\pi_P^\infty\otimes_EE_\infty)(\eta^{-1})_{N_{L_{P}}\!(\Qp)}\otimes_E \Hom_{(\varphi,\Gamma)}\Big(\R\big((\lambda\circ \lambda_{\alpha^\vee})\chi_{\pi_P^\infty,\alpha}^{-1}\big),-\Big)$$
avec action de $\Gal(E_\infty/E)$ sur le facteur $(\pi_P^\infty\otimes_EE_\infty)(\eta^{-1})_{N_{L_{P}}}$ comme dans la preuve du (i) du Lemme \ref{m+1}. En particulier, si $\pi_P^\infty$ est de plus de longueur finie, $F_\alpha(C^{\an}_c(N_{P}(\Qp),\pi_P))$ est nul si et seulement si $\pi_P^\infty$ n'a pas de constituant g\'en\'erique.
\end{prop}
\begin{proof}
Commen\c cons par calculer $F_{\alpha,m}(C^{\an}(N_{P}(\Qp)\cap N_m,\pi_P))$ pour $m\in \Z_{\geq 0}$. On a un isomorphisme topologique comme pour (\ref{isoflag})~:
\begin{eqnarray}\label{sortie}
C^{\an}(N_{P}(\Qp)\cap N_m,\pi_P)&\simeq & C^{\an}(N_{P}(\Qp)\cap N_m,E)\otimes_{E,\iota}\pi_P
\end{eqnarray}
en remarquant que l'espace de droite est de type compact (utiliser \cite[Prop.~1.1.32(i)]{Em} avec le fait que $\pi_P$ est une limite inductive topologique d\'enombrable de $E$-espaces vectoriels de dimension finie puisque $\pi_P^\infty$ est de dimension d\'enombrable), donc d\'ej\`a complet. Notons $N_{P}^\alpha\=N^\alpha\cap N_{P}$ (dans $N$), puisque $\alpha\notin S_P$ on a $N^\alpha=N_{P}^\alpha\rtimes N_{L_{P}}$ et $N_{P}=N_{P}^\alpha\rtimes N_\alpha$. On en d\'eduit un produit semi-direct des $\Qp$-alg\`ebres de Lie respectives $\mnn^\alpha\simeq \mnn_{P}^\alpha\rtimes \mnn_{L_{P}}$ et de groupes compacts $N_{m}^\alpha=(N_{P}^\alpha(\Qp)\cap N_{m}) \rtimes (N_{L_{P}}\!(\Qp)\cap N_{m})$, $N_{P}(\Qp)\cap N_{m}=(N_{P}^\alpha(\Qp)\cap N_{m}) \rtimes (N_\alpha(\Qp)\cap N_{m})$ pour tout $m\in \Z_{\geq 0}$. \'Ecrivant par \cite[Lem.~A.1]{ST3}~:
\begin{equation*}
C^{\an}(N_{P}(\Qp)\cap N_{m},E)\simeq C^{\an}\big(N_{\alpha}(\Qp)\cap N_{m},C^{\an}(N^\alpha_{P}(\Qp)\cap N_{m},E)\big)
\end{equation*}
on obtient (en notant $\otimes_E$ au lieu de $\otimes_{E,\iota}$)~:
\begin{multline*}
(C^{\an}(N_{P}(\Qp)\cap N_{m},E)\otimes_E\pi_P)[\mnn_{P}^\alpha]\\
\simeq C^{\an}\big(N_{\alpha}(\Qp)\cap N_{m},C^{\an}(N^\alpha_{P}(\Qp)\cap N_{m},E)[\mnn_{P}^\alpha]\big)\otimes_E\pi_P\\
\simeq C^{\an}\big(N_{\alpha}(\Qp)\cap N_{m},C^\infty(N^\alpha_{P}(\Qp)\cap N_{m},E)\big)\otimes_E\pi_P.
\end{multline*}
Par \cite[Prop.~A.2]{ST3} et le fait que $C^\infty(N^\alpha_{P}(\Qp)\cap N_{m},E)$ est un espace de type compact de dimension d\'enombrable muni de la topologie localement convexe la plus fine, on en d\'eduit~:
\begin{multline*}
(C^{\an}(N_{P}(\Qp)\cap N_{m},E)\otimes_E\pi_P)[\mnn_{P}^\alpha]\\
\simeq \big(C^{\an}(N_\alpha(\Qp)\cap N_{m},E)\widehat \otimes_{E,\pi}C^\infty(N^\alpha_{P}(\Qp)\cap N_{m},E)\big)\otimes_E \pi_P\\
\simeq C^{\an}(N_\alpha(\Qp)\cap N_{m},E)\otimes_{E}C^\infty(N^\alpha_{P}(\Qp)\cap N_{m},E)\otimes_E \pi_P
\end{multline*}
d'o\`u on obtient avec (\ref{sortie})~:
\begin{multline*}
C^{\an}(N_{P}(\Qp)\cap N_{m},\pi_P)[\mnn^\alpha]\simeq \big(C^{\an}(N_{P}(\Qp)\cap N_{m},\pi_P)[\mnn_{P}^\alpha]\big)[\mnn_{L_{P}}]\\
\simeq C^{\an}(N_\alpha(\Qp)\cap N_{m},E)\otimes_{E}C^\infty(N^\alpha_{P}(\Qp)\cap N_{m},E)\otimes_E (\pi_P[\mnn_{L_{P}}]).
\end{multline*}
Si $v_{-\lambda}$ est un vecteur de plus haut poids de $L(-\lambda)_{P}$ (pour $B\cap L_{P}$), on a $\pi_P[\mnn_{L_{P}}]\simeq Ev_{-\lambda} \otimes_E\pi_P^\infty$. On a un isomorphisme (puisque $N_{P}(\Qp)\cap N_m$ est un groupe compact)~:
$$C^\infty(N^\alpha_{P}(\Qp)\cap N_{m},E_m)(\eta^{-1})^{N^\alpha_{P}(\Qp)\cap N_{m}}\buildrel\sim\over\longrightarrow C^\infty(N^\alpha_{P}(\Qp)\cap N_{m},E_m)(\eta^{-1})_{N^\alpha_{P}(\Qp)\cap N_{m}}$$
qui montre que $C^\infty(N^\alpha_{P}(\Qp)\cap N_{m},E_m)(\eta^{-1})_{N^\alpha_{P}(\Qp)\cap N_{m}}$ est de dimension $1$ sur $E_m$ engendr\'e par la classe de la fonction localement constante $\eta\vert_{N^\alpha_{P}(\Qp)\cap N_{m}}$. De plus, l'automorphisme $t_g\circ g$ pour $g\in \Gal(E_\infty/E)$ (cf. (\ref{actiong}) et (\ref{actiontg})) \'etant $E_m$-semi-lin\'eaire, une application du Th\'eor\`eme 90 de Hilbert montre que $C^\infty(N^\alpha_{P}(\Qp)\cap N_{m},E_m)(\eta^{-1})_{N^\alpha_{P}(\Qp)\cap N_{m}}$ s'identifie \`a $E_m$ avec action usuelle de $\Gal(E_\infty/E)$ dessus. On en d\'eduit~:
\begin{multline}\label{firststep}
C^{\an}(N_{P}(\Qp)\cap N_{m},\pi_P\otimes_EE_m)[\mnn^\alpha](\eta^{-1})_{N^\alpha_{P}(\Qp)\cap N_{m}}\\
\simeq C^{\an}(N_\alpha(\Qp)\cap N_{m},E)\otimes_{E}(Ev_{-\lambda} \otimes_E\pi_P^\infty)\otimes_EE_m
\end{multline}
Via les produits semi-directs ci-dessus, on d\'eduit de (\ref{firststep}) et (\ref{actionP})~:
\begin{multline}\label{Malphaenfin}
C^{\an}(N_{P}(\Qp)\cap N_{m},\pi_P\otimes E_m)[\mnn^\alpha](\eta^{-1})_{N_m^\alpha}\\
\simeq C^{\an}(N_\alpha(\Qp)\cap N_{m},E)\otimes_{E}Ev_{-\lambda} \otimes_E(\pi_P^\infty\otimes_E E_m)(\eta^{-1})_{N_{L_{P}}\!(\Qp)\cap N_{m}}
\end{multline}
o\`u l'action (r\'esiduelle) de $N_\alpha(\Qp)\cap N_m\simeq \frac{1}{p^{m}}\Zp$ (cf. (\ref{zp})) est la translation \`a droite sur $C^{\an}(N_\alpha(\Qp)\cap N_{m},E)$ (et triviale sur les autres facteurs), o\`u l'action de $t=\lambda_{\alpha^{\!\vee}}(x)$ pour $x\in \Zp\backslash\{0\}$ est $f\mapsto (z\in \frac{1}{p^{m}}\Zp\mapsto f(z/x))$ sur $C^{\an}(N_\alpha(\Qp)\cap N_{m},E)\simeq C^{\an}(\frac{1}{p^{m}}\Zp,E)$ (vu dans $C^{\an}_c(N_\alpha(\Qp),E)$), est la multiplication par $(-\lambda)(t)$ sur $Ev_{-\lambda}$ et par $\chi_{\pi_P^\infty,\alpha}(t)$ sur $(\pi_P^\infty\otimes_E E_m)(\eta^{-1})_{N_{L_{P}}\!(\Qp)\cap N_{m}}$, et o\`u l'action de $\Gal(E_\infty/E)$ est triviale sur $C^{\an}(N_\alpha(\Qp)\cap N_{m},E)$ et est comme dans l'\'enonc\'e sur $Ev_{-\lambda} \otimes_E(\pi_P^\infty\otimes_E E_m)(\eta^{-1})_{N_{L_{P}}\!(\Qp)\cap N_{m}}$. 

Posons $V_m\=Ev_{-\lambda} \otimes_E(\pi_P^\infty\otimes_E E_m)(\eta^{-1})_{N_{L_{P}}\!(\Qp)\cap N_{m}}$ (de dimension d\'enombrable muni de la topologie localement convexe la plus fine), le sous-espace $C^{\an}(N_\alpha(\Qp)\cap N_{0},E)\otimes_{E}V_m$ de (\ref{Malphaenfin}) est ferm\'e et stable par l'action de $N_\alpha(\Qp)\cap N_0$ et $\lambda_{\alpha^{\!\vee}}(\Zp\backslash\{0\})$. De plus il existe $N\gg 0$ tel que l'action de $\lambda_{\alpha^{\!\vee}}(p)^N$ sur le quotient $(C^{\an}(N_\alpha(\Qp)\cap N_{m},E)/C^{\an}(N_\alpha(\Qp)\cap N_{0},E))\otimes_{E}V_m$ est nulle. Soit $r\in \Q_{>p-1}$, $T_r$ un $(\varphi,\Gamma)$-module g\'en\'eralis\'e sur $\R^r$ et~:
$$f:(C^{\an}(N_\alpha(\Qp)\cap N_{m},E)/C^{\an}(N_\alpha(\Qp)\cap N_{0},E))^\vee\widehat \otimes_{E}V_m^\vee\longrightarrow T_r\otimes_EE_m$$
un morphisme $\Rm^+$-lin\'eaire continu commutant \`a $\psi$ (on utilise tacitement ici la version duale de \cite[Prop.~20.13]{Sch}). On a $\psi^{N}=0$ \`a gauche, et (comme dans la preuve du (ii) du Lemme \ref{backtore}) par l'analogue du diagramme commutatif (\ref{psi5}) o\`u $\varphi$, $\psi$ sont remplac\'es par respectivement $\varphi^{N}$, $\psi^{N}$ et $pr$ par $p^Nr$ (cf. aussi (\ref{psi4}) et (\ref{psi6}) ci-dessous), on en d\'eduit que la compos\'ee~:
$$(C^{\an}(N_\alpha(\Qp)\cap N_{m},E)/C^{\an}(N_\alpha(\Qp)\cap N_{0},E))^\vee\widehat \otimes_{E}V_m^\vee\buildrel f\over \longrightarrow T_r\otimes_EE_m\longrightarrow T_{p^{N}r}\otimes_EE_m$$
est nulle, ce qui implique~:
$$\lim_{\substack{\longrightarrow \\ (r,f_r,T_r)\in I(T)}}\!\!\!\!\!\!\Hom_{\psi,\Gamma}\big((C^{\an}(N_\alpha(\Qp)\cap N_{m},E)/C^{\an}(N_\alpha(\Qp)\cap N_{0},E))^\vee\widehat \otimes_{E}V_m^\vee,T_r\otimes_EE_m\big)=0$$
pour tout $(\varphi,\Gamma)$-module g\'en\'eralis\'e $T$ sur $\R$. En utilisant que $V\mapsto V\widehat \otimes_EV_m$ pr\'eserve les suites exactes courtes d'espaces de Fr\'echet (\cite[Lem.~4.13]{Sc1}), que $\Hom_{\psi,\Gamma}$ est exact \`a gauche et que les limites inductives filtrantes sont exactes, on obtient un isomorphisme pour tout $(\varphi,\Gamma)$-module g\'en\'eralis\'e $T$ sur $\R$~:
\begin{multline*}
\lim_{\substack{\longrightarrow \\ (r,f_r,T_r)\in I(T)}}\!\!\!\!\!\!\Hom_{\psi,\Gamma}\big(C^{\an}(N_\alpha(\Qp)\cap N_{0},E)^\vee\widehat \otimes_{E}V_m^\vee,T_r\otimes_EE_m\big)\\
\buildrel\sim\over\longrightarrow F_{\alpha,m}(C^{\an}(N_{P}(\Qp)\cap N_{m},\pi_P))(T).
\end{multline*}
Par le Lemme \ref{begin} et l'action de $\lambda_{\alpha^{\!\vee}}(\Zp\backslash\{0\})$ comme explicit\'ee ci-dessus (dualis\'ee), on a un isomorphisme~:
\begin{multline*}
\lim_{\substack{\longrightarrow \\ (r,f_r,T_r)\in I(T)}}\!\!\!\!\!\!\Hom_{\psi,\Gamma}\big(C^{\an}(N_\alpha(\Qp)\cap N_{0},E)^\vee\widehat \otimes_{E}V_m^\vee,T_r\otimes_EE_m\big)\cong \\
V_m\otimes_{E_m}\!\!\!\lim_{\substack{\longrightarrow \\ (r,f_r,T_r)\in I(T)}}\!\!\!\!\!\!\Hom_{\psi,\Gamma}\Big(C^{\an}(N_\alpha(\Qp)\cap N_{0},E_m)^\vee\!\otimes_{\R^+}\!\R^+\big(((-\lambda)\circ \lambda_{\alpha^\vee}^{-1})\chi_{\pi_P^\infty,\alpha}^{-1}\big),\\
T_r\otimes_EE_m\Big).
\end{multline*}
Avec l'isomorphisme $N_\alpha(\Qp)\cap N_0\simeq \Zp$, les (iii), (iv) et (v) de la Remarque \ref{vrac} et le (iii) du Lemme \ref{abel}, on en d\'eduit~:
\begin{multline*}
F_{\alpha,m}(C^{\an}(N_{P}(\Qp)\cap N_{m},\pi_P))(T)\\
\simeq Ev_{-\lambda} \otimes_E(\pi_P^\infty\otimes_EE_m)(\eta^{-1})_{N_{L_{P}}\!(\Qp)\cap N_m}\otimes_E \Hom_{(\varphi,\Gamma)}\big(\R\big(((-\lambda)\circ \lambda_{\alpha^\vee}^{-1})\chi_{\pi_P^\infty,\alpha}^{-1}\big),T\big)
\end{multline*}
avec action de $\Gal(E_\infty/E)$ (sur les facteurs de gauche) comme ci-dessus. Par le Lemme \ref{niveaum} (et (\ref{falpha})), on a la premi\`ere assertion puisque $(-\lambda)\circ \lambda_{\alpha^\vee}^{-1}=\lambda\circ \lambda_{\alpha^\vee}$ et~:
$$\displaystyle \lim_{m\rightarrow +\infty}(\pi_P^\infty\otimes_EE_m)(\eta^{-1})_{N_{L_{P}}\!(\Qp)\cap N_m}=(\pi_P^\infty\otimes_EE_\infty)(\eta^{-1})_{N_{L_{P}}\!(\Qp)}.$$
La deuxi\`eme assertion vient du fait que $(\pi_P^\infty\otimes_EE_\infty)(\eta^{-1})_{N_{L_{P}}\!(\Qp)}\ne 0$ si et seulement s'il existe un constituant $C_P$ de $\pi_P^\infty$ tel que $(C_P\otimes_EE_\infty)(\eta^{-1})_{N_{L_{P}}\!(\Qp)}\ne 0$, ou de mani\`ere \'equivalente $((C_P\otimes_EE_\infty)(\eta^{-1})_{N_{L_{P}}\!(\Qp)})^\vee\ne 0$, ou encore il existe une fonctionnelle de Whittaker non nulle sur $C_P$ (\`a valeurs dans $E_\infty$), ou encore $C_P$ est g\'en\'erique.
\end{proof}

\subsection{Approximation des gradu\'es}\label{approx}

On montre un r\'esultat technique mais crucial de densit\'e sur certains gradu\'es (Proposition \ref{descgrdense}).

On garde les notations du \S~\ref{cellule} mais cette fois $\pi_P$ est une repr\'esentation localement analytique de $P(\Qp)$ sur un espace de type compact et on suppose $P\ne G$. On fixe une section localement analytique $s:P(\Qp)\backslash G(\Qp)\hookrightarrow G(\Qp)$ de la projection $G(\Qp)\twoheadrightarrow P(\Qp)\backslash G(\Qp)$, on a comme en (\ref{varflag}) un isomorphisme de vari\'et\'es localement $\Qp$-analytiques compatible \`a la multiplication \`a gauche par $P(\Qp)$ (la multiplication sur $P(\Qp)\backslash G(\Qp)$ \'etant triviale)~:
\begin{equation}\label{varflagbis}
P(\Qp)\times (P(\Qp)\backslash G(\Qp))\buildrel\sim\over\longrightarrow G(\Qp),\ \ (q,x)\longmapsto qs(x).
\end{equation}
Soit $C$ un ferm\'e quelconque de $G(\Qp)$ stable par multiplication \`a gauche par $P(\Qp)$ et consid\'erons l'induite $(\cInd_{P(\Qp)}^{G(\Qp)\backslash C}\pi_P)^{\an}$ des fonctions $f:G(\Qp)\backslash C\rightarrow \pi_P$ \`a support compact modulo $P(\Qp)$. Puisque $G(\Qp)\backslash C$ est un ouvert de $G(\Qp)$ stable par multiplication \`a gauche par $P(\Qp)$, on a par le (m\^eme argument que pour le) Lemme \ref{attention!}~:
\begin{equation}\label{attentionbis!}
\big(\cInd_{P(\Qp)}^{G(\Qp)\backslash C}\pi_P\big)^{\an}\cong C^{\an}(P(\Qp)\backslash G(\Qp),E)_{P(\Qp)\backslash(G(\Qp)\backslash C)}\widehat\otimes_E\pi_P.
\end{equation}
De \ plus $(\cInd_{P(\Qp)}^{G(\Qp)\backslash C}\pi_P)^{\an}$ \ est \ un \ sous-espace \ ferm\'e \ de \ $(\Ind_{P(\Qp)}^{G(\Qp)}\pi_P)^{\an}\cong C^{\an}(P(\Qp)\backslash G(\Qp),E)\widehat\otimes_E\pi_P$ (cf. la discussion avant le Lemme \ref{attention!}). Il suit de la preuve de \cite[Prop.~5.3]{Ko2} (avec \cite[(50)]{Ko2}, \cite[(53)]{Ko2} et \cite[Rem.~5.4]{Ko2}) que l'on a un isomorphisme topologique de $G(\Qp)$-repr\'esentations (cf. le d\'ebut du \S~\ref{topo} pour les notations):
\begin{equation}\label{dualsp}
\big((\Ind_{P(\Qp)}^{G(\Qp)} \pi_P)^{\an}\big)^{\vee}\simeq (\pi_P)^\vee \widehat \otimes_{D(P(\Qp),E),\iota}D(G(\Qp),E)
\end{equation}
o\`u $(\pi_P)^\vee$ est ici muni de sa structure de $D(P(\Qp),E)$-module \`a droite et o\`u $g\in G(\Qp)$ agit sur le membre de droite par la multiplication \`a droite par $\delta_{g^{-1}}\in D(G(\Qp),E)$ (plus pr\'ecis\'ement les r\'ef\'erences ci-dessus donnent un isomorphisme $((\Ind_{P(\Qp)}^{G(\Qp)} \pi_P)^{\an})^{\vee}\simeq (\pi_P)^\vee \widetilde \otimes_{D(P(\Qp),E),\iota}D(G(\Qp),E)$ o\`u le terme de droite est d\'efini comme dans \cite[p.~198]{Ko2}, mais comme $((\Ind_{P(\Qp)}^{G(\Qp)} \pi_P)^{\an})^{\vee}$ est un espace de Fr\'echet nucl\'eaire, en particulier est s\'epar\'e, on a un isomorphisme $(\pi_P)^\vee \widetilde \otimes_{D(P(\Qp),E),\iota}D(G(\Qp),E)\buildrel\sim\over\rightarrow (\pi_P)^\vee \widehat \otimes_{D(P(\Qp),E),\iota}D(G(\Qp),E)$). Posons~:
\begin{equation}\label{xc}
\pi_C\=\big(\Ind_{P(\Qp)}^{G(\Qp)} \pi_P\big)^{\an}/ \big(\cInd_{P(\Qp)}^{G(\Qp)\backslash C}\pi_P\big)^{\an}\ \ {\rm et}\ \ X_C\=(\pi_C)^\vee.
\end{equation}
Alors $\pi_C$ est un espace de type compact et $X_C$ est un sous-espace ferm\'e de $((\Ind_{P(\Qp)}^{G(\Qp)} \pi_P)^{\an})^{\vee}$, donc \'egalement un espace de Fr\'echet nucl\'eaire.

\begin{lem}\label{dualspx}
Conservons les notations ci-dessus.\\
(i) On a un isomorphisme topologique~:
\begin{equation*}
(\pi_P)^\vee \widehat \otimes_{D(P(\Qp),E),\iota}D(G(\Qp),E)_C\simeq X_C
\end{equation*}
o\`u le membre de gauche est muni de la topologie quotient de $(\pi_P)^\vee \widehat\otimes_{E,\iota}D(G(\Qp),E)_C$ (cf. \S~\ref{topo} pour $D(G(\Qp),E)_{C}$).\\
(ii) L'injection $D(G(\Qp),E)_C\hookrightarrow D(G(\Qp),E)$ induit une immersion ferm\'ee d'espaces de Fr\'echet~:
$$(\pi_P)^\vee \widehat \otimes_{D(P(\Qp),E),\iota}D(G(\Qp),E)_C\hookrightarrow (\pi_P)^\vee \widehat \otimes_{D(P(\Qp),E),\iota}D(G(\Qp),E).$$
\end{lem}
\begin{proof}
(i) Soit $F$ l'image de $C$ dans $P(\Qp)\backslash G(\Qp)$, qui est un ferm\'e (donc un compact) de $P(\Qp)\backslash G(\Qp)$. On a un isomorphisme topologique par (\ref{varflagbis}) et le Lemme \ref{supproduit}~:
\begin{equation}\label{densetenseur}
D(P(\Qp),E)\widehat\otimes_{E,\iota}D(P(\Qp)\backslash G(\Qp),E)_F\buildrel\sim\over\longrightarrow D(G(\Qp),E)_C.
\end{equation}
On a \'egalement un isomorphisme topologique par (\ref{attentionbis!}), (\ref{isoflag}) et \cite[Lem.~4.13]{Sc1} avec \cite[Prop.~20.13]{Sch}~:
\begin{equation*}
\pi_C\simeq \Big(C^{\an}\big(P(\Qp)\backslash G(\Qp),E\big)/C^{\an}\big(P(\Qp)\backslash G(\Qp),E\big)_{P(\Qp)\backslash(G(\Qp)\backslash C)}\Big)\widehat\otimes_{E} \pi_P
\end{equation*}
qui induit en dualisant un isomorphisme $X_C\simeq (\pi_P)^\vee \widehat\otimes_{E} D(P(\Qp)\backslash G(\Qp),E)_F$ (utiliser la version duale de \cite[Prop.~20.13]{Sch} puis la discussion apr\`es le Lemme \ref{supportcompact}). Par le Lemme \ref{librehat}, on a par ailleurs un isomorphisme~:
\begin{multline*}
(\pi_P)^\vee \widehat\otimes_{E} D(P(\Qp)\backslash G(\Qp),E)_F\\
\simeq (\pi_P)^\vee\widehat \otimes_{D(P(\Qp),E),\iota}\big(D(P(\Qp),E)\widehat\otimes_{E,\iota}D(P(\Qp)\backslash G(\Qp),E)_F\big)
\end{multline*}
d'o\`u le r\'esultat avec (\ref{densetenseur}).\\
(ii) s'obtient en consid\'erant le dual de la surjection topologique d'espaces de type compact~:
$$\big(\Ind_{P(\Qp)}^{G(\Qp)} \pi_P\big)^{\an}\twoheadrightarrow \big(\Ind_{P(\Qp)}^{G(\Qp)} \pi_P\big)^{\an}/ \big(\cInd_{P(\Qp)}^{G(\Qp)\backslash C}\pi_P\big)^{\an}=\pi_C$$
et en utilisant (\ref{dualsp}) et le (i).
\end{proof}

Le lemme suivant, qui explicite le cas $C=P(\Qp)$, sera utile au \S~\ref{premierscindage}.

\begin{lem}\label{support1}
Avec les notations ci-dessus, on a un isomorphisme d'espaces de Fr\'echet~:
$$(\pi_P)^\vee \widehat \otimes_{D(P(\Qp),E),\iota}D(G(\Qp),E)_{P(\Qp)}\simeq (\pi_P)^\vee\widehat\otimes_ED(N_P^-(\Qp),E)_{\{1\}}.$$
\end{lem}
\begin{proof}
Par \cite[Prop.~1.2.12]{Ko1}, on a un isomorphisme de $D(P(\Qp),E)$-modules \`a gauche s\'epar\'ement continus~:
\begin{equation}\label{ko1}
D(G(\Qp),E)_{P(\Qp)}\simeq D(P(\Qp),E)\widehat\otimes_{D(P(\Qp),E)_{\{1\}}}D(G(\Qp),E)_{\{1\}}.
\end{equation}
Comme $P(\Qp)N_P^-(\Qp)=P(\Qp)P^-(\Qp)$ est un ouvert de $G(\Qp)$ contenant $1$, on a un isomorphisme $D(P(\Qp)N_P^-(\Qp),E)_{\{1\}}\buildrel\sim\over\rightarrow D(G(\Qp),E)_{\{1\}}$. Comme la multiplication dans $G(\Qp)$ induit un isomorphisme de vari\'et\'es localement $\Qp$-analytiques $P(\Qp)\times N_P^-(\Qp)\buildrel\sim\over\rightarrow P(\Qp)N_P^-(\Qp)$, il suit du Lemme \ref{supproduit} que la multiplication dans l'alg\`ebre de Fr\'echet $D(G(\Qp),E)_{\{1\}}$ induit un isomorphisme de $D(P(\Qp),E)_{\{1\}}$-modules \`a gauche continus $D(P(\Qp),E)_{\{1\}}\widehat\otimes_ED(N_P^-(\Qp),E)_{\{1\}}\buildrel\sim\over\rightarrow D(G(\Qp),E)_{\{1\}}$. Par le Lemme \ref{librehat} et (\ref{ko1}), on en d\'eduit un isomorphisme $D(G(\Qp),E)_{P(\Qp)}\simeq D(P(\Qp),E)\widehat\otimes_ED(N_P^-(\Qp),E)_{\{1\}}$ de $D(P(\Qp),E)$-modules \`a gauche s\'epar\'ement continus. On obtient l'\'enonc\'e en appliquant le Lemme \ref{librehat} une deuxi\`eme fois.
\end{proof}

Soit maintenant $C'\subsetneq C\subseteq G(\Qp)$ deux ferm\'es distincts de $G(\Qp)$ stables par multiplication \`a gauche par $P(\Qp)$ tels que $C\backslash C'$ est un ouvert dense de $C$. On suppose que $P(\Qp)\backslash (C\backslash C')$ est une sous-vari\'et\'e localement $\Qp$-analytique paracompacte ferm\'ee de la vari\'et\'e $P(\Qp)\backslash (G(\Qp)\backslash C')$ (qui est une sous-vari\'et\'e ouverte de $P(\Qp)\backslash G(\Qp)$). La notation $P(\Qp)\backslash (C\backslash C')$ pouvant pr\^eter \`a confusion, pr\'ecisons qu'il s'agit de ``$C$ priv\'e de $C'$, le tout quotient\'e par l'action de $P(\Qp)$ par multiplication \`a gauche''. Noter que (\ref{varflagbis}) induit un isomorphisme~:
\begin{equation}\label{varflagbisc}
P(\Qp)\times (P(\Qp)\backslash (C\backslash C'))\buildrel\sim\over\longrightarrow C\backslash C'
\end{equation}
qui munit $C\backslash C'$ d'une structure de sous-vari\'et\'e localement $\Qp$-analytique ferm\'ee de $G(\Qp)\backslash C'$. Attention que $C$ et $C'$ {\it ne} sont {\it pas} eux-m\^emes suppos\'es \^etre des vari\'et\'es localement $\Qp$-analytiques (i.e. ils peuvent avoir des ``singularit\'es''). Par la Remarque \ref{rematopo} l'ouvert $P(\Qp)\backslash (G(\Qp)\backslash C')$ de la vari\'et\'e compacte $P(\Qp)\backslash G(\Qp)$ admet un recouvrement disjoint par un nombre {\it d\'enombrable} d'ouverts compacts. On en d\'eduit la m\^eme chose pour son ferm\'e $P(\Qp)\backslash (C\backslash C')$ en utilisant que l'intersection d'un ferm\'e et d'un compact est compacte. On note $(\cInd_{P(\Qp)}^{C\backslash C'}\pi_P)^{\an}$ l'espace des fonctions localement analytiques $f:C\backslash C'\rightarrow \pi_P$ \`a support compact modulo $P(\Qp)$ telles que $f(qc)=q(f(c))$ pour $q\in P(\Qp)$ et $c\in C\backslash C'$, qui s'identifie par (\ref{varflagbisc}) \`a l'espace de type compact $C^{\an}_c(P(\Qp)\backslash (C\backslash C'),\pi_P)$, cf. le Lemme \ref{supportcompact}. De plus les inclusions $D(G(\Qp),E)_{C'}\subseteq D(G(\Qp),E)_{C}\subseteq D(G(\Qp),E)$ induisent des immersions ferm\'ees d'espaces de Fr\'echet (nucl\'eaires) par le (ii) du Lemme \ref{dualspx}~:
\begin{multline}\label{morphci}
(\pi_P)^\vee \widehat \otimes_{D(P(\Qp),E),\iota}D(G(\Qp),E)_{C'}\hookrightarrow (\pi_P)^\vee \widehat \otimes_{D(P(\Qp),E),\iota}D(G(\Qp),E)_{C}\\
\hookrightarrow (\pi_P)^\vee \widehat \otimes_{D(P(\Qp),E),\iota}D(G(\Qp),E),
\end{multline}
et par le (i) du Lemme \ref{dualspx}, l'immersion de gauche est la duale de la surjection topologique d'espaces de type compact~:
\begin{multline}\label{piC}
\pi_C=\big(\Ind_{P(\Qp)}^{G(\Qp)} \pi_P\big)^{\an}/ \big(\cInd_{P(\Qp)}^{G(\Qp)\backslash C}\pi_P\big)^{\an}\\
\twoheadrightarrow \pi_{C'}=\big(\Ind_{P(\Qp)}^{G(\Qp)} \pi_P\big)^{\an}/ \big(\cInd_{P(\Qp)}^{G(\Qp)\backslash C'}\pi_P\big)^{\an}.
\end{multline}

La proposition suivante donne une ``approximation'' des ``gradu\'es'' $X_{C}/X_{C'}$ qui sera cruciale dans la preuve de la Proposition \ref{support} ci-dessous.

\begin{prop}\label{descgrdense}
On a un morphisme continu canonique d'image dense~:
$$\big((\cInd_{P(\Qp)}^{C\backslash C'}\pi_P)^{\an}\big)^\vee\otimes_{E,\iota}D(G(\Qp),E)_{\{1\}}\longrightarrow X_{C}/X_{C'}.$$
\end{prop}
\begin{proof}
Noter d'abord que par le (i) du Lemme \ref{dualspx} le quotient $X_{C}/X_{C'}$ s'identifie au quotient de l'espace de Fr\'echet $(\pi_P)^\vee \widehat \otimes_{D(P(\Qp),E),\iota}D(G(\Qp),E)_{C}$ par le sous-espace ferm\'e $(\pi_P)^\vee \widehat \otimes_{D(P(\Qp),E),\iota}D(G(\Qp),E)_{C'}$. En utilisant le Lemme \ref{supportcompact}, on a un diagramme commutatif de $E$-espaces vectoriels localement convexes s\'epar\'es avec applications continues~:
\begin{equation}\label{grdense}
\begin{gathered}
\xymatrix{ \big(\cInd_{P(\Qp)}^{C\backslash C'}\pi_P\big)^{\an} \ar@{^{(}->}[r] & \big(\Ind_{P(\Qp)}^{C\backslash C'}\pi_P\big)^{\an} \\
 \big(\cInd_{P(\Qp)}^{G(\Qp)\backslash C'}\pi_P\big)^{\an} \ar@{^{(}->}[r] \ar[u]^{\rm res} & \big(\Ind_{P(\Qp)}^{G(\Qp)}\pi_P\big)^{\an}\ar[u]^{\rm res} }
\end{gathered}
\end{equation}
o\`u les applications horizontales (resp. verticales) sont les injections canoniques (resp. les restrictions \`a $C\backslash C'$) et o\`u tous les espaces sont de type compact sauf \'eventuellement $(\Ind_{P(\Qp)}^{C\backslash C'}\pi_P)^{\an}$ (car $P(\Qp)\backslash (C\backslash C')$ n'est pas forc\'ement une vari\'et\'e compacte). Noter que la restriction \`a $C\backslash C'$ d'une fonction dans $(\cInd_{P(\Qp)}^{G(\Qp)\backslash C'}\pi_P)^{\an}$ tombe bien dans le sous-$E$-espace vectoriel $(\cInd_{P(\Qp)}^{C\backslash C'}\pi_P)^{\an}$ de $(\Ind_{P(\Qp)}^{C\backslash C'}\pi_P)^{\an}$ en utilisant que l'intersection d'un compact et d'un ferm\'e est encore un compact. Avec le Lemme \ref{dualspx}, le dual du diagramme (\ref{grdense}) devient le diagramme commutatif~:
\begin{equation}\small\label{grdensed}
\begin{gathered}
\xymatrix{\big((\cInd_{P(\Qp)}^{C\backslash C'}\pi_P)^{\an}\big)^\vee \ar[d]^{f}&\big((\Ind_{P(\Qp)}^{C\backslash C'}\pi_P)^{\an}\big)^\vee \ar[l]^{h} \ar[d]\\
\big((\pi_P)^\vee \widehat \otimes_{D(P(\Qp),E),\iota}D(G(\Qp),E)\big)/X_{C'} & (\pi_P)^\vee \widehat \otimes_{D(P(\Qp),E),\iota}D(G(\Qp),E). \ar[l] }
\end{gathered}
\end{equation}
Soit $(U_i)_{I}$ un recouvrement disjoint d\'enombrable de $P(\Qp)\backslash (C\backslash C')$ par des ouverts compacts, on a des isomorphismes topologiques (cf. la preuve du Lemme \ref{supportcompact})~:
\begin{eqnarray*}
\big((\cInd_{P(\Qp)}^{C\backslash C'}\pi_P)^{\an}\big)^\vee&\cong &\prod_{i\in I}C^{\an}(U_i,\pi_P)^\vee\\
\big((\Ind_{P(\Qp)}^{C\backslash C'}\pi_P)^{\an}\big)^\vee&\cong &\bigoplus_{i\in I}C^{\an}(U_i,\pi_P)^\vee
\end{eqnarray*}
qui montrent que l'application $h$ dans (\ref{grdensed}) est d'image dense. En utilisant que $D(G(\Qp),E)$ et $D(G(\Qp),E)_{C'}$ sont des $D(G(\Qp),E)_{\{1\}}$-modules \`a droite, le diagramme (\ref{grdensed}) s'\'etend naturellement en un diagramme commutatif de $D(G(\Qp),E)_{\{1\}}$-modules \`a droite avec applications $D(G(\Qp),E)_{\{1\}}$-lin\'eaires conti\-nues~:
\begin{equation}\footnotesize\label{grdensedo}
\begin{gathered}
\xymatrix{\big((\cInd_{P(\Qp)}^{C\backslash C'}\pi_P)^{\an}\big)^\vee\!\!\otimes_{E,\iota}\!D(G(\Qp),E)_{\{1\}} \ar[d]^{f\otimes \Id}&\big((\Ind_{P(\Qp)}^{C\backslash C'}\pi_P)^{\an}\big)^\vee\!\!\otimes_{E,\iota}\!D(G(\Qp),E)_{\{1\}} \ar[l]^{\ \ h\otimes \Id} \ar[d]\\
\big((\pi_P)^\vee \widehat \otimes_{D(P(\Qp),E),\iota}D(G(\Qp),E)\big)/X_{C'} & (\pi_P)^\vee \widehat \otimes_{D(P(\Qp),E),\iota}D(G(\Qp),E) \ar[l] }
\end{gathered}
\end{equation}
o\`u l'image de l'application $h\otimes \Id$ est encore dense par le Lemme \ref{tenseur}. En particulier l'adh\'erence de l'image de $f\otimes \Id$ est \'egale \`a l'adh\'erence de l'image de $(f\otimes \Id)\circ (h\otimes \Id)$ (utiliser que la compos\'ee de deux applications continues d'image dense est d'image dense). 

Soit $D(C,E)$ l'adh\'erence dans $D(G(\Qp),E)$ du sous-$E$-espace vectoriel engendr\'e par les distributions de Dirac $\delta_{c}$ pour $c\in C$ (lorsque le ferm\'e $C$ est une sous-vari\'et\'e localement $\Qp$-analytique de $G(\Qp)$, noter que cela co\"\i ncide bien avec les distributions localement analytiques sur $C$ par la preuve de \cite[Prop.~1.1.2]{Ko1} et la densit\'e des distributions de Dirac). Comme $C\backslash C'$ est un ouvert dense de $C$, on v\'erifie facilement que l'adh\'erence de $D(C\backslash C',E)$ dans $D(G(\Qp),E)$ est encore $D(C,E)$ (observer que cette adh\'erence est contenue dans $D(C,E)$ et pour la surjectivit\'e utiliser que si une suite $(c_{1,i})_{i\in I}$ de $C\backslash C'$ converge vers $c\in C$, alors $(\delta_{c_{1,i}})_{i\in I}$ converge vers $\delta_{c}$ dans $D(G(\Qp),E)$, par exemple en utilisant \cite[Lem.~5.10]{DLB}). Donc l'injection~:
$$D(C\backslash C',E)\otimes_{E,\iota}D(G(\Qp),E)_{\{1\}}\hookrightarrow D(C,E)\otimes_{E,\iota}D(G(\Qp),E)_{\{1\}}$$
est encore d'image dense par le Lemme \ref{tenseur}. Par ailleurs, il suit de la preuve de \cite[Lem.~1.2.10]{Ko1} et de \cite[Lem.~1.2.5]{Ko1} que l'adh\'erence de l'image de l'application naturelle $D(C,E)\otimes_{E,\iota}D(G(\Qp),E)_{\{1\}}\rightarrow D(G(\Qp),E)$ est exactement $D(G(\Qp),E)_{C}$. Donc l'adh\'erence de l'image de l'application $D(C\backslash C',E)\otimes_{E,\iota}D(G(\Qp),E)_{\{1\}}\rightarrow D(G(\Qp),E)$ est aussi $D(G(\Qp),E)_{C}$. 

Par ailleurs, par un argument similaire (en plus simple) \`a celui de la preuve du (i) du Lemme \ref{dualspx} en rempla\c cant les vari\'et\'es localement $\Qp$-analytiques $G(\Qp)$, $P(\Qp)\backslash G(\Qp)$ par les vari\'et\'es $C\backslash C'$, $P(\Qp)\backslash (C\backslash C')$ respectivement et en utilisant (\ref{varflagbisc}) et le Lemme \ref{librehat}, on a un isomorphisme topologique~:
\begin{equation*}
\big((\Ind_{P(\Qp)}^{C\backslash C'}\pi_P)^{\an}\big)^\vee\cong (\pi_P)^\vee \widehat \otimes_{D(P(\Qp),E),\iota}D(C\backslash C',E).
\end{equation*}
Ainsi la moiti\'e ``sud-est'' du diagramme (\ref{grdensedo}) revient \`a~:
\begin{multline}\label{moitie}
\big(\big(\pi_P)^\vee \widehat \otimes_{D(P(\Qp),E),\iota}D(C\backslash C',E)\big)\otimes_{E,\iota}D(G(\Qp),E)_{\{1\}}\\
\longrightarrow (\pi_P)^\vee \widehat \otimes_{D(P(\Qp),E),\iota}D(G(\Qp),E)\\
\twoheadrightarrow \big((\pi_P)^\vee \widehat \otimes_{D(P(\Qp),E),\iota}D(G(\Qp),E)\big)/X_{C'}.
\end{multline}
Comme les applications naturelles~:
$$(\pi_P)^\vee \otimes_{E,\iota}*\longrightarrow \big(\pi_P)^\vee \widehat \otimes_{D(P(\Qp),E),\iota}*$$
o\`u $*\in \{D(C\backslash C',E),D(G(\Qp),E)\}$ sont d'images denses, par le Lemme \ref{tenseur} et ce qui pr\'ec\`ede on en d\'eduit que la compos\'ee (\ref{moitie}) est d'image dense dans le sous-espace~:
$$\big((\pi_P)^\vee \widehat \otimes_{D(P(\Qp),E),\iota}D(G(\Qp),E)_{C}\big)/X_{C'}\cong X_{C}/X_{C'}$$
de $((\pi_P)^\vee \widehat \otimes_{D(P(\Qp),E),\iota}D(G(\Qp),E))/X_{C'}$. Comme il s'agit d'un sous-espace ferm\'e, on en d\'eduit finalement avec la phrase suivant (\ref{grdensedo}) que l'adh\'erence de l'image de $(f\otimes \Id)\circ (h\otimes \Id)$, donc aussi celle de l'image de $(f\otimes \Id)$, est $X_{C}/X_{C'}$, ce qui ach\`eve la preuve.
\end{proof}

Comme $X_{C}/X_{C'}$ est un espace complet, le morphisme dans la Proposition \ref{descgrdense} se factorise en un morphisme continu d'image dense~:
\begin{equation}\label{grcomplet}
\big((\cInd_{P(\Qp)}^{C\backslash C'}\pi_P)^{\an}\big)^\vee\widehat\otimes_{E,\iota}D(G(\Qp),E)_{\{1\}}\longrightarrow X_{C}/X_{C'}.
\end{equation}
Supposons maintenant que les ferm\'es $C$ et $C'$ sont de plus stables par multiplication {\it \`a droite} par $B(\Qp)$ dans $G(\Qp)$, de sorte que $(\cInd_{P(\Qp)}^{C\backslash C'}\pi_P)^{\an}$ est muni d'une action \`a gauche de $B(\Qp)$ par translation \`a droite sur les fonctions. Alors $((\cInd_{P(\Qp)}^{C\backslash C'}\pi_P)^{\an})^\vee$ est un $D(B(\Qp),E)$-module \`a droite et l'action \`a gauche duale de $B(\Qp)$, c'est-\`a-dire l'action $(bv)(f)\=v(b^{-1}f)$ pour $b\in B(\Qp)$, $v\in ((\cInd_{P(\Qp)}^{C\backslash C'}\pi_P)^{\an})^\vee$ et $f\in (\cInd_{P(\Qp)}^{C\backslash C'}\pi_P)^{\an}$, s'\'ecrit $b(v)=v\delta_{b^{-1}}$ o\`u $\delta_{b^{-1}}\in D(B(\Qp),E)$ est la distribution de Dirac associ\'ee \`a $b^{-1}$. On munit $((\cInd_{P(\Qp)}^{C\backslash C'}\pi_P)^{\an})^\vee\widehat\otimes_{E,\iota}D(G(\Qp),E)_{\{1\}}$ de l'action \`a gauche de $B(\Qp)$ donn\'ee par~:
\begin{equation}\label{actiondual}
b(v\otimes \mu)=b(v)\otimes \delta_b\mu\delta_{{b}^{-1}}
\end{equation}
pour $v\in ((\cInd_{P(\Qp)}^{C\backslash C'}\pi_P)^{\an})^\vee$ et $\mu\in D(G(\Qp),E)_{\{1\}}$. De m\^eme $X_{C}$ et $X_{C'}$ sont alors aussi des $D(B(\Qp),E)$-modules \`a droite o\`u l'action de $b\in B(\Qp)$ est induite par la multiplication \`a droite par $\delta_{b^{-1}}$ dans $D(G(\Qp),E)_{C}$ et $D(G(\Qp),E)_{C'}$ (via le (i) du Lemme \ref{dualspx}), et on v\'erifie en utilisant le diagramme (\ref{grdensedo}) que le morphisme (\ref{grcomplet}) est $B(\Qp)$-\'equivariant pour ces actions. Par ailleurs l'action (\`a gauche) de $B(\Qp)$ sur $D(G(\Qp),E)_{\{1\}}$ donn\'ee par la conjugaison par $\delta_b$ provient par dualit\'e (comme ci-dessus) de l'action \`a gauche de $B(\Qp)$ sur $C^\omega_1(G(\Qp),E)=\ilim{U}C^{\an}(U,E)$ (cf. (\ref{germe})) induite par~:
\begin{equation}\label{actiongerme}
f\in C^{\an}(U,E)\longmapsto (g\mapsto f(b^{-1}gb))\in C^{\an}(bUb^{-1},E),
\end{equation}
et on v\'erifie facilement que cette action de $B(\Qp)$ sur l'espace de type compact $C^\omega_1(G(\Qp),E)$ est une repr\'esentation localement analytique de $B(\Qp)$. Ainsi, par \cite[Prop.~20.13]{Sch} et \cite[Prop.~3.6.18]{Em}, l'action de $B(\Qp)$ en (\ref{actiondual}) provient par dualit\'e de la repr\'esentation localement analytique de $B(\Qp)$ agissant (de mani\`ere diagonale) sur l'espace de type compact~:
$$\big(\cInd_{P(\Qp)}^{C\backslash C'}\pi_P\big)^{\an}\widehat\otimes_{E}C^\omega_1(G(\Qp),E).$$
Notons~:
\begin{equation*}
\pi_{C/C'}\=\big(\cInd_{P(\Qp)}^{G(\Qp)\backslash C'} \pi_P\big)^{\an}/ \big(\cInd_{P(\Qp)}^{G(\Qp)\backslash C}\pi_P\big)^{\an},
\end{equation*}
alors $\pi_{C/C'}$ est aussi dans $\Rep(B(\Qp))$ et $\pi_{C/C'}^\vee\cong X_C/X_{C'}$ par (\ref{xc}) et (\ref{piC}).	

\begin{cor}\label{m'r'}
Supposons que $C$ et $C'$ soient de plus stables par multiplication \`a droite par $B(\Qp)$. Alors pour tout $m\in \Z_{\geq 0}$ le morphisme (\ref{grcomplet}) induit un morphisme continu d'image dense de $\Rm^+$-modules qui commute \`a $\psi$~:
$$M_\alpha\big(\big(\cInd_{P(\Qp)}^{C\backslash C'}\pi_P\big)^{\an}\widehat\otimes_{E}C^\omega_1(G(\Qp),E_m)\big)\longrightarrow M_\alpha(\pi_{C/C'}\otimes_EE_m)$$
o\`u $M_\alpha(-)$ est comme en (\ref{malpha}).
\end{cor}
\begin{proof}
L'existence du morphisme et sa compatibilit\'e aux structures de $\Rm^+$-modules et \`a $\psi$ d\'ecoulent de la commutation de (\ref{grcomplet}) \`a $B(\Qp)$. Par (\ref{malphaf}) il suffit de montrer que l'application~:
$$\Big(\big((\cInd_{P(\Qp)}^{C\backslash C'}\pi_P)^{\an}\big)^\vee\widehat\otimes_{E}D(G(\Qp),E_m)_{\{1\}}\Big)(\eta)^{\sep}_{N_m^\alpha}\longrightarrow ((X_{C}/X_{C'})\otimes_EE_m)(\eta)^{\sep}_{N_m^\alpha}$$
induite par (\ref{grcomplet}) (tensoris\'ee par $E_m$) est d'image dense. Mais c'est clair car $(-)\twoheadrightarrow (-)(\eta)^{\sep}_{N_m^\alpha}$ est un morphisme surjectif et l'image de (\ref{grcomplet}) est dense.
\end{proof}

\begin{rem}\label{grcomplettilde}
{\rm (i) Lorsque les ferm\'es $C$ et $C'$ sont stables par multiplication \`a droite par $B(\Qp)$ dans $G(\Qp)$, on v\'erifie aussi (avec (\ref{grdensedo}) encore) que le morphisme (\ref{grcomplet}) se factorise par un morphisme continu $B(\Qp)$-\'equivariant d'image dense~:
\begin{equation}\label{grcompletotimes}
\big((\cInd_{P(\Qp)}^{C\backslash C'}\pi_P)^{\an}\big)^\vee\widehat\otimes_{D(B(\Qp),E)_{\{1\}},\iota}D(G(\Qp),E)_{\{1\}}\longrightarrow X_{C}/X_{C'}.
\end{equation}
(ii) Un cas particulier de la Proposition \ref{descgrdense} est lorsque l'on prend $C'=\emptyset$ (donc $C=C\backslash C'$ est alors une vari\'et\'e localement $\Qp$-analytique paracompacte ferm\'ee dans $G(\Qp)$). Dans ce cas, (\ref{grcomplet}) donne un morphisme continu canonique d'image dense~:
$$\big((\Ind_{P(\Qp)}^{C}\pi_P)^{\an}\big)^\vee\widehat\otimes_{E,\iota}D(G(\Qp),E)_{\{1\}}\longrightarrow X_{C}$$
qui se factorise comme dans le (i), et le Corollaire \ref{m'r'} donne un morphisme continu d'image dense de $\Rm^+$-modules qui commute \`a $\psi$~:
$$M_\alpha\big(\big(\Ind_{P(\Qp)}^{C}\pi_P\big)^{\an}\widehat\otimes_{E}C^\omega_1(G(\Qp),E_m)\big)\longrightarrow M_\alpha(\pi_{C}\otimes_EE_m).$$}
\end{rem}

\subsection{Le cas des cellules non ouvertes}\label{celluledevisse}

On montre que le foncteur $F_\alpha$ est nul sur le compl\'ementaire de la cellule ouverte d'induites paraboliques localement analytiques.

On conserve toutes les notations du \S~\ref{approx} et on suppose de plus que l'action de $P(\Qp)$ sur $\pi_P$ se factorise par $L_P(\Qp)$. On note $\preceq$ l'ordre de Bruhat sur $W$, $\lg(w)$ la longueur d'un \'el\'ement $w\in W$, $w_0$ l'\'el\'ement de $W$ de longueur maximale, $W_{P}$ le groupe de Weyl de $L_{P}$ et~:
$$W_P^{\min}\=\{w\in W,\ \lg(w)\ {\rm est\ minimale\ dans\ la\ classe\ \grave a\ gauche}\ W_{P}w\}\subseteq W.$$
(Noter que $w_0\in W_P^{\min}$ si et seulement si $P=B$.) Chaque $P(\Qp) w B(\Qp)$ pour $w\in W$ est une vari\'et\'e localement $\Qp$-analytique paracompacte et si $w\in W_P^{\min}$ on a $\dim(P(\Qp) w B(\Qp))=\dim(P(\Qp))+\lg(w)$. De plus pour $w\in W_P^{\min}$ on a (o\`u $\overline{(-)}$ est l'adh\'erence de $(-)$ dans $G(\Qp)$)~:
$$\overline{P(\Qp) w B(\Qp)}=\!\!\!\coprod_{w'\in W_P^{\min}, w'\preceq w} \!\!\!P(\Qp) w' B(\Qp)\ \ {\rm et} \ \ G(\Qp)\!=\coprod_{w'\in W_P^{\min}} P(\Qp) w' B(\Qp)$$
avec $P(\Qp) w B(\Qp)$ ouvert (dense) dans $\overline{P(\Qp) w B(\Qp)}$ (toutes ces assertions se d\'eduisent par exemple de \cite[\S~2.3]{Ha} et des r\'ef\'erences dans {\it loc.cit.}). On d\'eduit de cela qu'il existe un entier $t\geq 1$ et une suite strictement croissante de ferm\'es $C_0\=P(\Qp)\subsetneq C_{1}\subsetneq \cdots \subsetneq C_t\subsetneq G(\Qp)$ tels que~:
\begin{equation}\small\label{decompbruhat}
C_i\backslash C_{i-1}=\!\!\!\coprod_{w\in W_P^{\min}, \lg(w)=i}\!\!\!P(\Qp) w B(\Qp),\ 1\leq i\leq t \ \ \ {\rm et}\ \ \ G(\Qp)\backslash C_t=P(\Qp) w_0 B(\Qp).
\end{equation}
De plus $C_i\backslash C_{i-1}$ pour $1\leq i\leq t$ est un ouvert dense de $C_i$ qui v\'erifie toutes les conditions avant et apr\`es (\ref{varflagbisc}) (notons que les vari\'et\'es $P(\Qp)\backslash (C_i\backslash C_{i-1})$ sont m\^eme ici affines). 

Nous aurons besoin du lemme suivant.

\begin{lem}\label{combigroupe}
(i) Notons $R_P$ les racines du Levi $L_{P}$ et soit $w\in W_P^{\min}$ tel que $w\notin W_{P} w_0$. Alors il existe $\gamma\in S$ tel que $w(\gamma)\in R^+\backslash (R^+\cap R_{P})$.\\
(ii) Soit $w\in W_P^{\min}$, alors on a $N(\Qp)\cap w^{-1}N(\Qp)w\buildrel\sim\over\rightarrow N(\Qp)\cap w^{-1}P(\Qp)w$.
\end{lem}
\begin{proof}
(i) En fait le m\^eme \'enonc\'e est vrai plus g\'en\'eralement avec n'importe quel \'el\'ement $w$ de $W$ tel que $w\notin W_Pw_0$. Supposons en effet $w(S)\cap (R^+\backslash R^+_P)=\emptyset$ (o\`u $R^+_P=R^+\cap R_P=$ racines positives de $L_P$), i.e. $w(S)\subseteq (-R^+)\cup R^+_P$ (puisque $R=(R^+\backslash R_P^+) \amalg ((-R^+)\cup R_P^+)$), alors $S\subseteq (-w^{-1}(R^+))\cup w^{-1}(R^+_P)$ et donc $B\subseteq w^{-1}P^-w$. En particulier $w^{-1}P^-w$ est un sous-groupe parabolique de $G$ contenant $B$. Mais il est conjugu\'e \`a $w_0P^-w_0$ qui est un autre sous-groupe parabolique de $G$ contenant $B$. Il s'ensuit que ces deux sous-groupes paraboliques de $G$ \`a la fois standard et conjugu\'es sont \'egaux, et donc $P^-=w_0w^{-1}P^-ww_0$. Cela implique $w_0w^{-1}\in W_{P}$ (= le groupe de Weyl de $L_P$), i.e. $w\in W_Pw_0$ ce qui contredit l'hypoth\`ese de d\'epart.\\
(ii) Par la preuve de \cite[Lem.~2.3.2]{Ha} (elle-m\^eme d\'eduite d'un lemme de Borel-Tits), on a un hom\'eomorphisme~:
$$B(\Qp)\backslash B(\Qp)wB(\Qp)\buildrel\sim\over\longrightarrow P(\Qp)\backslash P(\Qp)wB(\Qp).$$
Or le terme de gauche est isomorphe \`a $(N(\Qp)\cap w^{-1}N(\Qp)w)\backslash N(\Qp)$ et celui de droite \`a $(N(\Qp)\cap w^{-1}P(\Qp)w)\backslash N(\Qp)$, d'o\`u le r\'esultat.
\end{proof}

Si $K$ est une extension finie de $\Qp$, on note $U(\mg,K)\=K\otimes_{\Qp}U(\mg)$ et $U(\mb,K)\=K\otimes_{\Qp}U(\mb)$.

\begin{prop}\label{support}
On a~:
$$F_{\alpha}\Big(\big(\Ind_{P(\Qp)}^{G(\Qp)} \pi_P\big)^{\an}/ \big(\cInd_{P(\Qp)}^{P(\Qp)w_0B(\Qp)}\pi_P\big)^{\an}\Big)=0.$$
\end{prop}
\begin{proof}
{\bf \'Etape $1$}\\
Par (\ref{decompbruhat}), il faut montrer $F_{\alpha}(\pi_{C_t})=0$. Pour $i\in \{1,\dots,t\}$ on a une suite exacte (stricte) dans $\Rep(B(\Qp))$ (cf. (\ref{piC}))~:
$$0\longrightarrow \pi_{C_i/C_{i-1}}\longrightarrow \pi_{C_i}\longrightarrow \pi_{C_{i-1}}\rightarrow 0.$$ 
Par le (i) de la Proposition \ref{drex} et un d\'evissage \'evident, il suffit de montrer $F_{\alpha}(\pi_{C_0})=F_{\alpha}(\pi_{C_i/C_{i-1}})=0$ pour $i\in \{1,\dots,t\}$. Par la d\'efinition de $F_\alpha$ (cf. (\ref{foncteurm}) et (\ref{falpha})), il suffit de montrer que, pour $-\in \{\pi_{C_0},\pi_{C_1/C_{0}},\dots,\pi_{C_t/C_{t-1}}\}$, pour $m\in \Z_{\geq 0}$, $r\in \Q_{>p-1}$, $T_r$ un $(\varphi,\Gamma)$-module g\'en\'eralis\'e sur $\R^r$ et $f$ un morphisme $\Rm^+$-lin\'eaire continu commutant \`a $\psi$~:
$$f:M_\alpha(-\otimes_EE_m)\longrightarrow T_r\otimes_EE_m,$$
il existe $m'\gg m$ et $r'\gg r$ tels que la compos\'ee~:
\begin{equation}\small\label{composeemr}
M_\alpha(-\otimes_EE_{m'})\buildrel (\ref{augmentem})\over\longrightarrow M_\alpha(-\otimes_EE_m)\otimes_{E_m}E_{m'}\buildrel f\otimes \Id\over \longrightarrow (T_r\otimes_EE_m)\otimes_{E_m}E_{m'}\longrightarrow T_{r'}\otimes_EE_{m'}
\end{equation}
est nulle.

\noindent
{\bf \'Etape $2$}\\
Pour sa simplicit\'e (et pour rendre plus lisibles les \'etapes qui suivent), on traite d'abord le cas de $X_{C_0}$, m\^eme s'il est en fait un cas particulier des \'etapes suivantes. Fixons $m\in \Z_{\geq 0}$ tel que $\eta\vert_{N_P(\Qp)\cap N_m}\ne 1$, $r\in \Q_{>p-1}$, $T_r$ un $(\varphi,\Gamma)$-module g\'en\'eralis\'e sur $\R^r$ et $f:M_\alpha(\pi_{C_0}\otimes_EE_m)\longrightarrow T_r\otimes_EE_m$ un morphisme $\Rm^+$-lin\'eaire continu commutant \`a $\psi$. L'objectif de cette \'etape est de montrer que $f$ est nul (il n'est n\'ecessaire ici d'augmenter ni $m$ ni $r$). 

Par la Proposition \ref{descgrdense} et les (i) et (ii) de la Remarque \ref{grcomplettilde}, on a une application continue $P(\Qp)$-\'equivariante d'image dense~:
\begin{equation}\label{imagedense0}
(\pi_P)^\vee\widehat\otimes_{D(B(\Qp),E)_{\{1\}},\iota}D(G(\Qp),E)_{\{1\}}\longrightarrow X_{C_0}
\end{equation}
(en fait, elle se factorise m\^eme ici par $\widehat \otimes_{D(P(\Qp),E)_{\{1\}},\iota}$, mais nous n'en aurons pas besoin). Par \cite[Prop.~1.2.8]{Ko1} et le Lemme \ref{tenseur} on a une application continue d'image dense~:
\begin{equation}\label{imagedense}
(\pi_P)^\vee \otimes_{E,\iota}U(\mg,E)\longrightarrow (\pi_P)^\vee\widehat\otimes_{D(B(\Qp),E)_{\{1\}},\iota}D(G(\Qp),E)_{\{1\}}
\end{equation}
o\`u $U(\mg,E)$ est muni de la topologie localement convexe la plus fine. L'application (\ref{imagedense}) se factorise en une application continue $B(\Qp)$-\'equivariante d'image dense~:
\begin{equation}\label{imagedense5}
(\pi_P)^\vee \otimes_{U(\mb,E)}U(\mg,E)\longrightarrow (\pi_P)^\vee\widehat\otimes_{D(B(\Qp),E)_{\{1\}},\iota}D(G(\Qp),E)_{\{1\}}
\end{equation}
o\`u le terme de gauche est muni de la topologie quotient et o\`u l'action de $B(\Qp)$ y est d\'efinie par exactement la m\^eme formule qu'en (\ref{actiondual}) (la conjugaison par des distributions de Dirac pr\'eservant $U(\mg,E)$). En composant avec (\ref{imagedense0}) on en d\'eduit une application continue $B(\Qp)$-\'equivariante d'image dense~:
\begin{equation}\label{imagedenseX}
(\pi_P)^\vee \otimes_{U(\mb,E)}U(\mg,E)\longrightarrow X_{C_0}.
\end{equation}
Enfin (\ref{imagedenseX}) induit une application continue d'image dense en appliquant $(-\otimes_EE_m)(\eta)_{N_m^\alpha}$ des deux c\^ot\'es et en notant que $-\otimes_EE_m\rightarrow (-\otimes_EE_m)(\eta)_{N_m^\alpha}$ est continu et surjectif (pour la topologie quotient \`a droite)~:
\begin{equation}\label{imagedense5alpha}
\big((\pi_P)^\vee \otimes_{U(\mb,E)}U(\mg,E_m)\big)(\eta)_{N_m^\alpha}\longrightarrow (X_{C_0}\otimes_EE_m)(\eta)_{N_m^\alpha}.
\end{equation}
Comme on a une surjection continue $(X_{C_0}\otimes_EE_m)(\eta)_{N_m^\alpha}\twoheadrightarrow M_\alpha(\pi_{C_0}\otimes_EE_m)$ par (\ref{malphaf}), il suffit donc avec (\ref{imagedense5alpha}) de montrer que pour $m\in \Z_{\geq 0}$, $r\in \Q_{>p-1}$, $T_r$ un $(\varphi,\Gamma)$-module g\'en\'eralis\'e sur $\R^r$, tout morphisme continu commutant \`a l'action de $N_0/N_0^\alpha\simeq N_\alpha(\Qp)\cap N_0\simeq \Zp$~:
\begin{equation}\label{imagedense4}
\big((\pi_P)^\vee \otimes_{U(\mb,E)}U(\mg,E_m)\big)(\eta)_{N_m^\alpha}\longrightarrow T_r\otimes_EE_m
\end{equation}
est nul.

En \'ecrivant $U(\mg,E_m)=U(\mb,E)\otimes_EU(\mnn^-,E_m)$ et en consid\'erant l'action par adjonction de $N(\Qp)$ sur les \'el\'ements de $U(\mnn^-,E_m)$ (dans $U(\mg,E_m)$), on voit que $(\pi_P)^\vee \otimes_{U(\mb,E)}U(\mg,E_m)$ peut s'\'ecrire comme r\'eunion de sous-$U(\mb,E_m)$-modules \`a droite stables par l'action de $N(\Qp)$ qui sont des extensions successives finies de la $N(\Qp)$-repr\'esentation $(\pi_P\otimes_EE_m)^\vee$ (cet argument s'inspire de la discussion pr\'ec\'edant \cite[Lem.~8.4]{Ko2}). En uti\-lisant que le foncteur $(-)(\eta)_{N_m^\alpha}$ commute aux limites inductives filtrantes et est exact \`a droite, il suffit pour montrer (\ref{imagedense4}) de montrer qu'un morphisme de $\Rm^+$-modules~:
\begin{equation}\label{casqbeta}
(\pi_P\otimes_EE_m)^\vee(\eta)_{N_m^\alpha}\longrightarrow T_r\otimes_EE_m
\end{equation}
est nul. Supposons d'abord $\alpha\notin S_P$, alors l'action de $N_0/N_0^\alpha\simeq \Zp$ sur $(\pi_P)^\vee$ est triviale puisque $\pi_P$ est une repr\'esentation de $L_{P}(\Qp)$. En particulier $X\in \Rm^+$ annule le terme de gauche en (\ref{casqbeta}) alors qu'il est inversible sur celui de droite, d'o\`u le r\'esultat. Supposons maintenant $\alpha\in S_P$. Comme $N_P(\Qp)$ agit trivialement sur $(\pi_P)^\vee$ et $\eta\vert_{N_P(\Qp)\cap N_m}\ne 1$, on a en fait $(\pi_P\otimes_EE_m)^\vee(\eta)_{N_P(\Qp)\cap N_m}=0$, d'o\`u directement $(\pi_P\otimes_EE_m)^\vee(\eta)_{N^\alpha_m}=0$ puisque $N_P(\Qp)\cap N_m\subseteq N^\alpha_m$ dans ce cas.

\noindent
{\bf \'Etape 3}\\
Soit $i\in \{1,\dots,t\}$, $m\in \Z_{\geq 0}$, $r\in \Q_{>p-1}$, $T_r$ un $(\varphi,\Gamma)$-module g\'en\'eralis\'e sur $\R^r$ et fixons jusqu'\`a la fin de cette preuve un morphisme $\Rm^+$-lin\'eaire continu commutant \`a $\psi$~:
\begin{equation}\label{ffixe}
f:M_\alpha(\pi_{C_i/C_{i-1}}\otimes_EE_m)\longrightarrow T_r\otimes_EE_m.
\end{equation}
Rappelons qu'il suffit de montrer qu'il existe $m'\gg m$ et $r'\gg r$ tels que la compos\'ee (\ref{composeemr}) est nulle. 
Par le Corollaire \ref{m'r'}, il suffit de le montrer en rempla\c cant le morphisme $f$ par sa compos\'ee avec le morphisme~:
\begin{equation}\label{remplacantcomposee}
D:M_\alpha\big(\big(\cInd_{P(\Qp)}^{C_i\backslash C_{i-1}}\pi_P\big)^{\an}\widehat\otimes_{E}C^\omega_1(G(\Qp),E_m)\big)\longrightarrow M_\alpha(\pi_{C_i/C_{i-1}}\otimes_EE_m).
\end{equation}
Par (\ref{decompbruhat}), \cite[Lem.~1.2.13]{Ko1} (et ce qui pr\'ec\`ede le Corollaire \ref{m'r'}), on a un isomorphisme dans $\Repm(B(\Qp))$~:
\begin{multline*}
(\cInd_{P(\Qp)}^{C_i\backslash C_{i-1}}\pi_P)^{\an}\widehat\otimes_EC^\omega_1(G(\Qp),E_m)\\
\cong \bigoplus_{w\in W_{P}^{\min},\lg(w)=i}(\cInd_{P(\Qp)}^{P(\Qp)w B(\Qp)}\pi_P)^{\an}\widehat\otimes_{E}C^\omega_1(G(\Qp),E_m),
\end{multline*}
et il suffit donc de le montrer en rempla\c cant $(\cInd_{P(\Qp)}^{C_i\backslash C_{i-1}}\pi_P)^{\an}$ dans (\ref{remplacantcomposee}) par $(\cInd_{P(\Qp)}^{P(\Qp)w B(\Qp)}\pi_P)^{\an}$.

\noindent
{\bf \'Etape 4}\\
Consid\'erons un morphisme $\Rm^+$-lin\'eaire continu commutant \`a $\psi$ quelconque~:
\begin{equation}\label{ffixew}
h:M_\alpha\big(\big(\cInd_{P(\Qp)}^{P(\Qp)w B(\Qp)}\pi_P\big)^{\an}\widehat\otimes_{E}C^\omega_1(G(\Qp),E_m)\big)\longrightarrow T_r\otimes_EE_m.
\end{equation}
On montre dans cette \'etape que $h$ se factorise par un quotient ``\`a support born\'e'' de la source (comme dans la preuve du Lemme \ref{niveaum}).

Soit $(\pi_P)^{w}$ la repr\'esentation de $w^{-1} {P}(\Qp)w$ donn\'ee par $\pi_P$ avec $g\in w^{-1} {P}(\Qp)w$ agissant par $w gw^{-1}\in {P}(\Qp)$. Comme dans la discussion avant le Lemme \ref{niveaum}, on v\'erifie en utilisant (\ref{varflagbisc}) avec $P(\Qp)w B(\Qp)$ au lieu de $C\backslash C'$ puis le (ii) du Lemme \ref{combigroupe} que l'on a un isomorphisme topologique $B(\Qp)$-\'equivariant~:
\begin{eqnarray}\label{passagan}
\big(\cInd_{P(\Qp)}^{P(\Qp)w B(\Qp)}\pi_P\big)^{\an}&\buildrel\sim\over\longrightarrow &\big(\cInd_{N(\Qp)\cap w^{-1} N(\Qp)w}^{N(\Qp)}(\pi_P)^{w}\big)^{\an}\\
\nonumber H&\longmapsto &\big(g\mapsto H(w g)\big)
\end{eqnarray}
o\`u \`a droite il s'agit des fonctions localement analytiques sur $N(\Qp)$ \`a support compact modulo $N(\Qp)\cap w^{-1} N(\Qp)w$ (v\'erifiant l'\'equation fonctionnelle) avec l'action de $N(\Qp)$ donn\'ee par la translation \`a droite sur les fonctions et l'action de $T(\Qp)$ donn\'ee par $(t\lambda)(g)=t(\lambda(t^{-1}gt))$ pour $\lambda\in (\cInd_{N(\Qp)\cap w^{-1} N(\Qp)w}^{N(\Qp)}(\pi_P)^{w})^{\an}$ (o\`u $t\in T(\Qp)\subseteq w^{-1} P(\Qp)w$ agit sur $\lambda(t^{-1}gt)$ par $(\pi_P)^{w}$). Notons~:
$$N_{w<0}(\Qp)\=\prod_{\gamma\in R^+,w(\gamma)<0}N_\gamma(\Qp),$$
alors on a $N_{w<0}(\Qp)\buildrel\sim\over \rightarrow (N(\Qp)\cap w^{-1} N(\Qp)w)\backslash N(\Qp)$ (pour un ordre quelconque sur les $\gamma\in R^+$ tels que $w(\gamma)<0$) d'o\`u on d\'eduit~:
\begin{equation}\label{restrw}
\big(\cInd_{N(\Qp)\cap w^{-1} N(\Qp)w}^{N(\Qp)}(\pi_P)^{w}\big)^{\an}\buildrel\sim\over\longrightarrow C_c^{\an}\big(N_{w<0}(\Qp),(\pi_P)^{w}\big)
\end{equation}
o\`u \`a droite rappelons qu'il s'agit des fonctions localement analytiques \`a support compact (cf. \S~\ref{topo}). Comme $N_m$ est totalement d\'ecompos\'e (cf. (\ref{decomp})), on a~:
$$\big(w^{-1} N(\Qp)w\cap N_m\big) \times \big(N_{w<0}(\Qp)\cap N_m\big)\buildrel\sim\over\longrightarrow N_m$$
d'o\`u un isomorphisme~:
\begin{equation}\label{restrwm}
\big(\Ind_{N_m\cap w^{-1} N(\Qp)w}^{N_m}(\pi_P)^{w}\big)^{\an}\buildrel\sim\over\longrightarrow C^{\an}\big(N_{w<0}(\Qp)\cap N_m,(\pi_P)^{w}\big)
\end{equation}
et avec (\ref{restrw})~:
\begin{equation}\label{injwm}
\ilim{m} \big(\Ind_{N_m\cap w^{-1} N(\Qp)w}^{N_m}(\pi_P)^{w}\big)^{\an}\buildrel\sim\over\longrightarrow \big(\cInd_{N(\Qp)\cap w^{-1} N(\Qp)w}^{N(\Qp)}(\pi_P)^{w}\big)^{\an}
\end{equation}
o\`u les applications de transition \`a gauche sont clairement injectives. De plus $(\Ind_{N_m\cap w^{-1} N(\Qp)w}^{N_m}(\pi_P)^{w})^{\an}$, vu dans $(\cInd_{N(\Qp)\cap w^{-1} N(\Qp)w}^{N(\Qp)}(\pi_P)^{w})^{\an}$, est stable par l'action ci-dessus de $N_m$ et $\lambda_{\alpha^{\!\vee}}(\Zp\backslash\{0\})$, ainsi que par celle d'un sous-groupe ouvert compact suffisamment petit $T_0$ de $T(\Qp)$ (par exemple $T(\Zp)$ lorsque $G=\G$). Comme dans (\ref{decompo0}), on peut d\'ecomposer~:
\begin{multline}\small\label{decompo}
C^{\an}_c\big(N_{w<0}(\Qp),(\pi_P)^{w}\big)\simeq C^{\an}\big(N_{w<0}(\Qp)\cap N_m,(\pi_P)^{w}\big)\\
\bigoplus \Big(\bigoplus_{m'\geq m+1}C^{\an}\big(N_{w<0}(\Qp)\cap (N_{m'}\backslash N_{m'-1}),(\pi_P)^{w}\big)\Big)
\end{multline}
o\`u l'on v\'erifie que {\it chaque} facteur direct dans le terme de droite est stable par l'action de $N_m$ et $\lambda_{\alpha^{\!\vee}}(\Zp^\times)$. Par exemple, pour la stabilit\'e de $C^{\an}(N_{w<0}(\Qp)\cap (N_{m'}\backslash N_{m'-1}),(\pi_P)^{w})$ (vu dans $(\Ind_{N_{m'}\cap w^{-1} N(\Qp)w}^{N_{m'}}(\pi_P)^{w})^{\an}$ via (\ref{restrwm})) par translation \`a droite par $N_m$ pour $m\leq m'-1$, on utilise (cf. (\ref{decomp}))~:
\begin{multline*}
N_{w<0}(\Qp)\cap (N_{m'}\backslash N_{m'-1})=\{(n_\gamma)\in \Pi_{\gamma\in R^+,w(\gamma)<0}(N_\gamma(\Qp)\cap N_{m'})\\
{\rm t.q.}\ \exists \ \gamma\ {\rm avec}\ n_\gamma\in N_\gamma(\Qp)\cap (N_{m'}\backslash N_{m'-1})\}
\end{multline*}
et les r\`egles de commutation entre les $N_\gamma$ pour $\gamma\in R^+$ (cf. \cite[II.1.2(5)]{Ja}), et on en d\'eduit en particulier que l'``entr\'ee'' $n_\gamma$ en une racine $\gamma$ qui est de longeur minimale parmi les racines $\gamma'$ telles que $w(\gamma')<0$ et $n_{\gamma'}\in N_{\gamma'}(\Qp)\cap (N_{m'}\backslash N_{m'-1})$ reste dans $N_\gamma(\Qp)\cap (N_{m'}\backslash N_{m'-1})$ apr\`es translation \`a droite par un \'el\'ement quelconque de $N_m$ (ce qui assure que l'on reste bien dans $C^{\an}(N_{w<0}(\Qp)\cap (N_{m'}\backslash N_{m'-1}),(\pi_P)^{w})$). De plus, toujours dans le terme de droite en (\ref{decompo}), l'action de $\lambda_{\alpha^{\!\vee}}(\Zp\backslash\{0\})$ envoie $C^{\an}(N_{w<0}(\Qp)\cap (N_{m'}\backslash N_{m'-1}),(\pi_P)^{w})$ dans $C^{\an}(N_{w<0}(\Qp)\cap N_{m'},(\pi_P)^{w})$. Par \cite[Lem.~1.2.13]{Ko1} on d\'eduit de (\ref{decompo}) un isomorphisme d'espaces de type compact~:
\begin{multline}\label{decompo2}
\big(\cInd_{N(\Qp)\cap w^{-1} N(\Qp)w}^{N(\Qp)}(\pi_P)^{w}\big)^{\an}\widehat \otimes_{E}C^\omega_1(G(\Qp),E_m)\\
\simeq C^{\an}\big(N_{w<0}(\Qp))\cap N_m,(\pi_P)^{w}\big)\widehat \otimes_{E}C^\omega_1(G(\Qp),E_m)\\
\bigoplus \bigoplus_{m'\geq m+1}\!\big(C^{\an}\big(N_{w<0}(\Qp)\cap (N_{m'}\backslash N_{m'-1}),(\pi_P)^{w}\big)\widehat \otimes_{E}C^\omega_1(G(\Qp),E_m)\big).
\end{multline}
Notons que les deux termes dans l'isomorphisme (\ref{decompo2}) sont munis de l'action diagonale de $N_m$ et $\lambda_{\alpha^{\!\vee}}(\Zp\backslash \{0\})$ (par (\ref{actiongerme}) sur $C^\omega_1(G(\Qp),E_m)$), et que (\ref{decompo2}) commute \`a ces actions. On d\'eduit de (\ref{decompo2}) et de \cite[Prop.~9.10]{Sch} un isomorphisme de $\Rm^+$-modules commutant \`a $\psi$~:
\begin{multline}\label{decompo3}
M_\alpha\Big(\big(\cInd_{N(\Qp)\cap w^{-1} N(\Qp)w}^{N(\Qp)}(\pi_P)^{w}\big)^{\an}\widehat \otimes_{E}C^\omega_1(G(\Qp),E_m)\Big)\\
\simeq M_\alpha\Big(C^{\an}\big(N_{w<0}(\Qp))\cap N_m,(\pi_P)^{w}\big)\widehat \otimes_{E}C^\omega_1(G(\Qp),E_m)\Big) \\
\bigoplus\prod_{m'\geq m+1}M_\alpha\Big(\!\big(C^{\an}\big(N_{w<0}(\Qp)\cap (N_{m'}\backslash N_{m'-1}),(\pi_P)^{w}\big)\widehat \otimes_{E}C^\omega_1(G(\Qp),E_m)\big)\Big)
\end{multline}
o\`u $M_\alpha(-)$ est comme en (\ref{malpha}), o\`u chaque facteur direct \`a droite est un $\Rm^+$-module mais o\`u $\psi$ (qui vient de l'action de $\lambda_{\alpha^{\!\vee}}(p^{-1})$ sur $M_\alpha(-)$, cf. le d\'ebut du \S~\ref{def}) envoie $M_\alpha(C^{\an}(N_{w<0}(\Qp)\cap (N_{m'}\backslash N_{m'-1}),(\pi_P)^{w})\widehat \otimes_{E}C^\omega_1(G(\Qp),E_m))$ seulement dans~:
$$\bigoplus_{m''\geq m'}M_\alpha\Big(C^{\an}\big(N_{w<0}(\Qp)\cap (N_{m''}\backslash N_{m''-1}),(\pi_P)^{w}\big)\widehat \otimes_{E}C^\omega_1(G(\Qp),E_m)\Big).$$
Le m\^eme argument que dans la preuve du Lemme \ref{niveaum} (plus pr\'ecis\'ement l'argument bornant le support de la fonction $f$ qui suit (\ref{pourciter})), qui est aussi le m\^eme argument que pour d\'emontrer la surjectivit\'e du morphisme (\ref{topologie}), montre alors qu'il existe un entier $m''\gg m$ tel que le morphisme $h$ en (\ref{ffixew}) se factorise par le quotient (via les isomorphismes (\ref{passagan}) et (\ref{decompo3}))~:
\begin{multline*}
M_\alpha\Big(C^{\an}\big(N_{w<0}(\Qp))\cap N_m,(\pi_P)^{w}\big)\widehat \otimes_{E}C^\omega_1(G(\Qp),E_m)\Big)\\
\bigoplus \bigoplus_{m''\geq m'\geq m+1}M_\alpha\Big(\big(C^{\an}\big(N_{w<0}(\Qp)\cap (N_{m'}\backslash N_{m'-1}),(\pi_P)^{w}\big)\widehat \otimes_{E}C^\omega_1(G(\Qp),E_m)\big)\Big).
\end{multline*}
De mani\`ere \'equivalente, par le m\^eme raisonnement que pour montrer (\ref{decompo2}), (\ref{decompo3}) mais invers\'e et en utilisant (\ref{restrwm}), le morphisme $h$ en (\ref{ffixew}) se factorise par $M_\alpha((\Ind_{N_{m''}\cap w^{-1} N(\Qp)w}^{N_{m''}}(\pi_P)^{w})^{\an}\widehat \otimes_{E}C^\omega_1(G(\Qp),E_m))$. En appliquant cela au morphisme $h=f\circ D$ de l'\'Etape $3$ (avec $P(\Qp)w B(\Qp)$ au lieu de $C_i\backslash C_{i-1}$), on obtient que $f\circ D$ se factorise en un morphisme $\Rm^+$-lin\'eaire continu commutant \`a $\psi$ (encore not\'e $f\circ D$)~:
\begin{equation}\label{fronD}
f\circ D : M_\alpha\big(\big(\Ind_{N_{m''}\cap w^{-1} N(\Qp)w}^{N_{m''}}(\pi_P)^{w}\big)^{\an}\widehat \otimes_{E}C^\omega_1(G(\Qp),E_m)\big)\longrightarrow T_r\otimes_EE_m
\end{equation}
pour un $m''\gg m$. Il suffit donc de montrer qu'il existe $m'\gg m$ et $r'\gg r$ tels que la compos\'ee analogue \`a (\ref{composeemr}) avec $f\circ D$ en (\ref{fronD}) au lieu de $f$ est nulle. 

\noindent
{\bf \'Etape 5}\\
Par le (i) du Lemme \ref{combigroupe}, il existe une racine simple $\gamma$ telle que $N_\gamma(\Qp) \subseteq w^{-1}N(\Qp)w$ et $N_\gamma(\Qp) $ agit trivialement sur la $w^{-1}P(\Qp)w$-repr\'esentation $(\pi_P)^{w}$ (rappelons que l'action de $P(\Qp)$ sur $\pi_P$ se factorise par $L_{P}(\Qp)$). Dans cette \'etape, on suppose que l'on peut prendre $\gamma=\alpha$ et on termine la preuve dans ce cas.

Par (\ref{malphaf}) et la discussion pr\'ec\'edant le Corollaire \ref{m'r'}, le terme de gauche dans (\ref{fronD}) est un quotient de~:
\begin{equation}\label{quotientde}
\Big(\big(\big(\Ind_{N_{m''}\cap w^{-1} N(\Qp)w}^{N_{m''}}(\pi_P)^{w}\big)^{\an}\big)^\vee\widehat\otimes_E D(G(\Qp),E_m)_{\{1\}}\Big)(\eta)_{N_m^\alpha}
\end{equation}
et il suffit donc de montrer l'analogue de (\ref{fronD}) en rempla\c cant ce terme de gauche par (\ref{quotientde}). Comme dans le (i) de la Remarque \ref{grcomplettilde}, le (nouveau) morphisme $f\circ D$ se factorise par un morphisme $\Rm^+$-lin\'eaire continu commutant \`a $\psi$ (encore not\'e $f\circ D$)~:
\begin{multline*}
f\circ D : \Big(\big(\big(\Ind_{N_{m''}\cap w^{-1} N(\Qp)w}^{N_{m''}}(\pi_P)^{w}\big)^{\an}\big)^\vee\widehat\otimes_{D(B(\Qp),E)_{\{1\}},\iota}D(G(\Qp),E_m)_{\{1\}}\Big)(\eta)_{N_m^\alpha} \\
\longrightarrow T_r\otimes_EE_m
\end{multline*}
(noter que $((\Ind_{N_{m''}\cap w^{-1} N(\Qp)w}^{N_{m''}}(\pi_P)^{w})^{\an})^\vee$ est bien un $D(B(\Qp),E)_{\{1\}}$-module \`a droite car $D(B(\Qp),E)_{\{1\}}=D(B_0,E)_{\{1\}}$ pour tout sous-groupe ouvert compact $B_0$ de $B(\Qp)$ par (\ref{germe})). En argumentant maintenant exactement comme dans la premi\`ere moiti\'e de l'\'Etape $2$ (cf. ce qui suit (\ref{imagedense0})), il suffit alors de montrer que tout morphisme continu commutant \`a l'action de $N_0/N_0^\alpha\simeq N_\alpha(\Qp)\cap N_0\simeq \Zp$~:
\begin{equation}\label{morphismegl3}
\Big(\big(\big(\Ind_{N_{m''}\cap w^{-1} N(\Qp)w}^{N_{m''}}(\pi_P)^{w}\big)^{\an}\big)^\vee\otimes_{U(\mb,E)}U(\mg,E_m)\Big)(\eta)_{N_m^\alpha}\longrightarrow T_r\otimes_EE_m
\end{equation}
devient nul dans $T_{r'}\otimes_EE_m$ pour $r'\gg r$ o\`u la topologie \`a gauche est comme dans (\ref{imagedense5alpha}) (en fait dans ce cas il ne sera pas n\'ecessaire d'augmenter $m$).

Comme $N_\alpha(\Qp) $ agit trivialement sur $(\pi_P)^{w}$, l'action de $N_\alpha(\Qp)\cap N_0$ sur $(\Ind_{N_{m''}\cap w^{-1} N(\Qp)w}^{N_{m''}}(\pi_P)^{w})^{\an}$ est alors donn\'ee par~:
\begin{equation}\label{actionalpha}
\lambda \longmapsto \big(n_{m''}\mapsto \lambda(n_{m''}n_\alpha)=\lambda(n_\alpha^{-1}n_{m''}n_\alpha)\big).
\end{equation}
Notons $N^{(2)}\=\prod_{\delta\in R^+\backslash S}N_\delta$, qui est un sous-groupe normal de $N$ tel que $N/N^{(2)}$ est ab\'elien, et $N^{(2)}_m\=N^{(2)}(\Qp)\cap N_m$, il existe un sous-groupe ouvert compact suf\-fisamment petit $p^M\Zp$ de $N_\alpha(\Qp)\cap N_0\cong \Zp$ tel que pour tout $n_\alpha\in p^M\Zp$ et tout $n_{m''}\in N_{m''}$ on a~:
$$n_\alpha^{-1}n_{m''}n_\alpha\in n_{m''}N^{(2)}_m$$
(on utilise que les ``d\'enominateurs en $p$'' dans $N_{m''}$ sont born\'es). Cela implique que l'action de $n_\alpha\in p^M\Zp$ induite par (\ref{actionalpha}) sur le sous-espace~:
$$\big(\big(\Ind_{N_{m''}\cap w^{-1} N(\Qp)w}^{N_{m''}}(\pi_P)^{w}\big)^{\an}\big)^{N^{(2)}_m}$$
est triviale, donc {\it a fortiori} aussi sur~:
\begin{multline}\label{casgamma=alpha}
\big((\Ind_{N_{m''}\cap w^{-1} N(\Qp)w}^{N_{m''}}(\pi_P)^{w})^{\an}\otimes_EE_m\big)(\eta^{-1})^{N_m^\alpha}\\
\subseteq \big((\Ind_{N_{m''}\cap w^{-1} N(\Qp)w}^{N_{m''}}(\pi_P)^{w})^{\an}\otimes_EE_m\big)^{N^{(2)}_m}
\end{multline}
(rappelons que $\eta\vert_{N^{(2)}(\Qp)}=1$). Donc $(1+X)^{p^M}-1\in \Rm^+$ agit par $0$ sur le dual de (\ref{casgamma=alpha}). Soit maintenant $r'\gg r$ tel que $(1+X)^{p^M}-1$ est inversible dans ${\mathcal R}^{r'}_{E_m}$ (cf. la preuve du Lemme \ref{Falphanul}). Comme dans la deuxi\`eme moiti\'e de l'\'Etape $2$, en \'ecrivant $U(\mg,E_m)=U(\mb,E)\otimes_EU(\mnn^-,E_m)$ et en consid\'erant l'action par adjonction de $N(\Qp)$ sur $U(\mnn^-,E_m)$, on voit que $((\Ind_{N_{m''}\cap w^{-1} N(\Qp)w}^{N_{m''}}(\pi_P)^{w})^{\an})^\vee\otimes_{U(\mb,E)}U(\mg,E_m)$ s'\'ecrit comme r\'eunion de sous-$U(\mb,E_m)$-modules \`a droite stables par $N_{m''}$ qui sont des extensions successives finies de la $N_{m''}$-repr\'esentation $((\Ind_{N_{m''}\cap w^{-1} N(\Qp)w}^{N_{m''}}(\pi_P)^{w})^{\an})^\vee\otimes_EE_m$. En utilisant que le foncteur $(-)(\eta)_{N_m^\alpha}$ commute aux limites inductives filtrantes et est exact \`a droite, il suffit pour montrer (\ref{morphismegl3}) de montrer qu'un morphisme continu de $\Rm^+$-modules~:
$$\big((\Ind_{N_{m''}\cap w^{-1} N(\Qp)w}^{N_{m''}}(\pi_P)^{w})^{\an}\otimes_EE_m\big)^\vee(\eta)_{N_m^\alpha}\longrightarrow T_{r'}\otimes_EE_m$$
est nul. Mais un tel morphisme se factorise par~:
$$\big(\big(\Ind_{N_{m''}\cap w^{-1} N(\Qp)w}^{N_{m''}}(\pi_P)^{w}\big)^{\an}\otimes_EE_m\big)^\vee(\eta)^{\sep}_{N_m^\alpha}$$
qui est le dual de (\ref{casgamma=alpha}) par la preuve du (i) du Lemme \ref{gnouf}. Or on a vu que $(1+X)^{p^M}-1$ annule ce (nouveau) terme de gauche alors qu'il est inversible sur celui de droite ($=T_{r'}\otimes_EE_m$) par choix de $r'$, donc un tel morphisme est nul. Ceci ach\`eve la preuve lorsque l'on peut prendre $\gamma=\alpha$.

\noindent
{\bf \'Etape 6}\\
On suppose maintenant que la racine simple $\gamma$ au d\'ebut de l'\'Etape $5$ est distincte de $\alpha$, i.e. $N_\gamma\subseteq N^\alpha$ (cf. \S~\ref{prel}), et on ach\`eve la preuve.

Quitte \`a augmenter $m$ comme dans (\ref{composeemr}) (avec $f\circ D$ en (\ref{fronD}) au lieu de $f$), on peut supposer $m=m''$ dans (\ref{fronD}). Quitte \`a augmenter encore $m''$ si n\'ecessaire, on peut de plus supposer $\eta\vert_{N_\gamma(\Qp)\cap N_{m''}}\ne 1$. Par le m\^eme raisonnement que dans la premi\`ere partie de l'\'Etape $5$, il suffit donc de montrer que tout morphisme continu commutant \`a l'action de $N_0/N_0^\alpha\simeq N_\alpha(\Qp)\cap N_0\simeq \Zp$~:
\begin{equation*}
\Big(\big(\big(\Ind_{N_{m''}\cap w^{-1} N(\Qp)w}^{N_{m''}}(\pi_P)^{w}\big)^{\an}\big)^\vee\otimes_{U(\mb,E)}U(\mg,E_{m''})\Big)(\eta)_{N_{m''}^\alpha}\longrightarrow T_r\otimes_EE_{m''}
\end{equation*}
est nul (dans ce cas il ne sera pas n\'ecessaire d'augmenter $r$). \'Ecrivant $U(\mg,E_m)=U(\mb,E)\otimes_EU(\mnn^-,E_m)$ et raisonnant comme dans la deuxi\`eme moiti\'e de l'\'Etape $5$ ou de l'\'Etape $2$, il suffit de montrer~:
$$\big(\big(\Ind_{N_{m''}\cap w^{-1} N(\Qp)w}^{N_{m''}}(\pi_P)^{w}\big)^{\an}\otimes_EE_{m''}\big)^\vee(\eta)^{\sep}_{N_{m''}^\alpha}=0$$
c'est-\`a-dire par la preuve du (i) du Lemme \ref{gnouf}~:
$$\big(\big(\Ind_{N_{m''}\cap w^{-1} N(\Qp)w}^{N_{m''}}(\pi_P)^{w}\big)^{\an}\otimes_EE_{m''}\big)(\eta^{-1})^{N_{m''}^\alpha}=0.$$
Comme dans l'\'Etape $5$, l'action de $n_\gamma\in N_\gamma(\Qp)\cap N_{m''}$ sur $(\Ind_{N_{m''}\cap w^{-1} N(\Qp)w}^{N_{m''}}(\pi_P)^{w})^{\an}$ est donn\'ee par $\lambda\longmapsto (n_{m''}\mapsto \lambda(n_{m''}n_\gamma)=\lambda(n_\gamma^{-1}n_{m''}n_\gamma))$ avec $n_\gamma^{-1}n_{m''}n_\gamma\in n_{m''}N_{m''}^{(2)}$, donc elle est triviale sur le sous-espace $((\Ind_{N_{m''}\cap w^{-1} N(\Qp)w}^{N_{m''}}(\pi_P)^{w})^{\an})^{N_{m''}^{(2)}}$. Comme $\eta^{-1}\vert_{N_{m''}^{(2)}}=1$ mais $\eta^{-1}\vert_{N_\gamma(\Qp)\cap N_{m''}}\ne 1$, on en d\'eduit~:
\begin{multline*}
\Big(\big(\Ind_{N_{m''}\cap w^{-1} N(\Qp)w}^{N_{m''}}(\pi_P)^{w}\big)^{\an}\otimes_EE_{m''}\big)^{N_{m''}^{(2)}}\Big)(\eta^{-1})^{N_\gamma(\Qp)\cap N_{m''}}\\
=\big(\Ind_{N_{m''}\cap w^{-1} N(\Qp)w}^{N_{m''}}(\pi_P)^{w}\big)^{\an}\otimes_EE_{m''}\big)(\eta^{-1})^{N_{m''}^{(2)}(N_\gamma(\Qp)\cap N_{m''})}=0
\end{multline*}
d'o\`u le r\'esultat puisque $N_{m''}^{(2)}(N_\gamma(\Qp)\cap N_{m''})\subseteq N_{m''}^\alpha$. Ceci ach\`eve la preuve de la proposition.
\end{proof}

\subsection{Foncteurs $F_\alpha$ et repr\'esentations ${\mathcal F}_{P^-}^G(M,\pi_P^\infty)$}\label{endofproof}

On \'enonce et d\'emontre les r\'esultats principaux du \S~\ref{alphaqques}.

On conserve toutes les notations pr\'ec\'edentes. On rappelle que les caract\`eres alg\'ebriques $\chi_\lambda:\Gal(E_\infty/E)\rightarrow E^\times$ pour $\lambda\in X(T)$ sont d\'efinis avant le Th\'eor\`eme \ref{caslisse}.

\begin{thm}\label{plusgeneral}
Soit $P\subseteq G$ un sous-groupe parabolique contenant $B$ tel que $\alpha\notin S_P$, $L(-\lambda)_P$ une repr\'esentation alg\'ebrique irr\'eductible de $L_P(\Qp)$ sur $E$ (o\`u $\lambda\in X(T)$), $\pi_P^\infty$ une repr\'esentation lisse de longueur finie de $L_P(\Qp)$ sur $E$ admettant un caract\`ere central $\chi_{\pi_P^\infty}$ et $d_{\pi_P^\infty}\=\dim_E(\pi_P^\infty\otimes_EE_\infty)(\eta^{-1})_{N_{L_{P}}\!(\Qp)}$. Alors le foncteur $F_\alpha\big(\big(\Ind_{P^-(\Qp)}^{G(\Qp)}L(-\lambda)_P\otimes_E\pi_P^\infty\big)^{\an}\big)$ est isomorphe \`a~:
$$E_\infty(\chi_{-\lambda})\otimes_E \Hom_{(\varphi,\Gamma)}\Big(\R\big((\lambda\circ \lambda_{\alpha^{\!\vee}})(\chi_{\pi_P^\infty}^{-1}\circ \lambda_{\alpha^{\!\vee}})\big)^{\oplus d_{\pi_P^\infty}},-\Big).$$
En particulier $F_\alpha\big(\big(\Ind_{P^-(\Qp)}^{G(\Qp)}L(-\lambda)_P\otimes_E\pi_P^\infty\big)^{\an}\big)$ est nul si et seulement si $\pi_P^\infty$ n'a pas de constituant g\'en\'erique.
\end{thm}
\begin{proof}
On note $Q\=w_0P^-w_0\subseteq G$. Comme $w_0(S)=-S$, c'est un sous-groupe parabolique contenant $B$. On note $\widetilde \pi_Q$ la repr\'esentation localement alg\'ebrique de longueur finie de $L_{Q}(\Qp)$ donn\'ee par $L(-\lambda)_{P}\otimes_E\pi_P^\infty$ mais avec $g\in L_{Q}(\Qp)$ agissant par $w_0gw_0\in L_{P}(\Qp)$. On a un isomorphisme $G(\Qp)$-\'equivariant~:
\begin{equation}\label{de-a+}
\big(\Ind_{P^-(\Qp)}^{G(\Qp)} L(-\lambda)_{P}\otimes_E\pi_P^\infty\big)^{\an}\buildrel\sim\over\longrightarrow \big(\Ind_{Q(\Qp)}^{G(\Qp)} \widetilde \pi_Q\big)^{\an},\ \ f\longmapsto \big(g\mapsto f(w_0g)\big)
\end{equation}
qui envoie le sous-espace ferm\'e $(\cInd_{P^-(\Qp)}^{P^-(\Qp)P(\Qp)}L(-\lambda)_{P}\otimes_E\pi_P^\infty)^{\an}$ sur le sous-espace ferm\'e $(\cInd_{Q(\Qp)}^{Q(\Qp)w_0P(\Qp)}\widetilde \pi_Q)^{\an}$. Une application du th\'eor\`eme $90$ de Hilbert montre que le $E_\infty$-espace vectoriel $(\pi_P^\infty\otimes_EE_\infty)(\eta^{-1})_{N_{L_P}(\Qp)}$ de dimension $d_{\pi_P^\infty}$ muni de l'action semi-lin\'eaire de $\Gal(E_\infty/E)$ d\'efinie dans la preuve du (i) du Lemme \ref{m+1} est isomorphe \`a $E_\infty^{\oplus d_{\pi_P^\infty}}$ (avec action semi-lin\'eaire \'evidente de $\Gal(E_\infty/E)$ sur chaque facteur). Par la Proposition \ref{celluleouverte} avec (\ref{deIaC}) et le (i) de la Proposition \ref{drex}, on voit donc qu'il suffit de montrer~:
\begin{equation}\label{zero}
F_\alpha\Big(\big(\Ind_{Q(\Qp)}^{G(\Qp)} \widetilde \pi_Q\big)^{\an}/ \big(\cInd_{Q(\Qp)}^{Q(\Qp)w_0P(\Qp)}\widetilde \pi_Q\big)^{\an}\Big)=0.
\end{equation}
Mais on a $Q(\Qp)w_0P(\Qp)=Q(\Qp)w_0B(\Qp)$ puisque~:
\begin{multline}\label{baP}
Q(\Qp)w_0B(\Qp)\subseteq Q(\Qp)w_0P(\Qp)=Q(\Qp)(w_0L_{P}(\Qp)w_0)w_0N_{P}(\Qp)\\
=Q(\Qp)L_{Q}(\Qp)w_0N_{P}(\Qp)=Q(\Qp)w_0N_{P}(\Qp)\subseteq Q(\Qp)w_0B(\Qp),
\end{multline}
donc (\ref{zero}) suit de la Proposition \ref{support} (appliqu\'ee avec $P=Q$).
\end{proof}

Soit $P\subseteq G$ un sous-groupe parabolique contenant $B$, si $\alpha\in S_P$, alors on peut d\'efinir le foncteur $F_\alpha$ pour le triplet $(L_P, B\cap L_P,T)$ au lieu du triplet $(G,B,T)$ en gardant le m\^eme cocaract\`ere $\lambda_{\alpha^\vee}:{\mathbb G}_{\rm m}\rightarrow T$, le foncteur ne d\'ependant pas des autres choix par la Proposition \ref{choix} (noter que le centre $Z_{L_P}$ de $L_P$ est bien encore connexe, cf. \cite[\S~6]{Br2}).

\begin{thm}\label{encoreplusgeneral}
Soit $P\subseteq G$ un sous-groupe parabolique contenant $B$ tel que $\alpha\in S_P$ et $\pi_P$ une repr\'esentation localement analytique de $L_P(\Qp)$ sur un $E$-espace vectoriel de type compact. Alors on a un isomorphisme de foncteurs dans $F(\varphi,\Gamma)_\infty$~:
$$F_\alpha\big(\big(\Ind_{P^-(\Qp)}^{G(\Qp)} \pi_P\big)^{\an}\big)\cong F_\alpha(\pi_P).$$
\end{thm}
\begin{proof}
Par la Proposition \ref{choix}, pour d\'efinir $F_\alpha$ pour $(L_P, B\cap L_P,T)$ on peut choisir $(\iota_\alpha)_{\alpha\in S_P}$, $\eta$ et $(L_P(\Qp)\cap N_m)_{m\geq 0}$ (o\`u $(\iota_\alpha)_{\alpha\in S}$, $\eta$ et $(N_m)_{\geq 0}$ sont comme au \S~\ref{prel} pour $(G, B,T)$). On a des produits semi-directs analogues \`a ceux au d\'ebut de la preuve de la Proposition \ref{celluleouverte}~: $N^\alpha=N_{P}\rtimes N_{L_{P}}^\alpha$, $\mnn^\alpha\simeq \mnn_{P}\rtimes \mnn_{L_{P}}^\alpha$ et $N_{m}^\alpha=(N_{P}(\Qp)\cap N_{m}) \rtimes (N_{L_{P}}^\alpha(\Qp)\cap N_{m})$ pour tout $m\in \Z_{\geq 0}$. Par le Lemme \ref{niveaum}, la Proposition \ref{support} (en raisonnant comme dans la preuve du Th\'eor\`eme \ref{plusgeneral}) et le (i) de la Proposition \ref{drex}, il suffit de montrer $F_{\alpha,m}(C^{\an}(N_{P}(\Qp)\cap N_m,\pi_P))\cong F_{\alpha,m}(\pi_P)$ pour tout $m\in \Z_{\geq 0}$. On a des isomorphismes d'espaces de type compact (cf. (\ref{actionP}))~:
\begin{eqnarray*}
C^{\an}(N_{P}(\Qp)\cap N_m,\pi_P)[\mnn^\alpha]&=&(C^{\an}(N_{P}(\Qp)\cap N_m,\pi_P)[\mnn_{P}])[\mnn_{L_{P}}^\alpha]\\
&=&C^{\infty}(N_{P}(\Qp)\cap N_m,\pi_P)[\mnn_{L_{P}}^\alpha]\\
&\cong & (C^{\infty}(N_{P}(\Qp)\cap N_m,E)\otimes_E\pi_P))[\mnn_{L_{P}}^\alpha]\\
&\cong & C^{\infty}(N_{P}(\Qp)\cap N_m,E)\otimes_E(\pi_P[\mnn_{L_{P}}^\alpha])
\end{eqnarray*}
(rappelons que les topologies projective et injective co\"\i ncident sur les deux produits tensoriels du bas, qui sont par ailleurs d\'ej\`a complets par le m\^eme argument que pour (\ref{sortie})). On en d\'eduit (avec le (ii) du Lemme \ref{gnouf}) un isomorphisme d'espaces de type compact~:
\begin{multline*}
C^{\an}(N_{P}(\Qp)\cap N_m,\pi_P\otimes_EE_m)[\mnn^\alpha](\eta^{-1})_{N_m^\alpha}\\
\cong \big(C^{\infty}(N_{P}(\Qp)\cap N_m,E_m)(\eta^{-1})_{N_{P}(\Qp)\cap N_m}\otimes_{E_m}( \pi_P\otimes_EE_m)[\mnn_{L_{P}}^\alpha]\big)(\eta^{-1})_{N_{L_{P}}^\alpha(\Qp)\cap N_{m}}.
\end{multline*}
Comme dans la preuve de la Proposition \ref{celluleouverte}, $C^{\infty}(N_{P}(\Qp)\cap N_m,E_m)(\eta^{-1})_{N_{P}(\Qp)\cap N_m}$ a dimension $1$ engendr\'e par la classe de $\eta\vert_{N_{P}(\Qp)\cap N_m}$ et l'action de $\Gal(E_\infty/E)$ dessus est celle sur $E_m$ par Hilbert 90. Par (\ref{actionP}) on voit de plus que $N_{L_{P}}^\alpha(\Qp)\cap N_{m}$ et $\lambda_{\alpha^{\!\vee}}(\Zp\backslash \{0\})$ agissent trivialement sur $\eta\vert_{N_{P}(\Qp)\cap N_m}$. On en d\'eduit un isomorphisme d'espaces de type compact compatible aux actions de $N_m/N_m^\alpha\buildrel\sim\over\rightarrow (N_{L_{P}}(\Qp)\cap N_{m})/(N_{L_{P}}^\alpha(\Qp)\cap N_{m})$, $\lambda_{\alpha^{\!\vee}}(\Zp\backslash \{0\})$ et $\Gal(E_\infty/E)$~:
$$C^{\an}(N_{P}(\Qp)\cap N_m,\pi_P\otimes_EE_m)[\mnn^\alpha](\eta^{-1})_{N_m^\alpha}\cong (\pi_P\otimes_EE_m)[\mnn_{L_{P}}^\alpha](\eta^{-1})_{N_{L_{P}}^\alpha(\Qp)\cap N_{m}}$$
d'o\`u le r\'esultat avec (\ref{foncteurm}) et (\ref{malpha}).
\end{proof}

\begin{cor}\label{generaljoli}
Soit $P\subseteq G$ un sous-groupe parabolique contenant $B$, $\pi_P^\infty$ une repr\'esentation lisse de longueur finie de $L_P(\Qp)$ sur $E$ admettant un caract\`ere central, $M$ un objet quelconque de ${\mathcal O}^{\mpp^-}_{\rm alg}$ et $\pi\={\mathcal F}_{P^-}^G(M,\pi_P^\infty)$. Alors il existe des caract\`eres alg\'ebriques distincts $\chi_{\lambda_1},\dots,\chi_{\lambda_r}$ de $\Gal(E_\infty/E)$ et des $(\varphi,\Gamma)$-modules g\'en\'eralis\'es $D_{\alpha,1}(\pi),\dots ,D_{\alpha,r}(\pi)$ sur $\R$ tels que $F_{\alpha}(\pi)\simeq \bigoplus_{i=1}^r\big(E_\infty(\chi_{\lambda_i})\otimes_E\Hom_{(\varphi,\Gamma)}(D_{\alpha,i}(\pi),-)\big)$.
\end{cor}
\begin{proof}
Soit $W$ une sous-repr\'esentation alg\'ebrique de $L_P(\Qp)$ de dimension finie sur $E$ qui engendre $M$, de sorte que l'on a une surjection $U(\mg)\otimes_{U(\mpp^-)}W\twoheadrightarrow M$. Comme le noyau de cette surjection est aussi dans ${\mathcal O}^{\mpp^-}_{\rm alg}$, on peut encore l'\'ecrire comme quotient d'une surjection analogue. Puisque la cat\'egorie des repr\'esentations alg\'ebriques de $L_P(\Qp)$ sur $E$ est semi-simple, on voit qu'il existe $\mu_1,\dots,\mu_t,\lambda_1,\dots,\lambda_s\in X(T)$ et une suite exacte dans ${\mathcal O}^{\mpp^-}_{\rm alg}$~:
\begin{equation*}
\bigoplus_{i=1}^tU(\mg)\otimes_{U(\mpp^-)}L^-(\mu_i)_P\longrightarrow \bigoplus_{i=1}^sU(\mg)\otimes_{U(\mpp^-)}L^-(\lambda_i)_P\longrightarrow M\longrightarrow 0
\end{equation*}
qui donne une suite exacte de repr\'esentations localement analytiques~:
\begin{multline}
0\longrightarrow {\mathcal F}_{P^-}^G(M,\pi_P^\infty)\longrightarrow \bigoplus_{i=1}^s\big(\Ind_{P^-(\Qp)}^{G(\Qp)}L(-\lambda_i)_P\otimes_E\pi_P^\infty\big)^{\an}\\
\longrightarrow \bigoplus_{i=1}^t\big(\Ind_{P^-(\Qp)}^{G(\Qp)}L(-\mu_i)_P\otimes_E\pi_P^\infty\big)^{\an}.
\end{multline}
Quitte \`a remplacer $\pi_P^\infty$ par une somme de copies de $\pi_P^\infty$, on peut mettre ensemble les $\lambda_i$ (resp. $\mu_i$) qui sont \'egaux. Le r\'esultat d\'ecoule alors facilement du Th\'eor\`eme \ref{plusgeneral} (lorsque $\alpha\notin S_P$), du Th\'eor\`eme \ref{encoreplusgeneral} et du Th\'eor\`eme \ref{caslisse} appliqu\'e avec le groupe r\'eductif $G=L_P$ (lorsque $\alpha\in S_P$), et du (ii) de la Proposition \ref{drex} avec la Remarque \ref{plusieurschi}.
\end{proof}

\begin{rem}\label{catabelienne}
{\rm (i) Par la Proposition \ref{drex} avec la Remarque \ref{plusieurschi}, le foncteur $\pi\mapsto D_\alpha(\pi)\=\oplus_iD_{\alpha,i}(\pi)$ envoie une suite exacte $0\rightarrow \pi''\rightarrow \pi\rightarrow \pi'$ de repr\'esentations de la forme ${\mathcal F}_{P^-}^G(M,\pi_P^\infty)$ (pour $P$, $\pi_P^\infty$, $M$ variables comme dans le (i) du Corollaire \ref{generaljoli}) vers une suite exacte $D_\alpha(\pi')\rightarrow D_\alpha(\pi)\rightarrow D_\alpha(\pi'')\rightarrow 0$ de $(\varphi,\Gamma)$-modules g\'en\'eralis\'es sur $\R$. Mais il suit facilement du Th\'eor\`eme \ref{liesocle} ci-dessous que ce foncteur {\it n'}est {\it pas} exact (\`a gauche).\\
(ii) L'hypoth\`ese que $\pi_P^\infty$ admet un caract\`ere central ne devrait pas \^etre n\'ecessaire pour le Corollaire \ref{generaljoli}. Elle n'intervient que via le caract\`ere $\chi_{\pi_P^\infty,\alpha}=\chi_{\pi_P^\infty}\circ \lambda_{\alpha^\vee}$ de la Proposition \ref{celluleouverte}, mais cette derni\`ere peut probablement se g\'en\'eraliser (au moins lorsque $\pi_P^\infty$ est de longueur finie) en rempla\c cant (les copies de) $\R\big((\lambda\circ \lambda_{\alpha^\vee})\chi_{\pi_P^\infty,\alpha}^{-1}\big)$ par un $(\varphi,\Gamma)$-module non forc\'ement semi-simple.}
\end{rem}

Nous allons maintenant expliciter $F_\alpha({\mathcal F}_{P^-}^G(M,\pi_P^\infty))$ dans certains cas utiles. Pour cela, nous avons d'abord besoin de plusieurs lemmes pr\'eliminaires.

\begin{lem}\label{singular}
Soit $\lambda\in X(T)$ tel que $\langle \lambda, \alpha^\vee\rangle \leq 0$.\\
(i) Le constituant $L^-(s_\alpha\cdot \lambda)$ appara\^\i t avec multiplicit\'e $1$ dans $U(\mg)\otimes_{U(\mb^-)}\lambda$.\\
(ii) Le $U(\mg)$-module $U(\mg)\otimes_{U(\mb^-)}\lambda$ admet en quotient une unique extension non scind\'ee $L^-(s_\alpha\cdot \lambda)-L^-(\lambda)$ dans ${\mathcal O}^{\mb^-}_{\rm alg}$ (o\`u $L^-(s_\alpha\cdot \lambda)$ est en sous-objet).
\end{lem}
\begin{proof}
Cela d\'ecoule de \cite[Th.~1.4.2(i)]{Ir} et de propri\'et\'es standard des polyn\^omes de Kazhdan-Lusztig. L'extension en (ii) est non scind\'ee puisque $L^-(\lambda)$ engendre tout $U(\mg)\otimes_{U(\mb^-)}\lambda$.
\end{proof}

\begin{lem}\label{quotientM}
Soit $P\subseteq G$ un sous-groupe parabolique contenant $B$, $L(-\lambda)_P$ une repr\'esentation alg\'ebrique irr\'eductible de $L_P(\Qp)$ sur $E$ et $M$ un quotient de $U(\mg)\otimes_{U(\mpp^-)}L^-(\lambda)_{P}$. On suppose $\alpha\notin S_P$, $\langle \lambda,\alpha\rangle \leq 0$ et on note $Q$ le parabolique contenant $P$ tel que $S_Q=S_P\amalg \{\alpha\}$. Alors $M$ est dans ${\mathcal O}^{\mq^-}_{\rm alg}$ (o\`u $\mq^-$ est l'alg\`ebre de Lie de $Q^-(\Qp)$) si et seulement si $M$ ne contient pas le constituant $L^-(s_\alpha\cdot\lambda)$.
\end{lem}
\begin{proof}
Par \cite[Th.~9.4(b)]{Hu} on a un diagramme commutatif de suites exactes de $U(\mg)$-modules~:
\begin{equation*}
\begin{gathered}
\xymatrix{\bigoplus_{\beta\in S_P}U(\mg)\otimes_{U(\mb^-)}s_\beta\cdot\lambda\ar[r]\ar@{^{(}->}[d]&U(\mg)\otimes_{U(\mb^-)}\lambda\ar@{=}[d] \ar[r]& U(\mg)\otimes_{U(\mpp^-)}L^-(\lambda)_{P}\ar@{^{}->>}[d] \ar[r]&0\\
\bigoplus_{\beta\in S_Q}U(\mg)\otimes_{U(\mb^-)}s_\beta\cdot\lambda\ar[r] & U(\mg)\otimes_{U(\mb^-)}\lambda \ar[r] & U(\mg)\otimes_{U(\mq^-)}L^-(\lambda)_{Q}\ar[r]&0}
\end{gathered}
\end{equation*}
qui par une chasse au diagramme triviale donne une suite exacte~:
\begin{equation}\label{encoresuite}
U(\mg)\otimes_{U(\mb^-)}s_\alpha\cdot\lambda\longrightarrow U(\mg)\otimes_{U(\mpp^-)}L^-(\lambda)_{P} \longrightarrow U(\mg)\otimes_{U(\mq^-)}L^-(\lambda)_{Q} \longrightarrow 0.
\end{equation}
Comme $L^-(s_\alpha\cdot\lambda)$ est l'unique quotient irr\'eductible de $U(\mg)\otimes_{U(\mb^-)}s_\alpha\cdot\lambda$ et qu'il a multiplicit\'e $1$ dans $U(\mg)\otimes_{U(\mpp^-)}L^-(\lambda)_{P} $ par le (i) du Lemme \ref{singular}, on en d\'eduit qu'un quotient $M$ de $U(\mg)\otimes_{U(\mpp^-)}L^-(\lambda)_{P}$ ne contient pas $L^-(s_\alpha\cdot\lambda)$ si et seulement si c'est en fait un quotient de $U(\mg)\otimes_{U(\mq^-)}L^-(\lambda)_{Q}$. On en d\'eduit l'\'enonc\'e avec les propri\'et\'es de la cat\'egorie ${\mathcal O}^{\mq^-}_{\rm alg}$, cf. \cite[\S~1]{OS} et \cite[\S~9.4]{Hu}.
\end{proof}

\begin{thm}\label{liesocle}
Soit $P\subseteq G$ un sous-groupe parabolique contenant $B$, $L(-\lambda)_P$ une repr\'esentation alg\'ebrique irr\'eductible de $L_P(\Qp)$ sur $E$, $\pi_P^\infty$ une repr\'esentation lisse de longueur finie de $L_P(\Qp)$ sur $E$ et $M$ un quotient non nul de $U(\mg)\otimes_{U(\mpp^-)}L^-(\lambda)_{P}$. On suppose ou bien $\alpha \in S_P$, ou bien $\alpha\notin S_P$ et $M\notin {\mathcal O}^{\mq^-}_{\rm alg}$ pour $Q$ le parabolique contenant $P$ tel que $S_Q=S_P\amalg \{\alpha\}$ et $\mq^-$ l'alg\`ebre de Lie de $Q^-(\Qp)$. Alors l'injection $G(\Qp)$-\'equivariante~:
\begin{equation*}
{\mathcal F}_{P^-}^G(M,\pi_P^\infty)\hookrightarrow \big(\Ind_{P^-(\Qp)}^{G(\Qp)} L(-\lambda)_{P}\otimes_E\pi_P^\infty\big)^{\an}
\end{equation*}
induit un isomorphisme dans $F(\varphi,\Gamma)_\infty$~:
$$F_\alpha\big({\mathcal F}_{P^-}^G(M,\pi_P^\infty)\big)\buildrel\sim\over\longrightarrow F_\alpha\big(\big(\Ind_{P^-(\Qp)}^{G(\Qp)} L(-\lambda)_{P}\otimes_E\pi_P^\infty\big)^{\an}\big).$$
\end{thm}
\begin{proof}
{\bf \'Etape 1}\\
Soit $\pi_1\!\=\!{\mathcal F}_{P^-}^G(M,\pi_P^\infty)$ et $\pi_2\!\=\!\big(\Ind_{P^-(\Qp)}^{G(\Qp)} L(-\lambda)_{P}\otimes_E\pi_P^\infty\big)^{\an}=\big(\Ind_{P^-(\Qp)}^{G(\Qp)} L^-(\lambda)_{P}^\vee\otimes_E\pi_P^\infty\big)^{\an}$, nous allons montrer que l'injection $\pi_1\hookrightarrow \pi_2$ induit un isomorphisme $\pi_1[{\mnn^\alpha}]\buildrel\sim\over\rightarrow \pi_2[{\mnn^\alpha}]$, ce qui implique le r\'esultat par (\ref{malpha}). Si $\alpha\in S_P$ nous allons montrer cette assertion pour $M=L^-(\lambda)$, ce qui suffit par le (i) de la Proposition \ref{drex}. Si $\alpha\notin S_P$, on distingue deux cas. Ou bien on a $\langle \lambda,\alpha\rangle >0$ (i.e. $L^-(\lambda)\notin {\mathcal O}^{\mq^-}_{\rm alg}$) et nous allons encore montrer le r\'esultat pour $M=L^-(\lambda)$. Ou bien on a $\langle \lambda,\alpha\rangle \leq 0$, auquel cas le Lemme \ref{quotientM} implique que $L^-(s_\alpha\cdot \lambda)$ est un constituant de $M$ et le Lemme \ref{singular} implique alors que l'on a une surjection $M\twoheadrightarrow (L^-(s_\alpha\cdot \lambda)-L^-(\lambda))$ o\`u $L^-(s_\alpha\cdot \lambda)-L^-(\lambda)$ est l'extension du (ii) du Lemme \ref{singular}. Dans ce dernier cas, par le (i) de la Proposition \ref{drex}, il suffit de montrer le r\'esultat pour $M=L^-(s_\alpha\cdot \lambda)-L^-(\lambda)$. Dans la suite de la preuve, on suppose donc $M=L^-(\lambda)$ si $\alpha\in S_P$ ou si $\alpha\notin S_P$ mais $\langle \lambda,\alpha\rangle >0$, et $M=L^-(s_\alpha\cdot \lambda)-L^-(\lambda)$ si $\alpha\notin S_P$ et $\langle \lambda,\alpha\rangle \leq 0$.

\noindent
{\bf \'Etape 2}\\
Fixons une num\'erotation $\alpha_1,\dots,\alpha_h$ des racines de $R^+$ telle que $\alpha_1=\alpha$. Tout \'el\'ement de $U(\mnn)$ est une combinaison lin\'eaire \`a coefficients dans $E$ de mon\^omes $y_{\alpha_1}^{m_1}y_{\alpha_2}^{m_2}\cdots y_{\alpha_h}^{m_h}$ o\`u $m_i\in \Z_{\geq 0}$ et $y_{\alpha_i}$ est une base de la $\Qp$-alg\`ebre de Lie de $N_{\alpha_i}(\Qp)$ (cf. \cite[Th.~1.2(a)]{Hu}). Nous allons d'abord montrer que tout vecteur dans~:
\begin{equation}\label{ker}
\ker\big(U(\mg)\otimes_{U(\mpp^-)}L^-(\lambda)_{P}\twoheadrightarrow M\big)
\end{equation}
est l'image via $U(\mg)\otimes_{U(\mb^-)}\lambda\twoheadrightarrow U(\mg)\otimes_{U(\mpp^-)}L^-(\lambda)_{P}$ d'une combinaison lin\'eaire de vecteurs de poids de la forme $y_{\alpha_1}^{m_1}y_{\alpha_2}^{m_2}\cdots y_{\alpha_t}^{m_t}\otimes v_{\lambda}$ o\`u $v_{\lambda}$ est le vecteur de plus haut poids (par rapport \`a $B^-$) de $U(\mg)\otimes_{U(\mb^-)}\lambda$ et o\`u $t$ et les $m_i$ sont des entiers positifs ou nuls tels que $2\leq t\leq h$ et $m_t\ne 0$. Par \cite[\S\S~5.1~\&~5.2]{Hu} tout poids $\mu\in X(T)$ apparaissant dans (\ref{ker}) est tel que $\mu\in \lambda' + \sum_{i}\Z_{\geq 0}\alpha_i$ pour $\lambda'\in X(T)$ tel que $\lambda' \uparrow \lambda$ et $\lambda'\ne \lambda$, i.e. il existe $r\geq 1$ et $\gamma_1,\dots,\gamma_r\in R^+$ tels que $(s_{\gamma_j}\cdots s_{\gamma_r})\cdot \lambda - (s_{\gamma_{j+1}}\cdots s_{\gamma_r})\cdot \lambda\in \sum_{i}\Z_{\geq 0}\alpha_i$ et est non nul pour $j\in \{1,\dots,r\}$. De plus, comme les constituants irr\'eductibles de $U(\mg)\otimes_{U(\mpp^-)}L^-(\lambda)_{P}$ sont dans ${\mathcal O}^{\mpp^-}_{\rm alg}$, on a $\langle \lambda', \gamma^\vee\rangle \leq 0$ pour $\gamma\in S_P$ par \cite[Prop.~9.3(e)]{Hu}. Montrons que l'on a $\lambda'-\lambda=\sum_{i=1}^hm'_i\alpha_i$ pour $m'_i\in \Z_{\geq 0}$ avec au moins un $m'_i\ne 0$ pour $i\geq 2$. Sinon, on a $\lambda'-\lambda\in \Z_{\geq 0}\alpha$, et donc n\'ecessairement $\gamma_j=\alpha$ pour tout $j$ (puisque $\alpha$ est simple), c'est-\`a-dire $r=1$ et $\gamma_1=\alpha$, i.e. $\lambda'=s_\alpha \cdot \lambda=\lambda + m'_1\alpha$ pour $m'_1\geq 1$ (car $\lambda'\ne \lambda$). Par ailleurs $s_\alpha \cdot \lambda=s_\alpha(\lambda -\rho)+\rho=s_\alpha(\lambda)+\alpha$, d'o\`u $s_\alpha(\lambda)=\lambda + (m'_1-1)\alpha$ ce qui implique $\langle s_\alpha (\lambda), \alpha^\vee\rangle = \langle \lambda, \alpha^\vee\rangle + 2(m'_1-1)=-\langle \lambda, \alpha^\vee\rangle$ et donc $\langle \lambda, \alpha^\vee\rangle = 1-m'_1\leq 0$. Supposons d'abord $\alpha\in S_P$, alors $\langle s_\alpha \cdot \lambda, \alpha^\vee\rangle=\langle \lambda', \alpha^\vee\rangle\leq 0$ (cf. ci-dessus), et aussi $\langle s_\alpha \cdot \lambda, \alpha^\vee\rangle=\langle \lambda, \alpha^\vee\rangle + 2m'_1=(1-m'_1)+2m'_1=m'_1+1>0$, ce qui est une contradiction. Supposons maintenant $\alpha\notin S_P$ (et rappelons que $\langle \lambda, \alpha^\vee\rangle \leq 0$), comme $L^-(s_\alpha\cdot \lambda)$ a multiplicit\'e $1$ dans $U(\mg)\otimes_{U(\mb^-)}\lambda$ par le (i) du Lemme \ref{singular}, on en d\'eduit que $\lambda'$ ci-dessus ne peut \^etre $s_\alpha\cdot \lambda$ (sinon $L^-(s_\alpha\cdot \lambda)$ appara\^\i trait \`a la fois dans (\ref{ker}) et dans $M$), une contradiction. Finalement, on en d\'eduit $\mu=\lambda+\sum_{i=1}^hm_i\alpha_i$ pour des $m_i$ dans $\Z_{\geq 0}$ avec au moins un $m_i$ non nul pour $i\geq 2$, ce qui implique facilement le r\'esultat voulu (un tel poids $\mu$ ne pouvant venir de vecteurs de la forme $y_{\alpha}^{m_1}\otimes v_{\lambda}$).

\noindent
{\bf \'Etape 3}\\
Pour montrer $\pi_1[{\mnn^\alpha}]\buildrel\sim\over\rightarrow \pi_2[{\mnn^\alpha}]$, on reprend l'argument de la preuve de \cite[Prop.~3.2]{Br3}. Pour $f:G(\Qp)\rightarrow L^-( \lambda)^\vee_{P}\otimes_E\pi_P^\infty$ localement analytique consid\'erons le morphisme de $U(\mg)$-modules \`a gauche (non nul d\`es que $f\ne 0$)~:
$$\Delta_f:U(\mg)\otimes_{U(\mpp^-)}L^-( \lambda)_{P}\longrightarrow C^{\an}(G(\Qp),\pi_P^\infty),\ {\mathfrak d}\mapsto {\mathfrak d}\cdot f$$
o\`u ${\mathfrak d}\cdot f$ est la fonction ${\mathfrak x}\otimes v\mapsto {\mathfrak x}\cdot (f(-)(v))$ si ${\mathfrak d}={\mathfrak x}\otimes v$ avec $({\mathfrak y}\cdot h)(g)\=\frac{d}{dt}h\big(\exp(-t{\mathfrak y})g\big)\vert_{t=0}\in \pi_P^\infty$ si $h\in C^{\rm an}(G(\Qp),\pi_P^\infty)$ et ${\mathfrak y}\in \mg$. Fixons $0\ne f\in \pi_2[{\mnn^\alpha}]$. Par d\'efinition de ${\mathcal F}_{P^-}^G(M,\pi_P^\infty)$ (cf. \cite[\S~1]{OS}, et aussi \cite[(2.5)]{Br3}), il suffit de montrer que $\Delta_f(\ker \phi)=0$ o\`u $\phi: U(\mg)\otimes_{U(\mpp^-)}L^-( \lambda)_{P}\twoheadrightarrow M$. On note $M_f\=(U(\mg)\otimes_{U(\mpp^-)}L^-( \lambda)_{P})/\ker(\Delta_f)$, de sorte que $f\in {\mathcal F}_{P^-}^G(M_f,\pi_P^\infty)\subseteq \pi_2$, et il suffit de montrer que l'on a une surjection $M\twoheadrightarrow M_f$ (entre quotients de $U(\mg)\otimes_{U(\mpp^-)}L^-( \lambda)_{P}$) puisqu'alors ${\mathcal F}_{P^-}^G(M_f,\pi_P^\infty)\subseteq {\mathcal F}_{P^-}^G(M,\pi_P^\infty)=\pi_1$. Comme $\Delta_f\ne 0$ on a $M_f\ne 0$ et donc $M_f\twoheadrightarrow L^-( \lambda)$ (puisque $L^-( \lambda)$ est l'unique quotient irr\'eductible de $U(\mg)\otimes_{U(\mpp^-)}L^-( \lambda)_{P}$). 

\noindent
{\bf \'Etape 4}\\
On suppose d'abord ou bien $\alpha\in S_P$ ou bien $\alpha\notin S_P$ mais $\langle \lambda,\alpha\rangle > 0$, de sorte que $M=L^-(\lambda)$. Soit $C$ un constituant irr\'eductible (dans ${\mathcal O}^{\mpp^-}_{\rm alg}$) en sous-objet de $\ker(M_f\twoheadrightarrow L^-( \lambda))$ suppos\'e non nul. Alors $C$ est l'unique quotient irr\'eductible d'un module de Verma g\'en\'eralis\'e $U(\mg)\otimes_{U(\mpp^-)}V$ (par \cite[\S\S~1.3~\&~9.3]{Hu}) et l'application compos\'ee $\psi:U(\mg)\otimes_{U(\mpp^-)}V\twoheadrightarrow C\hookrightarrow M_f$ s'ins\`ere dans le diagramme (commutatif) de $U(\mg)$-modules~:
\begin{eqnarray}\label{wv}
\xymatrix{U(\mg)\otimes_{U(\mpp^-)}L^-( \lambda)_{P}\ar@{^{}->>}[r] &M_f\ar@{^{}->>}[r] &L^-(\lambda) \\ & U(\mg)\otimes_{U(\mpp^-)}V\ar[u]^>>>>>>{\psi} \ar[ur]^0& }
\end{eqnarray}
et induit un morphisme $G(\Qp)$-\'equivariant~:
$${\mathcal F}_{P^-}^G(M_f,\pi_P^\infty)\longrightarrow {\mathcal F}_{P^-}^G(U(\mg)\otimes_{U(\mpp^-)}V,\pi_P^\infty)=\big(\Ind_{P^-(\Qp)}^{G(\Qp)}V^\vee\otimes_E\pi_P^\infty\big)^{\an}.$$
Nous allons montrer que l'image $h$ de $f\in {\mathcal F}_{P^-}^G(M_f,\pi_P^\infty)$ dans $(\Ind_{P^-(\Qp)}^{G(\Qp)}V^\vee\otimes_E\pi_P^\infty)^{\an}$ est nulle, ce qui en reprenant le raisonnement de la preuve de \cite[Prop.~3.2]{Br3} (auquel on renvoie le lecteur pour plus de d\'etails) conduit \`a une contradiction et implique donc $M_f=L^-(\lambda)=M$. Comme $h\in (\Ind_{P^-(\Qp)}^{G(\Qp)}V^\vee\otimes_E\pi_P^\infty)^{\an}$, on voit facilement que $h=0$ si et seulement si $h(-)(v)\vert_{N_{P}(\Qp)}\equiv 0$ pour tout $v\in V$. Soit $v\in V$ et ${\mathfrak d}_v\in U(\mg)\otimes_{U(\mpp^-)}L^-( \lambda)_{P}$ qui s'envoie sur $\psi(1\otimes v)$ dans $M_f$, on a ${\mathfrak d}_v\in \ker(U(\mg)\otimes_{U(\mpp^-)}L^-( \lambda)_{P}\twoheadrightarrow M)$ par (\ref{wv}). On a aussi $h(-)(v)=({\mathfrak d}_v\cdot f)(-)$ dans $C^{\an}(G(\Qp),\pi_P^\infty)$ par \cite[Lem.~3.1]{Br3} et il suffit donc de montrer $({\mathfrak d}_v\cdot f)(n)=0$ pour tout $n\in N_{P}(\Qp)$. Par ce qui suit (\ref{ker}) on peut \'ecrire ${\mathfrak d}_v$ comme une combinaison lin\'eaire de vecteurs $y_{\alpha_1}^{m_1}y_{\alpha_2}^{m_2}\cdots y_{\alpha_t}^{m_t}\otimes v_{ \lambda}$ pour des $m_i\in \Z_{\geq 0}$ tels que $t\in \{2,\dots,h\}$ et $m_t\geq 1$. Ainsi $({\mathfrak d}_v\cdot f)(-)$ est une combinaison lin\'eaires de termes~:
$$y_{\alpha_1}^{m_1}y_{\alpha_2}^{m_2}\cdots y_{\alpha_t}^{m_t}\cdot (f(-)(v_{ \lambda}))=(y_{\alpha_1}^{m_1}y_{\alpha_2}^{m_2}\cdots y_{\alpha_{t-1}}^{m_{t-1}}y_{\alpha_t}^{m_t-1})\cdot (y_{\alpha_t}\cdot (f(-)(v_{ \lambda}))).$$
On a pour tout $g\in G(\Qp)$~:
\begin{eqnarray*}
y_{\alpha_t}\cdot (f(g)(v_{ \lambda}))&=&\frac{d}{dt}f\big(\exp(-ty_{\alpha_t})g\big)(v_{ \lambda})\vert_{t=0}\\
&=&\frac{d}{dt}f\big(g\exp(-tg^{-1}y_{\alpha_t}g)\big)(v_{ \lambda})\vert_{t=0}\\
&=&((g^{-1}(-y_{\alpha_t})g)(f))(g)(v_{ \lambda})
\end{eqnarray*}
o\`u $(g^{-1}(-y_{\alpha_t})g)(f)$ est l'action de Lie sur $f$ d\'eriv\'ee de l'action du groupe $G(\Qp)$ sur $\pi_2$ (par translation \`a droite). Lorsque $g=n\in N(\Qp)$, alors $n^{-1}(-y_{\alpha_t})n\in \mnn^\alpha$ puisque $-y_{\alpha_t}\in \mnn^\alpha$ (car $\alpha_t\ne \alpha_1=\alpha$) et $N^\alpha$ est normal dans $N$, et on obtient~:
$$y_{\alpha_t}\cdot (f(n)(v_{ \lambda}))=((n^{-1}(-y_{\alpha_t})n)(f))(n)(v_{ \lambda})=0$$
puisque $f\in \pi_2[\mnn^\alpha]$. On a donc $h(n)(v)=({\mathfrak d}_v\cdot f)(n)=0$ pour tout $v\in V$ et tout $n\in N(\Qp)$, donc {\it a fortiori} pour tout $n\in N_{P}(\Qp)$, ce qui ach\`eve la preuve dans ce cas.

\noindent
{\bf \'Etape 5}\\
On suppose maintenant $\alpha\notin S_P$ et $\langle \lambda,\alpha\rangle \leq 0$, donc $M=L^-(s_\alpha\cdot \lambda)-L^-(\lambda)$. Si $L^-(s_\alpha\cdot \lambda)$ appara\^\i t dans $M_f$, alors par le (ii) du Lemme \ref{singular} et en prenant un constituant irr\'eductible $C$ en sous-objet de $\ker(M_f\twoheadrightarrow M)$ (suppos\'e non nul) on a un diagramme commutatif analogue \`a (\ref{wv})~:
\begin{eqnarray}\label{wvalpha}
\xymatrix{U(\mg)\otimes_{U(\mpp^-)}L^-( \lambda)_{P}\ar@{^{}->>}[r] &M_f\ar@{^{}->>}[r] &M \\ & U(\mg)\otimes_{U(\mpp^-)}V\ar[u]^>>>>>>{\psi} \ar[ur]^0& }
\end{eqnarray}
et une preuve totalement analogue \`a celle de l'\'Etape $4$ (utilisant le r\'esultat de l'\'Etape $2$) montre que l'on doit avoir $M_f=M$, ce qui ach\`eve la preuve dans ce cas. 

On suppose maintenant que $L^-(s_\alpha\cdot \lambda)$ n'appara\^\i t pas dans $M_f$, de sorte que l'on est encore dans la situation de (\ref{wv}). Soit $v\in V$, on va montrer que l'on peut choisir ${\mathfrak d}_v\in U(\mg)\otimes_{U(\mpp^-)}L^-( \lambda)_{P}$ qui s'envoie sur $\psi(1\otimes v)$ dans $M_f$ par (\ref{wv}) et qui est tel que ${\mathfrak d}_v\in \ker(U(\mg)\otimes_{U(\mpp^-)}L^-( \lambda)_{P}\twoheadrightarrow M)$. D'une part la surjection $U(\mg)\otimes_{U(\mpp^-)}L^-( \lambda)_{P}\twoheadrightarrow M_f$ se factorise par $U(\mg)\otimes_{U(\mq^-)}L^-( \lambda)_{Q}$ par le Lemme \ref{quotientM}. D'autre part on a un diagramme commutatif de suites exactes avec (\ref{encoresuite}) et le Lemme \ref{singular}~:
\begin{equation*}
\begin{gathered}
\xymatrix{&U(\mg)\otimes_{U(\mb^-)}\!s_\alpha\cdot\lambda\ar[r]\ar@{^{}->>}[d] & U(\mg)\otimes_{U(\mpp^-)}\!L^-(\lambda)_{P} \ar@{^{}->>}[d] \ar[r]& U(\mg)\otimes_{U(\mq^-)}\!L^-(\lambda)_{Q} \ar@{^{}->>}[d] \ar[r]&0\\
0\ar[r]&L^-(s_\alpha\cdot\lambda)\ar[r] & M \ar[r] & L^-(\lambda)\ar[r]&0}
\end{gathered}
\end{equation*}
qui montre (par une chasse au diagramme triviale) qu'un \'el\'ement dans $\ker(U(\mg)\otimes_{U(\mq^-)}L^-( \lambda)_{Q}\twoheadrightarrow L^-(\lambda))$ se rel\`eve dans $\ker(U(\mg)\otimes_{U(\mpp^-)}L^-( \lambda)_{P}\twoheadrightarrow M)$. On en d\'eduit que l'on peut prendre ${\mathfrak d}_v\in \ker(U(\mg)\otimes_{U(\mpp^-)}L^-( \lambda)_{P}\twoheadrightarrow M)$. On peut alors d\'erouler la preuve comme dans l'\'Etape $4$ en utilisant l'\'Etape $2$, ce qui donne $M_f=L^-(\lambda)$ et ach\`eve la preuve.
\end{proof}

Le corollaire qui suit pr\'ecise le Corollaire \ref{generaljoli} dans des cas importants.

\begin{cor}\label{casparticuliers}
Soit $P\subseteq G$ un sous-groupe parabolique contenant $B$, $L(-\lambda)_P$ une repr\'esentation alg\'ebrique irr\'eductible de $L_P(\Qp)$ sur $E$, $\pi_P^\infty$ une repr\'esentation lisse de longueur finie de $L_P(\Qp)$ sur $E$ admettant un caract\`ere central $\chi_{\pi_P^\infty}$, $d_{\pi_P^\infty}\=\dim_E(\pi_P^\infty\otimes_EE_\infty)(\eta^{-1})_{N_{L_{P}}\!(\Qp)}$, $M$ un quotient non nul de $U(\mg)\otimes_{U(\mpp^-)}L^-(\lambda)_{P}$ et $\pi\={\mathcal F}_{P^-}^G(M,\pi_P^\infty)$. Quand $\alpha\notin S_P$ on note $Q$ le parabolique contenant $P$ tel que $S_Q=S_P\amalg \{\alpha\}$ et $\mq^-$ l'alg\`ebre de Lie de $Q^-(\Qp)$.\\
(i) Si $\alpha\notin S_P$ et $M\notin {\mathcal O}^{\mq^-}_{\rm alg}$, on a~:
\begin{equation*}
F_\alpha(\pi)\simeq
E_\infty(\chi_{-\lambda})\otimes_{E}\Hom_{(\varphi,\Gamma)}\Big(\R\big((\lambda\circ \lambda_{\alpha^\vee})(\chi_{\pi_P^\infty}^{-1}\circ \lambda_{\alpha^{\!\vee}})\big)^{\oplus d_{\pi_P^\infty}},-\Big).
\end{equation*}
(ii) Si $\alpha\notin S_P$ et $M\in {\mathcal O}^{\mq^-}_{\rm alg}$ ou si $\alpha \in S_P$, on a~:
$$F_\alpha(\pi)\simeq E_\infty(\chi_{-\lambda})\otimes_E\Hom_{(\varphi,\Gamma)}\Big(\big(\R(\lambda\circ \lambda_{\alpha^{\!\vee}})/(t^{1-\langle \lambda,\alpha^\vee\rangle})\big)^{\oplus d_{\pi_P^\infty}},-\Big).$$
\end{cor}
\begin{proof}
(i) Cela d\'ecoule du Th\'eor\`eme \ref{liesocle} avec le Th\'eor\`eme \ref{plusgeneral}.\\
(ii) Si $\alpha\in S_P$ cela d\'ecoule du Th\'eor\`eme \ref{liesocle} avec le Th\'eor\`eme \ref{encoreplusgeneral} et le Th\'eor\`eme \ref{caslisse} (appliqu\'e avec $G=L_P$). Supposons $\alpha\notin S_P$ et $M\in {\mathcal O}^{\mq^-}_{\rm alg}$. Par \cite[Prop.~4.9(b)]{OS} on a~:
$${\mathcal F}_{P^-}^G(M,\pi_P^\infty)\simeq {\mathcal F}_{Q^-}^G\big(M,\big(\Ind_{L_Q(\Qp)\cap P^-(\Qp)}^{L_Q(\Qp)}\pi_P^\infty\big)^\infty\big).$$
On d\'eduit du Th\'eor\`eme \ref{liesocle} appliqu\'e avec $P=Q$ un isomorphisme~:
$$F_\alpha\big({\mathcal F}_{P^-}^G(M,\pi_P^\infty)\big)\buildrel\sim\over\longrightarrow F_\alpha\big(\big(\Ind_{Q^-(\Qp)}^{G(\Qp)} L(-\lambda)_{Q}\otimes_E\big(\Ind_{L_Q(\Qp)\cap P^-(\Qp)}^{L_Q(\Qp)}\pi_P^\infty\big)^\infty\big)^{\an}\big)$$
puis avec le Th\'eor\`eme \ref{encoreplusgeneral} (aussi appliqu\'e avec $P=Q$) un isomorphisme~:
$$F_\alpha\big({\mathcal F}_{P^-}^G(M,\pi_P^\infty)\big)\buildrel\sim\over\longrightarrow F_\alpha\big(L(-\lambda)_{Q}\otimes_E\big(\Ind_{L_Q(\Qp)\cap P^-(\Qp)}^{L_Q(\Qp)}\pi_P^\infty\big)^\infty\big).$$
Le r\'esultat d\'ecoule alors du Th\'eor\`eme \ref{caslisse} (appliqu\'e avec $G=L_Q$) et de l'\'egalit\'e $d_{\pi_P^\infty}=\dim_{E_\infty}\big(\big(\Ind_{L_Q(\Qp)\cap P^-(\Qp)}^{L_Q(\Qp)}\pi_P^\infty\big)^\infty\otimes_EE_\infty\big)(\eta^{-1})_{N_{L_Q}(\Qp)}$ (\cite{Ro}).
\end{proof}

\section{Quelques r\'esultats d'exactitude pour $F_\alpha$}\label{gl3}

On \'enonce une conjecture de repr\'esentabilit\'e et d'exactitude (Conjecture \ref{representable}) pour le foncteur $F_\alpha$ appliqu\'e \`a certaines repr\'esentations localement analytiques de $G(\Qp)$ qui {\it ne} sont {\it pas} n\'ecessairement de la forme ${\mathcal F}_{P^-}^G(-,-)$ (mais dont les constituants irr\'eductibles le sont). Puis on en d\'emontre deux cas particuliers (Th\'eor\`eme \ref{localgenplus} et Th\'eor\`eme \ref{gl3enplus}), dont les preuves prennent l'essentiel du paragraphe.

\subsection{Pr\'eliminaires}\label{prelgen}

On \'enonce divers r\'esultats pr\'eliminaires et une conjecture (Conjecture \ref{representable}) valables pour tout groupe $G$ comme au \S~\ref{prel}.

On conserve les notations des \S\S~\ref{prel} et \ref{notabene}, et on fixe $\lambda\in X(T)$ dominant par rapport \`a $B^-$. On fixe aussi une racine simple $\alpha\in S$. Consid\'erons une repr\'esentation irr\'eductible $\pi$ de la forme ${\mathcal F}_{P^-}^G(L^-(w\cdot \lambda),\pi_P^\infty)$ pour $w\in W$ et $\pi_P^\infty$ lisse irr\'eductible. Alors $F_\alpha(\pi)$ est calcul\'e dans le Corollaire \ref{casparticuliers}. En se souvenant que $d_{\pi_P^\infty}= 1$ si $\pi_P^\infty$ est irr\'eductible g\'en\'erique et $d_{\pi_P^\infty}= 0$ si $\pi_P^\infty$ est irr\'eductible non g\'en\'erique (\cite{Ro}), on a $F_\alpha(\pi)=0$ si $\pi_P^\infty$ est non g\'en\'erique, et si $\pi_P^\infty$ est g\'en\'erique on obtient $F_\alpha(\pi)\simeq E_\infty(\chi_{-w\cdot \lambda})\otimes_{E}\Hom_{(\varphi,\Gamma)}\big(D_\alpha(\pi),-\big)$ avec (en utilisant \cite[Lem.~4.2]{Br3})~:
$$\left\{\begin{array}{ccccc}
D_\alpha(\pi)&=&\R\big(((w\cdot\lambda)\circ \lambda_{\alpha^\vee})(\chi_{\pi_P^\infty}^{-1}\circ \lambda_{\alpha^\vee})\big)&{\rm\ si\ }&w^{-1}(\alpha)<0\\
D_\alpha(\pi)&=&\R((w\cdot \lambda)\circ \lambda_{\alpha^\vee})/(t^{1-\langle w\cdot \lambda,\alpha^\vee\rangle})&{\rm\ si\ }&w^{-1}(\alpha)>0.
\end{array}\right.$$

\begin{lem}\label{egalchi}
On a $E_\infty(\chi_{-\lambda})\simeq E_\infty(\chi_{-w\cdot\lambda})$ si et seulement si $w\in \{1,s_\alpha\}$.
\end{lem}
\begin{proof}
Comme $\chi_{-\lambda}$ et $\chi_{-w\cdot\lambda}$ sont alg\'ebriques, on a $E_\infty(\chi_{-\lambda})\simeq E_\infty(\chi_{-w\cdot\lambda})$ si et seulement si $\chi_{-\lambda}=\chi_{-w\cdot\lambda}$ si et seulement si $\chi_{\lambda-w\cdot \lambda}=1$. Il suffit donc de montrer $\chi_{\lambda-w\cdot \lambda}=1$ si et seulement si $w\in \{1,s_\alpha\}$. On a $\langle \sum_{\beta\in S\setminus\{\alpha\}}\lambda_{\beta^\vee},\gamma\rangle \geq 0$ pour $\gamma\in S$ avec nullit\'e si et seulement si $\gamma=\alpha$. En revenant \`a la d\'efinition de $\chi_\mu$ avant le Th\'eor\`eme \ref{caslisse}, et en utilisant que $\lambda - w\cdot \lambda \in \sum_{\gamma\in S}\Z_{\geq 0}\gamma$ (car $\lambda$ est dominant par rapport \`a $B^-$), on en d\'eduit $\chi_{\lambda-w\cdot \lambda}=1$ si et seulement si $\lambda-w\cdot \lambda\in \Z_{\geq 0}\alpha$ si et seulement si $w\in \{1,s_\alpha\}$.
\end{proof}

On d\'esigne par $C_{\lambda,\alpha}$ la sous-cat\'egorie pleine (artinienne) de la cat\'egorie ab\'elienne des repr\'esentations localement analytiques admissibles de $G(\Qp)$ sur $E$ form\'ee des repr\'esentations de longueur finie dont les constituants irr\'eductibles sont de l'une des formes suivantes~: ${\mathcal F}_{P^-}^G(L^-(w\cdot \lambda),\pi_P^\infty)$ pour $w\in \{1,s_\alpha\}$ et $\pi_P^\infty$ quelconque ou bien ${\mathcal F}_{P^-}^G(L^-(w\cdot \lambda),\pi_P^\infty)$ pour $w$ quelconque et $\pi_P^\infty$ {\it non} g\'en\'erique. Noter qu'un objet (r\'eductible) de $C_{\lambda,\alpha}$ n'est en g\'en\'eral plus n\'ecessairement de la forme ${\mathcal F}_{P^-}^G(M,\pi_P^\infty)$ pour un $M$ dans ${\mathcal O}^{\mpp^-}_{\rm alg}$. Noter \'egalement que si $\pi$ dans $C_{\lambda,\alpha}$ est irr\'eductible de la forme ${\mathcal F}_{P^-}^{{\rm GL}_3}(L^-(w\cdot \lambda),\pi_P^\infty)$ avec $\pi_P^\infty$ non g\'en\'erique, on a n\'ecessairement $L_P\ne T$, donc $w\ne w_0$.

La proposition qui suit n'est pas n\'ecessaire pour les applications que l'on a en vue, mais est simple \`a d\'emontrer.

\begin{prop}
Soit $\lambda_1$, $\lambda_2$ distincts dans $X(T)$ et dominants par rapport \`a $B^-$, et soit $\pi_{\lambda_1}$, $\pi_{\lambda_2}$ de la forme ${\mathcal F}_{P_1^-}^G(L^-(w_1\cdot \lambda_1),\pi_{P_1}^\infty)$, ${\mathcal F}_{P_2^-}^G(L^-(w_2\cdot \lambda_2),\pi_{P_2}^\infty)$ respectivement o\`u $w_i\in W$ et $\pi_{P_i}^\infty$ est de longueur finie pour $i\in \{1,2\}$. Alors on a ${\rm Ext}^j_G(\pi_{\lambda_1},\pi_{\lambda_2})= {\rm Ext}^j_G(\pi_{\lambda_2},\pi_{\lambda_1})=0$ pour $j\in \{0,1\}$.
\end{prop}
\begin{proof}
Il suffit de montrer que le centre $Z(\mg)$ de $U(\mg)$ agit par des caract\`eres distincts sur $\pi_{\lambda_1}$ et $\pi_{\lambda_2}$. L'inclusion $\pi_{\lambda_i}\subseteq (\Ind_{P_i^-(\Qp)}^{G(\Qp)} L(-w_i\cdot \lambda_i)_{P_i}\otimes_E\pi_{P_i}^\infty)^{\an}$ montre que $Z(\mg)$ agit sur $\pi_{\lambda_i}$ par le caract\`ere infinit\'esimal $\chi_{-w_i\cdot \lambda_i}$ avec les notations de \cite[\S~1.7]{Hu} (cf. \cite[Prop.~3.7]{ST1}). Le r\'esultat d\'ecoule alors de \cite[Th.~1.10(b)]{Hu}.
\end{proof}

\begin{conj}\label{representable}
Soit $0\rightarrow \pi''\rightarrow \pi\rightarrow \pi'\rightarrow 0$ une suite exacte (stricte) dans $\Rep(G(\Qp))$. On suppose $\pi''$ irr\'eductible dans $C_{\lambda,\alpha}$ et $F_\alpha(\pi')(-)\simeq E_\infty(\chi_{-\lambda})\otimes_{E}\Hom_{(\varphi,\Gamma)}(D_\alpha(\pi'),-)$ pour un $(\varphi,\Gamma)$-module g\'en\'eralis\'e $D_{\alpha}(\pi')$ sur $\R$ tel que $D_\alpha(\pi')$ est sans torsion et $\Hom_{(\varphi,\Gamma)}(D_\alpha(\pi''),D_\alpha(\pi'))=0$. Alors on a~:
$$F_\alpha(\pi)(-)\simeq E_\infty(\chi_{-\lambda})\otimes_{E}\Hom_{(\varphi,\Gamma)}(D_\alpha(\pi),-)$$
pour un (unique) $(\varphi,\Gamma)$-module g\'en\'eralis\'e $D_{\alpha}(\pi)$ sur $\R$ qui s'ins\`ere dans une suite exacte $0\rightarrow D_{\alpha}(\pi')\rightarrow D_{\alpha}(\pi)\rightarrow D_{\alpha}(\pi'')\rightarrow 0$.
\end{conj}

\begin{rem}
{\rm (i) Sans l'hypoth\`ese $\Hom_{(\varphi,\Gamma)}(D_\alpha(\pi''),D_\alpha(\pi'))=0$, il n'est pas vrai en g\'en\'eral que $F_\alpha(\pi)$ est de la forme $E_\infty(\chi_{-\lambda})\otimes_{E}\Hom_{(\varphi,\Gamma)}(D,-)$ dans $F(\varphi,\Gamma)_\infty$ pour un $(\varphi,\Gamma)$-module g\'en\'eralis\'e $D$ sur $\R$~: consid\'erer par exemple $G=\G$ pour $n>2$ et $\pi=\pi_1\otimes_EE[\epsilon]/(\epsilon^2)$ avec $\pi_1$ irr\'eductible dans $C_{\lambda,\alpha}$ et $G(\Qp)$ agissant diagonalement par la multiplication par $1+\epsilon \log(\det)$ \`a droite (o\`u log est une branche du logarithme $p$-adique). Si l'on oublie l'action de $\Gal(E_\infty/E)$, il est possible que $F_\alpha(\pi)$ soit de la forme $E_\infty\otimes_{E}\Hom_{(\varphi,\Gamma)}(D,-)$ sans cette hypoth\`ese (c'est par exemple le cas, si l'on examine leur preuve, dans les situations du Th\'eor\`eme \ref{localgenplus} et du Th\'eor\`eme \ref{gl3enplus} ci-dessous), mais alors $D$ n'est pas canoniquement d\'efini (seul $E_\infty\otimes_ED$ l'est).\\
(ii) Il est par contre possible que la Conjecture \ref{representable} reste vraie sans l'hypoth\`ese que $D_\alpha(\pi')$ est sans torsion (par exemple on peut encore montrer que c'est le cas pour ${\rm GL}_2$). L'hypoth\`ese que $D_{\alpha}(\pi')$ est sans torsion sera satisfaite dans les applications \`a la compatibilit\'e local-global du \S~\ref{locglob}.\\
(iii) Si l'on rajoute \`a la cat\'egorie $C_{\lambda,\alpha}$ des constituants de la forme ${\mathcal F}_{P^-}^G(L^-(w\cdot \lambda),\pi_P^\infty)$ pour $w\notin\{1,s_\alpha\}$ et $\pi_P^\infty$ g\'en\'erique, il r\'esulte facilement du Th\'eor\`eme \ref{liesocle} que (l'analogue de) la Conjecture \ref{representable} devient fausse en g\'en\'eral.}
\end{rem}

L'objectif du \S~\ref{gl3} est de montrer les deux th\'eor\`emes suivants.

\begin{thm}\label{localgenplus}
La conjecture \ref{representable} est vraie lorsque $\pi''$ est localement alg\'ebrique.
\end{thm}

\begin{thm}\label{gl3enplus}
La conjecture \ref{representable} est vraie pour $G={\rm GL}_2$ et $G={\rm GL}_3$.
\end{thm}

L'hypoth\`ese $\Hom_{(\varphi,\Gamma)}(D_\alpha(\pi''),D_\alpha(\pi'))=0$ est automatiquement v\'erifi\'ee dans le cas du Th\'eor\`eme \ref{localgenplus} par le Th\'eor\`eme \ref{caslisse} et l'hypoth\`ese que $D_\alpha(\pi')$ est sans torsion. La preuve du Th\'eor\`eme \ref{gl3enplus} est nettement plus d\'elicate (pour $G={\rm GL}_3$) que celle du Th\'eor\`eme \ref{localgenplus}. Nous continuons avec quelques lemmes pr\'eliminaires valables pour tout $G$ et utilis\'es ci-dessous dans les preuves de ces deux th\'eor\`emes.

\begin{lem}\label{actiongalois}
Soit $r\in \Q_{>p-1}$, $m\in \Z_{\geq 0}$ et $D'_r$, $D''_r$ deux $(\varphi,\Gamma)$-modules g\'en\'eralis\'es sur $\Rr$ tels que $\Hom_{(\varphi,\Gamma)}(D''_r,D'_r)=0$. On munit $E_m\otimes_ED'_r$ et $E_m\otimes_ED''_r$ d'une action de $\Gal(E_\infty/E)$ d\'efinie par $g(x\otimes d)=\chi_{\lambda}(g)g(x)\otimes d$ pour $g\in \Gal(E_\infty/E)$, $x\in E_m$ et $d\in D'_r$ ou $D''_r$. Soit $V$ un $(\psi,\Gamma)$-module de Fr\'echet sur $\Rrm$ muni d'une action de $\Gal(E_\infty/E)$ qui commute \` a $\Rr$, $\psi$, $\Gamma$, v\'erifie $g(xv)=g(x)g(v)$ pour $g\in \Gal(E_\infty/E)$, $x\in E_m$, $v\in V$, et s'ins\`ere dans une suite exacte courte~:
$$0\longrightarrow E_m\otimes_ED'_r \longrightarrow V \longrightarrow E_m\otimes_ED''_r \longrightarrow 0$$
de $(\psi,\Gamma)$-modules de Fr\'echet sur $\Rrm$ commutant \`a l'action de $\Gal(E_\infty/E)$. Alors on a $V\simeq E_m\otimes_ED_r$ pour un unique $(\varphi,\Gamma)$-module g\'en\'eralis\'e $D_r$ sur $\Rr$ qui est une extension de $D''_r$ par $D'_r$ avec action de $\Gal(E_\infty/E)$ sur $E_m\otimes_ED_r$ donn\'ee par $g(x\otimes d)=\chi_{\lambda}(g)g(x)\otimes d$ ($g\in \Gal(E_\infty/E)$, $x\in E_m$, $d\in D_r$).
\end{lem}
\begin{proof}
Quitte \`a tordre partout par $\chi_{-\lambda}$, on peut supposer $\chi_{\lambda}=1$. L'unicit\'e de $D_r$ est claire car $D_r$ s'identifie alors \`a $(E_m\otimes_ED_r)^{\Gal(E_\infty/E)}=V^{\Gal(E_\infty/E)}$. Soit $g\in \Gal(E_\infty/E_m)$, alors $g-\Id:V\rightarrow V$ se factorise en $V\twoheadrightarrow E_m\otimes_ED''_r \rightarrow E_m\otimes_ED'_r \hookrightarrow V$, qui est nul puisque $\Hom_{(\varphi,\Gamma)}(D''_r,D'_r)=0$ (avec les (iv) et (v) de la Remarque \ref{vrac}). Donc $\Gal(E_\infty/E_m)$ agit trivialement sur $V$, d'o\`u on d\'eduit une suite exacte courte~:
$$0\longrightarrow (E_m\otimes_ED'_r)^{\Gal(E_m/E)}\longrightarrow V^{\Gal(E_m/E)} \longrightarrow (E_m\otimes_ED''_r)^{\Gal(E_m/E)}\longrightarrow 0$$
de $(\psi,\Gamma)$-modules de Fr\'echet sur $\Rr$, ainsi qu'un isomorphisme $E_m\otimes_EV^{\Gal(E_m/E)}\buildrel\sim\over\rightarrow V$ de $(\psi,\Gamma)$-modules de Fr\'echet sur $\Rrm$ compatible \`a $\Gal(E_m/E)$. Comme $(E_m\otimes_ED'_r)^{\Gal(E_m/E)}$ et $(E_m\otimes_ED''_r)^{\Gal(E_m/E)}$ sont des $(\varphi,\Gamma)$-modules g\'en\'eralis\'es sur $\Rr$, (\ref{psi6}) ci-dessous (pour $j=1$) et un d\'evissage \'evident montrent que (\ref{psi4}) (pour $j=1$) induit un isomorphisme $\R^{pr}\otimes_{\Rr}V^{\Gal(E_m/E)}\buildrel\sim\over\rightarrow \R^{pr}\otimes_{\varphi,\Rr}V^{\Gal(E_m/E)}$, d'o\`u il suit (en prenant son inverse) que $V^{\Gal(E_m/E)}$ est aussi un $(\varphi,\Gamma)$-module g\'en\'eralis\'e sur $\Rr$ v\'erifiant les conditions de l'\'enonc\'e.
\end{proof}

On d\'esigne par $H^\cdot(\mnn^\alpha,-)$ la cohomologie usuelle d'alg\`ebre de Lie (cf. par exemple \cite[\S~7]{We}, cf. aussi \cite[\S~3]{ST3} ou \cite[p.~73]{Sc1}). C'est la cohomologie d'un complexe dont le terme g\'en\'eral est $\Hom_{\Qp}(\bigwedge^\cdot \mnn^\alpha,-)$. Soit $\pi$ dans $\Rep(B(\Qp))$, alors $H^i(\mnn^\alpha,\pi)$ pour $i\in \Z$ est nul si $i<0$ et $i> \dim_{\Qp}\mnn^\alpha$, et est naturellement muni d'une topologie localement convexe (pas forc\'ement s\'epar\'ee) en utilisant que chaque terme $\Hom_{\Qp}(\bigwedge^\cdot \mnn^\alpha,\pi)$ est un espace localement convexe de type compact puisque $\bigwedge^\cdot \mnn^\alpha$ est de dimension finie, cf. \cite[p.~75]{Sc1}. Chaque $\Hom_{\Qp}(\bigwedge^i \mnn^\alpha,\pi)$ est muni de l'action de $B(\Qp)$ par automorphismes continus~:
\begin{equation}\label{actionlie}
f\in \Hom_{\Qp}\big(\bigwedge^i \mnn^\alpha,\pi\big)\longmapsto \big(x_1\wedge \cdots \wedge x_i\mapsto bf({\rm Ad}_{b^{-1}}(x_1)\wedge \cdots \wedge{\rm Ad}_{b^{-1}}(x_i))\big)
\end{equation}
qui en fait un objet de $\Rep(B(\Qp))$ et induit une action de $B(\Qp)$ par automorphismes continus sur chaque $H^i(\mnn^\alpha,\pi)$. Si la topologie sur $H^i(\mnn^\alpha,\pi)$ est s\'epar\'ee, alors $H^i(\mnn^\alpha,\pi)$ est un espace de type compact et l'action de $B(\Qp)$ en fait aussi un objet de $\Rep(B(\Qp))$ tel que l'action de $N^\alpha(\Qp)$ est lisse. On peut dans ce cas d\'efinir un $(\psi,\Gamma)$-module de Fr\'echet $M_\alpha(H^i(\mnn^\alpha,\pi\otimes_EE_m))=(H^i(\mnn^\alpha,\pi\otimes_EE_m)(\eta^{-1})_{N_m^\alpha})^\vee$ sur $\Rm^+$ comme en (\ref{malpha}).

\begin{lem}\label{separexact}
Soit $m\in \Z_{\geq 0}$, $0\rightarrow \pi''\rightarrow \pi\rightarrow \pi'\rightarrow 0$ une suite exacte (stricte) dans $\Repm(B(\Qp))$ et supposons que $H^1(\mnn^\alpha,\pi'')$ soit s\'epar\'e. Alors on a un complexe de $(\psi,\Gamma)$-modules de Fr\'echet sur $\Rm^+$~:
$$M_\alpha(H^1(\mnn^\alpha,\pi''))\buildrel f\over\longrightarrow M_\alpha(\pi')\buildrel g\over\longrightarrow M_\alpha(\pi)\longrightarrow M_\alpha(\pi'')\longrightarrow 0$$
qui est exact en $M_\alpha(\pi'')$ et $M_\alpha(\pi)$ et tel que l'image de $f$ est dense dans $\ker(g)$.
\end{lem}
\begin{proof}
On a une suite exacte d'espaces de type compact $0\rightarrow \pi''[\mnn^\alpha]\rightarrow \pi[\mnn^\alpha]\rightarrow \pi'[\mnn^\alpha]\buildrel\delta\over \rightarrow H^1(\mnn^\alpha,\pi'')$, qui se d\'ecompose donc en deux suites exactes d'espaces de type compact~:
$$0\rightarrow \pi''[\mnn^\alpha]\rightarrow \pi[\mnn^\alpha]\rightarrow \ker(\delta)\rightarrow 0\ \ {\rm et}\ \ 0\rightarrow \ker(\delta)\rightarrow \pi'[\mnn^\alpha]\rightarrow H^1(\mnn^\alpha,\pi'')$$
(noter que, si $H^1(\mnn^\alpha,\pi'')$ n'\'etait pas s\'epar\'e, on aurait encore ces deux suites exactes mais on ne saurait pas que $\ker(\delta)$ est un espace de type compact). Comme dans la preuve de la Proposition \ref{exactdense} on en d\'eduit deux suites exactes d'espaces de type compact (o\`u la premi\`ere est stricte et o\`u l'injection dans la deuxi\`eme est stricte)~:
\begin{multline*}
0\rightarrow \pi''[{\mnn^\alpha}](\eta^{-1})_{N_m^\alpha}\rightarrow \pi[{\mnn^\alpha}](\eta^{-1})_{N_m^\alpha}\rightarrow \ker(\delta)(\eta^{-1})_{N_m^\alpha}\rightarrow 0\ \ {\rm et}\\
0\rightarrow \ker(\delta)(\eta^{-1})_{N_m^\alpha}\rightarrow \pi'[{\mnn^\alpha}](\eta^{-1})_{N_m^\alpha} \rightarrow H^1(\mnn^\alpha,\pi'')(\eta^{-1})_{N_m^\alpha}.
\end{multline*}
En dualisant, on en d\'eduit une suite exacte stricte d'espaces de Fr\'echet $0\rightarrow (\ker(\delta)(\eta^{-1})_{N_m^\alpha})^\vee\rightarrow M_\alpha(\pi)\rightarrow M_\alpha(\pi'')\rightarrow 0$ et un complexe d'espaces de Fr\'echet $M_\alpha(H^1(\mnn^\alpha,\pi''))\buildrel f\over\rightarrow M_\alpha(\pi')\buildrel g\over\rightarrow (\ker(\delta)(\eta^{-1})_{N_m^\alpha})^\vee\rightarrow 0$ o\`u, comme \`a la fin de la preuve de la Proposition \ref{exactdense}, $g$ est une surjection topologique et l'image de $f$ est dense dans $\ker(g)$. On a l'\'enonc\'e en combinant ces deux complexes. 
\end{proof}

Une variante du lemme suivant est utilis\'ee de mani\`ere tacite dans \cite[\S~4.5]{Sc1}.

\begin{lem}[Schraen]\label{separedevisse}
Soit $i\in \Z$ et $0\rightarrow \pi''\rightarrow \pi\rightarrow \pi'\rightarrow 0$ une suite exacte (stricte) dans $\Rep(B(\Qp))$. Si $H^i(\mnn^\alpha,\pi')$ est s\'epar\'e et si la fl\`eche continue $H^{i-1}(\mnn^\alpha,\pi')\rightarrow H^i(\mnn^\alpha,\pi'')$ est d'image ferm\'ee, alors $H^i(\mnn^\alpha,\pi)$ est aussi s\'epar\'e.
\end{lem}
\begin{proof}
Notons $C^i(-)\=\Hom_{\Qp}(\bigwedge^i \mnn^\alpha,-)$ avec $-\in \{\pi'',\pi,\pi'\}$, qui est un espace de type compact, $d^i:C^i(-)\rightarrow C^{i+1}(-)$, $Z^i(-)\=\ker(d^i)$ et $B^i(-)=\im(d^i)$. On a $B^{i-1}(-)\subseteq Z^i(-)\subseteq C^i(-)$ et $H^i(\mnn^\alpha,-)=Z^i(-)/B^{i-1}(-)$. Le sous-espace $Z^i(-)$ est ferm\'e dans $C^i(-)$ et on le munit de la topologie d'espace de type compact induite. Le quotient $C^i(-)/Z^i(-)$ est alors aussi de type compact et via l'isomorphisme $C^i(-)/Z^i(-)\buildrel\sim\over\rightarrow B^i(-)$ on munit $B^i(-)$ de cette topologie de type compact. Les injections $B^{i-1}(-)\hookrightarrow Z^i(-)$ sont continues.

Par d\'efinition $Z^i(-)/B^{i-1}(-)$ est muni de la topologie quotient de $Z^i(-)$, il faut donc montrer que $B^{i-1}(\pi)$ est ferm\'e dans $Z^i(\pi)$ sous les conditions du lemme. Le morphisme $H^{i}(\mnn^\alpha,\pi'')\rightarrow H^i(\mnn^\alpha,\pi)$ se factorise comme suit~:
$$Z^i(\pi'')/B^{i-1}(\pi'')\twoheadrightarrow Z^i(\pi'')/(B^{i-1}(\pi)\cap Z^i(\pi''))\hookrightarrow Z^i(\pi)/B^{i-1}(\pi)$$
et la suite exacte $H^{i-1}(\mnn^\alpha,\pi')\rightarrow H^i(\mnn^\alpha,\pi'')\rightarrow H^i(\mnn^\alpha,\pi)$ montre donc que $Z^i(\pi'')/(B^{i-1}(\pi)\cap Z^i(\pi''))$ avec la topologie quotient de $Z^i(\pi'')$ est isomorphe au {\it conoyau} de $H^{i-1}(\mnn^\alpha,\pi')\rightarrow H^i(\mnn^\alpha,\pi'')$ avec la topologie quotient de $H^i(\mnn^\alpha,\pi'')$. Or, l'image de $H^{i-1}(\mnn^\alpha,\pi')$ dans $H^i(\mnn^\alpha,\pi'')$ est ferm\'ee, donc ce conoyau est s\'epar\'e, i.e. $B^{i-1}(\pi)\cap Z^i(\pi'')$ est ferm\'e dans $Z^i(\pi'')$.

Munissons $B^{i-1}(\pi)\cap Z^i(\pi'')$ de la topologie induite par $B^{i-1}(\pi)$, les suites exactes courtes $0\rightarrow C^{i}(\pi'')\rightarrow C^{i}(\pi)\rightarrow C^{i}(\pi')\rightarrow 0$ pour tout $i$ montrent que l'on a un diagramme commutatif de suites exactes d'espaces de type compacts~:
\begin{equation*}
\xymatrix{0\ar[r] & B^{i-1}(\pi)\cap Z^i(\pi'') \ar[r]\ar@{^{(}->}[d]^{f'} & B^{i-1}(\pi) \ar[r]^g\ar@{^{(}->}[d]^{f} & B^{i-1}(\pi')\ar@{^{(}->}[d]^{f''} \ar[r]&0\\
0\ar[r] & Z^i(\pi'') \ar[r] &Z^i(\pi) \ar[r] &Z^i(\pi')& }
\end{equation*}
o\`u les injections $f'$, $f''$ sont strictes par ce qui pr\'ec\`ede et l'hypoth\`ese, et o\`u $g$ est automatiquement stricte car une surjection entre espaces de type compact. Le Lemme \ref{separetypecompact} donne alors que l'injection $f$ est aussi stricte, i.e. d'image ferm\'ee. Donc $H^i(\mnn^\alpha,\pi)=Z^i(\pi)/B^{i-1}(\pi)$ est s\'epar\'e.
\end{proof}

\subsection{Un r\'esultat d'exactitude dans le cas localement alg\'ebrique}\label{preuvelocalg}

On montre le Th\'eor\`eme \ref{localgenplus}.

On conserve les notations du \S~\ref{prelgen} (en particulier $G$ est comme au \S~\ref{prel}). Notons d'abord que, si $\pi$ dans $\Rep(B(\Qp))$ est tel que $F_\alpha(\pi)(-)\simeq E_\infty(\chi_{-\lambda})\otimes_{E}\Hom_{(\varphi,\Gamma)}(D_\alpha(\pi),-)$ pour un $(\varphi,\Gamma)$-module g\'en\'eralis\'e $D_{\alpha}(\pi)$ sur $\R$, il suit de la d\'efinition de $F_\alpha$ (cf. (\ref{foncteurm}) et (\ref{falpha})) qu'il existe $m\gg 0$, $r\gg 0$ et un $(\varphi,\Gamma)$-module g\'en\'eralis\'e $D_{\alpha}(\pi)_r$ sur $\Rr$ tels que $\R\otimes_{\Rr}D_{\alpha}(\pi)_r\buildrel\sim\over\rightarrow D_{\alpha}(\pi)$ et tels que l'on a un morphisme continu non nul de $\Rm^+$-modules commutant \`a $\psi$, $\Gamma$~:
\begin{equation}\label{canonimor}
f:M_\alpha(\pi\otimes_EE_m)\longrightarrow D_{\alpha}(\pi)_r\otimes_EE_m
\end{equation}
s'envoyant (\`a scalaire pr\`es dans $E_\infty^\times$) vers le morphisme identit\'e $D_\alpha(\pi)\rightarrow D_\alpha(\pi)$ dans $F_\alpha(\pi)(D_\alpha(\pi))$.

On fixe $\pi''\simeq L(-\lambda)\otimes_E\pi^\infty$ avec $\pi^\infty$ lisse irr\'eductible ainsi qu'une suite exacte courte $0\rightarrow \pi''\rightarrow \pi\rightarrow \pi'\rightarrow 0$ comme dans la Conjecture \ref{representable}.\\

\noindent
{\bf \'Etape $1$}\\
Il est clair que $H^1(\mnn^\alpha,\pi'')$ est s\'epar\'e puisque muni de la topologie localement convexe la plus fine. Par le Lemme \ref{separexact} on a donc un complexe de $(\psi,\Gamma)$-modules de Fr\'echet sur $\Rm^+$~:
$$M_\alpha(H^1(\mnn^\alpha,\pi'')\otimes_EE_m)\rightarrow M_\alpha(\pi'\otimes_EE_m)\rightarrow M_\alpha(\pi\otimes_EE_m)\rightarrow M_\alpha(\pi''\otimes_EE_m)\rightarrow 0$$
comme dans {\it loc.cit.} Montrons d'abord que le morphisme compos\'e $M_\alpha(H^1(\mnn^\alpha,\pi'')\otimes_EE_m)\rightarrow M_\alpha(\pi'\otimes_EE_m)\rightarrow D_{\alpha}(\pi')_r\otimes_EE_m$ devient nul si l'on augmente suf\-fisamment $m$ et $r$, o\`u le morphisme de droite est comme en (\ref{canonimor}). On a $H^1(\mnn^\alpha,\pi'')=\oplus_{\beta\in S\backslash\{\alpha\}}L(-s_\beta\cdot \lambda)_{P_\alpha}\otimes_E\pi^\infty$ par \cite[Th.~4.10]{Sc1}, de sorte que l'on peut consid\'erer le foncteur $F_\alpha(H^1(\mnn^\alpha,\pi''))$. La m\^eme preuve que celle du Th\'eor\`eme \ref{caslisse} donne $F_\alpha(H^1(\mnn^\alpha,\pi''))\simeq \oplus_{\beta\in S\backslash\{\alpha\}}(E_\infty(\chi_{-s_\beta\cdot\lambda})\otimes_E \Hom_{(\varphi,\Gamma)}(D_\beta,-))$ pour $D_\beta$ convenable. Comme par hypoth\`ese $F_\alpha(\pi')\simeq E_\infty(\chi_{-\lambda})\otimes_E \Hom_{(\varphi,\Gamma)}(D_\alpha(\pi'),-)$, le Lemme \ref{egalchi} implique alors que le morphisme $F_\alpha(\pi')\rightarrow F_\alpha(H^1(\mnn^\alpha,\pi''))$ est nul dans $F(\varphi,\Gamma)_\infty$, d'o\`u on d\'eduit la nullit\'e ci-dessus. Notant $\widetilde M_\alpha(\pi\otimes_EE_m)\= (D_{\alpha}(\pi')_r\otimes_EE_m)\oplus_{f,M_\alpha(\pi'\otimes_EE_m)}M_\alpha(\pi\otimes_EE_m)$ (un $(\psi,\Gamma)$-module de Fr\'echet sur $\Rrm$ par la Remarque \ref{strictIf}), on obtient donc une suite exacte de $(\psi,\Gamma)$-module de Fr\'echet sur $\Rrm$~:
\begin{equation}\label{avecmtilde}
0\longrightarrow D_{\alpha}(\pi')_r\otimes_EE_m\longrightarrow \widetilde M_\alpha(\pi\otimes_EE_m)\longrightarrow M_\alpha(\pi''\otimes_EE_m)\longrightarrow 0.
\end{equation}

\noindent
{\bf \'Etape $2$}\\
Un examen de la preuve du Th\'eor\`eme \ref{caslisse} (cf. en particulier l'\'Etape $5$ et la preuve du Lemme \ref{backtore}) montre d'une part que le morphisme $f$ pour $\pi''$ en (\ref{canonimor}) existe pour tout $m\in \Z_{\geq 0}$ et se factorise en un morphisme $\Rm^+$-lin\'eaire qui commute \`a $\psi$ et $\Gamma$ et devient surjectif en tensorisant \`a gauche par $\Rm^+[1/X]$~:
\begin{equation}\label{surjcellulealg}
f:M_\alpha(\pi''\otimes_EE_m)\longrightarrow D_\alpha(\pi'')_+[1/X]\otimes_EE_m
\end{equation}
o\`u $D_\alpha(\pi'')_+\=\R^+(\lambda\circ \lambda_{\alpha^\vee})/(t^{1-\langle \lambda,\alpha^\vee\rangle})$ si $\pi^\infty$ est une repr\'esentation g\'en\'erique de $G(\Qp)$ et $D_\alpha(\pi'')_+\=0$ sinon, d'autre part que si $r\in \Q_{>p-1}$ et $T_r$ est un $(\varphi,\Gamma)$-module g\'en\'eralis\'e sur $\Rr$, tout morphisme dans $\Hom_{\psi,\Gamma}(\ker(f),T_r\otimes_EE_m)$ devient nul quitte \`a augmenter $r$ et $m$. Par ailleurs, comme $D_{\alpha}(\pi')_r$ est sans torsion, la multiplication par $t^{1-\langle \lambda,\alpha^\vee\rangle}$ sur la suite exacte (\ref{avecmtilde}) donne une suite exacte de $(\psi,\Gamma)$-modules de Fr\'echet sur $\Rm^+$ (o\`u $(-)(\delta)\=(-)\otimes_{\R^+}\R^+(\delta)$ si $\delta:\Qp^\times\rightarrow E^\times$)~:
\begin{equation*}
0\rightarrow \widetilde M_\alpha(\pi\otimes_EE_m)[t^{1-\langle \lambda,\alpha^\vee\rangle}]\rightarrow M_\alpha(\pi''\otimes_EE_m)\buildrel {h}\over \rightarrow D_{\alpha}(\pi')_r(\varepsilon^{\langle \lambda,\alpha^\vee\rangle-1})/(t^{1-\langle \lambda,\alpha^\vee\rangle})\otimes_EE_m.
\end{equation*}
Par ce qui pr\'ec\`ede appliqu\'e \`a $T_r=D_{\alpha}(\pi')_r(\varepsilon^{\langle \lambda,\alpha^\vee\rangle-1})/(t^{1-\langle \lambda,\alpha^\vee\rangle})$ et la compos\'ee $\ker(f)\hookrightarrow M_\alpha(\pi''\otimes_EE_m)\buildrel h\over\rightarrow T_r\otimes_EE_m$, quitte \`a augmenter $m$ et $r$ on en d\'eduit $\ker(f)\hookrightarrow \ker(h)\buildrel\sim\over\rightarrow \widetilde M_\alpha(\pi\otimes_EE_m)[t^{1-\langle \lambda,\alpha^\vee\rangle}]$, et en particulier une suite exacte de $(\psi,\Gamma)$-modules de Fr\'echet sur $\Rm^+$~:
\begin{equation}\small\label{versletruc}
0\longrightarrow D_{\alpha}(\pi')_r\otimes_EE_m\longrightarrow \widetilde M_\alpha(\pi\otimes_EE_m)/\ker(f)\longrightarrow M_\alpha(\pi''\otimes_EE_m)/\ker(f)\longrightarrow 0.
\end{equation} 
On en d\'eduit aussi pour tout $(\varphi,\Gamma)$-module g\'en\'eralis\'e $T$ sur $\R$ un isomorphisme~:
\begin{multline}\label{limker(f)}
\lim_{m\rightarrow +\infty}\lim_{\substack{\longrightarrow \\ (s,f_s,T_s)\in I(T)}}\!\!\Hom_{\psi,\Gamma}\big(\widetilde M_\alpha(\pi\otimes_EE_m)/\ker(f),T_s\otimes_EE_m\big)\\
\buildrel\sim\over\longrightarrow \lim_{m\rightarrow +\infty}\lim_{\substack{\longrightarrow \\ (s,f_s,T_s)\in I(T)}}\!\!\Hom_{\psi,\Gamma}\big(\widetilde M_\alpha(\pi\otimes_EE_m),T_s\otimes_EE_m\big).
\end{multline}
Par ailleurs, tout morphisme dans $\Hom_{\psi,\Gamma}(M_\alpha(\pi'\otimes_EE_m),T_s\otimes_EE_m)$ se factorise par $D_{\alpha}(\pi')_r\otimes_EE_m$ quitte \`a augmenter $s$ et $m$, d'o\`u on d\'eduit un isomorphisme~:
\begin{multline}\label{limsomamalg}
\lim_{m\rightarrow +\infty}\lim_{\substack{\longrightarrow \\ (s,f_s,T_s)\in I(T)}}\!\!\Hom_{\psi,\Gamma}\big(\widetilde M_\alpha(\pi\otimes_EE_m),T_s\otimes_EE_m\big)\\
\buildrel\sim\over\longrightarrow \lim_{m\rightarrow +\infty}\lim_{\substack{\longrightarrow \\ (s,f_s,T_s)\in I(T)}}\!\!\Hom_{\psi,\Gamma}\big(M_\alpha(\pi\otimes_EE_m),T_s\otimes_EE_m\big).
\end{multline}

\noindent
{\bf \'Etape $3$}\\
Notons $W\!\=\!\widetilde M_\alpha(\pi\otimes_EE_m)/\ker(f)$, on d\'eduit de (\ref{versletruc}) et (\ref{surjcellulealg}) (en se souvenant que $X$ est inversible dans $\Rrm$) une suite exacte de $\Rrm$\!-modules dont les morphismes commutent \`a $\psi$ et $\Gamma$~:
\begin{equation}\small\label{tensorRr}
0\rightarrow \Rrm\!\otimes_{\Rm^+}\!(D_{\alpha}(\pi')_r\otimes_EE_m) \rightarrow \Rrm\!\otimes_{\Rm^+}\!W\rightarrow \Rrm\!\otimes_{\Rm^+}\!(D_\alpha(\pi'')_+\otimes_EE_m)\rightarrow 0
\end{equation}
(on montre qu'il y a bien encore un endomorphisme $\psi$ sur chaque terme en tensorisant (\ref{psi4}) pour $j=1$ par $\Rprm$ au-dessus de $\Rm^+$). Consid\'erons le ``push-out'' de (\ref{tensorRr}) le long de la surjection canonique $\Rrm\otimes_{\Rm^+}(D_{\alpha}(\pi')_r\otimes_EE_m)\twoheadrightarrow D_{\alpha}(\pi')_r\otimes_EE_m$ (dont on v\'erifie facilement qu'elle commute \`a $\psi$ et $\Gamma$). C'est une suite exacte $0\rightarrow D_{\alpha}(\pi')_r\otimes_EE_m\rightarrow V\rightarrow D_{\alpha}(\pi'')_r\otimes_EE_m\rightarrow 0$ de $(\psi,\Gamma)$-modules de Fr\'echet sur $\Rrm$ telle que tout morphisme continu $\Rm^+$-lin\'eaire $W\rightarrow T_s\otimes_EE_m$ (pour $s\geq r$) se factorise en un morphisme (continu) $\Rrm$-lin\'eaire $V\rightarrow T_s\otimes_EE_m$ (utiliser le (iii) de la Remarque \ref{vrac} pour montrer que $\Rrm\otimes_{\Rm^+}(D_{\alpha}(\pi')_r\otimes_EE_m)\rightarrow T_s\otimes_EE_m$ se factorise par $D_{\alpha}(\pi')_r\otimes_EE_m\rightarrow T_s\otimes_EE_m$). On en d\'eduit pour $s\geq r$~:
\begin{equation}\label{dernieriso}
\Hom_{\psi,\Gamma}(V,T_s\otimes_EE_m)\buildrel\sim\over\longrightarrow \Hom_{\psi,\Gamma}(W,T_s\otimes_EE_m).
\end{equation}
Comme $D_{\alpha}(\pi')_r\otimes_EE_m$ et $D_{\alpha}(\pi'')_r\otimes_EE_m$ sont des $(\varphi,\Gamma)$-modules g\'en\'eralis\'es sur $\Rrm$, le m\^eme argument qu'\`a la fin de la preuve du Lemme \ref{actiongalois} montre que $V$ est aussi un $(\varphi,\Gamma)$-module g\'en\'eralis\'e sur $\Rrm$. Par (\ref{dernieriso}), (\ref{limker(f)}) et (\ref{limsomamalg}) (sans se pr\'eoccuper de $\Gal(E_\infty/E)$ pour le moment), on voit que l'on a~:
\begin{equation}\label{avecgalois}
F_\alpha(\pi)(-)\simeq E_\infty\otimes_{E_m}\Hom_{(\varphi,\Gamma)}\big(V,(-)\otimes_EE_m\big)
\end{equation}
avec une suite exacte $0\rightarrow D_{\alpha}(\pi')_r\otimes_EE_m\rightarrow V\rightarrow D_{\alpha}(\pi'')_r\otimes_EE_m\rightarrow 0$. En uti\-lisant $F_\alpha(\pi'')(-)\simeq E_\infty(\chi_{-\lambda})\otimes_{E}\Hom_{(\varphi,\Gamma)}(D_\alpha(\pi''),-)$ (Th\'eor\`eme \ref{caslisse}) et l'hypoth\`ese $F_\alpha(\pi')(-)\simeq E_\infty(\chi_{-\lambda})\otimes_{E}\Hom_{(\varphi,\Gamma)}(D_\alpha(\pi'),-)$, on v\'erifie que l'action semi-lin\'eaire de $\Gal(E_\infty/E)$ sur $M_\alpha((-)\otimes_EE_m)$ pour $(-)\in \{\pi',\pi,\pi''\}$ d\'efinie dans la preuve du (i) du Lemme \ref{m+1} induit une action de $\Gal(E_\infty/E)$ sur toutes les suites exactes de la preuve ci-dessus, o\`u l'action sur $D_{\alpha}(\pi')_r\otimes_EE_m$, $D_\alpha(\pi'')_+\otimes_EE_m$ et $D_\alpha(\pi'')_r\otimes_EE_m$ est comme dans le Lemme \ref{actiongalois}. On d\'eduit alors de {\it loc.cit.} que l'on a $V\simeq E_m\otimes_E D_\alpha(\pi)_r$ pour un unique $(\varphi,\Gamma)$-module $D_\alpha(\pi)_r$ sur $\Rr$ qui s'ins\`ere dans une suite exacte $0\rightarrow D_\alpha(\pi')_r\rightarrow D_\alpha(\pi)_r\rightarrow D_\alpha(\pi'')_r\rightarrow 0$ et tel que $F_\alpha(\pi)(-)\simeq E_\infty(\chi_{-\lambda})\otimes_{E}\Hom_{(\varphi,\Gamma)}(D_\alpha(\pi),(-))$ par (\ref{avecgalois}) (o\`u $D_\alpha(\pi)\=\R\otimes_{\Rr}D_\alpha(\pi)_r$). Cela ach\`eve la preuve du Th\'eor\`eme \ref{localgenplus}.

\begin{rem}\label{pourapplication}
{\rm Un examen de la preuve ci-dessus du Th\'eor\`eme \ref{localgenplus} montre qu'elle reste valable {\it verbatim} en supposant que $\pi''$ est de longueur finie avec un unique constituant irr\'eductible g\'en\'erique qui est en sous-objet.}
\end{rem}

\subsection{Premiers d\'evissages pour $G={\rm GL}_3$}\label{devis1}

On commence les d\'evissages pour montrer le Th\'eor\`eme \ref{gl3enplus} pour ${\rm GL}_3$, en particulier on se d\'ebarrasse des $H^1(\mnn^\alpha,-)$.

On conserve les notations du \S~\ref{prelgen} et on suppose maintenant $G={\rm GL}_3$ et $B=$ Borel sup\'erieur. On rappelle que l'on a fix\'e une racine simple $\alpha$, et on note $\beta$ l'unique autre racine simple. On a les deux sous-groupes paraboliques maximaux associ\'es $P_{\alpha}$ et $P_{\beta}$.

On montre d'abord un r\'esultat de s\'eparation pour $H^1(\mnn^\alpha,-)$. Pour $G$ quelconque et $\pi$ dans $\Rep(B(\Qp))$, la s\'eparation des espaces localement convexes $H^i(\mnn^\alpha,\pi)$ semble un probl\`eme d\'elicat en g\'en\'eral. Nous nous contenterons de la proposition suivante, qui utilise de mani\`ere essentielle les r\'esultats de \cite[\S~4.5]{Sc1} pour ${\rm GL}_3$ (et que nous ne savons pas montrer pour $G$ plus g\'en\'eral).

\begin{prop}\label{separegl3}
Soit $\pi=\big(\Ind_{P^-(\Qp)}^{G(\Qp)} L(-\mu)_{P}\otimes_E\pi_P^\infty\big)^{\an}$ o\`u $P\in \{P_{\alpha},P_{\beta}\}$, $L(-\mu)_{P}$ est une repr\'esentation alg\'ebrique de $L_P(\Qp)$ ($\mu\in X(T)$) et $\pi_P^\infty$ est une repr\'esenta\-tion lisse de longueur finie de $L_P(\Qp)$. Alors $H^1(\mnn^\alpha,\pi)$ est s\'epar\'e.
\end{prop}
\begin{proof}
On note $\pi_P\=L(-\mu)_{P}\otimes_E\pi_P^\infty$.

\noindent
Cas $1$~: $P=P_\alpha$.\\
Par (\ref{de-a+}) on a $\pi\simeq (\Ind_{P_\beta(\Qp)}^{G(\Qp)} \pi_\beta)^{\an}$ o\`u $\pi_\beta\= \pi_{P_\alpha}^{w_0}$ (avec une notation \'evidente). Notons $C\=P_\beta(\Qp)P_\alpha(\Qp)$, la d\'ecomposition ${\rm GL}_3(\Qp)=P_\beta(\Qp) w_0 P_\alpha(\Qp) \amalg C$ o\`u $P_\beta(\Qp)w_0 P_\alpha(\Qp)$ est ouvert dans ${\rm GL}_3(\Qp)$ donne une suite exacte dans $\Rep(B(\Qp))$ (cf. le d\'ebut du \S~\ref{approx})~:
\begin{equation}\label{decompgl3}
0\longrightarrow \big(\cInd_{P_\beta(\Qp)}^{P_\beta(\Qp) w_0 P_\alpha(\Qp)}\pi_\beta\big)^{\an}\longrightarrow \big(\Ind_{P_\beta(\Qp)}^{{\rm GL}_3(\Qp)}\pi_\beta\big)^{\an}\longrightarrow \pi_{C}\longrightarrow 0.
\end{equation} 
Comme $P_\alpha\cap w_0P_\beta w_0=L_{P_\alpha}$, on a~:
\begin{multline*}
\big(\cInd_{P_\beta(\Qp)}^{P_\beta(\Qp) w_0 P_\alpha(\Qp)}\pi_\beta\big)^{\an}\buildrel\sim\over\longrightarrow \big(\cInd_{L_{P_\alpha}(\Qp)}^{P_\alpha(\Qp) }\pi_\beta^{w_0}\big)^{\an}\\
\simeq \big(\cInd_{L_{P_\alpha}(\Qp)}^{P_\alpha(\Qp) }\pi_{P_\alpha}\big)^{\an}\simeq C^{\an}_c(N_{P_\alpha}(\Qp),\pi_{P_\alpha})\simeq C^{\an}_c(N^\alpha(\Qp),E)\otimes_{E,\iota}\pi_{P_\alpha}
\end{multline*}
o\`u le troisi\`eme (resp. dernier) isomorphisme se montre comme (\ref{deIaC}) (resp. (\ref{sortie})) et o\`u l'action de $N^\alpha(\Qp)$ est par translation \`a droite sur $C^{\an}_c(N^\alpha(\Qp),E)$ (et tri\-viale sur $\pi_{P_\alpha}$). Il suit alors de la preuve de \cite[Prop.~3.1]{ST3} que l'on a $H^i(\mnn^\alpha,(\cInd_{P_\beta(\Qp)}^{P_\beta(\Qp) w_0 P_\alpha(\Qp)}\pi_\beta)^{\an})=0$ pour $i>0$. 

Une preuve analogue \`a celle de \cite[(4.102)~\&~(4.103)]{Sc1}, mais en rempla\c cant $H^q(\mnn_2,-)$, resp. $H_q(\mnn_2,-)$, dans {\it loc.cit.} par $H_q(\mnn_2,-)$, resp. $H^q(\mnn_2,-)$ (ce qui ne change pas les arguments) donne que les espaces $H_i(\mnn^\alpha,(\pi_C)^\vee)$ pour $i\in \Z$ sont isomorphes \`a (avec les notations de {\it loc.cit.} adapt\'ees \`a nos conventions)~:
\begin{equation}\small\label{techniquei=2}
H_i\big(\mnn^\alpha,\pi_\beta^\vee\widehat \otimes_{D(P_\beta(\Qp),E)_{\{1\}}}D({\rm GL}_3(\Qp),E)_{\{1\}}\big)\widehat\otimes_{D(P_\alpha(\Qp),E)_{P_\beta(\Qp)\cap P_\alpha(\Qp)}}D(P_\alpha(\Qp),E)
\end{equation}
et sont s\'epar\'es. Par un raisonnement analogue \`a la discussion suivant \cite[Cor.~3.18]{Sc1}, les espaces $H^i(\mnn^\alpha,\pi_C)$ sont aussi s\'epar\'es (et on a un isomorphisme d'espaces de Fr\'echet r\'eflexifs $H_i(\mnn^\alpha,\pi_C^\vee)\simeq H^i(\mnn^\alpha,\pi_C)^\vee$). Appliquant le Lemme \ref{separedevisse} avec la suite exacte (\ref{decompgl3}) et $i=1$, on obtient que $H^1(\mnn^\alpha,\pi)$ est s\'epar\'e.

\noindent
Cas $2$~: $P=P_\beta$.\\
Par (\ref{de-a+}) on a $\pi\simeq (\Ind_{P_\alpha(\Qp)}^{G(\Qp)} \pi_\alpha)^{\an}$ o\`u $\pi_\alpha\= \pi_{P_\beta}^{w_0}$ et une suite exacte analogue \`a (\ref{decompgl3}) o\`u $C=P_\alpha(\Qp)$. On note $(\cInd_{N^\alpha(\Qp)\cap P_\beta^-(\Qp)}^{N^\alpha(\Qp)}\pi_{P_\beta})^{\an}$ le $E$-espace vectoriel de type compact des fonctions localement analytiques $f:N^\alpha(\Qp)\rightarrow \pi_{P_\beta}$ satisfaisant l'\'equation fonctionnelle usuelle avec $\pi_{P_\beta}$ et dont le support est compact dans $N^\alpha(\Qp)/N^\alpha(\Qp)\cap P_\beta^-(\Qp)\simeq \Qp$. On le munit de l'action usuelle de $N^\alpha(\Qp)$ par translation \`a droite et de l'action de $L_{P_\alpha}(\Qp)\cap P_\beta^-(\Qp)=L_{P_\alpha}(\Qp)\cap B^-(\Qp)$~:
\begin{equation}\label{actionnalphabis}
(l_\alpha f)(n^\alpha)\=l_\alpha(f(l_\alpha^{-1}n^\alpha l_\alpha))
\end{equation}
pour $l_\alpha\in L_{P_\alpha}(\Qp)\cap P_\beta^-(\Qp)$, $n^\alpha\in N^\alpha(\Qp)$. Par \cite[Lem.~4.15]{Sc1} on a des isomorphismes $P_\alpha(\Qp)$-\'equivariants~:
\begin{eqnarray}\scriptstyle\label{piCbis}
\ \ \ \ \ \ \ \big(\cInd_{P_\alpha(\Qp)}^{P_\alpha(\Qp) w_0 P_\alpha(\Qp)}\pi_\alpha\big)^{\an}&\buildrel\sim\over\longrightarrow &\big(\cInd_{P_\alpha(\Qp) \cap w_0 P_\alpha(\Qp)w_0}^{P_\alpha(\Qp)}\pi_{P_\beta}\big)^{\an} \\
\nonumber &\buildrel\sim\over\longrightarrow &\big(\Ind_{L_{P_\alpha}(\Qp)\cap P_\beta^-(\Qp)}^{L_{P_\alpha}(\Qp)}\big(\cInd_{N^\alpha(\Qp)\cap P_\beta^-(\Qp)}^{N^\alpha(\Qp)}\pi_{P_\beta}\big)^{\an}\big)^{\an}
\end{eqnarray}
o\`u le deuxi\`eme isomorphisme est donn\'e par $f\mapsto \big(l_\alpha\mapsto (n^\alpha\mapsto f(n^\alpha l_\alpha))\big)$, o\`u l'action de $L_{P_\alpha}(\Qp)$ en bas \`a droite est l'action usuelle par translation \`a droite et l'action de $N^\alpha(\Qp)$ y est~:
\begin{equation}\label{actionnalpha}
(n^\alpha F)(l_\alpha)=(l_\alpha n^\alpha l_\alpha^{-1})(F(l_\alpha))
\end{equation}
(l'action \ \ de \ \ $l_\alpha n^\alpha l_\alpha^{-1}\in N^\alpha(\Qp)$ \ \ sur \ \ $F(l_\alpha)$ \ \ \'etant \ \ celle \ \ de \ \ $N^\alpha(\Qp)$ \ \ sur $(\cInd_{N^\alpha(\Qp)\cap P_\beta^-(\Qp)}^{N^\alpha(\Qp)}\pi_{P_\beta})^{\an}$). Posons $\pi_{N^\alpha}\=(\cInd_{N^\alpha(\Qp)\cap P_\beta^-(\Qp)}^{N^\alpha(\Qp)}\pi_{P_\beta})^{\an}$, l'argument de la preuve de \cite[Prop.~4.16]{Sc1} (en plus simple car on peut directement argumenter sur le complexe $\Hom_{\Qp}(\bigwedge^\cdot \mnn^\alpha,-)$) bas\'e sur \cite[Cor.~4.14]{Sc1} montre que {\it si} les $H^i(\mnn^\alpha,(\cInd_{N^\alpha(\Qp)\cap P_\beta^-(\Qp)}^{N^\alpha(\Qp)}\pi_{P_\beta})^{\an})$ sont s\'epar\'es (donc des espaces de type compact), alors on a des isomorphismes topologiques ${P_\alpha}(\Qp)$-\'equivariants pour $i\in \Z$~:
\begin{equation}\label{grosse}
H^i\Big(\mnn^\alpha,\big(\cInd_{P_\alpha(\Qp)}^{P_\alpha(\Qp) w_0 P_\alpha(\Qp)}\pi_\alpha\big)^{\an}\Big)
\buildrel\sim\over\longrightarrow \big(\Ind_{L_{P_\alpha}(\Qp)\cap P_\beta^-(\Qp)}^{L_{P_\alpha}(\Qp)}H^i(\mnn^\alpha,\pi_{N^\alpha})\big)^{\an}
\end{equation}
et en particulier les $H^i(\mnn^\alpha,(\cInd_{P_\alpha(\Qp)}^{P_\alpha(\Qp) w_0 P_\alpha(\Qp)}\pi_\alpha)^{\an})$ sont aussi s\'epar\'es. Notons $\mnn_{\beta}$ la $\Qp$-alg\`ebre de Lie de $N^\alpha(\Qp)\cap P_\beta^-(\Qp)=N(\Qp)\cap L_{P_\beta}(\Qp)=N_\beta(\Qp)$, on a $\mnn^\alpha=(\mnn^\alpha/\mnn_{\beta})\oplus \mnn_{\beta}$ (car $N^\alpha$ est commutatif pour ${\rm GL}_3$) et une preuve analogue \`a celle de \cite[(4.63)]{Sc1} donne des isomorphismes d'espaces de type compact~:
\begin{equation*}
H^i(\mnn_{\beta},\pi_{N^\alpha})\simeq C_c^{\an}\big(N^\alpha(\Qp)/N^\alpha(\Qp)\cap P_\beta^-(\Qp),E\big)\otimes_{E,\iota} H^i(\mnn_{\beta},\pi_{P_\beta})
\end{equation*}
o\`u $H^i(\mnn_{\beta},\pi_{P_\beta})\simeq H^i(\mnn_{\beta},L(-\mu)_{P})\otimes \pi_{P_\beta}^\infty$ est muni de la topologie localement convexe la plus fine. Une application de la suite spectrale de Hochshild-Serre (\cite[\S~7.5]{We}) donne alors un isomorphisme continu pour $i\in \Z$~:
\begin{multline}\label{grossei=1}
H^i(\mnn^{\alpha},\pi_{N^\alpha})\\
\buildrel\sim\over\longrightarrow H^0\Big(\mnn^\alpha/\mnn_{\beta},C_c^{\an}\big(N^\alpha(\Qp)/N^\alpha(\Qp)\cap P_\beta^-(\Qp),E\big)\otimes_E H^i(\mnn_{\beta},\pi_{P_\beta})\Big)\\
\simeq C_c^{\infty}\big(N^\alpha(\Qp)/N^\alpha(\Qp)\cap P_\beta^-(\Qp),E\big)\otimes_E H^i(\mnn_{\beta},\pi_{P_\beta})
\end{multline}
qui est forc\'ement un isomorphisme topologique puisque le terme de droite est muni de la topologie localement convexe la plus fine. En particulier les espaces $H^i(\mnn^{\alpha},\pi_{N^\alpha})$ sont s\'epar\'es, donc aussi les espaces $H^i(\mnn^\alpha,(\cInd_{P_\alpha(\Qp)}^{P_\alpha(\Qp) w_0 P_\alpha(\Qp)}\pi_\alpha)^{\an})$ par (\ref{grosse}). De plus on d\'eduit facilement de (\ref{grossei=1}) que l'on a un isomorphisme de repr\'esentations localement alg\'ebriques de $L_{P_\alpha}(\Qp)\cap P_\beta^-(\Qp)=L_{P_\alpha}(\Qp)\cap B^-(\Qp)$~:
\begin{equation}\label{localgi=1}
H^i(\mnn^{\alpha},\pi_{N^\alpha})\simeq H^i(\mnn_{\beta},L(-\mu)_{P_\beta})\otimes_E \big(\cInd_{N^\alpha(\Qp)\cap P_\beta^-(\Qp)}^{N^\alpha(\Qp)}\pi_{P_\beta}^\infty\big)^{\infty}
\end{equation}
o\`u $H^i(\mnn_{\beta},L(-\mu)_{P_\beta})$ est nul si $i\notin \{0,1\}$ et a dimension $1$ si $i\in \{0,1\}$ (l'action de $L_{P_\alpha}(\Qp)\cap B^-(\Qp)$ sur $H^i(\mnn_{\beta},L(-\mu)_{P_\beta})$ se factorise par $L_{P_\alpha}(\Qp)\cap B^-(\Qp)\twoheadrightarrow T(\Qp)$).

Une preuve analogue \`a celle de \cite[(4.76)]{Sc1} (bas\'ee sur les trois premi\`eres lignes de la preuve de \cite[Th.~7.4]{Ko2} et sur celles de \cite[Lem.~8.4]{Ko2} et \cite[Th.~8.5]{Ko2}) montre que les groupes d'homologie $H_i(\mnn^\alpha,(\pi_C)^\vee)$ pour $i\in \Z$ (\cite[\S~3]{ST3}) sont de la forme $V_i^\vee \otimes_E ((\pi_{P_\beta}^\infty)^{w_0})^\vee$ o\`u $V_i$ est une repr\'esentation alg\'ebrique de $L_{P_\alpha}(\Qp)$ de dimension finie, donc sont s\'epar\'es. L'argument \`a la fin de la preuve de \cite[Th.~3.15]{Sc1} donne alors des isomorphismes topologiques $L_{P_\alpha}(\Qp)$-\'equivariants $H^i(\mnn^\alpha,\pi_C)\simeq H_i(\mnn^\alpha,(\pi_C)^\vee)^\vee\simeq V_i\otimes_E (\pi_{P_\beta}^\infty)^{w_0}$, en particulier les $H^i(\mnn^\alpha,\pi_C)$ sont munis de la topologie localement convexe la plus fine et sont des repr\'esentations localement alg\'ebriques de $L_{P_\alpha}(\Qp)$.

Pour appliquer le Lemme \ref{separedevisse} avec l'analogue de (\ref{decompgl3}) et $i=1$, il reste \`a v\'erifier que l'image de $H^0(\mnn^\alpha,\pi_C)\rightarrow H^1(\mnn^\alpha,(\cInd_{P_\alpha(\Qp)}^{P_\alpha(\Qp) w_0 P_\alpha(\Qp)}\pi_\alpha)^{\an})$ est ferm\'ee. Mais par ce qui pr\'ec\`ede, l'image tombe dans les vecteurs $V_0$-localement alg\'ebriques (pour $L_{P_\alpha}(\Qp)$) de l'induite (\ref{grosse}). Il suffit donc de v\'erifier que, si $V$ est une repr\'esentation alg\'ebrique de dimension finie de $L_{P_\alpha}(\Qp)$ sur $E$, les vecteurs $V$-localement alg\'ebriques de (\ref{grosse}) sont ferm\'es et munis de la topologie localement convexe la plus fine. Notons ${\mathfrak l}_{P_\alpha}$ la $\Qp$-alg\`ebre de Lie de $L_{P_\alpha}(\Qp)$, $W\= H^1(\mnn_{\beta},L(-\mu))$ et $\pi_{N^\alpha}^\infty\=(\cInd_{N^\alpha(\Qp)\cap P_\beta^-(\Qp)}^{N^\alpha(\Qp)}\pi_{P_\beta}^\infty)^{\infty}$ (avec topologie localement convexe la plus fine). En raisonnant comme dans \cite[p.~20]{OS} on a un isomorphisme topologique ${\mathfrak l}_{P_\alpha}$-\'equivariant~:
\begin{equation}\label{grossebis}
\big(\Ind_{L_{P_\alpha}(\Qp)\cap P_\beta^-(\Qp)}^{L_{P_\alpha}(\Qp)}W\otimes \pi_{N^\alpha}^\infty\big)^{\an}\simeq \big(\Ind_{L_{P_\alpha}(\Qp)\cap P_\beta^-(\Qp)}^{L_{P_\alpha}(\Qp)}W\big)^{\an}\otimes_{E,\iota}\pi_{N^\alpha}^\infty
\end{equation}
et par \cite[Prop.~4.2.10]{Em} les vecteurs $V$-localement alg\'ebriques de (\ref{grossebis}) sont~:
$$\Big(\!\Hom_{{\mathfrak l}_{P_\alpha}}\big(V,\big(\Ind_{L_{P_\alpha}(\Qp)\cap P_\beta^-(\Qp)}^{L_{P_\alpha}(\Qp)}W\big)^{\an})\otimes_E V\Big)\otimes_{E,\iota}\pi_{N^\alpha}^\infty$$
o\`u $\Hom_{{\mathfrak l}_{P_\alpha}}(V,\big(\Ind_{L_{P_\alpha}(\Qp)\cap P_\beta^-(\Qp)}^{L_{P_\alpha}(\Qp)}W)^{\an})\otimes_E V$ est muni de la topologie induite par $(\Ind_{L_{P_\alpha}(\Qp)\cap P_\beta^-(\Qp)}^{L_{P_\alpha}(\Qp)}W)^{\an}$ via l'application d'\'evaluation (qui est injective). Mais en se souvenant que $L_{P_\alpha}(\Qp)\simeq {\rm GL}_2(\Qp)\times \Qp^\times$ et que $W$ est un caract\`ere alg\'ebrique du tore $T(\Qp)$, la structure de $(\Ind_{L_{P_\alpha}(\Qp)\cap P_\beta^-(\Qp)}^{L_{P_\alpha}(\Qp)}W)^{\an}$ est la m\^eme que celle des s\'eries principales localement analytiques de ${\rm GL}_2(\Qp)$, o\`u il est clair que cette topologie induite est la topologie localement convexe la plus fine. Cela ach\`eve la preuve.
\end{proof}

La Proposition \ref{separegl3} et le Lemme \ref{separexact} impliquent le corollaire suivant.

\begin{cor}\label{separexactgl3}
Soit $m\in \Z_{\geq 0}$ et $0\rightarrow \pi''\rightarrow \pi\rightarrow \pi'\rightarrow 0$ une suite exacte dans $\Repm(B(\Qp))$ avec $\pi''\simeq (\Ind_{P^-(\Qp)}^{G(\Qp)} L(-\mu)_{P}\otimes_E\pi_{P}^\infty)^{\an}$ o\`u $P\in \{P_\alpha,P_\beta\}$, $\mu\in X(T)$ et $\pi_{P}^\infty$ est de longueur finie. On a un complexe de $(\psi,\Gamma)$-modules de Fr\'echet sur $\Rm^+$~:
$$M_\alpha(H^1(\mnn^\alpha,\pi''))\buildrel f\over\longrightarrow M_\alpha(\pi')\buildrel g\over\longrightarrow M_\alpha(\pi)\longrightarrow M_\alpha(\pi'')\longrightarrow 0$$
qui est exact en $M_\alpha(\pi'')$ et $M_\alpha(\pi)$ et tel que l'image de $f$ est dense dans $\ker(g)$.
\end{cor}

On note ${\det}:L_{P_\alpha}\rightarrow {\mathbb G}_{\rm m}\times {\mathbb G}_{\rm m}$ le morphisme de groupes alg\'ebriques induit par le d\'eterminant de ${\rm GL}_2$. La proposition suivante permet de se d\'ebarrasser des $H^1(\mnn^\alpha,-)$ en se ramenant \`a une suite exacte comme en (\ref{avecmtilde}).

\begin{prop}\label{pushtechnique}
Soit $0\rightarrow \pi''\rightarrow \pi\rightarrow \pi'\rightarrow 0$ une suite exacte dans $\Rep(B(\Qp))$ avec $\pi''$ comme dans le Corollaire \ref{separexactgl3} et $F_\alpha(\pi')(-)\simeq E_\infty(\chi_{-\lambda})\otimes_{E}\Hom_{(\varphi,\Gamma)}(D_\alpha(\pi'),-)$ pour un $(\varphi,\Gamma)$-module g\'en\'eralis\'e $D_{\alpha}(\pi')$ sur $\R$. Si $P=P_\alpha$ on suppose de plus que l'action de $L_{P_\alpha}(\Qp)$ sur $\pi_{P_\alpha}^\infty$ se factorise par ${\det}:L_{P_\alpha}(\Qp)\rightarrow \Qp^\times \times \Qp^\times$. Alors il existe $m\gg 0$, $r\gg 0$ tels que le ``push-out'' de la suite exacte $M_\alpha(\pi'\otimes_EE_m)\rightarrow M_\alpha(\pi\otimes_EE_m)\rightarrow M_\alpha(\pi''\otimes_EE_m)\rightarrow 0$ (cf. Proposition \ref{exactdense}) le long du morphisme $f:M_\alpha(\pi'\otimes_EE_m)\longrightarrow D_{\alpha}(\pi')_r\otimes_EE_m$ (cf. (\ref{canonimor})) donne une suite exacte courte de $(\psi,\Gamma)$-modules de Fr\'echet sur $\Rm^+$~:
$$0\longrightarrow D_{\alpha}(\pi')_r\otimes_EE_m\longrightarrow \widetilde M_\alpha(\pi\otimes_EE_m)\longrightarrow M_\alpha(\pi''\otimes_EE_m)\longrightarrow 0$$
o\`u $\widetilde M_\alpha(\pi\otimes_EE_m)\= (D_{\alpha}(\pi')_r\otimes_EE_m)\oplus_{f,M_\alpha(\pi'\otimes_EE_m)}M_\alpha(\pi\otimes_EE_m)$.
\end{prop}
\begin{proof}
Par le Corollaire \ref{separexactgl3}, il suffit comme dans l'\'Etape $1$ au \S~\ref{preuvelocalg} de montrer que, quitte \`a augmenter $m$ et $r$, la compos\'ee $M_\alpha(H^1(\mnn^\alpha,\pi'')\otimes_EE_m)\rightarrow M_\alpha(\pi'\otimes_EE_m)\buildrel f \over\rightarrow D_{\alpha}(\pi')_r\otimes_EE_m$ est nulle. Comme dans la preuve de la Proposition \ref{separegl3}, on distingue deux cas.

\noindent
Cas $1$~: $P=P_\alpha$.\\
Par le Cas $1$ dans la preuve de la Proposition \ref{separegl3} et un d\'evissage sur la suite exacte (\ref{decompgl3}) (appliqu\'ee avec $\pi''$), on a $H^1(\mnn^\alpha,\pi'')\buildrel\sim\over\rightarrow H^1(\mnn^\alpha,\pi_C)$ et donc $M_\alpha(H^1(\mnn^\alpha,\pi_C)\otimes_EE_m)\buildrel\sim\over\rightarrow M_\alpha(H^1(\mnn^\alpha,\pi'')\otimes_EE_m)$. Mais il d\'ecoule de (\ref{techniquei=2}) pour $i=1$ avec la phrase qui suit, (des r\'ef\'erences) du Cas $1$ de la preuve de la Proposition \ref{separegl3} (cf. en particulier la discussion pr\'ec\'edent \cite[(4.100)]{Sc1}), de (\ref{malphaf}) et de $P_\beta\cap P_\alpha=B$ que l'on a un morphisme d'image dense~:
\begin{multline}\label{pluspeniblequeprevu}
\big(H_1(\mnn^\alpha,\pi_\beta^\vee \otimes_{U({\mathfrak p}_\beta,E)}U({\mathfrak gl}_3,E_m))\otimes_{D(B(\Qp),E_m)}D(P_\alpha(\Qp),E_m)\big)(\eta)_{N_m^\alpha}\\
\longrightarrow M_\alpha(H^1(\mnn^\alpha,\pi_C)\otimes_EE_m)
\end{multline}
o\`u ${\mathfrak p}_\beta$ est la $\Qp$-alg\`ebre de Lie de $P_\beta(\Qp)$. En se souvenant que $\pi_\beta\simeq (L(-\mu)_{P_\alpha})^{w_0}\otimes_E(\pi_{P_\alpha}^\infty)^{w_0}$ on a~:
\begin{multline*}
H_1(\mnn^\alpha,\pi_\beta^\vee \otimes_{U({\mathfrak p}_\beta,E)}U({\mathfrak gl}_3,E_m))\\
\simeq H_1(\mnn^\alpha,(L^-(\mu)_{P_\alpha})^{w_0} \otimes_{U({\mathfrak p}_\beta,E)}U({\mathfrak gl}_3,E_m))\otimes _{E} (\pi_{P_\alpha}^\infty)^{w_0}
\end{multline*}
d'o\`u on d\'eduit que l'action de $N^\alpha(\Qp)$ est triviale sur $H_1(\mnn^\alpha,\pi_\beta^\vee \otimes_{U({\mathfrak p}_\beta,E)}U({\mathfrak gl}_3,E_m))$ car elle est triviale sur chaque facteur (on utilise ici l'hypoth\`ese sur $\pi_{P_\alpha}^\infty$). Comme $P_\alpha$ normalise $N^\alpha$, on voit que l'action de $N^\alpha(\Qp)$, et {\it a fortiori} de $N_m^\alpha$, est aussi triviale sur $H_1(\mnn^\alpha,\pi_\beta^\vee \otimes_{U({\mathfrak p}_\beta,E)}U({\mathfrak gl}_3,E_m))\otimes_{D(B(\Qp),E_m)}D(P_\alpha(\Qp),E_m)$. Cela implique que le terme de gauche en (\ref{pluspeniblequeprevu}) est nul (pour $m$ tel que $\eta\vert_{N_m^\alpha}\ne 1$), et donc aussi $M_\alpha(H^1(\mnn^\alpha,\pi_C)\otimes_EE_m)$ et $M_\alpha(H^1(\mnn^\alpha,\pi'')\otimes_EE_m)$.

\noindent
Cas $2$~: $P=P_\beta$.\\
Par le Cas $2$ de la preuve de la Proposition \ref{separegl3} et par la suite exacte (\ref{decompgl3}) (appliqu\'ee avec $\pi''$), on a une suite exacte d'espaces de type compact munis d'une action lisse de $N^\alpha(\Qp)$~:
\begin{multline*}
H^0(\mnn^\alpha\!,\pi_C)\longrightarrow H^1\big(\mnn^\alpha\!,\!\big(\cInd_{P_\alpha(\Qp)}^{P_\alpha(\Qp) w_0 P_\alpha(\Qp)}\pi_\alpha\big)^{\an} \big) \longrightarrow H^1(\mnn^\alpha\!,\pi'')\longrightarrow H^1(\mnn^\alpha\!,\pi_C).
\end{multline*}
L'action de $N^\alpha(\Qp)$ \'etant triviale sur $H^i(\mnn^\alpha,\pi_C)$ pour tout $i$ (cf. l'avant-dernier paragraphe de la preuve de la Proposition \ref{separegl3}), on en d\'eduit un isomorphisme d'espaces de type compact (pour $m\gg 0$)~:
$$\big(H^1\big(\mnn^\alpha,\big(\cInd_{P_\alpha(\Qp)}^{P_\alpha(\Qp) w_0 P_\alpha(\Qp)}\pi_\alpha\big)^{\an}\big) \otimes_EE_m\big)(\eta^{-1})_{N_m^\alpha} \buildrel\sim\over\longrightarrow \big(H^1(\mnn^\alpha,\pi'')\otimes_EE_m\big)(\eta^{-1})_{N_m^\alpha}$$
et donc un isomorphisme~:
\begin{equation}\label{malphah1}
M_\alpha\big(H^1(\mnn^\alpha,\pi'')\otimes_EE_m\big)\buildrel\sim\over\longrightarrow M_\alpha\big(H^1\big(\mnn^\alpha,\big(\cInd_{P_\alpha(\Qp)}^{P_\alpha(\Qp) w_0 P_\alpha(\Qp)}\pi_\alpha\big)^{\an}\big)\otimes_EE_m\big).
\end{equation}
Avec (\ref{grosse}), il suffit donc de montrer que le morphisme dans $F(\varphi,\Gamma)_\infty$~:
$$F_\alpha(\pi')\longrightarrow F_\alpha\big(\big(\Ind_{L_{P_\alpha}(\Qp)\cap P_\beta^-(\Qp)}^{L_{P_\alpha}(\Qp)}H^1(\mnn^\alpha,\pi_{N^\alpha})\big)^{\an}\big)$$
est nul. Par le Lemme \ref{egalchi} et l'hypoth\`ese sur $\pi'$, il suffit (comme dans l'\'Etape $1$ au \S~\ref{preuvelocalg}) de montrer que le terme de droite est de la forme $E_\infty(\chi_{-\nu})\otimes_E \Hom_{(\varphi,\Gamma)}(D,-)$ pour un certain $(\varphi,\Gamma)$-module $D$ et un caract\`ere $\nu\in X(T)$ {\it distinct} de $\lambda$ et $s_\alpha\cdot \lambda$.

Pour $i\in \Z$ notons~:
$$\pi^i_{N_\alpha}\=H^i(\mnn^\alpha,\pi_{N^\alpha})\buildrel{(\ref{localgi=1})}\over\simeq H^i(\mnn_{\beta},L(-\mu)_{P_\beta})\otimes_E \pi_{N_\alpha}^\infty.$$
En d\'ecomposant~:
\begin{multline*}
L_{P_\alpha}(\Qp)=(L_{P_\alpha}(\Qp)\cap B^-(\Qp))(L_{P_\alpha}(\Qp)\cap B(\Qp))\\
\amalg (L_{P_\alpha}(\Qp)\cap B^-(\Qp))s_\alpha(L_{P_\alpha}(\Qp)\cap B(\Qp)),
\end{multline*}
on peut d\'evisser $(\Ind_{L_{P_\alpha}(\Qp)\cap P_\beta^-(\Qp)}^{L_{P_\alpha}(\Qp)}\pi^i_{N_\alpha})^{\an}$ comme au \S~\ref{alphaqques}. La contribution de la cellule ouverte donne $C_c^{\an}(N_\alpha(\Qp),\pi^i_{N^\alpha})$ comme au \S~\ref{cellule}, mais il faut ici faire attention \`a l'action (\ref{actionnalpha}) de $N^\alpha_m$. Si $\delta\in \{\alpha,\beta\}$ notons $N_{\delta,m}\=N_{\delta}(\Qp)\cap N_m$. Par un argument comme pour le Lemme \ref{niveaum}, on se ram\`ene \`a d\'eterminer $\displaystyle \lim_{m\rightarrow +\infty}F_{\alpha,m}(C^{\an}(N_{\alpha,m},\pi^i_{N^\alpha}))$. Soit $\gamma\=\alpha+\beta=e_1-e_3$ (cf. Exemple \ref{gln}), le caract\`ere $\eta$ \'etant trivial sur $N_\gamma(\Qp)$ et l'action de $N_{\gamma}(\Qp)$ commutant \`a celle de $N_\alpha(\Qp)$ on a~:
$$C^{\an}(N_{\alpha,m},\pi^i_{N^\alpha}\otimes_EE_m)(\eta^{-1})_{N_{\gamma,m}}\simeq C^{\an}\big(N_{\alpha,m},(\pi^i_{N^\alpha}\otimes_EE_m)_{N_{\gamma,m}}\big)$$
et l'action de $N_{\beta}(\Qp)$ commutant \`a celle de $N_\alpha(\Qp)$ dans le quotient $N(\Qp)/N_{\gamma}(\Qp)$, on d\'eduit~:
$$C^{\an}(N_{\alpha,m},\pi^i_{N^\alpha}\otimes_EE_m)(\eta^{-1})_{N^\alpha_{m}}\simeq C^{\an}(N_{\alpha,m},E_m)\otimes_{E_m}\big((\pi^i_{N^\alpha}\otimes_EE_m)(\eta^{-1})_{N^\alpha_{m}}\big).$$
Par \cite[Th.~4.10]{Sc1} on a $H^1(\mnn_{\beta},L(-\mu)_{P_\beta})\simeq E(-s_\beta\cdot\mu)$, d'o\`u $(\pi^1_{N^\alpha}\otimes_EE_m)(\eta^{-1})_{N^\alpha_{m}}\simeq E(-s_\beta\cdot\mu)\otimes_E (\pi_{N^\alpha}^\infty\otimes_EE_m)(\eta^{-1})_{N^\alpha_{m}}$. De plus, on v\'erifie facilement (en utilisant $\eta\vert_{N_\gamma(\Qp)}=1$) que l'on a~:
$$\displaystyle \lim_{m\rightarrow +\infty}(\pi_{N^\alpha}^\infty\otimes_EE_m)(\eta^{-1})_{N^\alpha_{m}}=(\pi_{N^\alpha}^\infty\otimes_EE_\infty)(\eta^{-1})_{N^\alpha(\Qp)}\buildrel\sim\over\rightarrow (\pi_{P_\beta}^\infty\otimes_EE_\infty)(\eta^{-1})_{N_{L_{P_\beta}}\!(\Qp)}$$
o\`u le deuxi\`eme isomorphisme est induit par $f\mapsto f(1)$. Un argument comme dans la preuve de la Proposition \ref{celluleouverte} donne alors $\displaystyle \lim_{m\rightarrow +\infty}F_{\alpha,m}(C^{\an}(N_{\alpha,m},\pi^1_{N^\alpha}))\simeq E_\infty(\chi_{-s_\beta\cdot\mu})\otimes_E \Hom_{(\varphi,\Gamma)}(D,-)$ pour $D$ convenable.

Lorsque $i=0$, il suit de (\ref{piCbis}) que l'on a un isomorphisme $(\cInd_{P_\alpha(\Qp)}^{P_\alpha(\Qp) w_0 B(\Qp)}\!\pi_\alpha)^{\an}\!\buildrel\sim\over\rightarrow C_c^{\an}(N_\alpha(\Qp),\pi_{N^\alpha})$ d'o\`u on d\'eduit avec (\ref{grosse})~:
\begin{equation}\label{grosseter}
\big(\cInd_{P_\alpha(\Qp)}^{P_\alpha(\Qp) w_0 B(\Qp)}\pi_\alpha\big)^{\an}[\mnn^\alpha]\buildrel\sim\over\longrightarrow C_c^{\an}(N_\alpha(\Qp),\pi^0_{N^\alpha}).
\end{equation}
Par l'exactitude \`a gauche de $F_\alpha$ et la Proposition \ref{support} on sait que~:
\begin{multline*}
F_\alpha\big(\big(\cInd_{P_\alpha(\Qp)}^{P_\alpha(\Qp) w_0 P_\alpha(\Qp)}\pi_\alpha\big)^{\an}[\mnn^\alpha]/\big(\cInd_{P_\alpha(\Qp)}^{P_\alpha(\Qp) w_0 B(\Qp)}\pi_\alpha\big)^{\an}[\mnn^\alpha] \big)\\
\hookrightarrow F_\alpha\big(\big(\Ind_{P_\alpha(\Qp)}^{{\rm GL}_3(\Qp)}\pi_\alpha\big)^{\an}/\big(\cInd_{P_\alpha(\Qp)}^{P_\alpha(\Qp) w_0 B(\Qp)}\pi_\alpha\big)^{\an} \big)=0
\end{multline*}
d'o\`u par (\ref{grosse}) et (\ref{grosseter})~:
\begin{equation}\label{celluleinter}
F_\alpha\big((\Ind_{L_{P_\alpha}(\Qp)\cap P_\beta^-(\Qp)}^{L_{P_\alpha}(\Qp)}\pi^0_{N^\alpha})^{\an}/C_c^{\an}(N_\alpha(\Qp),\pi^0_{N^\alpha})\big)=0.
\end{equation}
Mais les repr\'esentations $\pi^i_{N^\alpha}$ sont toutes de la forme : (caract\`ere alg\'ebrique)$\otimes_E \pi_{N^\alpha}^\infty$, et il est clair que la nullit\'e en (\ref{celluleinter}) n'a rien \`a voir avec le caract\`ere alg\'ebrique particulier apparaissant dans $\pi^0_{N^\alpha}$. Autrement dit on a pour tout $i$~:
\begin{equation}\label{celluleinteri}
F_\alpha\big((\Ind_{L_{P_\alpha}(\Qp)\cap P_\beta^-(\Qp)}^{L_{P_\alpha}(\Qp)}\pi^i_{N^\alpha})^{\an}/C_c^{\an}(N_\alpha(\Qp),\pi^i_{N^\alpha})\big)=0.
\end{equation}
Par le paragraphe d'avant, on a donc en particulier~:
\begin{multline*}
E_\infty(\chi_{-s_\beta\cdot\mu})\otimes_E \Hom_{(\varphi,\Gamma)}(D,-)\simeq F_\alpha(C_c^{\an}(N_\alpha(\Qp),\pi^1_{N^\alpha}))\\
\buildrel\sim\over\longrightarrow F_\alpha\big((\Ind_{L_{P_\alpha}(\Qp)\cap P_\beta^-(\Qp)}^{L_{P_\alpha}(\Qp)}\pi^1_{N^\alpha})^{\an}\big).
\end{multline*}
Mais comme $\mu$ est par hypoth\`ese dominant par rapport \`a $B^-\cap L_{P_\beta}$, $s_\beta\cdot \mu$ ne peut plus l'\^etre, et est donc toujours distinct de $\lambda$ et $s_\alpha\cdot \lambda$ qui eux le sont. Cela ach\`eve la preuve du Cas $2$.
\end{proof}

\begin{rem}
{\rm La Proposition \ref{pushtechnique} reste en fait valable sans supposer que l'action de $L_{P_\alpha}(\Qp)$ sur $\pi_{P_\alpha}^\infty$ se factorise par det. Il faut pour cela d\'ecomposer $P_\alpha(\Qp)=B(\Qp) s_\alpha B(\Qp) \amalg B(\Qp)$ dans le Cas $1$ et faire un d\'evissage comme dans le Cas $2$ en proc\'edant comme dans la preuve de la Proposition \ref{support}. Comme nous nous limiterons \`a la cat\'egorie $C_{\lambda,\alpha}$, nous n'aurons besoin que du cas de l'\'enonc\'e, qui a l'avantage d'\^etre plus court \`a montrer. Par ailleurs (\ref{celluleinteri}) peut aussi se montrer directement par des consid\'erations analogues \`a celles de la preuve de la Proposition \ref{support} (noter que, du point de vue du groupe ${\rm GL}_3(\Qp)$, on est sur la cellule $P_\alpha(\Qp) s_\beta B(\Qp)$).}
\end{rem}

\subsection{Un scindage technique}\label{technique1}

On d\'emontre un r\'esultat technique de scindage sur des $(\psi,\Gamma)$-module de Fr\'echet sur $\R^+$ (Proposition \ref{scinde3} ci-dessous).

\begin{lem}\label{scinde2}
Soit $B$ un module de Fr\'echet sur $\R^+$ qui est un espace de Banach, $M_m$ pour $m\geq 1$ des $\R^+$-modules libres de rang fini et $\prod_{m\geq 1}M_m$ le module de Fr\'echet sur $\R^+$ obtenu en prenant la topologie produit. Alors toute suite exacte (stricte)~:
$$0\longrightarrow B\longrightarrow V\longrightarrow \prod_{m\geq 1}M_m\longrightarrow 0$$
de modules de Fr\'echet sur $\R^+$ o\`u la topologie de $V$ peut \^etre d\'efinie par des semi-normes $q$ telles qu'il existe $C_{q}\in {\mathbb R}_{>0}$ v\'erifiant $q(Xv)\leq C_q q(v)$ pour tout $v\in V$ est scind\'ee.
\end{lem}
\begin{proof}
La topologie de $B$ est d\'efinie par une norme $\Vert\! -\!\Vert_B$ et celle de $V$ par des semi-normes $q_k$ (pour $k$ dans un ensemble d\'enombrable) telles que $q_k(Xv)\leq C_{q_k}q_k(v)$ pour tout $k$ et tout $v\in V$. Comme la topologie de $B$ est aussi d\'efinie par les restrictions $(q_k\vert_B)_k$, par \cite[Cor.~6.2]{Sch} il existe $C_1\in {\mathbb R}_{>0}$, $t\in \Z_{\geq 1}$ et $k_1,\dots,k_t$ tels que $\Vert b\Vert_B\leq C_1 \max(q_{k_1}(b),\dots,q_{k_t}(b))$ pour tout $b\in B$. Par ailleurs, pour toute semi-norme continue $q$ sur $V$, par \cite[Cor.~6.2]{Sch} (encore) il existe $C_2\in {\mathbb R}_{>0}$ tel que $q(b)\leq C_2\Vert b\Vert_B$ pour tout $b\in B$. En appliquant cela \`a $q\=\max(q_{k_1},\dots,q_{k_t})$, on voit que $q$ induit une semi-norme sur $B$ qui est \'equivalente \`a la norme $\Vert - \Vert_B$ (et donc est une norme sur $B$). Notons que l'on a encore $q(Xv)\leq C_qq(v)$ (pour un $C_q\in {\mathbb R}_{>0}$) pour tout $v\in V$. Notons $M\= \prod_{m\geq 1}M_m$, la semi-norme quotient $\overline q$ sur $M$ est continue, donc (toujours par \cite[Cor.~6.2]{Sch}) il existe $C_3\in {\mathbb R}_{>0}$ $s\in \Z_{\geq 1}$ et des semi-normes $p_1,\dots,p_s$ sur $M$ (dans une famille d\'enombrable de semi-normes d\'efinissant sa topologie) tels que $\overline q(x)\leq C_3\max(p_1(x),\dots,p_s(x))$ pour tout $x\in M$. Mais par d\'efinition de la topologie produit (\cite[\S~5.D]{Sch}), il existe un entier $N\gg 0$ tel que $p_j\vert_{\prod_{m\geq N}M_m}=0$ pour tout $j\in \{1,\dots,s\}$, donc on a aussi $\overline q\vert_{\prod_{m\geq N}M_m}=0$. Munissons maintenant $V$ (resp. $M$) de la topologie localement convexe d\'efinie par la seule semi-norme $q$ (resp. $\overline q$) et consid\'erons le s\'epar\'e-compl\'et\'e $\widehat V_q$ (resp. $\widehat M_{\overline q}$) associ\'e (\cite[\S~7]{Sch}). Il suit de la d\'efinition explicite de ce s\'epar\'e-compl\'et\'e (cf. \cite[p.~38]{Sch}) et du fait que $q\vert_B$ induit d\'ej\`a la topologie de $B$ que l'on a un diagramme commutatif de suites exactes~:
$$\xymatrix{0\ar[r] &B\ar[r]\ar@{=}[d]&V\ar[r]\ar[d]^{c_q}&M\ar[r]\ar[d]^{c_{\overline q}}&0\\
0\ar[r] &B\ar[r] & \widehat V_q \ar[r] & \widehat M_{\overline q} &}$$
o\`u les espaces en bas sont tous des espaces de Banach et o\`u toutes les applications sont continues. On en d\'eduit (avec \cite[Cor.~8.7]{Sch}) un isomorphisme d'espaces de Fr\'echet $\ker(c_q)\buildrel\sim\over\rightarrow \ker(c_{\overline q})$ o\`u $\ker(c_q)$ (resp. $\ker(c_{\overline q})$) s'identifie au noyau de la semi-norme $q$ (resp. $\overline q$). Il suit facilement de $q(X-)\leq C_qq(-)$ que $\ker(c_q)$ est un sous-$\R^+$-module de $V$ et de $\overline q\vert_{\prod_{m\geq N}M_m}=0$ que l'on a une injection (stricte) $\prod_{m\geq N}M_m\hookrightarrow \ker(c_{\overline q})\simeq \ker(c_q)\subseteq V$ de modules de Fr\'echet sur $\R^+$. On voit donc que la suite exacte (stricte) de modules de Fr\'echet sur $\R^+$ obtenue \`a partir de celle de l'\'enonc\'e par ``pull-back'' le long de $\prod_{m\geq N}M_m\hookrightarrow M$ (cf. (i) du Lemme \ref{strictI} et Remarque \ref{strictIf}) est scind\'ee. En choisissant une section $\R^+$-lin\'eaire continue $s:\prod_{m\geq N}M_m\hookrightarrow V$, on est ramen\'e \`a montrer que la suite exacte courte (stricte) $0\longrightarrow B\rightarrow V/s(\prod_{m\geq N}M_m) \rightarrow \prod_{N>m\geq 1}M_m\rightarrow 0$ de modules de Fr\'echet sur $\R^+$ est aussi scind\'ee (on obtient alors un scindage de la suite exacte de l'\'enonc\'e en consid\'erant $V\twoheadrightarrow V/s(\prod_{m\geq N}M_m)\simeq B\oplus \prod_{N>m\geq 1}M_m\twoheadrightarrow B$). Mais cela d\'ecoule facilement du fait que $\prod_{N>m\geq 1}M_m$ est libre de rang fini sur $\R^+$ (il suffit de relever une base sur $\R^+$ et d'utiliser \cite[Cor.~8.7]{Sch}).
\end{proof}

\begin{rem}
{\rm On peut montrer que le Lemme \ref{scinde2} est faux si l'on suppose seulement que $B$ est un module de {\it Fr\'echet} sur $\R^+$.}
\end{rem}

La suite du paragraphe est consacr\'ee \`a la d\'emonstration de la proposition suivante.

\begin{prop}\label{scinde3}
Soit $r\in \Q_{>p-1}$, $D_r$ un $(\varphi,\Gamma)$-module g\'en\'eralis\'e sur $\Rr$ sans torsion, $M_m$ pour $m\geq 1$ des $\R^+$-modules libres de rang fini et $\prod_{m\geq 1}M_m$ le module de Fr\'echet sur $\R^+$ obtenu en prenant la topologie produit. On suppose que $M\=\prod_{m\geq 1}M_m$ est muni d'une structure de $(\psi,\Gamma)$-module de Fr\'echet sur $\R^+$ telle que, pour tout $m\geq 1$, $\Gamma$ pr\'eserve $M_m$ et $\psi$ induit un morphisme de $E$-espaces vectoriels $M_m\rightarrow M_{m+1}$. Alors toute suite exacte (stricte)~:
\begin{equation}\label{rstrict}
0\longrightarrow D_r\longrightarrow V\longrightarrow M\longrightarrow 0
\end{equation}
de $(\psi,\Gamma)$-modules de Fr\'echet sur $\R^+$ o\`u la topologie de $V$ peut \^etre d\'efinie par des semi-normes $q$ telles qu'il existe $C_{q}\in {\mathbb R}_{>0}$ v\'erifiant $q(Xv)\leq C_q q(v)$ pour tout $v\in V$ devient scind\'ee apr\`es ``push-out'' le long de $D_r\rightarrow D_{r'}$ pour $r'\gg r$.
\end{prop}

Notons que le (ii) du Lemme \ref{strictI} avec la Remarque \ref{strictIf} montrent que la suite exacte obtenue apr\`es ``push-out'' le long de $D_r\rightarrow D_{r'}$ est encore une suite exacte (stricte) de $(\psi,\Gamma)$-modules de Fr\'echet sur $\R^+$.

Si $C$ est un $(\psi,\Gamma)$-module de Fr\'echet sur $\R^+$, pour tout $j\geq 1$ on a une application $\R^+$-lin\'eaire qui commute \`a $\Gamma$ (cf. (\ref{psi3}) pour $j=1$)~:
\begin{eqnarray}\label{psi4}
\nonumber \widetilde \psi_j:C&\longrightarrow &\R^+\otimes_{\varphi^j,\R^+}C\simeq \oplus_{i=0}^{p^j-1}(1+X)^{i}\otimes C\\
v&\longmapsto &\sum_{i=0}^{p^j-1}(1+X)^i\otimes \frac{\psi^j((1+X)^{p^j-i}v)}{(1+X)}.
\end{eqnarray}
Si $C$ est un $(\varphi,\Gamma)$-module g\'en\'eralis\'e sur $\Rr$, l'application $\widetilde \psi_j$ se factorise comme suit (cf. (\ref{verspsi}) et (\ref{psi2}) pour $j=1$)~:
\begin{equation}\label{psi6}
C\buildrel 1\otimes \Id\over \longrightarrow {\mathcal R}_E^{p^jr}\otimes_{{\mathcal R}_E^{r}}C\buildrel {\buildrel (\Id\otimes \varphi^j)^{-1}\over\sim}\over\longrightarrow {\mathcal R}_E^{p^jr}\otimes_{\varphi^j,\Rr}C\buildrel\sim\over\longleftarrow \R^+\otimes_{\varphi^j,\R^+}C.
\end{equation}

\begin{lem}\label{scindepsi}
Avec les notations et hypoth\`eses de la Proposition \ref{scinde3}, il suffit de montrer l'existence d'un scindage de (\ref{rstrict}) pour les modules de Fr\'echet sur $\R^+$ sous-jacents, i.e. en oubliant $\psi$, $\Gamma$.
\end{lem} 
\begin{proof}
On suppose que l'on a un scindage $V\cong D_r\oplus M$ de modules de Fr\'echet sur $\R^+$, et il faut montrer que, quitte \`a modifier ce scindage de $\R^+$-modules (en augmentant \'eventuellement $r$), on peut se ramener \`a une situation o\`u il est de plus compatible aux actions de $\psi$ et de $\Gamma$. On note $\gamma$ un g\'en\'erateur topologique de $\Gamma$, par continuit\'e de l'action de $\Gamma$, il suffit de consid\'erer les actions de $\psi$ et $\gamma$.

Le morphisme $\psi:D_r\oplus M\rightarrow D_r\oplus M$ (resp. avec $\gamma$ au lieu de $\psi$) induit le morphisme $E$-lin\'eaire (compos\'e) $\delta\psi:M\hookrightarrow D_r\oplus M\buildrel\psi\over\longrightarrow D_r\oplus M\twoheadrightarrow D_r$ (resp. avec $\delta\gamma$ au lieu de $\delta\psi$). Par le m\^eme argument que celui bornant le support de la compos\'ee en (\ref{topologielibre}) (qui utilise le fait que $D_r$ est libre de rang fini), on obtient $N\gg 0$ tel que $\delta\psi\vert_{\prod_{m\geq N}M_m}=\delta\gamma\vert_{\prod_{m\geq N}M_m}=0$, autrement dit le ``pull-back'' de (\ref{rstrict}) le long de $\prod_{m\geq N}M_m\hookrightarrow \prod_{m\geq 1}M_m$ est scind\'e comme suite exacte de $(\psi,\Gamma)$-modules de Fr\'echet sur $\R^+$. 

Consid\'erons alors la suite exacte (stricte) de $(\psi,\Gamma)$-modules de Fr\'echet sur $\R^+$~:
$$0\longrightarrow D_r\longrightarrow W\longrightarrow \prod_{N>m\geq 1}M_m\longrightarrow 0$$
o\`u $W\=V/\prod_{m\geq N}M_m$ et $\prod_{N>m\geq 1}M_m\simeq \prod_{m\geq 1}M_m/\prod_{m\geq N}M_m$, donc par hypoth\`ese $\psi^N=0$ sur $\prod_{N>m\geq 1}M_m$. Le morphisme ${\widetilde\Psi}_N$ en (\ref{psi4}) induit alors avec (\ref{psi6}) une suite exacte o\`u tous les morphismes sont $\R^+$-lin\'eaires et commutent \`a $\Gamma$~:
\begin{equation}\label{ppsiN}
\begin{gathered}
\xymatrix{0\ar[r] &D_{p^Nr}\ar[r]\ar[d]^\wr&D_{p^Nr}\oplus_{D_r}W\ar[r]\ar[d]^{{\widetilde\Psi}_N}&M\ar[r]\ar[d]^{0}&0\\
0\ar[r] &\R^+\otimes_{\varphi^N,\R^+}D_r\ar[r] & \R^+\otimes_{\varphi^N,\R^+} W \ar[r] & \R^+\otimes_{\varphi^N,\R^+}M\ar[r]&0.}
\end{gathered}
\end{equation}
Avec le lemme des serpents on en d\'eduit facilement un isomorphisme $\ker({\widetilde\Psi}_N)\buildrel\sim\over\rightarrow M$. C'est un exercice de v\'erifier que $\ker({\widetilde\Psi}_N)$ est stable par $\psi$ (manipuler les \'egalit\'es $\psi^N((1+X)^{p^N-i}v)=0$ pour tout $i$ en (\ref{psi4})) et il est clairement stable par $\Gamma$. Comme $D_{p^Nr}\oplus_{D_r}W\rightarrow M$ commute \`a $\psi$, on en d\'eduit finalement une injection $\R^+$-lin\'eaire continue $M\hookrightarrow D_{p^Nr}\oplus_{D_r}W$ qui commute \`a $\psi$ et $\Gamma$. En argumentant comme \`a l'avant-derni\`ere phrase de la preuve du Lemme \ref{scinde2}, on obtient le scindage cherch\'e en rempla\c cant $r$ par $p^Nr$.
\end{proof}

On d\'emontre maintenant l'existence d'un scindage de (\ref{rstrict}) pour les modules de Fr\'echet sur $\R^+$ sous-jacents.

Pour $r\in \Q_{>p-1}$ et $s>r$ dans $\Q$, on v\'erifie facilement que l'on a $\varphi(\R^{[r,s]})\subseteq \R^{[pr,ps]}$ o\`u $\R^{[r,s]}$ est la $E$-alg\`ebre affino\"\i de de Banach d\'efinie dans la preuve du Lemme \ref{plat} et $\varphi$ est le Frobenius (induit par $\varphi(X)=(1+X)^p-1$), ainsi que $\R^{[pr,ps]}\simeq \oplus_{i=0}^{p-1}(1+X)^i\varphi(\R^{[r,s]})$. Par extension des scalaires de $\R^{pr}$ \`a $\R^{[pr,ps]}$ sur (\ref{frob}) on d\'eduit un isomorphisme $\R^{[pr,ps]}$-lin\'eaire o\`u $D_{[r',s']}\=\R^{[r',s']}\otimes_{\Rr}D_r$ pour $s'>r'\geq r$~:
\begin{equation}\label{frobborne}
\Id\otimes \varphi : \R^{[pr,ps]}\otimes_{\varphi,\R^{[r,s]}}D_{[r,s]} \buildrel\sim\over\longrightarrow D_{[pr,ps]}
\end{equation}
et l'on a comme en (\ref{verspsi}) un op\'erateur encore not\'e $\psi:D_{[r,ps]}\rightarrow D_{[r,s]}$ commutant \`a $\Gamma$ d\'efini comme la compos\'ee~:
\begin{equation}\label{verspsiborne}
D_{[r,ps]}\longrightarrow D_{[pr,ps]}\buildrel{\buildrel (\Id\otimes \varphi)^{-1}\over\sim} \over \longrightarrow \R^{[pr,ps]}\otimes_{\varphi,\R^{[r,s]}}D_{[r,s]}\buildrel {\rm pr}_0 \over\longrightarrow D_{[r,s]}.
\end{equation}
De plus, les fl\`eches de restriction induisent des diagrammes commutatifs \'evidents quand $s$ grandit qui, lorsque l'on prend la limite projective pour $s\rightarrow +\infty$ sur (\ref{verspsiborne}), redonnent (\ref{verspsi}). Comme en (\ref{psi6}) (pour $j=1$) on note~:
\begin{equation}\label{verspsi6borne}
\widetilde\psi:D_{[r,ps]}\longrightarrow \R^{[pr,ps]}\otimes_{\varphi,\R^{[r,s]}}D_{[r,s]}\buildrel\sim\over\longleftarrow \R^+\otimes_{\varphi^j,\R^+}D_{[r,s]}
\end{equation}
(o\`u la premi\`ere fl\`eche est la compos\'ee des deux fl\`eches de gauche en (\ref{verspsiborne}))
qui est $\R^+$-lin\'eaire, commute \`a $\Gamma$ et redonne (\ref{psi6}) pour $j=1$ quand on prend la limite projective pour $s\rightarrow +\infty$.

Pour tout $n\geq 1$ soit $I_n\=[r,p^nr]$, par le (ii) du Lemme \ref{strictI} et la Remarque \ref{strictIf} on a des suites exactes strictes $0\rightarrow D_{I_n}\rightarrow D_{I_n}\oplus_{D_r}V\rightarrow M\rightarrow 0$ de modules de Fr\'echet sur $\R^+$ et deux s\'eries de diagrammes commutatifs de $\R^+$-modules pour $n\geq 2$ (la premi\`ere induite par les restrictions et la deuxi\`eme par (\ref{verspsi6borne}) et (\ref{psi4}))~:
\begin{equation}\label{diagn1}
\begin{gathered}
\xymatrix{0\ar[r] &D_{I_n}\ar[r]\ar[d]&D_{I_n}\oplus_{D_r}V\ar[r]\ar[d]&M\ar[r]\ar@{=}[d]&0\\
0\ar[r] & D_{I_{n-1}}\ar[r]&D_{I_{n-1}}\oplus_{D_r}V\ar[r]&M\ar[r]&0}
\end{gathered}
\end{equation}
\begin{equation}\scriptsize\label{diagn2}
\begin{gathered}
\xymatrix{0\ar[r] &D_{I_n}\ar[r]\ar[d]^{\widetilde\psi}&D_{I_n}\oplus_{D_r}V\ar[r]\ar[d]^{\widetilde\psi}&M\ar[r]\ar[d]^{\widetilde\psi}&0\\
0\ar[r] &\R^+\otimes_{\varphi,\R^+}D_{I_{n-1}}\ar[r] & \R^+\otimes_{\varphi,\R^+}(D_{I_{n-1}}\oplus_{D_r} V) \ar[r] & \R^+\otimes_{\varphi,\R^+}M\ar[r]&0.}
\end{gathered}
\end{equation}
De plus une chasse au diagramme \'evidente sur (\ref{diagn1}) donne~:
$$V\buildrel\sim\over\longrightarrow \lim_{\substack{\longleftarrow \\ n\rightarrow +\infty}}D_{I_n}\oplus_{D_r}V.$$
Par le Lemme \ref{scinde2} (que l'on peut appliquer car $D_{I_n}$ est un espace de Banach), on a pour $n\geq 1$ des scindages de $\R^+$-modules $D_{I_n}\oplus_{D_r}V\simeq D_{I_n}\oplus M$ et pour $n\geq 2$ on d\'efinit $\delta_n{\rm res}: M\hookrightarrow D_{I_n}\oplus M\rightarrow D_{I_{n-1}}\oplus M\twoheadrightarrow D_{I_{n-1}}$. L'application continue $\delta_n{\rm res}: M=\prod_{m\geq 1}M_m\rightarrow D_{I_{n-1}}$ est toujours nulle en restriction \`a tous les $M_m$ sauf un nombre fini par l'argument de support qui suit (\ref{topologielibre}). On voit avec (\ref{diagn1}) qu'il suffit de montrer que l'on peut modifier ces scindages \`a chaque cran $n$ pour $n$ suffisamment grand de sorte que $\delta_n{\rm res}=0$.

Pour $n\geq 2$ on d\'efinit $m_n\=\max\{m\geq 1, \delta_n{\rm res}\vert_{M_m}\ne 0\}$. Si la suite $m_n$ est born\'ee quand $n\rightarrow +\infty$, alors il existe $N\gg 0$ tel que $\delta_n{\rm res}\vert_{\prod_{m\geq N}M_m}=0$ pour tout $n\geq 2$, i.e. le ``pull-back'' de (\ref{rstrict}) le long de $\prod_{m\geq N}M_m\hookrightarrow M$ est scind\'e comme suite exacte de modules de Fr\'echet sur $\R^+$, et le m\^eme argument qu'\`a la fin de la preuve du Lemme \ref{scinde2} donne alors un scindage de (\ref{rstrict}) (pour les structures de modules de Fr\'echet sur $\R^+$). On suppose donc que la suite $m_n$ n'est pas born\'ee. 

Le m\^eme argument de support que pour $\delta_n{\rm res}$ donne que pour $n\geq 2$ la compos\'ee~:
$$\delta_n\widetilde\psi: M\hookrightarrow D_{I_n}\oplus M\buildrel\widetilde\psi\over\rightarrow \R^+\otimes_{\varphi,\R^+}D_{I_{n-1}}\oplus \R^+\otimes_{\varphi,\R^+}M\twoheadrightarrow \R^+\otimes_{\varphi,\R^+}D_{I_{n-1}}$$
est nulle sur tous les $M_m$ sauf un nombre fini. On d\'efinit $j_2\=\max\{m\geq 1, \delta_2\widetilde\psi\vert_{M_m}\ne 0\}$ (seule valeur dont on aura besoin). On note $n_2$ le plus petit entier $\geq 3$ tel que $m_{n_2}>\max\{j_2,m_2\}$ (comme la suite $(m_n)_{n\geq 2}$ n'est pas born\'ee un tel entier existe bien). On a en particulier $m_n\leq \max\{j_2,m_2\}\leq m_{n_2}-1$ pour tout $2\leq n<n_2$. On va montrer que l'on peut modifier le scindage $D_{I_{n_2}}\oplus_{D_r}V\simeq D_{I_{n_2}}\oplus M$ au cran $n_2$ de fa\c con \`a diminuer strictement la valeur de $m_{n_2}$. Pour cela, il suffit de montrer que $(\delta_{n_2}{\rm res})(M_{m_{n_2}})\subseteq D_{I_{n_2}}\subseteq D_{I_{n_2-1}}$ (rappelons qu'{\it a priori} on a seulement $(\delta_{n_2}{\rm res})(M_{m_{n_2}})\subseteq D_{I_{n_2-1}}$). En effet, on peut alors modifier une base sur $\R^+$ de $M_{m_{n_2}}$ dans $D_{I_{n_2}}\oplus M$ par des \'el\'ements de $D_{I_{n_2}}$ de sorte que le nouveau scindage au cran $n_2$ v\'erifie $(\delta_{n_2}{\rm res})\vert_{M_{m_{n_2}}}=0$.

Pour cela, consid\'erons le diagramme commutatif~:
\begin{equation}\label{mn2}
\begin{gathered}
\xymatrix{D_{I_{n_2}}\oplus M_{m_{n_2}}\ar[r]^{\!\!\!\!\!\!\!\!\!\!\!\!\!\!\!\!\!\!\!\!\!\!\!\!\!\!\!\!\!\!\!\!\!\!\!\!\!\!\!\!\!\!\!\!\widetilde\psi}\ar[d]&\R^+\otimes_{\varphi,\R^+}D_{I_{n_2-1}}\oplus \R^+\otimes_{\varphi,\R^+}M_{m_{n_2}+1}\ar[d]\\
D_{I_{2}}\oplus M_{m_{n_2}}\ar[r]^{\!\!\!\!\!\!\!\!\!\!\!\!\!\!\!\!\!\!\!\!\!\!\!\!\!\!\!\!\!\!\!\!\!\!\!\!\!\!\!\!\widetilde\psi}&\R^+\otimes_{\varphi,\R^+}D_{I_{1}}\oplus \R^+\otimes_{\varphi,\R^+}M_{m_{n_2}+1}}
\end{gathered}
\end{equation}
o\`u la fl\`eche verticale \`a gauche (resp. \`a droite) est la restriction du cran $n_2$ au cran $2$ (resp. du cran $n_2-1$ au cran $1$). Comme $m_{n_2}>j_2$, par d\'efinition de $j_2$ l'application $\widetilde\psi$ du bas respecte les sommes directes \`a gauche et \`a droite. Comme $m_{n_2}>m_n$, et donc {\it a fortiori} $m_{n_2}+1>m_n$, pour $2\leq n\leq n_2-1$, la fl\`eche verticale de droite respecte aussi les sommes directes en haut et en bas. On en d\'eduit un diagramme commutatif de $\R^+$-modules (o\`u $\delta{\rm res}$ est la compos\'ee $M_{m_{n_2}}\subseteq D_{I_{n_2}}\oplus M_{m_{n_2}}\rightarrow D_{I_{2}}\oplus M_{m_{n_2}}\twoheadrightarrow D_{I_{2}}$)~:
\begin{equation*}
\begin{gathered}
\xymatrix{M_{m_{n_2}}\ar[r]^{\!\!\!\!\!\!\!\!\!\!\!\!\!\!\!\!\!\!\!\!\!\!\!\delta_{n_2}\widetilde\psi}\ar[d]^{\delta{\rm res}}&\R^+\otimes_{\varphi,\R^+}D_{I_{n_2-1}}\ar@{^{(}->}[d]\\
D_{I_{2}}\ar[r]^{\!\!\!\!\!\!\!\!\!\!\!\!\!\!\!\!\!\!\!\!\widetilde\psi}&\R^+\otimes_{\varphi,\R^+}D_{I_{1}}}
\end{gathered}
\end{equation*}
qui montre que $(\delta{\rm res})(M_{m_{n_2}})\subseteq \widetilde\psi^{-1}(\R^+\otimes_{\varphi,\R^+}D_{I_{n_2-1}})\subseteq D_{I_2}$ (en voyant $\R^+\otimes_{\varphi,\R^+}D_{I_{n_2-1}}$ comme sous-$\R^+$-module de $\R^+\otimes_{\varphi,\R^+}D_{I_{1}}$). Le diagramme commutatif de $\R^+$-modules (cart\'esien puisque les applications verticales sont injectives)~:
\begin{equation*}
\begin{gathered}
\xymatrix{D_{pI_{n_2-1}}\ar[r]^{\!\!\!\!\!\!\!\!\!\!\!\!\!\!\!\!\!\!\sim}\ar@{^{(}->}[d]&\R^+\otimes_{\varphi,\R^+}D_{I_{n_2-1}}\ar@{^{(}->}[d]\\
D_{pI_{1}}\ar[r]^{\!\!\!\!\!\!\!\!\!\!\!\!\!\!\!\!\!\!\sim}&\R^+\otimes_{\varphi,\R^+}D_{I_{1}}}
\end{gathered}
\end{equation*}
o\`u $pI_n\=[pr,p^{n+1}r]$ implique $\widetilde\psi^{-1}(\R^+\otimes_{\varphi,\R^+}D_{I_{n_2-1}})=D_{I_2}\cap D_{pI_{n_2-1}}$ o\`u l'intersection est dans $D_{pI_1}$. Comme $\{p^{-1/r}\leq \norm \leq p^{-1/p^2r}\} \cup \{p^{-1/pr}\leq \norm \leq p^{-1/p^{n_2}r}\}$ est un recouvrement affino\"\i de de $\{p^{-1/r}\leq \norm \leq p^{-1/p^{n_2}r}\}$, on a $D_{I_2}\cap D_{pI_{n_2-1}}=D_{I_{n_2}}\subseteq D_{pI_1}$. On en d\'eduit $(\delta{\rm res})(M_{m_{n_2}})\subseteq D_{I_{n_2}}\subseteq D_{I_2}$. Comme $m_{n_2}>m_n$ pour $2\leq n\leq n_2-1$, cela implique bien $(\delta_{n_2}{\rm res})(M_{m_{n_2}})\subseteq D_{I_{n_2}}$ dans $D_{I_{n_2-1}}$.

On peut alors recommencer tout le raisonnement pr\'ec\'edent avec ce nouveau scindage $D_{I_{n_2}}\oplus_{D_r}V\simeq D_{I_{n_2}}\oplus M$ au cran $n_2$. On obtient une nouvelle suite $m'_n$, qui est \'egale \`a $m_n$ pour $n<n_2$ et $n>n_2+1$, et une nouvelle valeur $n'_2$, et on voit que l'on est dans l'une des deux situations suivantes : soit $n'_2=n_2$ et $\max\{j_2,m_2\}<m'_{n_2}<m_{n_2}$, soit $n'_2>n_2$ et $m'_{n_2}\leq \max\{j_2,m_2\}$. Puis on recommence encore etc. Par r\'ecurrence, on voit que l'on obtient au bout du compte un scindage $D_{I_{n}}\oplus_{D_r}V\simeq D_{I_{n}}\oplus M$ pour tout $n\geq 1$ tel que $m_n\leq \max\{j_2,m_2\}$ pour tout $n\geq 2$. Autrement dit on est ramen\'e \`a la situation o\`u la suite $m_n$ est born\'ee quand $n\rightarrow +\infty$, qui a \'et\'e trait\'ee au d\'ebut du cas $D_r$ libre.

\begin{rem}\label{restevalable}
{\rm (i) L'argument ci-dessus du scindage de (\ref{rstrict}) pour les $\R^+$-modules sous-jacents reste valable si l'on a un entier $C\geq 0$ tel que $\psi:M_m\rightarrow M_m\oplus M_{m+1}\oplus \cdots \oplus M_{m+C}$ pour tout $m\geq 1$ (au lieu de $\psi:M_m\rightarrow M_{m+1}$) en rempla\c cant $M_{m_{n_2}+1}$ par $M_{m_{n_2}}\oplus \cdots \oplus M_{m_{n_2}+C}$ dans (\ref{mn2}).\\
(ii) On peut montrer que la Proposition \ref{scinde3} et le Lemme \ref{scindepsi} restent vrais sans supposer que $D_r$ est libre, mais cela n\'ecessite d'autres arguments (en particulier pour la preuve du Lemme \ref{scindepsi}).}
\end{rem}

\subsection{Un deuxi\`eme scindage technique}\label{premierscindage}

On d\'emontre un autre r\'esultat technique de scindage dans le cas $G={\rm GL}_3$ (Proposition \ref{scinde1dur} ci-dessous).

On conserve les notations des \S\S~\ref{prelgen} et \ref{devis1}, en particulier $G={\rm GL}_3$ et $B=$ Borel sup\'erieur. Soit $\pi=(\Ind_{P^-(\Qp)}^{G(\Qp)} \pi_P)^{\an}$ avec $P\in \{P_{\alpha},P_{\beta}\}$ et $\pi_P=L(-\mu)_{P}\otimes_E\pi_P^\infty$ comme dans la Proposition \ref{separegl3} et sa preuve. On rappelle que, si $P=P_\alpha$, on a la suite exacte (\ref{decompgl3}) dans $\Rep(B(\Qp))$~:
$$0\longrightarrow C^{\an}_c(N^\alpha(\Qp),\pi_{P_\alpha})\longrightarrow \pi \longrightarrow \pi_C\longrightarrow 0,$$
et si $P=P_\beta$, on a l'analogue de (\ref{decompgl3}) avec (\ref{piCbis})~:
$$0\longrightarrow \big(\Ind_{L_{P_\alpha}(\Qp)\cap B^-(\Qp)}^{L_{P_\alpha}(\Qp)}\pi_{N^\alpha}\big)^{\an}\longrightarrow \pi\longrightarrow \pi_{C}\longrightarrow 0$$
o\`u $\pi_{N^\alpha}=(\cInd_{N^\alpha(\Qp)\cap P_\beta^-(\Qp)}^{N^\alpha(\Qp)}\pi_{P_\beta})^{\an}$. Comme dans la Proposition \ref{pushtechnique}, on suppose de plus que l'action de $L_{P_\alpha}(\Qp)$ sur $\pi_{P_\alpha}^\infty$ se factorise par ${\det}:L_{P_\alpha}(\Qp)\rightarrow \Qp^\times \times \Qp^\times$.

\begin{lem}\label{scinde1}
Soit $m\in \Z_{\geq 0}$.\\
(i) Si $P=P_\alpha$, on a $M_\alpha(\pi\otimes_EE_m)\buildrel\sim\over\rightarrow M_\alpha(C^{\an}_c(N^\alpha(\Qp),\pi_{P_\alpha})\otimes_EE_m)$.\\
(ii) Si $P\!=\!P_\beta$ et $\eta\vert_{N_m^\alpha}\ne 1$, on a $M_\alpha(\pi\otimes_EE_m)\!\buildrel\sim\over\rightarrow \! M_\alpha((\Ind_{L_{P_\alpha}(\Qp)\cap B^-(\Qp)}^{L_{P_\alpha}(\Qp)}\pi_{N^\alpha})^{\an}\otimes_EE_m)$.
\end{lem}
\begin{proof}
(i) C'est la m\^eme preuve qu'au d\'ebut du Cas $1$ de la preuve de la Proposition \ref{pushtechnique} mais avec $H^0(\mnn^\alpha,-)$ au lieu de $H^1(\mnn^\alpha,-)$. (ii) C'est la m\^eme preuve que (\ref{malphah1}) mais avec (encore) $H^0(\mnn^\alpha,-)$ au lieu de $H^1(\mnn^\alpha,-)$.
\end{proof}

On suppose maintenant $P=P_\beta$ jusqu'\`a la fin de ce paragraphe. Rappelons que l'on a un diagramme commutatif dans $\Rep(B(\Qp))$ o\`u toutes les injections sont strictes~:
\begin{equation*}
\begin{gathered}
\xymatrix{ C^{\an}_c(N_\alpha(\Qp),\pi_{N^\alpha}) \ar@{^{(}->}[r] & \big(\Ind_{L_{P_\alpha}(\Qp)\cap B^-(\Qp)}^{L_{P_\alpha}(\Qp)}\pi_{N^\alpha}\big)^{\an} \\
C^{\an}_c(N_\alpha(\Qp),\pi^0_{N^\alpha}) \ar@{^{(}->}[r] \ar@{^{(}->}[u] & \big(\Ind_{L_{P_\alpha}(\Qp)\cap B^-(\Qp)}^{L_{P_\alpha}(\Qp)}\pi^0_{N^\alpha}\big)^{\an}\ar@{^{(}->}[u]}
\end{gathered}
\end{equation*}
et o\`u $\pi^0_{N^\alpha}=H^0(\mnn_{\beta},L(-\mu)_{P_\beta})\otimes_E \pi_{N_\alpha}^\infty = E(-\mu)\otimes_E(\cInd_{N^\alpha(\Qp)\cap P_\beta^-(\Qp)}^{N^\alpha(\Qp)}\pi_{P_\beta}^\infty)^{\infty}$ (cf. (\ref{piCbis}) et ce qui suit pour l'action de $B(\Qp)$). Rappelons de plus que $\pi^0_{N^\alpha}$ est muni d'une action lisse de $N^\alpha(\Qp)$ par translation \`a droite sur les fonctions sur $N^\alpha(\Qp)$ (et action triviale sur $E(-\mu)$). Par (\ref{grosse}) pour $i=0$ et (\ref{grosseter}), avec (\ref{malpha}) et la Proposition \ref{exactdense}, on en d\'eduit un diagramme commutatif de $(\psi,\Gamma)$-modules de Fr\'echet sur $\Rm^+$~:
\begin{equation}\small\label{malpha0}
\begin{gathered}
\xymatrix{ M_\alpha\big(C^{\an}_c(N_\alpha(\Qp),\pi_{N^\alpha})\otimes_EE_m\big) \ar@{=}[d] & M_\alpha\big((\Ind_{L_{P_\alpha}(\Qp)\cap B^-(\Qp)}^{L_{P_\alpha}(\Qp)}\pi_{N^\alpha})^{\an}\otimes_EE_m\big) \ar@{^{}->>}[l] \ar@{=}[d] \\
M_\alpha\big(C^{\an}_c(N_\alpha(\Qp),\pi^0_{N^\alpha})\otimes_EE_m\big) & M_\alpha\big((\Ind_{L_{P_\alpha}(\Qp)\cap B^-(\Qp)}^{L_{P_\alpha}(\Qp)}\pi^0_{N^\alpha})^{\an}\otimes_EE_m\big). \ar@{^{}->>}[l]}
\end{gathered}
\end{equation}
Comme l'action du groupe compact $N_m^\alpha$ sur $H^0(\mnn^\alpha,(\Ind_{L_{P_\alpha}(\Qp)\cap B^-(\Qp)}^{L_{P_\alpha}(\Qp)}\pi_{N^\alpha})^{\an})\simeq (\Ind_{L_{P_\alpha}(\Qp)\cap B^-(\Qp)}^{L_{P_\alpha}(\Qp)}\pi^0_{N^\alpha})^{\an}$ est lisse, le noyau de la surjection du bas en (\ref{malpha0}) est exactement $M_\alpha(\pi^0_C\otimes_EE_m)$ o\`u $\pi^0_C\=(\Ind_{L_{P_\alpha}(\Qp)\cap B^-(\Qp)}^{L_{P_\alpha}(\Qp)}\pi^0_{N^\alpha})^{\an}/C^{\an}_c(N_\alpha(\Qp),\pi^0_{N^\alpha})$. On utilisera l'isomorphisme~:
\begin{equation}\small\label{conjsalpha}
(\Ind_{L_{P_\alpha}(\Qp)\cap B^-(\Qp)}^{L_{P_\alpha}(\Qp)}\pi^0_{N^\alpha})^{\an}\buildrel\sim\over\longrightarrow (\Ind_{L_{P_\alpha}(\Qp)\cap B(\Qp)}^{L_{P_\alpha}(\Qp)}\pi^{0,s_\alpha}_{N^\alpha})^{\an},\ F\longmapsto \big(l_\alpha\mapsto F(s_\alpha l_\alpha)\big)
\end{equation}
o\`u $l_\alpha\in L_{P_\alpha}(\Qp)\cap B(\Qp)$ agit sur $\pi^{0,s_\alpha}_{N^\alpha}$ par l'action de $s_\alpha l_\alpha s_\alpha$ sur $\pi^{0}_{N^\alpha}$ (cf. (\ref{actionnalphabis})). Cet isomorphisme commute aux actions de $L_{P_\alpha}(\Qp)$ (par translation \`a droite des deux c\^ot\'es) et aux actions de $N^\alpha(\Qp)$ en d\'efinissant cette derni\`ere sur le membre de droite comme suit (cf. (\ref{actionnalpha}))~:
\begin{equation}\label{actionnalphas}
(n^\alpha F)(l_\alpha)=(s_\alpha l_\alpha n^\alpha l_\alpha^{-1} s_\alpha)(F(l_\alpha)),\ \ n^\alpha\in N^\alpha(\Qp),\ l_\alpha\in L_{P_\alpha}(\Qp)
\end{equation}
et en remarquant que $s_\alpha N^\alpha(\Qp) s_\alpha = N^\alpha(\Qp)$. 

\begin{lem}\label{malpha0bis}
On a un isomorphisme d'espaces de Fr\'echet $M_\alpha(\pi^0_C\otimes_EE_m)\simeq M_\alpha(\pi^{0,s_\alpha}_{N^\alpha}\otimes_EE_m)\widehat\otimes_ED(N_\alpha^-(\Qp),E)_{\{1\}}$ o\`u $N_\alpha^-\=L_{P_\alpha}\cap N^-$.
\end{lem}
\begin{proof}
Par (\ref{conjsalpha}), le (i) du Lemme \ref{dualspx} puis le Lemme \ref{support1} appliqu\'es avec $G=L_{P_\alpha}$ et $P= L_{P_\alpha}\cap B$, on a un isomorphisme d'espaces de Fr\'echet $(\pi^0_C)^\vee\simeq (\pi^{0,s_\alpha}_{N^\alpha})^\vee\widehat\otimes_ED(N_\alpha^-(\Qp),E)_{\{1\}}$.

Explicitons d'abord l'action de $N^\alpha(\Qp)$ sur $(\pi^{0,s_\alpha}_{N^\alpha})^\vee\widehat\otimes_ED(N_\alpha^-(\Qp),E)_{\{1\}}$. Le plongement ferm\'e de vari\'et\'es localement $\Qp$-analytiques $N^\alpha(\Qp)\times N_\alpha^-(\Qp)\hookrightarrow N^\alpha(\Qp)\times L_{P_\alpha}(\Qp)\simeq P_\alpha(\Qp)$ induit par \cite[Prop.~1.1.2]{Ko1} et \cite[Prop.~A.3]{ST3} une immersion ferm\'ee d'espaces localement convexes~:
$$D(N^\alpha(\Qp),E)\widehat\otimes_ED(N_\alpha^-(\Qp),E)\!\hookrightarrow \!D(N^\alpha(\Qp),E)\widehat\otimes_ED(L_{P_\alpha}(\Qp),E)\!\simeq \! D(P_\alpha(\Qp),E)$$
et l'\'egalit\'e $n^\alpha l_\alpha (n^\alpha)^{-1}=(n^\alpha (l_\alpha (n^\alpha)^{-1}l_\alpha^{-1}))l_\alpha\in N^\alpha(\Qp)l_\alpha$ pour $(n^\alpha,l_\alpha)\in N^\alpha(\Qp)\times L_{P_\alpha}(\Qp)$ montre que la conjugaison par $\delta_{n^\alpha}$ dans $D(P_\alpha(\Qp),E)$ pr\'eserve le sous-espace $D(N^\alpha(\Qp),E)\widehat\otimes_ED(N_\alpha^-(\Qp),E)$, et donc aussi (via le Lemme \ref{supproduit}) le sous-espace $D(N^\alpha(\Qp),E)_{\{1\}}\widehat\otimes_ED(N_\alpha^-(\Qp),E)_{\{1\}}$. En voyant $(\pi^{0,s_\alpha}_{N^\alpha})^\vee$ comme $D(N^\alpha(\Qp),E)$-module s\'epar\'ement continu via l'action de $N^\alpha(\Qp)$ sur $\pi^{0,s_\alpha}_{N^\alpha}$ comme dans le terme de droite de (\ref{actionnalphas}), on v\'erifie facilement que l'action (\`a gauche) de $n^\alpha\in N^\alpha(\Qp)$ sur $(\pi^{0,s_\alpha}_{N^\alpha})^\vee\widehat\otimes_ED(N_\alpha^-(\Qp),E)_{\{1\}}$ par multiplication \`a droite par $\delta_{n^\alpha}^{-1}$ s'obtient (de mani\`ere similaire \`a (\ref{actiondual})) en envoyant $v\otimes \mu$ vers (via le Lemme \ref{librehat})~:
\begin{multline}\label{actionn}
n^\alpha(v)\otimes\delta_{n^\alpha}\mu\delta_{n^\alpha}^{-1}\\
\in (\pi^{0,s_\alpha}_{N^\alpha})^{\vee}\widehat\otimes_{D(N^\alpha(\Qp),E)_{\{1\}}}\big(D(N^\alpha(\Qp),E)_{\{1\}}\widehat\otimes_ED(N_\alpha^-(\Qp),E)_{\{1\}}\big) \\
\simeq (\pi^{0,s_\alpha}_{N^\alpha})^\vee\widehat\otimes_ED(N_\alpha^-(\Qp),E)_{\{1\}}
\end{multline}
(pour $n^\alpha(v)=v\delta_{n^\alpha}^{-1}$, voir les conventions juste avant (\ref{actiondual})). Soit $\mathfrak x\in \mnn_\alpha^-$ (= la $\Qp$-alg\`ebre de Lie de $N_\alpha^-(\Qp)$) et $n^\alpha\in N^\alpha(\Qp)$, on v\'erifie que $n^\alpha {\mathfrak x} (n^\alpha)^{-1}-{\mathfrak x}\in \mnn^\alpha$ dans $\mpp_\alpha$ (= la $\Qp$-alg\`ebre de Lie de $P_\alpha(\Qp)$). En uti\-lisant $[\mnn^\alpha,\mnn_\alpha^-]\subseteq \mnn^\alpha$ dans $\mpp_\alpha$ et le fait que $\mnn^\alpha$ agit par $0$ sur $(\pi^{0,s_\alpha}_{N^\alpha})^\vee$, on en d\'eduit que si $\mu\in U(\mnn^-_\alpha,E)\subseteq D(N_\alpha^-(\Qp),E)_{\{1\}}$, on a $w\otimes \delta_{n^\alpha}\mu\delta_{n^\alpha}^{-1}=w\otimes \mu$ dans $(\pi^{0,s_\alpha}_{N^\alpha})^\vee\widehat\otimes_ED(N_\alpha^-(\Qp),E)_{\{1\}}$ pour tout $w\in (\pi^{0,s_\alpha}_{N^\alpha})^\vee$. Comme $(\pi^{0,s_\alpha}_{N^\alpha})^\vee \otimes_E U(\mnn^-_\alpha,E)$ est dense dans $(\pi^{0,s_\alpha}_{N^\alpha})^\vee\widehat\otimes_ED(N_\alpha^-(\Qp),E)_{\{1\}}$ par \cite[Prop.~1.2.8]{Ko1} et le Lemme \ref{tenseur}, on en d\'eduit finalement que l'action de $N^\alpha(\Qp)$ sur $(\pi^{0,s_\alpha}_{N^\alpha})^\vee\widehat\otimes_ED(N_\alpha^-(\Qp),E)_{\{1\}}$ est juste donn\'ee par l'action sur le ``facteur'' $(\pi^{0,s_\alpha}_{N^\alpha})^\vee$.

Si $0\rightarrow V'\rightarrow V\rightarrow V''$ est une suite exacte stricte de $E$-espaces vectoriels de Fr\'echet (i.e. $V\rightarrow V''$ est d'image ferm\'ee), alors par \cite[Lem.~4.13]{Sc1} pour tout $E$-espace vectoriel de Fr\'echet $W$ on a encore une suite exacte stricte $0\rightarrow V'\widehat\otimes_EW\rightarrow V\widehat\otimes_EW\rightarrow V''\widehat\otimes_EW$ d'espaces de Fr\'echet. En utilisant que $\pi^{0,s_\alpha}_{N^\alpha}$ est muni de la topologie localement convexe la plus fine puis en dualisant, on v\'erifie ais\'ement que l'on a des suites exactes {\it strictes} d'espaces de Fr\'echet pour tout $x\in N_m^\alpha$~:
$$0\longrightarrow ((\pi^{0,s_\alpha}_{N^\alpha})^\vee\otimes_EE_m)(\eta)^{x=1}\longrightarrow (\pi^{0,s_\alpha}_{N^\alpha}\otimes_EE_m)^\vee(\eta)\buildrel ^{x-\Id}\over \longrightarrow (\pi^{0,s_\alpha}_{N^\alpha}\otimes_EE_m)^\vee(\eta)$$
qui restent donc exactes apr\`es $\widehat\otimes_ED(N_\alpha^-(\Qp),E)_{\{1\}}$. En d\'evissant $N^\alpha_m$ comme dans la preuve du Lemme \ref{gnouf} (en fait, ici $N^\alpha_m$ est m\^eme commutatif), on d\'eduit alors facilement de tout ce qui pr\'ec\`ede $(\pi^{0,s_\alpha}_{N^\alpha}\otimes_EE_m)^\vee(\eta)^{N_m^\alpha}\widehat\otimes_ED(N_\alpha^-(\Qp),E)_{\{1\}}\!\buildrel\sim\over\rightarrow ((\pi^{0,s_\alpha}_{N^\alpha}\!\otimes_EE_m)^\vee\widehat\otimes_ED(N_\alpha^-(\Qp),E)_{\{1\}})(\eta)^{N_m^\alpha}$, d'o\`u le r\'esultat.
\end{proof}

Soit ${\mathfrak y}_\alpha=\smat{0 & 0\\ 1 & 0}$ (vu dans ${\mathfrak gl}_3$) une base du $E$-espace vectoriel $\mnn_\alpha^-\otimes_{\Qp}E$, de sorte que $U(\mnn_\alpha^-,E)=\oplus_{n\geq 0}E{\mathfrak y}_\alpha^n=E[{\mathfrak y}_\alpha]$. Par \cite[Prop.~1.2.8]{Ko1} on a~:
$$D(N_\alpha^-(\Qp),E)_{\{1\}}\simeq \left\{\sum_{n\geq 0}a_n{\mathfrak y}_\alpha^n,\ a_n\in E,\ \forall\ r\in \Q_{>0}\ \vert a_n\vert r^{-n}\rightarrow 0\ {\rm qd}\ n\rightarrow +\infty\right\}.$$
De mani\`ere \'equivalente, $D(N_\alpha^-(\Qp),E)_{\{1\}}$ est le compl\'et\'e de $U(\mnn_\alpha^-,E)$ pour les semi-normes $q_{r}(\sum_{n=0}^Na_n{\mathfrak y}_\alpha^n)\=\sup_n\vert a_n\vert r^{-n}$ o\`u $r\in \Q_{>0}$, ou de mani\`ere \'equivalente pour les $q_{r_m}$ avec $(r_m)_m$ une suite d\'enombrable (quelconque) tendant vers $0$ dans $\Q_{>0}$. Si $W$ est un $E$-espace vectoriel de Fr\'echet, on note~:
$$W\{\{{\mathfrak y}_\alpha\}\}\=\left\{\sum_{n\geq 0}w_n\otimes {\mathfrak y}_\alpha^n,\ w_n\in W,\ \forall\ q,\ \forall\ r\in \Q_{>0}\ \ q(w_n) r^{-n}\rightarrow 0\ {\rm qd}\ n\rightarrow +\infty\right\}$$
o\`u $q$ parcourt les semi-normes continues sur $W$. C'est aussi un $E$-espace vectoriel de Fr\'echet pour les semi-normes $\sup_{n\geq 0}q(v_n)r^{-n}$.

\begin{lem}\label{wyalpha}
On a un isomorphisme de $E$-espaces vectoriels de Fr\'echet~:
$$W\widehat\otimes_ED(N_\alpha^-(\Qp),E)_{\{1\}}\simeq W\{\{{\mathfrak y}_\alpha\}\}.$$
\end{lem}
\begin{proof}
Munissons $U(\mnn_\alpha^-,E)$ de la topologie induite par $D(N_\alpha^-(\Qp),E)_{\{1\}}$, alors il r\'esulte de \cite[Cor.~17.5(ii)]{Sch}, de \cite[Lem.~7.3]{Sch} avec le Lemme \ref{tenseur}, et de la propri\'et\'e universelle du s\'epar\'e-compl\'et\'e (\cite[Prop.~7.5]{Sch}) que $W\widehat\otimes_ED(N_\alpha^-(\Qp),E)_{\{1\}}$ s'identifie au compl\'et\'e de $W\otimes_{E,\pi}U(\mnn_\alpha^-,E)$, c'est-\`a-dire par \cite[Lem.~17.2]{Sch} au compl\'et\'e de $W\otimes_{E}U(\mnn_\alpha^-,E)$ pour les semi-normes produit tensoriel $q\otimes q_r$ pour $r\in \Q_{>0}$ et $q$ semi-norme continue sur $W$. Mais en appliquant \cite[Lem.~17.3]{Sch} avec $V\=U(\mnn_\alpha^-,E)$, $v_i\={\mathfrak y}_\alpha^i$ (et $W\=W$ !), et par l'argument de la preuve de \cite[Prop.~17.4(i)]{Sch}, on a pour tout $N\geq 0$ et tout $w_n\in W$, $0\leq n\leq N$~:
$$q\otimes q_r\Big(\sum_{n=0}^Nw_n\otimes {\mathfrak y}_\alpha^n\Big)=\sup_nq(w_n)q_r({\mathfrak y}_\alpha^n)=\sup_nq(v_n)r^{-n}$$
d'o\`u on d\'eduit que le compl\'et\'e de $W\otimes_{E,\pi}U(\mnn_\alpha^-,E)$ s'identifie \`a $W\{\{{\mathfrak y}_\alpha\}\}$.
\end{proof}

Pour all\'eger les notations, on pose dans la suite $M\=M_\alpha(\pi^0_C\otimes_EE_m)$ et $M^\alpha\=M_\alpha(\pi^{0,s_\alpha}_{N^\alpha}\otimes_EE_m)$, de sorte que, par le Lemme \ref{malpha0bis} et le Lemme \ref{wyalpha}, on a $M\simeq M^\alpha\{\{{\mathfrak y}_\alpha\}\}$. On rappelle que $t=\log(1+X)\in \R^+$ (\S~\ref{gen}).

\begin{lem}\label{techMalpha}
Pour $N\in \Z_{\geq 0}$, le $E$-espace vectoriel $M^\alpha\otimes_E (E\oplus E{\mathfrak y}_\alpha\oplus \cdots \oplus E{\mathfrak y}_\alpha^N)$ est un sous-$\R^+$-module de $M^\alpha\widehat\otimes_ED(N_\alpha^-(\Qp),E)_{\{1\}}$ stable par $\psi$ et $\Gamma$ tel que $t^{N+1}(M^\alpha\otimes_E (E\oplus E{\mathfrak y}_\alpha\oplus \cdots \oplus E{\mathfrak y}_\alpha^N))=0$.
\end{lem}
\begin{proof}
L'action de $L_{P_\alpha}(\Qp)\cap B(\Qp)$ sur $(\pi^{0,s_\alpha}_{N^\alpha})^\vee\widehat\otimes_ED(N_\alpha^-(\Qp),E)_{\{1\}}$ est analogue \`a (\ref{actionn}) en se souvenant que $(\pi^{0,s_\alpha}_{N^\alpha})^\vee$ est un $D(L_{P_\alpha}(\Qp)\cap B(\Qp),E)$-module via (\ref{actionnalphabis}) (conjugu\'e par $s_\alpha$). Par ailleurs un calcul donne dans la $\Qp$-alg\`ebre de Lie ${\mathfrak l}_{P_\alpha}$ de $L_{P_\alpha}(\Qp)$ en notant $n_\alpha(x)\=\smat{1 & x \\ 0 & 1}\in N_\alpha(\Qp)$ et ${\mathfrak h}_\alpha\=\smat{1 & 0 \\ 0 & -1}\in {\mathfrak l}_{P_\alpha}$ (o\`u l'on n'a \'ecrit que ce qui concerne le facteur ${\rm GL}_2$)~:
\begin{equation}\label{nalpha(x)}
n_\alpha(x) {\mathfrak y}_\alpha n_\alpha(-x) - ({\mathfrak y}_\alpha + x{\mathfrak h}_\alpha) \in \mnn_\alpha.
\end{equation}
On en d\'eduit facilement la premi\`ere assertion. Pour la deuxi\`eme, il suffit de montrer $t(w \otimes_E {\mathfrak y}_\alpha^N)\in M^\alpha\otimes_E (E\oplus E{\mathfrak y}_\alpha\oplus \cdots \oplus E{\mathfrak y}_\alpha^{N-1})$ pour tout $w\in M^\alpha$. Comme la suite $p^{-n}\varphi^n(X)=p^{-n}((1+X)^{p^n}-1)$ converge dans $\R^+$ vers $t$ quand $n$ tend vers $+\infty$, il suffit de montrer que la limite de $p^{-n}\varphi^n(X)(w \otimes_E {\mathfrak y}^N)$ quand $n\rightarrow +\infty$ dans le $\R^+$-module continu $M^\alpha\widehat\otimes_ED(N_\alpha^-(\Qp),E)_{\{1\}}$ tombe dans le sous-espace $M^\alpha\otimes_E (E\oplus E{\mathfrak y}_\alpha\oplus \cdots \oplus E{\mathfrak y}_\alpha^{N-1})$. En utilisant les crochets de Lie habituels dans ${\mathfrak l}_{P_\alpha}$, on d\'eduit facilement de (\ref{nalpha(x)})~:
\begin{multline*}
p^{-n}(1+X)^{p^n}(w\otimes {\mathfrak y}_\alpha^N)=p^{-n}\big((1+X)^{p^n}w\otimes n_\alpha(p^n){\mathfrak y}_\alpha^N n_\alpha(-p^n)\big)\\
\in \big(p^{-n}(1+X)^{p^n}w\big)\otimes {\mathfrak y}_\alpha^N + M^\alpha\otimes_E (E\oplus E{\mathfrak y}_\alpha\oplus \cdots \oplus E{\mathfrak y}_\alpha^{N-1})
\end{multline*}
d'o\`u le r\'esultat puisque $\displaystyle \lim_{n\rightarrow +\infty} ((p^{-n}(1+X)^{p^n}w)\otimes {\mathfrak y}_\alpha^N-p^{-n}w\otimes {\mathfrak y}_\alpha^N)=tw\otimes {\mathfrak y}_\alpha^N =0$ (car $t(\pi^{0,s_\alpha}_{N^\alpha})^\vee=0$) et $M^\alpha\otimes_E (E\oplus E{\mathfrak y}_\alpha\oplus \cdots \oplus E{\mathfrak y}_\alpha^{N-1})$ est ferm\'e dans $M^\alpha\widehat\otimes_ED(N_\alpha^-(\Qp),E)_{\{1\}}$ (p. ex. par le Lemme \ref{wyalpha}).
\end{proof}

Le but de la suite du paragraphe est de montrer la proposition suivante.

\begin{prop}\label{scinde1dur}
Soit $m\in \Z_{\geq 0}$, $r\in \Q_{>p-1}$ et $D_r$ un $(\varphi,\Gamma)$-module sans torsion sur $\Rrm$. Alors toute suite exacte (pour $M$ comme ci-dessus)~:
\begin{equation}\label{rstrictbis}
0\longrightarrow D_r\longrightarrow V\longrightarrow M\longrightarrow 0
\end{equation}
de $(\psi,\Gamma)$-modules de Fr\'echet sur $\Rm^+$, o\`u la topologie de $V$ peut \^etre d\'efinie par des semi-normes $q$ telles que $\forall\ S\in \Rm^+$ il existe $C_{q,S}\in {\mathbb R}_{>0}$ v\'erifiant $q(Sv)\leq C_{q,S}q(v)$ pour tout $v\in V$, devient scind\'ee apr\`es ``push-out'' le long de $D_r\rightarrow D_{r'}$ pour $r'\gg r$.
\end{prop}

On commence par un cas particulier utile de la Proposition \ref{scinde1dur}.

\begin{lem}\label{casdavant}
Avec les notations et hypoth\`eses de la Proposition \ref{scinde1dur}, supposons de plus que pour tout $N\geq 0$ le ``pull-back'' $V_N$ de (\ref{rstrictbis}) le long de $M_N\=M^\alpha \otimes_E(E\oplus E{\mathfrak y}_\alpha\oplus \cdots \oplus E{\mathfrak y}_\alpha^N)\hookrightarrow M$ est une suite exacte scind\'ee de $(\psi,\Gamma)$-modules de Fr\'echet sur $\Rm^+$. Alors la suite exacte (\ref{rstrictbis}) est aussi scind\'ee.
\end{lem}
\begin{proof}
{\bf \'Etape $1$}\\
Pour $N\geq 0$, on a donc par hypoth\`ese un isomorphisme $V_N\simeq D_r\oplus M_N$. Soit $x\in M_N$, si $\widehat x$ est un relev\'e quelconque de $x$ dans $V$ (ou de mani\`ere \'equivalente $V_N$), on a $t^{N+1}\widehat x\in t^{N+1}D_r$ (puisque $t^{N+1}M_N=0$ par le Lemme \ref{techMalpha}), de sorte que l'on peut d\'efinir~:
\begin{equation}\label{s()Malpha}
s(x)\=\widehat x-\frac{1}{t^{N+1}}(t^{N+1}\widehat x)\in V_N\subseteq V
\end{equation}
o\`u $\frac{1}{t^{N+1}}(t^{N+1}\widehat x)\in D_r$. On v\'erifie facilement que $s(x)$ ne d\'epend pas de $N$ tel que $x\in M_N$ et d\'efinit ainsi une application $s:M^\alpha[{\mathfrak y}_\alpha]\rightarrow V$ o\`u $M^\alpha[{\mathfrak y}_\alpha]\=\cup_{N\geq 0}M_N\subseteq M$ et on v\'erifie aussi facilement qu'elle est $\R^+$-lin\'eaire et commute \`a $\psi$ et $\Gamma$. Nous allons montrer qu'elle s'\'etend par continuit\'e (de mani\`ere n\'ecessairement unique) \`a $M=M^\alpha\{\{{\mathfrak y}_\alpha\}\}$ et que l'application $s:M\rightarrow V$ obtenue est continue (automatiquement $\R^+$-lin\'eaire et commutant \`a $\psi$ et $\Gamma$).

Soit $s>r$ dans $\Q_{>0}$, ${\mathcal R}_{E_m}^{[r,s]}$ l'alg\`ebre de Banach dans la preuve du Lemme \ref{plat} et $D_{[r,s]}\=D_r\otimes_{\Rrm}{\mathcal R}_{E_m}^{[r,s]}$. On note $p_s$ une norme sur ${\mathcal R}_{E_m}^{[r,s]}$ et on note encore $p_s$ la norme ``somme directe'' induite sur le Banach $D_{[r,s]}$. Par le (ii) du Lemme \ref{strictI} et la Remarque \ref{strictIf} on a des suites exactes strictes $0\rightarrow D_{[r,s]}\rightarrow D_{[r,s]}\oplus_{D_r}V\rightarrow M\rightarrow 0$ de modules de Fr\'echet sur $\R^+$, et comme pour (\ref{diagn1}) on a un isomorphisme topologique $\displaystyle V\buildrel\sim\over\rightarrow \lim_{\substack{\longleftarrow \\ s\rightarrow +\infty}}D_{[r,s]}\oplus_{D_r}V$, de sorte qu'il suffit de montrer que la compos\'ee $M^\alpha[{\mathfrak y}_\alpha]\rightarrow V \rightarrow D_{[r,s]}\oplus_{D_r}V$ se prolonge par continuit\'e \`a $M^\alpha\{\{{\mathfrak y}_\alpha\}\}$ et que l'application obtenue est continue. Changeant de notations, on note $V$ l'espace de Fr\'echet $D_{[r,s]}\oplus_{D_r}V$, $s:M^\alpha[{\mathfrak y}_\alpha]\rightarrow V$ la compos\'ee ci-dessus (ne pas confondre $s\in \Q_{>0}$ et la section $s$ !) et $V_N$ le ``pull-back'' de $V$ le long de $M_N\hookrightarrow M$.

Par hypoth\`ese la topologie de $V$ peut \^etre d\'efinie par des semi-normes $q_k$ (pour $k$ dans un ensemble d\'enombrable) telles que $q_k(tv)\leq C_{q_k,t}q_k(v)$ pour tout $k$ et tout $v\in V$ (o\`u $C_{q_k,t}\in {\mathbb R}_{>0}$). Nous allons montrer que si $\sum_{n\geq 0}w_n\otimes {\mathfrak y}_\alpha^n\in M$ alors $q_k(s(w_n\otimes {\mathfrak y}_\alpha^n))\rightarrow 0$ quand $n\rightarrow +\infty$.

\noindent
{\bf \'Etape $2$}\\
L'id\'eal $t{\mathcal R}_{E_m}^{[r,s]}$ est ferm\'e dans ${\mathcal R}_{E_m}^{[r,s]}$ et on le munit de la topologie induite. La multiplication par $1/t:t{\mathcal R}_{E_m}^{[r,s]}\rightarrow {\mathcal R}_{E_m}^{[r,s]}$ \'etant continue (utiliser p. ex. \cite[Prop.~2.1(iii)]{ST2}), si $p_s$ est une norme sur ${\mathcal R}_{E_m}^{[r,s]}$ d\'efinissant sa topologie il existe $c_s\in {\mathbb R}_{>0}$ tel que $p_s(\frac{1}{t}x)\leq c_sp_s(x)$ pour tout $x\in t{\mathcal R}_{E_m}^{[r,s]}$. Comme la topologie de $D_{[r,s]}$ est aussi d\'efinie par les restrictions $(q_k\vert_{D_{[r,s]}})_k$ (car l'injection $D_{[r,s]}\hookrightarrow V$ est stricte), par \cite[Cor.~6.2]{Sch} (encore) appliqu\'e avec $V=D_{[r,s]}$ et $q=p_s$, il existe $D_{s}\in {\mathbb R}_{>0}$, $t\in \Z_{\geq 1}$ et $k_{1},\dots,k_{t}$ tels que pour tout $d\in D_{[r,s]}$ on a~:
\begin{equation}\label{majorbanach}
p_s(d)\leq D_s\max(q_{k_{1}}(d),\dots,q_{k_{t}}(d)).
\end{equation}
Fixons $q_k$ une semi-norme comme \`a la fin de l'\'Etape $1$, par \cite[Cor.~6.2]{Sch} appliqu\'e avec $V=D_{[r,s]}$ et $q=q_k\vert_{D_{[r,s]}}$, il existe enfin $C_k\in {\mathbb R}_{>0}$ tel que $q_k(d)\leq C_k p_s(d)$ pour tout $d\in D_{[r,s]}$. 

Consid\'erons la semi-norme continue $q\=\max(q_k,q_{k_{i}}, i\in\{1,\dots,t\})$ sur $V$ et soit $\overline q$ la semi-norme sur $M$ quotient de $q$, comme $V\twoheadrightarrow M$ est une surjection topologique, $\overline q$ est une semi-norme continue sur $M$ (cf. \cite[\S~5.B]{Sch}). Par \cite[Cor.~6.2]{Sch} appliqu\'e avec $V=M$ et $q=\overline q$, il existe $B_k\in {\mathbb R}_{>0}$, $u\in \Z_{\geq 1}$, $\chi_1,\dots,\chi_u$ des semi-normes continues sur $M^\alpha$ et $r_1,\dots,r_u\in \Q_{>0}$ tels que pour tout $x=\sum_{n\geq 0}w_n\otimes {\mathfrak y}_\alpha^n\in M=M^\alpha\{\{{\mathfrak y}_\alpha\}\}$~:
$$\overline q(x)\leq B_k {\max}\big(\sup_n \chi_1(w_n)r_1^{-n},\dots,\sup_n \chi_u(w_n) r_u^{-n}\big).$$
Soit $n\in \Z_{\geq 0}$. Si $\overline q(w_n\otimes{\mathfrak y}_\alpha^n)\ne 0$, alors par d\'efinition de $\overline q$ il existe $\widehat{w_n\otimes{\mathfrak y}_\alpha^n}\in V$ relevant $w_n\otimes{\mathfrak y}_\alpha^n\in M$ tel que~:
\begin{equation}\small\label{borneqhat}
q(\widehat{w_n\!\otimes\!{\mathfrak y}_\alpha^n})\!\leq \! B_k {\max}\big(\chi_1(w_n)r_1^{-n},\dots,\chi_u(w_n)r_u^{-n}\big)\!\leq \!B_k {\max}\big(\chi_1(w_n),\dots,\chi_u(w_n)\big) r_k^{-n}
\end{equation}
o\`u $r_k\=\min(r_1,\dots,r_u)$. Si $\chi_i(w_n)=0$ pour $i\in \{1,\dots,u\}$, ce qui implique $\overline q(w_n\otimes{\mathfrak y}_\alpha^n)= 0$, en \'ecrivant $V_n\simeq D_{[r,s]}\oplus M_n$, il existe une suite $(d_m)_m$ dans $D_{[r,s]}$ telle que $q(d_m+w_n\otimes{\mathfrak y}_\alpha^n)\rightarrow 0$ quand $m\rightarrow +\infty$, donc {\it a fortiori} $q_{k_i}(d_m+w_n\otimes{\mathfrak y}_\alpha^n)\rightarrow 0$ quand $m\rightarrow +\infty$ pour tout $i\in \{1,\dots,t\}$. Avec (\ref{majorbanach}) on en d\'eduit facilement que $(d_m)_m$ est une suite de Cauchy dans le Banach $D_{[r,s]}$, donc converge vers un \'el\'ement $d\in D_{[r,s]}$. Posant $\widehat{w_n\otimes{\mathfrak y}_\alpha^n}\=d+w_n\otimes{\mathfrak y}_\alpha^n\in D_{[r,s]}\oplus M_n$, par continuit\'e de $q$ on a $q(\widehat{w_n\otimes{\mathfrak y}_\alpha^n})=0$. On en d\'eduit finalement dans tous les cas et pour tout $n\geq 0$~:
\begin{equation}\label{dn0alpha}
q_k(\widehat{w_n\otimes{\mathfrak y}_\alpha^n})\leq q(\widehat{w_n\otimes{\mathfrak y}_\alpha^n})\leq B_k {\max}(\chi_1(w_n),\dots,\chi_u(w_n))r_k^{-n}.
\end{equation}

\noindent
{\bf \'Etape $3$}\\
Posons $C\=\max(C_{q_{k_i},t}, i\in\{1,\dots,t\})$, on a pour tout $n\geq 0$ en utilisant tout ce qui pr\'ec\`ede~:
\begin{eqnarray}\label{analysealpha}
\nonumber q_k\Big(\frac{1}{t^{n+1}}(t^{n+1}\widehat{w_n\!\otimes\!{\mathfrak y}_\alpha^n})\Big) &\leq & C_kp_{s}\Big(\frac{1}{t^{n+1}}(t^{n+1}\widehat{w_n\!\otimes\!{\mathfrak y}_\alpha^n})\Big)\\
\nonumber &\leq & C_k c_s^{n+1}p_{s}(t^{n+1}\widehat{w_n\!\otimes\!{\mathfrak y}_\alpha^n})\\
\nonumber &\leq & C_k c_s^{n+1}D_s\max_{i}\big(q_{k_{i}}(t^{n+1}\widehat{w_n\!\otimes\!{\mathfrak y}_\alpha^n})\big)\\
 \nonumber &\leq & C_k c_s^{n+1}D_sC^{n+1}\max_{i}(q_{k_{i}}(\widehat{w_n\!\otimes\!{\mathfrak y}_\alpha^n}))\\
&\leq & C_k c_sD_s C B_k{\max}_j(\chi_j(w_n))((c_sC)^{-1}r_k)^{-n}.
\end{eqnarray}

Il suit donc de (\ref{dn0alpha}), (\ref{analysealpha}) et (\ref{s()Malpha}) que, pour toute semi-norme $q_k$ sur $V$ comme pr\'ec\'edemment, il existe $H_k,h_k\in {\mathbb R}_{>0}$ et des semi-normes continues $\chi_1,\dots,\chi_u$ sur $M^\alpha$ tels que pour tout $n\geq 0$ et tout $w_n\in M^\alpha$ on a~:
\begin{equation}\label{majoralpha}
q_k(s(w_n\otimes {\mathfrak y}_\alpha^n))\leq H_k {\max}_j(\chi_j(w_n)) h_k^{-n}.
\end{equation}
On en d\'eduit en particulier $q_k(s(w_n\otimes {\mathfrak y}_\alpha^n))\rightarrow 0$ pour tout $q_k$, donc la suite $(s(\sum_{n\geq 0}^Nw_n\otimes {\mathfrak y}_\alpha^n))_N$ est une suite de Cauchy dans $V$ (cf. \cite[\S~7]{Sch}). Comme $V$ est complet, elle converge vers un \'el\'ement que l'on note $s(\sum_{n\geq 0}^{+\infty}w_n\otimes {\mathfrak y}_\alpha^n)\in V$.

\noindent
{\bf \'Etape $4$}\\
On montre finalement que l'application $s:M\rightarrow V$ est continue. Par \cite[Prop.~8.5]{Sch} (appliqu\'e au graphe de $s$) et comme on est entre espaces de Fr\'echet (donc m\'etriques), il suffit de montrer que si $x_m\rightarrow x$ dans $M$ alors $s(x_m)\rightarrow s(x)$ (pour $m\rightarrow +\infty$ dans $\Z_{\geq 0}$). Comme $s$ est clairement $E$-lin\'eaire, on peut supposer $x=0$ et par \cite[Rem.~7.1(vi)]{Sch} il suffit de montrer $q_k(s(x_m))\rightarrow 0$ quand $m\rightarrow +\infty$ pour toute semi-norme $q_k$ sur $V$ comme ci-dessus. Soit $H_k$, $h_k$, $\chi_j$ comme en (\ref{majoralpha}) (associ\'es \`a $q_k$), comme $q(x_m)\rightarrow 0$ pour toute semi-norme $q$ continue sur $M$, on a en particulier par d\'efinition de $(-)\{\{{\mathfrak y}_\alpha\}\}$ et en notant $x_m=\sum_{n\geq 0}w_{m,n}\otimes {\mathfrak y}_\alpha^n$~:
\begin{equation}\label{tendvers0alpha}
\sup_n {\max}_j(\chi_j(w_{m,n})) h_k^{-n}\rightarrow 0{\rm\ quand\ }m\rightarrow +\infty.
\end{equation}
Soit maintenant $x=\sum_{n\geq 0}w_n\otimes{\mathfrak y}_\alpha^n\in M$ quelconque. Si ${\max}_j(\chi_j(w_{n}))=0$ pour tout $n\geq 0$ alors $q_k(s(w_n\otimes {\mathfrak y}_\alpha^n))=0$ pour tout $n$ par (\ref{majoralpha}), donc $q_k(s(\sum_{n=0}^Nw_n\otimes {\mathfrak y}_\alpha^n))=0$ pour tout $N$, d'o\`u $q_k(s(x))=0$ par continuit\'e de $q_k$. Sinon $\sup_n {\max}_j(\chi_j(w_{n})) h_k^{-n}$ est un r\'eel $>0$ et il existe donc $N\gg 0$ tel que $q_k(s(x)-\sum_{n=0}^Ns(w_n\otimes {\mathfrak y}_\alpha^n))\leq H_k\sup_n {\max}_j(\chi_j(w_{n})) h_k^{-n}$, d'o\`u avec (\ref{majoralpha})~:
\begin{multline*}
q_k(s(x))\leq \max\Big(q_k\Big(\sum_{n=0}^Ns(w_n\otimes {\mathfrak y}_\alpha^n))\Big),q_k\Big(s(x)-\sum_{n=0}^Ns(w_n\otimes{\mathfrak y}_\alpha^n)\Big)\Big)\\
\leq H_k\sup_n {\max}_j(\chi_j(w_{n})) h_k^{-n}.
\end{multline*}
Donc dans tous les cas on a $q_k(s(x))\leq H_k\sup_n {\max}_j(\chi_j(w_{n})) h_k^{-n}$. Appliquant cela aux $x_m$, on a $q_k(s(x_m))\leq H_k\sup_n {\max}_j(\chi_j(w_{m,n})) h_k^{-n}$, d'o\`u le r\'esultat par (\ref{tendvers0alpha}).
\end{proof}

On d\'emontre maintenant la Proposition \ref{scinde1dur}. On rappelle que $\gamma=\alpha+\beta$.

\noindent
{\bf \'Etape $1$}\\
On a $\pi^{0,s_\alpha}_{N^\alpha}\simeq C^\infty(N_{\gamma,m},\pi_{P_\beta}^\infty)\bigoplus \big(\bigoplus_{m'\geq m+1}C^\infty(N_{\gamma,m'}\backslash N_{\gamma,m'-1},\pi_{P_\beta}^\infty)\big)$ o\`u $N_{\gamma,m'}\=N_{\gamma}(\Qp)\cap N_{m'}$ et o\`u l'action de $N^\alpha_m$ pr\'eserve chaque facteur direct. Comme dans la preuve du Lemme \ref{niveaum} (cf. en particulier (\ref{pourciter})), on en d\'eduit un isomorphisme de $(\psi,\Gamma)$-modules de Fr\'echet sur $\Rm^+$~:
\begin{equation}\label{decompmalpha}
M^\alpha\simeq \prod_{m'\geq m}M^\alpha_{m'}
\end{equation}
o\`u $M^\alpha_{m'}\= M_\alpha(E_m(-\mu)\otimes_EC^\infty(N_{\gamma,m'}\backslash N_{\gamma,m'-1},\pi_{P_\beta}^\infty))$ si $m'\geq m+1$, $M^\alpha_{m}\= M_\alpha(E_m(-\mu)\otimes_EC^\infty(N_{\gamma,m},\pi_{P_\beta}^\infty))$ (avec $M_\alpha(-)$ comme en (\ref{malpha})) et o\`u l'on v\'erifie avec l'action (\ref{actionnalphabis}) (conjugu\'ee par $s_\alpha$) que les $M^\alpha_{m'}$ pour $m'\geq m$ sont des $(\psi,\Gamma)$-modules de Fr\'echet sur $\Rm^+$. De plus, on v\'erifie facilement que (\ref{decompmalpha}) induit un isomorphisme de $E$-espaces vectoriels de Fr\'echet $M\simeq \prod_{m'\geq m}(M^\alpha_{m'}\{\{{\mathfrak y}_\alpha\}\})$ o\`u chaque facteur direct \`a droite est stable par $\R^+$, $\psi$ et $\Gamma$ dans $M$, i.e. est un $(\psi,\Gamma)$-module de Fr\'echet sur $\R^+$. Noter que $\psi$ et $\Gamma$ respectent chaque sous-espace $M^\alpha_{m'}\otimes_EE{\mathfrak y}_\alpha^n$.

Pour $N\in \Z_{\geq 0}$ et $m'\in \Z_{\geq m}$ on note $M_{m',N}\=M^\alpha_{m'}\otimes_E(E\oplus E{\mathfrak y}_\alpha\oplus \cdots \oplus E{\mathfrak y}_\alpha^N)$. Il suit de (\ref{nalpha(x)}) que $M_{m',N}$ est aussi un $(\psi,\Gamma)$-module de Fr\'echet sur $\Rm^+$ et de (\ref{decompmalpha}) que l'on a $M_N\simeq \prod_{m'\geq m}M^\alpha_{m',N}$. Comme $D_r$ est libre, la multiplication par $t^{N+1}$ sur $V_N$ donne par le Lemme \ref{techMalpha} une suite exacte de $(\psi,\Gamma)$-modules de Fr\'echet sur $\Rm^+$~:
\begin{equation}\label{cobord}
0\longrightarrow V_N[t^{N+1}]\longrightarrow M_N\buildrel f_{N+1}\over \longrightarrow D_r(\varepsilon^{-N-1})/(t^{N+1}).
\end{equation}
L'argument dans la preuve du Lemme \ref{begin} montre alors que $f_{N+1}\vert_{\prod_{m'\geq D}M_{m',N}}=0$ pour $D\gg 0$. On d\'efinit par r\'ecurrence une suite croissante d'entiers $(m_N)_{N\geq 0}$ comme suit : $m_0$ est le plus petit entier $\geq m-1$ tel que $f_{1}\vert_{\prod_{m'\geq m_0+1}M_{m',0}}=0$ et $m_N$ pour $N>0$ est le plus petit entier $\geq m_{N-1}$ tel que $f_{N+1}\vert_{\prod_{m'\geq m_N+1}M_{m',N}}=0$.

\noindent
{\bf \'Etape $2$}\\
On suppose que la suite $(m_N)_{N\geq 0}$ est major\'ee par un entier $D$ et on note $M_{> D}^\alpha\=\prod_{m'\geq D+1}M^\alpha_{m'}\subseteq M^\alpha$, $M_{> D}\=M_{> D}^\alpha\{\{{\mathfrak y}_\alpha\}\}\subseteq M$ et $V_{> D}$ le ``pull-back'' de $V$ le long de $M_{> D}\hookrightarrow M$. Pour $N\geq 0$ on d\'efinit $M_{> D,N}\=M^\alpha_{> D}\otimes_E(E\oplus E{\mathfrak y}_\alpha\oplus \cdots \oplus E{\mathfrak y}_\alpha^N)$ et $V_{> D,N}$ le ``pull-back'' de $V_{> D}$ le long de $M_{> D,N}\hookrightarrow M_{> D}$. Par d\'efinition de $D$ et par (\ref{cobord}), on a pour tout $N$ une section $M_{> D,N}\simeq V_{> D,N}[t^{N+1}]\subseteq V_{> D,N}$ qui donne un isomorphisme $V_{> D,N}\simeq D_r\oplus M_{> D,N}$ de $(\psi,\Gamma)$-modules de Fr\'echet sur $\Rm^+$. Par l'argument du Lemme \ref{casdavant}, on en d\'eduit un isomorphisme $V_{> D}\simeq D_r\oplus M_{> D}$ et une suite exacte $0\rightarrow D_r\rightarrow V_{\leq D}\rightarrow M_{\leq D}\rightarrow 0$ de $(\psi,\Gamma)$-modules de Fr\'echet sur $\Rm^+$ o\`u $V_{\leq D}\=V/M_{> D}$ et $M_{\leq D}\=M/M_{> D}\simeq M^\alpha_{\leq D}\{\{{\mathfrak y}_\alpha\}\}$ avec $M^\alpha_{\leq D}\=\prod_{D\geq m'\geq m}M^\alpha_{m'}$. 

Un calcul montre que, si $n^\alpha\in N_\beta(\Qp)N_{\gamma,D}\subseteq N^\alpha(\Qp)$, alors il existe un sous-groupe ouvert (suffisamment petit d\'ependant de $D$) de $N_\alpha(\Qp)$ tel que, pour tout $n_\alpha$ dans ce sous-groupe ouvert, on a $(s_\alpha n_\alpha s_\alpha)n^\alpha(s_\alpha n_\alpha^{-1} s_\alpha)= n^\alpha n_\beta$ avec $n_\beta\in N_\beta(\Qp)\cap N_m^\alpha$ tel que $\eta(n_\beta)=1$. On en d\'eduit que l'action (\ref{actionnalphabis}) (conjugu\'ee par $s_\alpha$) de ce sous-groupe ouvert sur~:
$$\big(E_m(-\mu)\otimes_EC^\infty(N_{\gamma,D},\pi_{P_\beta}^\infty)\big)(\eta^{-1})^{N_m^\alpha}\simeq E_m(-\mu)\otimes_E \big((\cInd_{N_\beta(\Qp)}^{N_\beta(\Qp)N_{\gamma,D}}\pi_{P_\beta}^\infty)^{\infty}(\eta^{-1})^{N_m^\alpha}\big)$$
est triviale. Avec la Remarque \ref{fixe}, on voit qu'il existe $n_0\gg 0$ tel que $\varphi^{n_0}(X)$ annule $M^\alpha_{\leq D}$. Quitte \`a augmenter $r$, on peut supposer que $\varphi^{n_0}(X)$ est inversible dans $\Rrm$. Avec la preuve du Lemme \ref{techMalpha} (avec $M^\alpha_{\leq D}$ au lieu de $M^\alpha$ et $\varphi^{n_0}(X)$ au lieu de $t$) on peut alors v\'erifier que la formule (\ref{s()Malpha}) avec $\varphi^{n_0}(X)$ au lieu de $t$ d\'efinit une section $\Rm^+$-lin\'eaire $s:M^\alpha_{\leq D}[{\mathfrak y}_\alpha]\rightarrow V_{\leq D}$ qui commute \`a $\psi$ et $\Gamma$, par exemple pour la commutation \`a $\psi$ (pour $x\in M^\alpha_{\leq D}\otimes_E(E\oplus E{\mathfrak y}_\alpha\oplus \cdots \oplus E{\mathfrak y}_\alpha^N)$)~:
\begin{multline*}
\frac{1}{\varphi^{n_0}(X)^{N+1}}(\varphi^{n_0}(X)^{N+1}\psi(\widehat x))=\frac{1}{\varphi^{n_0}(X)^{N+1}}\psi(\varphi^{n_0+1}(X)^{N+1}\widehat x)\\
=\frac{1}{\varphi^{n_0}(X)^{N+1}}\psi\Big(\varphi^{n_0+1}(X)^{N+1}\frac{1}{\varphi^{n_0}(X)^{N+1}}(\varphi^{n_0}(X)^{N+1}\widehat x)\Big)\\
=\psi\Big(\frac{1}{\varphi^{n_0}(X)^{N+1}}(\varphi^{n_0}(X)^{N+1}\widehat x)\Big).
\end{multline*}
Un raisonnement strictement analogue \`a celui de la preuve du Lemme \ref{casdavant} avec $M^\alpha_{\leq D}$ au lieu de $M^\alpha$ et $\varphi^{n_0}(X)$ au lieu de $t$ donne alors que la suite $0\rightarrow D_r\rightarrow V_{\leq D}\rightarrow M_{\leq D}\rightarrow 0$ est scind\'ee. De la surjection compos\'ee $V\twoheadrightarrow V/M_{>D}=V_{\leq D}\twoheadrightarrow D_r$ on d\'eduit ais\'ement que (\ref{rstrictbis}) est aussi scind\'ee.

\noindent
{\bf \'Etape $3$}\\
On montre finalement que la suite $(m_N)_{N\geq 0}$ est n\'ecessairement major\'ee. Supposons le contraire. Il existe alors une sous-suite $(m_{i(N)})_N$ de $(m_N)_N$ tendant vers $+\infty$ telle que l'application $f_{i(N)+1}$ en (\ref{cobord}) est {\it non nulle} en restriction \`a $M^\alpha_{m_{i(N)}}\otimes {\mathfrak y}_\alpha^{i(N)}\subseteq M_{i(N)}$. Posons par ailleurs~:
$$M^{0}\=\left\{\sum_{n\geq 0}w_{n}\otimes {\mathfrak y}_\alpha^n\in M,\ w_n\in \prod_{m'\geq m_{n}+1}M^\alpha_{m'}\right\},$$
on v\'erifie que $M^{0}$ est ferm\'e dans $M$ et stable par toutes les structures (pour la stabilit\'e par $\Rm^+$, on utilise que la suite $(m_n)_n$ est croissante), en particulier $M^{0}$ est un $(\psi,\Gamma)$-module de Fr\'echet sur $\Rm^+$. De plus, par d\'efinition de $(m_n)_n$ on a pour tout $N\geq 0$ une section $M^{0}\cap M_N\hookrightarrow V_{N}[t^{N+1}]\subseteq V_{N}$ (cf. (\ref{cobord})). Soit $V^{0}$ le ``pull-back'' de $V$ le long de $M^{0}\hookrightarrow M$, la m\^eme preuve que celle du Lemme \ref{casdavant} mais en rempla\c cant $M$ par $M^{0}$ et $V$ par $V^{0}$ (ce qui ne change pas les arguments) donne un isomorphisme $V^{0}\simeq D_r\oplus M^{0}$ et une suite exacte de $(\psi,\Gamma)$-modules de Fr\'echet sur $\Rm^+$~:
\begin{equation}\label{scinde(0)}
0\longrightarrow D_r\longrightarrow V/M^{0}\longrightarrow M/M^{0}\longrightarrow 0.
\end{equation}
En utilisant que toute semi-norme continue sur $M^\alpha$ est nulle sur $\prod_{m'\geq D+1}M^\alpha_{m'}$ pour $D\gg 0$ (par d\'efinition de la topologie produit), on v\'erifie facilement sur la d\'efinition de $M=M^\alpha\{\{{\mathfrak y}_\alpha\}\}$ que l'on a une immersion ferm\'ee d'espaces de Fr\'echet~:
\begin{equation}\label{mtilde}
\widetilde M\=\prod_{n\geq 0} \Big(\Big(\prod_{m_n\geq m'\geq m_{n-1}+1}\!\!\!M^\alpha_{m'}\Big)\otimes_EE{\mathfrak y}^n_\alpha\Big) \hookrightarrow M/M^{0}
\end{equation}
o\`u $m_{-1}\=m-1$ et le sous-espace $\widetilde M$ de gauche est muni de la topologie produit (le facteur au cran $n$ \'etant nul si $m_n=m_{n-1}$). On v\'erifie de plus que $\widetilde M$ est stable par $\Rm^+$, $\psi$ et $\Gamma$ dans $M/M^{0}$, et m\^eme que l'on a $\widetilde M\subseteq (M/M^{0})[t]$. La multiplication par $t$ sur (\ref{scinde(0)}) donne donc comme en (\ref{cobord}) une application $\widetilde M\hookrightarrow (M/M^{0})[t] \buildrel f_1 \over \rightarrow D_r(\varepsilon^{-1})/(t)$ de $(\psi,\Gamma)$-modules de Fr\'echet sur $\Rm^+$. Il suit alors de la d\'efinition de $M^{0}$ et du fait que $D_r$ est libre de rang fini sur $\Rm^+$ que, pour tout $N\geq 0$, on a des diagrammes commutatifs de $(\psi,\Gamma)$-modules de Fr\'echet sur $\Rm^+$~:
\begin{equation}\label{cobord2}
\begin{gathered}
\xymatrix{M_N\ar[r]^{\!\!\!\!\!\!\!\!\!\!\!\!\!\!\!\!\!\!\!\!\!\!\!\!f_{N+1}}\ar@{^{}->>}[d] & D_r(\varepsilon^{-N-1})/(t^{N+1})\ar@{=}[d]\\
M_N/(M^{0}\cap M_N)\ar[r]^{f_{N+1}} & D_r(\varepsilon^{-N-1})/(t^{N+1})\\
\widetilde M_N\ar@{^{(}->}[u]\ar[r]^{\!\!\!\!f_1} & D_r(\varepsilon^{-1})/(t)\ar@{^{(}->}[u]^{t^N}}
\end{gathered}
\end{equation}
o\`u $\widetilde M_N\=\widetilde M\cap (M_N/(M^{0}\cap M_N))$ et $f_{N+1}$ est comme en (\ref{cobord}). Comme l'action de $\psi$ sur $\widetilde M$ respecte chaque facteur dans le produit $\prod_{n\geq 0}$ en (\ref{mtilde}), l'argument usuel dans la preuve du Lemme \ref{begin} montre que $f_1:\widetilde M\rightarrow D_r(\varepsilon^{-1})/(t)$ est {\it nul} sur tous ces facteurs sauf un nombre fini. Le diagramme (\ref{cobord2}) avec $i(N)$ au lieu de $N$ montre alors clairement (en regardant $f_1\vert_{M^\alpha_{m_{i(N)}}\otimes {\mathfrak y}_\alpha^{i(N)}}$) qu'une suite $i(N)$ tendant vers $+\infty$ comme au d\'ebut ne peut exister. Cela termine la preuve de la Proposition \ref{scinde1dur}.

\begin{rem}\label{addendum}
{\rm Un argument similaire, en plus simple, montre que pour tout $m\in \Z_{\geq 0}$ et tout $(\varphi,\Gamma)$-module g\'en\'eralis\'e $T$ sur $\R$, on a $F_{\alpha,m}(M)(T)=0$~: pour $r\in \Q_{>p-1}$ et $f:M\rightarrow T_r$, on consid\`ere la suite croissante $(m_N)_{N\geq 0}$ o\`u $m_0$ est le plus petit entier $\geq m-1$ tel que $f\vert_{\prod_{m'\geq m_0+1}M_{m',0}}=0$ et $m_N$ pour $N>0$ le plus petit entier $\geq m_{N-1}$ tel que $f\vert_{\prod_{m'\geq m_N+1}M_{m',N}}=0$ et on montre comme dans l'\'Etape $3$ ci-dessus qu'elle est major\'ee.}
\end{rem}

\subsection{Un r\'esultat d'exactitude pour ${\rm GL}_3$ et ${\rm GL}_2$}\label{gl32}

On montre le Th\'eor\`eme \ref{gl3enplus}.

On commence avec ${\rm GL}_3$. On conserve les notations et hypoth\`eses du \S~\ref{premierscindage} et on commence par une proposition pr\'eliminaire. On suppose $P=P_\beta$, i.e. $\pi=(\Ind_{P_\beta^-(\Qp)}^{G(\Qp)} \pi_{P_\beta})^{\an}$. Le deuxi\`eme isomorphisme en (\ref{piCbis}) induit un isomorphisme $C^{\an}_c(N^\beta(\Qp),\pi_{P_\beta})\simeq (\cInd_{B(\Qp) \cap P_\beta^-(\Qp)}^{B(\Qp)}\pi_{P_\beta})^{\an}\buildrel\sim\over\rightarrow C^{\an}_c(N_\alpha(\Qp),\pi_{N^\alpha})$ dans $\Rep(B(\Qp))$ o\`u l'action de $N^\beta(\Qp)$ \`a gauche est la translation \`a droite et celle de $B(\Qp) \cap P_\beta^-(\Qp)$ donn\'ee par $(l_\beta f)(n^\beta)\=l_\beta(f(l_\beta^{-1}n^\beta l_\beta))$ pour $l_\beta\in B(\Qp) \cap P_\beta^-(\Qp)=B(\Qp) \cap L_{P_\beta}(\Qp)$ et $n^\beta\in N^\beta(\Qp)$. De plus les preuves du Lemme \ref{niveaum} et de la Proposition \ref{celluleouverte} donnent pour $m\in \Z_{\geq 0}$ un morphisme surjectif de $(\psi,\Gamma)$-modules de Fr\'echet sur $\Rm^+$~:
\begin{equation}\label{surjcellule}
f:M_\alpha\big(C^{\an}_c(N^\beta(\Qp),\pi_{P_\beta})\otimes_EE_m\big)\twoheadrightarrow D_\alpha\big(C^{\an}_c(N^\beta(\Qp),\pi_{P_\beta})\big)_+\otimes_EE_m
\end{equation}
o\`u $D_\alpha(C^{\an}_c(N^\beta(\Qp),\pi_{P_\beta}))_+\=\Rm^+(((-\mu)\circ \lambda_{\alpha^\vee}^{-1})(\chi_{\pi_{P_\beta}^\infty}\circ \lambda_{\alpha^\vee}^{-1}))$ si $\pi_{P_\beta}^\infty$ est une repr\'esentation g\'en\'erique de $L_{P_\beta}(\Qp)$ et $D_\alpha(C^{\an}_c(N^\beta(\Qp),\pi_{P_\beta}))_+\=0$ sinon.

\begin{prop}\label{scinde2ouvert}
Soit $m\in \Z_{\geq 0}$, $r\in \Q_{>p-1}$ et $D_r$ un $(\varphi,\Gamma)$-module sans torsion sur $\Rrm$. Alors, quitte \`a augmenter $m$ (par ``pull-back'') et $r$ (par ``push-out''), toute suite exacte (stricte) de $(\psi,\Gamma)$-modules de Fr\'echet sur $\Rm^+$~:
\begin{equation}\label{rstrictter}
0\longrightarrow D_r\longrightarrow V\longrightarrow M_\alpha\big(C^{\an}_c(N^\beta(\Qp),\pi_{P_\beta})\otimes_EE_m\big)\longrightarrow 0
\end{equation}
o\`u la topologie de $V$ peut \^etre d\'efinie par des semi-normes $q$ telles qu'il existe $C_{q}\in {\mathbb R}_{>0}$ v\'erifiant $q(Xv)\leq C_{q}q(v)$ pour tout $v\in V$, s'ins\`ere dans un diagramme commutatif de $(\psi,\Gamma)$-modules de Fr\'echet sur $\Rm^+$~:
\begin{equation*}
\begin{gathered}
\xymatrix{0\ar[r] & D_r\ar[r] \ar@{=}[d] & V \ar[r] \ar@{^{}->>}[d] & M_\alpha\big(C^{\an}_c(N^\beta(\Qp),\pi_{P_\beta})\otimes_EE_m\big)\ar[r] \ar@{^{}->>}[d]^{f} & 0\\
0\ar[r] & D_r\ar[r] & W \ar[r] & D_\alpha\big(C^{\an}_c(N^\beta(\Qp),\pi_{P_\beta})\big)_+\otimes_EE_m\ar[r] & 0}
\end{gathered}
\end{equation*}
o\`u l'on a pour tout $(\varphi,\Gamma)$-module g\'en\'eralis\'e $T$ sur $\R$~:
\begin{multline*}
\lim_{m\rightarrow +\infty}\!\lim_{\substack{\longrightarrow \\ (s,f_s,T_s)\in I(T)}}\!\!\!\Hom_{\psi,\Gamma}(W,T_s\otimes_EE_m)\buildrel\sim\over\longrightarrow \lim_{m\rightarrow +\infty}\!\lim_{\substack{\longrightarrow \\ (s,f_s,T_s)\in I(T)}}\!\!\!\Hom_{\psi,\Gamma}(V,T_s\otimes_EE_m).
\end{multline*}
\end{prop}
\begin{proof}
On ne traite que le cas o\`u $D_\alpha(C^{\an}_c(N^\beta(\Qp),\pi_{P_\beta}))_+\ne 0$, le cas $D_\alpha(C^{\an}_c(N^\beta(\Qp),\pi_{P_\beta}))_+=0$ \'etant analogue en plus facile (noter qu'alors l'\'enonc\'e implique que la suite (\ref{rstrictbis}) est scind\'ee quitte \`a augmenter $r$ et $m$). Par ailleurs, les preuves \'etant de m\^eme nature que celles des \S\S~\ref{cellule}, \ref{celluledevisse}, \ref{technique1}~\&~\ref{premierscindage}, on se permet de donner moins de d\'etails. Si $N'$ est un sous-groupe de $N(\Qp)$, on note $N'_m\=N'\cap N_m$.

\noindent
{\bf \'Etape $1$}\\
Il r\'esulte d'abord facilement des preuves du Lemme \ref{niveaum} et de la Proposition \ref{celluleouverte} que l'on a $\displaystyle \lim_{m\rightarrow +\infty}\lim_{\substack{\longrightarrow \\ (s,f_s,T_s)\in I(T)}}\!\!\Hom_{\psi,\Gamma}(\ker(f),T_s\otimes_EE_m)=0$, d'o\`u le dernier isomorphisme si l'on a $W\simeq V/\ker(f)$. Il suffit donc de montrer l'existence du diagramme commutatif. Par le d\'ebut de la preuve du Lemme \ref{structure}, on a une suite exacte $0\rightarrow W_E^\infty \rightarrow \pi_{P_\beta}^\infty\rightarrow (\pi_{P_\beta}^\infty)_{N_\beta(\Qp)}\rightarrow 0$ qui commute aux actions de $N_\beta(\Qp)$ et $\lambda_{\alpha^\vee}(\Qp^\times)$ (ce dernier agissant partout par le caract\`ere central de $\pi_{P_\beta}^\infty$). On en d\'eduit une suite exacte qui commute aux actions de $N(\Qp)$ et $\lambda_{\alpha^\vee}(\Qp^\times)$ (d\'efinies comme avant (\ref{surjcellule}))~:
\begin{multline*}
0\longrightarrow C^{\an}_c(N^\beta(\Qp),E)\otimes_EL(-\mu)_{P_\beta}\otimes_EW_E^\infty \longrightarrow C^{\an}_c(N^\beta(\Qp),E)\otimes_EL(-\mu)_{P_\beta}\otimes_E\pi_{P_\beta}^\infty\\
\longrightarrow C^{\an}_c(N^\beta(\Qp),E)\otimes_EL(-\mu)_{P_\beta}\otimes_E\otimes_E(\pi_{P_\beta}^\infty)_{N_\beta(\Qp)}\longrightarrow 0.
\end{multline*}
Comme l'action de $\mnn^\alpha$ ne touche pas les facteurs lisses, on a encore une suite e\-xacte en rempla\c cant partout $C^{\an}_c(N^\beta(\Qp),E)\otimes_EL(-\mu)_{P_\beta}$ par $(C^{\an}_c(N^\beta(\Qp),E)\otimes_EL(-\mu)_{P_\beta})[\mnn^\alpha]$, d'o\`u on d\'eduit une suite exacte de $(\psi,\Gamma)$-modules de Fr\'echet sur $\Rm^+$~:
\begin{multline}\label{mwinfty}
0\longrightarrow M_\alpha\big(C^{\an}_c(N^\beta(\Qp),E)\otimes_EL(-\mu)_{P_\beta}\otimes_E(\pi_{P_\beta}^\infty)_{N_\beta(\Qp)}\otimes_EE_m\big)\\
\longrightarrow M_\alpha(C^{\an}_c(N^\beta(\Qp),E)\otimes_EL(-\mu)_{P_\beta}\otimes_E\pi_{P_\beta}^\infty\otimes_EE_m\big)\\
\longrightarrow M_\alpha\big(C^{\an}_c(N^\beta(\Qp),E)\otimes_EL(-\mu)_{P_\beta}\otimes_EW_E^\infty\otimes_EE_m\big) \longrightarrow 0.
\end{multline}
\'Ecrivant $C^{\an}_c\big(N^\beta(\Qp),E)\simeq C^{\an}(N^{\beta}_m,E)\bigoplus (\bigoplus_{m'\geq m+1}C^{\an}(N^{\beta}_{m'}\backslash N^{\beta}_{m'-1},E))$, on en d\'eduit avec les notations usuelles (cf. (\ref{pourciter})) un isomorphisme de modules de Fr\'echet sur $\Rm^+$ pour $(-)\in \{(\pi_{P_\beta}^\infty)_{N_\beta(\Qp)},\pi_{P_\beta}^\infty,W_E^\infty\}$~:
\begin{equation}\label{decomp(-)}
M_\alpha\big(C^{\an}_c(N^\beta(\Qp),E)\otimes_EL(-\mu)_{P_\beta}\otimes_E (-)\otimes_EE_m\big)\simeq \prod_{m'\geq m}M_{\alpha,m'}
\end{equation}
o\`u $\Gamma$ pr\'eserve chaque facteur direct \`a droite, o\`u $\psi$ pr\'eserve $M_{\alpha,m}$ et envoie $M_{\alpha,m'}$ dans $M_{\alpha,m'+1}$ si $m'\geq m+1$. Noter que (si $m'\geq m+1$)~:
\begin{multline}\label{malphaexpl}
M_\alpha\big(C^{\an}_c(N^{\beta}_{m'}\backslash N^{\beta}_{m'-1},E)\otimes_EL(-\mu)_{P_\beta}\otimes_E (-)\otimes_EE_m\big)\\
\simeq \big(\big(C^{\an}_c(N^{\beta}_{m'}\backslash N^{\beta}_{m'-1},E)[\mnn_\gamma]\otimes_EE(-\mu)\otimes (-)\otimes_EE_m\big)(\eta^{-1})_{N^\alpha_m}\big)^\vee
\end{multline}
et idem avec $N^{\beta}_{m}$ ($\mnn_\gamma$ est la $\Qp$-alg\`ebre de Lie de $N_\gamma(\Qp)$). Soit $m'\geq m+1$, comme $(C^{\an}_c(N^{\beta}_{m'}\backslash N^{\beta}_{m'-1},E)[\mnn_\gamma]\otimes_EE(-\mu)\otimes_E (-)\otimes_EE_m)^\vee$ est un $\Rm^+$-module pro-libre, on d\'eduit que le sous-$\Rm^+$-module $M_{\alpha,m'}$ est sans torsion. Supposons maintenant $(-)=(\pi_{P_\beta}^\infty)_{N_\beta(\Qp)}$, dont la dimension est finie. Comme $(C^{\an}_c(N^{\beta}_{m'}\backslash N^{\beta}_{m'-1},E)[\mnn_\gamma]^{N_{\gamma,m}}\otimes_EE(-\mu)\otimes_E (-)\otimes_EE_m)^\vee$ est un $\Rm^+$-module de type fini, on d\'eduit avec la Remarque \ref{fixe} que le $\Rm^+$-module $M_{\alpha,m'}$ est de type fini. Donc les $M_{\alpha,m'}$ sont libres de type fini. On peut alors appliquer la Proposition \ref{scinde3} au ``pull-back'' de (\ref{rstrictter}) le long de $\prod_{m'\geq m+1}M_{\alpha,m'}$, qui est donc scind\'e quitte \`a augmenter $r$. Par ailleurs, un calcul montre que l'action de $N_{\beta,m}$ sur $(C^{\an}_c(N^{\beta}_{m},E)[\mnn_\gamma]^{N_{\gamma,m}}\otimes_EE(-\mu)\otimes_E (-))^\vee$ est triviale, ce qui implique $M_{\alpha,m}=0$ (quitte \`a augmenter $m$ pour avoir $\eta\vert_{N_{\beta,m}}\ne 1$). Quotientant $V$ par $\prod_{m'\geq m+1}M_{\alpha,m'}=\prod_{m'\geq m}M_{\alpha,m'}$, on est alors ramen\'e \`a d\'emontrer l'\'enonc\'e pour une suite exacte~:
\begin{equation}\label{rstrictquarto}
0\longrightarrow D_r\longrightarrow V\longrightarrow M_\alpha\big(C^{\an}_c(N^\beta(\Qp),E)\otimes_EL(-\mu)_{P_\beta}\otimes_EW_E^\infty\otimes_EE_m\big)\longrightarrow 0.
\end{equation}

\noindent
{\bf \'Etape $2$}\\
Par (\ref{steinberg}) on a une suite exacte $0\rightarrow W_E^\infty \rightarrow C_c^\infty(N_\beta(\Qp),E)\rightarrow E \rightarrow 0$ qui commute aux actions de $N_\beta(\Qp)$ et de $\lambda_{\alpha^\vee}(\Qp^\times)$ en faisant partout agir ce dernier par le caract\`ere central de $\pi_{P_\beta}^\infty$, et qui donne une suite exacte de $(\psi,\Gamma)$-modules de Fr\'echet sur $\Rm^+$ analogue \`a (\ref{mwinfty}). Par l'analogue de (\ref{decomp(-)}), l'argument du Lemme \ref{begin} bornant le support et l'argument \`a la fin de l'\'Etape $1$ montrant $M_{\alpha,m}=0$ quand $(-)=(\pi_{P_\beta}^\infty)_{N_\beta(\Qp)}$, on d\'eduit en particulier que tout morphisme de $(\psi,\Gamma)$-modules de Fr\'echet $M_\alpha(C^{\an}_c(N^\beta(\Qp),E)\otimes_EL(-\mu)_{P_\beta}\otimes_EE_m)\rightarrow D_r$ devient nul en augmentant $m$. Quitte \`a prendre le ``pull-back'' de (\ref{rstrictquarto}) et \`a augmenter $m$, on en d\'eduit qu'il suffit de d\'emontrer l'\'enonc\'e pour une suite exacte~:
\begin{equation}\label{rstrictcinquo}
0\longrightarrow D_r\longrightarrow V\longrightarrow M_\alpha\big(C^{\an}_c(N(\Qp),E)\otimes_EE(-\mu)\otimes_EE_m\big)\longrightarrow 0
\end{equation}
en remarquant que (cf. p. ex. (\ref{malphaexpl}))~:
\begin{multline}\label{malphaexpl2}
M_\alpha\big(C^{\an}_c(N(\Qp),E)\otimes_EE(-\mu)\otimes_EE_m\big)\\
\simeq M_\alpha\big(C^{\an}_c(N^\beta(\Qp),E)\otimes_EL(-\mu)_{P_\beta}\otimes_EC_c^\infty(N_\beta(\Qp),E)\otimes_EE_m\big).
\end{multline}
\'Ecrivant $C^{\an}_c(N(\Qp),E)\simeq C^{\an}(N_m,E)\bigoplus (\bigoplus_{m'\geq m+1}C^{\an}(N_{m'}\backslash N_{m'-1},E))$, on a comme en (\ref{decomp(-)})~:
\begin{equation}\label{decomp(-)bis}
M_\alpha\big(C^{\an}_c(N(\Qp),E)\otimes_EE(-\mu)\otimes_EE_m\big)\simeq \prod_{m'\geq m}M_{\alpha,m'}
\end{equation}
o\`u $\Gamma$ respecte chaque $M_{\alpha,m'}$, $\psi$ envoie $M_{\alpha,m'}$ dans $M_{\alpha,m'}\oplus M_{\alpha,m'+1}$ et o\`u, comme dans l'\'Etape $1$, on v\'erifie que chaque $M_{\alpha,m'}$ est libre de rang fini sur $\Rm^+$. L'argument de la deuxi\`eme partie de la preuve de la Proposition \ref{scinde3}, i.e. le scindage pour les mo\-dules de Fr\'echet sur $\R^+$ sous-jacents (cf. le (i) de la Remarque \ref{restevalable}) avec l'argument de la premi\`ere partie de la preuve du Lemme \ref{scindepsi} montrent qu'il existe $M\gg 0$ tel que le ``pull-back'' de (\ref{rstrictcinquo}) sur $\prod_{m'\geq M+1}M_{\alpha,m'}$ est scind\'e comme suite exacte de $(\psi,\Gamma)$-modules de Fr\'echet sur $\Rm^+$. Quotientant $V$ par $\prod_{m'\geq M+1}M_{\alpha,m'}$, on est donc ramen\'e \`a d\'emontrer l'\'enonc\'e pour une suite exacte~:
\begin{equation}\label{rstrictsex}
0\longrightarrow D_r\longrightarrow V\longrightarrow M_\alpha\big(C^{\an}_c(N_M,E)\otimes_EE(-\mu)\otimes_EE_m\big)\longrightarrow 0.
\end{equation}
Remarquant que, comme pour (\ref{malphaexpl2})~:
\begin{multline*}
M_\alpha\big(C^{\an}_c(N_M,E)\otimes_EE(-\mu)\otimes_EE_m\big)\\
\simeq M_\alpha\big(C^{\an}_c(N^\beta_M,E)\otimes_EL(-\mu)_{P_\beta}\otimes_EC^{\infty}_c(N_{\beta,M},E)\otimes_EE_m\big)
\end{multline*}
on a une suite exacte de $(\psi,\Gamma)$-modules de Fr\'echet sur $\Rm^+$~:
\begin{multline}\label{mwinftybis}
0\longrightarrow M_\alpha\big(C^{\an}_c(N^\beta_M\backslash N^\beta_m,E)\otimes_EL(-\mu)_{P_\beta}\otimes_EC^{\infty}_c(N_{\beta,M},E)\otimes_EE_m\big)\\
\longrightarrow M_\alpha\big(C^{\an}_c(N^\beta_M,E)\otimes_EL(-\mu)_{P_\beta}\otimes_EC^{\infty}_c(N_{\beta,M},E)\otimes_EE_m\big)\\
\longrightarrow M_\alpha\big(C^{\an}_c(N^\beta_m,E)\otimes_EL(-\mu)_{P_\beta}\otimes_EC^{\infty}_c(N_{\beta,M},E)\otimes_EE_m\big) \longrightarrow 0.
\end{multline}
Comme $\psi^{M-m}=0$ sur le premier terme en (\ref{mwinftybis}), l'argument dans la derni\`ere partie du Lemme \ref{scindepsi} montre que, quitte \`a augmenter $r$, le ``pull-back'' de (\ref{rstrictsex}) sur $M_\alpha(C^{\an}_c(N^\beta_M\backslash N^\beta_m,E)\otimes_EL(-\mu)_{P_\beta}\otimes_EC^{\infty}_c(N_{\beta,M},E)\otimes_EE_m)$ est scind\'e comme suite exacte de $(\psi,\Gamma)$-modules de Fr\'echet sur $\Rm^+$. Quotientant $V$ et explicitant le terme du bas en (\ref{mwinftybis}), on est donc ramen\'e \`a d\'emontrer l'\'enonc\'e pour une suite exacte~:
\begin{equation}\small\label{rstrictsept}
0\rightarrow D_r\rightarrow V\rightarrow C^{\an}_c(N_{\alpha,m},E)^\vee\otimes_E\big(E(-\mu)\otimes_E C^{\infty}_c(N_{\beta,M},E_m)(\eta^{-1})_{N_{\beta,m}}\big)^\vee \rightarrow 0.
\end{equation}
Le m\^eme argument encore permet de remplacer $C^{\an}_c(N_{\alpha,m},E)^\vee$ par $C^{\an}_c(N_{\alpha,0},E)^\vee$ (cf. la preuve de la Proposition \ref{celluleouverte}). Enfin, quitte \`a remplacer $m$ par $M$, on est finalement ramen\'e \`a $C^{\an}_c(N_{\alpha,0},E)^\vee\otimes_E\big(E(-\mu)\otimes_E C^{\infty}_c(N_{\beta,m},E_m)(\eta^{-1})_{N_{\beta,m}}\big)^\vee\simeq D_\alpha\big(C^{\an}_c(N^\beta(\Qp),\pi_{P_\beta})\big)_+\otimes_EE_m$, ce qui termine la preuve.
\end{proof}

On d\'emontre maintenant le Th\'eor\`eme \ref{gl3enplus}. On commence par le lemme suivant, valable pour tout groupe $G$ comme au \S~\ref{prel}.

\begin{lem}\label{borneS}
Soit $\pi$ une repr\'esentation admissible de $G(\Qp)$ sur $E$, $S\in \R^+$ et $m\in \Z_{\geq 0}$. Alors la topologie de Fr\'echet de $M_\alpha(\pi\otimes_EE_m)$ peut \^etre d\'efinie par des semi-normes $q$ v\'erifiant la condition~: pour tout $q$ il existe $C_{q,S}\in {\mathbb R}_{>0}$ tel que $q(Sv)\leq C_{q,S}q(v)$ pour tout $v\in M_\alpha(\pi\otimes_EE_m)$.
\end{lem}
\begin{proof}
Puisque $\pi^\vee\otimes_EE_m$ est admissible, par \cite[Prop.~6.5]{Sch} on peut \'ecrire $\pi^\vee\otimes_EE_m$ comme limite inductive de type compact $\pi^\vee\otimes_EE_m=\ilim{l}B_l$ o\`u $l$ parcourt un ensemble d\'enombrable d'indices et les $B_l$ sont des repr\'esentations localement analytiques de $N_0$ sur des $E_m$-espaces de Banach. Choisissons une norme $r_l$ sur chaque Banach dual $B_l^\vee$, de sorte que les semi-normes induites (encore not\'ees) $r_l:\pi^\vee\otimes_EE_m\simeq \plim{l}B_l^\vee\rightarrow B_l^\vee \buildrel r_l\over\rightarrow {\mathbb R}_{\geq 0}$ sur $\pi^\vee\otimes_EE_m$ d\'efinissent sa topologie de Fr\'echet. Comme chaque $B_l^\vee$ est un $D(N_0,E_m)$-module s\'epar\'ement continu par \cite[Prop.~3.2]{ST1}, donc {\it a fortiori} un $D(N_\alpha(\Qp)\cap N_0,E_m)$-module s\'epar\'ement continu, la multiplication par $S\in \R^+\otimes_EE_m\simeq D(N_\alpha(\Qp)\cap N_0,E_m)$ est continue sur $B_l$, i.e. il existe $D_{l,S}\in {\mathbb R}_{>0}$ tel que $r_l(Sv)\leq D_{l,S}r_l(v)$ pour tout $v\in \pi^\vee\otimes_EE_m$. Quitte \`a remplacer les semi-normes $r_l$ (resp. les constantes $D_{l,S}$) par les semi-normes $q$ (resp. les constantes $C_{q,S}$) d\'efinies comme $\max(r_l,l\in F)$ (resp. $\max(D_{l,S},l\in F)$) pour $F$ parcourant les sous-ensembles finis d'indices $l$ (cf. \cite[Rem.~4.7]{Sch}), on en d\'eduit le r\'esultat en utilisant (\ref{malpha}) et \cite[\S~5.B]{Sch}.
\end{proof}

Fixons maintenant une suite exacte $0\rightarrow \pi''\rightarrow \pi\rightarrow \pi'\rightarrow 0$ dans $\Rep(G(\Qp))$ ($G(\Qp)={\rm GL}_3(\Qp)$) satisfaisant les hypoth\`eses de la Conjecture \ref{representable} avec $\pi''$ {\it non} localement alg\'ebrique (cf. \S~\ref{preuvelocalg} pour le cas localement alg\'ebrique). Par la d\'efinition de $C_{\lambda,\alpha}$ (cf. d\'ebut du \S~\ref{prelgen}), on a une injection $\pi''\hookrightarrow I(\pi'')\=(\Ind_{P^-(\Qp)}^{G(\Qp)} L(-w\cdot \lambda)_{P}\otimes_E\pi_P^\infty)^{\an}$ avec $w\notin \{1,w_0\}$ et un conoyau soit nul, soit de la forme ${\mathcal F}_{P^-}^G(L^-(w'\cdot \lambda),\pi_P^\infty)$ pour $w'\notin \{1,s_\alpha\}$. Si $\pi_P^\infty$ est non g\'en\'erique, on a $F_\alpha(\pi'')=F_\alpha(I(\pi''))=0$ par le Th\'eor\`eme \ref{plusgeneral}, le Th\'eor\`eme \ref{encoreplusgeneral} et le (i) de la Proposition \ref{drex}, et si $\pi_P^\infty$ est g\'en\'erique, on a $F_\alpha(\pi'')\buildrel\sim\over\rightarrow F_\alpha(I(\pi''))$ par le Th\'eor\`eme \ref{liesocle}.

On suppose d'abord $P=P_\alpha$. Il r\'esulte du (i) du Lemme \ref{scinde1}, ou directement du Th\'eor\`eme \ref{encoreplusgeneral}, que $t^{1-\langle \lambda,\alpha^\vee\rangle}M_\alpha(\pi''\otimes_EE_m)=0$ pour $m\in \Z_{\geq 0}$ et que $F_\alpha(\pi'')=0$. Partant de la suite exacte courte dans la Proposition \ref{pushtechnique}, l'argument de la preuve du cas $\pi''=L(-\lambda)\otimes \pi^\infty$ au \S~\ref{preuvelocalg} lorsque $D_\alpha(\pi'')_+=0$ (i.e. pour $\pi^\infty$ {\it non} g\'en\'erique) s'\'etend alors essentiellement tel quel et montre que $F_\alpha(\pi)(-)\simeq E_\infty(\chi_{-\lambda})\otimes_{E}\Hom_{(\varphi,\Gamma)}(D_\alpha(\pi'),(-))\simeq F_\alpha(\pi')$.

Le lemme ci-dessous, valable pour $P\in \{P_\alpha,P_\beta\}$, ne sera utilis\'e que pour $P=P_\beta$.

\begin{lem}\label{remplace}
Supposons que $\pi''$ ne soit pas localement alg\'ebrique, il suffit de d\'emontrer le Th\'eor\`eme \ref{gl3enplus} avec $I(\pi'')$ au lieu de $\pi''$.
\end{lem}
\begin{proof}
On suppose que le conoyau $J(\pi'')$ de $\pi''\hookrightarrow I(\pi'')$ est non nul, i.e. $J(\pi'')\simeq {\mathcal F}_{P^-}^G(L^-(w'\cdot \lambda),\pi_P^\infty)$, sinon il n'y a rien \`a montrer. Soit $\widetilde \pi\=\pi\oplus_{\pi''}I(\pi'')$, par hypoth\`ese on a $F_\alpha(\widetilde\pi)(-)\simeq E_\infty(\chi_{-\lambda})\otimes_{E}\Hom_{(\varphi,\Gamma)}(D_\alpha(\widetilde\pi),-)$ avec une suite exacte $0\rightarrow D_{\alpha}(\pi')\rightarrow D_{\alpha}(\widetilde\pi)\rightarrow D_{\alpha}(\pi'')\rightarrow 0$ (o\`u l'on a utilis\'e $F_\alpha(\pi'')\buildrel\sim\over\rightarrow F_\alpha(I(\pi''))$). La suite exacte $0\rightarrow \pi \rightarrow \widetilde \pi\rightarrow J(\pi'')\rightarrow 0$ et le (i) de la Proposition \ref{drex} donnent une suite exacte $0\rightarrow F_\alpha(\pi)\rightarrow F_\alpha(\widetilde \pi)\rightarrow F_\alpha(J(\pi''))$. Mais par le Corollaire \ref{casparticuliers}, le Lemme \ref{egalchi} et le fait que $w'\notin \{1,s_\alpha\}$, le morphisme $F_\alpha(\widetilde \pi)\rightarrow F_\alpha(J(\pi''))$ est nul dans $F(\varphi,\Gamma)_\infty$, l'action de $\Gal(E_\infty/E)$ n'\'etant pas la m\^eme des deux c\^ot\'es. On a donc $F_\alpha(\pi)\buildrel\sim\over\rightarrow F_\alpha(\widetilde \pi)$ ce qui termine la preuve.
\end{proof}

On suppose maintenant $P=P_\beta$ et, changeant de notations, $\pi''=(\Ind_{P_\beta^-(\Qp)}^{G(\Qp)} L(-w\cdot \lambda)_{P_\beta}\otimes_E\pi_{P_\beta}^\infty)^{\an}$ (via le Lemme \ref{remplace}). La Proposition \ref{pushtechnique} donne alors (quitte \`a augmenter $m$ et $r$) une suite exacte de $(\psi,\Gamma)$-modules de Fr\'echet sur $\Rm^+$~:
\begin{equation}\label{avecmtildebis}
0\longrightarrow D_{\alpha}(\pi')_r\otimes_EE_m\longrightarrow \widetilde M_\alpha(\pi\otimes_EE_m)\longrightarrow M_\alpha(\pi''\otimes_EE_m)\longrightarrow 0
\end{equation}
avec un isomorphisme comme en (\ref{limsomamalg}) (par le m\^eme argument). Par le (ii) du Lemme \ref{scinde1} avec (\ref{malpha0}) et la phrase d'apr\`es, on a une suite exacte courte de $(\psi,\Gamma)$-modules de Fr\'echet sur $\Rm^+$~:
\begin{equation*}
0\longrightarrow M_\alpha(\pi^0_C\otimes_EE_m)\longrightarrow M_\alpha(\pi''\otimes_EE_m) \longrightarrow M_\alpha\big(C^{\an}_c(N^\beta(\Qp),\pi_{P_\beta})\otimes_EE_m\big) \longrightarrow 0.
\end{equation*}
De plus on d\'eduit facilement du Lemme \ref{borneS} (et de l'assertion analogue pour $D_r$ au lieu de $M_\alpha(\pi\otimes_EE_m)$) que la topologie de Fr\'echet sur $\widetilde M_\alpha(\pi\otimes_EE_m)$ peut \^etre d\'efinie par des semi-normes $q$ v\'erifiant $q(Sv)\leq C_{q,S}q(v)$. On peut donc appliquer la Proposition \ref{scinde1dur} au ``pull-back'' de (\ref{avecmtildebis}) le long de $M_\alpha(\pi^0_C\otimes_EE_m)\rightarrow M_\alpha(\pi''\otimes_EE_m)$. Avec la Remarque \ref{addendum}, on en d\'eduit (quitte \`a augmenter $r$) une section $M_\alpha(\pi^0_C\otimes_EE_m)\hookrightarrow \widetilde M_\alpha(\pi\otimes_EE_m)$ et un isomorphisme~:
\begin{multline*}
\lim_{\substack{\longrightarrow \\ (s,f_s,T_s)\in I(T)}}\!\!\Hom_{\psi,\Gamma}\big(\widetilde M_\alpha(\pi\otimes_EE_m)/M_\alpha(\pi^0_C\otimes_EE_m),T_s\otimes_EE_m\big)\\\buildrel\sim\over\longrightarrow \lim_{\substack{\longrightarrow \\ (s,f_s,T_s)\in I(T)}}\!\!\Hom_{\psi,\Gamma}\big(\widetilde M_\alpha(\pi\otimes_EE_m),T_s\otimes_EE_m\big).
\end{multline*}
Rempla\c cant $\widetilde M_\alpha(\pi\otimes_EE_m)$ par $\widetilde M_\alpha(\pi\otimes_EE_m)/M_\alpha(\pi^0_C\otimes_EE_m)$, on est ainsi ramen\'e \`a une suite exacte comme dans la Proposition \ref{scinde2ouvert}. Par {\it loc.cit.} et ce qui pr\'ec\`ede, on peut alors argumenter comme dans l'\'Etape $3$ du \S~\ref{preuvelocalg}. Cela ach\`eve la preuve du Th\'eor\`eme \ref{gl3enplus} pour $G={\rm GL}_3$. Noter que l'hypoth\`ese $\Hom_{(\varphi,\Gamma)}(D_\alpha(\pi''),D_\alpha(\pi'))=0$ n'est utilis\'ee que via le Lemme \ref{actiongalois}.

La preuve pour ${\rm GL}_2$ est beaucoup plus simple que pour ${\rm GL}_3$. On a d'abord $N^\alpha=N_m^\alpha=\{1\}$ et $\lambda_{\alpha^\vee}(x)=\diag(x,1)$, et par la Proposition \ref{choix} on peut choisir $N_m=\smat{1&\frac{1}{p^m}\Zp\\0&1}$ pour $m\in \Z_{\geq 0}$. Pour $\pi$ dans $\Rep(B(\Qp))$, on a $M_\alpha(\pi\otimes_E E_m)=(\pi\otimes_E E_m)^\vee\cong \pi^\vee\otimes_E E_m$ o\`u la structure de $\R^+$-module vient de l'action de $N_0=\smat{1&\Zp\\0&1}$ sur $\pi^\vee$, l'action de $\Gamma$ de celle de $\smat{\Zp^\times&0\\0&1}$, l'action de $\psi$ de celle de $\smat{\frac{1}{p}&0\\0&1}$ (rappelons que $B(\Qp)$ agit sur le dual $\pi^\vee$ comme en (\ref{actionsurdual})) et o\`u l'action de $\Gal(E_\infty/E)$ est triviale sur $\pi^\vee$. On voit donc que l'on a $\displaystyle F_\alpha(\pi)(T)=
E_\infty\otimes_E\!\!\lim_{\substack{\longrightarrow \\ (r,f_r,T_r)\in I(T)}}\Hom_{\psi,\Gamma}(\pi^\vee,T_r)$ (on pourrait en fait oublier $E_\infty$ dans ce cas). Par ailleurs les objets irr\'eductibles dans $C_{\lambda,\alpha}$ sont soit des repr\'esentations localement alg\'ebriques $L(-\lambda)\otimes_E\pi^\infty$, soit des s\'eries principales localement analytiques $(\Ind_{B^-(\Qp)}^{{\rm GL}_2(\Qp)} \chi_1\otimes \chi_2 )^{\an}$ pour des caract\`eres localement alg\'ebriques $\chi_i:\Qp^\times\rightarrow E^\times$ tels que la d\'eriv\'ee de $\chi_1\otimes \chi_2$ est $-s_\alpha\cdot\lambda$. Fixons une suite exacte $0\rightarrow \pi''\rightarrow \pi\rightarrow \pi'\rightarrow 0$ dans $\Rep(G(\Qp))$ ($G(\Qp)={\rm GL}_2(\Qp)$) comme dans la Conjecture \ref{representable}. Le cas $\pi''$ localement alg\'ebrique est trait\'e au \S~\ref{preuvelocalg}. Pour l'autre cas, on d\'ecompose la s\'erie principale $\pi''$ de la mani\`ere usuelle~:
$$0\longrightarrow (\pi''_C)^\vee\longrightarrow (\pi'')^\vee\longrightarrow \big((\cInd_{B^-(\Qp)}^{B^-(\Qp)B(\Qp)} \chi_1\otimes\chi_2)^{\an}\big)^\vee\longrightarrow 0.$$
La Proposition \ref{scinde1dur} prend dans ce cas la forme plus simple suivante.

\begin{prop}\label{scinde1gl2}
Soit $r\in \Q_{>p-1}$ et $D_r$ un $(\varphi,\Gamma)$-module g\'en\'eralis\'e sur $\Rr$. Alors toute suite exacte~:
$$0\longrightarrow D_r\longrightarrow V\longrightarrow (\pi''_C)^\vee\longrightarrow 0$$
de $(\psi,\Gamma)$-modules de Fr\'echet sur $\R^+$ o\`u la topologie de $V$ peut \^etre d\'efinie par des semi-normes $q$ telles qu'il existe $C_{q}\in {\mathbb R}_{>0}$ v\'erifiant $q(Xv)\leq C_q q(v)$ pour tout $v\in V$ est scind\'ee.
\end{prop}
\begin{proof}
En remarquant que $(\pi_C)^\vee\simeq (\chi_2\otimes \chi_1)^\vee \otimes_ED(N^-(\Qp),E)_{\{1\}}$ par le Lemme \ref{support1} et que $X$ annule $(\chi_2\otimes \chi_1)^\vee$, la preuve est la m\^eme, en beaucoup plus simple, que celle du Lemme \ref{casdavant} avec l'espace de dimension un $(\chi_2\otimes \chi_1)^\vee$ au lieu de $M^\alpha$ et $X$ au lieu de $t$ (cf. aussi la fin de l'\'Etape $2$ dans la preuve de la Proposition \ref{scinde1dur}). On laisse les d\'etails au lecteur int\'eress\'e (noter que l'on n'a pas besoin de l'hypoth\`ese $D_r$ sans torsion).
\end{proof}

La \ \ Proposition \ \ \ref{scinde2ouvert} \ \ reste \ \ \'egalement \ \ valable \ \ en \ \ rempla\c cant $M_\alpha(C^{\an}_c(N^\beta(\Qp),\pi_{P_\beta})\otimes_EE_m)$ par $((\cInd_{B^-(\Qp)}^{B^-(\Qp)B(\Qp)}\chi_1\otimes\chi_2)^{\an})^\vee\otimes_EE_m$. On termine la preuve de la m\^eme mani\`ere (noter que l'on n'utilise pas $\Hom_{(\varphi,\Gamma)}(D_\alpha(\pi''),D_\alpha(\pi'))=0$).

\section{Foncteurs $F_\alpha$, groupes Ext${}^1$ et compatibilit\'e local-global}\label{calpha}

On revisite la conjecture de compatibilit\'e local-global \cite[Conj.~6.1.1]{Br1} lorsque $F^+=\Q$, en particulier on la reformule de mani\`ere plus pr\'ecise en termes de $(\varphi,\Gamma)$-modules en utilisant les foncteurs $F_\alpha$. On donne plusieurs cas particuliers et r\'esultats partiels sur cette reformulation.

\subsection{Pr\'eliminaires de th\'eorie de Hodge $p$-adique}\label{hodge}

Les r\'esultats de ce paragraphe, pas exemple la Proposition \ref{lienfilmax} ci-dessous, sont enti\`erement ``c\^ot\'e Galois'' et ne sont (probablement) pas nouveaux, mais nous donnons une formulation adapt\'ee pour \^etre utilis\'ee dans l'\'enonc\'e de la Conjecture \ref{extglobmieux} ci-dessous ainsi que dans la preuve de cas particuliers (\S~\ref{carparticuliers}).

On appelle ``morphisme naturel d'espaces vectoriels'' tout morphisme qui est cano\-nique \`a multiplication pr\`es par un scalaire non nul. Dans ce paragraphe, on note avec une majuscule $\mathcal D$ les $(\varphi,\Gamma)$-modules sur $\R$ afin de les distinguer des modules de Deligne-Fontaine $D$.

On fixe un syst\`eme compatible de racines primitives $p^n$-i\`emes de l'unit\'e $(\zeta_{p^n})_{n\geq 1}$ dans $\Qpbar\hookrightarrow B^+_{\rm dR}$, ce qui permet de voir $t=\log(1+X)$ comme un \'el\'ement de $B^+_{\rm dR}$. Rappelons que, pour $r\in \Q_{>p-1}$ et $n\geq n(r)$ (o\`u $n(r)$ est le plus petit entier $n$ tel que $p^{n(r)-1}(p-1) \geq r$), on a un morphisme de $\Qp$-alg\`ebres~:
\begin{equation*}
\iota_n: {\mathcal R}_{\Q_p}^r \longrightarrow \Q_p(\zeta_{p^n})[[t]]
\end{equation*}
qui envoie $\sum a_i X^i$ vers $\sum a_i (\zeta_{p^n}\exp(t/p^n)-1)^i$. Il est clair que $\iota_n$ est $\Gamma$-\'equivariant. On note encore $\iota_n$ la compos\'ee $\iota_n: {\mathcal R}_{\Q_p}^r \rightarrow \Q_p(\zeta_{p^n})[[t]] \hookrightarrow B_{\rm dR}^+$, qui est alors $\gp$-\'equivariante (o\`u l'action de $\gp$ sur $ {\mathcal R}_{\Q_p}^r$ se factorise par $\Gamma$). Si ${\mathcal D}$ est un $(\varphi,\Gamma)$-module sans torsion sur $\R$ et ${\mathcal D}_r\subseteq {\mathcal D}$ un $(\varphi,\Gamma)$-module sur $\Rr$ tel que $\R\otimes_{\Rr}{\mathcal D}_r\buildrel\sim\over\rightarrow {\mathcal D}$, pour $n\geq n(r)$ on note $D_{\mathrm{diff},n}^+({\mathcal D})\=\Q_p(\zeta_{p^n})[[t]]\otimes_{ {\mathcal R}_{\Q_p}^r}^{\iota_n} {\mathcal D}_{r}$ (un $\Q_p(\zeta_{p^n})[[t]]\otimes_{\Qp}E$-module libre de rang fini muni de l'action diagonale de $\Gamma$) et $W_{\rm dR}^+({\mathcal D})\= B_{\rm dR}^+ \otimes_{ {\mathcal R}_{\Q_p}^r}^{\iota_n} {\mathcal D}_{r}\simeq B_{\rm dR}^+ \otimes_{\Q_p(\zeta_{p^n})[[t]]} D_{\mathrm{diff},n}^+({\mathcal D})$ (muni de l'action diagonale de $\gp$) qui ne d\'epend pas de $n$ suffisamment grand (\cite[Prop. 2.2.6(2)]{Be3}). On rappelle que $H^i_{(\varphi,\Gamma)}({\mathcal D})\=\Ext^i_{(\varphi,\Gamma)}(\R,{\mathcal D})$.

\begin{lem}\label{h0t}
Soit ${\mathcal D}_1 \subseteq {\mathcal D}_2$ deux $(\varphi,\Gamma)$-modules sans torsion sur ${\mathcal R}_E$ de m\^eme rang, alors on a un isomorphisme canonique~:
 \begin{equation}\label{tor0}
 H^0_{(\varphi,\Gamma)}({\mathcal D}_2/{\mathcal D}_1) \buildrel\sim\over\longrightarrow H^0\big(\gp, W_{\rm dR}^+({\mathcal D}_2)/W_{\rm dR}^+({\mathcal D}_1)\big).
 \end{equation}
\end{lem}
\begin{proof}
Les deux $E$-espaces vectoriels en (\ref{tor0}) sont de dimension finie~: par \cite[Th.~5.3(1)]{Li} pour celui de gauche et par \cite[Cor.~5.7]{Na2} (avec la suite exacte longue de cohomologie associ\'ee \`a $0\rightarrow W_{\rm dR}^+({\mathcal D}_1)\rightarrow W_{\rm dR}^+({\mathcal D}_2)\rightarrow W_{\rm dR}^+({\mathcal D}_2)/W_{\rm dR}^+({\mathcal D}_1)\rightarrow 0$) pour celui de droite. Par d\'efinition de $W_{\rm dR}^+(-)$ on a un isomorphisme $\gp$-\'equivariant pour $n\gg 0$~:
\begin{equation}\label{tor3}
W_{\rm dR}^+({\mathcal D}_2)/W_{\rm dR}^+({\mathcal D}_1) \simeq B_{\rm dR}^+ \otimes_{\Q_p(\zeta_{p^n})[[t]]} \big(D_{\mathrm{diff},n}^+({\mathcal D}_2)/D_{\mathrm{diff},n}^+({\mathcal D}_1)\big).
\end{equation}

\noindent
{\bf \'Etape 1}\\
On montre que (\ref{tor3}) induit un isomorphisme quitte \`a augmenter encore $n$~:
\begin{equation}\label{tor1}
H^0\big(\gp, W_{\rm dR}^+({\mathcal D}_2)/W_{\rm dR}^+({\mathcal D}_1)\big) \buildrel\sim\over\longleftarrow \big(D_{\mathrm{diff},n}^+({\mathcal D}_2)/D_{\mathrm{diff},n}^+({\mathcal D}_1)\big)^{\Gamma}.
\end{equation}
Notons ${\mathbb Q}_{p,\infty}\=\cup_n\Qp(\zeta_n)$, $H\={\rm Gal}(\overline\Qp/{\mathbb Q}_{p,\infty})$ et $L_{\rm dR}^+\=(B^+_{\rm dR})^H$ qui est un anneau de valuation discr\`ete complet d'uniformisante $t$ et de corps r\'esiduel le compl\'et\'e $p$-adique de ${\mathbb Q}_{p,\infty}$ (\cite[Prop.~3.3]{Fo}). En particulier $L_{\rm dR}^+$ contient ${\mathbb Q}_{p,\infty}[[t]]$. Soit $W\=W_{\rm dR}^+({\mathcal D}_2)/W_{\rm dR}^+({\mathcal D}_1)$, qui est un $B^+_{\rm dR}$-module de longueur finie. Il suit de \cite[Th.~3.5]{Fo} que l'on a un isomorphisme $B^+_{\rm dR}\otimes_{L^+_{\rm dR}}W^H\buildrel\sim\over\rightarrow W$ (et donc que $W^H$ est un $L_{\rm dR}^+$-module de longueur finie). De (\ref{tor3}) on d\'eduit facilement une injection $\Gamma$-\'equivariante de $L_{\rm dR}^+$-modules de longueur finie $L_{\rm dR}^+ \otimes_{\Q_p(\zeta_{p^n})[[t]]} (D_{\mathrm{diff},n}^+({\mathcal D}_2)/D_{\mathrm{diff},n}^+({\mathcal D}_1))\hookrightarrow W^H$ (pour $n\gg 0$) qui, par extension des scalaires de $L_{\rm dR}^+$ \`a $B_{\rm dR}^+$, devient donc un isomorphisme de $B_{\rm dR}^+$-modules par (\ref{tor3}) et ce qui pr\'ec\`ede. On en d\'eduit facilement que l'on a en fait un isomorphisme $\Gamma$-\'equivariant~:
\begin{equation}\label{tor4}
W^H\buildrel\sim\over\longleftarrow L_{\rm dR}^+ \otimes_{\Q_p(\zeta_{p^n})[[t]]} (D_{\mathrm{diff},n}^+({\mathcal D}_2)/D_{\mathrm{diff},n}^+({\mathcal D}_1)).
\end{equation}
De m\^eme, par (\ref{tor4}) et \cite[Th.~3.6]{Fo} appliqu\'e \`a $X=W^H$, on d\'eduit une injection $\Gamma$-\'equivariante de ${\mathbb Q}_{p,\infty}[[t]]$-modules de longueur finie ${\mathbb Q}_{p,\infty}[[t]]\otimes_{\Q_p(\zeta_{p^n})[[t]]} (D_{\mathrm{diff},n}^+({\mathcal D}_2)/D_{\mathrm{diff},n}^+({\mathcal D}_1))\hookrightarrow (W^H)_f$ (cf. \cite[Th.~3.6]{Fo} pour $(W^H)_f\subseteq W^H$) qui devient un isomorphisme apr\`es extension des scalaires de ${\mathbb Q}_{p,\infty}[[t]]$ \`a $L_{\rm dR}^+$, donc qui est en fait un isomorphisme $\Gamma$-\'equivariant (pour $n\gg 0$)~:
\begin{equation}\label{tor5}
(W^H)_f\buildrel\sim\over\longleftarrow {\mathbb Q}_{p,\infty}[[t]] \otimes_{\Q_p(\zeta_{p^n})[[t]]} (D_{\mathrm{diff},n}^+({\mathcal D}_2)/D_{\mathrm{diff},n}^+({\mathcal D}_1)).
\end{equation}
Comme $(W^H)^\Gamma=H^0(\gp, W_{\rm dR}^+({\mathcal D}_2)/W_{\rm dR}^+({\mathcal D}_1))$ est de dimension finie, on a $((W^H)_f)^\Gamma=(W^H)^\Gamma$ et il existe $m\gg n$ tel que tous les \'el\'ements de $((W^H)_f)^\Gamma$ proviennent via (\ref{tor5}) d'\'el\'ements de ${\mathbb Q}_{p}(\zeta_{p^m})[[t]] \otimes_{\Q_p(\zeta_{p^n})[[t]]} (D_{\mathrm{diff},n}^+({\mathcal D}_2)/D_{\mathrm{diff},n}^+({\mathcal D}_1))$. Par l'isomorphisme~:
\begin{eqnarray*}
{\mathbb Q}_{p}(\zeta_{p^m})[[t]] \otimes_{\Q_p(\zeta_{p^n})[[t]]} (D_{\mathrm{diff},n}^+({\mathcal D}_2)/D_{\mathrm{diff},n}^+({\mathcal D}_1))&\buildrel\sim\over\longrightarrow &D_{\mathrm{diff},m}^+({\mathcal D}_2)/D_{\mathrm{diff},m}^+({\mathcal D}_1)\\
a\otimes (b\otimes d)&\longmapsto &ab\varphi^{m-n}(d)
\end{eqnarray*}
($a\in {\mathbb Q}_{p}(\zeta_{p^m})[[t]]$, $b\in {\mathbb Q}_{p}(\zeta_{p^n})[[t]]$, $d\in {\mathcal D}_{2,r}/{\mathcal D}_{1,r}$), quitte \`a remplacer $n$ par $m$, on en d\'eduit (\ref{tor1}).

\noindent
{\bf \'Etape 2}\\
On termine la preuve. Soit $r\in \Q_{>p-1}$ tel que l'injection ${\mathcal D}_1\hookrightarrow {\mathcal D}_2$ est induite par une injection ${\mathcal D}_{1,r} \hookrightarrow {\mathcal D}_{2,r}$ de $(\varphi,\Gamma)$-modules sur $ {\mathcal R}_E^r$ ((iii) du Lemme 2.2.3). Par \cite[Prop.~4.4(1)]{Li} et la preuve de \cite[Th.~4.7]{Li} (en parti\-culier la partie montrant que $H^0(S)$ est de dimension finie), le morphisme canonique surjectif~:
\begin{equation*}
{\mathcal D}_{2,r}/{\mathcal D}_{1,r} \longrightarrow D_{\mathrm{diff},n}^+({\mathcal D}_2)/D_{\mathrm{diff},n}^+({\mathcal D}_1)\cong \Q_p(\zeta_{p^n})[[t]]\otimes_{ {\mathcal R}_{\Q_p}^r}^{\iota_n} ({\mathcal D}_{2,r}/{\mathcal D}_{1,r})
\end{equation*}
induit un isomorphisme canonique pour $n\gg 0$~:
\begin{equation*}
H^0_{(\varphi,\Gamma)}({\mathcal D}_2/{\mathcal D}_1) \buildrel\sim\over\longrightarrow \big(D_{\mathrm{diff},n}^+({\mathcal D}_2)/D_{\mathrm{diff},n}^+({\mathcal D}_1)\big)^{\Gamma}.
\end{equation*}
En le composant avec (\ref{tor1}), on obtient un isomorphisme comme en (\ref{tor0}), dont on v\'erifie facilement qu'il ne d\'epend pas du choix de $n\gg 0$.
\end{proof}

Si $L$ est une extension finie de $\Qp$, on note $L_0$ sa sous-extension non ramifi\'ee maximale, et l'on renvoie \`a \cite[\S~2.2]{Br1} pour la d\'efinition des modules de Deligne-Fontaine $(D,\varphi,N,\glp)$ sur $L_0\otimes_{\Qp}E$ qui deviennent ``non ramifi\'es sur $L$'' (l'extension $L$ est not\'ee $L'$ dans {\it loc.cit.}). Dans la suite, on note un module de Deligne-Fontaine juste par son $L_0\otimes_{\Qp}E$-module libre sous-jacent $D$ (les op\'erateurs $\varphi$, $N$ sur $D$ et l'action de $\glp$ sur $D_L\=L\otimes_{L_0}D$ \'etant sous-entendus). Si $D$ est un module de Deligne-Fontaine, on note $\mathcal D$ le $(\varphi,\Gamma)$-module sur $\R$ associ\'e \`a $D$ muni de la filtration d\'ecroissante triviale, i.e. ${\rm Fil}^0(D_L)=D_L$ et ${\rm Fil}^1(D_L)=0$, par l'\'equivalence de cat\'egories \cite[Th.~A]{Be2} de Berger ($\mathcal D$ est aussi appel\'e l'\'equation diff\'erentielle $p$-adique associ\'ee \`a $D$). Noter que si l'on remplace cette filtration par ${\rm Fil}^{-h}(D_L)=D_L$ et ${\rm Fil}^{-h+1}(D_L)=0$ pour $h\in \Z$, alors on remplace $\mathcal D$ par ${\mathcal D}(x^h)={\mathcal D}\otimes_{\R}\R(x^h)$. Rappelons par ailleurs que le Th\'eor\`eme de Hilbert $90$ implique $L\otimes_{\Qp}D_L^{\glp}\buildrel\sim\over\rightarrow D_L$ et que se donner une filtration $\glp$-\'equivariante sur $D_L$ est \'equivalent via $L\otimes_{\Qp}(-)$ \`a se donner une filtration sur le $E$-espace vectoriel $D_L^{\glp}$.

\begin{prop}\label{lienfilmax}
(i) Soit $h_1>h_2$ dans $\Z$ et $F\subseteq D_L^{\glp}$ un sous-$E$-espace vectoriel de dimension $1$, alors le $(\varphi,\Gamma)$-module libre sur $\R$ associ\'e par \cite[Th.~A]{Be2} \`a $D$ muni de la filtration \`a deux crans ${\rm Fil}^{-h_1}(D_L)=D_L$, ${\rm Fil}^{-h}(D_L)=L\otimes_{\Qp}F$ pour $h\in \{h_2,\dots,h_1-1\}$, ${\rm Fil}^{-h_2+1}(D_L)=0$ est une extension non scind\'ee de $\R(x^{h_2})/(t^{h_1-h_2})$ par ${\mathcal D}(x^{h_1})$. De plus toute extension non scind\'ee de $\R(x^{h_2})/(t^{h_1-h_2})$ par ${\mathcal D}(x^{h_1})$ provient comme cela d'un (unique) sous-$E$-espace vectoriel de dimension $1$ de $D_L^{\glp}$.\\
(ii) Soit $h_1>h_2$ dans $\Z$, il existe un isomorphisme naturel de $E$-espaces vectoriels~:
\begin{equation}\label{versfilmax}
{\rm Ext}^1_{(\varphi,\Gamma)}\big(\R(x^{h_2})/(t^{h_1-h_2}),{\mathcal D}(x^{h_1})\big)\buildrel\sim\over\longrightarrow D_L^{\glp}
\end{equation}
qui envoie la droite dans ${\rm Ext}^1_{(\varphi,\Gamma)}(\R(x^{h_2})/(t^{h_1-h_2}),{\mathcal D}(x^{h_1}))$ engendr\'ee par une extension non scind\'ee vers le sous-$E$-espace vectoriel de dimension $1$ de $D_L^{\glp}$ associ\'e par (i) \`a cette extension.
\end{prop}
\begin{proof}
On note $d\=\dim_E D_L^{\mathrm{Gal}(L/\mathbb{Q}_p)}$.\\
(i) Montrons la premi\`ere assertion. Notons $\mathrm{Fil}_1^\cdot (D_L)$ (resp. $\mathrm{Fil}_2^\cdot (D_L)$) la filtration sur $D_L$ donn\'ee par $\mathrm{Fil}_1^{-h_1}(D_L)=D_L$ et $\mathrm{Fil}_1^{-h_1+1}(D_L)=0$ (resp. par $\mathrm{Fil}_2^{-h_1}(D_L)=D_L$, $\mathrm{Fil}^{-h}(D_L)= L\otimes_{\mathbb{Q}_p} F$ pour $h\in \{h_2, \dots, h_1-1\}$ et $\mathrm{Fil}^{-h_2+1}(D_L)=0$). Notons $\mathcal{{\mathcal D}}_i$ le $(\varphi,\Gamma)$-module sur $ {\mathcal R}_E$ associ\'e \`a $(D, \mathrm{Fil}_i^\cdot(D_L))$ (cf. \cite[D\'ef.~II.2.4]{Be2}), qui est alors libre de rang $d$ sur $ {\mathcal R}_E$. Le morphisme naturel $(D, \mathrm{Fil}_1^\cdot) \rightarrow (D, \mathrm{Fil}_2^\cdot)$ induit un morphisme $f:\mathcal{D}_1 \rightarrow \mathcal{D}_2$. Par \cite[Th.II.2.6]{Be2}, ce morphisme est injectif (sinon le noyau donnerait un sous-objet de $(D, \mathrm{Fil}_1^\cdot)$ envoy\'e vers $0$ dans $(D, \mathrm{Fil}_2^\cdot)$). Soit $r\in \mathbb{Q}_{>0}$ suffisamment grand tel que $\mathcal{D}_1 \hookrightarrow \mathcal{D}_2$ est induit par un morphisme $f_r: \mathcal{D}_{1,r} \hookrightarrow \mathcal{D}_{2,r}$ de $(\varphi,\Gamma)$-modules sur ${\mathcal R}_E^r$ ((iii) du Lemme 2.2.3). Alors $\mathrm{coker}(f_r)$ est un $(\varphi,\Gamma)$-module g\'en\'eralis\'e de torsion sur ${\mathcal R}_E^r$. Par \cite[Prop.~II.2.5]{Be2} (et sa preuve) on a un isomorphisme naturel pour $n\gg 0$ (en fait $n\geq n(r)$)~:
\begin{equation}\label{diff}
D_{\mathrm{diff},n}^+(\mathcal{D}_i)\buildrel\sim\over\longrightarrow \mathrm{Fil}^0 \big(\Q_p(\zeta_{p^n})((t))\otimes_{\Q_p} D_L^{\mathrm{Gal}(L/\mathbb{Q}_p)}\big)
\end{equation}
o\`u $D_L$ est muni de la filtration $\mathrm{Fil}_i^\cdot$ et $\mathrm{Fil}^j (\Q_p(\zeta_{p^n})((t)))\=t^j \Q_p(\zeta_{p^n})[[t]]$ pour $j\in \Z$. Par (\ref{diff}), on v\'erifie que $D_{\mathrm{diff},n}^+(\mathcal{D}_2)/D_{\mathrm{diff},n}^+(\mathcal{D}_1)$ est isomorphe \`a $(E\otimes_{\Q_p} \Q_p(\zeta_{p^n}))[[t]]/(t^{h_1-h_2})$. Par \cite[Prop. 4.4]{Li} (et la discussion avant \cite[Th.4.3]{Li}), on en d\'eduit un isomorphisme $\mathrm{coker}(f)\cong {\mathcal R}_E/(t^{h_1-h_2})$ comme $ {\mathcal R}_E$-modules. Soit $W_{\mathrm{dR}}^+({\mathcal D}_i)\=B_{\rm dR}^+ \otimes_{ {\mathcal R}_E^r, \iota_n} \mathcal{D}_{i,r} \cong B_{\rm dR}^+\otimes_{\Q_p(\zeta_{p^n})[[t]]} D_{\mathrm{diff},n}^+(\mathcal{D}_i)$ pour $n\gg 0$, par le Lemme \ref{h0t} on a pour $h\in \mathbb{Z}$ un isomorphisme~:
\begin{equation}\label{bpa}
 H^0_{(\varphi,\Gamma)}(\mathrm{coker}(f)(x^{-h})) \cong H^0\big(\gp, t^{-h} W_{\mathrm{dR}}^+({\mathcal D}_2)/ t^{-h}W_{\mathrm{dR}}^+({\mathcal D}_1)\big).
\end{equation}
Par (\ref{diff}) et (\ref{tor1}), on voit facilement que~:
\begin{equation}\label{dr}
H^0\big(\gp, t^{-h} W_{\mathrm{dR}}^+({\mathcal D}_2)/ t^{-h}W_{\mathrm{dR}}^+({\mathcal D}_1)\big)\cong \begin{cases}
 E & h_2\leq h\leq h_1-1 \\
 0 & \text{sinon}.
 \end{cases}
\end{equation}
Comme $H^0_{(\varphi,\Gamma)}(\mathrm{coker}(f)(x^{-h_2}))=(\mathrm{coker}(f)(x^{-h_2}))^{\varphi=1, \Gamma=1}\simeq E$ par (\ref{bpa}) et (\ref{dr}), il existe $s\leq h_1-h_2$ et un morphisme injectif de $(\varphi,\Gamma)$-modules sur $\R$~:
\begin{equation}\label{injs}
{\mathcal R}_E/(t^s) \hookrightarrow \mathrm{coker}(f)(x^{-h_2}), \ x\longmapsto xe
\end{equation}
o\`u $e$ est une $E$-base de $(\mathrm{coker}(f)(x^{-h_2}))^{\varphi=1, \Gamma=1}$. Si $s\neq h_1-h_2$, il existe alors $e'\in \mathrm{coker}(f)$ tel que $e=te'$. Mais cela implique $e'\in ( \mathrm{coker}(f)(x^{1-h_2}))^{\varphi=1, \Gamma=1}$, une contradiction puisque ce dernier espace est nul par (\ref{bpa}) et (\ref{dr}). On en d\'eduit un isomorphisme $\mathrm{coker}(f)\cong {\mathcal R}_E(x^{h_2})/(t^{h_1-h_2})$ de $(\varphi,\Gamma)$-modules.\\
Montrons la deuxi\`eme assertion de (i). Notons $\mathcal{D}_1\=\mathcal{D}(x^{h_1})$ et soit $\mathcal{D}_2$ un $(\varphi,\Gamma)$-module g\'en\'eralis\'e sur ${\mathcal R}_E$ qui est une extension non scind\'ee de ${\mathcal R}_E(x^{h_2})/(t^{h_1-h_2})$ par $\mathcal{D}_1$. Montrons que $\mathcal{D}_2$ est sans torsion. Sinon, on voit facilement que $\mathcal{D}_2'\=\mathcal{D}_2/(\mathcal{D}_2)_{\mathrm{tor}}$ est une extension non scind\'ee de ${\mathcal R}_E(x^{h_2})/(t^{h})$ par $\mathcal{D}_1$ pour un $h$ tel que $0<h<h_1-h_2$. On a comme en (\ref{dim}) ci-dessous~:
\begin{equation*}
 \dim_E \Ext^1_{(\varphi,\Gamma)}\big(\mathcal{R}_E(x^{h_2})/(t^h), \mathcal{D}_1\big) =\dim_E H^0_{(\varphi,\Gamma)}(\mathcal{D}_1(x^{-h_2-h})/(t^h)).
\end{equation*}
Mais par le Lemme \ref{h0t} et (\ref{diff}), on a $H^0_{(\varphi,\Gamma)}(\mathcal{D}_1(x^{-h_2-h})/(t^h))=0$ et donc $\Ext^1_{(\varphi,\Gamma)}(\mathcal{R}_E(x^{h_2})/(t^h), \mathcal{D}_1)=0$, une contradiction. Comme $\mathcal{D}_1$ est de de Rham (i.e. la connexion associ\'ee est localement triviale, cf. \cite[Th.~A]{Be2}), $\mathcal{D}_2$ l'est aussi (rappelons que $D_{\mathrm{dR}}(\mathcal{D}_1)\cong (D_{\mathrm{diff},n}^+(\mathcal{D}_1)[\frac{1}{t}])^{\Gamma}=(D_{\mathrm{diff},n}^+(\mathcal{D}_2)[\frac{1}{t}])^{\Gamma}$ pour $n\gg 0$). Soit $(D, \mathrm{Fil}_2^\cdot )$ le module de Deligne-Fontaine \emph{filtr\'e} associ\'e \`a $\mathcal{D}_2$. Comme $\mathcal{D}_2/\mathcal{D}_1\cong \mathcal{R}_E(x^{h_2})/(t^{h_1-h_2})$, par \cite[Th.~4.3~\&~Prop. 4.4]{Li} on a pour $n\gg 0$~:
\begin{equation}\label{diff2}
D_{\mathrm{diff},n}^+(\mathcal{D}_2)/D_{\mathrm{diff},n}^+(\mathcal{D}_1) \cong (E\otimes_{\Q_p} \Q_p(\zeta_{p^n}))[[t]]/(t^{h_1-h_2}).
\end{equation}
En utilisant (\ref{diff}), on v\'erifie alors facilement que la filtration $\mathrm{Fil}_2^\cdot$ doit \^etre comme dans le (i) de la proposition.\\
(ii) En appliquant $\Hom _{(\varphi,\Gamma)}(-, \mathcal{D}(x^{h_1}))$ \`a la suite exacte $0 \rightarrow {\mathcal R}_E(x^{h_1}) \rightarrow {\mathcal R}_E(x^{h_2}) \rightarrow {\mathcal R}_E(x^{h_2})/(t^{h_1-h_2})\rightarrow 0$, on obtient une suite exacte de $E$-espaces vectoriels~:
\begin{multline}\label{pga}
0 \rightarrow H^0 _{(\varphi,\Gamma)}(\mathcal{D}(x^{h_1-h_2})) \rightarrow H^0 _{(\varphi,\Gamma)}(\mathcal{D}) \rightarrow \Ext^1 _{(\varphi,\Gamma)}\big({\mathcal R}_E(x^{h_2})/(t^{h_1-h_2}), \mathcal{D}(x^{h_1})\big) \\ \rightarrow H^1 _{(\varphi,\Gamma)}(\mathcal{D}(x^{h_1-h_2})) \rightarrow H^1 _{(\varphi,\Gamma)}(\mathcal{D}).
\end{multline}
La suite exacte longue de cohomologie appliqu\'ee \`a $0 \rightarrow \mathcal{D}(x^{h_1-h_2}) \rightarrow \mathcal{D} \rightarrow \mathcal{D}/(t^{h_1-h_2})\rightarrow 0$ donne aussi une suite exacte de $E$-espaces vectoriels~:
\begin{multline*}
0 \longrightarrow H^0 _{(\varphi,\Gamma)}(\mathcal{D}(x^{h_1-h_2})) \longrightarrow H^0 _{(\varphi,\Gamma)}(\mathcal{D}) \longrightarrow H^0_{(\varphi,\Gamma)}\big(\mathcal{D}/(t^{h_1-h_2})\big) \\ \longrightarrow H^1 _{(\varphi,\Gamma)}(\mathcal{D}(x^{h_1-h_2})) \longrightarrow H^1 _{(\varphi,\Gamma)}(\mathcal{D})
\end{multline*}
d'o\`u on d\'eduit avec (\ref{pga}) l'\'egalit\'e~:
\begin{equation}\label{dim}
\dim_E \Ext^1 _{(\varphi,\Gamma)}\big({\mathcal R}_E(x^{h_2})/(t^{h_1-h_2}), \mathcal{D}(x^{h_1})\big)=\dim_E H^0 _{(\varphi,\Gamma)}\big(\mathcal{D}/(t^{h_1-h_2})\big).
\end{equation}
En outre, en appliquant $\Hom _{(\varphi,\Gamma)}(\mathcal{R}_E(x^{h_2})/(t^{h_1-h_2}), -)$ \`a la suite exacte $0\rightarrow \mathcal{D}(x^{h_1}) \rightarrow \mathcal{D}(x^{h_2}) \rightarrow \mathcal{D}(x^{h_2})/(t^{h_1-h_2}) \rightarrow 0$, on obtient une injection~:
\begin{equation}\footnotesize\label{tor2}
\Hom _{(\varphi,\Gamma)}\big( {\mathcal R}_E(x^{h_2})/(t^{h_1-h_2}), \mathcal{D}(x^{h_2})/(t^{h_1-h_2}) \big) \hookrightarrow \Ext^1 _{(\varphi,\Gamma)}\big({\mathcal R}_E(x^{h_2})/(t^{h_1-h_2}), \mathcal{D}(x^{h_1})\big).
\end{equation}
Comme $\Hom _{(\varphi,\Gamma)}( {\mathcal R}_E(x^{h_2})/(t^{h_1-h_2}), \mathcal{D}(x^{h_2})/(t^{h_1-h_2}) ) \buildrel\sim\over\rightarrow H^0 _{(\varphi,\Gamma)}(\mathcal{D}/(t^{h_1-h_2}))$, on d\'eduit de (\ref{tor2}) et (\ref{dim}) un isomorphisme naturel~:
\begin{equation}\label{isonat}
H^0 _{(\varphi,\Gamma)}(\mathcal{D}/(t^{h_1-h_2}))\buildrel\sim\over\longrightarrow \Ext^1 _{(\varphi,\Gamma)}\big({\mathcal R}_E(x^{h_2})/(t^{h_1-h_2}), \mathcal{D}(x^{h_1})\big).
\end{equation}
En combinant (\ref{isonat}) avec le Lemme \ref{h0t} appliqu\'e \`a $\mathcal{D}_2=\mathcal{D}$ et $\mathcal{D}_1=\mathcal{D}(x^{h_1-h_2})$, et en notant $W_{\rm dR}^+(\mathcal{D})$ le $B_{\rm dR}^+$-module associ\'e \`a $\mathcal{D}$, on obtient un isomorphisme naturel~:
\begin{equation*}
 H^0\big(\gp, W_{\rm dR}^+(\mathcal{D})/(t^{h_1-h_2})\big) \buildrel\sim\over \longrightarrow \Ext^1 _{(\varphi,\Gamma)}\big({\mathcal R}_E(x^{h_2})/(t^{h_1-h_2}), \mathcal{D}(x^{h_1})\big).
\end{equation*}
Utilisant les isomorphismes naturels~:
\begin{equation}\small\label{basdroite}
D_L^{\mathrm{Gal}(L/\mathbb{Q}_p)}\cong H^0\big(\gp, W_{\rm dR}^+(\mathcal{D})\big) \buildrel\sim\over\longrightarrow H^0\big(\gp, W_{\rm dR}^+(\mathcal{D})/(t^{h_1-h_2})\big)
\end{equation}
(proc\'eder comme en (\ref{dr}) pour l'isomorphisme de droite), on obtient finalement un isomorphisme naturel~:
\begin{equation}
\iota: \Ext^1 _{(\varphi,\Gamma)}\big({\mathcal R}_E(x^{h_2})/(t^{h_1-h_2}), \mathcal{D}(x^{h_1})\big) \buildrel\sim\over\longrightarrow D_L^{\mathrm{Gal}(L/\mathbb{Q}_p)}.
\end{equation}
Montrons qu'il satisfait la propri\'et\'e du (ii) de la proposition. Soit $F\subseteq D_L^{\mathrm{Gal}(L/\mathbb{Q}_p)}$ une $E$-droite et $\mathcal{D}_2$ le $(\varphi,\Gamma)$-module associ\'e en (i), on doit donc montrer $\iota(E[\mathcal{D}_2])=F$ (avec une notation \'evidente). On a un diagramme commutatif~:
\begin{equation}\label{h12}
\begin{CD}
0 @>>> \mathcal{D}(x^{h_1}) @>>> \mathcal{D}_2 @>>> {\mathcal R}_E(x^{h_2})/(t^{h_1-h_2}) @>>> 0 \\
@. @| @VVV @V j VV @. \\
0 @>>> \mathcal{D}(x^{h_1}) @>>> \mathcal{D}(x^{h_2}) @>>> \mathcal{D}(x^{h_2})/(t^{h_1-h_2}) @>>> 0
\end{CD}
\end{equation}
qui, combin\'e avec la suite exacte~:
\begin{equation}\label{r12}
0 \longrightarrow {\mathcal R}_E(x^{h_1}) \longrightarrow {\mathcal R}_E(x^{h_2}) \longrightarrow {\mathcal R}_E(x^{h_2})/(t^{h_1-h_2}) \longrightarrow 0
\end{equation}
induit un diagramme commutatif (en notant $\mathcal R$ au lieu de $\R$)~:
\begin{equation*}\small
\begindc{\commdiag}[70]
\obj(0,0)[a]{$\Hom\Big({\mathcal R}(x^{h_2}), \frac{{\mathcal R}(x^{h_2})}{(t^{h_1-h_2})}\Big)$}
\obj(25,0)[b]{$\Hom\Big({\mathcal R}(x^{h_2}), \frac{\mathcal{D}(x^{h_2})}{(t^{h_1-h_2})}\Big)$}
\obj(10,10)[c]{$\Ext^1\big({\mathcal R}(x^{h_2}), \mathcal{D}(x^{h_1})\big)$}
\obj(35,10)[d]{$\Ext^1\big({\mathcal R}(x^{h_2}), \mathcal{D}(x^{h_1})\big)$}
\obj(0,16)[e]{$\Hom\Big(\frac{{\mathcal R}(x^{h_2})}{(t^{h_1-h_2})}, \frac{{\mathcal R}(x^{h_2})}{(t^{h_1-h_2})}\Big)$}
\obj(25,16)[f]{$\Hom\Big(\frac{{\mathcal R}(x^{h_2})}{(t^{h_1-h_2})}, \frac{\mathcal{D}(x^{h_2})}{(t^{h_1-h_2})}\Big)$}
\obj(10,26)[g]{$\Ext^1\Big(\frac{{\mathcal R}(x^{h_2})}{(t^{h_1-h_2})}, \mathcal{D}(x^{h_1})\Big)$}
\obj(35,26)[h]{$\Ext^1\Big(\frac{{\mathcal R}(x^{h_2})}{(t^{h_1-h_2})}, \mathcal{D}(x^{h_1})\Big)$}
\mor{a}{b}{$j_1$}
\mor{a}{c}{}
\mor{c}{d}{}[+1,\equalline]
\mor{b}{d}{}
\mor{e}{f}{$j_2$}
\mor{e}{g}{$\delta_1$}
\mor{g}{h}{}[+1,\equalline]
\mor{f}{h}{$\delta_2$}
\mor{e}{a}{\!$v_1$}
\mor{f}{b}{\!$v_2$}
\mor{g}{c}{}
\mor{h}{d}{}
\enddc
\end{equation*}
(par exemple, la face tout \`a gauche est induite par la suite exacte en haut de (\ref{h12}) et la suite exacte (\ref{r12})). Les morphismes $v_1$, $v_2$ sont clairement des isomorphismes. On a $\im(\delta_1)=E[\mathcal{D}_2]$ donc $\im(\delta_2\circ j_2)=E[\mathcal{D}_2]$ et $\im(\delta_2 \circ v_2^{-1}\circ j_1)=\im(\delta_2\circ v_2^{-1}\circ j_1\circ v_1)=E[\mathcal{D}_2]$. Soit $W_{\mathrm{dR}}^+(t^{-h_2} \mathcal{D}_2)$ le $B_{\mathrm{dR}}^+$-module associ\'ee \`a $t^{-h_2} \mathcal{D}_2$, par le Lemme \ref{h0t} appliqu\'e avec ${\mathcal D}_2$ et ${\mathcal D}_1={\mathcal D}(x^{h_1})$, l'isomorphisme ${\mathcal D}_2/{\mathcal D}(x^{h_1})\simeq {\mathcal R}_E(x^{h_2})/(t^{h_1-h_2})$ (premi\`ere suite exacte en (\ref{h12})) et (\ref{tor0}), on a un diagramme commutatif~:
\begin{equation*}
 \begin{CD}
 \Hom _{(\varphi,\Gamma)}\Big({\mathcal R}_E(x^{h_2}), \frac{{\mathcal R}_E(x^{h_2})}{(t^{h_1-h_2})}\Big) @> j_1 >> \Hom _{(\varphi,\Gamma)}\Big({\mathcal R}_E(x^{h_2}), \frac{\mathcal{D}(x^{h_2})}{(t^{h_1-h_2})}\Big) \\
 @V \wr VV @V \wr VV \\
 H^0\Big(\gp, \frac{W_{\mathrm{dR}}^+(t^{-h_2} \mathcal{D}_2)}{t^{h_1-h_2} W_{\mathrm{dR}}^+(\mathcal{D})}\Big) @>>> H^0\Big(\gp, \frac{W_{\mathrm{dR}}^+(\mathcal{D})}{(t^{h_1-h_2})}\Big) \\
 @V \wr VV @V \wr VV\\
 \big(\mathrm{Fil}^{0} D_{\rm dR}(t^{-h_2}\mathcal{D}_2)\big)^{\mathrm{Gal}(L/\mathbb{Q}_p)} @>>> D_L^{\mathrm{Gal}(L/\mathbb{Q}_p)}
 \end{CD}
\end{equation*}
o\`u les applications horizontales sont induites par $j$, la commutativit\'e du carr\'e du haut vient de la fonctorialit\'e de (\ref{tor0}), o\`u $D_{\rm dR}(t^{-h_2}\mathcal{D}_2)$ est le module filtr\'e sur $L$ associ\'e au $(\varphi,\Gamma)$-module de de Rham $t^{-h_2}\mathcal{D}_2$, l'isomorphisme du bas \`a droite est celui en (\ref{basdroite}) et en remarquant que (\ref{basdroite}) induit des isomorphismes (en proc\'edant comme pour (\ref{dr}))~:
\begin{multline*}
\big(\mathrm{Fil}^{0} D_{\rm dR}(t^{-h_2}\mathcal{D}_2)\big)^{\mathrm{Gal}(L/\mathbb{Q}_p)}\simeq H^0\big(\gp, W_{\mathrm{dR}}^+(t^{-h_2} \mathcal{D}_2)\big)\\
\buildrel\sim\over\longrightarrow H^0\Big(\gp, \frac{W_{\mathrm{dR}}^+(t^{-h_2} \mathcal{D}_2)}{t^{h_1-h_2} W_{\mathrm{dR}}^+(\mathcal{D})}\Big).
\end{multline*}
L'image du morphisme horizontal en bas est $F\subseteq D_L^{\mathrm{Gal}(L/\mathbb{Q}_p)}$. De plus la compos\'ee verticale \`a droite est \'egale \`a $\iota \circ \delta_2 \circ v_2^{-1}$ par la construction de $\iota$. Comme $\im(\delta_2 \circ v_2^{-1}\circ j_1)=E[\mathcal{D}_2]$, on obtient bien $\iota(E[\mathcal{D}_2])=F$, ce qui termine la preuve de (ii).
\end{proof}

\begin{rem}\label{torsion}
{\rm Soit $\delta:\gp\rightarrow E^\times$ un caract\`ere localement constant, il r\'esulte facilement de \cite[Th.~4.3~\&~Prop.~4.4]{Li} que $\R(x^{h_2})/(t^{h_1-h_2})\simeq \R(x^{h_2}\delta)/(t^{h_1-h_2})$ (on laisse cela en exercice au lecteur). En tordant les extensions par $\R(\delta)$, on en d\'eduit un isomorphisme ${\rm Ext}^1_{(\varphi,\Gamma)}(\R(x^{h_2})/(t^{h_1-h_2}),T)\simeq {\rm Ext}^1_{(\varphi,\Gamma)}(\R(x^{h_2})/(t^{h_1-h_2}),T(\delta))$ pour tout $(\varphi,\Gamma)$-module $T$. Par (\ref{versfilmax}) on obtient donc un isomorphisme naturel pour tout $\delta$ localement constant~:
$${\rm Ext}^1_{(\varphi,\Gamma)}\big(\R(x^{h_2})/(t^{h_1-h_2}),{\mathcal D}(x^{h_1})(\delta)\big)\buildrel\sim\over\longrightarrow D_L^{\glp}.$$}
\end{rem}

Si $V$ est une repr\'esentation potentiellement semi-stable de $\gp$ sur $E$ qui devient semi-stable en restriction \`a ${\Gal}(\Qpbar/L)$, rappelons qu'on lui associe le module de Deligne-Fontaine $D=(B_{\rm st}\otimes_{\Qp}V)^{{\Gal}(\Qpbar/L)}$ muni des op\'erateurs $\varphi,N$ agissant sur $B_{\rm st}$ et de l'action r\'esiduelle de $\glp$, et que $D_L\buildrel\sim\over\rightarrow (B_{\rm dR}\otimes_{\Qp}V)^{{\Gal}(\Qpbar/L)}$ est muni de la filtration $\Fil^\cdot D_L$ (d\'ecroissante, exhaustive, s\'epar\'ee) induite par celle sur $B_{\rm dR}$.

\subsection{Nullit\'e de certaines extensions}\label{conjcris}

On d\'emontre la nullit\'e d'un certain ${\rm Ext}^1_G$ (Th\'eor\`eme \ref{ext1nul}), ce qui lorsque $G=\GL_n$ implique la Conjecture $3.3.1$ de \cite{Br1} pour ${\rm GL}_n(\Qp)$. Ce r\'esultat sera utilis\'e dans le cas particulier cristallin de la Conjecture \ref{extglobmieux}, cf. le (ii) du Th\'eor\`eme \ref{ncris}.

On conserve les notations des \S\S~\ref{prel}, \ref{notabene} et on rappelle que $w_0$ est l'\'el\'ement de longueur maximale dans $W$. On fixe une racine simple $\alpha\in S$, un poids $\lambda\in X(T)$ dominant par rapport \`a $B^-$ et un caract\`ere lisse $\chi:T(\Qp)\rightarrow E^\times$. On note $W^\alpha$ le groupe de Weyl de $L_{Q_\alpha}$, c'est-\`a-dire le sous-groupe de $W$ engendr\'e par les r\'eflexions simples diff\'erentes de $s_\alpha$, et $w_0^\alpha$ l'\'el\'ement de longueur maximale de $W^\alpha$. Si $\pi$, $\pi'$ sont deux repr\'esentations localement analytiques admissibles de $G(\Qp)$ sur $E$, on note ${\rm Ext}^i_G(\pi,\pi')\={\rm Ext}^i_{D(G(\Qp),E)}({\pi'}^\vee,\pi^\vee)$ pour $i\geq 0$ o\`u ${\rm Ext }^i_{D(G(\Qp),E)}$ est dans la cat\'egorie des $D(G(\Qp),E)$-modules \`a gauche (abstraits). Lorsque $i=1$ (resp. $i=0$), cela co\"\i ncide avec le ${\rm Ext}^1_G(\pi,\pi')$ comme \`a la fin du \S~\ref{intro} (resp. avec les homomorphismes continus $G(\Qp)$-\'equivariants de $\pi$ dans $\pi'$), cf. \cite[Th.~6.3]{ST2} et \cite[Lem.~2.1.1]{Br1}.

On commence par plusieurs r\'esultats pr\'eliminaires. Par le Lemme \ref{singular}, le $U(\mg)$-module $U(\mg)\otimes_{U(\mb^-)}\lambda$ admet en quotient une unique extension non scind\'ee $\!\begin{xy}(-60,0)*+{L^-(s_\alpha\cdot \lambda)}="a";(-39,0)*+{L^-(\lambda)}="b";{\ar@{-}"a";"b"}\end{xy}$. Appliquant le foncteur exact contravariant ${\mathcal F}^G_{B^-}(-,\chi)$ et utilisant \cite[Cor.~2.5~\&~Cor.~2.7]{Br3}, on en d\'eduit que la s\'erie principale $(\Ind_{B^-(\Qp)}^{G(\Qp)}(-\lambda)\otimes \chi)^{\an}={\mathcal F}^G_{B^-}(U(\mg)\otimes_{U(\mb^-)}\lambda,\chi)$ admet en sous-repr\'esentation une unique extension non scind\'ee~:
\begin{equation}\label{extnonsc}
\begin{xy}(-60,0)*+{{\mathcal F}^G_{B^-}(L^-(\lambda),\chi)}="a";(-22,0)*+{{\mathcal F}^G_{B^-}(L^-(s_\alpha\cdot \lambda),\chi)}="b";{\ar@{-}"a";"b"}\end{xy}
\end{equation}
o\`u ${\mathcal F}^G_{B^-}(L^-(\lambda),\chi)=L(-\lambda)\otimes_E(\Ind_{B^-(\Qp)}^{G(\Qp)}\chi)^\infty$ est en sous-objet. Pour $w\in W$, on note pour all\'eger $C(w\cdot \lambda)\={\mathcal F}^G_{B^-}(L^-(w\cdot \lambda),\chi)$, $I(w\cdot\lambda)\=(\Ind_{B^-(\Qp)}^{G(\Qp)}(-w\cdot \lambda)\otimes \chi)^{\an}={\mathcal F}^G_{B^-}(U(\mg)\otimes_{U(\mb^-)}w\cdot \lambda,\chi)$ (oubliant dans la notation le caract\`ere $\chi$ qui ne joue aucun r\^ole).

\begin{prop}\label{lem03}
Supposons $\lg(w_0^\alpha)\geq 1$, on a une suite exacte de repr\'esenta\-tions localement analytiques admissibles de $G(\Qp)$ sur $E$~:
\begin{multline}\label{BGGL}
0\longrightarrow X \longrightarrow I(\lambda)\buildrel d^0\over \longrightarrow \bigoplus_{w\in W^\alpha \atop \lg(w)=1}I(w\cdot \lambda) \buildrel d^1\over \longrightarrow \bigoplus_{w\in W^\alpha \atop \lg(w)=2}I(w\cdot \lambda)\\
\buildrel d^2 \over \longrightarrow \cdots \buildrel d^{\lg(w_0^\alpha)-1}\over \longrightarrow I(w_0^\alpha\cdot\lambda)\longrightarrow 0
\end{multline}
o\`u la repr\'esentation $X$ est de la forme $\begin{xy}(-60,0)*+{C(\lambda)}="a";(-42,0)*+{C(s_\alpha\!\cdot \!\lambda)}="b";(-27,0)*+{Y}="c";{\ar@{-}"a";"b"};{\ar@{-}"b";"c"}\end{xy}$ avec $\begin{xy}(-60,0)*+{C(\lambda)}="a";(-42,0)*+{C(s_\alpha\!\cdot \!\lambda)}="b";{\ar@{-}"a";"b"}\end{xy}$ donn\'e par (\ref{extnonsc}) et $Y$ admettant une filtration dont les gradu\'es sont $C(w'\cdot \lambda)$ pour des $w'\in W$ tels que $\lg(w')\geq 3$.
\end{prop}
\begin{proof}
Soit $\mq_\alpha^-$ (resp. $\mathfrak{l}_{Q_\alpha}$) l'alg\`ebre de Lie de $Q_\alpha^-(\Qp)$ (resp. $L_{Q_\alpha}(\Qp)$). En appliquant le foncteur exact $U(\mg)\otimes_{U(\mq_\alpha^-)}(-)$ \`a la r\'esolution BGG pour $\mathfrak{l}_{Q_\alpha}$~:
\begin{multline*}
0\longrightarrow U(\mathfrak{l}_{Q_\alpha})\otimes_{U(\mathfrak{l}_{Q_\alpha}\cap \mb^-)}w_0^\alpha\cdot\lambda\longrightarrow \cdots \longrightarrow \bigoplus_{w\in W^\alpha \atop \lg(w)=1}U(\mathfrak{l}_{Q_\alpha})\otimes_{U(\mathfrak{l}_{Q_\alpha}\cap \mb^-)}w\cdot\lambda\\
\longrightarrow U(\mathfrak{l}_{Q_\alpha})\otimes_{U(\mathfrak{l}_{Q_\alpha}\cap \mb^-)}\lambda \longrightarrow L^-(\lambda)_{Q_\alpha}\longrightarrow 0,
\end{multline*}
on obtient la suite exacte de $U(\mg)$-modules~:
\begin{multline}\label{BGGq}
0\longrightarrow U(\mg)\otimes_{U(\mb^-)}w_0^\alpha\cdot\lambda\longrightarrow \cdots \longrightarrow \bigoplus_{w\in W^\alpha \atop \lg(w)=1}U(\mg)\otimes_{U(\mb^-)}w\cdot\lambda\\
\longrightarrow U(\mg)\otimes_{U(\mb^-)}\lambda \longrightarrow U(\mg)\otimes_{U(\mq_\alpha^-)}L^-(\lambda)_{Q_\alpha}\longrightarrow 0,
\end{multline}
puis en appliquant le foncteur exact contravariant ${\mathcal F}_{B^-}^G(-,\chi)$ \`a (\ref{BGGq}), on obtient (\ref{BGGL}) avec $X={\mathcal F}_{B^-}^G(U(\mg)\otimes_{U(\mq_\alpha^-)}L^-(\lambda)_{Q_\alpha},\chi)$.\\
 
Montrons l'assertion sur $X$. Comme le constituant $L^-(s_\alpha\cdot\lambda)$ de $U(\mg)\otimes_{U(\mb^-)}\lambda$ n'appara\^it dans aucun $U(\mg)\otimes_{U(\mb^-)}w\cdot\lambda$ lorsque $w\in W^\alpha$, $\lg(w)=1$ (par exemple par \cite[\S~5.2]{Hu}), on d\'eduit du Lemme \ref{singular} que le quotient $\!\begin{xy}(-60,0)*+{L^-(s_\alpha\cdot \lambda)}="a";(-39,0)*+{L^-(\lambda)}="b";{\ar@{-}"a";"b"}\end{xy}\!$ de $U(\mg)\otimes_{U(\mb^-)}\lambda$ est \'egalement quotient de $U(\mg)\otimes_{U(\mq_\alpha^-)}L^-(\lambda)_{Q_\alpha}$. Comme $L(w\cdot \lambda)$ n'est pas dans la cat\'egorie ${\mathcal O}^{\mq_\alpha^-}_{\rm alg}$ lorsque $w\in W^\alpha$, $\lg(w)=1$, on d\'eduit que $U(\mg)\otimes_{U(\mq_\alpha^-)}L^-(\lambda)_{Q_\alpha}\in {\mathcal O}^{\mq_\alpha^-}_{\rm alg}$ est de la forme $\!\begin{xy}(-76,0)*+{Z}="z";(-60,0)*+{L^-(s_\alpha\cdot \lambda)}="a";(-40,0)*+{L^-(\lambda)}="b";{\ar@{-}"z";"a"};{\ar@{-}"a";"b"}\end{xy}\!$
o\`u les constituants de $Z$ sont des $L^-(w'\cdot \lambda)$ pour $w'\in W$, $\lg(w')\geq 2$. Lorsque $w'\in W^\alpha$, le constituant $L^-(w'\cdot \lambda)$ n'est pas dans ${\mathcal O}^{\mq_\alpha^-}_{\rm alg}$ (car $w'\cdot \lambda$ n'est pas dominant par rapport \`a $L_{Q_\alpha}\cap B^-$), et lorsque $w'=s_\beta s_\alpha\cdot\lambda$ pour $\beta\in S\backslash\{\alpha\}$, $L^-(w'\cdot \lambda)=L^-(s_\beta\cdot(s_\alpha\cdot\lambda))$ n'est pas non plus dans ${\mathcal O}^{\mq_\alpha^-}_{\rm alg}$ (car $s_\alpha\cdot \lambda$ est dominant par rapport \`a $L_{Q_\alpha}\cap B^-$, donc $s_\beta\cdot(s_\alpha\cdot\lambda)$ ne peut plus l'\^etre). Donc les seuls constituants $L^-(w'\cdot \lambda)$ pour $w'\in W$, $\lg(w')=2$ qui peuvent appara\^\i tre dans $Z$ sont les $L^-(s_\alpha s_\beta\cdot\lambda)$ pour $\beta\in S\backslash\{\alpha\}$. Mais par \cite[8.3(a)]{Hu}, le constituant $L^-(s_\alpha s_\beta\cdot\lambda)$ appara\^\i t avec multiplicit\'e $1$ dans $U(\mg)\otimes_{U(\mb^-)}\lambda$, donc est n\'ecessairement dans l'image de l'injection $U(\mg)\otimes_{U(\mb^-)}s_\beta\cdot\lambda\hookrightarrow U(\mg)\otimes_{U(\mb^-)}\lambda$ (rappelons qu'il appara\^\i t aussi dans $U(\mg)\otimes_{U(\mb^-)}s_\beta\cdot\lambda$), et par (\ref{BGGq}) ne peut donc plus appara\^\i tre dans le quotient $U(\mg)\otimes_{U(\mq_\alpha^-)}L^-(\lambda)_{Q_\alpha}$ de $U(\mg)\otimes_{U(\mb^-)}\lambda$. Appliquant ${\mathcal F}_{B^-}^G(-,\chi)$, on en d\'eduit avec (\ref{extnonsc}) que $X={\mathcal F}_{B^-}^G(U(\mg)\otimes_{U(\mq_\alpha^-)}L^-(\lambda)_{Q_\alpha},\chi)$ a bien la forme voulue.
\end{proof}

On fixe $\pi$ une repr\'esentation localement alg\'ebrique admissible de $G(\Qp)$ sur $E$, c'est-\`a-dire le produit tensoriel (sur $E$) d'une repr\'esentation alg\'ebrique de dimension finie de $G(\Qp)$ par une repr\'esentation lisse admissible de $G(\Qp)$.

\begin{lem}\label{lem01}
Pour $w\in W$ et $i\in \{0,\dots,\lg(w_0)-1\}$ on a ${\rm Ext}^i_G(I(w\cdot \lambda),\pi)=0$.
\end{lem}
\begin{proof}
C'est un cas particulier de \cite[(4.17)~\&~Rem.~4.14]{Sc1} avec l'\'egalit\'e $\dim (G(\Qp))-\dim(B^-(\Qp))=\lg(w_0)$.
\end{proof}

\begin{lem}\label{lem02}
Pour $w\in W$ et $i\in \{0,\dots,\lg(w)-1\}$ on a ${\rm Ext}^i_G(C(w\cdot \lambda),\pi)=0$.
\end{lem}
\begin{proof}
Il n'y a rien \`a montrer si $\lg(w)=0$. Pour $j\in \{1,\dots,\lg(w_0)\}$ consid\'erons l'hypoth\`ese (de r\'ecurrence) HR$(j)$~: l'\'enonc\'e du lemme est vrai lorsque $\lg(w)\geq j$. Comme $C(w_0\cdot \lambda)=I(w_0\cdot\lambda)$, le Lemme \ref{lem01} implique HR$(\lg(w_0))$. Soit $j\in \{2,\dots,\lg(w_0)\}$ et montrons que HR$(j)$ implique HR$(j-1)$. Soit $w\in W$ tel que $\lg(w)=j-1$ (il suffit de consid\'erer ce cas), il r\'esulte de \cite[\S~5.2]{Hu} et de l'exactitude du foncteur ${\mathcal F}_{B^-}^G(-,\chi)$ que l'on a une suite exacte de repr\'esentations localement analytiques admissibles de $G(\Qp)$ sur $E$~:
$$0\longrightarrow C(w\cdot \lambda)\longrightarrow I(w\cdot\lambda)\longrightarrow X\longrightarrow 0$$
o\`u $X$ admet une filtration dont les gradu\'es sont les $C(w'\cdot \lambda)$ pour $w'\in W$ tel que $\lg(w')>\lg(w)=j-1$, i.e. $\lg(w')\geq j$. On en d\'eduit une suite exacte de $E$-espaces vectoriels pour tout $i\geq 0$~:
$${\rm Ext}^i_G(I(w\cdot\lambda),\pi)\longrightarrow {\rm Ext}^i_G(C(w\cdot\lambda),\pi)\longrightarrow {\rm Ext}^{i+1}_G(X,\pi).$$
Pour $i\in \{0,\dots,j-2\}$, on a ${\rm Ext}^i_G(I(w\cdot\lambda),\pi)=0$ par le Lemme \ref{lem01} et ${\rm Ext}^{i+1}_G(X,\pi)=0$ par HR$(j)$ et un d\'evissage \'evident. On en d\'eduit ${\rm Ext}^i_G(C(w\cdot\lambda),\pi)=0$ ce qui montre HR$(j-1)$ et ach\`eve la preuve.
\end{proof}

\begin{lem}\label{lem04}
Supposons $1\leq\lg(w_0^\alpha)\leq \lg(w_0)-2$. Pour $j\in \{0,\dots,\lg(w_0^\alpha)-1\}$ et $i\in \{0,\dots,j+2\}$ on a ${\rm Ext}^i_G(\im(d^j),\pi)=0$ o\`u $d^j$ est le morphisme dans la Proposition \ref{lem03}.
\end{lem}
\begin{proof}
On fait une r\'ecurrence descendante sur $j$. Par le Lemme \ref{lem01} l'\'enonc\'e est vrai lorsque $j=\lg(w_0^\alpha)-1$ car $\im(d^{\lg(w_0^\alpha)-1})=I(w_0^\alpha\cdot \lambda)$ et $j+2=\lg(w_0^\alpha)+1<\lg(w_0)$. Supposons l'\'enonc\'e vrai pour $j\in \{1,\dots,\lg(w_0^\alpha)-1\}$ (en supposant $2\leq \lg(w_0^\alpha)$ sinon il ne reste rien \`a montrer) et consid\'erons la suite exacte de repr\'esentations localement analytiques admissibles de $G(\Qp)$ sur $E$ issue de (\ref{BGGL})~:
$$0\longrightarrow \im(d^{j-1})\longrightarrow \bigoplus_{w\in W^\alpha \atop \lg(w)=j}I(w\cdot \lambda)\longrightarrow \im(d^{j})\longrightarrow 0.$$
On en d\'eduit une suite exacte courte de $E$-espaces vectoriels pour tout $i\geq 0$~:
$${\rm Ext}^i_G\big(\bigoplus_{w\in W^\alpha \atop \lg(w)=j}I(w\cdot \lambda),\pi\big)\longrightarrow {\rm Ext}^i_G(\im(d^{j-1}),\pi)\longrightarrow {\rm Ext}^{i+1}_G(\im(d^{j}),\pi).$$
Pour $i\in \{0,\dots,j+1\}$, le ${\rm Ext}^i_G$ de gauche est nul par le Lemme \ref{lem01} (car $j+1\leq \lg(w_0^\alpha)<\lg(w_0)$) et ${\rm Ext}^{i+1}_G(\im(d^{j}),\pi)=0$ par l'hypoth\`ese de r\'ecurrence au cran $j$. On en d\'eduit ${\rm Ext}^i_G(\im(d^{j-1}),\pi)=0$ ce qui ach\`eve la preuve.
\end{proof}

On peut maintenant montrer le r\'esultat principal de ce paragraphe.

\begin{thm}\label{ext1nul}
Soit $\pi$ une repr\'esentation localement alg\'ebrique admissible de $G(\Qp)$ sur $E$, c'est-\`a-dire le produit tensoriel d'une repr\'esentation alg\'ebrique de dimension finie de $G(\Qp)$ par une repr\'esentation lisse admissible de $G(\Qp)$, et supposons $1\leq\lg(w_0^\alpha)\leq \lg(w_0)-2$. Alors on a~:
$${\rm Ext}^1_G\big(\!\begin{xy}(-60,0)*+{{\mathcal F}^G_{B^-}(L^-(\lambda),\chi)}="a";(-23,0)*+{{\mathcal F}^G_{B^-}(L^-(s_\alpha\cdot \lambda),\chi)}="b";{\ar@{-}"a";"b"}\end{xy}\!,\pi\big)=0$$
o\`u la repr\'esentation de gauche est celle en (\ref{extnonsc}).
\end{thm}
\begin{proof}
Consid\'erons la suite exacte de repr\'esentations localement analytiques admissibles de $G(\Qp)$ sur $E$ issue de (\ref{BGGL})~:
$$0\longrightarrow X\longrightarrow I(\lambda)\longrightarrow \im(d^{0})\longrightarrow 0.$$
On en d\'eduit une suite exacte de $E$-espaces vectoriels~:
$${\rm Ext}^1_G(I(\lambda),\pi)\longrightarrow {\rm Ext}^1_G(X,\pi)\longrightarrow {\rm Ext}^{2}_G(\im(d^{0}),\pi).$$
Par le Lemme \ref{lem01} on a ${\rm Ext}^1_G(I(\lambda),\pi)=0$ et par le Lemme \ref{lem04} appliqu\'e pour $j=0$, $i=2$ on a ${\rm Ext}^{2}_G(\im(d^{0}),\pi)=0$. On en d\'eduit ${\rm Ext}^1_G(X,\pi)=0$. Vue la forme de $X$ dans la Proposition \ref{lem03}, on a aussi une suite exacte de $E$-espaces vectoriels~:
$${\rm Ext}^1_G(X,\pi)\longrightarrow {\rm Ext}^1_G\big(\!\begin{xy}(-60,0)*+{C(\lambda)}="a";(-42,0)*+{C(s_\alpha\!\cdot \!\lambda)}="b";{\ar@{-}"a";"b"}\end{xy}\!,\pi\big)\longrightarrow {\rm Ext}^{2}_G(Y,\pi).$$
Vue la forme de $Y$ (cf. Proposition \ref{lem03}), on a ${\rm Ext}^{2}_G(Y,\pi)$ par le Lemme \ref{lem02} et un d\'evissage \'evident. Comme ${\rm Ext}^1_G(X,\pi)=0$, on en d\'eduit le r\'esultat.
\end{proof}

\begin{cor}\label{cor05}
Supposons $1\leq\lg(w_0^\alpha)\leq \lg(w_0)-2$ et la repr\'esentation ${\mathcal F}^G_{B^-}(L^-(\lambda),\chi)$ irr\'eductible, alors on a~:
$$\dim_E{\rm Ext}^1_G\big({\mathcal F}^G_{B^-}(L^-(s_\alpha\cdot \lambda),\chi),{\mathcal F}^G_{B^-}(L^-(\lambda),\chi)\big)=1.$$
En particulier \cite[Conj.~3.3.1]{Br1} est vraie pour ${\GL}_n(\Qp)$.
\end{cor}
\begin{proof}
On a une suite exacte de $E$-espaces vectoriels~:
\begin{multline*}
\Hom_G\big(\!\begin{xy}(-60,0)*+{C(\lambda)}="a";(-42,0)*+{C(s_\alpha\!\cdot \!\lambda)}="b";{\ar@{-}"a";"b"}\end{xy}\!,C(\lambda)\big)\rightarrow \Hom_G(C(\lambda),C(\lambda)) \rightarrow {\rm Ext}^1_G(C(s_\alpha\cdot \lambda),C(\lambda))\\
\rightarrow {\rm Ext}^1_G\big(\!\begin{xy}(-60,0)*+{C(\lambda)}="a";(-42,0)*+{C(s_\alpha\!\cdot \!\lambda)}="b";{\ar@{-}"a";"b"}\end{xy}\!,C(\lambda)\big)
\end{multline*}
o\`u le terme tout \`a gauche est nul car (\ref{extnonsc}) est non scind\'e et $C(\lambda)$ est irr\'eductible, et celui tout \`a droite est nul par le Th\'eor\`eme \ref{ext1nul} appliqu\'e avec $\pi=C(\lambda)$. On en d\'eduit la premi\`ere assertion. La deuxi\`eme suit lorsque $n\geq 3$ (pour avoir $1\leq\lg(w_0^\alpha)\leq \lg(w_0)-2$), mais lorsque $n=2$ elle est montr\'ee dans \cite[Lem.~3.1.1]{Br1}.
\end{proof}

\begin{rem}
{\rm (i) Il est possible que l'hypoth\`ese d'irr\'eductibilit\'e de ${\mathcal F}^G_{B^-}(L^-(\lambda),\chi)$ dans le Corollaire \ref{cor05} soit en fait superflue.\\
(ii) Rappelons que la raison pour laquelle cette preuve de \cite[Conj~3.3.1]{Br1} ne s'\'etend pas \`a ${\GL}_n(L)$ et des extensions localement $\sigma$-analytiques (pour $\sigma:L\hookrightarrow E$) est que les r\'esultats de \cite[\S~6]{ST3} (en particulier \cite[Prop.~6.5(i)]{ST3}) utilis\'es de mani\`ere essentielle dans la preuve de \cite[(4.17)]{Sc1}, et donc dans celle du Lemme \ref{lem01}, ne sont montr\'es que dans le cadre localement $\Qp$-analytique.}
\end{rem} 

\subsection{Foncteurs $F_\alpha$ et compatibilit\'e local-global conjecturale}\label{locglob}

Apr\`es un rappel du cadre global, on donne une version plus pr\'ecise de la conjecture principale de \cite{Br1} utilisant les foncteurs $F_\alpha$.

On conserve les notations du \S~\ref{hodge}. On rappelle d'abord le cadre global de \cite[\S~6.1]{BD} (cf. aussi \cite[\S~6.1]{Br1}) auquel on renvoie le lecteur pour plus de d\'etails. On fixe une fois pour toutes des plongements $\overline{\Q} \hookrightarrow \C$ et $\overline{\Q} \hookrightarrow \overline{\Qp}$ que l'on utilise de mani\`ere tacite. On fixe un corps de nombres totalement r\'eel $F^+$, une extension quadratique totalement imaginaire $F$ de $F^+$ et un groupe unitaire $G/F^+$ attach\'e \`a $F/F^+$ comme dans \cite[\S~6.2.2]{BChe} tel que $G\times_{F^+} F\cong {\rm GL}_n$ ($n\geq 2$) et $G(F^+\otimes_{\Q} {\mathbb R})$ est compact. Pour une place finie $v$ de $F^+$ totalement d\'ecompos\'ee dans $F$ et $\tilde{v}$ une place de $F$ divisant $v$, on a des isomorphismes $i_{G, \widetilde{v}}: G(F^+_v)\buildrel\sim\over \rightarrow G(F_{\tilde{v}})\simeq {\rm GL}_n(F_{\tilde{v}})$. On note $\Sigma_p$ les places de $F^+$ divisant $p$ et on suppose que chaque place dans $\Sigma_p$ est totalement d\'ecompos\'ee dans $F$. Pour toute extension finie $E$ de $\Qp$ et tout sous-groupe ouvert compact $U^p=\prod_{v\nmid p} U_v$ de $G({\mathbb A}_{F^+}^{p,\infty})$ (o\`u ${\mathbb A}_{F^+}^{p,\infty}$ d\'esigne les ad\`eles finis hors $p$ de $F^+$) on pose~:
\begin{equation}\label{supe}
\widehat{S}(U^p,E)\=\Big\{f: G(F^+) \setminus G({\mathbb A}_{F^+}^{\infty})/U^p\longrightarrow E,\ f \text{ continue}\Big\}
\end{equation}
qui est un espace de Banach $p$-adique sur $E$ de boule unit\'e~:
\begin{equation*}
\widehat{S}(U^p,\oE)\=\Big\{f: G(F^+)\setminus G({\mathbb A}_{F^+}^{\infty})/U^p \longrightarrow \oE,\ f \text{ continue}\Big\}.\end{equation*}
On munit $\widehat{S}(U^p,E)$ de l'action continue de $G(F^+\otimes_{\Q}\Qp)$ donn\'ee par $(g'f)(g)=f(gg')$ (pour $f\in \widehat{S}(U^p,E)$, $g'\in G(F^+\otimes_{\Q} \Qp)$, $g\in G({\mathbb A}_{F^+}^{\infty})$), qui pr\'eserve la boule unit\'e et fait de $\widehat{S}(U^p,E)$ un Banach unitaire admissible. 

Soit $D(U^p)$ les places finies $v$ de $F^+$ telles que $v\nmid p$, $v$ est totalement d\'ecompos\'e dans $F$ et $U_v$ est un sous-groupe ouvert compact maximal de $G(F^+_v)$. On note ${\mathbb T}(U^p)\=\oE[T_{\tilde{v}}^{(j)}]$ la $\oE$-alg\`ebre polynomiale commutative engendr\'ee par des variables formelles $T_{\tilde{v}}^{(j)}$ o\`u $j\in \{1,\cdots,n\}$ et $\tilde{v}$ est une place de $F$ au-dessus d'une place $v$ de $D(U^p)$. La $\oE$-alg\`ebre ${\mathbb T}(U^p)$ agit sur $\widehat{S}(U^p,E)$ et $\widehat{S}(U^p, \oE)$ en faisant agir $T_{\tilde{v}}^{(j)}$ par l'op\'erateur associ\'e \`a la double classe~:
\begin{equation}\label{equ: ord-hecke1}
 T_{\tilde{v}}^{(j)}\=\Big[U_{v} g_v i_{G,\tilde{v}}^{-1}\Big(\begin{pmatrix}
 \text{\textbf{1}}_{n-j} & 0 \\ 0 & \text{$\varpi_{\tilde{v}}$ \textbf{1}}_{j}
 \end{pmatrix}\Big) g_v^{-1}U_{v}\Big]
\end{equation}
o\`u $\varpi_{\tilde{v}}$ est une uniformisante de $F_{\tilde{v}}$ et $g_v\in G(F_v^+)$ est tel que $i_{G,\tilde{v}}(g_v^{-1} U_v g_v)={\rm GL}_n({\mathcal O}_{F_{\tilde{v}}})$. Cette action commute avec celle de $G(F^+ \otimes_{\Q} \Qp)$. On note $\widehat{S}(U^p,E)^{\alg}\subset \widehat{S}(U^p,E)^{\an}$ les sous-espaces de $\widehat{S}(U^p,E)$ des vecteurs respectivement localement alg\'ebriques et localement analytiques pour l'action de $(\Res_{F^+/\Q}G)(\Qp)=G(F^+\otimes_{\Q} \Qp)$. Ils sont stables par $G(F^+\otimes_{\Q} \Qp)$ et ${\mathbb T}(U^p)$. Le sous-espace $\widehat{S}(U^p,E)^{\an}$ se d\'ecrit comme (\ref{supe}) en rempla\c cant ``continue'' par ``localement analytique'' et on a un isomorphisme $G(F^+\otimes_{\Q} \Qp)\times {\mathbb T}(U^p)$-\'equivariant (cf. \cite[(6.2)]{BD})~:
\begin{equation}\label{equ: lalgaut}
 \widehat{S}(U^p,E)^{\alg} \otimes_E \overline{\Qp} \cong \bigoplus_{\pi}\Big((\pi^{\infty,p})^{U^p} \otimes_{\overline{\Q}}(\pi_p \otimes_{\overline{\Q}} W_p)\Big)^{\oplus m(\pi)}
\end{equation}
o\`u $\pi\cong \pi_{\infty}\otimes_{\overline{\Q}} \pi^{\infty}$ parcourt les repr\'esentations automorphes de $G({\mathbb A}_{F^+})$, $W_p$ est une repr\'esentation alg\'ebrique irr\'eductible de $(\Res_{F^+/\Q}G)(\Qp)$ sur $\overline{\Qp}$ d\'etermin\'ee par $\pi_\infty$ (voir \cite[\S~6.2.3]{BChe} et \cite[\S~6.1]{BD}), $m(\pi)\geq 1$ est une multiplicit\'e, et o\`u $T_{\tilde{v}}^{(j)}\in {\mathbb T}(U^p)$ agit sur $(\pi^{\infty,p})^{U^p}$ par la double classe (\ref{equ: ord-hecke1}).

On fixe d\'esormais une place $\wp$ de $F^+$ au-dessus de $p$ telle que $F^+_\wp=\Qp$ et une place $\widetilde{\wp}$ de $F$ divisant $\wp$ (donc $\Qp=F^+_\wp\cong F_{\widetilde{\wp}}$ et on a un isomorphisme $i_{G,\widetilde{\wp}}: G(F^+_\wp)\xrightarrow{\sim} {\rm GL}_n(\Qp)$). On fixe \'egalement une repr\'esentation $\Qp$-alg\'ebrique irr\'eductible $W^\wp$ de $\prod_{v|p, v\neq \wp} G(F^+_v)$ sur $E$, que l'on voit comme repr\'esentation de $G(F^+\otimes_{\Q} \Qp)$ via $G(F^+\otimes_{\Q} \Qp)\twoheadrightarrow \prod_{v|p, v\neq \wp} G(F^+_v)$, ainsi qu'un sous-groupe ouvert compact $U_p^\wp=\prod_{v|p,v\neq \wp} U_v$ de $\prod_{v|p, v\neq \wp}G(F^+_v)$. On pose $U^\wp\=U^pU_p^\wp$ et~:
$$\widehat{S}(U^\wp,W^\wp)\=\big(\widehat{S}(U^p,E)\otimes_E W^\wp\big)^{U_p^\wp}.$$
Muni de l'action naturelle de $G(F^+_\wp)$ induite par celle sur $\widehat{S}(U^p,E)$, l'espace $\widehat{S}(U^\wp,W^\wp)$ est encore un Banach unitaire admissible. On le munit aussi de l'action de ${\mathbb T}(U^p)$ induite par celle sur $\widehat{S}(U^p,E)$, qui commute \`a $G(F^+_\wp)$, et on d\'efinit les sous-espaces $\widehat{S}(U^\wp,W^\wp)^{\alg}\subset \widehat{S}(U^\wp,W^\wp)^{\an}$ (avec des notations \'evidentes) qui sont stables par $G(F^+_\wp)\times {\mathbb T}(U^p)$. On peut v\'erifier que~:
\begin{multline}\label{equ: ord-aut}
\widehat{S}(U^{\wp},W^{\wp})=\Big\{f: G(F^+) \setminus G({\mathbb A}_{F^+}^{\infty})/U^p \longrightarrow W^{\wp},\ \ f \text{ continue et} \\ f(gg_p^{\wp})=(g_p^{\wp})^{-1} (f(g)) \ \text{pour tout}\ g\in G({\mathbb A}_{F^+}^{\infty}) \ \text{et tout} \ g_p^{\wp}\in U_p^{\wp}\Big\}
\end{multline}
o\`u l'action de $G(F^+_\wp)$ est par translation \`a droite sur les fonctions et celle de ${\mathbb T}(U^p)$ comme en (\ref{equ: ord-hecke1}) ($\widehat{S}(U^\wp,W^\wp)^{\an}$ s'identifie alors au sous-espace de (\ref{equ: ord-aut}) des fonctions $f$ localement analytiques), et que~:
\begin{equation}\label{equ: lalgaut2}
\widehat{S}(U^{\wp},W^{\wp})^{\alg}\otimes_{E} \overline{\Qp} \cong \bigoplus_{\pi} \big( (\pi^{\infty,p})^{U^p} \otimes_{\overline{\Q}} (\otimes_{v|p, v\neq {\wp}} \pi_v^{U_v})\otimes_{\overline{\Q}} (\pi_{\wp} \otimes_{\overline{\Q}} W_{\wp})\big)^{m(\pi)}
\end{equation}
o\`u $\pi\cong \pi_{\infty} \otimes_{{\mathbb C}} \pi^{\infty}$ parcourt les repr\'esentations automorphes de $G({\mathbb A}_{F_+})$ telles que $W_p$ en (\ref{equ: lalgaut}) est de la forme $W_p\cong W_{\wp}\otimes_{E}(W^{\wp})^{\vee}$ en notant $(W^{\wp})^{\vee}$ le dual de $W^{\wp}$ et $W_{\wp}$ une repr\'esentation alg\'ebrique de $G(F^+_\wp)$ sur $\overline{\Qp}$. On voit dans la suite $\widehat{S}(U^{\wp},W^{\wp})$ et ses sous-repr\'esentations comme repr\'esentations de ${\rm GL}_n(\Qp)$ via $\iota_{G,{\widetilde{\wp}}}$.

Soit ${\rho}: {\rm Gal}(\overline F/F)\rightarrow {\rm GL}_n(E)$ une repr\'esentation continue non ramifi\'ee aux places dans $D(U^p)$. On associe \`a ${\rho}$ l'unique id\'eal maximal $\mrho$ de corps r\'esiduel $E$ de ${\mathbb T}(U^p)[1/p]$ tel que, pour tout $v\in D(U^p)$ et $\tilde{v}$ place de $F$ au-dessus de $v$, le polyn\^ome caract\'eristique ${\rho}(\Frob_{\tilde{v}})$ (o\`u $\Frob_{\tilde{v}}$ est un Frobenius {\it geom\'etrique} en $\tilde v$) est donn\'e par~:
\begin{equation}\label{equ: ord-ideal}
X^n+\cdots + (-1)^j (N\tilde{v})^{\frac{j(j-1)}{2}} \theta_{{\rho}}(T_{\tilde{v}}^{(j)}) X^{n-j}+\cdots +(-1)^n(N\tilde{v})^{\frac{n(n-1)}{2}} \theta_{{\rho}}(T_{\tilde{v}}^{(n)})
\end{equation}
en notant $N\tilde{v}$ le cardinal du corps r\'esiduel de $F_{\tilde v}$ et $\theta_{{\rho}}$ la compos\'ee ${\mathbb T}(U^p)[1/p]\twoheadrightarrow {\mathbb T}(U^p)[1/p]/\mrho\buildrel\sim\over\rightarrow E$. Lorsque $\widehat{S}(U^{\wp}, W^{\wp})[\mrho]^{\alg}\ne 0$ on sait de plus que $\rho_{\widetilde{\wp}}\=\rho\vert_{{\rm Gal}(\overline F_{\widetilde{\wp}}/F_{\widetilde{\wp}})}$ est potentiellement semi-stable \`a poids de Hodge-Tate distincts.

On normalise la correspondance de Langlands locale pour ${\rm GL}_n(\Qp)$ de sorte que la repr\'esentation $\chi_1\oplus \cdots \oplus \chi_n$ du groupe de Weil de $\Qp$ corresponde (pour des $\chi_i$ g\'en\'eriques) \`a~:
\begin{equation}\label{norma}
\big(\Ind_{B^-(\Qp)}^{{\rm GL}_n(\Qp)}\chi_1|\cdot|^{1-n}\otimes\chi_2|\cdot|^{2-n}\otimes\cdots \otimes \chi_n\big)^\infty
\end{equation}
o\`u les $\chi_i$ dans l'induite sont vus comme caract\`eres de $\Qp^\times$ via la normalisation du corps de classe local au \S~\ref{intro} et o\`u (comme dans l'Exemple \ref{gln}) $B^-$ est le Borel inf\'erieur.

On fixe un module de Deligne-Fontaine $D$ de rang $n$ sur $L_0\otimes_{\Qp}E$ (pour une extension finie $L$ de $\Qp$, cf. \S~\ref{hodge}). On note $W$ la repr\'esentation de Weil-Deligne associ\'ee par Fontaine \`a $D$ (cf. \cite[\S~2.2]{Br1}) et l'on suppose que la repr\'esentation de Weil sous-jacente \`a $W$ est absolument semi-simple, i.e. $W$ est $F$-semi-simple. On note $\pi^{\infty}(W)$ la repr\'esentation lisse admissible de longueur finie de ${\rm GL}_n(\Qp)$ sur $E$ correspondant \`a $W$ par la correspondance de Langlands locale normalis\'ee comme ci-dessus {\it puis} modifi\'ee comme dans \cite[\S~4]{BS}. La repr\'esentation $\pi^{\infty}(W)$ a un unique constituant irr\'eductible g\'en\'erique qui est en sous-objet, et si $\chi_{\pi^{\infty}(W)}$ est son caract\`ere central alors $\chi_{\pi^{\infty}(W)}=\chi_W\norm^{-\frac{n(n-1)}{2}}$ o\`u $\chi_W\={\det}(W)$.

On fixe \'egalement $(h_{i})_{i\in \{1,\dots,n\}}$ dans $\Z^n$ tel que $h_{1}>h_{2}>\cdots >h_{n}$ et on pose $\lambda=(\lambda_{i})_{i\in \{1,\dots,n\}}$ avec $\lambda_{i}\=(n-i)-h_i$, de sorte que $\lambda$ est dominant par rapport \`a $B^-$. On rappelle que $L(-\lambda)$ est la repr\'esentation alg\'ebrique irr\'eductible de ${\rm GL}_n$ sur $E$ de plus haut poids $-\lambda$ par rapport \`a $B$ (cf. \S~\ref{notabene}). On dispose donc de la repr\'esentation localement alg\'ebrique $L(-\lambda)\otimes_E\pi^{\infty}(W)$. Rappelons que, si $\alpha\in S$ et si l'on a une extension $0\rightarrow \pi'\rightarrow \pi \rightarrow L(-\lambda)\otimes_E\pi^{\infty}(W)\rightarrow 0$ dans $\Rep({\rm GL}_n(\Qp))$ avec $F_\alpha(\pi')(-)\simeq E_\infty(\chi_{-\lambda})\otimes_{E}\Hom_{(\varphi,\Gamma)}(D_\alpha(\pi'),-)$ pour un $(\varphi,\Gamma)$-module sans torsion $D_{\alpha}(\pi')$, alors par le Th\'eor\`eme \ref{localgenplus} (avec la Remarque \ref{pourapplication}) et le Th\'eor\`eme \ref{caslisse} on a $F_\alpha(\pi)(-)\simeq E_\infty(\chi_{-\lambda})\otimes_{E}\Hom_{(\varphi,\Gamma)}(D_\alpha(\pi),-)$ pour un $(\varphi,\Gamma)$-module $D_{\alpha}(\pi)$ qui est une extension de $D_{\alpha}(L(-\lambda)\otimes_E\pi^{\infty}(W))\simeq \R(\lambda\circ \lambda_{\alpha^{\!\vee}})/(t^{1-\langle \lambda,\alpha^\vee\rangle})$ par $D_{\alpha}(\pi')$. On a donc une application $E$-lin\'eaire~:
\begin{equation}\small\label{mapext}
{\rm Ext}^1_{{\rm GL}_n(\Qp)}\big(\pi',L(-\lambda)\otimes_E\pi^{\infty}(W)\big)\longrightarrow {\rm Ext}^1_{(\varphi,\Gamma)}\big(D_{\alpha}(L(-\lambda)\otimes_E\pi^{\infty}(W)),D_{\alpha}(\pi')\big).
\end{equation}

Pour toute racine simple $\alpha=e_j-e_{j+1}\in S$ o\`u $j\in \{1,\dots,n-1\}$ on pose~:
\begin{equation}\label{dalphaw}
{D}_\alpha(\lambda,W)\=(\wedge^{n-j}_{\R}{\mathcal D})\otimes_{\R}\R((s_\alpha\cdot \lambda)\circ \lambda_{\alpha^{\!\vee}})\otimes_{\R}\R(\norm^{n-j+(n-j+1)+\cdots+n-1}\chi_W^{-1})
\end{equation}
o\`u $\mathcal D$ est le $(\varphi,\Gamma)$-module sur $\R$ donn\'e par l'\'equation diff\'erentielle $p$-adique associ\'ee \`a $D$ (cf. \S~\ref{hodge}). Alors ${D}_\alpha(\lambda,W)$ est, \`a la torsion localement constante $\norm^{n-j+\cdots+n-1}\chi_W^{-1}$ pr\`es, le $(\varphi,\Gamma)$-module libre de rang $n\choose j$ sur $\R$ associ\'e par l'\'equivalence de \cite[Th.~A]{Be2} au module de Deligne-Fontaine $\wedge_{L_0\otimes_{\Qp}E}^{n-j}D$ muni de la filtration~:
$$\left\{\begin{array}{ccl}
{\rm Fil}^{-(\lambda_1+\cdots+\lambda_{j-1}+\lambda_{j+1}+1)}(\wedge_{L\otimes_{\Qp}E}^{n-j}D_L)&=&\wedge_{L\otimes_{\Qp}E}^{n-j}D_L\\
{\rm Fil}^{-(\lambda_1+\cdots+\lambda_{j-1}+\lambda_{j+1}+1)+1}(\wedge_{L\otimes_{\Qp}E}^{n-j}D_L)&=&0.
\end{array}\right.$$
De plus par (\ref{versfilmax}) (appliqu\'e avec $h_1=\lambda_1+\cdots+\lambda_{j-1}+\lambda_{j+1}+1$, $h_2=\lambda_1+\cdots+\lambda_{j-1}+\lambda_{j}$ et $D=\wedge_{L_0\otimes_{\Qp}E}^{n-j}D$) avec la Remarque \ref{torsion} (appliqu\'ee avec $\delta=\norm^{n-j+\cdots+n-1}\chi_W^{-1}$), on d\'eduit un isomorphisme naturel~:
\begin{equation}\label{mapfil}
{\mathcal F}_\alpha:{\rm Ext}^1_{(\varphi,\Gamma)}\big(\R(\lambda\circ \lambda_{\alpha^{\!\vee}})/(t^{1-\langle \lambda,\alpha^\vee\rangle}),D_\alpha(\lambda,W)\big)\buildrel\sim\over\longrightarrow (\wedge_{L\otimes_{\Qp}E}^{n-j}D_L)^{\glp}.
\end{equation}

Soient maintenant $U^\wp$, $W^\wp$ et ${\rho}: {\rm Gal}(\overline F/F)\rightarrow {\rm GL}_n(E)$ comme ci-dessus tels que~:
\begin{itemize}
\item[(i)]$\rho$ est absolument irr\'eductible non ramifi\'ee aux places de $F$ au-dessus de $D(U^p)$;\\
\item[(ii)]$\widehat S(U^\wp,W^\wp)^{\alg}[\mrho]\ne 0$ (donc $\rho_{\widetilde{\wp}}$ est potentiellement semi-stable);\\
\item[(iii)]$\rho_{\widetilde{\wp}}$ a pour poids de Hodge-Tate $h_1,\dots,h_n$ et module de Deligne-Fontaine $D$
\end{itemize}
(en fait $W^\wp$ est d\'etermin\'e par $\rho$). Il r\'esulte alors des normalisations, de (\ref{equ: lalgaut2}) et de \cite{Ca} que l'on a pour tout $U^\wp$, $W^\wp$, $\rho$ v\'erifiant (i), (ii) et (iii)~:
\begin{equation}\label{localgrho}
\widehat S(U^\wp,W^\wp)^{\alg}[\mrho]\cong \big(L(-\lambda)\otimes_E\pi^{\infty}(W)\big)^{\oplus d(U^\wp,\rho)}\otimes_E \varepsilon^{n-1}\circ{\det}
\end{equation}
o\`u $d(U^\wp,\rho)\geq 1$ est un entier ne d\'ependant que de $U^\wp$ et $\rho$. La filtration $\Fil^\cdot$ sur $(B_{\rm dR}\otimes_{\Qp}\rho_{\widetilde{\wp}})^{{\Gal}(\Qpbar/L)}\simeq D_L\simeq L\otimes_{\Qp}D_L^{\glp}$ permet pour $\alpha=e_j-e_{j+1}$ de d\'efinir la droite~:
\begin{equation}\footnotesize\label{filmax}
\Fil_\alpha^{\rm max}(\rho_{\widetilde{\wp}})\=\Fil^{-h_{j+1}}(D_L)^{\glp}\wedge \Fil^{-h_{j+2}}(D_L)^{\glp}\wedge \cdots \wedge \Fil^{-h_{n}}(D_L)^{\glp}
\end{equation}
dans $\wedge_{E}^{n-j}(D_L^{\glp})\simeq (\wedge_{L\otimes_{\Qp}E}^{n-j}D_L)^{\glp}$.

Enfin, si $\pi'$, $\pi''$ sont des repr\'esentations admissibles dans $\Rep({\rm GL}_n(\Qp))$ et si $H$ est un sous-$E$-espace vectoriel de dimension $\leq 1$ du $E$-espace vectoriel $\Ext^1_{{\rm GL}_n(\Qp)}(\pi'',\pi')$ des extensions de $\pi''$ par $\pi'$ (dans la cat\'egorie ab\'elienne des repr\'esentations admissibles de ${\rm GL}_n(\Qp)$), on note $[\![H]\!]$ l'unique repr\'esentation de ${\rm GL}_n(\Qp)$ sur $E$ sous-jacente \`a $H$ au sens de Yoneda. Si $\pi_1,\dots,\pi_{m}$ sont admissibles dans $\Rep({\rm GL}_n(\Qp))$ avec $\pi$ comme sous-repr\'esentation commune, on note $\oplus_\pi^j\pi_j\=\pi_1\oplus_{\pi}\pi_2\oplus_{\pi}\cdots\oplus_\pi\pi_{m}$.

La conjecture suivante est une version nettement plus pr\'ecise (et un peu plus g\'en\'erale) de \cite[Conj.~6.1.1]{Br1} lorsque $F^+_{v_0}=\Qp$, qu'elle implique (sauf l'unicit\'e de ${\mathcal R}^j$ dans {\it loc.cit.} dont on ne se pr\'eoccupe pas ici).

\begin{conj}\label{extglobmieux}
Pour toute racine simple $\alpha=e_j-e_{j+1}$ il existe une repr\'esentation localement analytique admissible de longueur finie $\pi^\alpha(\lambda,W)$ de ${\rm GL}_n(\Qp)$ sur $E$ ne d\'ependant que de $\lambda$, $W$ et $\alpha$ et v\'erifiant les propri\'et\'es suivantes~:

\noindent
(i) $F_\alpha(\pi^\alpha(\lambda,W))\simeq E_\infty(\chi_{-\lambda})\!\otimes_{E}\!\Hom_{(\varphi,\Gamma)}\!\big(D_\alpha(\lambda,W),-\big)$ (cf. (\ref{dalphaw}) pour $D_\alpha(\lambda,W)$);

\noindent
(ii) l'application (\ref{mapext}) pour $\pi'=\pi^\alpha(\lambda,W)$ est un isomorphisme~:
\begin{multline*}
{\mathcal E}_\alpha:{\rm Ext}^1_{{\rm GL}_n(\Qp)}\big(\pi^\alpha(\lambda,W),L(-\lambda)\otimes_E\pi^{\infty}(W)\big)\\\buildrel\sim\over\longrightarrow {\rm Ext}^1_{(\varphi,\Gamma)}\big(\R(\lambda\circ \lambda_{\alpha^{\!\vee}})/(t^{1-\langle \lambda,\alpha^\vee\rangle}),D_\alpha(\lambda,W)\big);
\end{multline*}
(iii) pour $U^\wp$, $W^\wp$, $\rho$ v\'erifiant (i), (ii), (iii) ci-dessus, on a une injection ${\rm GL}_n(\Qp)$-\'equivariante (cf. (\ref{mapfil}) pour ${\mathcal F}^\alpha$, (\ref{filmax}) pour $\Fil^{\rm max}_\alpha(\rho_{\widetilde{\wp}})$ et (\ref{localgrho}) pour $d(U^\wp,\rho)$)~:
\begin{equation*}
\bigg(\bigoplus^{\alpha}_{L(-\lambda)\otimes_E\pi^{\infty}(W)}[\![({\mathcal F}_\alpha\circ {\mathcal E}_\alpha)^{-1}(\Fil^{\rm max}_\alpha(\rho_{\widetilde{\wp}}))]\!] \bigg)^{\oplus d(U^\wp,\rho)}\!\!\otimes \varepsilon^{n-1}\circ{\det}\hookrightarrow \widehat S(U^\wp,W^\wp)^{\an}[\mrho].
\end{equation*}
\end{conj}

\begin{rem}
{\rm (i) Les repr\'esentations $\pi^{e_i-e_{i+1}}(-)$ sont not\'ees $\Pi^{n-i}(-)$ dans \cite{Br1}.\\
(ii) En g\'en\'eral il devrait exister plus d'une repr\'esentation $\pi^\alpha(\lambda,W)$ satisfaisant toutes les conditions de la Conjecture \ref{extglobmieux} (par exemple, si le foncteur $F_\alpha$ annule des constituants en cosocle de $\pi^\alpha(\lambda,W)$, on peut les enlever de $\pi^\alpha(\lambda,W)$), mais il est possible qu'il en existe une maximale pour l'inclusion.\\
(iii) Si les constituants de $\pi^{\infty}(W)$ sont des sous-quotients de s\'eries principales, on devrait pouvoir prendre $\pi^\alpha(\lambda,W)$ dans la cat\'egorie $C_{\lambda,\alpha}$ du \S~\ref{prelgen}.\\
(iv) Pour $\lambda=0$, les auteurs aiment \`a penser les repr\'esentations localement ana\-lytiques $\pi^\alpha(W)\=\pi^\alpha(0,W)$ comme une ``correspondance de Langlands locale (conjecturale) pour chaque racine simple $\alpha$''.}
\end{rem}

\subsection{Cas partiels ou particuliers}\label{carparticuliers}

On \'enonce et d\'emontre divers cas partiels de la Conjecture \ref{extglobmieux}, notamment lorsque $\dim_EW=3$ et $N^2\ne 0$ sur $D$.

Commen\c cons par le cas $n=2$. On dispose alors de la correspondance de Langlands localement analytique $\rho_p\mapsto \pi^{\an}(\rho_p)$ pour ${\rm GL}_2(\Qp)$ (normalis\'ee de sorte que le caract\`ere central de $\pi^{\an}(\rho_p)$ est $\varepsilon^{-1}{\det}(\rho_p)$), cf. \cite{Co2}, \cite{CDP}, \cite{CD}. Via les r\'esultats de compatibilit\'e local-global sur cette correspondance (cf. \cite{Em3}, \cite{CS1}, \cite{CS2}), on peut ici remplacer la Conjecture \ref{extglobmieux} par la conjecture purement locale suivante (o\`u $\alpha$ est l'unique racine de ${\rm GL}_2$ et $D$, $W$, $h_1,h_2$ et $\lambda$ sont comme au \S~\ref{locglob}).

\begin{conj}\label{extglobmieuxn=2}
Il existe une repr\'esentation localement analytique admissible de longueur finie $\pi^\alpha(\lambda,W)$ de ${\rm GL}_2(\Qp)$ sur $E$ ne d\'ependant que de $\lambda$, $W$ et v\'erifiant les propri\'et\'es suivantes~:

\noindent
(i) $F_\alpha(\pi^\alpha(\lambda,W))\simeq E_\infty\otimes_E\Hom_{(\varphi,\Gamma)}\!\big(D_\alpha(\lambda,W),-\big)$;

\noindent
(ii) l'application (\ref{mapext}) pour $\pi'=\pi^\alpha(\lambda,W)$ est un isomorphisme~:
\begin{multline*}
{\mathcal E}_\alpha:{\rm Ext}^1_{{\rm GL}_n(\Qp)}\big(\pi^\alpha(\lambda,W),L(-\lambda)\otimes_E\pi^{\infty}(W)\big)\\\buildrel\sim\over\longrightarrow {\rm Ext}^1_{(\varphi,\Gamma)}\big(\R(\lambda\circ \lambda_{\alpha^{\!\vee}})/(t^{1-\langle \lambda,\alpha^\vee\rangle}),D_\alpha(\lambda,W)\big);
\end{multline*}
(iii) pour toute $\rho_p:{\rm Gal}(\Qpbar/\Qp)\rightarrow {\rm GL}_2(E)$ potentiellement semi-stable de poids de Hodge-Tate $h_1,h_2$ et module de Deligne-Fontaine $D$, on a $\pi^{\an}(\rho_p)\simeq [\![({\mathcal F}_\alpha\circ {\mathcal E}_\alpha)^{-1}(\Fil^{-h_2}(\rho_p))]\!]$ o\`u $\Fil^\cdot (\rho_p)$ est la filtration de Hodge sur $D_L$ donn\'ee par $\rho_p$.
\end{conj}

En tenant compte des normalisations, on peut v\'erifier que, si (i) et (ii) sont vrais, alors (iii) est \'equivalent \`a~:
\begin{eqnarray*}
F_\alpha(\pi^{\an}(\rho_p))&\simeq &E_\infty\otimes_E\Hom_{(\varphi,\Gamma)}\big(D_{\rm rig}(\rho_p)({\det}(\rho_p)^{-1}\varepsilon),-\big)\\
&\simeq &E_\infty\otimes_E\Hom_{(\varphi,\Gamma)}(D_{\rm rig}(\check \rho_p),-)
\end{eqnarray*}
o\`u $\check\rho_p\= \rho_p^\vee\otimes\varepsilon\simeq \rho_p\otimes {\det}(\rho_p)^{-1}\varepsilon$ est le dual de Cartier de $\rho_p$ et $D_{\rm rig}(\rho_p)$, $D_{\rm rig}(\check \rho_p)$ les $(\varphi,\Gamma)$-modules (\'etales) sur $\R$ de (respectivement) $\rho_p$, $\check\rho_p$. Rappelons \'egalement que $\pi^{\an}(\rho_p)/(L(-\lambda)\otimes_E\pi^{\infty}(W))$ ne d\'epend que de $\lambda$ et $W$ par \cite[Th.~VI.6.43]{Co2}. Le th\'eor\`eme suivant r\'esume ce que l'on sait montrer de la Conjecture \ref{extglobmieuxn=2}.

\begin{thm}\label{n=2}
(i) La Conjecture \ref{extglobmieuxn=2} est vraie lorsque $W$ est r\'eductible.\\
(ii) Si $W$ est irr\'eductible, pour tout $\rho_p$ de poids de Hodge-Tate $h_1,h_2$ et module de Deligne-Fontaine $D$ on a un morphisme surjectif $\Rr\otimes_{\R^+}\pi^{\an}(\rho_p)^\vee \twoheadrightarrow D_{\rm rig}(\check\rho_p)_r$ de $(\psi,\Gamma)$-modules de Fr\'echet pour $r\gg 0$. De plus, si le (i) de la Conjecture \ref{extglobmieuxn=2} est vrai avec $\pi^\alpha(\lambda,W)=\pi^{\an}(\rho_p)/(L(-\lambda)\otimes_E\pi^{\infty}(W))$, alors on a $F_\alpha(\pi^{\an}(\rho_p))\simeq E_\infty\otimes_E\Hom_{(\varphi,\Gamma)}(D_{\rm rig}(\check \rho_p),-)$.
\end{thm}
\begin{proof}
\noindent
{\bf \'Etape $1$}\\
Soit $\rho_p:{\rm Gal}(\Qpbar/\Qp)\rightarrow {\rm GL}_2(E)$ potentiellement semi-stable de poids de Hodge-Tate $h_1,h_2$ et module de Deligne-Fontaine $D$. Pour $r\gg 0$, on a $D_{\rm rig}(\check\rho_p)=\R\otimes_{\Rr}D_{\rm rig}(\check\rho_p)_r$. On montre d'abord (sans hypoth\`ese sur $W$) que l'on a une surjection $\Rr\otimes_{\R^+}\pi^{\an}(\rho_p)^\vee\twoheadrightarrow D_{\rm rig}(\check\rho_p)_r$ de $(\psi,\Gamma)$-modules de Fr\'echet quitte \'eventuellement \`a augmenter $r$. La compos\'ee (cf. \cite[\S~VI]{CD} pour les d\'efinitions, en particulier \cite[Cor.~VI.12~\&~Cor.~VI.13]{CD})~:
$$\pi^{\an}(\rho_p)^\vee\hookrightarrow D_{\rm rig}(\check\rho_p)_r\boxtimes {\mathbb P}^1(\Qp)\buildrel{\rm res}_{\Zp}\over\longrightarrow D_{\rm rig}(\check\rho_p)_r\boxtimes \Zp\simeq D_{\rm rig}(\check\rho_p)_r$$
donne un morphisme $\pi^{\an}(\rho_p)^\vee\rightarrow D_{\rm rig}(\check\rho_p)_r$ de $(\psi,\Gamma)$-modules de Fr\'echet sur $\R^+$. Soit $\pi(\rho_p)$ le Banach $p$-adique unitaire associ\'ee \`a $\rho_p$, ou alternativement par \cite[Th.~0.2]{CD} le compl\'et\'e unitaire universel de $\pi^{\an}(\rho_p)$, $D(\check\rho_p)_0$ le $(\varphi,\Gamma)$-module {\it continu} de $\rho_p$ (i.e. sur ${\mathcal E}=\oE[[X]][1/X]^\wedge[1/p]$) et $D(\check\rho_p)_0^\natural$ le sous-$\oE[[X]][1/p]$-module de $D(\check\rho_p)_0$ stable par $\psi$ et $\Gamma$ d\'efini dans \cite[\S~II.5]{Co}. Alors il r\'esulte de \cite[Cor.~III.22(ii)]{CD} et \cite[Prop.~III.23]{CD} (avec \cite[Cor.~II.7.2]{Co}) que l'image de la compos\'ee $\pi(\rho_p)^\vee\hookrightarrow \pi^{\an}(\rho_p)^\vee\rightarrow D_{\rm rig}(\check\rho_p)_r$ contient $D(\check\rho_p)_0^\natural$. Comme $D(\check\rho_p)_0^\natural$, vu dans $D_{\rm rig}(\check\rho_p)_r$, contient une base de $D_{\rm rig}(\check\rho_p)_r$ sur $\Rr$ (ce qui se d\'eduit par extension des scalaires de \cite[Cor.~II.7.2]{Co} et du fait que $D(\check\rho_p)_0^\natural$ contient une base de $D(\check\rho_p)_0$ sur $\mathcal E$), on en d\'eduit en particulier que l'application obtenue par extension des scalaires $\Rr\otimes_{\R^+}\pi^{\an}(\rho_p)^\vee\twoheadrightarrow D_{\rm rig}(\check\rho_p)_r$ est surjective (merci \`a G. Dospinescu pour les r\'ef\'erences et pour cet argument !).

\noindent
{\bf \'Etape $2$}\\
Supposons $W$ r\'eductible avec $N=0$ sur $D$. Supposons d'abord $W$ non scalaire. Alors comme dans \cite[\S~3.1]{Br1} la repr\'esentation $\pi^\alpha(\lambda,W)$ est la somme directe de deux s\'eries principales localement analytiques distinctes et il r\'esulte facilement du Th\'eor\`eme \ref{plusgeneral} et de (l'analogue de) \cite[Lem.~3.1.1]{Br1} avec le (ii) de la Proposition \ref{lienfilmax} que les (i) et (ii) de la Conjecture \ref{extglobmieuxn=2} sont vrais dans ce cas (l'application ${\mathcal E}_\alpha$ est en effet une injection entre $E$-espaces vectoriels de dimension $2$, donc une bijection). Soit $\rho_p$ comme dans l'\'Etape $1$, alors par \cite{Li2}, \cite{LXZ} et ce qui pr\'ec\`ede, le foncteur $F_\alpha(\pi^{\an}(\rho_p))$ est repr\'esentable par un $(\varphi,\Gamma)$-module libre de rang $2$ $D_\alpha(\pi^{\an}(\rho_p))$ extension de deux $(\varphi,\Gamma)$-modules libres de rang $1$ (en \'etendant les scalaires \`a $E_\infty$). De plus, par l'\'Etape $1$ et la d\'efinition du foncteur $F_\alpha$, on obtient un morphisme surjectif $D_\alpha(\pi^{\an}(\rho_p))\twoheadrightarrow D_{\rm rig}(\check\rho_p)$ de $(\varphi,\Gamma)$-modules sur $\R$, donc un isomorphisme puisque les deux sont des modules libres de rang $2$ sur $\R$. Cela ach\`eve la preuve du cas $W$ r\'eductible non scalaire avec $N=0$. Supposons maintenant $W$ scalaire (et $N=0$), alors les deux s\'eries principales localement analytiques du d\'ebut sont les m\^emes. En prenant pour $\pi^\alpha(\lambda,W)$ deux copies de cette s\'erie principale, on obtient de m\^eme les (i) et (ii) de la Conjecture \ref{extglobmieuxn=2}, et dans ce cas il n'existe pas de $\rho_p$ comme en (iii).

\noindent
{\bf \'Etape $3$}\\
Supposons maintenant $W$ r\'eductible (non scalaire) avec $N\ne 0$ sur $D$. Alors \`a torsion pr\`es par un caract\`ere de $\Qp^\times$, la repr\'esentation $\pi^\alpha(\lambda,W)$ est comme dans \cite[Lem.~3.1.2(i)]{Br1}. Par le Th\'eor\`eme \ref{plusgeneral} et le Th\'eor\`eme \ref{gl3enplus} (pour $G={\rm GL}_2$) on a que $F_\alpha(\pi^\alpha(\lambda,W))$ est repr\'esentable par un $(\varphi,\Gamma)$-module libre de rang $2$ $D_\alpha(\pi^\alpha(\lambda,W))$ extension de deux $(\varphi,\Gamma)$-modules libres de rang $1$. Pour montrer $D_\alpha(\pi^\alpha(\lambda,W))\simeq D_\alpha(\lambda,W)$, on voit ais\'ement avec \cite[Prop.~2.15]{Na1} (et \cite[Th.~5.11]{Na2}) qu'il suffit de montrer que cette extension est non scind\'ee. Quitte \`a tordre encore par un caract\`ere de $\Qp^\times$, il existe une repr\'esentation $\rho_p$ semi-stable de poids de Hodge-Tate $h_1,h_2$ et module de Deligne-Fontaine $D$. De plus il r\'esulte facilement de \cite[Lem.~3.1.2(ii)]{Br1} et sa preuve avec (encore) le Th\'eor\`eme \ref{plusgeneral} et le Th\'eor\`eme \ref{localgenplus} que le foncteur $F_\alpha$ appliqu\'e \`a toute extension {\it non scind\'ee} dans ${\rm Ext}^1_{{\rm GL}_n(\Qp)}(\pi^\alpha(\lambda,W),L(-\lambda)\otimes_E\pi^{\infty}(W))$, par exemple l'extension donn\'ee par $\pi^{\an}(\rho_p)$ (par \cite{LXZ}), est repr\'esentable par un $(\varphi,\Gamma)$-module {\it libre} de rang $2$ qui contient $D_\alpha(\pi^\alpha(\lambda,W))$. Par l'\'Etape $1$ et ce qui pr\'ec\`ede, on d\'eduit un morphisme surjectif $D_\alpha(\pi^{\an}(\rho_p))\twoheadrightarrow D_{\rm rig}(\check\rho_p)$ de $(\varphi,\Gamma)$-modules libres de rang $2$, donc un isomorphisme. En particulier on a une injection $D_\alpha(\pi^\alpha(\lambda,W))\hookrightarrow D_{\rm rig}(\check\rho_p)$ d'o\`u on d\'eduit facilement que $D_\alpha(\pi^\alpha(\lambda,W))$ est bien une extension non scind\'ee de $(\varphi,\Gamma)$-modules libres de rang $1$. Un examen de tout ce qui pr\'ec\`ede (avec le (ii) de la Proposition \ref{lienfilmax} pour $\dim_E{\rm Ext}^1_{(\varphi,\Gamma)}(-,-)=2$) montre que l'on a finalement d\'emontr\'e les (i), (ii) et (iii) de la Conjecture \ref{extglobmieuxn=2} dans ce cas. Cela ach\`eve la preuve du (i).

\noindent
{\bf \'Etape $4$}\\
Supposons enfin $W$ irr\'eductible, par l'\'Etape $1$ il reste \`a v\'erifier la derni\`ere assertion en (ii). Par le Th\'eor\`eme \ref{localgenplus} et l'\'Etape $1$, on a un morphisme surjectif de $(\varphi,\Gamma)$-modules g\'en\'eralis\'es $D_\alpha(\pi^{\an}(\rho_p))\twoheadrightarrow D_{\rm rig}(\check\rho_p)$ o\`u $D_\alpha(\pi^{\an}(\rho_p))\={\mathcal E}^\alpha([\![\pi^{\an}(\rho_p)]\!])\in {\rm Ext}^1_{(\varphi,\Gamma)}(\R(\lambda\circ \lambda_{\alpha^{\!\vee}})/(t^{1-\langle \lambda,\alpha^\vee\rangle}),D_\alpha(\lambda,W))$. Si $D_\alpha(\pi^{\an}(\rho_p))$ a de la torsion $D_\alpha(\pi^{\an}(\rho_p))_{\tor}$, celle-ci s'injecte forc\'ement dans $\R(\lambda\circ \lambda_{\alpha^{\!\vee}})/(t^{1-\langle \lambda,\alpha^\vee\rangle})$, et par ailleurs s'envoie sur $0$ dans $D_{\rm rig}(\check\rho_p)$ (qui est libre), de sorte que $D_\alpha(\pi^{\an}(\rho_p))/D_\alpha(\pi^{\an}(\rho_p))_{\tor}\buildrel\sim\over\rightarrow D_{\rm rig}(\check\rho_p)$ (puisque ces deux $\R$-modules sont libres de rang $2$). Si $D_\alpha(\pi^{\an}(\rho_p))_{\tor}\ne 0$, on en d\'eduit en utilisant le (i) de la Proposition \ref{lienfilmax} que les poids de Hodge-Tate de $D_\alpha(\pi^{\an}(\rho_p))/D_\alpha(\pi^{\an}(\rho_p))_{\tor}$ sont $1-h_2,h'_1$ avec $h'_1>1-h_1$, ce qui contredit l'isomorphisme pr\'ec\'edent (les poids de Hodge-Tate de $\check\rho_p$ \'etant $1-h_2,1-h_1$). Donc $D_\alpha(\pi^{\an}(\rho_p))\buildrel\sim\over\rightarrow D_{\rm rig}(\check\rho_p)$, ce qui ach\`eve la preuve du (ii).
\end{proof}

\begin{rem}
{\rm (i) Lorsque $N\ne 0$ sur $D$ et en supposant $W$ non ramifi\'e, l'isomorphisme ${\mathcal F}_\alpha\circ {\mathcal E}_\alpha:{\rm Ext}^1_{{\rm GL}_n(\Qp)}(\pi^\alpha(\lambda,W),L(-\lambda)\otimes_E\pi^{\infty}(W))\buildrel\sim\over\longrightarrow D_L^{\glp}=D$ est le m\^eme, \`a multiplication pr\`es par un scalaire non nul, que l'inverse de l'isomorphisme ${\mathcal R}^1$ de \cite[(25)]{Br1}, car un tel isomorphisme satisfaisant le (iii) de la Conjecture \ref{extglobmieuxn=2} est dans ce cas unique (\`a scalaire pr\`es).\\
(ii) Lorsque $W$ est irr\'eductible, on a le r\'esultat suivant dans la direction du (i) de la Conjecture \ref{extglobmieuxn=2}. Quitte \`a tordre $W$ par un caract\`ere, il existe $\rho_p$ potentiellement semi-stable de poids de Hodge-Tate $h_1,h_2$ et module de Deligne-Fontaine $D$. Au moins lorsque $(h_1,h_2)=(1,0)$, il r\'esulte de \cite[Prop.~8.2]{DLB} que le morphisme $\pi^{\an}(\rho_p)^\vee\rightarrow D_{\rm rig}(\check\rho_p)_r$ dans le (ii) du Th\'eor\`eme \ref{n=2} induit pour $r\gg 0$ un morphisme de $(\psi,\Gamma)$-modules de Fr\'echet~:
$$\big(\pi^{\an}(\rho_p)/(L(-\lambda)\otimes_E\pi^{\infty}(W))\big)^\vee\longrightarrow D_\alpha(\lambda,W)_r.$$}
\end{rem}

Continuons avec le cas o\`u $W$ est une somme directe de caract\`eres non ramifi\'es suffisamment g\'en\'eriques, i.e. $W=\oplus_{i=1}^{n}\nr(c_i)$ pour des $c_i\in E^\times$ tels que $c_ic_j^{-1}\notin \{1,p\}$ pour $i\ne j\in \{0,\dots,n-1\}$ o\`u $\nr(z):\Qp^\times\rightarrow E^\times$ est le caract\`ere non ramifi\'e qui envoie $p\in\Qp^\times$ sur $z\in E^\times$. Le module de Deligne-Fontaine associ\'e est $D=Ef_1 \oplus \cdots \oplus Ef_{n}$ avec $\varphi(f_i)=c_i f_i$ et on a $\pi^\infty(W)$ comme en (\ref{norma}). Un calcul donne pour $\alpha=e_j-e_{j+1}$, $j\in \{1,\dots,n-1\}$~:
\begin{equation*}
D_\alpha(\lambda,W)=\bigoplus_{\mathcal P} \R\big((s_\alpha\cdot \lambda)\circ \lambda_{\alpha^{\!\vee}}\norm^{n-j+(n-j+1)+\cdots+n-1}\nr(\Pi_{i\notin {\mathcal P}}c_i)^{-1}\big)
\end{equation*}
o\`u $\mathcal P$ parcourt les sous-ensembles de $\{c_1,\dots,c_n\}$ de cardinal $n-j$. Pour $w\in {\mathcal S}_n$, soit $\nr_w\= {\nr}(c_{w^{-1}(1)})|\cdot|^{1-n}\otimes{\nr}(c_{w^{-1}(2)})|\cdot|^{2-n}\otimes\cdots \otimes {\nr}(c_{w^{-1}(n)})$ (un caract\`ere non ramifi\'e du tore). On d\'efinit $\pi^\alpha(\lambda,W)$ comme la somme directe de tous les constituants (irr\'eductibles) ${\mathcal F}_{B^-}^{{\rm GL}_n}\big(L^-(s_\alpha\cdot \lambda),\nr_w)$ qui sont {\it distincts} quand $w$ parcourt ${\mathcal S}_n$. Il y en a exactement $n \choose j$, cf. \cite[\S~3.3]{Br1}.

\begin{thm}\label{ncris}
Supposons $W=\oplus_{i=1}^{n}\nr(c_i)$ avec $c_ic_j^{-1}\notin \{1,p\}$ pour $i\ne j$.\\
(i) On a $F_\alpha(\pi^\alpha(\lambda,W))\simeq E_\infty(\chi_{-\lambda})\!\otimes_{E}\Hom_{(\varphi,\Gamma)}(D_\alpha(\lambda,W),-)$.\\
(ii) L'application (\ref{mapext}) pour $\pi'=\pi^\alpha(\lambda,W)$ est un isomorphisme~:
\begin{multline*}
{\mathcal E}_\alpha:{\rm Ext}^{1}_{{\rm GL}_n(\Qp)}\big(\pi^\alpha(\lambda,W),L(-\lambda)\otimes_E\pi^{\infty}(W)\big)\\\buildrel\sim\over\longrightarrow {\rm Ext}^1_{(\varphi,\Gamma)}\big(\R(\lambda\circ \lambda_{\alpha^{\!\vee}})/(t^{1-\langle \lambda,\alpha^\vee\rangle}),D_\alpha(\lambda,W)\big).
\end{multline*}
(iii) Soit $U^\wp$, $W^\wp$, $\rho$ v\'erifiant les points (i), (ii), (iii) du \S~\ref{locglob} tels que $U^p$ est suffisamment petit, $\rho$ est r\'esiduellement absolument irr\'eductible et $\rho\vert_{{\rm Gal}(\overline F_{\widetilde{v}}/F_{\widetilde{v}})}$ est cristalline pour $v\vert p$, $v\ne \wp$ avec les ratios des valeurs propres du Frobenius distincts de $1,p$, alors sous les hypoth\`eses standard de Taylor-Wiles on a une injection ${\rm GL}_n(\Qp)$-\'equivariante~:
\begin{equation*}
\bigg(\bigoplus^{\alpha}_{L(-\lambda)\otimes_E\pi^{\infty}(W)}[\![({\mathcal F}_\alpha\circ {\mathcal E}_\alpha)^{-1}(\Fil^{\rm max}_\alpha(\rho_{\widetilde{\wp}}))]\!] \bigg)\!\otimes \varepsilon^{n-1}\circ{\det}\hookrightarrow \widehat S(U^\wp,W^\wp)^{\an}[\mrho].
\end{equation*}
\end{thm}
\begin{proof}
Le (i) suit facilement des d\'efinitions et du (i) du Corollaire \ref{casparticuliers}. Pour le (ii), il suffit de montrer que l'application est injective puisque les deux $E$-espaces vectoriels sont de dimension $n\choose j$~: par le Corollaire \ref{cor05} pour celui de gauche et en utilisant le (ii) de la Proposition \ref{lienfilmax} pour celui de droite. Autrement dit il suffit de montrer que ${\mathcal E}_\alpha$ envoie l'unique extension non scind\'ee de ${\mathcal F}_{B^-}^{{\rm GL}_n}(L^-(s_\alpha\cdot \lambda),\nr_w)$ par $L(-\lambda)\otimes_E\pi^\infty(W)$ vers une extension non nulle de $(\varphi,\Gamma)$-modules \`a droite. Cela d\'ecoule facilement du Th\'eor\`eme \ref{plusgeneral} appliqu\'e \`a la s\'erie principale $(\Ind_{B^-(\Qp)}^{{\rm GL}_n(\Qp)} (-\lambda) \otimes_E\nr_w)^{\an}$ (qui contient l'unique extension non scind\'ee ci-dessus, cf. (\ref{extnonsc})) avec le (ii) de la Proposition \ref{drex} appliqu\'ee avec $\pi''=0$. Enfin le (iii) r\'esulte facilement de \cite[Cor.~5.16]{BH2} avec \cite[Th.~1.3]{BHS3} et de \cite[Prop.~3.3.2]{Br1}.
\end{proof}

Rappelons que les ``hypoth\`eses standard de Taylor-Wiles'' dans le (iii) du Th\'eor\`eme \ref{ncris} sont les suivantes (cf. e.g. \cite[\S~1]{BHS3})~: $p>2$, $F/F^+$ non ramifi\'e, $G$ quasi-d\'eploy\'e en toute place finie, $U_v$ hypersp\'ecial en toute place finie $v$ de $F^+$ inerte dans $F$ et $\overline \rho({\rm Gal}(\overline F/F(\sqrt[p]{1})))$ ad\'equat o\`u $\overline\rho$ est la r\'eduction (irr\'eductible) de $\rho$ modulo $p$.

Terminons ce paragraphe avec le cas important $\dim_EW=3$ et $N^2\ne 0$ sur $D$ (certains r\'esultats \'etant d\'emontr\'es dans les paragraphes suivants). Quitte \`a tordre $D$ par un caract\`ere de $\Qp^\times$, on suppose de plus $W$ non ramifi\'e, i.e. on a pour un $c\in E^\times$~:
\begin{equation}\label{fontaine}
D=Ef_{1}\oplus Ef_{2}\oplus Ef_{3}\ \ \ \ 
\left\{\!\!\!\begin{array}{lcl}
\varphi(f_{3})&=&c f_{3}\\
\varphi(f_{2})&=&c p^{-1}f_{2}\\
\varphi(f_{1})&=&c p^{-2}f_{1}\\
\end{array}\right.
\ \ \ 
\left\{\!\!\!\begin{array}{lcl}
N(f_3)&=&f_2\\
N(f_2)&=&f_1\\
N(f_1)&=&0
\end{array}\right.
\end{equation}
et $\pi^\infty(W)=\St_3\otimes_E (\nr(c)\circ {\det})$ o\`u $\St_3$ est la repr\'esentation lisse de Steinberg de ${\rm GL}_3(\Qp)$. On v\'erifie alors que $D_{\alpha}(\lambda,W)$ est (\`a isomorphisme pr\`es) l'unique $(\varphi,\Gamma)$-module de de Rham de la forme~:
\begin{equation}\small\label{equadifp}
\begin{gathered}
\begin{array}{lll}
D_{e_1-e_2}(\lambda,W)\!\!&=&\!\!\big(\!\begin{xy}(-60,0)*+{\R(\norm^{2})}="a";(-38,0)*+{\R(\norm)}="b";(-21,0)*+{\R}="c";{\ar@{-}"a";"b"};{\ar@{-}"b";"c"}\end{xy}\!\big)\otimes_{\R}\R\big(x^{2-h_2}\nr(c^{-1})\big)\\
D_{e_2-e_3}(\lambda,W)\!\!&=&\!\!\big(\!\begin{xy}(-60,0)*+{\R(\norm^{2})}="a";(-38,0)*+{\R(\norm)}="b";(-21,0)*+{\R}="c";{\ar@{-}"a";"b"};{\ar@{-}"b";"c"}\end{xy}\!\big)\otimes_{\R}\R\big(x^{3-(h_1+h_3)}\nr(c^{-2})\big)
\end{array}
\end{gathered}
\end{equation}
o\`u $\!\begin{xy}(-60,0)*+{D_1}="a";(-48,0)*+{D_2}="b";{\ar@{-}"a";"b"}\end{xy}\!$ d\'esigne une extension {\it non scind\'ee} du $(\varphi,\Gamma)$-module $D_2$ par le $(\varphi,\Gamma)$-module $D_1$. Les repr\'esentations $\pi^\alpha(\lambda,W)$ ont \'et\'e construites dans \cite[\S~4.5]{Br1}. On rappelle juste ici qu'elles sont dans la cat\'egorie $C_{\lambda,\alpha}$ du \S~\ref{prelgen} et ont la forme~:
\begin{equation}\label{calphai}
\!\begin{xy}(-60,0)*+{C_{\alpha,1}}="a";(-45,0)*+{C_{\alpha,2}}="b";(-30,0)*+{C_{\alpha,3}}="c";(-15,0)*+{C_{\alpha,4}}="d";(0,0)*+{C_{\alpha,5}}="e";{\ar@{-}"a";"b"};{\ar@{-}"b";"c"};{\ar@{-}"c";"d"};{\ar@{-}"d";"e"}\end{xy}
\end{equation}
o\`u $C_{\alpha,i}$ est irr\'eductible si $i\in \{1,3,5\}$ et est une somme directe de deux irr\'eductibles si $i\in\{2,4\}$ (cf. \cite[\S~4.1]{Br1} o\`u $C_{\alpha,i}$ est not\'e $C_i\oplus \widetilde C_i$ si $i\in \{2,4\}$), et o\`u $\!\begin{xy}(-60,0)*+{C_1}="a";(-48,0)*+{C_2}="b";{\ar@{-}"a";"b"}\end{xy}$ d\'esigne encore une extension non scind\'ee de $C_2$ par $C_1$. De plus on a $F_\alpha(C_{\alpha,i})=E_\infty(\chi_{-\lambda})\otimes_{E}\Hom_{(\varphi,\Gamma)}(D_\alpha(C_{\alpha,i}),-)$ avec $D_\alpha(C_{\alpha,i})=\R(x^{2-h_2}\norm^{\frac{i-1}{2}}\nr(c^{-1}))$ si $i\in \{1,3,5\}$ et $\alpha=e_1-e_2$, $D_\alpha(C_{\alpha,i})=\R(x^{3-(h_1+h_3)}\norm^{\frac{i-1}{2}}\nr(c^{-2}))$ si $i\in \{1,3,5\}$ et $\alpha=e_2-e_3$, $D_\alpha(C_{\alpha,i})=0$ si $i\in \{2,4\}$.

Soit $\rho_p:\gp\rightarrow {\rm GL}_3(E)$ semi-stable de poids de Hodge-Tate $h_1,h_2,h_3$ et module de Deligne-Fontaine $D$ comme ci-dessus. On suppose de plus que le $(\varphi,\Gamma)$-module sur $\R$ associ\'e \`a $\rho_p$ v\'erifie les hypoth\`eses de g\'en\'ericit\'e dans la derni\`ere partie de \cite[\S~3.3.4]{BD} (que l'on ne rappelle pas ici, elles sont par exemple satisfaites lorsque $p>3$, $h_1=h_2+1=h_3+2$ et $\rho_p$ admet un r\'eseau stable dont la r\'eduction $\overline{\rho_p}$ a tous ses sous-quotients de dimension $2$ non scind\'es, cf. \cite[Prop.~3.30~\&~Prop.~3.32]{BD}). Alors dans \cite[\S~3.3.4]{BD} est associ\'e \`a un tel $\rho_p$ pour $\alpha\in S$ une extension non scind\'ee $\pi^\alpha(\rho_p)\in {\rm Ext}^1_{{\rm GL}_n(\Qp)}(\pi^\alpha(\lambda,W),L(-\lambda)\otimes_E\pi^{\infty}(W))$ ne d\'ependant que de $\rho_p$. 

Les deux th\'eor\`emes suivants r\'esument l'essentiel de ce qui est connu dans la direction de la Conjecture \ref{extglobmieux} lorsque $\dim_EW=3$ et $N^2\ne 0$.

\begin{thm}\label{gl3loc}
Supposons $\dim_EW=3$ et $N^2\ne 0$.\\
(i) On a $F_\alpha(\pi^\alpha(\lambda,W))\simeq E_\infty(\chi_{-\lambda})\!\otimes_{E}\Hom_{(\varphi,\Gamma)}(D_\alpha(\lambda,W),-)$.\\
(ii) L'application (\ref{mapext}) pour $\pi'=\pi^\alpha(\lambda,W)$ est un isomorphisme.
\end{thm}

\begin{thm}\label{gl3glob}
Supposons $\dim_EW=3$ et $N^2\ne 0$.\\
(i) Pour tout $\rho_p:\gp\rightarrow {\rm GL}_3(E)$ semi-stable comme ci-dessus, la repr\'esentation $\pi^\alpha(\rho_p)$ d\'etermine $h_1,h_2,h_3$, $D$, $\Fil^{\rm max}_\alpha(\rho_p)$ et ne d\'epend que de ces donn\'ees.\\
(ii) Soit $U^\wp$, $W^\wp$, $\rho$ v\'erifiant les points (i), (ii), (iii) du \S~\ref{locglob} et tels que $h_1=h_2+1=h_3+2$, $U_v$ est maximal si $v\vert p$, $v\ne \wp$, $\overline\rho$ est absolument irr\'eductible et les sous-quotients de dimension $2$ de $\overline \rho\vert_{{\rm Gal}(\overline F_{\widetilde{\wp}}/F_{\widetilde{\wp}})}$ sont non scind\'es (cf. \cite[Th.~1.1]{BD}). Alors on a une injection ${\rm GL}_3(\Qp)$-\'equivariante~:
\begin{equation*}
\big(\pi^{e_1-e_2}(\rho_{\widetilde{\wp}})\oplus_{L(-\lambda)\otimes_E\pi^\infty(W)}\pi^{e_2-e_3}(\rho_{\widetilde{\wp}})\big)^{\oplus d(U^\wp,\rho)}\!\!\otimes \varepsilon^{n-1}\circ{\det}\hookrightarrow \widehat S(U^\wp,W^\wp)^{\an}[\mrho].
\end{equation*}
\end{thm}

Le (ii) du Th\'eor\`eme \ref{gl3glob} est d\'ej\`a dans \cite[Th.~1.1]{BD} (et ne contient donc rien de nouveau). Lorsque $h_1$, $h_2$, $h_3$ ne sont plus n\'ecessairement cons\'ecutifs, il est conjectur\'e dans \cite[Conj.~1.2]{BD} que la somme amalgam\'ee $(\pi^{e_1-e_2}(\rho_{\widetilde{\wp}})\oplus_{L(-\lambda)\otimes_E\pi^\infty(W)}\pi^{e_2-e_3}(\rho_{\widetilde{\wp}}))\!\otimes \varepsilon^{n-1}\circ{\det}$ se plonge encore dans $\widehat S(U^\wp,W^\wp)^{\an}[\mrho]$. Le (i) du Th\'eor\`eme \ref{gl3glob} sera montr\'e au \S~\ref{filmaxbis}. Notons qu'il dit que les repr\'esentations $\pi^\alpha(\rho_p)$ et $[\![({\mathcal F}_\alpha\circ {\mathcal E}_\alpha)^{-1}(\Fil^{\rm max}_\alpha(\rho_p))]\!]$ contiennent {\it exactement} la m\^eme information, et la conjecture qui suit (purement locale) est donc tr\`es naturelle. 

\begin{conj}\label{deuxpialpha}
Pour tout $\rho_p:\gp\rightarrow {\rm GL}_3(E)$ semi-stable comme ci-dessus, on a $\pi^\alpha(\rho_p)\simeq [\![({\mathcal F}_\alpha\circ {\mathcal E}_\alpha)^{-1}(\Fil^{\rm max}_\alpha(\rho_p))]\!]$.
\end{conj}

On termine ce paragraphe avec la preuve du Th\'eor\`eme \ref{gl3loc} (voir \S~\ref{filmaxbis} pour celle du (i) du Th\'eor\`eme \ref{gl3glob}). Quitte \`a tordre $\pi^\alpha(\lambda,W)$ et $\pi^\infty(W)$ par un caract\`ere non ramifi\'e, on peut supposer $c=1$ (cf. la derni\`ere assertion de la Remarque \ref{change}) et par sym\'etrie, il suffit de traiter le cas $\alpha=e_1-e_2$. On commence par la preuve du (i).

\noindent
{\bf \'Etape $1$}\\
Par (\ref{equadifp}) et le paragraphe qui le suit, $D_\alpha(\lambda,W)$ est l'unique $(\varphi,\Gamma)$-module sur $\R$ de de Rham de la forme $\!\begin{xy}(-60,0)*+{D_\alpha(C_{\alpha,5})}="a";(-37,0)*+{D_\alpha(C_{\alpha,3})}="b";(-14,0)*+{D_\alpha(C_{\alpha,1})}="c";{\ar@{-}"a";"b"};{\ar@{-}"b";"c"}\end{xy}\!$. \'Ecrivons $C_{\alpha,i}=C_{\alpha,i}^{1}\oplus C_{\alpha,i}^{2}$ si $i\in \{2,4\}$ avec les deux facteurs directs irr\'eductibles. En utilisant par r\'ecurrence le Th\'eor\`eme \ref{gl3enplus} en partant de $\pi'=C_{\alpha,5}$ puis en consid\'erant successivement $\pi''=C_{\alpha,4}^{1}$, $\pi''=C_{\alpha,4}^{2}$, $\pi''=C_{\alpha,3}$, $\pi''=C_{\alpha,2}^{1}$, $\pi''=C_{\alpha,2}^{2}$ et $\pi''=C_{\alpha,1}$ (et en augmentant $\pi'$ \`a chaque \'etape via le Th\'eor\`eme \ref{gl3enplus}) on obtient $F_\alpha(\pi^\alpha(\lambda,W))\simeq E_\infty(\chi_{-\lambda})\!\otimes_{E}\Hom_{(\varphi,\Gamma)}(\widetilde D_\alpha(\lambda,W),-)$ o\`u $\widetilde D_\alpha(\lambda,W)$ est un $(\varphi,\Gamma)$-module sur $\R$ sans torsion de la forme $\!\begin{xy}(-60,0)*+{D_\alpha(C_{\alpha,5})}="a";(-37,0)*+{D_\alpha(C_{\alpha,3})}="b";(-14,0)*+{D_\alpha(C_{\alpha,1})}="c";{\ar@{--}"a";"b"};{\ar@{--}"b";"c"}\end{xy}\!$, la notation $\!\begin{xy}(-60,0)*+{D_1}="a";(-48,0)*+{D_2}="b";{\ar@{--}"a";"b"}\end{xy}\!$ d\'esignant une extension {\it quelconque} (i.e. possiblement scind\'ee) de $D_2$ par $D_1$.

\noindent
{\bf \'Etape $2$}\\
On montre que toutes les extensions dans $\widetilde D_\alpha(\lambda,W)$ sont non scind\'ees. On commence par $\!\begin{xy}(-60,0)*+{D_\alpha(C_{\alpha,3})}="a";(-37,0)*+{D_\alpha(C_{\alpha,1})}="b";{\ar@{--}"a";"b"}\end{xy}\!$. On note $I\={(\Ind_{B^-(\Qp)\cap L_{P_\alpha}(\Qp)}^{L_{P_\alpha}(\Qp)}(-s_\alpha\cdot \lambda))^{\rm an}}$ et $\widetilde I\={(\Ind_{B^-(\Qp)\cap L_{P_\alpha}(\Qp)}^{L_{P_\alpha}(\Qp)}(-s_\alpha\cdot \lambda)}\otimes_E (\vert\cdot\vert^{-1}\otimes \vert\cdot\vert\otimes 1))^{\rm an}$ (o\`u selon la notation usuelle on voit $-s_\alpha\cdot \lambda$ comme caract\`ere alg\'ebrique de $T(\Qp)$). Par \cite[Lem.~3.1.2(i)]{Br1} on a une unique repr\'esentation de $L_{P_\alpha}(\Qp)$ sur $E$ de la forme $\!\begin{xy}(-60,0)*+{I}="a"; (-45,0)*+{L(-\lambda)_{P_\alpha}}="b"; (-30,0)*+{\widetilde I}="c"; {\ar@{-}"a";"b"}; {\ar@{-}"b";"c"}\end{xy}$\!. En fait, si l'on note $D_2=Ef_1\oplus Ef_2$ le module de Deligne-Fontaine de dimension $2$ o\`u $\left\{\!\!\!\begin{array}{lcl}\varphi(f_{2})&=&f_{2}\\ \varphi(f_{1})&=&p^{-1}f_{1}\ \end{array}\right.$ $\left\{\!\!\!\begin{array}{lcl} N(f_{2})&=&f_{1}\\ N(f_{1})&=&0\end{array}\right.$ et $W_2$ la repr\'esentation de Weil-Deligne associ\'ee, on a $\!\begin{xy}(-60,0)*+{I}="a"; (-45,0)*+{L(-\lambda)_{P_\alpha}}="b"; (-30,0)*+{\widetilde I}="c"; {\ar@{-}"a";"b"}; {\ar@{-}"b";"c"}\end{xy}\!\simeq \pi^\alpha((\lambda_1,\lambda_2),W_2)\boxtimes_E (-\lambda_3)$ o\`u $\boxtimes$ est le produit tensoriel {\it ext\'erieur} d'une repr\'esentation de ${\rm GL}_2(\Qp)$ et d'une repr\'esentation de $\Qp^\times$. Un examen de la preuve de \cite[Prop.~4.4.2]{Br1} montre que l'on a une injection~:
$$\pi\hookrightarrow \big(\Ind_{P_\alpha(\Qp)}^{G(\Qp)}(\!\begin{xy}(-60,0)*+{I}="a"; (-45,0)*+{L(-\lambda)_{P_\alpha}}="b"; (-30,0)*+{\widetilde I}="c"; {\ar@{-}"a";"b"}; {\ar@{-}"b";"c"}\end{xy}\!)\big)^{\an}$$
o\`u $\pi$ est une repr\'esentation dans $C_{\lambda,\alpha}$ de la forme $\!\begin{xy}(-60,0)*+{\pi''}="a"; (-26,0)*+{\big(\!\begin{xy}(-60,0)*+{C_{\alpha,1}}="a";(-45,0)*+{C_{\alpha,2}}="b";(-30,0)*+{C_{\alpha,3}}="c";{\ar@{-}"a";"b"};{\ar@{-}"b";"c"}\end{xy}\!\big)}="b";{\ar@{--}"a";"b"}\end{xy}$
avec les constituants de $\pi''$ tous de la forme ${\mathcal F}_{P^-}^G(L^-(w\cdot \lambda),\pi_P^\infty)$ pour des $\pi_P^\infty$ {\it non} g\'en\'eriques. Il r\'esulte du Th\'eor\`eme \ref{gl3enplus} que l'on a~:
$$F_\alpha(\pi)\simeq E_\infty(\chi_{-\lambda})\!\otimes_{E}\Hom_{(\varphi,\Gamma)}\!\big(\!\!\begin{xy}(-60,0)*+{D_\alpha(C_{\alpha,3})}="a";(-37,0)*+{D_\alpha(C_{\alpha,1})}="b";{\ar@{--}"a";"b"}\end{xy}\!,-\big)$$
et du Th\'eor\`eme \ref{encoreplusgeneral} suivi du (i) du Th\'eor\`eme \ref{n=2} que l'on a~:
\begin{multline*}
F_\alpha\Big(\big(\Ind_{P_\alpha(\Qp)}^{G(\Qp)}(\!\begin{xy}(-60,0)*+{I}="a"; (-46,0)*+{L(-\lambda)_{P_\alpha}}="b"; (-32,0)*+{\widetilde I}="c"; {\ar@{-}"a";"b"}; {\ar@{-}"b";"c"}\end{xy})\big)^{\an}\Big)\\
\simeq E_\infty(\chi_{-\lambda})\otimes_{E}\Hom_{(\varphi,\Gamma)}\big(\begin{xy}(-60,0)*+{D_\alpha(C_{\alpha,3})}="a";(-39,0)*+{D_\alpha(C_{\alpha,1})}="b";{\ar@{-}"a";"b"}\end{xy},-\big)
\end{multline*}
o\`u $\!\begin{xy}(-60,0)*+{D_\alpha(C_{\alpha,3})}="a";(-37,0)*+{D_\alpha(C_{\alpha,1})}="b";{\ar@{-}"a";"b"}\end{xy}\!$ est l'unique extension non scind\'ee (\cite[Prop.~2.15]{Na1} et \cite[Th.~5.11]{Na2}). Par le (ii) de la Proposition \ref{drex} (appliqu\'ee avec $\pi''=0$) l'injection $\pi\hookrightarrow (\Ind_{P_\alpha(\Qp)}^{G(\Qp)}(\!\begin{xy}(-60,0)*+{I}="a"; (-46,0)*+{L(-\lambda)_{P_\alpha}}="b"; (-32,0)*+{\widetilde I}="c"; {\ar@{-}"a";"b"}; {\ar@{-}"b";"c"}\end{xy}\!))^{\an}$ induit une surjection de $(\varphi,\Gamma)$-modules~:
$$\!\begin{xy}(-60,0)*+{D_\alpha(C_{\alpha,3})}="a";(-37,0)*+{D_\alpha(C_{\alpha,1})}="b";{\ar@{-}"a";"b"}\end{xy}\!\twoheadrightarrow \!\begin{xy}(-60,0)*+{D_\alpha(C_{\alpha,3})}="a";(-37,0)*+{D_\alpha(C_{\alpha,1})}="b";{\ar@{--}"a";"b"}\end{xy}$$
qui doit donc \^etre un isomorphisme. La preuve pour la deuxi\`eme extension $\!\begin{xy}(-60,0)*+{D_\alpha(C_{\alpha,5})}="a";(-37,0)*+{D_\alpha(C_{\alpha,3})}="b";{\ar@{--}"a";"b"}\end{xy}\!$ est analogue en rempla\c cant $\!\begin{xy}(-60,0)*+{I}="a"; (-45,0)*+{L(-\lambda)_{P_\alpha}}="b"; (-30,0)*+{\widetilde I}="c"; {\ar@{-}"a";"b"}; {\ar@{-}"b";"c"}\end{xy}$\! par son tordu par le caract\`ere $\vert {\det}_2\vert^{-1}\boxtimes_E \vert \cdot\vert^{2}$ de $L_{P_\alpha}(\Qp)$ (cf. \cite[Prop.~4.4.4]{Br1}).

\noindent
{\bf \'Etape $3$}\\
On montre que $\widetilde D_\alpha(\lambda,W)$ est un $(\varphi,\Gamma)$-module de de Rham. Soit $\pi$ une extension dans ${\rm Ext}^1_{{\rm GL}_n(\Qp)}(\pi^\alpha(\lambda,W),L(-\lambda)\otimes_E\pi^{\infty}(W))$ telle que l'on a une extension non scind\'ee $\pi'\=\!\begin{xy}(-60,0)*+{L(-\lambda)\otimes_E\pi^{\infty}(W)}="a";(-33,0)*+{C_{\alpha,1}}="b";{\ar@{-}"a";"b"}\end{xy}\!$ en sous-objet de $\pi$ (c'est possible par \cite[\S~4.6]{Br1}). Il r\'esulte de \cite[\S~4.1~\&~Prop.~4.6.1~\&~Lem.~5.3.1~\&~(53)]{Br1} avec \cite[Cor.~2.5]{Br3} (on laisse les d\'etails faciles au lecteur) que l'on a un isomorphisme~:
\begin{equation}\label{extspe}
\begin{xy}(-60,0)*+{\pi''}="a"; (-48,0)*+{\pi'}="b"; {\ar@{--}"a";"b"}\end{xy}\!\simeq {\mathcal F}_{P_\beta^-}^G\big(U(\mg)\otimes_{U(\mpp_\beta^-)}L^-(\lambda)_{P_\beta}, 1\boxtimes_E {\rm St}_2^\infty\big)
\end{equation}
o\`u ${\rm St}_2^\infty$ est la repr\'esentation de Steinberg lisse de ${\rm GL}_2(\Qp)$ et $\pi''$ a tous ses constituants de la forme ${\mathcal F}_{P^-}^G(L^-(w\cdot \lambda),\pi_P^\infty)$ pour des $\pi_P^\infty$ {\it non} g\'en\'eriques. Il r\'esulte du Th\'eor\`eme \ref{gl3enplus} que l'on a $F_\alpha(\!\begin{xy}(-60,0)*+{\pi''}="a"; (-49,0)*+{\pi'}="b"; {\ar@{--}"a";"b"}\end{xy}\!)\buildrel\sim\over\rightarrow F_\alpha(\pi')$, d'o\`u par (\ref{extspe}) et le (i) du Corollaire \ref{casparticuliers} $F_\alpha(\pi')\simeq E_\infty(\chi_{-\lambda})\!\otimes_{E}\Hom_{(\varphi,\Gamma)}(\R(x^{2-h_1}),-)$. Autrement dit, par le Th\'eor\`eme \ref{gl3enplus} appliqu\'e \`a $\pi$ on a $F_\alpha(\pi)\simeq E_\infty(\chi_{-\lambda})\!\otimes_{E}\Hom_{(\varphi,\Gamma)}(D_\alpha(\pi),-)$ avec $D_\alpha(\pi)$ de la forme~:
$$\!\begin{xy}(-60,0)*+{\R(x^{2-h_2}\norm^{2})}="a";(-28,0)*+{\R(x^{2-h_2}\norm)}="b";(1,0)*+{\R(x^{2-h_1})}="c";{\ar@{-}"a";"b"};{\ar@{-}"b";"c"}\end{xy}\!\simeq \!\begin{xy}(-60,0)*+{\widetilde D_\alpha(\lambda,W)}="a";(-27,0)*+{\R(x^{2-h_1})/(t^{h_1-h_2})}="b";{\ar@{-}"a";"b"}\end{xy}\!.$$
Or, c'est un exercice facile en utilisant $2-h_1<2-h_2$ et l'\'equivalence de cat\'egories exactes entre modules filtr\'es et $(\varphi,\Gamma)$-modules de de Rham (\cite[Th.A]{Be2}) de v\'erifier que toutes les extensions de $\R(x^{2-h_1})$ par $\!\begin{xy}(-60,0)*+{\R(x^{2-h_2}\norm^{2})}="a";(-28,0)*+{\R(x^{2-h_2}\norm)}="b";{\ar@{-}"a";"b"}\end{xy}\!$ sont des $(\varphi,\Gamma)$-modules de de Rham (et m\^eme semi-stables). En particulier, $\widetilde D_\alpha(\lambda,W)$ est aussi de de Rham (car $\widetilde D_\alpha(\lambda,W)[1/t]\buildrel\sim\over\rightarrow D_\alpha(\pi)[1/t]$, cf. la discussion avant (\ref{diff2})). Cela termine la preuve du (i) du Th\'eor\`eme \ref{gl3loc}.

On d\'emontre maintenant le (ii) du Th\'eor\`eme \ref{gl3loc}. Comme les deux espaces sont de dimension $3$ (utiliser le (ii) de la Proposition \ref{lienfilmax} pour celui de droite), il suffit de montrer l'injectivit\'e de ${\mathcal E}_\alpha$. Si $\pi$ est une extension dans ${\rm Ext}^1_{{\rm GL}_n(\Qp)}(\pi^\alpha(\lambda,W),L(-\lambda)\otimes_E\pi^{\infty}(W))$, par le Th\'eor\`eme \ref{gl3enplus} (ou le Th\'eor\`eme \ref{localgenplus}) et le (i) on a $F_\alpha(\pi)=E_\infty(\chi_{-\lambda})\otimes_{E}\Hom_{(\varphi,\Gamma)}(D_\alpha(\pi),-)$ pour un $(\varphi,\Gamma)$-module $D_{\alpha}(\pi)$ qui est une extension de $\R(x^{2-h_1})/(t^{h_1-h_2})$ par $D_\alpha(\lambda,W)$. En proc\'edant comme dans l'\'Etape $3$ de la preuve du (i), il suit de \cite[\S~4.6]{Br1} qu'il existe une base $\pi_1$, $\pi_2$, $\pi_3$ de ${\rm Ext}^1_{{\rm GL}_n(\Qp)}(\pi^\alpha(\lambda,W),L(-\lambda)\otimes_E\pi^{\infty}(W))$ telle que $\pi_i$ contient en sous-quotient une repr\'esentation ind\'ecomposable $\pi'_i$ de la forme~:
\begin{eqnarray*}
\pi'_1&\=&\!\begin{xy}(-60,0)*+{L(-\lambda)\otimes_E\pi^{\infty}(W)}="a";(-33,0)*+{C_{\alpha,1}}="b";{\ar@{-}"a";"b"}\end{xy}\\
\pi'_2&\=&\!\begin{xy}(-60,0)*+{L(-\lambda)\otimes_E\pi^{\infty}(W)}="a";(-33,0)*+{C_{\alpha,2}}="b";(-19,0)*+{C_{\alpha,3}}="c";{\ar@{-}"a";"b"};{\ar@{-}"b";"c"}\end{xy}\\
\pi'_3&\=&\!\begin{xy}(-60,0)*+{L(-\lambda)\otimes_E\pi^{\infty}(W)}="a";(-33,0)*+{C_{\alpha,4}}="b";(-19,0)*+{C_{\alpha,5}}="c";{\ar@{-}"a";"b"};{\ar@{-}"b";"c"} \end{xy}\!.
\end{eqnarray*}
De plus, on a des isomorphismes par \cite[\S~4.1~\&~Prop.~4.6.1~\&~Lem.~5.3.1~\&~(53)]{Br1} avec \cite[Cor.~2.5]{Br3}~:
\begin{eqnarray*}
\!\begin{xy}(-60,0)*+{\pi''_1}="a"; (-48,0)*+{\pi'_1}="b"; {\ar@{--}"a";"b"}\end{xy}\!&\simeq &{\mathcal F}_{P_\beta^-}^G\big(U(\mg)\otimes_{U(\mpp_\beta^-)}L^-(\lambda)_{P_\beta}, 1\boxtimes_E {\rm St}_2^\infty\big)\\
\!\begin{xy}(-60,0)*+{\pi''_2}="a"; (-48,0)*+{\pi'_2}="b"; {\ar@{--}"a";"b"}\end{xy}\!&\simeq &{\mathcal F}_{P_\beta^-}^G\big(U(\mg)\otimes_{U(\mpp_\beta^-)}L^-(\lambda)_{P_\beta}, \big(\Ind_{B^-(\Qp)\cap L_{P_\beta}(\Qp)}^{L_{P_\beta}(\Qp)}\vert\cdot\vert^{-1}\otimes \vert\cdot\vert\otimes 1\big)^{\infty}\big)\\
\!\begin{xy}(-60,0)*+{\pi''_3}="a"; (-48,0)*+{\pi'_3}="b"; {\ar@{--}"a";"b"}\end{xy}\!&\simeq &{\mathcal F}_{P_\beta^-}^G\big(U(\mg)\otimes_{U(\mpp_\beta^-)}L^-(\lambda)_{P_\beta},\vert\cdot\vert^{-2}\boxtimes_E({\rm St}_2^\infty\otimes_E\vert{\det}_2\vert)\big)
\end{eqnarray*}
o\`u ${\det}_2$ est le d\'eterminant pour ${\rm GL}_2(\Qp)$ et $\pi''_i$ n'a que des constituants ${\mathcal F}_{P^-}^G(L^-(w\cdot \lambda),\pi_P^\infty)$ pour des $\pi_P^\infty$ non g\'en\'eriques. On en d\'eduit facilement comme dans l'\'Etape $3$ avec le Th\'eor\`eme \ref{gl3enplus} et le (i) du Corollaire \ref{casparticuliers} que $D_\alpha(\pi_i)$ est une extension non scind\'ee de $\R(x^{2-h_1})/(t^{h_1-h_2})$ par $D_\alpha(\lambda,W)$ de la forme~:
\begin{eqnarray*}
D_\alpha(\pi_1)&\simeq & \!\begin{xy}(-60,0)*+{\R(x^{2-h_2}\norm^{2})}="a";(-28,0)*+{\R(x^{2-h_2}\norm)}="b";(1,0)*+{\R(x^{2-h_2})}="c";(34.5,0)*+{\R(x^{2-h_1})/(t^{h_1-h_2})}="d";{\ar@{-}"a";"b"};{\ar@{-}"b";"c"};{\ar@{-}"c";"d"}\end{xy}\\
D_\alpha(\pi_2)&\simeq &\!\begin{xy}(-60,0)*+{\R(x^{2-h_2}\norm^{2})}="a";(-28,0)*+{\R(x^{2-h_2}\norm)}="b";(1,0)*+{\R(x^{2-h_2})}="c";(1,-9)*+{\R(x^{2-h_1})/(t^{h_1-h_2})}="d";{\ar@{-}"a";"b"};{\ar@{-}"b";"c"};{\ar@{-}"b";"d"}\end{xy}\\
D_\alpha(\pi_2)&\simeq &\!\begin{xy}(-60,0)*+{\R(x^{2-h_2}\norm^{2})}="a";(-28,0)*+{\R(x^{2-h_2}\norm)}="b";(1,0)*+{\R(x^{2-h_2})}="c";(-28,-9)*+{\R(x^{2-h_1})/(t^{h_1-h_2})}="d";{\ar@{-}"a";"b"};{\ar@{-}"b";"c"};{\ar@{-}"a";"d"}\end{xy}\!.
\end{eqnarray*}
Les trois extensions $D_\alpha(\pi_1)$, $D_\alpha(\pi_2)$, $D_\alpha(\pi_3)$ forment clairement une base de ${\rm Ext}^1_{(\varphi,\Gamma)}(\R(x^{2-h_1})/(t^{h_1-h_2}),D_\alpha(\lambda,W))$, d'o\`u le r\'esultat.

\subsection{Preuve du Th\'eor\`eme \ref{gl3glob}}\label{filmaxbis}

On montre le (i) du Th\'eor\`eme \ref{gl3glob} en uti\-lisant les r\'esultats de l'appendice (Proposition \ref{split}). La preuve est essentiellement ind\'ependante des autres r\'esultats de cet article.

On conserve les notations des paragraphes pr\'ec\'edents. Soit $D$ un $(\varphi,\Gamma)$-module (libre de rang fini) sur $\R$ et $W_{\mathrm{dR}}^+(D)$ le $B_{\mathrm{dR}}^+$-module avec action semi-lin\'eaire de $\gp$ associ\'e (cf. \S~\ref{hodge}). Il suit de la d\'efinition du foncteur (exact) $W_{\mathrm{dR}}^+$ que l'on a des applications canoniques pour $i=0,1$ (et m\^eme pour tout $i$, cf. \cite[\S~2.1]{Na1} et \cite[Th.~5.11]{Na2})~:
\begin{equation}\label{hg}
H^i_{(\varphi,\Gamma)}(D)\lra H^i(\gp, W_{\mathrm{dR}}^+(D)).
\end{equation}
En composant (\ref{hg}) \`a droite avec l'application d\'eduite de $W_{\mathrm{dR}}^+(D)\twoheadrightarrow W_{\mathrm{dR}}^+(D)/(t)$, on d\'efinit~:
\begin{equation}\label{htb}
H^1_{\mathrm{HT}}(D) \= \ker\big(H^1_{(\varphi,\Gamma)}(D) \ra H^1(\gp, W_{\mathrm{dR}}^+(D)/(t))\big).
\end{equation}
Rappelons qu'un $(\varphi,\Gamma)$-module $D$ sur $\R$ est dit de Hodge-Tate si le rang de $D$ est $\dim_E \sum_{i\in \Z} H^0(\gp, t^i W_{\mathrm{dR}}^+(D)/t^{i+1} W_{\mathrm{dR}}^+(D))$.

\begin{lem}
Supposons $D$ de Hodge-Tate et soit $D'$ un $(\varphi,\Gamma)$-module extension de $\R$ par $D$. Alors $E[D']\subseteq H^1_{\mathrm{HT}}(D)$ si et seulement si $D'$ est de Hodge-Tate.
\end{lem}
\begin{proof}
La suite exacte $0\ra D \ra D' \ra \R\ra 0$ induit une suite exacte de $B_{\mathrm{dR}}^+$-modules avec action semi-lin\'eaire de $\gp$~:
\begin{equation*}
0 \lra W_{\mathrm{dR}}^+(D) \lra W_{\mathrm{dR}}^+(D') \lra B_{\mathrm{dR}}^+\otimes_{\Qp} E \lra 0
\end{equation*}
qui (en regardant la multiplication par $t\in B_{\mathrm{dR}}^+$) induit une suite exacte~:
\begin{equation*}
0 \lra W_{\mathrm{dR}}^+(D)/(t) \lra W_{\mathrm{dR}}^+(D')/(t) \lra B_{\mathrm{dR}}^+/(t)\otimes_{\Qp} E \lra 0.
\end{equation*}
On en d\'eduit un diagramme commutatif (en \'ecrivant $H^i(-)$ pour $H^i(\gp,\!-)$)~\!:
\begin{equation*}\begin{CD}
 0 @>>> H^0_{(\varphi,\Gamma)}(D) @>>> H^0_{(\varphi,\Gamma)}(D') @>>> E @>>> H^1_{(\varphi,\Gamma)}(D) \\
@. @VVV @VVV @| @VVV \\
0 @>>> H^0\big(\frac{W_{\mathrm{dR}}^+(D)}{(t)}\big) @>>> H^0\big(\frac{W_{\mathrm{dR}}^+(D')}{(t)}\big) @>>> E @>>>H^1\big(\frac{W_{\mathrm{dR}}^+(D)}{(t)}\big).
\end{CD}
\end{equation*}
Comme l'image de $E$ dans la fl\`eche du haut \`a droite est $E[D']$ et comme $H^0(\gp, t^i B_{\mathrm{dR}}^+/t^{i+1} B_{\mathrm{dR}}^+)=0$ si $i\ne 0$, on en d\'eduit facilement le lemme.
\end{proof}

On pose $k_1\=h_1-2$, $k_2\=h_2-1$ et $k_3\=h_3$, on a donc $k_1\geq k_2\geq k_3$ (et $k_i=-\lambda_i$ avec les notations du \S~\ref{locglob}). Soit $D_1$ de la forme $D_1\simeq \!\begin{xy}(-60,0)*+{\R(\varepsilon^2 x^{k_1})}="a";(-34,0)*+{\R(\varepsilon x^{k_2})}="b";{\ar@{-}"a";"b"}\end{xy}\!$ (i.e. une extension non scind\'ee). C'est un $(\varphi,\Gamma)$-module de de Rham, et m\^eme semi-stable. De plus, par \cite[(2.2)]{BD} le cup-produit induit un diagramme commutatif d'accouplements~:
\begin{equation}\label{hta}
\begin{CD}
\Ext^1_{(\varphi,\Gamma)}(\R(x^{k_3}), D_1) \ \ \ \ @. \times @. \Ext^1_{(\varphi,\Gamma)}(D_1, D_1) @> \cup>> E \\
 @| @. @V \kappa VV @| \\
\!\!\!\!\!\!\!\Ext^1_{(\varphi,\Gamma)}(\R(x^{k_3}), D_1) @. \times @. \ \ \ \Ext^1_{(\varphi,\Gamma)}(D_1, \R(\varepsilon x^{k_2}))@> \cup >> E
 \end{CD}
\end{equation}
o\`u l'application $\kappa$ est surjective et l'accouplement du bas est parfait (\cite[Prop.~2.3(2)]{BD}). Soit $k_2'\=k_2+(k_2+1-k_3)$, on a aussi un diagramme commutatif d'accouplements parfaits (noter que $k'_2>k_2$ et rappelons que $k_2+1>k_3$)~:
\begin{equation}\label{ht1}
\begin{CD}
\Ext^1_{(\varphi,\Gamma)}(\R(x^{k_3}), D_1) \ @. \times @. \ \ \Ext^1_{(\varphi,\Gamma)}(D_1, \R(\varepsilon x^{k_2}))@> \cup>> E \\
@V u VV @. @A j AA @| \\
\Ext^1_{(\varphi,\Gamma)}(\R(x^{k_2+1}), D_1) \ \ @. \times @. \ \ \Ext^1_{(\varphi,\Gamma)}(D_1, \R(\varepsilon x^{k_2'}))@> \cup >> E
\end{CD}
\end{equation}
qui se d\'eduit des accouplements parfaits (\cite[\S~5.2]{Li})~:
\begin{equation}\label{ht2}
\begin{CD}
\ H^1_{(\varphi, \Gamma)}( D_1(x^{-k_3})) \ @. \times @. \ \ H^1_{(\varphi,\Gamma)}(D_1^{\vee}(\varepsilon x^{k_2})) @> \cup>> H^2_{(\varphi, \Gamma)}(\R(\varepsilon x^{k_2-k_3})) \\
@V VV @. @AAA @| \\
H^1_{(\varphi, \Gamma)}( D_1(x^{-1-k_2})) \ \ @. \times @. \ \ H^1_{(\varphi,\Gamma)}(D_1^{\vee}(\varepsilon x^{k_2'})) @> \cup >> H^2_{(\varphi, \Gamma)}(\R(\varepsilon x^{k_2-k_3}))
\end{CD}
\end{equation}
o\`u $D_1^\vee$ est le $(\varphi,\Gamma)$-module dual de $D_1$ (notons que $D_1^{\vee}(\varepsilon x^{k_2'})\simeq D_1(x^{-1-k_2})^{\vee}(\varepsilon x^{k_2-k_3})$).

\begin{lem}\label{lem1}
(i) On a les \'egalit\'es~:
$$\dim_E \Ext^1_{(\varphi,\Gamma)}(\R(x^{k_2+1}), D_1)=\dim_E \Ext^1_{(\varphi,\Gamma)}(D_1, \R(\varepsilon x^{k_2'}))=2.$$
(ii) Dans le diagramme (\ref{ht1}) l'application $u$ est surjective et l'application $j$ injective.
\end{lem}
\begin{proof}
Le (i) se d\'eduit facilement de \cite[Prop.~2.15]{Na1} et \cite[Th.~5.11]{Na2} (qui incluent le cas $p=2$). Montrons le (ii). Soit $D\=D_1(x ^{-1-k_2})$ et $D_{\mathrm dR}(D)$ le $E$-espace vectoriel filtr\'e de dimension $2$ associ\'e au $(\varphi,\Gamma)$-module de de Rham $D$ (par exemple par \cite[Th.~A]{Be2}). On a ${\rm Fil}^i(D_{\mathrm dR}(D))=({\rm Fil}^i(W_{\rm dR}(D)))^{\gp}$ pour $i\in \Z$ (cf. \cite[Def.~1.17]{Na1}) o\`u $W_{\rm dR}(D)\=B_{\rm dR}\otimes_{B^+_{\rm dR}}W^+_{\rm dR}(D)=W^+_{\rm dR}(D)[1/t]$ avec $W^+_{\rm dR}(D)$ comme au \S~\ref{hodge} et ${\rm Fil}^i(W_{\rm dR}(D))=t^iW^+_{\rm dR}(D)$ (en particulier ${\rm Fil}^0(W_{\rm dR}(D))=W^+_{\rm dR}(D)$). La forme de $D$ montre que, pour $i\in \Z$, on a ${\rm Fil}^i(D_{\mathrm dR}(D))=D_{\mathrm dR}(D)$ si $i\leq -(k_1-k_2+1)$, ${\rm Fil}^i(D_{\mathrm dR}(D))$ est une droite si $-(k_1-k_2+1)< i\leq 0$ et ${\rm Fil}^i(D_{\mathrm dR}(D))=0$ si $0< i$. Comme $D$ est de de Rham, on a un isomorphisme compatible \`a $\gp$ et aux filtrations $B_{\rm dR}\otimes_{\Qp}D_{\mathrm dR}(D)\buildrel\sim\over\rightarrow W_{\rm dR}(D)$, d'o\`u on d\'eduit par ce qui pr\'ec\`ede~:
\begin{eqnarray}\label{filw}
\nonumber W^+_{\rm dR}(D)&\simeq &\sum_{i+j=0}t^jB^+_{\rm dR}\otimes_{\Qp}{\rm Fil}^i(D_{\mathrm dR}(D))\\
&\simeq & B^+_{\rm dR}\otimes_{\Qp}E\ \oplus \ t^{k_1-k_2+1}B^+_{\rm dR}\otimes_{\Qp}E.
\end{eqnarray}
Utilisant le Lemme \ref{h0t} avec (\ref{filw}) (et $H^0(\gp, t^j B_{\mathrm{dR}}^+/t^{j+1} B_{\mathrm{dR}}^+)=0$ si $j\ne 0$), on obtient $\dim_E H^0_{(\varphi,\Gamma)}(D/(t^{k_2-k_3+1})) =1$. On a par ailleurs une suite exacte~:
\begin{equation}\label{fil1}
0 \ra H^0_{(\varphi,\Gamma)}(D_1(x ^{-1-k_2})/t^{k_2-k_3+1}) \ra H^1_{(\varphi, \Gamma)}( D_1(x^{-k_3})) \ra H^1_{(\varphi, \Gamma)}( D_1(x^{-1-k_2})).
\end{equation}
Comparant (\ref{fil1}) et (\ref{ht2}), on d\'eduit de ce qui pr\'ec\`ede $\dim_E \ker(u)=1$. Par \cite[Lem. 2.2]{BD}, on a $\dim_E \Ext^1_{(\varphi,\Gamma)}(\R(x^{k_3}), D_1) =3$, ce qui implique finalement la surjectivit\'e de $u$ par (i). La preuve de l'injectivit\'e de $j$ est analogue et laiss\'ee au lecteur.
\end{proof}

On note $D_0$ l'unique $(\varphi,\Gamma)$-module de de Rham de la forme (cf. (\ref{equadifp}))~:
\begin{multline*}
\begin{xy}(-60,0)*+{\R(\varepsilon^2 x^{k_2-1})}="a";(-33,0)*+{\R(\varepsilon x^{k_2})}="b";(-8,0)*+{\R(x^{k_2+1})}="c";{\ar@{-}"a";"b"};{\ar@{-}"b";"c"}\end{xy}\\
\simeq \big(\!\begin{xy}(-60,0)*+{\R(\norm^{2})}="a";(-37,0)*+{\R(\norm)}="b";(-19,0)*+{\R}="c";{\ar@{-}"a";"b"};{\ar@{-}"b";"c"}\end{xy}\!\big)\otimes_{\R}\R\big(x^{k_2+1}\big).
\end{multline*}

\begin{prop}\label{lemfil}
Soit $D$ une extension de $\R(x^{k_3})$ par $D_1$ telle que $N^2\neq 0$ sur $D_{\mathrm dR}(D)=D_{\mathrm st}(D)$ (cf. \cite[\S~1]{Na1} pour les notations, cf. aussi la preuve du (ii) du Lemme \ref{lem1}) et notons $u(D)$ le ``pull-back'' de $D$ le long de $\R(x^{k_2+1})\hookrightarrow \R(x^{k_3})$. Alors le dual $u(D)^\vee$ est une extension (non scind\'ee) de $\R(\varepsilon^{-2}x^{-k_1})/(t^{k_1-k_2+1})$ par $D_0^{\vee}$.
\end{prop}
\begin{proof}
Notons que par \cite[Lem. 3.2]{Di2} l'hypoth\`ese $N^2\neq 0$ sur $D_{\mathrm dR}(D)$ implique que le $(\varphi,\Gamma)$-module en sous-objet $D_1$ n'est pas cristallin. L'application $u$ en (\ref{ht1}) s'inscrit dans un diagramme commutatif~:
\begin{equation*}\small
\xymatrix{\Ext^1\!(\R(x^{k_3}), \R(\varepsilon^2x^k_1))\ar[r]\ar[d] & \Ext^1\!(\R(x^{k_3}), D_1)\ar[r] \ar[d]^{u} & \Ext^1\!(\R(x^{k_3}), \R(\varepsilon x^{k_2})) \ar[d]^{u_1}\\
\Ext^1\!(\R(x^{k_2+1}), \R(\varepsilon^2x^k_1)) \ar[r] & \Ext^1\!(\R(x^{k_2+1}), D_1) \ar[r] & \Ext^1\!(\R(x^{k_2+1}), \R(\varepsilon x^{k_2})).}
\end{equation*}
Via les isomorphismes $\Ext^1_{(\varphi,\Gamma)}(\R(x^{k_3}), \R(\varepsilon x^{k_2}))\simeq H^1_{(\varphi,\Gamma)}(\R(\varepsilon x^{k_2-k_3}))$ et $\Ext^1_{(\varphi,\Gamma)}(\R(x^{k_2+1}), \R(\varepsilon x^{k_2}))\simeq H^1_{(\varphi,\Gamma)}(\R(\varepsilon x^{-1}))$, $u_1$ est l'application naturelle~:
\begin{equation}\label{equ1}
u_1:H^1_{(\varphi,\Gamma)}(\R(\varepsilon x^{k_2-k_3})) \lra H^1_{(\varphi,\Gamma)}(\R(\varepsilon x^{-1}))
\end{equation}
o\`u $\dim_E H^1_{(\varphi,\Gamma)}(\R(\varepsilon x^{k_2-k_3}))=2$ et $\dim_E H^1_{(\varphi,\Gamma)}(\R(\varepsilon x^{-1}))=1$. Un argument comme dans la preuve du (ii) du Lemme \ref{lem1} (en plus simple) donne que $u_1$ est surjectif. Par ailleurs il suit facilement de \cite[Prop. 2.7]{Na1} que $\dim_EH^1_f(\R(\varepsilon x^{k_2-k_3}))=1$ et $H^1_f(\R(\varepsilon x^{-1}))=0$, de sorte que $H^1_f(\R(\varepsilon x^{k_2-k_3}))$ est inclus dans le noyau de (\ref{equ1}), et qu'il lui est m\^eme \'egal puisque $u_1$ est surjectif en comparant les dimensions. Donc $\ker(u_1)$ s'identifie aux extensions qui sont cristallines. Comme $N^2\neq 0$ sur $D_{\mathrm dR}(D)$, le quotient $\!\begin{xy}(-60,0)*+{\R(\varepsilon x^{k_2})}="a";(-36,0)*+{\R(\varepsilon x^{k_3})}="b";{\ar@{-}"a";"b"}\end{xy}\!$ de $D$ n'est pas cristallin. Par le diagramme commutatif et la discussion ci-dessus, on en d\'eduit que le quotient $\!\begin{xy}(-60,0)*+{\R(\varepsilon x^{k_2})}="a";(-34,0)*+{\R(\varepsilon x^{k_2+1})}="b";{\ar@{-}"a";"b"}\end{xy}\!$ de $u(D)$ est une extension {\it non scind\'ee}, i.e. $u(D)$ est de la forme $\!\begin{xy}(-60,0)*+{\R(\varepsilon^2 x^{k_1})}="a";(-34,0)*+{\R(\varepsilon x^{k_2})}="b";(-8,0)*+{\R(x^{k_2+1})}="c";{\ar@{-}"a";"b"};{\ar@{-}"b";"c"}\end{xy}\!$ et donc~:
\begin{equation}\label{uD}
u(D)^{\vee}\simeq \!\begin{xy}(-60,0)*+{\R(x^{-k_2-1})}="a";(-29,0)*+{\R(\varepsilon^{-1} x^{-k_2})}="b";(2.5,0)*+{\R(\varepsilon^{-2} x^{-k_1})}="c";{\ar@{-}"a";"b"};{\ar@{-}"b";"c"}\end{xy}\!.
\end{equation}
Comme $u(D)^{\vee}$ est de de Rham avec en quotient $D_1^{\vee}$ qui n'est pas cristallin, on en d\'eduit facilement que le ``pull-back'' de $u(D)^{\vee}$ le long de $\R(\varepsilon^{-2}x^{-k_2+1})\hookrightarrow \R(\varepsilon^{-2} x^{-k_1})$ est isomorphe \`a $D_0^{\vee}$, puis le r\'esultat.
\end{proof}

On note $\Ext^1_{\rm HT}(D_1,D_1)$ (resp. $\Ext^1_g(D_1,D_1)\subseteq \Ext^1_{\rm HT}(D_1,D_1)$) le sous-$E$-espace vectoriel de $\Ext^1_{(\varphi,\Gamma)}(D_1,D_1)$ des extensions qui sont de Hodge-Tate (resp. de de Rham). En voyant un \'el\'ement de $\Ext^1_{(\varphi,\Gamma)}(D_1,D_1)$ comme un $(\varphi,\Gamma)$-module $\widetilde{D_1}$ sur ${\mathcal R}_{E[\epsilon]/(\epsilon^2)}\=\R\otimes_EE[\epsilon]/(\epsilon^2)$ qui d\'eforme $D_1$, on d\'efinit $\Ext^1_Z(D_1, D_1)\subseteq \Ext^1_{(\varphi,\Gamma)}(D_1,D_1)$ comme le sous-espace des extensions telles que $\wedge^2_{{\mathcal R}_{E[\epsilon]/(\epsilon^2)}}\widetilde{D_1}\simeq \wedge^2_{{\mathcal R}_{E}}{D_1}\otimes_EE[\epsilon]/(\epsilon^2)$. Enfin on pose $\Ext^1_{{\rm HT},Z}(D_1, D_1)\=\Ext^1_{\rm HT}(D_1, D_1)\cap \Ext^1_{Z}(D_1, D_1)$.

\begin{lem}\label{HT1}
On a $\dim_E\Ext^1_{\rm HT}(D_1,D_1)=3$ et $\dim_E \Ext^1_{{\rm HT},Z}(D_1,D_1)=2$.
\end{lem}
\begin{proof}
Soit $\widetilde{D}_1$ une d\'eformation de $D_1$ sur ${\mathcal R}_{E[\epsilon]/(\epsilon^2)}$, les poids de Sen de $\widetilde{D}_1$ sont de la forme $(k_1+2+d_1\epsilon, k_2+1+d_2\epsilon)$ pour $(d_1,d_2)\in E^2$, ce qui donne une application $E$-lin\'eaire naturelle~:
\begin{equation}\label{sen}
\Ext^1_{(\varphi,\Gamma)}(D_1,D_1) \longrightarrow E^2, \ \ [\widetilde{D}_1] \mapsto (d_1, d_2).
\end{equation}
Montrons que cette application est surjective. Si $D_1$ est cristallin, cela suit facilement de \cite[Prop. 2.3.10]{BChe} appliqu\'e au raffinement $\!\begin{xy}(-60,0)*+{\R(\varepsilon x^{k_1+1})}="a";(-32,0)*+{\R(\varepsilon^2 x^{k_2-1})}="b";{\ar@{-}"a";"b"}\end{xy}\!$ de $D_1$. Si $D_1$ est (semi-stable) {\it non} cristallin, pour tout $(d_1,d_2)\in E^2$, en utilisant \cite[Th.~2.7]{BD} (appliqu\'e \`a $D_1$) il n'est pas difficile de trouver $\widetilde{D}_1\in \Ext^1_{(\varphi,\Gamma)}(D_1,D_1)$ triangulin de param\`etre $(\widetilde{\delta}_1, \widetilde{\delta}_2)$ tel que $(k_1+2+d_1\epsilon,k_2+1+d_2\epsilon)$ sont les poids de Sen respectifs de $\widetilde{\delta}_1,\widetilde{\delta}_2$ (la condition dans {\it loc.cit.} sur $\widetilde{\delta}_1,\widetilde{\delta}_2$ donne en fait ici une condition vide sur $d_1,d_2$). On a par ailleurs $[\widetilde{D}_1]\in \Ext^1_{\rm HT}(D_1,D_1)$ si et seulement si $d_1=d_2=0$~: utiliser par exemple que l'on peut toujours tordre $D_1$ par un caract\`ere non ramifi\'e pour le rendre \'etale, donc provenant d'une repr\'esentation galoisienne de Hodge-Tate. Comme $\dim_E \Ext^1_{(\varphi,\Gamma)}(D_1,D_1)=5$ (\cite[Lem. 3.5]{BD}), on obtient la premi\`ere \'egalit\'e. Pour la deuxi\`eme, notons que la restriction de (\ref{sen}) \`a $\Ext^1_Z(D_1,D_1)$ a pour image la droite $\{(x,-x),x\in E\} \hookrightarrow E^2$. Comme $\dim_E \Ext^1_Z(D_1,D_1)=3$ (\cite[Lem. 3.9]{BD}), on en d\'eduit $\dim_E \Ext^1_{{\rm HT},Z}(D_1,D_1)=2$.
\end{proof}

\begin{lem}\label{HT1bis}
On a $\dim_E \Ext^1_{\rm HT}(D_1, \R(\varepsilon x^{k_2}))=2$.
\end{lem}
\begin{proof}
L'isomorphisme $\Ext^1_{(\varphi,\Gamma)}(D_1, \R(\varepsilon x^{k_2}))\cong H^1_{(\varphi,\Gamma)}(D_1^{\vee}(\varepsilon x^{k_2}))$ induit un isomorphisme $\Ext^1_{\rm HT}(D_1, \R(\varepsilon x^{k_2}))\cong H^1_{\rm HT}(D_1^{\vee}(\varepsilon x^{k_2}))$. Par ailleurs une preuve analogue \`a celle de (\ref{filw}) donne $W_{\rm dR}^+(D_1^{\vee}(\varepsilon x^{k_2}))\cong (B_{\rm dR}^+ \oplus t^{-(k_1+1-k_2)} B_{\rm dR}^+)\otimes_{\Q_p} E$, qui implique en particulier (avec $H^1(\gp, t^j B_{\mathrm{dR}}^+/t^{j+1} B_{\mathrm{dR}}^+)=0$ si $j\ne 0$)~:
\begin{equation*}
\dim_E H^1\big(\gp, W_{\rm dR}^+(D_1^{\vee}(\varepsilon x^{k_2}))/(t)\big)=1.
\end{equation*}
Comme $\dim_E \Ext^1_{(\varphi,\Gamma)}(D_1, \R(\varepsilon x^{k_2}))=3$ (cf. \cite[Lem. 3.5]{BD}), on en d\'eduit $\dim_E \Ext^1_{\rm HT}(D_1, \R(\varepsilon x^{k_2}))\geq 2$ par la d\'efinition de $H^1_{\rm HT}(D_1^{\vee}(\varepsilon x^{k_2}))$. Par ailleurs, l'injection (venant de $D_1\twoheadrightarrow \R(\varepsilon x^{k_2})$)~:
\begin{equation*}
\Ext^1_{(\varphi,\Gamma)}(\R(\varepsilon x^{k_2}), \R(\varepsilon x^{k_2})) \hookrightarrow \Ext^1_{(\varphi,\Gamma)}(D_1, \R(\varepsilon x^{k_2}))
\end{equation*}
induit un isomorphisme (en utilisant que tout sous-$(\varphi,\Gamma)$-module d'un $(\varphi,\Gamma)$-module de Hodge-Tate tel que le conoyau n'a pas de torsion est encore de Hodge-Tate)~:
\begin{equation}\footnotesize\label{interHT}
\Ext^1_{\rm HT}(\R(\varepsilon x^{k_2}), \R(\varepsilon x^{k_2}))\cong \Ext^1_{\rm HT}(D_1, \R(\varepsilon x^{k_2}))\cap \Ext^1_{(\varphi,\Gamma)}(\R(\varepsilon x^{k_2}), \R(\varepsilon x^{k_2}))
\end{equation}
(le terme de droite est clairement inclus dans celui de gauche, qui a seulement dimension $1$, d'o\`u un isomorphisme). Comme par ailleurs~:
$$\Ext^1_{\rm HT}(\R(\varepsilon x^{k_2}), \R(\varepsilon x^{k_2}))\cong \Ext^1_g(\R(\varepsilon x^{k_2}), \R(\varepsilon x^{k_2}))\cong H^1_g(\R)\subsetneq H^1_{(\varphi,\Gamma)}(\R),$$
il existe des extensions de $\R(\varepsilon x^{k_2})$ par $\R(\varepsilon x^{k_2})$ qui ne sont pas de Hodge-Tate, d'o\`u on d\'eduit par (\ref{interHT})~:
\begin{equation*}
\Ext^1_{\rm HT}(D_1, \R(\varepsilon x^{k_2}))\subsetneq \Ext^1_{(\varphi,\Gamma)}(D_1, \R(\varepsilon x^{k_2})).
\end{equation*}
Cela implique $\dim_E \Ext^1_{\rm HT}(D_1, \R(\varepsilon x^{k_2}))\leq 2$ et termine la preuve.
\end{proof}

\begin{rem}
{\rm Si $D$ est un $(\varphi, \Gamma)$-module et $D'$ un sous-$(\varphi, \Gamma)$-module de m\^eme rang que $D$, alors $D'$ de Hodge-Tate {\it n}'implique {\it pas} $D$ de Hodge-Tate (contrairement au cas de Rham). Par exemple, si $D$ est une extension $\R$ par $\R$ qui n'est pas de Hodge-Tate, son ``pull-back'' le long de $t^n \R \hookrightarrow \R$ est de Hodge-Tate pour tout $n>0$.}
\end{rem}

Comme pour (\ref{interHT}), on v\'erifie que la surjection $\kappa$ en (\ref{hta}) induit une suite exacte~:
\begin{equation}\label{exa}
0 \lra \Ext^1_{\HT}(D_1, \R(\varepsilon^2 x^{k_1})) \lra \Ext^1_{\HT}(D_1,D_1) \buildrel {\kappa}\over\longrightarrow \Ext^1_{\HT}(D_1, \R(\varepsilon x^{k_2})).
\end{equation}

\begin{lem}\label{HTg}
(i) On a $\Ext^1_g(D_1, \R(\varepsilon^2 x^{k_1}))\buildrel\sim\over\rightarrow \Ext^1_{\HT}(D_1, \R(\varepsilon^2 x^{k_1}))$.\\
(ii) Supposons $D_1$ non cristallin, alors $\dim_E \Ext^1_{\HT}(D_1, \R(\varepsilon^2 x^{k_1}))=1$ et l'application $\kappa$ en (\ref{exa}) est surjective.
\end{lem}
\begin{proof}
(i) Notons d'abord que $\Ext^1_{*}(D_1, \R(\varepsilon^2 x^{k_1}))\cong H^1_{*}(D_1^{\vee}(\varepsilon^2 x^{k_1}))$ o\`u $*\in \{(\varphi,\Gamma), g, \HT\}$. Consid\'erons les applications naturelles~:
\begin{multline}\label{comp}
H^1_{(\varphi,\Gamma)}(D_1^{\vee}(\varepsilon^2 x^{k_1})) \lra H^1\big(\gp, W_{\dR}^+(D_1^{\vee}(\varepsilon^2 x^{k_1}))\big)\\
\lra H^1\big(\gp, W_{\dR}^+(D_1^{\vee}(\varepsilon^2 x^{k_1}))/(t)\big).
\end{multline}
Comme pour (\ref{filw}) on montre $W_{\dR}^+(D_1^{\vee}(\varepsilon^2 x^{k_1}))\cong (B_{\dR}^+ \oplus t^{k_1-k_2+1} B_{\dR}^+)\otimes_{\Q_p} E$, d'o\`u on d\'eduit que la seconde application en (\ref{comp}) est une bijection de $E$-espaces vectoriels de dimension $1$. Par \cite[Def.~2.4\ \&\ Lem. 2.6]{Na1} on en d\'eduit (i).\\
(ii) Supposons $D_1$ non cristallin, comme $D_1(\varepsilon^{-1} x^{-k_1})\simeq \!\begin{xy}(-60,0)*+{\R(\varepsilon)}="a";(-36,0)*+{\R(x^{k_2-k_1})}="b";{\ar@{-}"a";"b"}\end{xy}\!$ est aussi (semi-stable) non cristallin, on d\'eduit de \cite[Prop. 2.7]{Na1} que l'on a $H^1_e(D_1(\varepsilon^{-1} x^{-k_1}))=H^1_f(D_1(\varepsilon^{-1} x^{-k_1}))\simeq E$. Par ailleurs par \cite[Lem. 3.5]{BD} on a $\dim_E H^1_{(\varphi,\Gamma)}(D_1^{\vee}(\varepsilon^2 x^{k_1}))=2$. Il suit alors de \cite[Prop. 2.11]{Na1} (avec \cite[Th.~5.11]{Na2}) et $(D_1^{\vee}(\varepsilon^2 x^{k_1}))^{\vee} (\varepsilon) \cong D_1(\varepsilon^{-1} x^{-k_1})$ que l'on a $\dim_E H^1_{g}(D_1^{\vee}(\varepsilon^2 x^{k_1}))=1=\dim_E \Ext^1_g(D_1, \R(\varepsilon^2 x^{k_1}))$. On en d\'eduit le (ii) avec (i), le Lemme \ref{HT1} et le Lemme \ref{HT1bis}.
\end{proof}

\begin{rem}\label{noncris}
{\rm Dans le (ii) du Lemme \ref{HTg}, si $D_1$ est cristallin on peut montrer que $\dim_E \Ext^1_{\HT}(D_1, \R(\varepsilon^2 x^{k_1}))=2$, et dans ce cas $\kappa$ en (\ref{exa}) n'est donc pas surjective (par le Lemme \ref{HT1} et le Lemme \ref{HT1bis}).}
\end{rem}

\begin{lem}\label{imj}
On a $\im(j)= \Ext^1_{\HT}(D_1, \R(\varepsilon x^{k_2}))\subseteq \Ext^1_{(\varphi,\Gamma)}(D_1, \R(\varepsilon x^{k_2}))$ (voir (\ref{ht1}) pour $j$).
\end{lem}
\begin{proof}
Soit $k_2'$ comme en (\ref{ht1}), on a un diagramme commutatif (en notant \`a droite $H^1(-)$ pour $H^1(\gp,-)$)~:
\begin{equation*}
\begin{CD}
H^1_{(\varphi,\Gamma)}(D_1^{\vee}(\varepsilon x^{k_2'}))@>>> H^1\big(t^{k_2'-k_2} W_{\dR}^+(D_1^{\vee}(\varepsilon x^{k_2})\big)\\
@V j VV @V j' VV @. \\
H^1_{(\varphi,\Gamma)}(D_1^{\vee} (\varepsilon x^{k_2})) @>>> H^1\big(W_{\dR}^+(D_1^{\vee}(\varepsilon x^{k_2}))\big) @> v >> H^1\big(W_{\dR}^+(D_1^{\vee}(\varepsilon x^{k_2}))/(t)\big).
\end{CD}
\end{equation*}
Comme $k'_2-k_2\geq 1$, on a $v\circ j'=0$, donc $\im(j)\subseteq H^1_{\HT}(D_1^{\vee} (\varepsilon x^{k_2}))$. Mais par l'injectivit\'e de $j$ ((ii) du Lemme \ref{lem1}), le (i) du Lemme \ref{lem1} et le Lemme \ref{HT1bis}, ces deux espaces ont m\^eme dimension ($=2$), d'o\`u le r\'esultat.
\end{proof}

\begin{prop}\label{prop}\label{D1fin}
(i) L'accouplement du bas en (\ref{hta}) induit un accouplement parfait $\Ext^1_{(\varphi,\Gamma)}(\R(x^{k_2+1}), D_1) \times \Ext^1_{\HT}(D_1, \R(\varepsilon x^{k_2}))\buildrel{\cup}\over\longrightarrow E$.\\
(ii) Si $D_1$ est non cristallin, on a un diagramme commutatif d'accouplements compatible via la surjection $u$ avec le diagramme commutatif (\ref{hta})~:
\begin{equation}\label{ht00}
\begin{CD}
\Ext^1_{(\varphi,\Gamma)}(\R(x^{k_2+1}), D_1) \ \ \ \ @. \ \times @. \Ext^1_{\HT}(D_1, D_1) @> \cup>> E \\
@| @. @V \kappa VV @| \\
\Ext^1_{(\varphi,\Gamma)}(\R(x^{k_2+1}), D_1) \ \ @. \ \times @. \ \ \ \Ext^1_{\HT}(D_1, \R(\varepsilon x^{k_2}))@> \cup >> E.
\end{CD}
\end{equation}
\end{prop}
\begin{proof}
Le Lemme \ref{imj} et l'injectivit\'e de $j$ ((ii) du Lemme \ref{lem1}) permettent de remplacer $\Ext^1_{(\varphi,\Gamma)}(D_1, \R(\varepsilon x^{k_2'}))$ par $\im(j)=\Ext^1_{\HT}(D_1, \R(\varepsilon x^{k_2}))$ dans l'accouplement parfait du bas en (\ref{ht1}), d'o\`u (i). L'accouplement du haut en (\ref{ht00}) est alors d\'efini en d\'ecr\'etant que $\ker(\kappa)$ annule $\Ext^1_{(\varphi,\Gamma)}(\R(x^{k_2+1}), D_1)$. Les derni\`eres assertions du (ii) d\'ecoulent du (ii) du Lemme \ref{HTg} et du (i).
\end{proof}

\begin{rem}
{\rm Dans le (ii) de la Proposition \ref{D1fin}, si $D_1$ est cristallin on a encore un diagramme commutatif comme en (\ref{ht00}), mais dans ce cas $\kappa$ n'est plus surjective (Remarque \ref{noncris}) et on ne peut utiliser $\Ext^1_{\HT}(D_1, D_1)$ pour ``caract\'eriser'' des vecteurs dans $\Ext^1_{(\varphi,\Gamma)}(\R(x^{k_2+1}), D_1)$.}
\end{rem}

En rempla\c cant $D_1$ par une extension (semi-stable) non cristalline $D_2\simeq \!\begin{xy}(-60,0)*+{\R(\varepsilon x^{k_2})}="a";(-36,0)*+{\R(x^{k_3})}="b";{\ar@{-}"a";"b"}\end{xy}\!$, on obtient la proposition suivante par des arguments similaires.

\begin{prop}\label{ht02}
(i) L'injection $\R(\varepsilon^2 x^{k_1})\hookrightarrow \R(\varepsilon^2 x^{k_2-1})$ induit une application surjective~:
\begin{equation*}
u': \Ext^1_{(\varphi,\Gamma)}(D_2, \R(\varepsilon^2 x^{k_1})) \twoheadrightarrow \Ext^1_{(\varphi,\Gamma)}(D_2, \R(\varepsilon^2 x^{k_2-1})).
\end{equation*}
(ii) Soit $D$ une extension de $D_2$ par $\R(\varepsilon^2 x^{k_1})$ avec $N^2\neq 0$ sur $D_{\dR}(D)=D_{\rm st}(D)$ et $u'(D)$ le ``push-forward'' de $D$ le long de $\R(\varepsilon^2x^{k_1})\hookrightarrow \R(\varepsilon^2 x^{k_2-1})$, alors $u'(D)$ est une extension (non scind\'ee) de $\R(x^{k_3})/t^{k_2+1-k_3}$ par $D_0$.\\
(iii) Le diagramme commutatif d'accouplements (cf. \cite[(2.8)]{BD})~:
\begin{equation*}
\begin{CD}
\Ext^1_{(\varphi,\Gamma)}(D_2, \R(\varepsilon^2 x^{k_1})) \ \ \ \ @. \times @. \!\!\!\Ext^1_{(\varphi,\Gamma)}(D_2, D_2) @> \cup>> E \\
@| @. @V \kappa VV @| \\
\Ext^1_{(\varphi,\Gamma)}(D_2, \R(\varepsilon^2 x^{k_1})) \ \ \ \ @. \times @. \ \ \ \ \Ext^1_{(\varphi,\Gamma)}(\R(\varepsilon x^{k_2}), D_2)@> \cup >> E
\end{CD}
\end{equation*}
induit un diagramme commutatif d'accouplements~:
\begin{equation*}
\begin{CD}
\Ext^1_{(\varphi,\Gamma)}(D_2, \R(\varepsilon^2 x^{k_2-1})) \ \ \ \ @. \times @. \Ext^1_{\HT}(D_2, D_2) @> \cup>> E \\
@| @. @V \kappa VV @| \\
\Ext^1_{(\varphi,\Gamma)}(D_2, \R(\varepsilon^2 x^{k_2-1})) \ \ \ \ @. \times @. \ \ \ \ \Ext^1_{\HT}(\R(\varepsilon x^{k_2}), D_2)@> \cup >> E
\end{CD}
\end{equation*}
o\`u l'application $\kappa$ est surjective et l'accouplement du bas est parfait.
\end{prop}

Par \cite[Lem. 3.29]{BD}, il existe $a\in E^{\times}$ et une repr\'esentation $\rho_1: \gp \ra \GL_2(E)$ tels que $D_1\cong D_{\rm rig}(\rho_1)\otimes_{\R} \R(\nr(a))$ (o\`u $D_{\rm rig}(\rho_1)$ est le $(\varphi,\Gamma)$-module sur $\R$ associ\'e \`a $\rho_1$). On note $\pi^{\an}(D_1) \=\pi^{\an}(\rho_1) \otimes \nr(a)\circ{\det}$ o\`u $\pi^{\an}(\rho_1)$ est la repr\'esentation localement analytique de $\GL_2(\Q_p)$ associ\'ee \`a $\rho_1$ via la correspondance de Langlands localement analytique pour $\GL_2(\Q_p)$ (\cite{Co2}, \cite{CD}). On suppose que $\rho_1$ admet un $\oE$-r\'eseau invariant dont la r\'eduction $\overline{\rho_1}$ satisfait \cite[(A.2)]{BD}. Par \cite[Prop. 3.30]{BD}, on a alors une bijection naturelle~:
\begin{equation}\label{pLL}
{\rm pLL}: \Ext^1_{(\varphi,\Gamma)}(D_1, D_1) \buildrel{\sim}\over\longrightarrow \Ext^1_{\GL_2(\Q_p)}(\pi^{\an}(D_1), \pi^{\an}(D_1)).
\end{equation}
Noter que la normalisation de la correspondance de Langlands localement analytique (ou $p$-adique) que l'on utilise ici est celle de \cite[\S~3.2.3]{BD} tordue par le caract\`ere $\varepsilon^{-1} \circ {\det}$. Explicitement, notons $\Hom(\Q_p^{\times}, E)$ les morphismes de groupes continus pour la structure additive sur $E$, et soit $0\neq \eta \in {\mathcal L}_{FM}(D_1: \R(\varepsilon^2 x^{k_1}))\subseteq \Hom(\Q_p^{\times}, E)$ (cf. \cite[Cor. 2.9]{BD} et noter que $D_1$ est d\'etermin\'e par $\varepsilon^2 x^{k_1}$, $\varepsilon x^{k_2}$ et $\eta$), alors on a $\pi^{\an}(D_1) \cong \pi(\nu_{1,2}, \eta) \otimes \varepsilon \circ {\det}$ o\`u $\nu_{1,2}\=(k_1, k_2)$ et la repr\'esentation $\pi(\nu_{1,2}, \eta)$ est la repr\'esentation localement analytique de longueur finie d\'efinie dans \cite[(3.26)]{BD} (not\'ee $\pi(\lambda,\psi)$ dans {\it loc.cit.}).

Si $V$, $W$ sont des repr\'esentations admissibles dans $\Rep(\GL_2(\Q_p))$ telles que le centre $Z({\mathfrak gl}_2)$ de $\text{U}({\mathfrak gl}_2)$ (resp. le centre $Z_{{\rm GL}_2}(\Q_p)$) agit sur $V$, $W$ par le m\^eme ca\-ract\`ere, on note $\Ext^1_{\inf}(W,V)$ (resp. $\Ext^1_Z(W,V)$, resp. $\Ext^1_{\inf,Z}(W,V)$) le sous-espace de $\Ext^1_{\GL_2(\Q_p)}(W,V)$ des extensions sur lesquelles $Z({\mathfrak gl}_2)$ (resp. $Z_{{\rm GL}_2}(\Q_p)$, resp. $Z({\mathfrak gl}_2)$ {\it et} $Z_{{\rm GL}_2}(\Q_p)$) agit par ce caract\`ere.

\begin{prop}\label{pLLht}
Avec les notations pr\'ec\'edentes, supposons de plus $\End_{\gp}(\overline{\rho_1})\cong k_E$, alors (\ref{pLL}) induit une bijection~:
\begin{equation}\label{htinf}
\Ext^1_{\HT}(D_1,D_1) \buildrel{\sim}\over\longrightarrow \Ext^1_{\inf}(\pi^{\an}(D_1),\pi^{\an}(D_1)).
\end{equation}
\end{prop}
\begin{proof}
Par (\ref{gHt}), la deuxi\`eme partie du (i) du Lemme \ref{inf3} (en tordant par $\varepsilon \circ {\det}$) et \cite[Lem. 3.25(2)]{BD}, on a si $\eta$ est lisse $\Ext^1_g(\pi^{\an}(D_1),\pi^{\an}(D_1))\cong \Ext^1_{\inf}(\pi^{\an}(D_1),\pi^{\an}(D_1))$. Si $\eta$ est lisse, par \cite[Lem. 3.11(3)]{BD} et le Lemme \ref{HT1} on a aussi $\Ext^1_g(D_1, D_1)\cong \Ext^1_{\HT}(D_1, D_1)$. La bijection en (\ref{htinf}) d\'ecoule alors de \cite[Lem. 3.28]{BD} dans ce cas.\\
Soit $\pi(\overline{\rho_1})$ la rep\'esentation de $\GL_2(\Q_p)$ sur $k_E$ associ\'ee \`a $\overline{\rho_1}$ via la correspondance de Langlands modulo $p$ (normalis\'ee de sorte que le caract\`ere central de $\pi(\overline{\rho_1})$ est $\overline{\varepsilon}^{-1} {\det}(\overline{\rho}_1)$), et $\widehat{\pi}(\rho_1)$ la rep\'esentation de Banach unitaire de $\GL_2(\Q_p)$ sur $E$ associ\'ee \`a $\rho_1$ via la correspondance de Langlands $p$-adique. On utilise les notations de \cite[\S~A.2]{BD}. Comme $\End_{\gp}(\overline{\rho_1})\cong k_E$ (le corps r\'esiduel de $E$), on a comme en \cite[(A.7)]{BD} un isomorphisme ${\rm Def}_{\pi(\overline{\rho_1}), {\rm ortho}} \xrightarrow{\sim} {\rm Def}_{\overline{\rho_1}}$ (noter que notre $\pi(\overline{\rho}_1)$ est celui de \emph{loc.cit.} tordu par $\overline{\varepsilon}^{-1} \circ {\det}$, ce qui est sans cons\'equence). Soit $\zeta\={\det}(\rho_1)$, alors $\zeta\varepsilon^{-1}$ est le caract\`ere central de $\widehat{\pi}(\rho_1)$. Comme en \cite[(A.9)]{BD} (on utilise ici aussi $\End_{\gp}(\overline{\rho_1})\cong k_E$), on a un isomorphisme ${\rm Def}_{\pi(\overline{\rho_1}), {\rm ortho}}^{\zeta\varepsilon^{-1}} \cong {\rm Def}_{\overline{\rho_1}}^{\zeta}$ o\`u $ {\rm Def}_{\pi(\overline{\rho_1}), {\rm ortho}}^{\zeta\varepsilon^{-1}} $ (resp. $ {\rm Def}_{\overline{\rho_1}}^{\zeta}$) d\'esigne le sous-foncteur de ${\rm Def}_{\pi(\overline{\rho_1}), {\rm ortho}}$ (resp. de ${\rm Def}_{\overline{\rho_1}}$) des d\'eformations avec caract\`ere central \'egal \`a $\zeta \varepsilon^{-1}$ (resp. avec d\'eterminant \'egal \`a $\zeta$). Par le m\^eme argument que dans la preuve de \cite[Cor. A.2]{BD}, on voit que le foncteur de Colmez $\bf V$ induit une bijection~:
\begin{equation}
{\bf V}: \Ext^1_{Z}(\widehat{\pi}(\rho_1), \widehat{\pi}(\rho_1)) \buildrel{\sim}\over\longrightarrow \Ext^1_Z(\rho_1, \rho_1)
\end{equation}
o\`u $\Ext^1_Z(\rho_1, \rho_1)$ d\'esigne les d\'eformations de $\rho_1$ sur $E[\epsilon]/(\epsilon^2)$ de d\'eterminant $\zeta$. Par la preuve de \cite[Prop. 3.30]{BD}, on en d\'eduit facilement que (\ref{pLL}) induit un isomorphisme~:
\begin{equation}\label{pLLz}
\Ext^1_{Z}(D_1, D_1)\buildrel{\sim}\over\longrightarrow \Ext^1_{\GL_2(\Q_p),Z}(\pi^{\an}(D_1), \pi^{\an}(D_1)).
\end{equation}
Soit $\widetilde{\rho}_1\in \Ext^1_Z(\rho_1,\rho_1)$, et $\widetilde{\pi}$ l'image r\'eciproque de $\widetilde{\rho}_1$ via $\bf V$. Par \cite[Th.~III.45]{CD}, on a un morphisme \'equivariant sous l'action de $\GL_2(\Q_p)$~:
\begin{equation}\label{pLLc}
\Pi_{\zeta \varepsilon^{-1}}(D(\widetilde{\rho}_1)_0) \longrightarrow \widetilde{\pi}/\widetilde{\pi}^{{\rm SL}_2(\Q_p)}
\end{equation}
dont les noyau et conoyau sont de dimension finie sur $E$, o\`u $D(\widetilde{\rho}_1)_0$ est le $(\varphi,\Gamma)$-module continu de $\widetilde{\rho}_1$ et le foncteur $\Pi_{\zeta \varepsilon^{-1}}$ est comme dans \emph{loc.cit.} Comme $\widetilde{\pi}$ est isomorphe \`a une extension de $\widehat{\pi}(\rho_1)$ par $\widehat{\pi}(\rho_1)$, on voit (par la structure de $\pi^{\an}(D_1)$) que $\widetilde{\pi}^{{\rm SL}_2(\Q_p)}=0$ et que $\widetilde{\pi}$ n'a pas de quotient de dimension finie \'equivariant sous l'action de $\GL_2(\Q_p)$. Donc (\ref{pLLc}) est une application surjective $\Pi_{\zeta \varepsilon^{-1}}(D(\widetilde{\rho}_1)_0) \twoheadrightarrow \widetilde{\pi}$. Par \cite[Th.~3.3]{Do} (qui fait l'hypoth\`ese que $\rho_1$ est absolument irr\'eductible, mais le m\^eme argument s'applique au cas o\`u $\End_{\Gal_{\Q_p}}(\rho_1)\cong E$), si $\widetilde{\rho}_1$ est de Hodge-Tate alors $\Pi_{\zeta \varepsilon^{-1}}(D(\widetilde{\rho}_1)_0)^{\an}$ a un caract\`ere infinit\'esimal, donc $\widetilde{\pi}^{\an}$ aussi par ce qui pr\'ec\`ede. On en d\'eduit que l'inverse de $\bf V$ induit une injection (o\`u $\pi^{\an}(\rho_1)=\widehat{\pi}(\rho_1)^{\an}$)~:
\begin{equation}
\Ext^1_{{\rm HT},Z}(\rho_1, \rho_1) \hookrightarrow \Ext^1_{\inf,Z}(\pi^{\an}(\rho_1), \pi^{\an}(\rho_1)).
\end{equation}
Donc (\ref{pLLz}) induit une injection~:
\begin{equation}\label{htz}
\Ext^1_{{\rm HT},Z} (D_1, D_1)\hookrightarrow \Ext^1_{\inf,Z}(\pi^{\an}(D_1),\pi^{\an}(D_1)).
\end{equation}
En outre, il est clair que $\Ext^1_g(D_1,D_1)$ n'est pas contenu dans $\Ext^1_Z(D_1,D_1)$ (cf. \cite[Th.~2.7~\&~Lem. 3.11(2)]{BD}). Par \cite[Lem. 3.11(3)]{BD} et le Lemme \ref{HT1}, on en d\'eduit~: 
\begin{equation*}
\Ext^1_g(D_1,D_1)+\Ext^1_{{\rm HT},Z}(D_1,D_1)=\Ext^1_{\rm HT}(D_1,D_1). 
\end{equation*}
L'injection (\ref{htz}) combin\'ee avec l'injection (\ref{gHt}) ci-dessous et le premier isomorphisme de \cite[Lem. 3.28]{BD} induisent alors une injection~:
\begin{equation}\label{hty}
\Ext^1_{\HT} (D_1, D_1)\hookrightarrow \Ext^1_{\inf}(\pi^{\an}(D_1),\pi^{\an}(D_1)).
\end{equation}
En comparant les dimensions (par le Lemme \ref{HT1} et le (i) du Lemme \ref{inf3}), on voit que (\ref{hty}) est une bijection.
\end{proof}

On note $\St_2(\nu_{1,2})\=L(\nu_{1,2})\otimes_E\St_2$ o\`u $L(\nu_{1,2})$ est la repr\'esentation alg\'ebrique de $\GL_2(\Q_p)$ sur $E$ de plus haut poids $\nu_{1,2}$ (par rapport au Borel sup\'erieur de $\GL_2(\Q_p)$), $\nu\=(k_1,k_2,k_3)=-\lambda$ et $v_{P_i^-}^{\infty}(\nu)\=L(\nu)\otimes_Ev_{P_i^-}^{\infty}$ o\`u $v_{P_i^-}^{\infty}\=(\Ind_{P_i^-(\Qp)}^{{\rm GL}_3(\Qp)}1)^\infty/1$ pour $i\in \{1,2\}$. Par \cite[(3.90)]{BD} on a une suite de morphismes~:
\begin{multline}\label{hL}
\Ext^1_{(\varphi,\Gamma)}(D_1, D_1) \buildrel{\sim}\over\longrightarrow \Ext^1_{\GL_2(\Q_p)}(\pi(\nu_{1,2},\eta),\pi(\nu_{1,2}, \eta))\\ \twoheadrightarrow \Ext^1_{\GL_2(\Q_p)}(\St_2(\nu_{1,2}), \pi(\nu_{1,2}, \eta)) \buildrel{\iota}\over\longrightarrow \Ext^1_{\GL_3(\Q_p)}(v_{P_2^-}^{\infty}(\nu), \Pi^1(\nu, \eta)^+)
\end{multline}
o\`u la premi\`ere application est (\ref{pLL}) (en tordant par $\varepsilon^{-1}\circ {\det}$ \`a droite), o\`u $\Pi^1(\nu, \eta)^+$ est la repr\'esentation d\'efinie au d\'ebut de \cite[\S~3.3.4]{BD} et o\`u l'on renvoie \`a \cite[(3.87)~\&~Lem.~3.42(1)]{BD} pour l'application $\iota$.

\begin{prop}\label{th1}
Supposons que $\eta$ ne soit pas lisse, alors l'accouplement parfait dans \cite[Th.~3.45]{BD} induit un diagramme commutatif d'accouplements parfaits (voir \cite[Rem. 3.41]{BD} pour $\Pi^1(\nu, \eta)$):
\begin{equation}\label{htaa}
\begin{CD}
\Ext^1_{(\varphi,\Gamma)}(\R(x^{k_3}), D_1) \ @. \times @. \ \ \Ext^1_{\GL_3(\Q_p)}(v_{P_2^-}^{\infty}(\nu), \Pi^1(\nu, \eta)^+) @> \cup>> E \\
@V u VV @. @AAA @| \\
\Ext^1_{(\varphi,\Gamma)}(\R(x^{k_2+1}), D_1) \ \ @. \times @. \ \Ext^1_{\GL_3(\Q_p)}(v_{P_2^-}^{\infty}(\nu), \Pi^1(\nu, \eta))@> \cup >> E.
 \end{CD}
\end{equation}
\end{prop}
\begin{proof}
Par la phrase qui suit \cite[(3.90)]{BD}, la compos\'ee en (\ref{hL}) se factorise \`a travers un isomorphisme~:
\begin{equation}\label{hL0}
\Ext^1_{(\varphi,\Gamma)}(D_1, \R(x^{k_2} \varepsilon)) \buildrel{\sim}\over\longrightarrow \Ext^1_{\GL_3(\Q_p)}\Ext^1_{\GL_3(\Q_p)}(v_{P_2^-}^{\infty}(\nu), \Pi^1(\nu, \eta)^+).
\end{equation}
Par la Proposition \ref{pLLht}, la surjection en (\ref{exb}) et la Proposition \ref{split}, (\ref{hL}) induit~:
\begin{multline}\label{hL2}
\Ext^1_{\HT}(D_1,D_1) \buildrel{\sim}\over\longrightarrow \Ext^1_{\inf}(\pi(\nu_{1,2},\eta),\pi(\nu_{1,2}, \eta))\\ \twoheadrightarrow \Ext^1_{\inf}(\St_2(\nu_{1,2}), \pi(\nu_{1,2}, \eta)) \buildrel{\iota\atop\sim}\over\longrightarrow \Ext^1_{\GL_3(\Q_p)}(v_{P_2^-}^{\infty}(\nu), \Pi^1(\nu, \eta)).
\end{multline}
Comme (\ref{hL}) se factorise \`a travers (\ref{hL0}), on voit que (\ref{hL2}) induit une surjection (via le (ii) du Lemme \ref{HTg})~:
\begin{equation*}
\Ext^1_{\HT}(D_1, \R(\varepsilon x^{k_2})) \twoheadrightarrow \Ext^1_{\GL_3(\Q_p)}(v_{P_2^-}^{\infty}(\nu), \Pi^1(\nu, \eta)),
\end{equation*}
qui est un fait un isomorphisme en comparant les dimensions (voir le Lemme 5.4.5 et la preuve de \cite[Prop. 3.49]{BD}). Par construction, on a donc un diagramme commutatif~:
\begin{equation*}
\begin{CD}
\Ext^1_{\HT}(D_1, \R(\varepsilon x^{k_2})) @>\sim >> \Ext^1_{\GL_3(\Q_p)}(v_{P_2^-}^{\infty}(\nu), \Pi^1(\nu, \eta)) \\
@VVV @VVV \\
\Ext^1_{(\varphi,\Gamma)}(D_1, \R(\varepsilon x^{k_2})) @>\sim >> \Ext^1_{\GL_3(\Q_p)}(v_{P_2^-}^{\infty}(\nu), \Pi^1(\nu, \eta)^+).
\end{CD}
\end{equation*}
La proposition en d\'ecoule avec la Proposition \ref{D1fin} et \cite[Th.~3.45]{BD}.
\end{proof}

Par le m\^eme argument en utilisant la Proposition \ref{ht02} et un analogue sym\'etrique de la Proposition \ref{split}, on obtient la proposition suivante.

\begin{prop}\label{th2}
Soit $0\neq \eta' \in {\mathcal L}_{FM}(D_2: \R(\varepsilon x^{k_2}))\subseteq \Hom(\Q_p^{\times}, E)$ (cf. \cite[Cor. 2.9]{BD}), alors l'accouplement parfait dans \cite[Th.~3.50]{BD} induit un diagramme commutatif d'accouplements parfaits~:
\begin{equation}\label{htbb}
\begin{CD}
\Ext^1_{(\varphi,\Gamma)}(D_2, \R(\varepsilon^2x^{k_1})) \ @. \times @. \ \ \Ext^1_{\GL_3(\Q_p)}(v_{P_1^-}^{\infty}(\nu), \Pi^2(\nu, \eta')^+) @> \cup>> E \\
@V u' VV @. @AAA @| \\
\Ext^1_{(\varphi,\Gamma)}(D_2, \R(\varepsilon^2 x^{k_2-1})) \ \ @. \times @. \ \Ext^1_{\GL_3(\Q_p)}(v_{P_1^-}^{\infty}(\nu), \Pi^2(\nu, \eta'))@> \cup >> E.
\end{CD}
\end{equation}
\end{prop}

On peut enfin montrer le (i) du Th\'eor\`eme \ref{gl3glob}. Quitte \`a tordre $\rho_p$ et les repr\'esentations de ${\rm GL}_3(\Qp)$ par $\nr(c^{-1})$, on peut supposer $c=1$. Soit $D_{\rm rig}(\rho_p)$ le $(\varphi,\Gamma)$-module de $\rho_p$ sur $\R$ et $D_1$ (resp. $D_2$) l'unique sous-$(\varphi,\Gamma)$-module (resp. l'unique quotient) de $D_{\rm rig}(\rho_p)$ de rang $2$, donc $E[D_{\rm rig}(\rho_p)]\subseteq \Ext^1_{(\varphi,\Gamma)}(\R(x^{k_3}),D_1) $ et aussi $E[D_{\rm rig}(\rho_p)]\subseteq \Ext^1_{(\varphi,\Gamma)}(D_2, \R(\varepsilon^2x^{k_1}))$. Notons $\pi^{e_1-e_2}(\rho_p)^-$ l'unique extension de $v_{P_2^-}^{\infty}(\nu)$ par $\Pi^1(\nu, \eta)$ (avec les notations pr\'ec\'edentes) telle que $E[\pi^{e_1-e_2}(\rho_p)^-]\subseteq \Ext^1_{\GL_3(\Q_p)}(v_{P_2^-}^{\infty}(\nu), \Pi^1(\nu, \eta))$ est l'orthogonal de $E[u(D_{\rm rig}(\rho_p))]$ via l'accouplement parfait du bas de (\ref{htaa}) (o\`u les espaces ont dimension $2$). Donc la repr\'esentation $\pi^{e_1-e_2}(\rho_p)^-$ est unique \`a isomorphisme pr\`es, et elle d\'etermine et ne d\'epend que de $u(D_{\rm rig}(\rho_p))$. On d\'efinit $\pi^{e_2-e_3}(\rho_p)^-$ de mani\`ere analogue en utilisant $D_2$ et l'accouplement parfait du bas de (\ref{htbb}). La repr\'esentation $\pi^{e_2-e_3}(\rho_p)^-$ d\'etermine et ne d\'epend que de $u'(D_{\rm rig}(\rho_p))$. Il suit facilement de la Proposition \ref{lemfil}, du (ii) de la Proposition \ref{ht02}, du (ii) de la Proposition \ref{lienfilmax} et de la fonctorialit\'e dans \cite[Th.~A]{Be2} que la repr\'esentation $\pi^{\alpha}(\rho_p)^-$ pour $\alpha\in S$ d\'etermine et ne d\'epend que $h_1$, $h_2$, $h_3$, $D$ et $\Fil^{\max}_{\alpha}(\rho_p)$ (noter que $u(D_{\rm rig}(\rho_p))^\vee\otimes_{\R}\R(\varepsilon^3x^{k_1+k_2+k_3})$ est isomorphe au $(\varphi,\Gamma)$-module qui correspond via \cite[Th.~A]{Be2} \`a $\wedge^2_ED$ muni de la filtration \`a un cran ${\Fil}^{-h_1-h_2}(\wedge^2_ED)=\wedge^2_ED$, ${\Fil}^{-h_2-h_3}(\wedge^2_ED)={\Fil}^{-h_2}(D)\wedge {\Fil}^{-h_3}(D)=\Fil^{\max}_{e_1-e_2}(\rho_p)$). Par la fin de \cite[\S ~3.3.4]{BD} et \cite[Rem. 4.6.3]{Br1}, la repr\'esentation $\pi^{\alpha}(\rho_p)$ contient $\pi^{\alpha}(\rho_p)^-$ et ces deux repr\'esentations se d\'eterminent l'une l'autre. Le (i) du Th\'eor\`eme \ref{gl3glob} en d\'ecoule.

\section{Appendice}

Le but de cet appendice est de montrer la Proposition \ref{split} ci-dessous, essentielle dans la preuve du (i) du Th\'eor\`eme \ref{gl3glob}.

\subsection{Calculs de Lie}

On commence par des calculs techniques d'invariants sous l'unipotent dans des modules de Verma g\'en\'eralis\'es pour ${\mathfrak gl}_3$. 

On a besoin de plusieurs notations~: $\fx\=\smat{0 & 1 \\ 0 & 0}$, $\fy\=\smat{0 & 0 \\ 1 & 0}$, $\fh\=\smat{1 & 0 \\ 0 & -1}$, $\fz\=\smat{1 & 0 \\ 0 & 1}$ et $\fc\=\fh^2-2\fh+4\fx \fy=\fh^2+2\fh+4\fy\fx\in \text{U}({\mathfrak gl}_2)$ l'\'el\'ement de Casimir. Rappelons que le centre $Z({\mathfrak gl}_2)$ de $\text{U}({\mathfrak gl}_2)$ est isomorphe \`a la $E$-alg\`ebre polynomiale $E[\fc,\fz]$. On note d\'esormais $P_1\=P_{e_1-e_2}=\smat{\GL_2 & * \\ 0& \GL_1 \\}\subset \GL_3$, $L_1\=L_{P_1}$, $\fp_1$ (resp. $\fl_1$) la $\Qp$-alg\`ebre de Lie de $P_1(\Qp)$ (resp. $L_1(\Qp)$), $N_1=N_{P_1}$, $\fn_1$ la $\Qp$-alg\`ebre de Lie de $N_1(\Qp)$, et on note avec un $-$ en exposant les oppos\'es~: $P_1^-$, ${\fp}_1^-$, ${\fn}^-_1$... On note aussi~:
$$\begin{array}{cccc}
\fx_1\=\smat{0 & 1 & 0 \\ 0 & 0 & 0\\ 0 & 0 & 0} & \fx_2\=\smat{0 & 0 & 0 \\ 0 & 0 & 1 \\ 0 & 0 & 0} & \fx_3\=\smat{0 & 0 & 1 \\ 0 & 0 & 0 \\ 0 & 0 & 0} & \fy_1\=\smat{0 & 0 & 0 \\ 1 & 0 & 0\\ 0 & 0 & 0}\\
\fy_2\=\smat{0 & 0 & 0 \\ 0 & 0 & 0 \\ 0 & 1 & 0} & \fy_3\=\smat{0 & 0 & 0 \\ 0 & 0 & 0 \\ 1 & 0 & 0} & \fh_1\=\smat{1 & 0 & 0 \\ 0 & -1 & 0 \\ 0 & 0 & 0} & \fh_2=\smat{0 & 0 &0 \\ 0 & 1 & 0 \\ 0 & 0 & -1}.
\end{array}$$
Nous utiliserons les relations de commutations suivantes entre ces \'el\'ements~:
\begin{equation}\label{form}
\begin{cases} \fx_2 \fy_2^{r_2}=&\fy_2^{r_2}\fx_2-r_2\fy_2^{r_2-1}((r_2-1)-\fh_2)\\
\fx_2 \fy_3^{r_3}=& \fy_3^{r_3}\fx_2+r_3\fy_3^{r_3-1} \fy_1\\
\fx_3 \fy_2^{r_2}=& \fy_2^{r_2}\fx_3+r_2 \fy_2^{r_2-1} \fx_1\\
\fx_3 \fy_3^{r_3}=&\fy_3^{r_3} \fx_3-r_3 \fy_3^{r_3-1}((r_3-1)-(\fh_1+\fh_2))\\
\fh_2 \fy_3^{r_3}=& \fy_3^{r_3} \fh_2+(-r_3) \fy_3^{r_3}\\
\fx_1 \fy_3^{r_3}=& \fy_3^{r_3} \fx_1-r_3 \fy_2\fy_3^{r_3-1}.
\end{cases}
\end{equation}

On fixe deux entiers $k_1, k_2\in \Z$ tels que $k_1\geq k_2\geq 0$ et $M$ un $\text{U}({\mathfrak gl}_2)\otimes_{\Qp}E$-module quelconque tel que $\fz(u)=(k_1+k_2)u$ pour tout $u\in M$. On voit $M$ comme $\text{U}(\fl_1)$-module via $\text{U}(\fl_1)\cong \text{U}({\mathfrak gl}_2)\otimes_E \text{U}({\mathfrak gl}_1)\twoheadrightarrow \text{U}({\mathfrak gl}_2)$, i.e. le facteur ${\mathfrak gl}_1$ de $\fl_1$ agit trivialement sur $M$, puis comme $\text{U}(\fp_1)$-module via $\fp_1 \twoheadrightarrow \fl_1$.

\begin{lem}\label{ninv}
(i) Le $E$-espace vectoriel $H^0(\fn_1, \text{U}({\mathfrak gl}_3) \otimes_{\text{U}(\fp_1)} M)=(\text{U}({\mathfrak gl}_3) \otimes_{\text{U}(\fp_1)} M)[\fn_1]$ est engendr\'e par les \'el\'ements de $M$ et les vecteurs de la forme $v=\sum_{i=0}^{r+1} \fy_2^i \fy_3^{r+1-i} \otimes v_i$ pour $r\in \Z_{\geq 0}$ et $v_i\in M$ v\'erifiant (avec $v_{i-1}=0$ si $i=0$)~:
\begin{equation}\label{dieq}\begin{cases}
i(\fh+2r-(k_1+k_2)) v_i =2(r+2-i)\fy v_{i-1}\\
(r+2-i)(\fh-2r+(k_1+k_2)) v_{i-1} =-2i\fx v_i.
\end{cases}
\end{equation}
(ii) Supposons qu'il existe $s\in \Z_{\geq 1}$ tel que $(\fc-(k_1-k_2)(k_1-k_2+2))^s(u)=0$ pour tout $u\in M$, alors pour $v\neq 0$ comme dans le (i) on a $r=k_2$ ou $r=k_1+1$ et $\fc(v_i)=(k_1-k_2)(k_1-k_2+2)v_i$ pour tout $v_i$.
\end{lem}
\begin{proof}
(i) Tout \'el\'ement $v\in \text{U}({\mathfrak gl}_3)\otimes_{\text{U}(\fp_1)} M \cong \text{U}({\fn}_1^-) \otimes_E M$ s'\'ecrit de mani\`ere unique~:
\begin{equation}\label{exp}
v=\sum_{r_2, r_3\in \Z_{\geq 0}} \fy_2^{r_2} \fy_3^{r_3} \otimes v_{r_2,r_3}.
\end{equation}
Un calcul utilisant (\ref{form}) donne~:
\begin{eqnarray*}
\fx_2 (\fy_2^{r_2} \fy_3^{r_3}\! \otimes \!v_{r_2,r_3})\!&\!\!=\!\!&\!\fy_2^{r_2-1} \fy_3^{r_3} \! \otimes \! \big(r_2 (\fh_2-r_2-r_3+1)v_{r_2,r_3}\big) + \fy_2^{r_2} \fy_3^{r_3-1}\! \otimes\! r_3 \fy_1 v_{r_2,r_3}\\
\fx_3(\fy_2^{r_2} \fy_3^{r_3}\!\otimes \!v_{r_2,r_3})\!&\!\!=\!\!&\!\fy_2^{r_2} \fy_3^{r_3-1}\! \otimes\!\big(r_3(\fh_1+\fh_2-r_2-r_3+1) v_{r_2,r_3}\big)+\fy_2^{r_2-1} \fy_3^{r_3} \!\otimes \! r_2 \fx_1 v_{r_2,r_3}.
\end{eqnarray*}
En utilisant (\ref{exp}) on en d\'eduit $\fx_2 v=\fx_3 v=0$ si et seulement si pour tout $r_2$, $r_3\in \Z_{\geq 0}$~:
\begin{equation}\label{equri}
\begin{cases}
(r_2+1)(\fh_2-r_2-r_3) v_{r_2+1,r_3}+(r_3+1) \fy_1 v_{r_2,r_3+1}=0\\
(r_3+1)(\fh_1+\fh_2-r_2-r_3) v_{r_2,r_3+1} + (r_2+1) \fx_1 v_{r_2+1, r_3}=0.
\end{cases}
\end{equation}
Via l'action de $\fl_1$ sur $M$ comme ci-dessus et l'hypoth\`ese sur l'action de $\mathfrak z$, on voit que $\fh_2$ agit comme $-\frac{\fh-k_1-k_2}{2}$ et $(\fh_1+\fh_2)$ comme $\frac{\fh+k_1+k_2}{2}$, et (\ref{equri}) est \'equivalent \`a~:
\begin{equation*}
\begin{cases}
(r_2+1)\big(-\frac{\fh-k_1-k_2}{2}-r_2-r_3\big)v_{r_2+1,r_3}+(r_3+1) \fy v_{r_2,r_3+1}=0\\
(r_3+1)\big(\frac{\fh+k_1+k_2}{2}-r_2-r_3\big) v_{r_2,r_3+1} + (r_2+1) \fx v_{r_2+1, r_3}=0.
\end{cases}
\end{equation*}
On en d\'eduit (i) en remarquant que $v-\sum_{r\in \Z_{\geq 0}}\big(\sum_{i=0}^{r+1} \fy_2^i \fy_3^{r+1-i} \!\otimes \! v_{i,r+1-i}\big)\in M$.\\
(ii) En appliquant $\fx$ \`a la premi\`ere \'equation en (\ref{dieq}) on a (avec $\fx\fh=\fh\fx-2\fx$)~:
\begin{equation*}
i(\fh+2r-2-(k_1+k_2)) \fx v_i\\= \fx i(\fh+2r-(k_1+k_2)) v_i= 2(r+2-i) \fx \fy v_{i-1}
\end{equation*}
et en appliquant $\fh+2r-2-(k_1+k_2)$ \`a la seconde~:
\begin{equation*}
-2i(\fh+2r-2-(k_1+k_2))\fx v_i = (r+2-i)(\fh+2r-2-(k_1+k_2))(\fh-2r+(k_1+k_2)) v_{i-1}.
\end{equation*}
Ces deux \'egalit\'es impliquent pour $i\in \{0,\dots,r+1\}$ (donc $r+2-i\ne 0$)~:
\begin{equation*}
(\fh^2-2\fh+4\fx \fy) v_{i-1}=\fc v_{i-1}=(k_1+k_2-2r)(k_1+k_2+2-2r) v_{i-1}.
\end{equation*}
De plus, par l'hypoth\`ese en (ii) on en d\'eduit pour $i\in \{0,\dots,r+1\}$ et $v_{i-1}\neq 0$~:
\begin{equation*}
(k_1+k_2-2r)(k_1+k_2+2-2r)=(k_1-k_2)(k_1-k_2+2),
\end{equation*}
d'o\`u $r=k_2$ ou $r=k_1+1$.
De m\^eme on obtient $\fc(v_{i})=(k_1-k_2)(k_1-k_2+2) v_i$ pour $i\in \{0,\dots,r+1\}$ en appliquant $\fh-2r+2+(k_1+k_2)$ \`a la premi\`ere \'equation en (\ref{dieq}) et $\fy$ \`a la seconde, d'o\`u (ii).
\end{proof}

Le lemme technique suivant sera important dans la suite.

\begin{lem}\label{ninv2}
Soit $u\in M$ tel que $\fc(u)=(k_1-k_2)(k_1-k_2+2)u$, alors on a~:
$$\sum_{i=0}^{k_2+1} \fy_2^i \fy_3^{k_2+1-i} \!\otimes \! u_i\in H^0\big(\fn_1, \text{U}({\mathfrak gl}_3)\otimes_{\text{U}(\fp_1)} M\big)$$
o\`u $u_0\=\fx^{k_2+1} u$ et $\displaystyle{u_i\=\Big(\prod_{j=1}^i \big(-\frac{k_2+2-j}{2j}\big)\Big)\Big(\prod_{j=1}^i \big(\fh+(k_1-k_2)+2j\big)\Big) \fx^{k_2+1-i} u}$.
\end{lem}
\begin{proof}
Posons $a_0\=1$ et, pour $i\in \{1,\dots,k_2+1\}$, $a_i\=\prod_{j=1}^i \big(-\frac{k_2+2-j}{2j}\big)$, alors $(-2i) a_i=(k_2+2-i) a_{i-1}$. On a (avec $\fh\fx=\fx\fh+2\fx$)~:
\begin{align*}
\MoveEqLeft[5.5] (k_2+2-i)(\fh+k_1-k_2) u_{i-1} & \\
={} & (k_2+2-i)(\fh+k_1-k_2)a_{i-1} \Big(\prod_{j=1}^{i-1} \big(\fh+(k_1-k_2)+2j\big)\Big)\fx^{k_2+2-i} u \\
={} & (k_2+2-i)a_{i-1}\fx(\fh+k_1-k_2+2)\Big(\prod_{j=1}^{i-1}\big(\fh+k_1-k_2+2j+2\big)\Big) \fx^{k_2+1-i} u\\
={} & -2i \fx u_i.
\end{align*}
Avec $\fc(u)=(k_1-k_2)(k_1-k_2+2)u$, on a aussi par ailleurs~:
\begin{align*}
\MoveEqLeft[3.5] i(\fh-(k_1-k_2)) u_i & \\
={} & i a_i \big(\fh^2+2\fh-(k_1-k_2)(k_1-k_2+2)\big) \Big(\prod_{j=2}^i \big(\fh+(k_1-k_2)+2j\big)\Big) \fx^{k_2+1-i} u\\
={} & i a_i (-4\fy\fx) \Big(\prod_{j=2}^i \big(\fh+(k_1-k_2)+2j\big)\Big) \fx^{k_2+1-i} u\\
={} & 2(k_2+2-i) a_{i-1} \fy \Big(\prod_{j=1}^{i-1} \big(\fh+(k_1-k_2)+2j\big)\Big) \fx^{k_2+2-i} u=2(k_2+2-i) \fy u_{i-1}
\end{align*}
en utilisant encore $\fx\fh=\fh\fx-2\fx$ pour l'avant-derni\`ere \'egalit\'e. Le lemme suit alors du Lemme \ref{ninv} appliqu\'e avec $r=k_2$.
\end{proof}

\subsection{Extensions avec un caract\`ere infinit\'esimal}

On d\'emontre la Proposition \ref{split}.

On conserve les notations du paragraphe pr\'ec\'edent et celles du \S~\ref{filmaxbis}. On rappelle que $\nu_{1,2}=(k_1,k_2)$, $\eta:\Q_p^{\times}\rightarrow E$ est un morphisme continu de groupes (pour la structure additive \`a droite), $\St_2(\nu_{1,2})=L(\nu_{1,2})\otimes_E\St_2$ et que l'on dispose de la repr\'esentation localement analytique de longueur finie $\pi(\nu_{1,2}, \eta)$ (\cite[(3.26)]{BD}). On suppose dans la suite que $\pi(\nu_{1,2}, \eta)$ v\'erifie \cite[Hyp.~3.19]{BD}. C'est une hypoth\`ese faible. En effet, par la discussion qui suit \cite[Lem.~3.29]{BD}, quitte \`a tordre $\pi(\nu_{1,2}, \eta)$ par un caract\`ere non ramifi\'e on peut supposer $\pi(\nu_{1,2}, \eta)\simeq \pi^{\an}(\rho)$ o\`u $\rho:\gp\rightarrow {\rm GL}_2(E)$ et $\pi^{\an}(\rho)$ correspond \`a $\rho$ via la correspondance de Langlands localement analytique pour $\GL_2(\Q_p)$ (\cite{Co2}, \cite{CD}). Si $\rho$ admet un $\oE$-r\'eseau invariant dont la r\'eduction $\overline{\rho}$ satisfait \cite[(A.2)]{BD}, alors \cite[Hyp.~3.19]{BD} est v\'erifi\'ee par \cite[Prop.~3.30]{BD}.

\begin{lem}\label{gl2}
(i) Le centre $Z({\mathfrak gl}_2)$ agit sur $\pi(\nu_{1,2},\eta)$ par l'unique caract\`ere $\xi$ tel que $\xi(\fz)=k_1+k_2$ et $\xi(\fc)=(k_1-k_2)(k_1-k_2+2)$.\\
(ii) On a $\dim_E \Ext^1_{\inf}(\St_2(\nu_{1,2}), \pi(\nu_{1,2},\eta))=2$.\\
(iii) On a $\dim_E \Ext^1_{\inf,Z}(\St_2(\nu_{1,2}),\pi(\nu_{1,2},\eta))=1$.
\end{lem}
\begin{proof}
(i) Il est clair que $Z({\mathfrak gl}_2)$ agit sur $\St_2(\nu_{1,2})$ par $\xi$. Par \cite[Prop. 3.7]{ST1}, pour $\nabla\in Z({\mathfrak gl}_2)$ on a un morphisme $\GL_2(\Q_p)$-\'equivariant $\nabla-\xi(\nabla): \pi(\nu_{1,2}, \eta) \ra \pi(\nu_{1,2}, \eta)$, qui induit donc un morphisme $\pi(\nu_{1,2},\eta)/\St_2(\nu_{1,2}) \ra \pi(\nu_{1,2}, \eta)$. Comme le socle de $\pi(\nu_{1,2}, \eta)$ est $\St_2(\nu_{1,2})$, on en d\'eduit (i).\\
(ii) Comme $\dim_E \Ext^1_{\GL_2(\Q_p)}(\St_2(\nu_{1,2}), \pi(\nu_{1,2}, \eta))=4$ (\cite[Lem. 3.17(ii)]{BD}), on peut associer \`a cet espace une extension $\widetilde{\pi}(\nu_{1,2}, \eta)$ de $\St_2(\nu_{1,2})^{\oplus 4}$ par $\pi(\nu_{1,2}, \eta)$. Soit $\nabla \in Z({\mathfrak gl}_2)$, le morphisme $\GL_2(\Q_p)$-\'equivariant $\nabla-\xi(\nabla): \widetilde{\pi}(\nu_{1,2}, \eta) \lra \widetilde{\pi}(\nu_{1,2}, \eta)$ est nul sur $\pi(\nu_{1,2},\eta)$ par (i), donc se factorise comme suit~:
\begin{equation*}
\widetilde{\pi}(\nu_{1,2}, \eta) \twoheadrightarrow \St_2(\nu_{1,2})^{\oplus 4} \lra \St_2(\nu_{1,2})\hookrightarrow \widetilde{\pi}(\nu_{1,2}, \eta).
\end{equation*}
On en d\'eduit que le sous-espace de $\Ext^1_{\GL_2(\Q_p)}(\St_2(\nu_{1,2}), \pi(\nu_{1,2}, \eta))$ des extensions sur lesquelles $\nabla$ agit par $\xi(\nabla)$ est au moins de dimension $3$. Appliquant ceci \`a $\nabla=\fc$ et $\nabla=\fz$, on en d\'eduit $\dim_E \Ext^1_{\inf}(\St_2(\nu_{1,2}), \pi(\nu_{1,2},\eta))\geq 2$ puisque $Z(\ug)\cong E[\fc, \fz]$. Soit $\pi(\nu_{1,2},\eta)^-$ la sous-repr\'esentation de $\pi(\nu_{1,2},\eta)$ d\'efinie en \cite[(3.23)]{BD}, par \cite[(3.29)]{BD} on a une injection~:
\begin{equation}\label{depi-api}
\Ext^1_{\GL_2(\Q_p)}(\St_2(\nu_{1,2}), \pi(\nu_{1,2}, \eta)^{-})\hookrightarrow \Ext^1_{\GL_2(\Q_p)}(\St_2(\nu_{1,2}), \pi(\nu_{1,2}, \eta)),
\end{equation}
et il suit de \cite[Lem.~3.17(2)~\&~(4)]{BD} que $\dim_E \Ext^1_{\GL_2(\Q_p)}(\St_2(\nu_{1,2}), \pi(\nu_{1,2}, \eta)^{-})=3$. De plus, par \cite[Lem. 3.20(2)]{BD}, le foncteur de Jacquet-Emerton (par rapport au Borel sup\'erieur) induit une bijection (avec les notations de {\it loc.cit.} \`a droite)~:
\begin{equation}\small\label{jac}
\Ext^1_{\GL_2(\Q_p)}(\St_2(\nu_{1,2}), \pi(\nu_{1,2}, \eta)^{-}) \buildrel {\sim}\over \longrightarrow \Ext^1_{T(\Q_p)}\big(\delta_{\nu_{1,2}}(|\cdot|\otimes |\cdot|^{-1}), \delta_{\nu_{1,2}}(|\cdot|\otimes |\cdot|^{-1})\big)_{\eta}.
\end{equation}
On voit une extension \`a droite dans (\ref{jac}) comme un caract\`ere du tore $T(\Qp)$ \`a valeurs dans $(E[\epsilon]/(\epsilon^2))^\times$ et on la note $\delta_{\nu_{1,2}}(|\cdot| \otimes |\cdot|^{-1}) (1+\Psi \epsilon)$ avec $\Psi=(\psi_1, \psi_2)$ o\`u $\psi_i:\Qp^\times\rightarrow E$ (par d\'efinition de l'espace \`a droite on a $\psi_1-\psi_2 \in E \eta$). Si $V$ est l'extension de $\St_2(\nu_{1,2})$ par $ \pi(\nu_{1,2},\eta)^-$ associ\'ee \`a $\delta_{\nu_{1,2}}(|\cdot| \otimes |\cdot|^{-1}) (1+\Psi \epsilon)$ par (\ref{jac}), il suit de \cite[Rem.~3.21]{BD} et de la d\'efinition du foncteur de Jacquet-Emerton que l'on a une injection $\mathfrak b$-\'equivariante~:
\begin{equation}\label{jacquet}
\delta_{\nu_{1,2}}(|\cdot| \otimes |\cdot|^{-1}) (1+\Psi \epsilon) \hookrightarrow V
\end{equation}
o\`u $\mathfrak b$ est la $\Qp$-alg\`ebre de Lie de $B(\Qp)$ et l'action de $\fn$ \`a gauche est triviale. Un examen de l'action de $\fz$ et $\fc=\fh^2+2\fh+4\fy\fx$ sur l'image de (\ref{jacquet}) (en utilisant que $\fx$ annule cette image) donne que, si $V$ a un caract\`ere infinit\'esimal et si $\eta$ n'est pas lisse, alors $\psi_1=\psi_2\in E\val$. Par (\ref{jac}) cela implique $\dim_E \Ext^1_{\rm inf}(\St_2(\nu_{1,2}), \pi(\nu_{1,2}, \eta)^{-})\leq 1$, i.e. $\dim_E\big(\Ext^1_{\GL_2(\Q_p)}(\St_2(\nu_{1,2}), \pi(\nu_{1,2}, \eta)^{-}) \cap \Ext^1_{\rm inf}(\St_2(\nu_{1,2}),\pi(\nu_{1,2}, \eta))\big)\leq 1$ qui implique $\dim_E \Ext^1_{\inf}(\St_2(\nu_{1,2}), \pi(\nu_{1,2},\eta))\leq 2$ (par les dimensions de ces deux espaces). On en d\'eduit (ii) dans le cas $\eta$ non lisse. Si $\eta$ est lisse, l'injection $\widetilde{I}(\nu_{1,2})\hookrightarrow \pi(\nu_{1,2},\eta)$ (cf. \cite[\S~3.2.2]{BD} pour $ \widetilde{I}(\nu_{1,2})$) induit un isomorphisme~:
\begin{equation*}
\Ext^1_{\GL_2(\Q_p)}(\St_2(\nu_{1,2}), \widetilde{I}(\nu_{1,2}))\buildrel {\sim}\over \longrightarrow \Ext^1_{\GL_2(\Q_p)}(\St_2(\nu_{1,2}), \pi(\nu_{1,2},\eta)).
\end{equation*}
Par une variante facile de la preuve de \cite[Lem. 3.20]{BD}, le foncteur de Jacquet-Emerton induit un isomorphisme~:
\begin{equation*}
\Ext^1_{\GL_2(\Q_p)}(\St_2(\nu_{1,2}), \widetilde{I}(\nu_{1,2})) \buildrel {\sim}\over \longrightarrow \Ext^1_{T(\Q_p)}(\delta_{\nu_{1,2}}, \delta_{\nu_{1,2}})
\end{equation*}
d'o\`u on d\'eduit comme pr\'ec\'edemment $\dim_E \Ext^1_{\inf}(\St_2(\nu_{1,2}), \pi(\nu_{1,2},\eta))\leq 2$.\\
(iii) Par \cite[Lem. 3.17]{BD} on a $\dim_E\Ext^1_{Z}(\St_2(\nu_{1,2}), \pi(\nu_{1,2}, \eta))=2$. Comme $\Ext^1_Z(\St_2(\nu_{1,2}), \St_2(\nu_{1,2}))=0$ alors que $\Ext^1_{\inf}(\St_2(\nu_{1,2}), \pi(\nu_{1,2},\eta))$ contient clairement (via $\St_2(\nu_{1,2})\hookrightarrow \pi(\nu_{1,2},\eta)$) l'unique extension localement alg\'ebrique non scind\'ee de $\St_2(\nu_{1,2})$ par lui-m\^eme, on en d\'eduit $\dim_E \Ext^1_{\inf,Z}(\St_2(\nu_{1,2}),\pi(\nu_{1,2},\eta))\leq 1$. Soit $\Ext^1_{\fz}(\St_2(\nu_{1,2}), \pi(\nu_{1,2}, \eta))$ le sous-espace des extensions sur lesquelles $\fz$ agit par $k_1+k_2$. On a clairement une inclusion~:
\begin{equation*}
\Ext^1_{Z}(\St_2(\nu_{1,2}), \pi(\nu_{1,2}, \eta)) + \Ext^1_{\inf}(\St_2(\nu_{1,2}), \pi(\nu_{1,2}, \eta)) \subseteq \Ext^1_{\fz}(\St_2(\nu_{1,2}), \pi(\nu_{1,2}, \eta)).
\end{equation*}
Comme $\Ext^1_{\fz}(\St_2(\nu_{1,2}), \pi(\nu_{1,2}, \eta)) \neq \Ext^1_{\GL_2(\Q_p)}(\St_2(\nu_{1,2}), \pi(\nu_{1,2}, \eta))$ (par exemple par (\ref{jac}) et (\ref{depi-api})), on en d\'eduit $\dim_E \Ext^1_{\inf,Z}(\St_2(\nu_{1,2}),\pi(\nu_{1,2},\eta))\geq 1$ en comparant les dimensions. Cela termine la preuve de (iii).
\end{proof}

Le lemme suivant est utilis\'e au \S~\ref{filmaxbis}.

\begin{lem}\label{inf3}
(i) On a une suite exacte courte~:
\begin{multline}\label{exb}
0 \ra \Ext^1_{\GL_2(\Q_p)}(\pi(\nu_{1,2}, \eta)/\St_2(\nu_{1,2}), \pi(\nu_{1,2},\eta)) \ra \Ext^1_{\inf}(\pi(\nu_{1,2},\eta), \pi(\nu_{1,2}, \eta))\\
\ra \Ext^1_{\inf}(\St_2(\nu_{1,2}), \pi(\nu_{1,2}, \eta)) \ra 0.
\end{multline}
En particulier, $\dim_E \Ext^1_{\inf}(\pi(\nu_{1,2}, \eta), \pi(\nu_{1,2}, \eta))=3$.\\
(ii) Soit $\Ext^1_g(\pi(\nu_{1,2},\eta), \pi(\nu_{1,2}, \eta))$ le sous-espace engendr\'e par les extensions $\widetilde \pi$ telles que les vecteurs localement alg\'ebriques de $\widetilde\pi$ contiennent {\rm strictement} $\St_2(\nu_{1,2})$ (cf. la discussion avant \cite[Lem.~3.25]{BD}). On a une injection naturelle~:
\begin{equation}\label{gHt}
\Ext^1_g(\pi(\nu_{1,2}, \eta), \pi(\nu_{1,2}, \eta))\hookrightarrow \Ext^1_{\inf}(\pi(\nu_{1,2}, \eta),\pi(\nu_{1,2}, \eta)).
\end{equation}
\end{lem}
\begin{proof}
(i) Par \cite[(3.28)]{BD}, on a une suite exacte comme en (\ref{exb}) avec $\Ext^1_{\inf}$ remplac\'e par $\Ext^1_{\GL_2(\Q_p)}$. Il suffit donc de montrer que toute extension $\widetilde{\pi}$ de $\pi(\nu_{1,2},\eta)$ par $\pi(\nu_{1,2},\eta)$ admet un caract\`ere infinit\'esimal si et seulement si sa sous-repr\'esentation $\widetilde{\pi}_0$ donn\'ee par ``pull-back" le long de $\St_2(\nu_{1,2}) \hookrightarrow \pi(\nu_{1,2},\eta)$ admet un carat\`ere infinit\'esimal. Le sens ``seulement si" est clair. Si $\widetilde{\pi}_0$ admet un caract\`ere infinit\'esimal $\xi$, alors pour tout $\nabla\in Z({\mathfrak gl}_2)$, le morphisme $\nabla-\xi(\nabla): \widetilde{\pi} \lra \widetilde{\pi}$ se factorise \`a travers $\widetilde{\pi}/\widetilde{\pi}_0\cong \pi(\nu_{1,2}, \eta)/\St_2(\nu_{1,2})) \lra \widetilde{\pi}$. Comme le socle $\soc_{\GL_2(\Q_p)} \pi(\nu_{1,2}, \eta)$ de $\pi(\nu_{1,2}, \eta)$ est $\St_2(\nu_{1,2})$, le morphisme $\nabla-\xi(\nabla)$ est nul. Le sens ``si" en d\'ecoule. La seconde partie de (i) d\'ecoule alors de \cite[Lem. 3.17(2)]{BD} et du (ii) du Lemme 6.2.1.\\
(ii) Si $\widetilde{\pi}\in \Ext^1_g(\pi(\nu_{1,2},\eta),\pi(\nu_{1,2},\eta))$, on voit facilement que $\St_2(\nu_{1,2})$ a multiplicit\'e $2$ dans la sous-repr\'esentation localement alg\'ebrique $\widetilde{\pi}^{\alg}$ de $\widetilde{\pi}$. Pour tout $\nabla\in Z({\mathfrak gl}_2)$, $\nabla-\xi(\nabla)$ annule $\widetilde{\pi}^{\alg}$ (o\`u $\xi$ est le caract\`ere infinit\'esimal). Utilisant (encore) $\soc_{\GL_2(\Q_p)} \pi(\nu_{1,2}, \eta)\cong \St_2(\nu_{1,2})$, on en d\'eduit que $\nabla-\xi(\nabla)$ annule tout $\widetilde{\pi}$.
\end{proof}

On note $\pi(\nu_{1,2},\eta)^{\rm inf}$ l'unique extension non scind\'ee de $\St_2(\nu_{1,2})$ par $\pi(\nu_{1,2}, \eta)$ (\`a isomorphisme pr\`es) avec un caract\`ere central et un caract\`ere infinit\'esimal donn\'ee par le (iii) du Lemme \ref{gl2}. La repr\'esentation localement analytique $\pi(\nu_{1,2},\eta)^{\rm inf}\boxtimes 1$ de $P_1(\Qp)$ est en particulier un $U(\fp_1)$-module.

\begin{lem}\label{key1}
Pour tout $u\in \St_2(\nu_{1,2})$, il existe~:
\begin{equation*}
\widetilde{u}\ \=\ \sum_{i=0}^{k_2+1} \fy_2^i \fy_3^{k_2+1-i} \!\otimes \widetilde{u}_i\in H^0\big(\fn_1, \text{U}({\mathfrak gl}_3)\otimes_{\text{U}(\fp_1)} (\pi(\nu_{1,2},\eta)^{\rm inf}\boxtimes 1)\big)
\end{equation*}
tel que $\widetilde{u}_{k_2+1}$ a pour image $u$ via la surjection naturelle $\pi(\nu_{1,2},\eta)^{\rm inf}\twoheadrightarrow \St_2(\nu_{1,2})$.
\end{lem}
\begin{proof}
Soit $v\in \St_2(\nu_{1,2})$ tel que $\big(\prod_{j=1}^i (-\frac{k_2+2-j}{2j})\big)\big(\prod_{j=1}^{k_2+1} (\fh+(k_1-k_2)+2j)\big) v=u$ (en notant que $\fh+k_1-k_2+2j$ est une bijection sur $\St_2(\nu_{1,2})$ pour $1\leq j\leq k_2+1$). Soit $\widetilde{v}$ un relev\'e arbitraire de $v$ dans $\pi(\nu_{1,2},\eta)^{\rm inf}$ via $\pi(\nu_{1,2},\eta)^{\rm inf}\twoheadrightarrow \St_2(\nu_{1,2})$. Le r\'esultat suit du Lemme \ref{ninv2} appliqu\'e \`a $u=\widetilde{v}$.
\end{proof}

Soit ${\rm Rep}^{\an,z}_E(L_{1}(\Q_p))$ la sous-cat\'egorie pleine de ${\rm Rep}^{\an}_E(L_{1}(\Q_p))$ des repr\'esentations qui sont unions croissantes de $BH$-sous-espaces stables sous $Z_{L_{1}}(\Q_p)$ (cf. \cite{Em2}). Pour $V$ dans ${\rm Rep}^{\an,z}_E(L_{1}(\Q_p))$, suivant \cite{Em2} on note $I_{{P}_1^-}^{\GL_3} (V)$ la sous-repr\'esentation ferm\'ee de $(\Ind_{{P}_1^-(\Q_p)}^{\GL_3(\Q_p)} V)^{\an}$ engendr\'ee par l'image de $V\otimes_E \delta_{P_1}\hookrightarrow J_{P_1}((\Ind_{{P}^-_1(\Q_p)}^{\GL_3(\Q_p)} V)^{\an})$ (\cite[Lem.~0.3]{Em2}) dans $(\Ind_{{P}^-_1(\Q_p)}^{\GL_3(\Q_p)} V)^{\an}$ via le rel\`evement canonique~:
\begin{equation}\label{lifting}
J_{P_1}\big((\Ind_{{P}^-_1(\Q_p)}^{\GL_3(\Q_p)} V)^{\an}\big)\longrightarrow (\Ind_{{P}^-_1(\Q_p)}^{\GL_3(\Q_p)} V)^{\an}
\end{equation}
o\`u $\delta_{P_1}$ est le caract\`ere module de $P_1(\Q_p)$ et $J_{P_1}$ le foncteur de Jacquet-Emerton relativement \`a $P_1$. Rappelons que l'application (\ref{lifting}) d\'epend du choix d'un sous-groupe ouvert compact de $N_1(\Qp)=N_{P_1}(\Q_p)$ mais pas la repr\'esentation $I_{{P}^-_1}^{\GL_3}(V)$. Le lemme facile suivant sera utile.

\begin{lem}\label{jacquetII}
Soit $0 \rightarrow V_1 \rightarrow V_2 \rightarrow V_3 \rightarrow 0$ une suite exacte courte dans ${\rm Rep}^{\an,z}_E(L_{1}(\Q_p))$.\\
(i) L'injection naturelle $(\Ind_{P_1^-(\Q_p)}^{\GL_3(\Q_p)} V_1)^{\an} \hookrightarrow (\Ind_{P_1^-(\Q_p)}^{\GL_3(\Q_p)} V_2)^{\an}$ induit une injection $I_{P_1^-}^{\GL_3} (V_1) \hookrightarrow I_{P_1^-}^{\GL_3} (V_2)$.\\
(ii) La surjection naturelle $(\Ind_{P_1^-(\Q_p)}^{\GL_3(\Q_p)} V_2)^{\an} \twoheadrightarrow (\Ind_{P_1^-(\Q_p)}^{\GL_3(\Q_p)} V_3)^{\an}$ induit une surjection $I_{P_1^-}^{\GL_3} (V_2) \twoheadrightarrow I_{P_1^-}^{\GL_3} (V_3)$.
\end{lem}
\begin{proof}
Le lemme suit facilement des d\'efinitions et du diagramme commutatif de suites exactes~:
\begin{equation*}
\begin{CD}
0 @>>> V_1\otimes_E \delta_{P_1} @>>> V_2\otimes_E \delta_{P_1} @>>> V_3 \otimes_E\delta_{P_1} @>>>0 \\
@. @VVV @VVV @VVV@. \\
0 @>>> (\Ind_{P_1^-(\Q_p)}^{\GL_3(\Q_p)} V_1)^{\an} @>>> (\Ind_{P_1^-(\Q_p)}^{\GL_3(\Q_p)} V_2)^{\an} @>>> (\Ind_{P_1^-(\Q_p)}^{\GL_3(\Q_p)} V_3)^{\an} @>>> 0.
\end{CD}
\end{equation*}
\end{proof}

Soit $\nu\=(k_1, k_2,k_3)$ vu comme poids dominant (par rapport au Borel sup\'erieur de ${\rm GL}_3$), on utilise dans la suite les notations de \cite[\S\S~3.2.2,~3.3.1~\&~3.3.3]{BD}, sauf que l'on remplace $\lambda$, $\lambda_{1,2}$ par $\nu$, $\nu_{1,2}$ et que la ``dot action'' est ici par rapport au Borel inf\'erieur (cf. \S~\ref{notabene}) alors que dans {\it loc.cit.} elle est par rapport au Borel sup\'erieur. Ainsi par exemple la repr\'esentation $I(s\cdot \lambda_{1,2})$ de \cite[\S~3.2.2]{BD} devient ici $I(-s\cdot (-\nu_{1,2}))$, le $U({\mathfrak gl}_3)$-module $\overline{L}(-s_1\cdot \lambda)$ de \cite[\S~3.3.1]{BD} devient ${L}^-(s_1\cdot (-\nu))$, etc. Pour $i\in \{1,3,5\}$ on note aussi $C_{1,i}\=C_{e_1-e_2,i}$ et $C_{2,i}\=C_{e_2-e_3,i}$ (cf. (\ref{calphai})).

On note $\cV_1\=L(\nu)\otimes_E(\Ind_{{P}^-_1(\Q_p)}^{\GL_3(\Q_p)} \St_2\boxtimes 1)^{\infty}$ (cf. \cite[(53)]{Br1}), $\cV_2$ l'unique sous-repr\'esentation de $(\Ind_{{P}^-_1(\Q_p)}^{\GL_3(\Q_p)} I(-s\cdot (-\nu_{1,2})) \boxtimes x^{k_3})^{\an}$ donn\'ee par $\cF_{{P}^-_2}^{\GL_3}({L}^-(s_1\cdot (-\nu)), 1\boxtimes (\Ind_{B_2^-(\Qp)}^{{\rm GL}_2(\Qp)}1)^\infty)$ (une extension non scind\'ee de $C_{1,1}=\cF_{{P}^-_2}^{\GL_3}({L}^-(s_1\cdot (-\nu)), 1\boxtimes \St_2)$ par $\cF_{{P}^-_2}^{\GL_3}({L}^-(s_1\cdot (-\nu)), 1)$), $\cV_3\=L(\nu)\otimes_E (\Ind_{{P}^-_1(\Q_p)}^{\GL_3(\Q_p)} 1)^{\infty}$ et~:
\begin{multline*}
\cV_4\=C_{1,3}= \cF_{{P}^-_2}^{\GL_3}\big({L}^-(s_1\cdot (-\nu)), |\cdot|^{-1} \boxtimes (\Ind_{{B}^-_2(\Q_p)}^{\GL_2(\Q_p)} |\cdot|\otimes 1)^{\infty}\big)\\ \cong \soc_{\GL_3(\Q_p)} \big(\Ind_{{P}^-_1(\Q_p)}^{\GL_3(\Q_p)} \widetilde{I}(-s\cdot (-\nu_{1,2})) \boxtimes x^{k_3}\big)^{\an}
\end{multline*}
o\`u $B_2^-(\Qp)$ est le Borel des matrices triangulaires inf\'erieures dans $\GL_2(\Q_p)$.

\begin{lem}\label{ipg1}
On a $\cV_1\cong I_{{P}^-_1}^{\GL_3}(\St_2(\nu_{1,2})\boxtimes x^{k_3})$, $\cV_2 \cong I_{{P}^-_1}^{\GL_3}(I(-s\cdot (-\nu_{1,2}))\boxtimes x^{k_3})$, $\cV_3 \cong I_{{P}^-_1}^{\GL_3}(L(\nu_{1,2})\boxtimes x^{k_3})$ et $\cV_4\cong I_{{P}^-_1}^{\GL_3}(\widetilde{I}(-s\cdot (-\nu_{1,2})) \boxtimes x^{k_3})$.
\end{lem}
\begin{proof}
On montre l'\'enonc\'e pour $\cV_2$, les autres cas \'etant analogues ou plus simples. Pour tout constituant irr\'eductible $W$ de $(\Ind_{{P}^-_1(\Q_p)}^{\GL_3(\Q_p)} I(-s\cdot (-\nu_{1,2})) \boxtimes x^{k_3})^{\an}/\cV_2$ (cf. la liste \cite[(3.70)]{BD}), on montre que $(I(-s\cdot (-\nu_{1,2}))\boxtimes x^{k_3})\otimes_E \delta_{P_1}$ n'est pas une sous-repr\'esentation de $J_{P_1}(W)$. En effet, si l'on a $(I(-s\cdot (-\nu_{1,2})) \boxtimes x^{k_3})\otimes_E \delta_{P_1} \hookrightarrow J_{P_1}(W)$, alors en appliquant $J_{B\cap L_1}$ puis \cite[Th.~5.3(2)]{HL}, on en d\'eduit une injection $T(\Q_p)$-\'equivariante~:
\begin{equation}\label{jac1}
\delta_{-(s_1\cdot (-\nu))}(|\cdot|^2 \otimes 1 \otimes |\cdot|^{-2}) \hookrightarrow J_B(W).
\end{equation}
Mais on d\'eduit facilement de \cite[Th.~4.3]{Br4} et \cite[Cor.~4.25]{OS2} (et de la structure des modules de Verma pour ${\mathfrak gl}_3$) que (\ref{jac1}) ne peut exister. Cela implique que l'image de $(I(-s\cdot (-\nu_{1,2}))\boxtimes x^{k_3})\otimes_E \delta_{P_1}\hookrightarrow J_{P_1}( (\Ind_{{P}^-_1(\Q_p)}^{\GL_3(\Q_p)} I(-s\cdot (-\nu_{1,2})) \boxtimes x^{k_3})^{\an})$ tombe dans $J_{P_1}(\cV_2)$, et donc par d\'efinition $I_{{P}^-_1}^{\GL_3}( I(-s\cdot (-\nu_{1,2})) \boxtimes x^{k_3}) \hookrightarrow \cV_2$. Un argument analogue montre aussi que $(I(-s\cdot (-\nu_{1,2}))\boxtimes x^{k_3})\otimes_E \delta_{P_1}$ ne peut se plonger dans $J_{P_1}(\cF_{{P}^-_2}^{\GL_3}({L}^-(s_1\cdot (-\nu)),1))$. On en d\'eduit $I_{{P}^-_1}^{\GL_3}( I(-s\cdot (-\nu_{1,2})) \boxtimes x^{k_3}) \cong \cV_2$.
\end{proof}

Soit $\cW\=\ker\big((\Ind_{{P}^-_1(\Q_p)}^{\GL_3(\Q_p)} \pi(\nu_{1,2}, \eta)\boxtimes x^{k_3})^{\an} \twoheadrightarrow \widetilde{\Pi}^1(\nu, \eta)\big)$ (cf. la discussion qui suit \cite[Rem. 3.41]{BD}), alors $\cW$ est isomorphe \`a une extension de $L(\nu)$ par $v_{{P}^-_2}^{\an}(\nu)$. On note $\cV$ l'unique sous-repr\'esentation de $(\Ind_{{P}^-_1(\Q_p)}^{\GL_3(\Q_p)} \pi(\nu_{1,2}, \eta)\boxtimes x^{k_3})^{\an}$ qui est une extension de $\Pi^1(\nu, \eta)$ par $\cW$ (cf. \cite[Lem. 3.40~\&~Rem. 3.41]{BD} pour $\Pi^1(\nu, \eta)$).

\begin{lem}\label{ipg2}
On a une injection $I_{{P}^-_1}^{\GL_3} (\pi(\nu_{1,2}, \eta)\boxtimes x^{k_3})\hookrightarrow \cV$ qui est un isomorphisme si $\eta$ n'est pas lisse.
\end{lem}
\begin{proof}
Par le m\^eme argument que celui dans la preuve du Lemme \ref{ipg1}, on montre que, pour tout constituant irr\'eductible $W$ de $(\Ind_{{P}^-_1(\Q_p)}^{\GL_3(\Q_p)} \pi(\nu_{1,2}, \eta)\boxtimes x^{k_3})^{\an}/\cV$, aucun des constituants irr\'eductibles de $\pi(\nu_{1,2}, \eta)\boxtimes x^{k_3}$ ne peut s'injecter dans $J_{P_1}(W)$. On en d\'eduit que l'injection $ \pi(\nu_{1,2}, \eta)\boxtimes x^{k_3}\hookrightarrow J_{P_1}((\Ind_{{P}^-_1(\Q_p)}^{\GL_3(\Q_p)} \pi(\nu_{1,2}, \eta)\boxtimes x^{k_3})^{\an})$ se factorise par $J_{P_1}(\cV)$ et donc $I_{{P}^-_1}^{\GL_3} (\pi(\nu_{1,2}, \eta)\boxtimes x^{k_3})\subseteq \cV $. Par le Lemme \ref{jacquetII} et le Lemme \ref{ipg1}, on montre facilement que tous les $\cV_i$ pour $i=1, \dots, 4$ apparaissent comme sous-quotients de $I_{{P}^-_1}^{\GL_3} (\pi(\nu_{1,2}, \eta)\boxtimes x^{k_3})$. Lorsque $\eta$ n'est pas lisse, par la structure de $\cV$ (cf. \cite[\S~3.3.3]{BD}, cf. aussi \cite[Rem. 3.41]{BD} avec \cite[Lem.~4.31~\&~Lem.~4.34]{Qi} pour la structure de $\Pi^1(\nu, \eta)$), il n'est pas difficile de montrer que $\cV$ est la plus petite sous-repr\'esentation de $(\Ind_{{P}^-_1(\Q_p)}^{\GL_3(\Q_p)} \pi(\nu_{1,2}, \eta)\boxtimes x^{k_3})^{\an}$ contenant le constituant $\cV_4$. Cela termine la preuve.
\end{proof}

Notons $\cV'$ l'unique sous-repr\'esentation de $(\Ind_{{P}^-_1(\Q_p)}^{\GL_3(\Q_p)} \pi(\nu_{1,2}, \eta)\boxtimes x^{k_3})^{\an}$ qui est une extension de $\Pi^1(\nu, \eta)^+$ par $\cW$ (cf. \cite[\S~3.3.4]{BD} pour $\Pi^1(\nu, \eta)^+$). La repr\'esentation $\cV'$ est donc aussi une extension de $C_{2,1}=\cF_{{P}^-_1}^{\GL_3}({L}^-(s_2\cdot (-\nu)), \St_2\boxtimes 1)$ par $\cV$. Par \cite[Cor. 4.9,~(4.41),~(4.42)~\&~(4.44)]{Sc1}, il n'est pas difficile de v\'erifier que l'on a~:
\begin{equation}\label{nul2}
\Ext^1_{\GL_3(\Q_p)}(\St_3(\nu), C_{2,1})=\Ext^1_{\GL_3(\Q_p)}(\St_3(\nu), \cF_{{P}^-_1}^{\GL_3}({L}^-(s_2\cdot (-\nu),1))=0.
\end{equation}
La deuxi\`eme nullit\'e avec \cite[Lem. 2.24]{Di2} impliquent~: 
\begin{equation}\label{nul1}
\Ext^1_{\GL_3(\Q_p)}\big(\St_3(\nu), \widetilde{\Pi}^1(\nu, \eta)/\Pi^1(\nu, \eta)\big)=0.
\end{equation}
Par \cite[(3.87)~\&~(3.77)]{BD} on a une application naturelle~:
\begin{equation*}
\iota: \Ext^1_{\GL_2(\Q_p)}\big(\St_2(\nu_{1,2}), \pi(\nu_{1,2}, \eta)\big) \twoheadrightarrow \Ext^1_{\GL_3(\Q_p)}\big(v_{{P}^-_2}^{\infty} (\nu), \Pi^1(\nu, \eta)^+\big)
\end{equation*}
telle que, pour $V\in \Ext^1_{\GL_2(\Q_p)}(\St_2(\nu_{1,2}), \pi(\nu_{1,2}, \eta))$, la repr\'esentation $\iota(V)$ est un sous-quotient de $(\Ind_{{P}^-_1(\Q_p)}^{\GL_3(\Q_p)} V\boxtimes x^{k_3})^{\an}$. Il suit de (\ref{nul1}) qu'il existe une unique sous-repr\'esentation $\iota(V)^+$ de $(\Ind_{{P}^-_1(\Q_p)}^{\GL_3(\Q_p)} V\boxtimes x^{k_3})^{\an}$ qui est une extension de $\cV_1= L(\nu)\otimes_E(\Ind_{{P}^-_1(\Q_p)}^{\GL_3(\Q_p)}\St_2\boxtimes 1)^{\infty}$ par $\cV'$ telle que $\iota(V)$ est un sous-quotient de $\iota(V)^+$. On peut visualiser la repr\'esentation $\iota(V)^+$ comme suit (son socle $v_{P_2^-}^{\infty}(\nu)$ \'etant \`a gauche)~:
\begin{equation*}\footnotesize
\begindc{\commdiag}[200]
\obj(0,4)[a]{$v_{P_2^-}^{\infty}(\nu)$}
\obj(4,4)[b]{$\St_3(\nu)$}
\obj(8,4)[c]{$C_{2,1}$}
\obj(2,2)[d]{$\cF_{P_2^-}^{\GL_3}(L^-(s_1\!\cdot \!(-\nu)),1)$}
\obj(6,2)[e]{$C_{1,1}$}
\obj(10,2)[f]{$\widetilde{C}_{1,2}$}
\obj(4,0)[g]{$L(\nu)$}
\obj(8,0)[h]{$v_{P_1^-}^{\infty}(\nu)$}
\obj(12,0)[i]{$C_{1,3}$}
\obj(16,0)[j]{$v_{P_2^-}^{\infty}(\nu)$}
\obj(20,0)[k]{$\St_3(\nu)$}
\mor{a}{b}{}[+1,\solidline]
\mor{b}{c}{}[+1,\solidline]
\mor{a}{d}{}[+1,\solidline]
\mor{b}{e}{}[+1,\solidline]
\mor{d}{e}{}[+1,\solidline]
\mor{e}{f}{}[+1,\solidline]
\mor{d}{g}{}[+1,\solidline]
\mor{e}{h}{}[+1,\solidline]
\mor{f}{i}{}[+1,\solidline]
\mor{g}{h}{}[+1,\solidline]
\mor{h}{i}{}[+1,\solidline]
\mor{i}{j}{}[+1,\solidline]
\mor{j}{k}{}[+1,\solidline]
\mor{c}{j}{}[+1,\dashline]
\enddc
\end{equation*} 
o\`u $\widetilde{C}_{1,2}\=\cF_{P_1^-}^{\GL_3}(L^-(s_2s_1\cdot (-\nu)),1)$. On voit que la repr\'esentation $\iota(V)^+$ est de la forme $\cW-\iota(V)-\St_3(\nu)$, et aussi de la forme $\begindc{\commdiag}[32]
\obj(14,10)[b]{$\cV$}
\obj(26,18)[d]{$C_{2,1}$}
\obj(38, 10)[e]{$\cV_1$.}
\mor{b}{d}{}[+1,\solidline]
\mor{d}{e}{}[+1,\dashline]
\mor{b}{e}{}[+1,\solidline]
\enddc$

\begin{lem}\label{equivalent}
Soit $V$ une extension de $\St_2(\nu_{1,2})$ par $\pi(\nu_{1,2}, \eta)$. On a une injection $I_{{P}^-_1}^{\GL_3}(V \boxtimes x^{k_3}) \hookrightarrow \iota(V)^+$. De plus, si $\eta$ n'est pas lisse, les assertions suivantes sont \'equivalentes~:
\begin{enumerate}
\item[(i)]$I_{{P}^-_1}^{\GL_3}(V \boxtimes x^{k_3}) \subsetneq \iota(V)^+$;
\item[(ii)]$C_{2,1}$ n'est pas un constituant de $I_{{P}^-_1}^{\GL_3}(V \boxtimes x^{k_3})$;
\item[(iii)]le sous-quotient $\!\begin{xy}(-60,0)*+{C_{2,1}}="a";(-42,0)*+{v_{{P}^-_2}^{\infty}(\nu)}="b";{\ar@{--}"a";"b"}\end{xy}\!$ de $\iota(V)$ est scind\'e.
\end{enumerate}
\end{lem}
\begin{proof}
La premi\`ere assertion se d\'emontre par le m\^eme argument que pour la premi\`ere assertion du Lemme \ref{ipg2}. Supposons maintenant $\eta$ non lisse. L'injection $\pi(\nu_{1,2},\eta)\hookrightarrow V$ et le Lemme \ref{jacquetII} donnent une injection~:
\begin{equation*}
\cV\cong I_{{P}^-_1}^{\GL_3} (\pi(\nu_{1,2}, \eta)\boxtimes x^{k_3})\hookrightarrow I_{{P}^-_1}^{\GL_3}(V \boxtimes x^{k_3}).
\end{equation*} 
Par le Lemme \ref{jacquetII} et le Lemme \ref{ipg1}, on sait aussi que $I_{{P}^-_1}^{\GL_3}(V \boxtimes x^{k_3})$ contient deux copies de $\cV_1$. Par la discussion pr\'ec\'edant ce Lemme \ref{equivalent}, on en d\'eduit facilement l'\'equivalence entre (i) et (ii). Si le sous-quotient $\!\begin{xy}(-60,0)*+{C_{2,1}}="a";(-43,0)*+{v_{{P}^-_2}^{\infty}(\nu)}="b";{\ar@{--}"a";"b"}\end{xy}\!$ de $\iota(V)$ n'est pas scind\'e, on obtient facilement que $\iota(V)^+$ est la plus petite sous-repr\'esentation ferm\'ee de $(\Ind_{{P}^-_1(\Q_p)}^{\GL_3(\Q_p)} V\boxtimes x^{k_3})^{\an}$ qui contient (en sous-quotient) $\cV_4$ et deux copies de $\cV_1$, d'o\`u $I_{{P}^-_1}^{\GL_3}(V \boxtimes x^{k_3}) \cong \iota(V)^+$. Cela montre (i) $\Rightarrow$ (ii) (et aussi (ii) $\Rightarrow $ (iii)). R\'eciproquement, supposons ce sous-quotient de $\iota(V)$ scind\'e. Avec (\ref{nul2}) on en d\'eduit $\iota(V)^+/\cV\cong \cV_1 \oplus C_{2,1}$ et en particulier que $\iota(V)^+$ contient une sous-repr\'esentation de la forme $\cV-\cV_1$, qui doit contenir $I_{{P}^-_1}^{\GL_3}(V \boxtimes x^{k_3})$ par le m\^eme argument que pour la premi\`ere assertion du Lemme \ref{ipg2}. Cela montre (iii) $\Rightarrow$ (ii) et termine la preuve.
\end{proof}

\begin{rem}
{\rm Si $\eta$ est lisse, on peut encore montrer par des arguments similaires que les (ii) et (iii) du Lemme \ref{equivalent} sont \'equivalents.}
\end{rem}

\begin{lem}\label{key2}
Si $\eta$ n'est pas lisse, alors $C_{2,1}$ n'est pas un constituant de $I_{{P}^-_1}^{\GL_3}(\pi(\nu_{1,2},\eta)^{\rm inf} \boxtimes x^{k_3})$ (cf. avant le Lemme \ref{key1} pour $\pi(\nu_{1,2},\eta)^{\rm inf}$).
\end{lem}
\begin{proof}
En tordant par $x^{-k_3}\circ {\det}$, on peut supposer $k_3=0$ (et $k_1\geq k_2\geq 0$). Les injections naturelles (cf. par exemple (\ref{deIaC}) pour la seconde)~:
\begin{multline*}
C^{\infty}_c(N_{1}(\Q_p), \pi(\nu_{1,2},\eta)^{\rm inf}\boxtimes 1) \hookrightarrow C^{\an}_c(N_{1}(\Q_p), \pi(\nu_{1,2},\eta)^{\rm inf}\boxtimes 1)\\
\hookrightarrow \big(\Ind_{{P}^-_1(\Q_p)}^{\GL_3(\Q_p)} \pi(\nu_{1,2},\eta)^{\rm inf}\boxtimes 1\big)^{\an}
\end{multline*}
induisent par \cite[(2.8.7)~\&~(2.5.27)]{Em2} un morphisme $(\text{U}({\mathfrak gl}_3)$, $P_1(\Q_p))$-\'equivariant~: 
\begin{equation}\small\label{adj00}
\iota_1: \text{U}({\mathfrak gl}_3) \otimes_{\text{U}(\fp_1)} C^{\infty}_c(N_1(\Q_p), \pi(\nu_{1,2},\eta)^{\rm inf} \boxtimes 1) \lra \big(\Ind_{{P}^-_1(\Q_p)}^{\GL_3(\Q_p)} \pi(\nu_{1,2},\eta)^{\rm inf}\boxtimes 1\big)^{\an}
\end{equation}
dont l'image tombe dans $I_{{P}^-_1}^{\GL_3}(\pi(\nu_{1,2}, \eta)^{\rm inf} \boxtimes 1)$. De plus par \cite[(2.8.7)]{Em2} le morphisme $\iota_1$ se factorise comme suit (cf. \cite[Def. 2.5.21]{Em2} pour $C^{\rm lp}_c$)~:
\begin{multline}
\label{adj01} \text{U}({\mathfrak gl}_3) \otimes_{\text{U}(\fp_1)} C^{\infty}_c(N_1(\Q_p), \pi(\nu_{1,2},\eta)^{\rm inf} \boxtimes 1) \lra C^{\rm lp}_c(N_1(\Q_p), \pi(\nu_{1,2},\eta)^{\rm inf}\boxtimes 1) \\ \hookrightarrow C^{\an}_c(N_1(\Q_p), E) \otimes_E (\pi(\nu_{1,2},\eta)^{\rm inf}\boxtimes 1),
\end{multline}
o\`u $\fn_1$ agit sur $C^{\an}_c(N_1(\Q_p), E) \otimes_E (\pi(\nu_{1,2},\eta)^{\rm inf}\boxtimes 1)$ via l'action diagonale avec l'action triviale sur le deuxi\`eme facteur. Comme $H^0(\fn_1, C^{\an}_c(N_1(\Q_p),E))\cong C^{\infty}_c(N_1(\Q_p),E)$ on a~:
\begin{multline*}
H^0\big(\fn_1, C^{\an}_c(N_1(\Q_p), E) \otimes_E (\pi(\nu_{1,2},\eta)^{\rm inf}\boxtimes 1)\big)\cong C^{\infty}_c(N_1(\Q_p), E)\otimes_E (\pi(\nu_{1,2},\eta)^{\rm inf}\boxtimes 1)\\ 
\cong C^{\infty}_c(N_1(\Q_p), \pi(\nu_{1,2},\eta)^{\rm inf}\boxtimes 1).
\end{multline*}
Appliquant $H^0(\fn_1,-)$ \`a (\ref{adj01}) et comme $\fn_1$ annule $C^{\infty}_c(N_1(\Q_p), \pi(\nu_{1,2},\eta)^{\rm inf} \boxtimes 1)$, on d\'eduit que (\ref{adj00}) induit un morphisme encore not\'e $\iota_1$~:
\begin{multline*}
\iota_1\! : \! H^0\big(\fn_1, \text{U}({\mathfrak gl}_3) \otimes_{\text{U}(\fp_1)} C^{\infty}_c(N_1(\Q_p), \pi(\nu_{1,2},\eta)^{\rm inf} \boxtimes 1)\big) \twoheadrightarrow C^{\infty}_c(N_1(\Q_p), \pi(\nu_{1,2},\eta)^{\rm inf}\boxtimes 1)
 \\
 \hookrightarrow I_{{P}^-_1}^{\GL_3}(\pi(\nu_{1,2}, \eta)^{\rm inf} \boxtimes 1).
\end{multline*}
L'\'enonc\'e du Lemme \ref{key1} est encore valable, par la m\^eme preuve, en rempla\c cant $\St_2(\nu_{1,2})$ par $C_c^{\infty}(N_1(\Q_p), \St_2(\nu_{1,2}) \boxtimes 1)$ et le $\fp_1$-module $\pi(\nu_{1,2},\eta)^{\rm inf}\boxtimes 1$ par $C_c^{\infty}(N_1(\Q_p), \pi(\nu_{1,2},\eta)^{\rm inf} \boxtimes 1)$ (en remarquant que le facteur en plus $C_c^{\infty}(N_1(\Q_p), E)$ n'affecte pas l'action de $\fp_1$). Donc, pour tout $u\in C_c^{\infty}(N_1(\Q_p), \St_2(\nu_{1,2}) \boxtimes 1)$, il existe $\widetilde{u}$ dans la source de $\iota_1$ v\'erifiant les conditions du Lemme \ref{key1}. Soit $v=\iota_1(\widetilde{u}) \in C^{\infty}_c(N_1(\Q_p), \pi(\nu_{1,2},\eta)^{\rm inf} \boxtimes 1) $, on a donc $\iota_1(\widetilde{u}-1\otimes v)=0$. On a par ailleurs un diagramme commutatif~:
\begin{equation*}\tiny
\begin{CD}
0 @>>> (\pi(\nu_{1,2}, \eta)\boxtimes 1)\otimes_E \delta_{P_1} @>>> (\pi(\nu_{1,2}, \eta)^{\rm inf} \boxtimes 1)\otimes_E \delta_{P_1} @>>> (\St_2(\nu_{1,2}) \boxtimes 1)\otimes_E \delta_{P_1} \\
@. @VVV @VVV @VVV \\
0 @>>> J_{P_1}(\cV) @>>> J_{P_1}(I_{{P}^-_1}^{\GL_3}(\pi(\nu_{1,2}, \eta)^{\rm inf} \boxtimes 1 )) @>>> J_{P_1}\big(I_{{P}^-_1}^{\GL_3}(\pi(\nu_{1,2}, \eta)^{\rm inf} \boxtimes 1 )/\cV\big) \\
@. @VVV @VVV @VVV \\
0 @>>> \cV @>>> I_{{P}^-_1}^{\GL_3}(\pi(\nu_{1,2}, \eta)^{\rm inf} \boxtimes 1 ) @>>> I_{{P}^-_1}^{\GL_3}(\pi(\nu_{1,2}, \eta)^{\rm inf} \boxtimes 1 )/\cV
\end{CD}
\end{equation*}
qui induit un diagramme commutatif avec fl\`eches horizontales surjectives~:
\begin{equation*}\footnotesize
\begin{CD}
\text{U}({\mathfrak gl}_3) \otimes_{\text{U}(\fp_1)} C^{\infty}_c(N_1(\Q_p),\pi(\nu_{1,2}, \eta)^{\rm inf}\boxtimes1) @> {\rm pr} >> \text{U}({\mathfrak gl}_3) \otimes_{\text{U}(\fp_1)} C^{\infty}_c(N_1(\Q_p), \St_2(\nu_{1,2})\boxtimes 1) \\
@V \iota_1 VV @V \iota_2 VV \\
I_{{P}^-_1}^{\GL_3}(\pi(\nu_{1,2}, \eta)^{\rm inf} \boxtimes 1 ) @>>> I_{{P}^-_1}^{\GL_3}(\pi(\nu_{1,2}, \eta)^{\rm inf} \boxtimes 1 )/\cV
\end{CD}
\end{equation*}
o\`u $\iota_2$ est induit par la compos\'ee verticale de droite du diagramme juste avant. Soit $0\neq u\in C_c^{\infty}(N_1(\Q_p), \St_2(\nu_{1,2}) \boxtimes 1)$, comme $\iota_1(\widetilde{u}-1\otimes v)=0$ on a $\iota_2({\rm pr}(\widetilde{u}-1\otimes v))=0$. Mais ${\rm pr}(\widetilde{u}-1\otimes v)\neq 0$ puisque, par construction (cf. Lemme \ref{key1}), il contient le terme non nul $\fy_2^{k_2+1} \otimes u$. On en d\'eduit que l'application $\iota_2$ n'est pas injective.\\
Supposons maintenant que $C_{2,1}$ apparaisse dans $I_{{P}^-_1}^{\GL_3}(\pi(\nu_{1,2},\eta)^{\rm inf} \boxtimes 1)$. Par l'\'equivalence entre (ii) et (iii) dans le Lemme \ref{equivalent} et le fait que l'unique extension non scind\'ee de $\cV_1$ par $C_{2,1}$ est isomorphe \`a $\cF_{{P}^-_1}^{\GL_3}((\text{U}({\mathfrak gl}_3)\otimes_{\text{U}({\fp}^-_1)} L^-(-\nu)_{P_1})^{\vee}, \St_2\boxtimes 1)$ o\`u $(-)^\vee$ est la dualit\'e en \cite[\S~3.2]{Hu} (car d'une part on a $\dim_E \Ext^1_{\GL_3(\Q_p)}(\cV_1, C_{2,1})\leq 1$ par \cite[Lem. 3.42(2)]{BD} et (\ref{nul2}), d'autre part $\cF_{{P}^-_1}^{\GL_3}((\text{U}({\mathfrak gl}_3)\otimes_{\text{U}({\fp}^-_1)} L^-(-\nu)_{P_1})^{\vee}, \St_2\boxtimes 1)$ est une telle extension non scind\'ee par \cite[Cor.~4.25]{OS2}), on obtient~:
\begin{equation*}
I_{{P}^-_1}^{\GL_3}(\pi(\nu_{1,2}, \eta)^{\rm inf} \boxtimes 1)/\cV \cong \cF_{{P}^-_1}^{\GL_3}\big((\text{U}({\mathfrak gl}_3)\otimes_{\text{U}({\fp}^-_1)} L^-(-\nu)_{P_1})^{\vee}, \St_2\boxtimes 1\big).
\end{equation*}
Mais alors il suit de \cite[Prop. 3.4]{Br4} que l'application $\iota_2$ doit \^etre injective, ce qui est une contradiction.
\end{proof}

\begin{prop}\label{split}
Si $\eta$ n'est pas lisse, alors la compos\'ee~:
\begin{multline*}
\Ext^1_{\inf}(\St_2(\nu_{1,2}), \pi(\nu_{1,2}, \eta)) \hookrightarrow \Ext^1_{\GL_2(\Q_p)}(\St_2(\nu_{1,2}), \pi(\nu_{1,2}, \eta)) \\ \buildrel\iota\over \longrightarrow \Ext^1_{\GL_3(\Q_p)}(v_{{P}^-_2}^{\infty}(\nu), \Pi^1(\nu, \eta)^+)
\end{multline*}
se factorise en un isomorphisme~:
\begin{equation}\label{film0}
\Ext^1_{\inf}(\St_2(\nu_{1,2}), \pi(\nu_{1,2}, \eta)) \buildrel{\sim}\over\longrightarrow \Ext^1_{\GL_3(\Q_p)}(v_{{P}^-_2}^{\infty}(\nu), \Pi^1(\nu, \eta)).
\end{equation}
\end{prop}
\begin{proof}
On montre d'abord que pour $V\in \Ext^1_{\inf}(\St_2(\nu_{1,2}), \pi(\nu_{1,2}, \eta))$, le sous-quotient $\!\begin{xy}(-60,0)*+{C_{2,1}}="a";(-43,0)*+{v_{{P}^-_2}^{\infty}(\nu)}="b";{\ar@{--}"a";"b"}\end{xy}\!$ de $\iota(V)$ est scind\'e. Si $V\in \Ext^1_{\inf,Z}(\St_2(\nu_{1,2}), \pi(\nu_{1,2}, \eta))$, i.e. $V\simeq \pi(\nu_{1,2},\eta)^{\rm inf}$, cela suit du Lemme \ref{key2} et de l'\'equivalence entre (ii) et (iii) du Lemme \ref{equivalent}. Si $V$ est isomorphe au ``push-forward'' (le long de $\St_2(\nu_{1,2})\hookrightarrow \pi(\nu_{1,2}, \eta)$) de l'unique extension localement alg\'ebrique non scind\'ee de $\St_2(\nu_{1,2})$ par $\St_2(\nu_{1,2})$, en utilisant l'argument dans la preuve de \cite[Lem. 3.44(2)]{BD} et \cite[Prop. 3.35]{BD} il n'est pas difficile de montrer que $\iota(V)$ est isomorphe au ``push-forward'' (le long de $\St_3(\nu)\hookrightarrow \Pi^1(\nu, \eta)^+$) d'une extension de $v_{{P}^-_2}^{\infty}(\nu)$ par $\St_3(\nu)$, et donc $\!\begin{xy}(-60,0)*+{C_{2,1}}="a";(-42,0)*+{v_{{P}^-_2}^{\infty}(\nu)}="b";{\ar@{--}"a";"b"}\end{xy}\!$ est encore scind\'e dans $\iota(V)$. Par le (ii) du Lemme \ref{gl2}, ces deux extensions $V$ forment une base de $\Ext^1_{\inf}(\St_2(\nu_{1,2}), \pi(\nu_{1,2}, \eta))$. L'application en (\ref{film0}) s'obtient donc en ``oubliant le constituant $C_{2,1}$" dans $\iota(V)$ (ce que l'on peut faire car $\!\begin{xy}(-60,0)*+{C_{2,1}}="a";(-42,0)*+{v_{{P}^-_2}^{\infty}(\nu)}="b";{\ar@{--}"a";"b"}\end{xy}\!$ est scind\'e). Comme $\eta$ n'est pas lisse, par \cite[Lem. 3.44(2)]{BD} et la preuve du Lemme \ref{gl2}, on voit que l'application (\ref{film0}) est injective. Comme $\dim_E \Ext^1_{\GL_3(\Q_p)}(v_{{P}^-_2}^{\infty}(\nu), \Pi^1(\nu, \eta))=2$ (cf. la preuve de \cite[Prop. 3.49]{BD}), elle est bijective en comparant les dimensions.
\end{proof}

\begin{rem}
{\rm Si $\eta$ est lisse, l'\'enonc\'e du Lemme \ref{key2} est encore vrai (et la preuve est bien plus facile). Par contre l'application (\ref{film0}) n'est plus injective.}
\end{rem}

\end{document}